\documentclass{amsart}

\usepackage{amsmath}
\usepackage{amssymb,amsthm}
\usepackage[all]{xy}
\usepackage{graphicx}
\usepackage{mathrsfs}
\usepackage{stmaryrd}
\usepackage[active]{srcltx}
\usepackage{euscript}

\usepackage{mathabx,epsfig}

\RequirePackage{ifpdf}
\ifpdf
  \IfFileExists{pdfsync.sty}{\RequirePackage{pdfsync}}{}
  \RequirePackage[pdftex, 
   colorlinks = true,
   pagebackref,
   bookmarksopen=true]{hyperref}
\else
  \RequirePackage[hypertex]{hyperref}
\fi

\usepackage{tikz}
\tikzset{
	MyPersp/.style={scale=2,x={(0.8cm,0cm)},y={(0cm,0.25cm)},
    z={(0cm,1cm)}},
	MyPoints/.style={fill=white,draw=black,thick}
		}

\usetikzlibrary{decorations.pathmorphing,arrows,calc}

 \definecolor{darkgreen}{HTML}{336633}
 \definecolor{darkred}{HTML}{993333}

\RequirePackage{color}
\definecolor{myred}{rgb}{0.75,0,0}
\definecolor{mygreen}{rgb}{0,0.5,0}
\definecolor{myblue}{rgb}{0,0,0.65}

\newcommand{\C}{\mathbb{C}}

\newcommand{\Q}{\mathbb{Q}}
\newcommand{\F}{\mathbb{F}}
\newcommand{\bbL}{\mathbb{L}}
\newcommand{\K}{\mathbb{K}}

\newcommand{\bk}{\Bbbk}

\newcommand{\Z}{\mathbb{Z}}

\newcommand{\cO}{\mathcal{O}}

\newcommand{\Gm}{\mathbb{G}_m}

\DeclareMathOperator{\Spec}{Spec}
\newcommand{\pt}{\mathrm{pt}}

\newcommand{\cone}{\mathrm{cone}}

\newcommand{\Mod}{\mathsf{-Mod}}

\newcommand{\Tilt}{\mathsf{Tilt}}
\newcommand{\Db}{D^{\mathrm{b}}}
\newcommand{\Kb}{K^{\mathrm{b}}}
\newcommand{\Par}{\mathsf{Parity}}

\newcommand{\Gv}{G^{\vee}}
\newcommand{\Tv}{T{^\vee}}
\newcommand{\Bv}{B^{\vee}}
\newcommand{\Gw}{G^{\wedge}}
\newcommand{\Tw}{T{^\wedge}}
\newcommand{\Bw}{B^{\wedge}}
\newcommand{\Iv}{I^{\vee}}
\newcommand{\Iw}{I^{\wedge}}

\newcommand{\GvK}{\Gv(\mathscr{K})}
\newcommand{\GvO}{\Gv(\mathscr{O})}
\newcommand{\GwK}{\Gw(\mathscr{K})}
\newcommand{\GwO}{\Gw(\mathscr{O})}

\newcommand{\fg}{\mathfrak{g}}
\newcommand{\fh}{\mathfrak{h}}

\newcommand{\GKM}{\mathscr{G}}
\newcommand{\BKM}{\mathscr{B}}
\newcommand{\PKM}{\mathscr{P}}
\newcommand{\TKM}{\mathscr{T}}
\newcommand{\UKM}{\mathscr{U}}
\newcommand{\LKM}{\mathscr{L}}

\newcommand{\bX}{\mathbf{X}}
\newcommand{\hdot}{{\hspace{3.5pt} \raisebox{2.5pt}{\text{\circle*{2.5}}}} \hspace{0.5pt}}
\newcommand{\hhdot}{\, \raisebox{2.5pt}{\text{\circle*{2.5}}}}

\newcommand{\la}{\lambda}

\newcommand{\Flag}{\mathscr{X}}
\newcommand{\Gr}{{\EuScript Gr}}
\newcommand{\Fl}{{\EuScript Fl}}

\newcommand{\cF}{\mathcal{F}}
\newcommand{\cE}{\mathcal{E}}
\newcommand{\cG}{\mathcal{G}}

\newcommand{\cH}{\mathcal{H}}

\newcommand{\cL}{\mathcal{L}}

\newcommand{\cA}{\mathcal{A}}

\newcommand{\simto}{\xrightarrow{\sim}}

\newcommand{\End}{\mathrm{End}}

\makeatletter
\def\lotimes{\@ifnextchar_{\@lotimessub}{\@lotimesnosub}}
\def\@lotimessub_#1{\mathchoice{\mathbin{\mathop{\otimes}^L}_{#1}}%
  {\otimes^L_{#1}}{\otimes^L_{#1}}{\otimes^L_{#1}}}
\def\@lotimesnosub{\mathbin{\mathop{\otimes}^L}}
\makeatother

\newcommand{\cD}{\mathcal{D}}

\newcommand{\onto}{\twoheadrightarrow}
\newcommand{\into}{\hookrightarrow}

\newcommand{\id}{\mathrm{id}}

\newcommand{\cC}{\mathcal{C}}

\newcommand{\bbH}{\mathbb{H}}

\DeclareMathOperator{\Hom}{Hom}
\DeclareMathOperator{\Ext}{Ext}
\DeclareMathOperator{\Ind}{Ind}

\DeclareMathOperator{\cok}{cok}

\newcommand{\triright}{\xrightarrow{[1]}}
\newcommand{\excise}[1]{}

\newcommand{\Xbb}{\mathbb{X}}
\newcommand{\Tbb}{\mathbb{T}}
\newcommand{\Trans}{\mathsf{T}}

\newcommand{\JW}{{}^J \hspace{-1pt} W}
\newcommand{\Whit}{\mathrm{Whit}}
\newcommand{\Av}{\mathsf{Av}}

\newcommand{\quo}{/ \hspace{-2.5pt} / \hspace{1pt}}

\newtheorem*{conj*}{Conjecture}
\newtheorem*{thm*}{Theorem}
\newtheorem*{cor*}{Corollary}
\numberwithin{equation}{section}
\newtheorem{thm}{Theorem}[section]
\newtheorem{lem}[thm]{Lemma}
\newtheorem{prop}[thm]{Proposition}
\newtheorem{cor}[thm]{Corollary}
\newtheorem{conj}[thm]{Conjecture}
\theoremstyle{definition}
\newtheorem{defn}[thm]{Definition}

\theoremstyle{remark}
\newtheorem{rmk}[thm]{Remark}

\newtheorem*{rmk*}{Remark}

\DeclareMathOperator{\Rep}{Rep}
\DeclareMathOperator{\Tr}{Tr}
\newcommand{\un}{\underline}
\def\nat{\mathrm{nat}}

\renewcommand{\a}{\alpha}

\newcommand{\hg}{\mathfrak{h}}
\newcommand{\hgf}{\mathfrak{h}_{\mathrm{f}}}
\newcommand{\om}{\omega}
\newcommand{\Repp}{\mathsf{R}}

\newcommand{\Sim}{\mathbb{L}}  
\newcommand{\Sta}{\Delta}  
\newcommand{\Cos}{\nabla}  
\newcommand{\Til}{\mathbb{T}}  

\newcommand{\Wf}{W_{\mathrm{f}}}
\newcommand{\Waff}{W}
\newcommand{\fW}{{}^{\mathrm{f}} W}
\newcommand{\Sf}{S_{\mathrm{f}}}
\newcommand{\Saff}{S}
\newcommand{\asph}{\mathrm{asph}}
\newcommand{\uw}{\underline{w}}
\newcommand{\uv}{\underline{v}}
\newcommand{\uu}{\underline{u}}
\newcommand{\ux}{\underline{x}}
\newcommand{\uy}{\underline{y}}
\newcommand{\uz}{\underline{z}}
\newcommand{\ue}{\underline{e}}
\newcommand{\uf}{\underline{f}}
\newcommand{\Haff}{\mathcal{H}}
\newcommand{\Hf}{\mathcal{H}_{\mathrm{f}}}
\newcommand{\Masph}{\mathcal{M}^{\asph}}
\newcommand{\Msph}{\mathcal{M}^{\mathrm{sph}}}
\newcommand{\puH}{{}^p \hspace{-1pt} \underline{H}}
\newcommand{\puN}{{}^p \hspace{-1pt} \underline{N}}
\newcommand{\puM}{{}^p \hspace{-1pt} \underline{M}}
\newcommand{\ppn}{{}^p \hspace{-1pt} n}

\newcommand{\ppm}{{}^p \hspace{-1pt} m}
\newcommand{\uM}{\underline{M}}
\newcommand{\uH}{\underline{H}}
\newcommand{\uN}{\underline{N}}
\newcommand{\sgn}{\mathsf{sgn}}

\newcommand{\BSvar}{\mathsf{BS}}

\newcommand{\Diag}{\mathcal{D}}
\newcommand{\DiagBS}{\mathcal{D}_{\mathrm{BS}}}
\newcommand{\DiagBSZ}{\mathcal{D}_{\mathrm{BS},\Z'}}

\newcommand{\DiagBSk}{\mathcal{D}_{\mathrm{BS},\Bbbk}}
\newcommand{\Dasph}{\mathcal{D}^{\asph}}
\newcommand{\DasphBS}{\mathcal{D}^{\asph}_{\mathrm{BS}}}

\newcommand{\DasphBSk}{\mathcal{D}^{\asph}_{\mathrm{BS},\Bbbk}}
\newcommand{\oB}{\overline{B}}

\newcommand{\LL}{L}

\newcommand{\adj}{\mathrm{adj}}
\newcommand{\rel}{\,\text{---}\,}
\newcommand{\norel}{\hspace{4pt}\slash{\hspace{-7pt}}\text{---} \,}

\newcommand{\hgl}{\widehat{\mathfrak{gl}}}
\newcommand{\gl}{\mathfrak{gl}}
\newcommand{\e}{\varepsilon}
\def\g{\gamma}

\newcommand{\ext}{\mathrm{ext}}

\newcommand{\sfK}{\mathsf{K}}

\title{Tilting modules and the $p$-canonical basis}
\thanks{S.R. was partially supported by ANR Grant No.~ANR-13-BS01-0001-01. This project has received
funding from the European Research Council (ERC) under the European Union's Horizon 2020
research and innovation programme (grant agreement No 677147).}

\author{Simon Riche}
\address{Universit\'e Clermont Auvergne, CNRS, LMBP, F-63000 Clermont-Ferrand, France.}
\email{simon.riche@uca.fr}

\author{Geordie Williamson}
\address{Max-Planck-Institut f\"ur Mathematik, Vivatsgasse 7, 53111,
  Bonn, Germany.
}
\email{geordie@mpim-bonn.mpg.de}

\begin{document}

\begin{abstract}
In this paper we propose a new approach to tilting
  modules for reductive algebraic groups in positive
  characteristic. We conjecture that translation functors give an
  action of the (diagrammatic) Hecke category of the affine Weyl group
  on the principal block. Our conjecture
  implies character formulas for the simple and tilting modules
 in terms of the $p$-canonical basis, as well as a description of the
  principal block as the anti-spherical quotient of the Hecke
  category. We prove our conjecture for
  $\mathrm{GL}_n(\bk)$ using the theory of $2$-Kac--Moody actions.
  Finally, we prove that the diagrammatic Hecke category of a general
  crystallographic Coxeter group may be described in terms of parity
  complexes on the flag variety of the corresponding Kac--Moody
  group.
\end{abstract}

\maketitle

\tableofcontents

\ifdefined\PARTCOMPILE{

\section{Introduction}

\subsection{Overview}
\label{ss:intro-overview}

In this paper we give new 
conjectural character formulas for simple and
indecomposable tilting modules for a connected reductive algebraic group in
characteristic $p$,\footnote{This paper is written under the
  assumption that $p$ is (strictly) larger than the Coxeter number, and so $p$ cannot be too small. However
  there is a variant of our conjectures for any $p$ involving
  singular variants of the Hecke category. In particular, it seems likely that
  $p$-Kazhdan--Lusztig polynomials give correct character formulas for
  tilting modules in \emph{any} characteristic (see Conjecture
  \ref{conj:conj-tiltings-small-p}).  We hope to return to this
  subject in a future work.} and we prove our 
conjectures in the case of the group $\textrm{GL}_n(\bk)$ when $n \geq 3$.
 These conjectures are formulated in terms of the
$p$-canonical basis of the corresponding affine Hecke algebra.
They should be regarded as evidence
for the philosophy that Kazhdan--Lusztig polynomials should be
replaced by \emph{$p$-Kazhdan--Lusztig polynomials} in modular representation
theory. 
From this point of view
several conjectures (Lusztig's conjecture,
James' conjecture, Andersen's conjecture) become the
question of agreement between canonical (or Kazhdan--Lusztig) and
$p$-canonical bases.

In the general setting
we prove that the new character formulas follow from a
very natural conjecture of a more categorical nature, which has remarkable structural
consequences for the representation theory of reductive algebraic
groups. It is a classical observation that wall-crossing functors
provide an action of the affine Weyl group $W$ on the Grothendieck group of the
principal block; in this way, the principal block gives a
categorification of the antispherical module for $W$.
We conjecture that this action can be categorified:
namely, that the action of wall-crossing functors on the principal
block gives rise to an action of the diagrammatic Bott--Samelson Hecke category attached to $W$ as in~\cite{ew}. From this conjecture we deduce the following properties.
\begin{enumerate}
\item 
\label{it:intro-consequences-1}
The principal block is equivalent (as a module category over the
  diagrammatic Bott--Samelson Hecke category) to a categorification of the anti-spherical module
  defined by diagrammatic generators and relations.
\item 
\label{it:intro-consequences-2}
The principal block admits a grading. Moreover, this graded
  category arises via extension of scalars from a category defined
  over the integers. Thus the principal block of any reductive
  algebraic group admits a ``graded integral form''.
\item 
\label{it:intro-consequences-3}
The (graded) characters of the simple and tilting modules are determined
  by the $p$-canonical basis in the anti-spherical module for the
  Hecke algebra of $W$.
\end{enumerate}
From~\eqref{it:intro-consequences-1} one may describe the principal
block in terms of parity sheaves on the affine flag variety, which
raises the possibility of calculating simple and tilting characters
topologically. Point~\eqref{it:intro-consequences-2} gives a strong form of ``independence of 
$p$''. Finally, point~\eqref{it:intro-consequences-3} implies Lusztig's character formula for large $p$.

We prove this ``categorical'' conjecture (hence in particular the character formulas) for the groups $\textrm{GL}_n(\bk)$ using the Khovanov--Lauda--Rouquier theory of
$2$-Kac--Moody algebra actions. We view this, together with the agreement
with Lusztig's conjecture for large $p$ and character formulas of
Soergel and Lusztig in the context of quantum groups (see~\S\ref{ss:variants} for details) as strong evidence for our conjecture.


\subsection{The ``categorical'' conjecture}
\label{ss:intro-cat-conj}

Let $G$ be a connected reductive algebraic group over an algebraically closed
field $\Bbbk$ of characteristic $p$ with simply connected derived subgroup. We assume that $p > h$, where $h$ is the
Coxeter number of $G$. Let $T \subset B \subset G$ be a maximal
torus and a Borel subgroup in $G$.
Denote by $\bX:=X^*(T)$ the lattice of characters of $T$ and by $\bX^+
\subset \bX$ the subset of dominant weights. Consider a regular block
$\Rep_0(G)$ of the category of finite-dimensional algebraic $G$-modules,
corresponding to a weight $\lambda_0 \in \bX$ in the fundamental
alcove, with its natural highest weight structure. Then if $\Phi$ is the root system of $(G,T)$, $\Wf=N_G(T)/T$
is the corresponding Weyl group, and $\Waff:=\Wf \ltimes \Z \Phi$ is
the affine Weyl group, then the simple, standard, costandard and indecomposable
tilting objects in the highest weight category $\Rep_0(G)$ are all parametrized by $\Waff \hdot
\lambda_0 \cap \bX^+$. (Here ``$\hdot$" denotes ``$p$-dilated dot action
of $\Waff$ of $\bX$'', see \S\ref{ss:definitions-G}.) 
 If $\fW \subset \Waff$ is the subset of elements $w$ which are
 minimal in their coset $\Wf w$, then there is a natural bijection
\[
\fW \simto \Waff \hdot \lambda_0 \cap \bX^+
\]
sending $w$ to $w \hdot \lambda_0$. In this way we can parametrize
the  simple, standard, costandard and indecomposable tilting objects in $\Rep_0(G)$
by $\fW$, and denote them by $\Sim(x \hdot \lambda_0)$, $\Sta(x \hdot \lambda_0)$, $\Cos(x \hdot \lambda_0)$ and
$\Til(x \hdot \lambda_0)$ respectively. In particular, on the level of Grothendieck groups we have
\begin{align}
\label{anti1}
[\Rep_0(G)] = \bigoplus_{x \in \fW} \Z [\Cos(x\hdot \lambda_0)].
\end{align}

Let $\Saff \subset \Waff$ denote the simple reflections. To any
$s \in S$ one can associate a ``wall-crossing"
functor $\Xi_s$ by translating to and from an $s$-wall
of the fundamental alcove. 
Consider the ``anti-spherical'' right $\Z [\Waff]$-module $\Z_\varepsilon \otimes_{\Z [\Wf]} \Z [\Waff]$, where $\Z_\varepsilon$ denotes the sign module for the finite Weyl
group $\Wf$ (viewed as a right $\Z[\Wf]$-module). This module has a basis (as a $\Z$-module) consisting of the elements $1 \otimes w$ with $w \in \fW$. Then
we can reformulate~\eqref{anti1} as an isomorphism
\begin{equation}
\label{eqn:anti2}
\phi \colon
 \Z_\varepsilon \otimes_{\Z [\Wf]} \Z [\Waff] \simto [\Rep_0(G)] 
\end{equation}
defined by $\phi(1 \otimes w)=[\Cos(w \hdot \lambda_0)]$ for $w \in \fW$.
The advantage of this
formulation over~\eqref{anti1} is that $\phi$ becomes an isomorphism of right
$\Z[ \Waff]$-modules if we let $1 + s$ act on the right by $[\Xi_s]$ for
any $s \in S$ (as follows easily from standard translation functors
combinatorics).

Let $\Haff$ be the Hecke algebra of $(\Waff,\Saff)$ over $\Z[v,v^{-1}]$, with standard basis $\{H_w : w \in \Waff\}$ and Kazhdan--Lusztig basis $\{\underline{H}_w :
w \in \Waff\}$. (Our normalisation is as in~\cite{soergel-comb-tilting}.)
Let $\Diag$ denote the diagrammatic Hecke category over $\Bbbk$: this is an
additive monoidal $\Bbbk$-linear category introduced by B.~Elias and the second author in~\cite{ew}. This category is defined by diagrammatic generators
and relations. If $\Bbbk$ is a field of characteristic zero, then
$\Diag$ is equivalent to the category of Soergel bimodules for $\Waff$~\cite{soergel-bim}. In the modular case (and for affine Weyl groups as here) Soergel bimodules are not expected to be well behaved; but $\Diag$ provides a convenient replacement for these objects. (Some piece of evidence for this idea is provided by the fact that $\Diag$ can also be described in topological terms using parity sheaves on an affine flag variety; see~\S\ref{ss:Hecke-cat-parity} below.)

The fundamental properties of $\Diag$, which generalize well-known properties of Soergel bimodules, are the following (see~\cite{ew}):
\begin{enumerate}
\item $\Diag$ is graded with shift functor $\langle 1
  \rangle$;
\item $\Diag$ is idempotent complete and Krull--Schmidt;
\item $\Diag$ is generated as a graded monoidal category by some objects $B_s$ for $s \in S$;
\item 
\label{it:Diag-Hecke-alg}
we have a canonical isomorphism of $\Z[v,v^{-1}]$-algebras
\begin{equation}
\label{eqn:groth-D}
\Haff \simto [\Diag] \colon \underline{H}_s \mapsto [B_s]
\end{equation}
where $[\Diag]$ denotes
the split Grothendieck group of $\Diag$ (a $\Z[v,v^{-1}]$-module via $v \cdot
[M] := [M\langle 1 \rangle]$);
\item 
for each $w \in W$ there exists an object $B_w \in \Diag$
   (well-defined up to isomorphism) such that the map $(w, n) \mapsto B_w
   \langle n \rangle$ gives an identification between $W \times \Z$ and
   the isomorphism classes of indecomposable objects in $\Diag$.
\end{enumerate}
In particular,~\eqref{it:Diag-Hecke-alg} tells us that $\Diag$ provides a categorification
of $\Haff$. Let us note also that $\Diag$ is defined as the Karoubi envelope of the additive hull of a category $\DiagBS$ (which we will call the diagrammatic Bott--Samelson Hecke category) which can be obtained by base change to $\bk$ from a category $\DiagBSZ$ defined over a ring $\Z'$ which is either $\Z$ or $\Z[\frac{1}{2}]$. (In other words, $\DiagBS$ has a natural integral form.)

The following conjecture says that the above shadow of translation
functors on the Grothendieck group can be upgraded to a categorical action.

\begin{conj}[Rough version] \label{conj:roughfirst}
The assignment of $B_s$ to a wall-crossing functor $\Xi_s$ for all $s
\in S$ induces a right action of the diagrammatic Bott--Samelson Hecke category $\DiagBS$ on $\Rep_0(G)$.
\end{conj}

The reason that this formulation of the conjecture is only a rough statement is
that we also need to define the images of certain generating morphisms in
$\DiagBS$. We require that most of these morphisms arise from fixed choices of
adjunctions between translation functors. See Conjecture~\ref{conj:main} for a precise statement, and Remark~\ref{rmk:main-conj} for comments.

\subsection{Tilting modules and the anti-spherical module}
\label{ss:intro-Dasph}

We now explain the main results of the first part of this
paper. In short, we prove that Conjecture~\ref{conj:roughfirst} leads to an explicit
description of the regular block $\Rep_0(G)$ in terms of the diagrammatic Hecke category $\Diag$, 
which provides a character formula for tilting modules and allows us to construct
graded and integral forms of $\Rep_0(G)$.

Let $\Haff$ be the Hecke algebra of $\Waff$ as above, and let
$\Hf \subset \Haff$ be the Hecke algebra of $\Wf$ (i.e.~the
subalgebra with basis $\{H_w : w \in \Wf\}$). Let 
\[
\Masph:=\sgn \otimes_{\Hf} \Haff 
\]
be the corresponding antispherical (right) $\Haff$-module. (Here
$\sgn=\Z[v,v^{-1}]$, with $H_s$ acting via multiplication by $-v$ for
simple reflections $s$ in $\Wf$; the module $\Masph$ is denoted
$\mathcal{N}^0$ in~\cite{soergel-comb-tilting}.) This module has a
standard basis $\{ N_w : w \in \fW \}$ and a
Kazhdan--Lusztig basis $\{ \underline{N}_w : w \in \fW \}$; by definition we have $N_w = 1 \otimes H_w$, and it is not difficult to check that $\underline{N}_w = 1 \otimes \underline{H}_w$. The anti-spherical $\Haff$-module $\Masph$ is a
quantization of the anti-spherical $\Z[\Waff]$-module considered in~\S\ref{ss:intro-cat-conj} in the
sense that the specialization $v \mapsto 1$ gives a canonical isomorphism
\[
\Z \otimes_{\Z[v,v^{-1}]} \Masph \cong \Z_\varepsilon \otimes_{\Z [\Wf]} \Z [\Waff].
\]


Now define 
$\Dasph$ as the quotient of $\Diag$ by the morphisms which factor through
a sum of indecomposable objects of the form $B_w \langle n \rangle$
with $w \notin \fW$. Then $\Dasph$ is naturally a right $\Diag$-module
category and we have a canonical isomorphism of right $\Haff$-modules
\begin{equation}
\label{eqn:groth-Dasph}
\Masph \simto [\Dasph].
\end{equation}
In particular, $\Dasph$ provides a categorification of $\Masph$.

The categorified anti-spherical module $\Dasph$ has the following properties:
\begin{enumerate}
\item $\Dasph$ is a graded
  category (with shift functor $\langle 1 \rangle$) and its indecomposable objects (up to isomorphism) are the
  images $\oB_w \langle n \rangle$ of the
  objects $B_w\langle n \rangle$ for $w \in \fW$ and $n \in \Z$;
\item $\Dasph$ may be obtained as the Karoubi envelope of the additive hull of a right
  $\DiagBS$-module category obtained by base change from a category defined over an explicit localization $\mathfrak{R}$ of $\Z'$.
\end{enumerate}
Let us denote by $\Dasph_{\deg}$ the ``degrading" of $\Dasph$, i.e.~the category with the same
  objects as $\Dasph$ but with morphisms given by
\[
\Hom_{\Dasph_{\deg}}(M,N) := \bigoplus_{n \in \Z}
\Hom_{\Dasph}(M,N\langle n \rangle).
\]
The map on Grothendieck groups induced by the obvious
functor $\Dasph \to \Dasph_{\deg}$ is the specialization $v \mapsto 1$.

Let $\Tilt(\Rep_0(G))$ denote the additive category of tilting
modules in $\Rep_0(G)$. Because translation functors preserve tilting
modules, the wall-crossing functors preserve $\Tilt(\Rep_0(G))$; hence if
Conjecture~\ref{conj:roughfirst} holds (or in fact its more precise formulation in Conjecture~\ref{conj:main}) then the diagrammatic Bott--Samelson Hecke category $\DiagBS$
acts on $\Tilt(\Rep_0(G))$.

Our first main result is the following.

\begin{thm}
\label{thm:intro-main}
  Suppose that Conjecture~{\rm \ref{conj:main}} holds. Then we have an
  equivalence of additive right $\DiagBS$-module categories
\[
\Dasph_{\deg} \simto \Tilt(\Rep_0(G))
\]
which sends $\oB_w$ to $[\Til(w \hdot \lambda_0)]$ for any $w \in \fW$.
\end{thm}

In rough terms, this result says that once we have the action of
$\DiagBS$, we automatically know that $\Tilt(\Rep_0(G))$ ``is'' the
(categorified) antispherical module. Or in other words, we can upgrade
the isomorphism between Grothendieck groups deduced from~\eqref{eqn:anti2} and~\eqref{eqn:groth-Dasph} to an equivalence of categories.

In view of the above properties of $\Dasph_{\deg}$, the following
corollary is immediate. (See Theorem~\ref{thm:integral} for a more precise version.)

\begin{cor}
Suppose that $p> N \geq h$ and that Conjecture~{\rm \ref{conj:main}} holds in any characteristic $p'>N$. Then $\Tilt(\Rep_0(G))$
admits $\Dasph$ as a graded enhancement. Moreover, this
graded enhancement may be obtained as the Karoubi envelope of the additive hull of a graded $\DiagBS$-module category coming via base change from a
category defined over $\Z[\frac{1}{N!}]$.
\end{cor}

Because tilting modules have no higher extensions and generate
$\Rep_0(G)$ the inclusion functor gives an equivalence of triangulated
categories:\footnote{Indeed, it is from this property that tilting modules derive
their name.}
\[
\Kb(\Tilt(\Rep_0(G))) \simto \Db(\Rep_0(G)).
\]
The results above tell us that $\Tilt(\Rep_0(G))$ has a
graded $\Z[\frac{1}{N!}]$-form, and hence so does $\Kb(\Tilt(\Rep_0(G)))$. 
We show (see Theorems~ \ref{thm:rep0grading} and~\ref{thm:repintform} below) that the t-structure defining $\Rep_0(G)
\subset \Db(\Rep_0(G))$ may be lifted to the graded integral form of
$\Kb(\Tilt(\Rep_0(G)))$. Hence (under the assumption of our
conjecture in any characteristic $p>N$) there exists a graded abelian category from which
$\Rep_0(G)$ is obtained by extension of scalars and ``degrading'' for any field
$\Bbbk$ of characteristic $p>N$. This gives a strong form of the ``independence of $p$'' property of
\cite{ajs}.

\begin{rmk}
  In \cite{libedinsky-williamson} Libedinsky and the second author
  study the analogue of the $\Dasph$ for any Coxeter system and any
  subset of simple reflections. For such categories,
  structural results analogous to those of this paper are
  proved. (Indeed, some of the statements were motivated by the
  current work.) The
  techniques are different and do not rely on geometry or
  representation theory. In particular the authors deduce the
  positivity of Deodhar's signed Kazhdan-Lusztig polynomials.
\end{rmk}

\subsection{Tilting characters}
\label{ss:intro-tilting-characters}

The preceding theorem giving a description of the category
$\Tilt(\Rep_0(G))$ has a ``combinatorial shadow'',  namely a character formula
for tilting modules. First, recall that the \emph{$p$-canonical basis} $\{\puH_w :
w \in \Waff\}$ of $\Haff$ is defined as the inverse image
under~\eqref{eqn:groth-D} of the basis of $[\Diag]$ consisting of the
classes of the objects $\{B_w : w \in \Waff\}$. 
This basis has many
favorable properties similar to the usual Kazhdan--Lusztig basis
$\{\uH_w : w \in \Waff\}$ of $\Haff$, see~\cite{jw}. In particular: 
\begin{itemize}
\item
for each $w \in W$, we have $\puH_w = \uH_w$ if $p \gg 0$;
\item
the elements $\puH_w$
can be computed algorithmically, though this algorithm is much more
complicated than the algorithm for computing the usual
Kazhdan--Lusztig basis.
\end{itemize}
(Note that our category $\Diag$ does not coincide with the category used to define the $p$-canonical basis in~\cite{jw}; however the two choices lead to the same basis of $\Haff$; see Remark~\ref{rmk:p-can-basis-realization} for details.)

Similarly, the $p$-canonical basis $\{\puN_w : w \in \fW\}$ of $\Masph$ can be defined as the inverse image under~\eqref{eqn:groth-Dasph} of the basis of $[\Dasph]$ consisting of the classes of the objects $\overline{B}_w$. By construction, for any $w \in \fW$ we have
\[
\puN_w = 1 \otimes \puH_w.
\]
In particular, this basis is easy to compute if we know the $p$-canonical basis of $\Haff$.
We can then define the \emph{antispherical $p$-Kazhdan--Lusztig polynomials} $\ppn_{w,y}$ (for $w,y \in \fW$) via the formula
\[
\puN_w = \sum_{y \in \fW} \ppn_{y,w} N_y.
\]

As an immediate corollary of Theorem~\ref{thm:intro-main} we obtain the following result.

\begin{cor}
\label{cor:characters-tiltings}
Assume that Conjecture~{\rm \ref{conj:main}} holds.
Then for any $w \in \fW$ the isomorphism $\phi$ of~\eqref{eqn:anti2} satifies
\begin{equation}
\label{eqn:char-formula-tilt}
\phi^{-1}([\Til(w \hdot \lambda_0)]) = 1 \otimes \puN_w.
\end{equation}
In other words, for any $w,y \in \fW$ we have
\begin{equation}
\label{eqn:char-formula-tilt-2}
(\Til(w \hdot \lambda_0) : \Cos(y \hdot \lambda_0)) = \ppn_{y,w}(1).
\end{equation}
\end{cor}

Note that the characters of the costandard modules are known, and given by the Weyl character formula. Hence the formula~\eqref{eqn:char-formula-tilt-2} for multiplicities provides a character formula for tilting modules; this is 
the conjectural character formula for tilting modules referred to in~\S\ref{ss:intro-overview}.
Of course, there already exists a conjecture, due to Andersen, for the multiplicities
of indecomposable tilting modules whose weights lie in the lowest
$p^2$-alcove; see~\cite{andersen2}. Recall that $\{\underline{N}_w : w \in \fW\}$ is the Kazhdan--Lusztig
basis of $\Masph$. Then Andersen's conjecture can be expressed as follows:
\begin{equation}
\label{eqn:andersen-conj}
\begin{array}{c}
\text{if $w \in \fW$ and $\forall \alpha \in \Phi^+, \, \langle w \hdot \lambda_0 + \rho, \alpha^\vee \rangle < p^2$,}\\
\text{ then }
\phi^{-1}([\Til(w \hdot \lambda_0)]) = 1 \otimes \underline{N}_w.
\end{array}
\end{equation}
(Here $\Phi^+ \subset \Phi$ denotes the set of positive roots, and $\rho$ is the halfsum of positive roots.)
It is known that if $p \geq 2h-2$ and if this conjecture is true over $\bk$, then Lusztig's conjecture~\cite{lusztig}
holds over $\bk$; see~\S\ref{ss:intro-simple-characters} below for details. Hence the counterexamples to the expected bound in Lusztig's conjecture found by the second author~\cite{williamson} show that this
conjecture does not hold, except perhaps if $p$ is very large.

There are two important differences between the multiplicity
formula~\eqref{eqn:char-formula-tilt} and Andersen's
conjecture~\eqref{eqn:andersen-conj}. The first one is that the
Kazhdan--Lusztig basis has been replaced by the $p$-canonical
basis. The second one is that our formula applies to \emph{all}
tilting modules in the principal block, and not only to those in the
lowest $p^2$-alcove.

\begin{rmk}
\begin{enumerate}
\item
To see a concrete example where~\eqref{eqn:char-formula-tilt-2} differs from Andersen's formula, one can consider the case $G=\mathrm{SL}(2)$ and $p=3$. If $s$ and $s_0$ are the unique elements in $\Sf$ and $\Saff \smallsetminus \Sf$ respectively, one can easily check that $(\Til(s_0 s s_0 s \bullet 0) : \Cos(s_0 s \bullet 0)) = 1$. On the other hand, the coefficient of $\uN_{s_0 s s_0 s}$ on $N_{s_0 s}$ is $0$, while ${}^3 \hspace{-1pt}n_{s_0 s, s_0 s s_0 s}=1$ (see~\cite[\S 5.3]{jw}). Of course, $s_0 s s_0 s \bullet 0$ does not satisfy the condition in~\eqref{eqn:andersen-conj}.
\item
As noted several times already,
Andersen's conjecture is only concerned with weights in the lowest
$p^2$-alcove. There exists a ``tensor product theorem'' for tilting
modules, but with assumptions which are different from the assumptions
of Steinberg's tensor product theorem for simples modules. In
particular, whereas one can obtain the characters of \emph{all} simple
modules if one knows the characters of restricted simple modules, one
\emph{cannot} deduce from Andersen's formula character formulas for
all indecomposable tilting modules in $\Rep_0(G)$, not even when $p$
is large; see~\cite{lw} for an investigation of which characters can
be obtained in this way.
\item
 For
  fixed $w \in \fW$, since $\puH_w = \uH_w$ for large $p$, comparing our definition of $\puN_w$ with~\cite[Proposition~3.4]{soergel-comb-tilting} we see that also $\puN_w = \underline{N}_w$ for large $p$. However,
  this observation alone does not imply Andersen's conjecture for
  large $p$ as the set
\[
\{ w \in \fW \mid w \hdot \lambda_0 \text{ lies in the lowest
  $p^2$-alcove} \}
\]
grows with $p$. Hence as $p$ grows we would need the equality
$\puN_w = \underline{N}_w$ for an increasing number of elements $w$. Therefore, even
for large $p$, Andersen's conjecture would imply a highly
non-trivial stability property of the $p$-canonical
basis. (Note that in the setting of affine Weyl groups it is known
that there is no $p$ such that $\puH_w = \underline{H}_w$ for all $w \in W$.)
\item
Conjecture~\ref{conj:roughfirst} provides a very direct way to obtain
the multiplicity formula~\eqref{eqn:char-formula-tilt}. However this
conjecture might be difficult to prove in general (see in particular Remark~\ref{rmk:intro-other-types} below), and there might be other, more indirect, ways to prove this formula.
In fact, a proof will appear in~\cite{amrw} (a joint work with P.~Achar and S.~Makisumi), building on the results of Part~\ref{pt:parities} and those obtained by P.~Achar and the first author in~\cite{ar}.
\end{enumerate}
\end{rmk}

Corollary~\ref{cor:characters-tiltings} (contingent on our
Conjecture~\ref{conj:roughfirst}) has a reformulation which makes
sense for any $p$ and which we would like to state as a conjecture. So let
us temporarily allow $p$ to be arbitrary. Fix a weight $\lambda_0
\in \bX$ in the closure of the fundamental alcove. (Note that $\lambda_0$ will
be neither regular nor dominant in general.) Let $I \subset \Saff$ be such
that $\Waff_I$ is the stabiliser of $\lambda_0$ under the action
$\lambda_0 \mapsto w \hdot \lambda_0$ of $\Waff$. We have a bijection
\[
\fW^I \simto \Waff \hdot \lambda_0 \cap \bX^+
\]
where $\fW^I$ denotes the subset of elements $w \in \Waff$ which both belong
to $\fW$ and are maximal in their coset $w \Waff_I$. The bijection is
given by $w \mapsto w \hdot \lambda_0$.

\begin{conj}
\label{conj:conj-tiltings-small-p}
For any $w,y \in \fW^I$ we have
\begin{equation}
\label{eqn:conj-tiltings-small-p}
(\Til(w \hdot \lambda_0) : \Cos(y \hdot \lambda_0)) = \ppn_{y,w}(1).
\end{equation}
\end{conj}

\begin{rmk}
In the case when regular weights exist (i.e.~when $p \geq h$), the general case of Conjecture~\ref{conj:conj-tiltings-small-p} follows from the case when $\lambda_0$ is regular (in which case it coincides with~\eqref{eqn:char-formula-tilt-2}) using~\cite[Proposition~5.2]{andersen-sum}.
\end{rmk}

\subsection{The case of the group $\mathrm{GL}_n(\bk)$}

In the second part of the paper we restrict to the case $G=\mathrm{GL}_n(\bk)$ (in which case $h=n$), and $\lambda_0$ is the weight $(n, \cdots, n)$ (under the standard identification $\bX = \Z^n$).

The main result of this part, proved in~\S\ref{ss:strategy}, is the following.

\begin{thm}
\label{thm:intro-GLn}
Let $n \in \Z_{\geq 3}$.
If $G=\mathrm{GL}_n(\bk)$ and $\lambda_0=(n, \cdots, n)$, then Conjecture~{\rm \ref{conj:main}} holds.
\end{thm}

Our proof of this result uses in a crucial way the Khovanov--Lauda--Rouquier $2$-category $\mathcal{U}(\mathfrak{h})$ associated with a Kac--Moody Lie algebra $\mathfrak{h}$. In a sense, to check the relations between wall-crossing functors we decompose them in terms of ``simpler'' functors appearing in the KLR $2$-category, and use the known relations between these ``simpler functors'' to prove the desired relations. More precisely, our proof of Theorem~\ref{thm:intro-GLn} consists of 3 steps:
\begin{enumerate}
\item
\label{it:proof-GLn-1}
define an action of $\mathcal{U}(\hgl_p)$ on the category $\Rep(G)$ of finite dimensional algebraic $G$-modules;
\item
\label{it:proof-GLn-2}
show that one can restrict this action to an action of $\mathcal{U}(\hgl_n)$;
\item
\label{it:proof-GLn-3}
show that the categorical action of $\mathcal{U}(\hgl_n)$ induces an action of $\DiagBS$ on an appropriate weight space equal to $\Rep_0(G)$.
\end{enumerate}
Each of these steps is known: \eqref{it:proof-GLn-1} is due to Chuang--Rouquier~\cite{CR}, \eqref{it:proof-GLn-2} follows from a result of Maksimau~\cite{m}, and~\eqref{it:proof-GLn-3} follows from results of Mackaay--Sto{\v s}i{\'c}--Vaz and Mackaay--Thiel, see~\cite{msv, mt1, mt2}. However, as these authors do not all use the same conventions, and for the benefit of readers not used to the KLR formalism (like the authors), we give a detailed account of each step. (Our proof of step~\eqref{it:proof-GLn-3} is slightly different from the proof in~\cite{msv, mt1, mt2}, which does not use our step~\eqref{it:proof-GLn-2}.) We also make an essential use of a recent result of Brundan~\cite{Brundan} proving that the $2$-categories defined by Khovanov--Lauda and by Rouquier are equivalent.

In particular, this theorem implies that the character formula~\eqref{eqn:char-formula-tilt} is proved for the group $G=\mathrm{GL}_n(\bk)$ when $p>n \geq 3$. This case is especially interesting because, due to work of Donkin~\cite{donkin}, Erdmann~\cite{erdmann} and Mathieu~\cite{mathieu}, the knowledge of characters of indecomposable tilting modules for all groups $\mathrm{GL}_n(\bk)$ provides (in theory) a dimension formula for the irreducible representations of all the symmetric groups ${S}_m$ over $\Bbbk$. 
To use this idea in practice we would need a character formula valid for \emph{all} characteristics, not only when $p>n$. But our formula might still have interesting applications in this direction (which we have not investigated yet).

\begin{rmk}
\label{rmk:intro-other-types}
\begin{enumerate}
\item
\label{it:rmk-n2}
Theorem~\ref{thm:intro-GLn} is also trivially true in case $n=1$. The case $n=2$ can be proved by methods similar to those we use in Part~\ref{pt:GLn}. However, since the extended Dynkin diagram looks differently in this case, some arguments have to be modified, and for simplicity we decided to exclude this case. (In fact, in this case there are only ``one color relations'', so that one only needs to adapt the considerations of~\S\ref{ss:one-color}.) In any case, the tilting modules for the group $\mathrm{GL}_2(\bk)$ are essentially the same as for $\mathrm{SL}_2(\bk)$, and in the latter case they are well understood thanks to work of Donkin~\cite{donkin}.
\item
  From the above discussion it should be clear that the extra
  structure provided by the categorical Kac--Moody action on $\Rep (\mathrm{GL}_n(\bk))$ is absolutely
  central to our proof in this case. There has been recent work by
  several authors on possible replacements for Kac--Moody actions in
  other classical types (see in particular~\cite{bsww}); however it is unclear whether these works will
  open the way to a proof of our conjecture in these cases similar to our proof for $\mathrm{GL}_n(\bk)$. 
\end{enumerate}
\end{rmk}

\subsection{The diagrammatic Hecke category and parity sheaves}
\label{ss:Hecke-cat-parity}

As explained above,
in this paper a central role is played by the diagrammatic Hecke category
$\Diag$, viewed as a monoidal category
defined by certain diagrammatic generators and relations. 
In the third part of the paper we show that $\Diag$ is equivalent to the
``geometric" (or ``topological") Hecke category, i.e.~the additive category of Iwahori-equivariant parity
sheaves on the Langlands dual affine flag variety. This provides an
alternative and more intrinsic description of $\Diag$, and also allows
us (under the assumption that Conjecture~\ref{conj:main} holds) to relate $\Rep_0(G)$
to parity sheaves.

In fact, in view of other expected applications, we consider more generally the diagrammatic Hecke category $\cD^\K(\GKM)$ associated with the Weyl group $W$ of a Kac--Moody group $\GKM$, with coefficients in a Noetherian (commutative) complete local ring $\K$ (assuming that $2$ is invertible in $\K$ in some cases, see~\S\ref{ss:Diag-GKM}). We denote by $\BKM \subset \GKM$ the Borel subgroup, and by $\Flag:=\GKM/\BKM$ the associated flag variety. Then we can consider the category $\Par_{\BKM}(\Flag, \K)$ of $\BKM$-equivariant parity complexes on $\Flag$ with coefficients in $\K$ (in the sense of~\cite{jmw}). This category is a full subcategory in the $\BKM$-equivariant derived category $\Db_{\BKM}(\Flag, \K)$ which is closed under the convolution product $\star^\BKM$ and is graded, with shift functor $[1]$.

The main result of the third part of the paper is the following.

\begin{thm}
\label{thm:intro-parities}
There exists an equivalence of monoidal graded additive categories
\[
\Diag^\K(\GKM) \simto \Par_{\BKM}(\Flag,\K).
\]
\end{thm}

Coming back to the setting of~\S\ref{ss:intro-cat-conj}, 
we can define $\Gw$ as the simply connected cover of the derived subgroup of the complex reductive group which is Langlands dual to $G$, and consider an Iwahori subgroup $\Iw \subset \Gw \bigl( \C[ \hspace{-1.5pt} [z] \hspace{-1.5pt}] \bigr)$ and the associated affine flag variety $\Fl^\wedge := \Gw \bigl( \C( \hspace{-1.5pt} (z) \hspace{-1.5pt}) \bigr)/\Iw$. Then we have a monoidal category $\Par_{\Iw}(\Fl^\wedge,\bk)$ of parity complexes on $\Fl^\wedge$ with coefficients in $\bk$. The same methods as for Theorem~\ref{thm:intro-parities} provide a ``topological'' description of the category $\Diag$ of~\S\ref{ss:intro-cat-conj}, in the form of an equivalence of monoidal additive graded categories
\[
\Diag \simto \Par_{\Iw}(\Fl^\wedge,\bk).
\]

One can also deduce a topological description of $\Dasph$ in terms of Iwahori--Whittaker (\'etale) sheaves on $\Fl^\wedge$ (or rather a version of $\Fl^\wedge$ over an algebraically closed field of positive characteristic different from $p$); see Theorem~\ref{thm:Dasph-parities} for details.

\begin{rmk}
As explained above, if Conjecture~\ref{conj:main} holds, combining Theorem~\ref{thm:intro-main} with the above ``topological'' description of $\Dasph$ one obtains a description of $\Rep_0(G)$ in terms of constructible sheaves on $\Fl^\wedge$. This relation is \emph{not} the same as the relation conjectured by Finkelberg--Mirkovi{\'c} in~\cite{fm}; in fact it is Koszul dual (see~\cite{ar} for details).
\end{rmk}




\subsection{Variants}
\label{ss:variants}

Although we will not treat this in detail, methods similar to those of this paper apply in the following contexts:
\begin{enumerate}
\item
\label{it:cat-O}
regular block of category $\mathcal{O}$ of a reductive complex Lie algebra;
\item
\label{it:qgroups}
regular block of the category of finite-dimensional representations of
Lusz\-tig's quantum groups at a root of unity;
\item
\label{it:mod-cat-O}
Soergel's modular category $\mathcal{O}$.
\end{enumerate}
(In cases~\eqref{it:cat-O} and~\eqref{it:mod-cat-O} one needs to
replace the affine Weyl group $\Waff$ by $\Wf$, the antispherical module $\Masph$ by
the regular right module $\Hf$, and the affine flag variety by the
finite flag variety; in cases~\eqref{it:cat-O} and~\eqref{it:qgroups}
one needs to replace $\Bbbk$ by the appropriate field of
characteristic $0$.)

In each case one can formulate an analogue of our
Conjecture~\ref{conj:main} and show that this conjecture implies a
description of the category of tilting objects as
in~Theorem~\ref{thm:intro-main}, and a character formula for these tilting objects as in Corollary~\ref{cor:characters-tiltings}.
In case~\eqref{it:cat-O} one can show (using a Radon transform as in~\cite{bbm}, or an algebraic analogue) that this formula is equivalent to the Kazhdan--Lusztig conjecture proved by Be{\u\i}linson--Bernstein and Brylinski--Kashiwara. In case~\eqref{it:qgroups} this formula is Soergel's conjecture~\cite{soergel-comb-tilting} proved by Soergel~\cite{soergel-char-tilting}, which implies in particular Lusztig's conjecture on characters of simple representations of quantum groups at a root of unity~\cite{lusztig-q}, see~\S\ref{ss:comparison-lusztig} below. And in case~\eqref{it:mod-cat-O} these statements are equivalent to the main results of~\cite{soergel}.

In cases~\eqref{it:cat-O} and~\eqref{it:qgroups}, the appropriate
variant of our results for the group $G=\mathrm{GL}_n(\bk)$ also
apply. (In case~\eqref{it:cat-O} one uses the action of
$\mathcal{U}(\mathfrak{gl}_\infty)$ on the category $\mathcal{O}$
constructed in~\cite{CR}. In case~\eqref{it:qgroups} one uses an
action of $\mathcal{\hgl_\ell}$ where $\ell$ is the order of the root
of unity under consideration, assuming that $\ell > n$ and $\ell$ is odd; the
existence of such an action is suggested in~\cite[Remark~7.27]{CR}.)
Combining these approaches with the main result of~\cite{ew1}, in this way one can
obtain direct algebraic proofs of the Kazhdan--Lusztig conjecture and
of Lusztig's and Soergel's quantum conjectures in type $\mathbf{A}$, which bypass geometry
completely (and even use few results from Representation Theory).

In case~\eqref{it:mod-cat-O} the similar approach does not apply
(because we cannot embed the category in a larger category as in the
other cases). However, since we have checked the appropriate relations
on $\Rep(G)$, and since the translation functors in Soergel's modular
category $\mathcal{O}$ are obtained from the translation functors on
$\Rep(G)$ by restricting to a subcategory and then taking the induced
functors on the quotient, from the relations on $\Rep(G)$ one can
deduce the relations on the modular category $\mathcal{O}$. In this
way one gets an alternative proof of the main result
of~\cite{soergel}, which does not rely on any result from~\cite{ajs}
(contrary to Soergel's proof). Note that these results are the main
ingredient in the second author's construction of counterexamples to
the expected bound in Lusztig's conjecture in~\cite{williamson}.

\begin{rmk}
  Case~\eqref{it:qgroups} above has recently been considered by Andersen--Tubben\-hauer~\cite{AT}. In particular, they have obtained by more explicit methods a diagrammatic description of the principal block in type $\mathbf{A}_1$ which is similar to the one that can be deduced from our methods.
\end{rmk}

\subsection{Simple characters}
\label{ss:intro-simple-characters}

We conclude this introduction with some comments on
another motivation for the current work, which is to establish a formula for
the simple characters in $\Rep_0(G)$ in terms of $p$-Kazhdan--Lusztig
polynomials. (See also~\cite{williamson-takagi} for more details and references.)

The question of computing the characters of simple modules has a long
history. For a fixed characteristic $p$, Steinberg's tensor product theorem
reduces this question to the calculation of the characters of the
(finitely many) simple modules with restricted highest
weight. As long as $p \ge h$, classical results of Jantzen and Andersen
reduce this question further to the calculation of simple characters
corresponding to restricted highest weights in a regular block (a
finite set which is independent of $p \ge h$). As for tilting modules (see~\S\ref{ss:intro-tilting-characters}), 
it is then natural to seek an
expression (in the Grothendieck group of such a regular block) for
the classes of simple modules in terms of costandard 
(dual Weyl) 
modules, since the characters of the latter modules are known.

In 1979, Lusztig~\cite{lusztig} gave a conjectural expression for this
decomposition in terms of affine Kazhdan--Lusztig polynomials. This
conjecture was then proved in 1994/95, combining works of
Andersen--Jantzen--Soergel~\cite{ajs}, Kashiwara--Tani\-saki~\cite{kt},
Kazhdan--Lusztig~\cite{kl} and~Lusztig~\cite{LUSMon}, but only under the assumption that $p$
is bigger than a non-explicit bound depending on the root system of $G$. More recently
Fiebig~\cite{fiebig} obtained an explicit bound above which the
conjecture holds. However this bound is difficult to compute and in any
case several orders of magnitude bigger than $h$.

On the other hand, the second author~\cite{williamson} has recently
exhibited examples showing that Lusztig's conjecture \emph{cannot}
hold for all $p \geq h$. In fact these examples
 show that there does not exist any polynomial $P \in \Z[X]$ such that
 Lusztig's conjecture holds for the group $G=\mathrm{GL}_n(\bk)$ in all
 characteristics $p > P(n)$.\footnote{More precisely, to prove this statement one also needs some non-trivial results from number theory, see the appendix to \cite{williamson}.} 
 On the other hand, in this case the
 Coxeter number $h$ is equal to $n$.

Following a strategy due to Andersen, we can use a small part of
the information on tilting characters above to deduce character
formulas for simple modules. Of course, it is enough to do this when $G$ is quasi-simple, which we will assume for the rest of this subsection. Let us assume that $p \geq 2h-2$. As above, we denote by $\rho$ the half sum of positive roots, by $\alpha_0^\vee$ the highest coroot of $G$, and we set
\[
\fW_{0} := \{w \in \fW \mid \langle w \hdot \lambda_0 + \rho, \alpha_0^\vee \rangle < p(h-1) \}.
\]
(More concretely, this means that $w \hdot \lambda_0$ belongs to an alcove which lies below the hyperplane orthogonal to $\alpha_0^\vee$ containing the ``Steinberg weight'' $(p-1)\rho$.
Clearly, this subset of $\fW$ does not depend on the choice of $\lambda_0$, nor on $p$.) 
Following Soergel~\cite{soergel-comb-tilting} we define a map
\begin{equation}
\label{eqn:hat-fW}
\fW \to \fW \colon y \mapsto \widehat{y}
\end{equation}
as follows. Let $\Sigma$ be the set of simple roots, and let
\[
\bX_1^+:= \{\lambda \in \bX^+ \mid \forall \alpha \in \Sigma, \, \langle \lambda, \alpha^\vee \rangle \leq p-1\}
\]
be the set of restricted weights.
Then, for $\lambda \in (p-1)\rho + \bX^+$, we set $\widecheck{\lambda}=(p-1)\rho + p\gamma + w_0 \eta$ where $\gamma \in \bX^+$ and $\eta \in \bX^+_1$ are characterized by the fact that $\lambda=(p-1)\rho+p\gamma+\eta$. This map induces a bijection $(p-1)\rho + \bX^+ \simto \bX^+$, and we denote its inverse by $\mu \mapsto \widehat{\mu}$. Then the map~\eqref{eqn:hat-fW} is characterized by the fact that $\widehat{y \hdot \lambda_0} = \widehat{y} \hdot \lambda_0$.

The following result is closely related to~\cite[Proposition~2.6]{andersen2}.

\begin{prop}
\label{prop:multiplicities-simples}
For any $x,y \in \fW_0$ we have
\[
[\Sta(x \hdot \lambda_0) : \Sim(y \hdot \lambda_0)] = (\Til(\widehat{y} \hdot \lambda_0) : \Cos(x \hdot \lambda_0)).
\]
\end{prop}

\begin{proof}\footnote{This proof has benefitted from discussions with H.~H.~Andersen.}
Let us consider
\[
\bX^+_{<p(h-1)}:=\{\lambda \in \bX^+ \mid \langle \lambda, \alpha_0^\vee \rangle < p(h-1) \}.
\]
(Then $\bX^+_{<p(h-1)}$ contains the weights $x \hdot \lambda_0$ for $w \in \fW$.)
Let $\Rep(G)$ be the category of finite dimensional algebraic $G$-modules,
and denote by $\Rep_{<p(h-1)}(G)$ the Serre subcategory of $\Rep(G)$ generated by the simple objects $\Sim(\lambda)$ with $\lambda \in \bX^+_{<p(h-1)}$. Since $\bX^+_{<p(h-1)}$ is an ideal in $(\bX^+,\uparrow)$, $\Rep_{<p(h-1)}(G)$ has a natural highest weight structure, see Lemma~\ref{lem:hw-quotient}\eqref{it:hw-subcat}. We define in a similar way the subcategory $\Rep_{<2p(h-1)}(G)$.

Let us denote by
\[
\Pi_{<p(h-1)} : \Rep(G) \to \Rep_{<p(h-1)}(G)
\]
the functor which sends an object to its largest subobject which belongs to the subcategory $\Rep_{<p(h-1)}(G)$. This functor is right adjoint to the inclusion functor $\Rep_{<p(h-1)}(G) \to \Rep(G)$. To prove the proposition it suffices to prove that for $y \in \fW_0$, the object $\Pi_{<p(h-1)}(\Til(\widehat{y} \hdot \lambda_0))$ is the injective hull of $\Sim(y \hdot \lambda_0)$ in $\Rep_{<p(h-1)}(G)$. Indeed, if this is known then we obtain that for $x,y \in \fW_0$ we have
\begin{multline*}
[\Sta(x \hdot \lambda_0),\Sim(y \hdot \lambda_0)] = \dim_\bk \Hom_{\Rep_{<p(h-1)}(G)}(\Sta(x \hdot \lambda_0), \Pi_{<p(h-1)}(\Til(\widehat{y} \hdot \lambda_0))) \\
= \dim_\bk \Hom_{\Rep(G)}(\Sta(x \hdot \lambda_0), \Til(\widehat{y} \hdot \lambda_0))
= (\Til(\widehat{y} \hdot \lambda_0) : \Cos(x \hdot \lambda_0)).
\end{multline*}

So, what remains is to prove that $\Pi_{<p(h-1)}(\Til(\widehat{y} \hdot \lambda_0))$ is the injective hull of $\Sim(y \hdot \lambda_0)$ in $\Rep_{<p(h-1)}(G)$. Let $\mu=y \hdot \lambda_0$, and write $\mu=\mu^0+p\mu^1$ with $\mu^0 \in \bX_1^+$ and $\mu^1 \in \bX^+$. Then $\widehat{y} \hdot \lambda_0 = \widehat{\mu^0} + p\mu^1$. Moreover, the fact that $y \in \fW_0$ implies that $p \langle \mu^1, \alpha_0^\vee \rangle \leq p(h-1)$, hence that $\langle \mu^1, \alpha_0^\vee \rangle \leq (h-1)$. Since $p\geq 2(h-1)$, this shows that 
\[
\Sta(\mu^1)=\Cos(\mu^1)=\Sim(\mu^1) = \Til(\mu^1),
\]
see~\cite[Corollary~II.5.6]{jantzen}.
Now by Donkin's tensor product theorem (see~\cite[\S E.9]{jantzen}), under our assumption that $p \geq 2h-2$ we have $\Til(\widehat{\mu^0} + p \mu^1) \cong \Til(\widehat{\mu^0}) \otimes \Til(\mu^1)^{[1]}$, where $(-)^{[1]}$ is the Frobenius twist, see~\cite[\S II.3.16]{jantzen}. We deduce that
\begin{equation}
\label{eqn:tilting-hat}
\Til(\widehat{y} \hdot \lambda_0) \cong \Til(\widehat{\mu^0}) \otimes \Sim(\mu^1)^{[1]}.
\end{equation}

We claim that
\begin{enumerate}
\item
\label{it:til-inj}
$\Til(\widehat{\mu^0})$ belongs to $\Rep_{<2p(h-1)}(G)$, and is injective therein; 
\item
\label{it:til-soc}
the socle of $\Til(\widehat{\mu^0})$ as a $G_1$-module is $\Sim(\mu^0)$.
\end{enumerate}
Indeed, by~\cite[\S 4.5--4.6]{jantzen-darstellungen}, under our assumptions the injective hull of $\Sim(\mu^0)$ as a $G_1$-module has a unique $G$-module structure, and such a $G$-module is the injective hull of $\Sim(\mu^0)$ in $\Rep_{<2p(h-1)}(G)$. By an observation of Donkin (see~\cite[Example~1 on p.~47]{donkin} or~\cite[II.E.9(1)]{jantzen}), this $G$-module is isomorphic to $\Til(\widehat{\mu^0})$, proving our claims.

It follows from Claim~\eqref{it:til-soc} above that the socle of $\Pi_{<p(h-1)}(\Til(\widehat{y} \hdot \lambda_0))$ is $\Sim(y \hdot \lambda_0)$. Indeed,
since $\Sim(\mu^0) \otimes \Sim(\mu^1)^{[1]} \cong \Sim(y \hdot \lambda_0)$ by Steinberg's tensor product theorem, using~\eqref{eqn:tilting-hat} and~\cite[Lemma~4.6]{ak} we obtain that the socle of $\Til(\widehat{y} \hdot \lambda_0)$ is $\Sim(y \hdot \lambda_0)$, which implies our claim on $\Pi_{<p(h-1)}(\Til(\widehat{y} \hdot \lambda_0))$.

Now if $\nu \in \bX^+_{<p(h-1)}$ we have
\begin{multline*}
\Ext^1_{\Rep(G)}(\Sim(\nu),\Til(\widehat{y} \hdot \lambda_0)) \stackrel{\eqref{eqn:tilting-hat}}{=} \Ext^1_{\Rep(G)}(\Sim(\nu),\Til(\widehat{\mu^0}) \otimes \Sim(\mu^1)^{[1]}) \\
\cong
\Ext^1_{\Rep(G)}(\Sim(\nu) \otimes \Sim(-w_0 \mu^1)^{[1]},\Til(\widehat{\mu^0})),
\end{multline*}
where $w_0$ is the longest element in $\Wf$. We have
\[
\langle \nu - p w_0 \mu^1, \alpha_0^\vee \rangle < p(h-1) + p(h-1) = 2p(h-1),
\]
hence any composition factor of $\Sim(\nu) \otimes \Sim(-w_0 \mu^1)^{[1]}$ belongs to $\Rep_{<2p(h-1)}(G)$. This implies that
\[
\Ext^1_{\Rep(G)}(\Sim(\nu),\Til(\widehat{y} \hdot \lambda_0))=0
\]
by Claim~\eqref{it:til-inj} above.
From this vanishing we deduce that 
$\Pi_{<p(h-1)}(\Til(\widehat{y} \hdot \lambda_0))$ is injective in $\Rep_{<p(h-1)}(G)$, which finishes the proof.
\end{proof}

\begin{rmk}
As explained in~\cite[\S 4.6]{jantzen-darstellungen}, the statement from~\cite{jantzen-darstellungen} that we use above does not hold in general if $h \leq p < 2h-2$. See~\cite[Corollary~6.3.1]{kildetoft} for a statement similar to Proposition~\ref{prop:multiplicities-simples} which holds under slightly weaker assumptions.
\end{rmk}

From~\eqref{eqn:char-formula-tilt-2} and Proposition~\ref{prop:multiplicities-simples} it follows that if Conjecture~\ref{conj:main} holds then
\begin{equation}
\label{eqn:character-formula-simples-pre}
[\Sta(x \hdot \lambda_0)] = \sum_{y \in \fW_0} \ppn_{x, \widehat{y}}(1) \cdot [\Sim(y \hdot \lambda_0)]
\end{equation}
in $[\Rep_0(G)]$. 
Inverting the matrix $(\ppn_{x, \widehat{y}}(1))_{x,y \in \fW^0}$ (which is upper-triangular with $1$'s on the diagonal for an appropriate order) we can express the classes of the modules $\Sim(y \hdot \lambda_0)$ for $y \in \fW_0$ in terms of standard modules, hence deduce their characters. Since these modules include the simple modules $\Sim(y \hdot \lambda_0)$ with $y \hdot \lambda_0 \in \bX_1^+$, this solves the problem of computing the characters of simple modules in $\Rep(G)$ (under the assumptions that $p \geq 2(h-1)$ and that Conjecture~\ref{conj:main} holds, or at least that~\eqref{eqn:char-formula-tilt-2} holds).

\begin{rmk}
\begin{enumerate}
\item
As explained in~\cite[Remark~2.7(ii)]{andersen2} and~\cite[\S E.10]{jantzen}, there is another way to obtain a character formula for $\Sim(y)$ out of a tilting character formula, this time in terms of baby Verma modules for the Lie algebra of $G$ (again under the assumption that $p \geq 2h-2$). 
In the context of Lusztig's conjecture, the relation between the two formulas one can obtain in this way is explained in~\cite[\S 3.3]{fiebig-moment-graph}.
\item
In the case of the group $\mathrm{GL}_n(\bk)$, there is still another way to obtain a simple character formula out of a tilting character formula using a ``reciprocity formula'' due to Donkin; see~\cite[E.10(4)]{jantzen}.
\end{enumerate}
\end{rmk}


\subsection{Comparison with Lusztig's conjecture}
\label{ss:comparison-lusztig}

To explain the relation between the formula for characters of simple modules obtained in~\S\ref{ss:intro-simple-characters} and
Lusztig's conjecture~\cite{lusztig}, we first consider the
characteristic $0$
situation, which corresponds to representations of Lusztig's quantum
groups at a $p$-th root of unity (see~\S\ref{ss:variants}). (The relation between character formulas for tilting and simple modules in this case is investigated by Soergel in~\cite[p.106--107]{soergel-comb-tilting}, and we essentially copy Soergel's discussion.) We assume that $p \geq h$ is odd (but not necessarily a prime number), and denote by $\Sta_q(\lambda)$, $\Sim_q(\lambda)$, etc.~the analogues of $\Sta(\lambda)$, $\Sim(\lambda)$, etc.~in this context.

Consider the ordinary Kazhdan--Lusztig basis $\{\uN_w : w\in \fW\}$ of $\Masph$ and the corresponding Kazhdan--Lusztig polynomials $\{n_{y,w} : y,w \in \fW\}$. Then, due to work of Andersen, we have
\[
[\Sta_q(x \hdot \lambda_0) : \Sim_q(y \hdot \lambda_0)] = (\Til_q(\widehat{y} \hdot \lambda_0) : \Cos_q(x \hdot \lambda_0))
\]
for \emph{all} $x,y \in \fW$. This implies in particular that the right-hand side vanishes unless $y \hdot \lambda_0 \preceq x \hdot \lambda_0$ for the standard order on dominant weights, hence that from Soergel's character formula for the tilting modules $\Til_q(w \hdot \lambda_0)$ one deduces that
\begin{equation}
\label{eqn:character-formula-simples-quantum}
[\Sta_q(x \hdot \lambda_0)] = \sum_{y \in \fW} n_{x, \widehat{y}}(1) \cdot [\Sim_q(y \hdot \lambda_0)],
\end{equation}
where now the sum runs over all $y \in \fW$. Introducing the
inverse Kazhdan--Lusztig polynomials $\{n^{y,w} : y,w \in \fW\}$ as in~\cite{soergel-comb-tilting}, from~\eqref{eqn:character-formula-simples-quantum} one can deduce the formula
\begin{equation*}
[\Sim_q(y)] = \sum_{x \in \fW} (-1)^{\ell(x) + \ell(\widehat{y})} \cdot n^{x,\widehat{y}}(1) \cdot [\Sta_q(x)].
\end{equation*}

Consider now the \emph{spherical} (right) $\Haff$-module
\[
\Msph:=\mathsf{triv} \otimes_{\Hf} \Haff,
\]
where $\mathsf{triv}$ is the ``trivial'' $\Hf$-module, i.e.~the free rank one $\Z[v,v^{-1}]$-module on which $H_s$ acts by multiplication by $v^{-1}$. Let $\{\uM_w : w\in \fW\}$ be the corresponding Kazhdan--Lusztig basis and $\{m_{y,w} : y,w \in \fW\}$ be the corresponding Kazhdan--Lusztig polynomials, see~\cite{soergel-comb-tilting}. Then~\cite[Theorem~5.1]{soergel-comb-tilting} implies that
\begin{equation}
\label{eqn:soergel-thm-klp}
n^{x,\widehat{y}}(1) = (-1)^{\ell(y) + \ell(\widehat{y})} m_{x,y}(1).
\end{equation}
Hence, in this setting, from~\eqref{eqn:character-formula-simples-quantum} one obtains the character formula 
\[
[\Sim_q(y)] = \sum_{x \in \fW} (-1)^{\ell(x) + \ell(y)} \cdot m_{x,y}(1) \cdot [\Sta_q(x)],
\]
which is exactly the formula conjectured by Lusztig and proved by the combination of works of Kazhdan--Lusztig~\cite{kl}, Lusztig~\cite{LUSMon} and Kashiwara--Tanisaki~\cite{kt}.

If one wants to generalize this analysis to the modular setting, the first difficulty is that we do not know whether one can replace $\fW_0$ by $\fW$ in the sum in~\eqref{eqn:character-formula-simples-pre}. The second difficulty is that we do not have an analogue of~\cite[Theorem~5.1]{soergel-comb-tilting} for $p$-canonical bases. (Note that one can define the $p$-canonical basis of $\Msph$ as follows. Consider the injective morphism $\zeta$ of~\cite[Proof of Proposition~3.4]{soergel-comb-tilting}. Then using~\cite[Lemma~4.3]{jw} it is not difficult to check that, for $w \in \fW$, the element $\puH_{w_0 w}$ belongs to the image of $\zeta$; then we can define $\puM_w$ by the property that $\zeta(\puM_w) = \puH_{w_0 w}$. One can also define the ``dual'' basis of $\Hom_{\Z[v,v^{-1}]}(\Msph, \Z[v,v^{-1}])$ as in~\cite{soergel-comb-tilting}, so that one can at least make sense of the $p$-analogues of all the ingredients in~\cite[Theorem~5.1]{soergel-comb-tilting}.)

But in any case one shouldn't expect to express the classes $[\Sim(y \hdot \lambda)]$ in terms of the classes $[\Sta(x \hdot \lambda_0)]$ using the polynomials $\ppm_{x,y}$. In fact, recall that
the polynomials $m_{x,y}$ encode the dimensions of cohomology groups of stalks of Iwahori-constructible simple $\mathbb{Q}$-perverse sheaves on the affine Grassmannian
\[
\Gr=\Gw \bigl( \C( \hspace{-1.5pt} (t) \hspace{-1.5pt} ) \bigr) / \Gw \bigl( \C[ \hspace{-1.5pt} [t] \hspace{-1.5pt} ] \bigr),
\]
where $\Gw$ is as in~\S\ref{ss:Hecke-cat-parity}.
With this in mind, Lusztig's character formula can be interpreted as the ``combinatorial shadow'' of the quantum analogue of the Finkelberg--Mirkovi{\'c} conjecture~\cite{fm} which relates $\Rep_0(G)$ to perverse sheaves on $\Gr$. (This quantum version was proved by Arkhipov--Bezukavnikov--Ginzburg~\cite{abg}.)

In the modular case,
it can be easily deduced from the results presented in~\S\ref{ss:Hecke-cat-parity} that
the polynomials $\ppm_{x,y}$ encode the dimensions of cohomology
groups of stalks of Iwahori-constructible \emph{parity sheaves} on $\Gr$. On
the other hand, according to the Finkelberg--Mirkovi{\'c} conjecture~\cite{fm},
the principal block $\Rep_0(G)$ should be equivalent, as a highest
weight category, to the category of Iwahori-constructible
$\Bbbk$-perverse sheaves on $\Gr$. If this conjecture is true, then
one would be able to express the classes of simple modules $\Sim(x \hdot \lambda_0)$ in terms of standard modules $\Sta(y \hdot \lambda_0)$ with
coefficients
given (up to sign) by the Euler characteristic of the stalk at
the point corresponding to $y$ of the simple perverse sheaf
corresponding to $x$. Now parity sheaves and simple perverse sheaves
on $\Gr$ 
do not coincide, so the coefficient considered above will \emph{not} be equal to $\ppm_{x,y}(1)$.

\subsection{Acknowledgements}

We thank Henning Haahr Andersen, Ben Elias, Thor\-ge Jensen, Nicolas Libedinsky, Daniel Tubbenhauer and Ben Webster
for useful discussions on the subject of this paper. We would also
like to thank Jon Brundan for not only admitting the existence of
\cite{Brundan} but also sending a preliminary version, and
Henning Haahr Andersen, Jens Carsten Jantzen and Weiqiang Wang for helpful remarks on preliminary versions of this paper.
Finally we thank Shrawan Kumar, Olivier Mathieu and Guy Rousseau for useful correspondence on the construction of Kac--Moody groups and their flag varieties.

After we started working on this project, we were informed by Ivan Losev that Ben Elias and he have obtained results similar to ours in type $\mathbf{A}$, and a generalization in a different direction, see~\cite{el}. Their approach is different, and uses deeper results in the theory of categorical actions of Kac--Moody Lie algebras (in particular, ``unicity of categorification" results). As far as we understand, it does not make sense (at least at this point) for a general reductive group. We thank Ivan Losev for keeping us informed of their work.

The crucial idea that a presentation of the category of Soergel
bimodules should help defining actions and equivalences of categories,
in particular in the context of modular representations of reductive
groups and quantum groups, comes from discussions of the second author with
Rapha{\"e}l Rou\-quier. This paper also owes a debt to the
ideas of Roman Bezrukavnikov and his collaborators, in particular to \cite{ab} (where categorifications of the antispherical module also play a crucial role).

\subsection{Organization of the paper}

The paper is divided into 3 parts. In Part~\ref{pt:general-conj} we study our main conjecture on the categorical action of $\DiagBS$ on $\Rep_0(G)$ (for a general reductive group $G$ as above) and study its consequences presented in~\S\ref{ss:intro-Dasph}. In Part~\ref{pt:GLn} we restrict to the special case $G=\mathrm{GL}_n(\bk)$, and prove our main conjecture in this case using the Khovanov--Lauda--Rouquier theory of categorical Lie algebra actions. Finally, in Part~\ref{pt:parities} we prove an equivalence of monoidal categories between the diagrammatic Hecke category attached to a Kac--Moody group and the category of parity sheaves on the corresponding flag variety, and deduce a topological description of the categories $\Diag$ and $\Dasph$.

We begin each part with a more detailed overview of its contents.

 }
\else { \newpage }
\fi

\part{General conjecture}
\label{pt:general-conj}

\ifdefined\PARTCOMPILE{\textbf{Overview}.
This part is concerned with the study of our main conjecture and its
consequences, in the context of a general reductive group with simply
connected derived subgroup. In Section~\ref{sec:tilting-hw} we recall
the basic notions related to tilting objects in highest weight
categories and introduce the concept of a ``section of the $\Cos$-flag", which will play a key technical role in our approach. In Section~\ref{sec:blocks} we concentrate on the case of regular blocks of representations of reductive groups; our main result is the technical Proposition~\ref{prop:morphisms-BS} which describe how morphisms between ``Bott--Samelson type" tilting modules can be generated inductively. (The proof of this result involves the study of sections of the $\Cos$-flag.) In Section~\ref{sec:Diag} we recall the construction of the diagrammatic Hecke category, and explain the construction of the categorical antispherical module. Finally, in Section~\ref{sec:main-conj} we state precisely our ``categorical'' conjecture, and study its main applications.

\section{Tilting objects and sections of the $\Cos$-flag}
\label{sec:tilting-hw}

\subsection{Highest weight categories}
\label{ss:hwcat}

Let $\Bbbk$ be a field, and let $\cA$ be a finite length $\Bbbk$-linear abelian category. We let $\Lambda$ be a set which parametrizes isomorphism classes of simple objects in $\cA$, and for any $\lambda \in \Lambda$ we fix a representative $\Sim(\lambda)$ of the corresponding isomorphism class of simple objects. 
We assume that $\Lambda$ is equipped with a partial order $\preceq$. In this setting, an \emph{ideal} of $\Lambda$ is a subset $\Omega \subset \Lambda$ such that for $\lambda,\mu \in \Lambda$,
\[
(\mu \in \Omega \ \ \& \ \ \lambda \preceq \mu) \ \ \Rightarrow \ \ \lambda \in \Omega.
\]
A \emph{coideal} is a subset $\Omega \subset \Lambda$ such that $\Lambda \smallsetminus \Omega$ is an ideal.

For $\Omega \subset \Lambda$, we denote by $\cA_{\Omega}$ the Serre subcategory of $\cA$ generated by the objects $\Sim(\la)$ for $\la \in \Omega$. We write $\cA_{\preceq \lambda}$ for $\cA_{\{\mu \in \Lambda \mid  \mu \preceq \lambda\}}$, and similarly for $\cA_{\prec \lambda}$.
Finally, we assume that we are given, for any $\lambda \in \Lambda$, objects $\Sta(\lambda)$ and $\Cos(\lambda)$, and nonzero morphisms $\Sta(\lambda) \to \Sim(\lambda)$ and $\Sim(\lambda) \to \Cos(\lambda)$.

The following definition is due to Cline--Parshall--Scott~\cite{cps}, although our version of it is closer to the some ideas developed in~\cite{bgs}.

\begin{defn}
\label{defn:hwcat}
The category $\cA$ (with the above data) is said to be a \emph{highest weight category} if the following conditions hold.
\begin{enumerate}
\item 
\label{it:hw-def-fin}
For any $\la \in \Lambda$, the set $\{\mu \in \Lambda \mid \mu \preceq \lambda\}$ is finite.
\item 
\label{it:hw-def-split}
For each $\lambda \in \Lambda$, we have 
$\Hom(\Sim(\la),\Sim(\la)) =
\Bbbk$.
\item 
\label{it:hw-def-proj-inj}
For any ideal $\Omega \subset \Lambda$ such that $\lambda \in \Omega$ is maximal, $\Sta(\la) \to \Sim(\la)$ is a projective cover in $\cA_{\Omega}$ and $\Sim(\lambda) \to \Cos(\lambda)$ is an injective envelope in $\cA_{\Omega}$.
\item 
\label{it:hw-def-ker}
The kernel of $\Sta(\la) \to \Sim(\la)$ and the cokernel of $\Sim(\la) \to \Cos(\la)$ belong to $\cA_{\prec \la}$.
\item
\label{it:hw-def-ext2}
We have $\Ext^2_\cA(\Sta(\la), \Cos(\mu)) = 0$ for all $\lambda, \mu \in \Lambda$.
\end{enumerate}
The poset $(\Lambda, \preceq)$ is called the \emph{weight poset} of this highest weight category.
\end{defn}

In the rest of this section we fix a category $\cA$ which satisfies Definition~\ref{defn:hwcat}.
The objects $\Sta(\lambda)$ are called \emph{standard objects}, and the objects $\Cos(\lambda)$ are called \emph{costandard objects}. These objects satisfy
\begin{equation}
\label{eqn:morph-Sta-Cos}
\Ext^n_\cA(\Sta(\la), \Cos(\mu)) = \begin{cases}
\Bbbk & \text{if $\la=\mu$ and $n=0$;} \\
0 & \text{otherwise.}
\end{cases}
\end{equation}
When $\la=\mu$ and $n=0$, the only non zero morphism $\Sta(\la) \to \Cos(\la)$ (up to scalar) is the composition $\Sta(\la) \to \Sim(\la) \to \Cos(\la)$.

\begin{defn}
Let $X$ be an object of $\cA$.  We say that $X$ \emph{admits a standard} (resp.~\emph{costandard}) \emph{filtration} if there exists a finite filtration $F_\bullet X$ of $X$ such that each $F_i X / F_{i-1} X$ is isomorphic to some $\Sta(\lambda)$ (resp.~$\Cos(\la)$) with $\la \in \Lambda$.
\end{defn}

If $X$ admits a standard, resp.~costandard, filtration, we denote by $(X : \Sta(\la))$, resp.~$(X : \Cos(\la))$, the number of times $\Sta(\la)$, resp.~$\Cos(\la)$, appears in a standard, resp.~costandard, filtration of $X$. It is well known that this integer does not depend on the choice of filtration; in fact by~\eqref{eqn:morph-Sta-Cos} it equals $\dim_{\Bbbk} \Hom_\cA(X, \Cos(\la))$, resp.~$\dim_{\Bbbk} \Hom_\cA(\Sta(\la), X)$.

The following well known lemma is proved e.g.~in~\cite[Lemma~2.2]{modrap3}. (More precisely,~\cite{modrap3} considers \emph{graded} highest weight categories, but the proof in the ungraded setting is identical.)

\begin{lem}
\label{lem:hw-quotient}
\begin{enumerate}
\item
\label{it:hw-subcat}
Let $\Omega \subset \Lambda$ be an ideal.
The subcategory $\cA_{\Omega} \subset \cA$ is a highest weight category with weight poset $(\Omega, \preceq)$ and standard (resp.~costandard) objects $\Sta(\omega)$ (resp.~$\Cos(\omega)$) for $\omega \in \Omega$. 
\item
\label{it:hw-quotient}
Let $\Omega \subset \Lambda$ be a coideal.
The Serre quotient $\cA^{\Omega}:=\cA / \cA_{\Lambda \smallsetminus \Omega}$ is a highest weight category with weight poset $(\Omega, \preceq)$. The standard (resp.~costandard) objects are the images in the quotient of the objects $\Sta(\omega)$ (resp.~$\Cos(\omega)$) for $\omega \in \Omega$.
\end{enumerate}
\end{lem}

In the setting of Lemma~\ref{lem:hw-quotient}\eqref{it:hw-quotient}, we will usually omit the quotient functor $\cA \to \cA^\Omega$ from the notation. This statement has the following consequence.

\begin{cor}
\label{cor:morph-quotient}
Let $\Omega \subset \Lambda$ be a coideal. If $M \in \cA$ admits a standard filtration and $N \in \cA$ admits a costandard filtration, the morphism
\[
\Hom_{\cA}(M,N) \to \Hom_{\cA^\Omega}(M,N)
\]
is surjective.
\end{cor}

\begin{proof}
Using the four-lemma, it is enough to prove the claim when $M=\Sta(\la)$ and $N=\Cos(\mu)$ for some $\lambda, \mu \in \Lambda$. In this case, it follows from Lemma~\ref{lem:hw-quotient}\eqref{it:hw-quotient} and~\eqref{eqn:morph-Sta-Cos}.
\end{proof}

\subsection{Canonical $\Cos$-flags}
\label{ss:canonical-cos-flag}

In the rest of the paper we will usually prefer costandard filtrations over standard filtrations. However, very similar constructions can be considered for standard filtrations; we leave the necessary modifications to the reader.

Let $X$ be an object of $\cA$ which admits a costandard filtration. A \emph{canonical $\Cos$-flag} of $X$ is the data, for any ideal $\Omega \subset \Lambda$, of a subobject $\Gamma_\Omega X \subset X$, such that:
\begin{itemize}
\item
$\bigcup_{\Omega} \Gamma_{\Omega} M = M$ and $\bigcap_\Omega \Gamma_{\Omega} M = 0$;
\item
$\Omega \subset \Omega' \ \Rightarrow \ \Gamma_\Omega X \subset \Gamma_{\Omega'} X$;
\item
for any ideal $\Omega$ and any $\la \in \Omega$ maximal, setting $\Omega':=\Omega \smallsetminus \{\la\}$ we have that $\Gamma_\Omega X / \Gamma_{\Omega'} X$ is isomorphic to a direct sum of copies of
$\Cos(\la)$.
\end{itemize}

\begin{lem}
\label{lem:can-nabla-flag}
Let $X$ be an object of $\cA$ which admits a costandard filtration.
A canonical $\Cos$-flag of $X$ exists and is unique.
\end{lem}

\begin{proof}
Existence follows, by induction on the length of a costandard filtration, from the property that $\Ext^1_{\cA}(\Cos(\la), \Cos(\mu))=0$ unless $\mu \prec \lambda$ (which itself follows from Property~\eqref{it:hw-def-proj-inj} in Definition~\ref{defn:hwcat}).

To prove unicity, it is enough to prove that if $\la \in \Lambda$ is minimal for the property that $(X : \Cos(\la)) \neq 0$, then there exists a unique subobject $X' \subset X$ which is isomorphic to a direct sum of copies of $\Cos(\la)$ and such that $X/X'$ admits a  costandard filtration such that $(X/X' : \Cos(\la))=0$. However we observe that
\[
\Hom_{\cA}(\Cos(\la), \Cos(\mu)) \neq 0 \quad \Rightarrow \quad \mu \preceq \la
\]
(because the image of any non-zero morphism $\Cos(\la) \to \Cos(\mu)$ must contain the socle $\Sim(\mu)$ of $\Cos(\mu)$), so that $X'$ is characterized as the sum of the images of all morphisms $\Cos(\la) \to X$.
\end{proof}

Because of Lemma~\ref{lem:can-nabla-flag}, we can consider \emph{the} canonical $\Cos$-flag of an object $X$ which admits a costandard filtration, and use the notation $\Gamma_\Omega X$ unambiguously (for $\Omega$ an ideal).

\subsection{Tilting objects and sections of the $\Cos$-flag}
\label{ss:tiltings}

Recall that
we say that an object $X$ of $\cA$ is \emph{tilting} if it admits both a standard and a costandard filtration.
We denote by $\Tilt(\cA)$ the additive full subcategory of $\cA$ whose objects are the tilting objects. It is well known that this category is Krull--Schmidt, and that its isomorphism classes of indecomposable objects are parametrized by $\Lambda$. In fact, for $\la \in \Lambda$, the corresponding indecomposable object $\Til(\la)$ is characterized (up to isomorphism) by the properties that
\[
(\Til(\la) : \Cos(\mu)) \neq 0 \ \Rightarrow \ \mu \preceq \la \qquad \text{and} \qquad (\Til(\la) : \Cos(\la)) \neq 0.
\]
We also note that $(\Til(\la) : \Sta(\mu)) \neq 0$ implies that $\mu \preceq \la$, and that $(\Til(\la) : \Cos(\la))=(\Til(\la) : \Sta(\la))=1$.

The following properties of tilting objects are well known.

\begin{lem}
\label{lem:Hom-Til-Sta-Cos}
Let $\lambda \in \Lambda$.
\begin{enumerate}
\item
We have $\Hom_{\cA}(\Sta(\la), \Til(\la)) \cong \bk$, and any nonzero morphism $\Sta(\la) \to \Til(\la)$ is injective.
\item
We have $\Hom_{\cA}(\Til(\la), \Cos(\la)) \cong \bk$, and any nonzero morphism $\Til(\la) \to \Cos(\la)$ is surjective.
\item
If $\varphi \colon \Sta(\la) \to \Til(\la)$ and $\psi \colon \Til(\la) \to \Cos(\la)$ are nonzero, then the composition $\psi \circ \varphi$ is nonzero.
\end{enumerate}
\end{lem}

For $\la \in \Lambda$ we set
\begin{equation}
\label{eqn:Asucceq}
\cA^{\succeq \la}:=\cA^{\{\mu \in \Lambda \mid \mu \succeq \la\}} = \cA/\cA_{\{\mu \in \Lambda \mid \mu \not\succeq \lambda\}}.
\end{equation}
Note that the images in $\cA^{\succeq \la}$ of the objects
\[
\Sta(\la), \quad \Cos(\la), \quad \Sim(\la), \quad \Til(\la)
\]
all coincide, and are equal to the standard object attached to $\la$ in $\cA^{\succeq \la}$.

\begin{defn}
Let $X$ be an object of $\cA$ which admits a costandard filtration. A \emph{section of the $\Cos$-flag} of $X$ is a triple $(\Pi, e, (\varphi_\pi^X)_{\pi \in \Pi})$ where:
\begin{itemize}
\item
$\Pi$ is a finite set;
\item
$e \colon \Pi \to \Lambda$ is a map;
\item
for each $\pi \in \Pi$, $\varphi^X_\pi$ is an element in $\Hom_{\cA}(\Til(e(\pi)), X)$
\end{itemize}
such that for any $\la \in \Lambda$, the images of the morphisms
\[
\{\varphi_\pi^X \colon \Til(\la) \to X : \pi \in e^{-1}(\la)\}
\]
form a basis of $\Hom_{\cA^{\succeq \la}}(\Til(\la), X) = \Hom_{\cA^{\succeq \la}}(\Cos(\la), X)$ (where, as usual, we omit the quotient functor $\cA \to \cA^{\succeq \la}$ from the notation). 
\end{defn}

Note that, by Lemma~\ref{lem:hw-quotient}\eqref{it:hw-quotient}, we have
\[
\dim_\Bbbk \Hom_{\cA^{\succeq \la}}(\Cos(\la), X) = \dim_\Bbbk \Hom_{\cA^{\succeq \la}}(\Sta(\la), X) = (X : \Cos(\la)),
\]
so that the number of maps $\varphi^X_\pi$ such that $e(\pi)=\lambda$ in a section of the $\Cos$-flag of $X$ is $(X : \Cos(\la))$. Note also that Corollary~\ref{cor:morph-quotient} guarantees that sections of the $\Cos$-flag always exist (for objects admitting a costandard filtration); they are far from unique, however.


\begin{rmk}
\label{rmk:AST}
As was pointed out to us by H.~H.~Andersen, our notion of a section of the $\Cos$-flag is not unrelated to the constructions in~\cite[\S 4.1]{ast}. In fact, our condition on the morphisms $\varphi^X_\pi$ can be equivalently stated as the requirement that for all $\la \in \Lambda$, the compositions
\[
\Sta(\la) \to \Til(\la) \xrightarrow{\varphi_\pi^X} X
\]
for $\pi \in e^{-1}(\la)$ (where the first morphism is a fixed nonzero morphism, which is unique up to scalar and injective) form a basis of $\Hom_{\cA}(\Sta(\la), X)$. From this point of view, a section of the $\Cos$-flag is the same as a choice of morphisms $\overline{g}_i^\lambda$ for all $\lambda$, using the notation of~\cite[\S 4.1]{ast}.
\end{rmk}

We conclude the generalities with two easy lemmas.

\begin{lem}
\label{lem:section-Gamma}
Let $X$ be an object which admits a costandard filtration, and let $(\Pi, e, (\varphi_\pi^X)_{\pi \in \Pi})$ be a section of the $\Cos$-flag of $X$. Let also $\Omega \subset \Lambda$ be an ideal. Then for any $\pi$ in
$\Pi_\Omega := \{ \pi \in \Pi \mid e(\pi) \in \Omega \}$,
the map $\varphi^X_\pi$ factors through a map $\varphi_\pi^{\Gamma_\Omega X} \colon \Til(e(\pi)) \to \Gamma_\Omega X$, and setting $e_\Omega := e_{| \Pi_\Omega}$, the triple
\[
(\Pi_\Omega, e_\Omega, (\varphi_\pi^{\Gamma_\Omega X})_{\pi \in \Pi_\Omega})
\]
is a section of the $\Cos$-flag of $\Gamma_\Omega X$.
\end{lem}

\begin{proof}
For any $\la \in \Omega$, the exact sequence $\Gamma_\Omega X \hookrightarrow X \twoheadrightarrow X/\Gamma_\Omega X$ induces an exact sequence
\[
\Hom_\cA(\Til(\la), \Gamma_\Omega X) \hookrightarrow \Hom_\cA(\Til(\la), X) \twoheadrightarrow \Hom_\cA(\Til(\la), X/\Gamma_\Omega X).
\]
Now $(\Til(\la) : \Sta(\mu)) = 0$ unless $\mu \in \Omega$, hence the third term in this exact sequence vanishes. Hence the first arrow in an isomorphism, which shows that, for any $\pi \in \Pi_\Omega$, $\varphi^X_\pi$ factors through a map $\varphi_\pi^{\Gamma_\Omega X} \colon \Til(e(\pi)) \to \Gamma_\Omega X$. Similar arguments show that for any $\la \in \Omega$ the natural morphism
\[
\Hom_{\cA^{\succeq \la}}(\Til(\la), \Gamma_\Omega X) \to \Hom_{\cA^{\succeq \la}}(\Til(\la), X)
\]
is an isomorphism, which implies that indeed $(\Pi_\Omega, e_{| \Pi_\Omega}, (\varphi_\pi^{\Gamma_\Omega X})_{\pi \in \Pi_\Omega})$
is a section of the $\Cos$-flag of $\Gamma_\Omega X$.
\end{proof}

The proof of the following lemma is similar to that of Lemma~\ref{lem:section-Gamma}, hence left to the reader.

\begin{lem}
\label{lem:section-Gamma-2}
Let $X$ be an object which admits a costandard filtration, and let $(\Pi, e, (\varphi_\pi^X)_{\pi \in \Pi})$ be a section of the $\Cos$-flag of $X$. Let also $\Omega \subset \Lambda$ be an ideal. For any $\pi$ in
$\Pi^\Omega := \{ \pi \in \Pi \mid e(\pi) \notin \Omega \}$, consider the composition
\[
\varphi^{X/\Gamma_\Omega X}_\pi \colon \Til(e(\pi)) \xrightarrow{\varphi^X_\pi} X \twoheadrightarrow X/\Gamma_\Omega X.
\]
Setting $e^\Omega := e_{| \Pi^\Omega}$, the triple
\[
(\Pi^\Omega, e^\Omega, (\varphi_\pi^{X/\Gamma_\Omega X})_{\pi \in \Pi^\Omega})
\]
is a section of the $\Cos$-flag of $X/ \Gamma_\Omega X$. \qed
\end{lem}

\section{Regular and subregular blocks of reductive groups}
\label{sec:blocks}

\subsection{Definitions}
\label{ss:definitions-G}

In this section we fix an algebraically closed field $\Bbbk$ of characteristic $p$ and a connected reductive algebraic group $G$ over $\Bbbk$ with simply-connected derived subgroup. We will assume that $p \geq h$, where $h$ is the Coxeter number of $G$ (i.e.~the maximum of the Coxeter numbers of all simple components of the root system of $G$, see~\cite[\S II.6.2]{jantzen}).

We fix a Borel subgroup $B \subset G$ and a maximal torus $T \subset B$, and denote by $\bX=X^*(T)$ the lattice of characters of $T$. We denote by $\Phi$ the root system of $(G,T)$, by $\Phi^\vee$ the corresponding coroots, and by $\Phi^+ \subset \Phi$ the system of positive roots consisting of the roots which do \emph{not} appear in the Lie algebra of $B$ (so that $B$ is the ``negative'' Borel subgroup). This choice determines a basis $\Sigma$ of $\Phi$, and a subset $\bX^+ \subset \bX$ of \emph{dominant weights}. 
We denote by $\rho \in \frac{1}{2}\bX$ the half sum of positive roots; then we have
$\langle \rho, \alpha^\vee\rangle = 1$ for all $\alpha \in \Sigma$, where $\alpha^\vee$ is the coroot associated with the root $\alpha$.

We denote by $\Wf:=N_G(T)/T$ the Weyl group of $(G,T)$. We will also consider the affine Weyl group $\Waff:=\Wf \ltimes \Z \Phi$.
To avoid confusions, for $\lambda \in \Z \Phi$ we denote by $t_\la$ the corresponding element of $\Waff$. The group $\Waff$ acts on $\bX$ via the ``dot-action'' defined by
\[
w t_\lambda \hdot \mu := w(\mu + p\la + \rho) -\rho.
\]
Let $\Saff \subset \Waff$ be the set of simple reflections, i.e.~the reflections whose associated hyperplane (for the dot-action) meets
\[
\overline{C}:=\{\la \in \bX \otimes_\Z \mathbb{R} \mid \forall \alpha \in \Phi^+, \ 0 \leq \langle \la + \rho, \alpha^\vee \rangle \leq p\}
\]
in a codimension $1$ facet.
Let also
$\Sf = \Saff \cap \Wf$; this set is in a natural bijection with $\Sigma$. The pairs $(\Waff, \Saff)$ and $(\Wf, \Sf)$ are Coxeter groups, and we denote by $\ell$ their length function. We will denote by $\fW$ the subset of $\Waff$ consisting of the elements $w$ which are of minimal length in their coset $\Wf w$.

We will denote by $\Rep(G)$ the abelian category of finite dimensional (algebraic) representations of $G$. It is well known that the simple objects in this category are parametrized by $\bX^+$; for any $\la \in \bX^+$ we denote by $\Sim(\la)$ a choice of a corresponding simple object.

Let us consider
\[
C_\Z := \{\la \in \bX \mid \forall \alpha \in \Phi^+, \ 0 < \langle \la + \rho, \alpha^\vee \rangle < p\}.
\]
Since $p \geq h$, we have $C_\Z \neq \varnothing$, see~\cite[\S II.6.2]{jantzen}, so that we can choose some $\lambda_0 \in C_\Z$. (A typical choice would be $\lambda_0=0$, but we allow an arbitrary choice.)
Then we set
\[
\bX_0^+ := (\Waff \hdot \la_0) \cap \bX^+.
\]
Recall that for $w \in \Waff$ we have $w \hdot \la_0 \in \bX^+$ iff $w \in \fW$, so that this set is a natural bijection with $\fW$.
We denote by
\[
\Rep_0(G)
\]
the Serre subcategory of $\Rep(G)$ generated by the simple objects $\Sim(\la)$ for $\la \in \bX^+_0$. For $\la \in \Waff \hdot \la_0$, by our choice of $\la_0$ there exists a \emph{unique} $w \in \Waff$ such that $\la = w \hdot \la_0$. Then for $s \in \Saff$ we set
\[
\la^s := ws \hdot \la_0.
\]

For any $s \in \Saff$, we fix a weight $\mu_s$ in 
\[
\overline{C}_\Z := \{\la \in \bX \mid \forall \alpha \in \Phi^+, \ 0 \leq \langle \la + \rho, \alpha^\vee \rangle \leq p\}
\]
which
lies on the reflection hyperplane of $s$ and is contained in no other
reflection hyperplane; such a weight exists under our assumptions by~\cite[\S II.6.3]{jantzen}. Then we set
\[
\bX^+_s := (\Waff \hdot \mu_s) \cap \bX^+,
\]
and denote by
\[
\Rep_s(G)
\]
the Serre subcategory of $\Rep(G)$ generated by the simple objects $\Sim(\mu)$ for $\mu \in \bX^+_s$.

For any $\la \in \bX^+$ we consider the $G$-modules
\[
\Cos(\la) := \mathrm{Ind}_B^G(\la), \qquad \Sta(\la) := (\Cos(-w_0 \la))^*,
\]
where $w_0 \in \Wf$ is the longest element. It is well known that for $\la \in \bX^+$ there exists (up to scalar) a unique non zero morphism $\Sta(\la) \to \Cos(\la)$. We fix such a morphism; its image is isomorphic to $\Sim(\la)$, so that it factors as a composition $\Sta(\la) \twoheadrightarrow \Sim(\la) \hookrightarrow \Cos(\la)$.
Recall the order $\uparrow$ on $\bX^+$ defined in~\cite[\S II.6.4]{jantzen}. (We will also write $\lambda \downarrow \mu$ for $\mu \uparrow \lambda$.)
It is well known that $\Rep(G)$, endowed with the parametrization of simple objects by the poset $(\bX^+,\uparrow)$, and with the collections of objects $\Cos(\la)$, $\Sta(\la)$, $\Sim(\la)$ and the morphisms $\Sta(\la) \to \Sim(\la) \to \Cos(\la)$ considered above, is a highest weight category. 
Since the subsets $\bX_0^+$ and $\bX_s^+$ are both ideals of the poset $(\bX^+,\uparrow)$, the subcategories
$\Rep_0(G)$ and $\Rep_s(G)$ also have canonical highest weight structures (see Lemma~\ref{lem:hw-quotient}\eqref{it:hw-subcat}), and we can consider the corresponding tilting objects $\Til(\la)$ (see~\S\ref{ss:tiltings}).

\subsection{Translation functors}
\label{ss:translation-functors}

Let us fix a simple reflection $s \in S$.
Then one can consider the \emph{translation functors}
\[
T_{\la_0}^{\mu_s} \colon \Rep_0(G) \to \Rep_s(G), \qquad T^{\la_0}_{\mu_s} \colon \Rep_s(G) \to \Rep_0(G),
\]
see~\cite[\S II.7]{jantzen}. It is more reasonable to consider that these functors are defined only up to isomorphism: they depend at least on the choice of the module by which one tensors, see in particular~\cite[Remark~II.7.6(1)]{jantzen}. So we fix some functors
\[
\Trans^s \colon \Rep_0(G) \to \Rep_s(G), \qquad \Trans_s \colon \Rep_s(G) \to \Rep_0(G)
\]
which are \emph{isomorphic} to some translation functors $T_{\la_0}^{\mu_s}$ and $T^{\la_0}_{\mu_s}$ respectively. We also fix some (arbitrary) adjunctions $(\Trans_s, \Trans^s)$ and $(\Trans^s, \Trans_s)$;
we will simply write
\[
\mathrm{id} \xrightarrow{\adj} \Trans_s \Trans^s, \quad \mathrm{id} \xrightarrow{\adj} \Trans^s \Trans_s, \quad \Trans_s \Trans^s \xrightarrow{\adj} \mathrm{id}, \quad \Trans^s \Trans_s \xrightarrow{\adj} \mathrm{id}
\]
for the corresponding adjunction morphisms.

\begin{rmk}
If the functors $T_{\la_0}^{\mu_s}$ and $T^{\la_0}_{\mu_s}$ are defined using the tensor product with a $G$-module and with its dual respectively, which is allowed, then there exists natural adjunctions $(T^{\la_0}_{\mu_s},T_{\la_0}^{\mu_s})$ and $(T_{\la_0}^{\mu_s},T^{\la_0}_{\mu_s})$; see~\cite[Lemma~II.7.6]{jantzen}. So, at least the adjunctions $(\Trans_s, \Trans^s)$ and $(\Trans^s, \Trans_s)$ exist. But we do \emph{not} assume that our chosen adjunctions are obtained in this way.
\end{rmk}

It is well known (see e.g.~\cite[Propositions~II.7.11 \&~II.7.19]{jantzen}) that the functors $\Trans_s$ and $\Trans^s$ send standard, resp.~costandard, objects to objects which admit a standard, resp.~costandard, filtration. Hence they send objects which admit a standard, resp.~costandard, filtration to objects which admit a standard, resp.~costandard, filtration. In particular, they send tilting objects to tilting objects.

Let $\mu \in \bX_s^+$, and let $\la$ be the unique weight in $\bX_0^+$ which belongs to an alcove containing $\mu$ in its closure and such that $\la \uparrow \la^s$. Then by~\cite[\S II.E.11]{jantzen} we have
\begin{equation}
\label{eqn:translation-tilting-1}
\Trans_s \Til(\mu) \cong \Til(\la^s).
\end{equation}
We fix such an isomorphism once and for all. Similarly, we have $(\Trans^s \Til(\la) : \Cos(\mu)) =1$, and $\mu$ is maximal in the collection of weights $\nu$ such that $(\Trans^s \Til(\la) : \Cos(\nu)) \neq 0$. Therefore $\Til(\mu)$ is a direct summand of $\Trans^s \Til(\la)$, with multiplicity $1$. We fix once and for all a split embedding and a split surjection
\begin{equation}
\label{eqn:translation-tilting-2}
\Til(\mu) \hookrightarrow \Trans^s \Til(\la) \twoheadrightarrow \Til(\mu).
\end{equation}

In Section~\ref{sec:main-conj} we will need the following fact (where we use the notation introduced in~\eqref{eqn:Asucceq}).

\begin{lem}
\label{lem:Hom-Dta-y-ys}
Let $y \in \fW$ and $s \in \Saff$, and assume that $ys>y$ the Bruhat order and $ys \in \fW$. Then, if $\lambda:=y \bullet \lambda_0$, then the morphism
\[
\Hom_{\Rep_0(G)}(\Sta(\la),\Trans_s \Trans^s \Sta(\la)) \to \Hom_{\Rep_0(G)^{\downarrow \la}}(\Sta(\la),\Trans_s \Trans^s \Sta(\la))
\]
induced by the quotient functor is an isomorphism, and both vector spaces are $1$-dimensional.
\end{lem}

\begin{proof}
By adjunction we have
\[
\Hom_{\Rep_0(G)}(\Sta(\la),\Trans_s \Trans^s \Sta(\la)) \cong \Hom_{\Rep_0(G)}(\Trans^s \Sta(\la),\Trans^s \Sta(\la)),
\]
and $\Trans^s \Sta(\la) \cong \Sta(\mu)$ where $\mu$ is the only weight in $\bX^+_s$ which belongs to the closure of the alcove of $\la$ (see~\cite[Proposition~II.7.11]{jantzen}). Hence the left-hand side is one-dimensional. Similar arguments, using~\cite[Proposition~II.7.15]{jantzen}, show that the morphisms
\begin{multline*}
\Hom_{\Rep_0(G)}(\Sta(\la),\Trans_s \Trans^s \Sta(\la)) \to \Hom_{\Rep_0(G)}(\Sta(\la),\Trans_s \Trans^s \Sim(\la)) \\
\to \Hom_{\Rep_0(G)}(\Sta(\la),\Trans_s \Trans^s \Cos(\la))
\end{multline*}
induced by our fixed morphisms $\Sta(\la) \to \Sim(\la) \to \Cos(\la)$ are isomorphisms.

Now we observe that the morphisms $\Trans_s \Trans^s \Sta(\la) \to \Trans_s \Trans^s \Sim(\la) \to \Trans_s \Trans^s \Cos(\la)$ are isomorphisms in $\Rep_0(G)^{\downarrow \la}$, so that we also have a canonical isomorphism
\[
\Hom_{\Rep_0(G)^{\downarrow \la}}(\Sta(\la),\Trans_s \Trans^s \Sta(\la)) \cong \Hom_{\Rep_0(G)^{\downarrow \la}}(\Sta(\la),\Trans_s \Trans^s \Cos(\la)),
\]
and what remains is to prove that the morphism
\[
\Hom_{\Rep_0(G)}(\Sta(\la),\Trans_s \Trans^s \Cos(\la)) \to \Hom_{\Rep_0(G)^{\downarrow \la}}(\Sta(\la),\Trans_s \Trans^s \Cos(\la))
\]
induced by the quotient functor is an isomorphism. This follows from the well-know fact that $\Trans_s \Trans^s \Cos(\la)$ admits a costandard filtration with $(\Trans_s \Trans^s \Cos(\la) : \Cos(\la)) = 1$ and from Lemma~\ref{lem:hw-quotient}\eqref{it:hw-quotient}.
\end{proof}

Finally we recall the following well-known fact, which follows e.g.~from~\cite[Proposition~II.7.19(a)]{jantzen}.

\begin{lem}
\label{lem:translation-Cos}
For any $\lambda \in \bX_0^+$ such that $\lambda^s \notin \bX^+$ and any $M \in \Rep_s(G)$ which admits a costandard filtration, we have $(\Trans_s M : \nabla(\la))=0$.
\end{lem}

\subsection{Sections of the $\Cos$-flag and translation to a wall}
\label{ss:sections-translation-to}

As in~\S\ref{ss:translation-functors}, let $s$ be a simple reflection.
In this subsection we explain how, given an object $M$ in $\Rep_0(G)$ which admits a costandard filtration and a section of the $\Cos$-flag of $M$, one can construct a section of the $\Cos$-flag of $\Trans^s M$. This construction is not canonical, but depends on the choices of morphisms in~\S\ref{ss:translation-functors}. The basic idea of this construction is inspired by a construction due to Libedinsky in the setting of Soergel bimodules, see~\cite{libedinsky}.

Let $M \in \Rep_0(G)$ be an object which admits a costandard filtration, and let $(\Pi, e, (\varphi^M_\pi)_{\pi \in \Pi})$ be a section of the $\Cos$-flag of $M$. We set
\[
\Pi' := \{\pi \in \Pi \mid e(\pi)^s \in \bX^+\},
\]
and for $\pi \in \Pi'$ we define $e'(\pi)$ as the unique weight in $\bX^+_s$ which lies in the closure of the alcove containing $e(\pi)$. (Our definition of $\Pi'$ ensures that such an element exists.) This defines a map $e' \colon \Pi' \to \bX_s^+$. Now we explain how to define, for any $\pi \in \Pi'$, a morphism $\varphi_{\pi}^{\Trans^s M} \colon \Til(e'(\pi)) \to \Trans^s M$.

First, let us consider the case where $e(\pi) \downarrow e(\pi)^s$. In this setting we have fixed an isomorphism 
$\Trans_s \Til(e'(\pi)) \cong \Til(e(\pi))$,
see~\eqref{eqn:translation-tilting-1}. We define $\varphi^{\Trans^s M}_{\pi}$ as the image of $\varphi_\pi^M$ under the composition
\begin{multline*}
\Hom_{\Rep_0(G)}(\Til(e(\pi)), M) \cong \Hom_{\Rep_0(G)}(\Trans_s \Til(e'(\pi)), M) \\
\cong \Hom_{\Rep_s(G)}(\Til(e'(\pi)), \Trans^s M),
\end{multline*}
where the second isomorphism is obtained from our chosen adjunction $(\Trans_s, \Trans^s)$. In other words, $\varphi^{\Trans^s M}_{\pi}$ is the composition
\[
\Til(e'(\pi)) \xrightarrow{\adj} \Trans^s \Trans_s (\Til(e'(\pi))) \cong \Trans^s \Til(e(\pi)) \xrightarrow{\Trans^s(\varphi^M_\pi)} \Trans^s M.
\]

Now, let us consider the case where $e(\pi) \uparrow e(\pi)^s$. In this case, we have fixed a split embedding $\Til(e'(\pi)) \hookrightarrow \Trans^s \Til(e(\pi))$, see~\eqref{eqn:translation-tilting-2}. We define $\varphi^{\Trans^s M}_\pi$ as the composition
\[
\Til(e'(\pi)) \hookrightarrow \Trans^s \Til(e(\pi)) \xrightarrow{\Trans^s(\varphi^M_\pi)} \Trans^s M.
\]

\begin{rmk}
\label{rmk:translation-section-to}
A point which will be important for us later is that in both cases the morphism $\varphi^{\Trans^s M}_\pi$ factors through the morphism $\Trans^s (\varphi^M_\pi) \colon \Trans^s \Til(e(\pi)) \to \Trans^s M$.
\end{rmk}

The main result of this subsection is the following.

\begin{prop}
\label{prop:translation-section-to}
The triple $(\Pi', e', (\varphi_\pi^{\Trans^s M})_{\pi \in \Pi'})$ constructed above is a section of the $\Cos$-flag of $\Trans^s M$.
\end{prop}

Before proving this result in general, we consider the special case where there exists $\la \in \bX^+_0$ such that $M$ is a direct sum of objects $\Cos(\la)$.

\begin{lem}
\label{lem:translation-section-to}
Let $\lambda \in \bX_0^+$. If
$M$ is isomorphic to a direct sum of objects $\Cos(\la)$, the triple $(\Pi', e', (\varphi_\pi^{\Trans^s M})_{\pi \in \Pi'})$ constructed above is a section of the $\Cos$-flag of $\Trans^s M$.
\end{lem}

\begin{proof}
For any $\pi \in \Pi$ we denote by $M_\pi$ the image of $\varphi_\pi^M$. Then each $M_\pi$ is isomorphic to $\Cos(\la)$, and we have $M=\bigoplus_{\pi \in \Pi} M_\pi$. Hence it is sufficient to prove the lemma in the case $M=\Cos(\la)$.

By Lemma~\ref{lem:Hom-Til-Sta-Cos}, a section of the $\Cos$-flag of $\Cos(\la)$ is unique up to scalar, so that we can assume that $\Pi=\{\la\}$, $e(\la)=\la$ and $\varphi_\la^{\Cos(\la)} \colon \Til(\la) \to \Cos(\la)$ is a non-zero (hence surjective) morphism. If $\lambda^s \notin \bX^+$ then $\Trans^s \Cos(\la)=0$, and there is nothing to prove.

Now, assume that $\lambda^s \in \bX^+$. Then $\Pi'=\{\la\}$, and we set $\mu:=e'(\la)$. In this case, by~\cite[Proposition~II.7.11]{jantzen} we have $\Trans^s \Cos(\la) \cong \Cos(\mu)$, so that to prove the claim it suffices to prove that $\varphi^{\Trans^s \Cos(\la)}_\la \neq 0$.
If $\la \downarrow \la^s$, then by construction $\varphi^{\Trans^s \Cos(\la)}_\la$ is non-zero, hence there is nothing to prove. If $\la \uparrow \la^s$, then $\Til(\mu)$ is the only direct summand of $\Trans^s \Til(\la)$ which has $\Sim(\mu)$ as a composition factor. Since the morphism $\Trans^s(\varphi^{\Cos(\la)}_\la)$ is non zero (in fact it is surjective, since $\varphi^{\Cos(\la)}_\la$ is surjective and $\Trans^s$ is exact), its image contains the socle $\Sim(\mu)$ of $\Trans^s \Cos(\la) \cong \Cos(\mu)$, so that its restriction to $\Til(\mu)$ must be non zero, proving the claim in this case also.
\end{proof}

\begin{proof}[Proof of Proposition~{\rm \ref{prop:translation-section-to}}]
We have to prove that, for any $\mu \in \bX_s^+$, the images of the morphisms $\varphi^{\Trans^s M}_\pi$ with $e'(\pi)=\mu$ form a basis of $\Hom_{\Rep_s(G)^{\downarrow \mu}}(\Til(\mu), \Trans^s M)$. If this $\Hom$-space is nonzero, then there exists a unique $\la$ in $\bX_0^+$ which belongs to an alcove containing $\mu$ in its closure and such that $\la \uparrow \la^s$ and $\la^s \in \bX^+$ (see Lemma~\ref{lem:translation-Cos}). Then $e'(\pi)=\mu$ iff $e(\pi) \in \{\la, \la^s\}$. Let also $\Omega \subset \bX_0^+$ be an ideal such that $\Omega \cap \{\la, \la^s\} = \{\la\}$, and such that both
\[
\Omega':=\Omega \smallsetminus \{\la\} \quad \text{and} \quad \Omega'':= \Omega \cup \{\la^s\}
\]
are ideals. (For instance, $\Omega = \{\nu \in \bX_0^+ \mid \nu \uparrow \la^s\} \smallsetminus \{\la^s\}$ satisfies these conditions.) Then we have inclusions
\[
\Gamma_{\Omega'} M \hookrightarrow \Gamma_\Omega M \hookrightarrow \Gamma_{\Omega''} M \hookrightarrow M,
\]
which induce inclusions
\[
\Trans^s(\Gamma_{\Omega'} M) \hookrightarrow \Trans^s(\Gamma_\Omega M) \hookrightarrow \Trans^s(\Gamma_{\Omega''} M) \hookrightarrow \Trans^s(M).
\]

By Lemma~\ref{lem:section-Gamma}, every morphism $\varphi^M_\pi$ with $e(\pi)=\la$, resp.~$e(\pi)=\la^s$, factors through a morphism $\varphi_\pi^{\Gamma_\Omega M} \colon \Til(\la) \to \Gamma_\Omega M$, resp.~$\varphi_\pi^{\Gamma_{\Omega''} M} \colon \Til(\la^s) \to \Gamma_{\Omega''} M$. Then by construction the corresponding morphisms $\varphi^{\Trans^s M}_\pi \colon \Til(\mu) \to \Trans^s M$ factor through morphisms
\[
\varphi^{\Trans^s (\Gamma_\Omega M)}_\pi \colon \Til(\mu) \to \Trans^s (\Gamma_\Omega M), \quad \text{resp.} \quad \varphi^{\Trans^s (\Gamma_{\Omega''} M)}_\pi \colon \Til(\mu) \to \Trans^s (\Gamma_{\Omega''} M).
\]
It is clear that these morphisms coincide with the morphisms obtained by the same procedure applied to the section of the $\Cos$-flag $(\Pi_\Omega, e_\Omega, (\varphi^{\Gamma_\Omega M}_\pi)_{\pi \in \Pi_\Omega})$ of $\Gamma_\Omega M$, resp.~the section of the $\Cos$-flag $(\Pi_{\Omega''}, e_{\Omega''}, (\varphi^{\Gamma_{\Omega''} M}_\pi)_{\pi \in \Pi_{\Omega''}})$ of $\Gamma_{\Omega''} M$.

Similarly (see Lemma~\ref{lem:section-Gamma-2}), if $e(\pi)=\la$, resp.~$e(\pi)=\la^s$, considering the compositions 
\[
\varphi^{\Gamma_\Omega M /\Gamma_{\Omega'} M}_\pi \colon \Til(\la) \xrightarrow{\varphi^{\Gamma_\Omega M}_\pi} \Gamma_\Omega M \twoheadrightarrow \Gamma_\Omega M /\Gamma_{\Omega'} M,
\]
resp.
\[
\varphi^{\Gamma_{\Omega''} M/\Gamma_{\Omega} M}_\pi \colon \Til(\la^s) \xrightarrow{\varphi_\pi^{\Gamma_{\Omega''} M}} \Gamma_{\Omega''} M \twoheadrightarrow \Gamma_{\Omega''} M/\Gamma_{\Omega} M,
\]
the morphisms
\[
\varphi^{\Trans^s (\Gamma_\Omega M/\Gamma_{\Omega'} M)}_\pi \colon \Til(\mu) \xrightarrow{\varphi^{\Trans^s (\Gamma_\Omega M)}_\pi} \Trans^s (\Gamma_\Omega M) \twoheadrightarrow \Trans^s (\Gamma_\Omega M/\Gamma_{\Omega'} M),
\]
resp.
\[
\varphi^{\Trans^s (\Gamma_{\Omega''} M/\Gamma_{\Omega} M)}_\pi \colon \Til(\mu) \xrightarrow{\varphi^{\Trans^s (\Gamma_{\Omega''} M)}_\pi} \Trans^s (\Gamma_{\Omega''} M) \twoheadrightarrow \Trans^s (\Gamma_{\Omega''} M/\Gamma_{\Omega} M),
\]
coincide with the morphisms obtained by the same procedure applied to the section of the $\Cos$-flag $(e^{-1}(\la), e_{|e^{-1}(\la)}, (\varphi^{\Gamma_\Omega M /\Gamma_{\Omega'} M}_\pi)_{\pi \in e^{-1}(\la)})$ of $\Gamma_\Omega M /\Gamma_{\Omega'} M$, resp.~the section of the $\Cos$-flag $(e^{-1}(\la^s), e_{|e^{-1}(\la^s)}, (\varphi^{\Gamma_{\Omega''} M/\Gamma_{\Omega} M}_\pi)_{\pi \in e^{-1}(\la^s)})$ of $\Gamma_{\Omega''} M/\Gamma_{\Omega} M$.

Finally, for $\pi$ in $e^{-1}(\la)$, resp.~$e^{-1}(\la^s)$, we will consider the compositions
\[
\varphi^{\Gamma_{\Omega''} M /\Gamma_{\Omega'} M}_\pi \colon \Til(\la) \xrightarrow{\varphi^{\Gamma_\Omega M /\Gamma_{\Omega'} M}_\pi} \Gamma_\Omega M /\Gamma_{\Omega'} M \hookrightarrow \Gamma_{\Omega''} M /\Gamma_{\Omega'} M,
\]
resp.
\[
\varphi^{\Gamma_{\Omega''} M/\Gamma_{\Omega'} M}_\pi \colon \Til(\la^s) \xrightarrow{\varphi_\pi^{\Gamma_{\Omega''} M}} \Gamma_{\Omega''} M \twoheadrightarrow \Gamma_{\Omega''} M/\Gamma_{\Omega'} M,
\]
and the corresponding morphisms
\[
\varphi^{\Trans^s (\Gamma_{\Omega''} M/\Gamma_{\Omega'} M)}_\pi \colon \Til(\mu) \to \Trans^s (\Gamma_{\Omega''} M/\Gamma_{\Omega'} M).
\]
Note that by construction for $\pi \in e^{-1}(\la^s)$, the composition
\[
\Til(\mu) \xrightarrow{\varphi^{\Trans^s (\Gamma_{\Omega''} M/\Gamma_{\Omega'} M)}_\pi} \Trans^s (\Gamma_{\Omega''} M/\Gamma_{\Omega'} M) \twoheadrightarrow \Trans^s (\Gamma_{\Omega''} M/\Gamma_{\Omega} M)
\]
coincides with $\varphi^{\Trans^s (\Gamma_{\Omega''} M/\Gamma_{\Omega} M)}_\pi$.

Consider now the exact sequence
\[
\Gamma_\Omega M / \Gamma_{\Omega'} M \hookrightarrow \Gamma_{\Omega''} M / \Gamma_{\Omega'} M \twoheadrightarrow \Gamma_{\Omega''} M / \Gamma_\Omega M,
\]
its image under $\Trans^s$. and the induced exact sequence
\begin{multline}
\label{eqn:es-Hom-Tilmu}
\Hom_{\Rep_s(G)^{\downarrow \mu}}(\Til(\mu), \Trans^s(\Gamma_\Omega M / \Gamma_{\Omega'} M)) \\
\hookrightarrow \Hom_{\Rep_s(G)^{\downarrow \mu}}(\Til(\mu), \Trans^s(\Gamma_{\Omega''} M / \Gamma_{\Omega'} M)) \\
\twoheadrightarrow \Hom_{\Rep_s(G)^{\downarrow \mu}}(\Til(\mu), \Trans^s(\Gamma_{\Omega''} M / \Gamma_\Omega M)).
\end{multline}
Since $\Gamma_\Omega M / \Gamma_{\Omega'} M$, resp.~$\Gamma_{\Omega''} M / \Gamma_\Omega M$, is a direct sum of copies of $\Cos(\la)$, resp.~of $\Cos(\la^s)$,
it follows from Lemma~\ref{lem:translation-section-to} that the images of the morphisms $\varphi^{\Trans^s(\Gamma_\Omega M /\Gamma_{\Omega'} M)}_\pi$ for $\pi \in e^{-1}(\la)$, resp.~$\varphi^{\Trans^s(\Gamma_{\Omega''} M /\Gamma_{\Omega} M)}_\pi$ for $\pi \in e^{-1}(\la^s)$, form a basis of the first, resp.~third, space in~\eqref{eqn:es-Hom-Tilmu}
Therefore, the images of the morphisms $\varphi^{\Trans^s (\Gamma_{\Omega''} M/\Gamma_{\Omega'} M)}_\pi$ for $\pi \in e^{-1}(\{\la, \la^s\})=(e')^{-1}(\mu)$ form a basis of the second term in the exact sequence.
Now,
by definition of $\Omega'$ and $\Omega''$, we have
\[
(\Trans^s(\Gamma_{\Omega'} M), \Cos(\mu)) = (\Trans^s(M) / \Trans^s(\Gamma_{\Omega''} M), \Cos(\mu)) = 0,
\]
so that we have canonical isomorphisms
\begin{multline*}
\Hom_{\Rep_s(G)^{\downarrow \mu}}(\Til(\mu), \Trans^s M) \cong \Hom_{\Rep_s(G)^{\downarrow \mu}}(\Til(\mu), \Trans^s (\Gamma_{\Omega''} M)) \\
\cong \Hom_{\Rep_s(G)^{\downarrow \mu}}(\Til(\mu), \Trans^s (\Gamma_{\Omega''} M) / \Trans^s(\Gamma_{\Omega'} M)).
\end{multline*}
Under these isomorphisms, for $\pi \in (e')^{-1}(\mu)$ the image of $\varphi^{\Trans^s M}_\pi$ in the first term correspond to the image of $\varphi^{\Trans^s (\Gamma_{\Omega''} M/\Gamma_{\Omega'} M)}_\pi$ in the third term. Therefore the image of the collection $(\varphi^{\Trans^s M}_\pi)_{\pi \in (e')^{-1}(\mu)}$ forms a basis of $\Hom_{\Rep_s(G)^{\downarrow \mu}}(\Til(\mu), \Trans^s M)$, and the proof is complete.
\end{proof}

\subsection{Sections of the $\Cos$-flag and translation from a wall}
\label{ss:sections-translation-from}

As in~\S\ref{ss:translation-functors}, let $s$ be a simple reflection.
In this subsection we explain how, given an object $M$ in $\Rep_s(G)$ which admits a costandard filtration and a section of the $\Cos$-flag of $M$, one can construct a section of the $\Cos$-flag of $\Trans_s M$. Again, this construction depends on the choices of morphisms in~\S\ref{ss:translation-functors}, and the idea goes back to~\cite{libedinsky}.

Let $M \in \Rep_s(G)$ be an object which admits a costandard filtration, and let $(\Pi, e, (\varphi^M_\pi)_{\pi \in \Pi})$ be a section of the $\Cos$-flag of $M$. We set
\[
\Pi':= \Pi \times \{0,1\}.
\]
We define the map $e' \colon \Pi' \to \bX^+_0$ as follows.
Given $\pi \in \Pi$, the images $e'(\pi, 0)$ and $e'(\pi, 1)$ are characterized by the following properties:
\begin{itemize}
\item
$e(\pi)$ belongs to the closures of the alcoves containing $e'(\pi,0)$ and $e'(\pi,1)$;
\item
$e'(\pi, 0)^s = e'(\pi,1)$;
\item
$e'(\pi, 0) \uparrow e'(\pi, 1)$.
\end{itemize}
Finally we need to define, for any $\pi \in \Pi$, morphisms $\varphi^{\Trans_s M}_{(\pi, 0)}$ and $\varphi^{\Trans_s M}_{(\pi, 1)}$. First, recall that we have fixed an isomorphism $\Trans_s \Til(e(\pi)) \cong \Til(e'(\pi,1))$, see~\eqref{eqn:translation-tilting-1}. Using this isomorphism we simply define $\varphi^{\Trans_s M}_{(\pi, 1)}$ as the composition
\[
\Til(e'(\pi, 1)) \simto \Trans_s \Til(e(\pi)) \xrightarrow{\Trans_s (\varphi_\pi^M)} \Trans_s M.
\]
On the other hand, we have also fixed a projection $\Trans^s \Til(e'(\pi, 0)) \twoheadrightarrow \Til(e(\pi))$, see~\eqref{eqn:translation-tilting-2}. We define $\varphi^{\Trans_s M}_{(\pi, 0)}$ as the composition
\[
\Til(e'(\pi, 0)) \xrightarrow{\adj} \Trans_s \Trans^s \Til(e'(\pi, 0)) \twoheadrightarrow \Trans_s \Til(e(\pi)) \xrightarrow{\Trans_s(\varphi_\pi^M)} \Trans_s M.
\]
In other words, $\varphi^{\Trans_s M}_{(\pi, 0)}$ is the image of the composition
\[
\Trans^s \Til(e'(\pi, 0)) \twoheadrightarrow \Til(e(\pi)) \xrightarrow{\varphi_\pi^M} M
\]
under the isomorphism
\[
\Hom_{\Rep_s(G)}(\Trans^s \Til(e'(\pi, 0)),M) \cong
\Hom_{\Rep_0(G)}(\Til(e'(\pi, 0)),\Trans_s M)
\]
induced by our adjunction $(\Trans^s, \Trans_s)$.

\begin{rmk}
\label{rmk:translation-section-from}
As in~\S\ref{ss:sections-translation-to},
an important point for us is that both $\varphi^{\Trans_s M}_{(\pi,0)}$ and $\varphi^{\Trans_s M}_{(\pi,1)}$ factor through the morphism $\Trans_s (\varphi^M_\pi) \colon \Trans_s \Til(e(\pi)) \to \Trans_s M$.
\end{rmk}

The main result of this subsection is the following.

\begin{prop}
\label{prop:translation-section-from}
The triple $(\Pi', e', (\varphi_\pi^{\Trans_s M})_{\pi \in \Pi'})$ constructed above is a section of the $\Cos$-flag of $\Trans_s M$.
\end{prop}

As in~\S\ref{ss:sections-translation-to},
before proving this result in general, we consider the special case where there exists $\mu \in \bX^+_s$ such that $M$ is a direct sum of objects $\Cos(\mu)$ for some $\mu \in \bX^+_s$.

\begin{lem}
\label{lem:translation-section-from}
Let $\mu \in \bX^+_s$. If
$M$ is isomorphic to a direct sum of objects $\Cos(\mu)$, the triple $(\Pi', e', (\varphi_\pi^{\Trans_s M})_{\pi \in \Pi'})$ constructed above is a section of the $\Cos$-flag of $\Trans_s M$.
\end{lem}

\begin{proof}
As in the proof of Lemma~\ref{lem:translation-section-to}, we can assume $M=\Cos(\mu)$, and then that $\Pi=\{\mu\}$, $e(\mu)=\mu$, and $\varphi_\mu^{\Cos(\mu)}$ is a non-zero (hence surjective) morphism $\Til(\mu) \to \Cos(\mu)$. Let $\la=e'(\mu, 0)$, so that $\la^s = e'(\mu,1)$ and $\la \uparrow \la^s$.
By~\cite[Proposition~II.7.19]{jantzen}, there exists a short exact sequence
\[
\Cos(\la) \hookrightarrow \Trans_s \Cos(\mu) \twoheadrightarrow \Cos(\la^s).
\]
We have morphisms $\varphi_{(\mu, 0)}^{\Trans_s \Cos(\mu)} \colon \Til(\la) \to \Trans_s \nabla(\mu)$ and  $\varphi_{(\mu, 1)}^{\Trans_s \Cos(\mu)} \colon \Til(\la^s) \to \Trans_s \nabla(\mu)$, and to prove the lemma it suffices to prove that the compositions
\[
\Sta(\la) \hookrightarrow \Til(\la) \xrightarrow{\varphi_{(\mu, 0)}^{\Trans_s \Cos(\mu)}} \Trans_s \Cos(\mu) \quad \text{and} \quad \Sta(\la^s) \hookrightarrow \Til(\la^s) \xrightarrow{\varphi_{(\mu, 1)}^{\Trans_s \Cos(\mu)}} \Trans_s \Cos(\mu)
\]
are non zero (see Remark~\ref{rmk:AST}).

Consider first the case of $\varphi_{(\mu, 1)}^{\Trans_s \Cos(\mu)}$. By construction and exactness of $\Trans_s$, the morphism $\varphi_{(\mu, 1)}^{\Trans_s \Cos(\mu)}$ is surjective. Hence the composition with the embedding of $\Sta(\la^s)$ cannot vanish, since $[\Til(\la^s)/\Sta(\la^s) : \Sim(\la^s)]=0$, so that no morphism $\Til(\la^s)/\Sta(\la^s) \to \Trans_s \Cos(\mu)$ can be surjective.

Consider now $\varphi_{(\mu, 0)}^{\Trans_s \Cos(\mu)}$. By the same arguments as in the proof of Lemma~\ref{lem:section-Gamma}, this morphism must factor through a morphism $\Til(\la) \to \Cos(\la)$, so that we only have to prove that $\varphi_{(\mu, 0)}^{\Trans_s \Cos(\mu)} \neq 0$ (see Lemma~\ref{lem:Hom-Til-Sta-Cos}). However, 
the composition
\[
\Trans^s \Til(\la) \twoheadrightarrow \Til(\mu) \xrightarrow{\varphi_\mu^{\Cos(\mu)}} \Cos(\mu)
\]
is surjective, hence non zero. By construction, this implies that $\varphi_{(\mu, 0)}^{\Trans_s \Cos(\mu)} \neq 0$, and finishes the proof.
\end{proof}

\begin{proof}[Proof of Proposition~{\rm \ref{prop:translation-section-from}}]
This proof is very similar to the proof of Proposition~\ref{prop:translation-section-to}.

Let $\la \in \bX_0^+$. If $\la^s \notin \bX_0^+$, then $(\Trans_s M : \Cos(\la))=0$ and $\la \notin e'(\Pi')$, so that there is nothing to prove for this $\la$. Assume now that $\la^s \in \bX^+$, and let $\mu \in \bX_s^+$ be the unique element which belongs to the closure of the alcoves of $\la$ and $\la^s$. Replacing $\la$ by $\la^s$ if necessary, we can assume that $\la \uparrow \la^s$, so that for $\pi \in \Pi$ we have $e(\pi,0)=\la$ iff $e(\pi,1)=\la^s$ iff $e(\pi)=\mu$. We have to prove that, for such $\pi$'s, the images of the morphisms $\varphi_{(\pi,0)}^{\Trans_s M}$ in $\Hom_{\Rep_0(G)^{\downarrow \la}}(\Til(\la), \Trans_s M)$, resp.~the images of the morphisms $\varphi_{(\pi,1)}^{\Trans_s M}$ in $\Hom_{\Rep_0(G)^{\downarrow \la^s}}(\Til(\la^s), \Trans_s M)$, form a basis.

Let $\Omega \subset \bX_s^+$ be an ideal such that $\mu \in \Omega$ is maximal; then $\Omega':=\Omega \smallsetminus \{\mu\}$ is also an ideal. We have embeddings
\[
\Gamma_{\Omega'} M \hookrightarrow \Gamma_\Omega M \hookrightarrow M,
\]
and $\Gamma_\Omega M / \Gamma_{\Omega'} M$ is a direct sum of copies of $\Cos(\mu)$. By Lemma~\ref{lem:section-Gamma}, for $\pi \in e^{-1}(\mu)$ the morphism $\varphi^M_\pi$ factors through a morphism $\varphi^{\Gamma_\Omega M}_\pi \colon \Til(\mu) \to \Gamma_\Omega M$, and using also Lemma~\ref{lem:section-Gamma-2}, if we consider the composition
\[
\varphi^{\Gamma_\Omega M / \Gamma_{\Omega'} M}_\pi \colon \Til(\mu) \xrightarrow{\varphi^{\Gamma_\Omega M}_\pi} \Gamma_\Omega M \twoheadrightarrow \Gamma_\Omega M / \Gamma_{\Omega'} M,
\]
then $(e^{-1}(\mu), e_{|e^{-1}(\mu)}, (\varphi^{\Gamma_\Omega M / \Gamma_{\Omega'} M}_\pi)_{\pi \in e^{-1}(\mu)})$ is a section of the $\Cos$-flag of the object $\Gamma_\Omega M / \Gamma_{\Omega'} M$.
By Lemma~\ref{lem:translation-section-from}, the images of the corresponding morphisms
\[
\varphi_{(\pi,0)}^{\Trans_s(\Gamma_\Omega M / \Gamma_{\Omega'} M)} \colon \Til(\la) \to \Trans_s(\Gamma_\Omega M / \Gamma_{\Omega'} M)
\]
form a basis of $\Hom_{\Rep_0(G)^{\downarrow \la}}(\Til(\la), \Trans_s (\Gamma_\Omega M / \Gamma_{\Omega'} M))$, and the images of the corresponding morphisms
\[
\varphi_{(\pi,1)}^{\Trans_s(\Gamma_\Omega M / \Gamma_{\Omega'} M)} \colon \Til(\la^s) \to \Trans_s(\Gamma_\Omega M / \Gamma_{\Omega'} M)
\]
form a basis of $\Hom_{\Rep_0(G)^{\downarrow \la^s}}(\Til(\la^s), \Trans_s (\Gamma_\Omega M / \Gamma_{\Omega'} M))$.

Now we have
\[
(\Trans_s (\Gamma_{\Omega'} M) : \Cos(\la)) = (\Trans_s (\Gamma_{\Omega'} M) : \Cos(\la^s)) = 0
\]
and
\[
(\Trans_s (M/\Gamma_{\Omega} M) : \Cos(\la)) = (\Trans_s (M/\Gamma_{\Omega} M) : \Cos(\la^s)) = 0,
\]
so that we have natural isomorphisms
\[
\Hom_{\Rep_0(G)^{\downarrow \la}}(\Til(\la), \Trans_s M ) \cong \Hom_{\Rep_0(G)^{\downarrow \la}}(\Til(\la), \Trans_s (\Gamma_\Omega M / \Gamma_{\Omega'} M))
\]
and
\[
\Hom_{\Rep_0(G)^{\downarrow \la^s}}(\Til(\la^s), \Trans_s M) \cong \Hom_{\Rep_0(G)^{\downarrow \la^s}}(\Til(\la^s), \Trans_s (\Gamma_\Omega M / \Gamma_{\Omega'} M)).
\]
Under these isomorphisms, the images of the morphisms $\varphi_{(\pi,0)}^{\Trans_s M}$, resp.~$\varphi_{(\pi,1)}^{\Trans_s M}$, correspond to the images of the morphisms $\varphi_{(\pi,0)}^{\Trans_s(\Gamma_\Omega M / \Gamma_{\Omega'} M)}$, resp.~$\varphi_{(\pi,1)}^{\Trans_s(\Gamma_\Omega M / \Gamma_{\Omega'} M)}$. We deduce the expected properties of these morphisms, which concludes the proof.
\end{proof}

\subsection{Morphisms between ``Bott--Samelson type'' tilting modules}
\label{ss:BStilting}

If $s \in S$, we will consider the ``wall crossing" functor
\[
\Theta_s := \Trans_s \Trans^s \colon \Rep_0(G) \to \Rep_0(G).
\]

As in~\cite{ew}, an \emph{expression} is a sequence $\uw = (s_1,s_2, \cdots ,s_m)$ with $s_i \in
\Saff$. The length $\ell(\uw)$ is the number $m$ of reflections appearing in this sequence. Omitting the underline denotes the product $s_1 \cdots s_m \in
\Waff$. We will often abuse notation and write expressions as $\uw = s_1 \cdots
s_m$, with the underline there to remind us that we consider an expression
and not the product in $\Waff$. An expression $\uw = s_1 \cdots s_m$ is \emph{reduced} if
$\ell(\un{w}) = \ell(w)$.

If $\uw = s_1 \cdots s_m$ is an expression, we set
\[
\Til(\uw) := \Theta_{s_m} \circ \cdots \circ \Theta_{s_1}(\Til(\la_0)).
\]
(Note the inversion in the order of the simple reflections.) This object is a tilting module in $\Rep_0(G)$. Note also that if $\uw$ is a reduced expression of an element $w \in \fW$, then $\Til(\uw)$ contains $\Til(w \hdot \lambda_0)$ as a direct summand (with multiplicity $1$).

We will now derive from the results of~\S\S\ref{ss:sections-translation-to}--\ref{ss:sections-translation-from} some properties of the morphism spaces between such objects, which will play a crucial role in Section~\ref{sec:main-conj}.

\begin{prop}
\label{prop:morphisms-BS}
Let $\ux$ and $\uv$ be expressions, and assume that $\ux$ is a reduced expression for some element $x \in \fW$. Let also $\lambda = x \hdot \lambda_0$.
\begin{enumerate}
\item
\label{it:morphisms-BS-1}
Assume that $\lambda \uparrow \lambda^s$ (so that $\ux s$ is a reduced expression for $xs \in \fW$). Let $(f_i)_{i \in I}$ be a family of morphisms in $\Hom_{\Rep_0(G)}(\Til(\ux), \Til(\uv))$ whose images span the vector space $\Hom_{\Rep_0(G)^{\downarrow \la}}(\Til(\ux), \Til(\uv))$, and let $(g_j)_{j \in J}$ be a family of morphisms in $\Hom_{\Rep_0(G)}(\Til(\ux s), \Til(\uv))$ whose images span the vector space $\Hom_{\Rep_0(G)^{\downarrow \la^s}}(\Til(\ux s), \Til(\uv))$. There exists morphisms $f_i' \colon \Til(\ux) \to \Theta_s \Til(\ux)=\Til(\ux s)$ (for $i \in I$) and $g_j' \colon \Til(\ux) \to \Theta_s \Til(\ux s) = \Til(\ux s s)$ (for $j \in J$) such that the images of the compositions
\[
\Til(\ux) \xrightarrow{f_i'} \Theta_s \Til(\ux) \xrightarrow{\Theta_s(f_i)} \Theta_s \Til(\uv) = \Til(\uv s)
\]
together with the images of the compositions
\[
\Til(\ux) \xrightarrow{g_j'} \Theta_s \Til(\ux s) \xrightarrow{\Theta_s(g_j)} \Theta_s \Til(\uv) = \Til(\uv s)
\]
span $\Hom_{\Rep_0(G)^{\downarrow \la}}(\Til(\ux), \Til(\uv s))$.
\item
\label{it:morphisms-BS-2}
Assume that $\ux=\uy s$ for some $\uy$ which is a reduced expression for an element $y \in \fW$. (Then $\lambda^s=y \hdot \lambda_0 \in \bX^+$, and
$\la^s \uparrow \la$.) 
Let $(f_i)_{i \in I}$ be a family of morphisms in $\Hom_{\Rep_0(G)}(\Til(\ux), \Til(\uv))$ whose images span 
the vector space 
$\Hom_{\Rep_0(G)^{\downarrow \la}}(\Til(\ux), \Til(\uv))$, and let $(g_j)_{j \in J}$ be a family of morphisms in $\Hom_{\Rep_0(G)}(\Til(\uy), \Til(\uv))$ whose images span 
$\Hom_{\Rep_0(G)^{\downarrow \la^s}}(\Til(\uy), \Til(\uv))$. There exists morphisms $f_i' \colon \Til(\ux) \to \Theta_s \Til(\ux) = \Til(\ux s)$ and $g_j' \colon \Til(\ux) \to \Theta_s \Til(\uy) = \Til(\ux)$ such that the images of the compositions
\[
\Til(\ux) \xrightarrow{f_i'} \Theta_s \Til(\ux) \xrightarrow{\Theta_s(f_i)} \Theta_s \Til(\uv) = \Til(\uv s)
\]
together with the images of the compositions
\[
\Til(\ux) \xrightarrow{g_j'} \Theta_s \Til(\uy) \xrightarrow{\Theta_s(g_j)} \Theta_s \Til(\uv) = \Til(\uv s)
\]
span $\Hom_{\Rep_0(G)^{\downarrow \la}}(\Til(\ux), \Til(\uv s))$.
\end{enumerate}
\end{prop}

\begin{proof}
\eqref{it:morphisms-BS-1}
We have $\Til(\ux) \cong \Til(\la)$ in $\Rep_0(G)^{\downarrow \la}$, and $\Til(\ux s) \cong \Til(\la^s)$ in $\Rep_0(G)^{\downarrow \la^s}$. Hence we can fix split embeddings $\Til(\la) \hookrightarrow \Til(\ux)$ and $\Til(\la^s) \hookrightarrow \Til(\ux s)$ and assume that the compositions
\[
\Til(\la) \hookrightarrow \Til(\ux) \xrightarrow{f_i} \Til(\uv) \quad \text{and} \quad \Til(\la^s) \hookrightarrow \Til(\ux s) \xrightarrow{g_j} \Til(\uv)
\]
are part of a section of the $\Cos$-flag of $\Til(\uv)$. Let $\mu \in \bX_s^+$ be the unique element which belongs to the closures of the alcoves of $\la$ and $\la^s$. Then Proposition~\ref{prop:translation-section-to} provides a section of the $\Cos$-flag of $\Trans^s \Til(\uv)$ whose morphisms $\Til(\mu) \to \Trans^s \Til(\uv)$ are parametrized by $I \sqcup J$, in such a way that the morphism corresponding to $i \in I$ factors through the morphism $\Trans^s(f_i) \colon \Trans^s \Til(\ux) \to \Trans^s \Til(\uv)$, and that the morphism corresponding to $j \in J$ factors through the morphism $\Trans^s(g_j) \colon \Trans^s \Til(\ux s) \to \Trans^s \Til(\uv)$ (see in particular Remark~\ref{rmk:translation-section-to}). Applying Proposition~\ref{prop:translation-section-from}, we then obtain a section of the $\Cos$-flag of $\Til(\uv s) = \Theta_s \Til(\uv)$ whose morphisms $\Til(\la) \to \Til(\uv s)$ are parametrized by $I \sqcup J$, in such a way that the morphism corresponding to $i \in I$ factors through the morphism $\Theta_s(f_i) \colon \Theta_s \Til(\ux) \to \Theta_s \Til(\uv)$, and that the morphism corresponding to $j \in J$ factors through the morphism $\Theta_s(g_j) \colon \Theta_s \Til(\ux s) \to \Theta_s \Til(\uv)$ (see in particular Remark~\ref{rmk:translation-section-from}). Composing with a fixed surjection $\Til(\ux) \to \Til(\la)$ we deduce the desired morphisms.

\eqref{it:morphisms-BS-2} The proof is identical to the proof of~\eqref{it:morphisms-BS-1}, and therefore omitted.
\end{proof}

\section{Diagrammatic Hecke category and the antispherical module}
\label{sec:Diag}

\subsection{The affine Hecke algebra and the antispherical module}
\label{ss:Haff-Masph}

To the Coxeter groups $(\Waff,\Saff)$ and $(\Wf, \Sf)$ one can associate the Hecke algebras $\Haff$ and $\Hf$ over $\Z[v,v^{-1}]$. We will follow the notation of~\cite{soergel} rather closely: in particular $\Haff$ has a ``standard'' basis $\{H_w : w \in \Waff\}$ and a ``Kazhdan--Lusztig'' basis $\{\uH_w : w \in \Waff\}$, and $\Hf$ identifies with the $\Z[v,v^{-1}]$-subalgebra spanned (as a $\Z[v,v^{-1}]$-module) by the element $H_w$ for $w \in \Wf$. If $\uw=s_1 \cdots s_m$ is an expression, we set
\[
\uH_{\uw} := \uH_{s_1} \cdots \uH_{s_m} \ \in \Haff.
\]

The algebra $\Hf$ has a natural ``sign'' right module $\mathsf{sgn}$ such that $\mathsf{sgn}=\Z[v,v^{-1}]$, and $H_s$ acts as multiplication by $-v$ for $s \in \Sf$. Then we can consider the ``antispherical'' right $\Haff$-module defined as
\[
\Masph := \mathsf{sgn} \otimes_{\Hf} \Haff.
\]
(This module is denoted $\mathcal{N}^0$ in \cite[\S 5]{soergel}.) 
This module has a ``standard'' basis $\{N_w : w \in \fW\}$, where $N_w=1 \otimes H_w$ for $w \in \fW$, and a ``Kazhdan--Lusztig'' basis
$\{\uN_w : w \in \fW \}$, see \cite[Theorem~3.1]{soergel}. Let $\varphi \colon \Haff \to \Masph$ be the morphism of right $\Haff$-modules sending $H$ to $1 \otimes H$. Then it is explained in \cite[Proof of Proposition 3.4]{soergel} that
\begin{equation}
\label{eqn:canonical-basis-asph}
\varphi(\uH_w) = \begin{cases}
\uN_w & \text{if $w \in \fW$;} \\
0 & \text{otherwise.}
\end{cases}
\end{equation}
Recall also that
if $s \in \Saff$ and $w \in \fW$ then we have
\begin{equation} \label{eq:sact}
N_w \cdot \uH_s = \begin{cases} N_{ws} + vN_w & \text{if $ws \in \fW$ and $ws > w$}, \\
N_{ws} + v^{-1}N_w & \text{if $ws \in \fW$ and $ws < w$}, \\
0 & \text{otherwise,} \end{cases}
\end{equation}
see~\cite[p.~86]{soergel}. The third case in this formula relies on the observation that, for $w \in \fW$ and $s \in S$,
\begin{equation}
\label{eqn:not-in-fW}
ws \notin \fW \ \Rightarrow \ (\exists r \in \Sf, \ ws=rw), 
\end{equation}
see~\cite{soergel}.

Fix an expression $\un{w} = s_1s_2 \cdots s_m$.
A \emph{subexpression}, denoted $\un{e} \subset \un{w}$, is a sequence $e_1
\cdots e_m$ with $e_i \in \{ 0, 1 \}$. Its \emph{end-point} is
$\un{w}^{\un{e}} := s_1^{e_1} \cdots s_m^{e_m} \in \Waff$; we will also say that $\ue$ \emph{expresses} $\un{w}^{\un{e}}$. The \emph{Bruhat
  stroll} of $\un{e}$ is the sequence
\[
x_0 := 1, \, x_1 := s_1^{e_1}, \, x_2 := s_1^{e_1}s_2^{e_2}, \ \cdots, \ x_m := \un{w}^{\un{e}}
\]
of elements of $\Waff$. To each index $i \in \{1, \cdots, m\}$ we assign a symbol:
\begin{itemize}
\item
U1 if $e_i=1$ and $x_i = x_{i-1} s_i > x_{i-1}$;
\item
D1 if $e_i=1$ and $x_i = x_{i-1} s_i < x_{i-1}$;
\item
U0 if $e_i=0$ and $x_{i-1} s_i > x_{i-1}$;
\item
D0 if $e_i=0$ and $x_{i-1} s_i < x_{i-1}$.
\end{itemize}
(Here ``U'' stands for Up, and ``D'' stands for Down.) The defect $d(\ue) \in \Z$ of $\ue$ is the difference between the number of symbols U0 and the number of symbols D0 in this list (see e.g.~\cite[\S 2.4]{ew}).

Given subexpressions $\un{e}', \un{e}'' \subset
\un{w}$, denote their Bruhat strolls by $x_0', x_1', \dots, x_m'$ and
$x_0'', x_1'', \dots, x_m''$ respectively.
We say that $\un{e}' \le \un{e}''$ if $x_i' \le x_i''$ in the Bruhat order for all $0 \le
i \le m$. This gives a partial order on the set of subexpressions of $\un{w}$
which we call the \emph{path dominance order} (see \cite[\S 2.4]{ew}).

Finally,
given a subset $K \subset W$ we will say that $\un{e}$ \emph{avoids}
$K$ if
\[
x_{i-1}s_i \notin K \quad \text{for $1 \le i \le m$}.
\]
(In particular, this condition is automatically satisfied if $m=0$.)


The following easy lemma is the analogue in our antispherical setting of~\cite[Lemma~2.10]{ew} (which is due to Deodhar).

\begin{lem}
\label{lem:number-subexpr-avoids}
  For any expression $\un{w}$, in $\Masph$ we have
\[
N_{1} \cdot \uH_{\un{w}} = \sum_{\substack{\un{e} \subset \un{w} \\ \text{$\un{e}$
    avoids $W \smallsetminus \fW$ }}} v^{d(\un{e})}N_{\un{w}^{\un{e}}}.
\]
\end{lem}

\begin{proof}
  The formula is obvious if $\ell(\un{w}) = 0$.
  Now, let $\un{w}$ be an
  expression of length $m \ge 0$ and let $s \in \Saff$. By induction and
  \eqref{eq:sact} we have
\begin{gather*}
N_{1} \cdot \uH_{\uw s} = (N_{1} \cdot \uH_{\un{w}}) \cdot \uH_s = \left( \sum_{\substack{\un{e} \subset \un{w} \\ \text{$\un{e}$
    avoids $\Waff \smallsetminus \fW$} }}
  v^{d(\un{e})}N_{\un{w}^{\un{e}}} \right) \cdot \uH_s = \\
   \sum_{\substack{\un{e} \subset \un{w}, \\ \text{$\un{e}$
    avoids $\Waff \smallsetminus \fW$,} \\ \un{w}^{\un{e}}s \in \fW, \un{w}^{\un{e}} s
  > \un{w}^{\un{e}}}}(v^{d(\un{e})}N_{\un{w}^{\un{e}}s} + v^{d(\un{e})+1}N_{\un{w}^{\un{e}}}) + 
\sum_{\substack{\un{e} \subset \un{w}, \\ \text{$\un{e}$
    avoids $\Waff \smallsetminus \fW$,} \\ \un{w}^{\ue}s \in \fW, \un{w}^{\un{e}} s
  < \un{w}^{\un{e}}}} (v^{d(\un{e})}N_{\un{w}^{\un{e}}s} +
v^{d(\un{e})-1}N_{\un{w}^{\un{e}}}) \\
=  \sum_{\substack{\un{e}' \subset \un{w}s, \\ \text{$\un{e}'$
    avoids  $\Waff \smallsetminus \fW$} \\ \un{w}^{\tau(\un{e}')} s
  >\un{w}^{\tau(\un{e}')}  }} v^{d(\un{e}')} N_{(\un{w}s)^{\un{e}'}}+
\sum_{\substack{\un{e}' \subset \un{w}s, \\ \text{$\un{e}'$
    avoids  $\Waff \smallsetminus \fW$} \\ \un{w}^{\tau(\un{e}')} s
  <\un{w}^{\tau(\un{e}')}  }} v^{d(\un{e}')} N_{(\un{w}s)^{\un{e}'}} \\
= \sum_{\substack{\un{e}' \subset \un{w}s, \\ \text{$\un{e}'$
    avoids  $\Waff \smallsetminus \fW$}}} v^{d(\un{e}')} N_{(\un{w}s)^{\un{e}'}}
\end{gather*}
where in the third line $\tau(\un{e}') = e'_1 \cdots e'_m$ is the
subexpression of $\un{w}$ obtained by omitting the last term in $\un{e}'$. The lemma follows.
\end{proof}

\subsection{Diagrammatic Soergel bimodules}
\label{ss:diag-SB}

From now on in this section we fix an integral domain $\K$. We set $\fh := \K \otimes_\Z \Z \Phi$, and define for any $s \in S$ elements $\alpha_s \in \fh^*:=\Hom_{\K}(\fh,\K)$ and $\alpha_s^\vee \in \fh$ as follows:
\begin{itemize}
\item
if $s \in \Sf$, then $\alpha_s$ and $\alpha_s^\vee$ are the images in $\fh^*$ and $\fh$ of the simple coroot and simple root associated with $s$ respectively;
\item
if $s \in \Saff \smallsetminus \Sf$, then the image of $s$ under the natural projection $\Waff \to \Wf$ is a reflection $s_\gamma$ for some $\gamma \in \Phi^+$; we define $\alpha_s$ as the image in $\fh^*$ of $-\gamma^\vee$, and $\alpha_s^\vee$ as the image in $\fh$ of $-\gamma$.
\end{itemize}
(This notation might be misleading, but it will be abandoned very soon.) We make the following assumptions:
\begin{gather}
\label{eqn:assumption-pairing}
\text{the natural pairing $(\K \otimes_\Z \Z\Phi) \times (\K \otimes_\Z \Z\Phi^\vee) \to \K$ is a perfect pairing;} \\
\label{eqn:assumption-Dem-surjectivity}
\text{for any $s \in \Saff$, the morphisms $\alpha_s \colon \fh \to \K$ and $\alpha_s^\vee \colon \fh^* \to \K$ are surjective.}
\end{gather}
Assumption~\eqref{eqn:assumption-pairing} is equivalent to requiring that the determinant of the Cartan matrix of the root system $\Phi$ is invertible in $\K$, and assumption~\eqref{eqn:assumption-Dem-surjectivity} is clearly always satisfied if $2$ is invertible in $\K$. Note also that if these assumptions hold for a commutative ring $\K$, then they hold for all commutative $\K$-algebras.

\begin{rmk}
\label{rmk:condition-p-h}
In our main application of the results of this section, the ring $\K$ will be the field $\bk$ of Section~\ref{sec:blocks}. In this case, the condition $p \geq h$ ensures that~\eqref{eqn:assumption-Dem-surjectivity} holds, except maybe when $\Phi$ is a (non empty) direct sum of root systems of type $\mathbf{A}_1$. To ensure that~\eqref{eqn:assumption-pairing} holds, we will assume that $p>h$.
\end{rmk}

The datum of $\fh$ and the subsets $\{\alpha_s : s \in S\} \subset \fh^*$ and $\{\alpha_s^\vee : s \in S\} \subset \fh$ defines a realization of $(\Waff,\Saff)$ over $\K$ in the sense of~\cite[Definition~3.1]{ew}.
(In fact it is easy to check that the technical condition in~\cite[(3.3)]{ew} is satisfied.). Moreover, this realization is balanced in the sense of~\cite[Definition~3.6]{ew}. The representation of $\Waff$ on $\fh$ associated with this realization factors through the natural representation of $\Wf$. Finally, our assumption~\eqref{eqn:assumption-Dem-surjectivity} precisely says that Demazure Surjectivity (see~\cite[Assumption~3.7]{ew}) holds for this realization. (This condition is necessary for the category $\DiagBS$ to be well behaved, see~\cite[Remark~5.5]{ew}.) As in~\cite{ew}, we denote by $R$ the symmetric algebra (over $\K$) of $\fh^*$. Our assumption~\eqref{eqn:assumption-pairing} implies that
\begin{equation}
\label{eqn:generators-R}
\text{$R$ is generated, as a $\K$-algebra, by the images of the coroots $\alpha^\vee$ for $\alpha \in \Sigma$.}
\end{equation}


Let us consider the category associated with this realization as defined
in~\cite[Definition~5.2]{ew}. In the present paper we will denote this category by $\DiagBS$. (Here ``$\mathrm{BS}$'' stands for ``Bott--Samelson''; this category is denoted $\mathcal{D}$ in~\cite{ew}.) We will consider $\DiagBS$ as a category endowed with a ``shift of the grading'' autoequivalence $\langle 1 \rangle$ (denoted $(1)$ in~\cite{ew}) rather than as a graded category; therefore this category has objects $B_{\uw} \langle n \rangle$ parametrized by pairs consisting of an expression $\uw$ and an integer $n \in \Z$, and we have $(B_{\uw} \langle n \rangle) \langle 1 \rangle = B_{\uw} \langle n+1 \rangle$. We will write $B_{\uw}$ instead of $B_{\uw} \langle 0 \rangle$. The category $\DiagBS$ is monoidal, with product defined on objects by
\[
(B_{\uv} \langle n \rangle) \cdot (B_{\uw} \langle m \rangle) = B_{\uv \uw} \langle n + m \rangle,
\]
where $\uv \uw$ is the concatenation of the expressions $\uv$ and $\uw$.

As in~\cite{ew} we will use diagrams to denote morphisms in $\DiagBS$:
a morphism in
$\Hom_{\DiagBS}(B_{\uv} \langle n \rangle, B_{\uw} \langle m \rangle)$
is a $\K$-linear combination of certain equivalence classes of diagrams whose bottom has strands labeled by the simple reflections appearing in $\uv$, and whose top has strands labeled by the simple reflections appearing in $\uw$. (In particular, diagrams should be read from bottom to top.)
Diagrammatically, the product corresponds to horizontal concatenation, and composition corresponds to vertical concatenation.
The diagrams are constructed by (horizontal and vertical) concatenation of images under powers of $\langle 1 \rangle$ of $4$ different types of generators:
\begin{enumerate}
\item
\label{it:morphisms-DBS-f}
morphisms
$B_{\varnothing} \to B_{\varnothing} \langle \deg(f) \rangle$
for any homogeneous $f \in R$, represented diagrammatically as the diagram
\[
    \begin{tikzpicture}[thick,scale=0.07,baseline]
      \node at (0,0) {$f$};
    \end{tikzpicture}
\]
with empty top and bottom;
\item
\label{it:morphisms-DBS-dots}
the upper and lower dots
$B_s \to B_{\varnothing}\langle 1 \rangle$ and $B_{\varnothing} \to B_s \langle 1 \rangle$
(for $s \in \Saff$), represented diagrammatically as
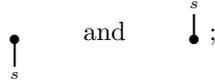
\[
    \begin{tikzpicture}[thick,scale=0.07,baseline]
      \draw (0,-5) to (0,0);
      \node at (0,0) {$\bullet$};
      \node at (0,-6.7) {\tiny $s$};
    \end{tikzpicture}
    \qquad
    \text{and}
    \qquad
      \begin{tikzpicture}[thick,baseline,xscale=0.07,yscale=-0.07]
      \draw (0,-5) to (0,0);
      \node at (0,0) {$\bullet$};
      \node at (0,-6.7) {\tiny $s$};
    \end{tikzpicture};
\]
\item
\label{it:morphisms-DBS-trivalent}
the trivalent vertices
$B_s \to B_{ss}\langle -1 \rangle$ and $B_{ss} \to B_s\langle -1 \rangle$
(again for $s \in \Saff$), represented diagrammatically as
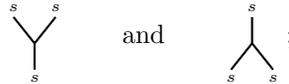
\[
    \begin{tikzpicture}[thick,baseline,scale=0.07]
      \draw (-4,5) to (0,0) to (4,5);
      \draw (0,-5) to (0,0);
      \node at (0,-6.7) {\tiny $s$};
      \node at (-4,6.7) {\tiny $s$};
      \node at (4,6.7) {\tiny $s$};      
    \end{tikzpicture}
    \qquad
    \text{and}
    \qquad
        \begin{tikzpicture}[thick,baseline,scale=-0.07]
      \draw (-4,5) to (0,0) to (4,5);
      \draw (0,-5) to (0,0);
      \node at (0,-6.7) {\tiny $s$};
      \node at (-4,6.7) {\tiny $s$};
      \node at (4,6.7) {\tiny $s$};    
    \end{tikzpicture};
\]
\item
\label{it:morphisms-DBS-2mvalent}
for pairs $(s,t)$ of distinct elements of $\Saff$ such that $st$ has finite order $m_{st}$ in $\Waff$, the $2m_{st}$-valent vertex
$B_{st \cdots} \to B_{ts \cdots}$
(where each index has $m_{st}$ simple reflections appearing), represented diagrammatically as
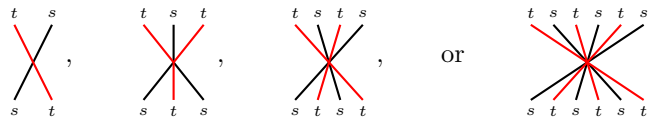
\[
\begin{tikzpicture}[scale=0.5,baseline,thick]
\draw (-0.5,-1) to (0,0) to (0.5,1);
\draw[red] (0.5,-1) to (0,0) to (-0.5,1);
\node at (-0.5,-1.3) {\tiny $s$};
\node at (0.5,1.3) {\tiny $s$};
\node at (0.5,-1.3) {\tiny $t$};
\node at (-0.5,1.3) {\tiny $t$};
    \end{tikzpicture},
    \qquad
\begin{tikzpicture}[yscale=0.5,xscale=0.4,baseline,thick]
\draw (-1,-1) to (0,0) to (0,1);
\draw (1,-1) to (0,0);
\draw[red] (0,-1) to (0,0) to (-1,1);
\draw[red] (0,0) to (1,1);
\node at (-1,-1.3) {\tiny $s$};
\node at (0,1.3) {\tiny $s$};
\node at (1,-1.3) {\tiny $s$};
\node at (0,-1.3) {\tiny $t$};
\node at (-1,1.3) {\tiny $t$};
\node at (1,1.3) {\tiny $t$};
    \end{tikzpicture},
    \qquad
\begin{tikzpicture}[yscale=0.5,xscale=0.3,baseline,thick]
\draw (-1.5,-1) to (0,0) to (-0.5,1);
\draw (0.5,-1) to (0,0) to (1.5,1);
\draw[red] (-0.5,-1) to (0,0) to (-1.5,1);
\draw[red] (1.5,-1) to (0,0) to (0.5,1);
\node at (-1.5,-1.3) {\tiny $s$};
\node at (0.5,-1.3) {\tiny $s$};
\node at (-0.5,1.3) {\tiny $s$};
\node at (1.5,1.3) {\tiny $s$};
\node at (-0.5,-1.3) {\tiny $t$};
\node at (1.5,-1.3) {\tiny $t$};
\node at (0.5,1.3) {\tiny $t$};
\node at (-1.5,1.3) {\tiny $t$};
    \end{tikzpicture},
    \qquad \text{or} \qquad
\begin{tikzpicture}[yscale=0.5,xscale=0.3,baseline,thick]
\draw (-2.5,-1) to (0,0) to (-1.5,1);
\draw (-0.5,-1) to (0,0) to (0.5,1);
\draw (1.5,-1) to (0,0) to (2.5,1);
\draw[red] (-1.5,-1) to (0,0) to (-2.5,1);
\draw[red] (0.5,-1) to (0,0) to (-0.5,1);
\draw[red] (2.5,-1) to (0,0) to (1.5,1);
\node at (-2.5,-1.3) {\tiny $s$};
\node at (-1.5,1.3) {\tiny $s$};
\node at (0.5,1.3) {\tiny $s$};
\node at (2.5,1.3) {\tiny $s$};
\node at (-0.5,-1.3) {\tiny $s$};
\node at (1.5,-1.3) {\tiny $s$};
\node at (-1.5,-1.3) {\tiny $t$};
\node at (0.5,-1.3) {\tiny $t$};
\node at (2.5,-1.3) {\tiny $t$};
\node at (-2.5,1.3) {\tiny $t$};
\node at (-0.5,1.3) {\tiny $t$};
\node at (1.5,1.3) {\tiny $t$};
    \end{tikzpicture}
\]
if $m_{st}$ is $2$, $3$, $4$ or $6$.
\end{enumerate}
(The colors used in the diagrams have no particular significance: they are here only to make it easier to see which strands have the same label.)
These generators satisfy a number of relations described in~\cite[\S 5]{ew}, which we will not repeat here. These relations define the ``equivalence relation'' considered above; let us only recall that isotopic diagrams are equivalent.

\begin{rmk}
\label{rmk:morphisms-Dasph}
Let $\alpha \in \Sigma$, and let $s=s_\alpha \in \Sf$. By~\cite[(5.1)]{ew}, the morphism associated with the image of $\alpha^\vee$ in $\fh^* \subset R$ as in~\eqref{it:morphisms-DBS-f} above is the composition
\[
B_{\varnothing} \to B_s \langle 1 \rangle \to B_{\varnothing}\langle 2 \rangle
\]
where the first morphism is the ``lower dot'' morphism for $s$ and the second one is the shift by $\langle 1 \rangle$ of the ``upper dot" morphism for $s$ (see~\eqref{it:morphisms-DBS-dots}). Therefore, by~\eqref{eqn:generators-R}, the morphisms defined in~\eqref{it:morphisms-DBS-dots}, \eqref{it:morphisms-DBS-trivalent} and~\eqref{it:morphisms-DBS-2mvalent} are sufficient to generate all the morphisms in~$\DiagBS$.
\end{rmk}

For $X, Y$ in $\DiagBS$ we set
\[
\Hom^\bullet_{\DiagBS}(X,Y) := \bigoplus_{n \in \Z} \Hom_{\DiagBS}(X, Y \langle n \rangle).
\]
This graded $\K$-module is a graded bimodule over $R$, where for $f \in R$ homogeneous and $\phi \in \Hom^\bullet_{\DiagBS}(X,Y)$, the element $f \cdot \phi$, resp.~$\phi \cdot f$, is the composition of the morphism
\begin{multline*}
X = B_\varnothing \cdot X \to (B_\varnothing \langle \deg(f) \rangle) \cdot X = X \langle \deg(f) \rangle, \\
\text{resp.} \quad X = X \cdot B_\varnothing \to X \cdot (B_\varnothing \langle \deg(f) \rangle)= X \langle \deg(f) \rangle,
\end{multline*}
induced by the morphism in~\eqref{it:morphisms-DBS-f} above, with $\phi \langle \deg(f) \rangle$. (Of course, one can equivalently compose with the similar morphisms for $Y$.) It follows from~\cite[Corollary~6.13]{ew} that $\Hom^\bullet_{\DiagBS}(X,Y)$ is free of finite rank as a left $R$-module and as a right $R$-module.

The category $\DiagBS$ admits an autoequivalence
\[
\imath \colon \DiagBS \simto \DiagBS
\]
which sends the object $B_{s_1 \cdots s_m} \langle n \rangle$ to $B_{s_m \cdots s_1} \langle n \rangle$, and acts on a morphism corresponding to a diagram by reflecting this diagram along a vertical axis. This autoequivalence satisfies
\[
\imath(X \cdot Y) = \imath(Y) \cdot \imath(X) \quad \text{for all $X,Y$ in $\DiagBS$.}
\]
There also exists a monoidal equivalence
\[
\tau \colon \DiagBS \simto \DiagBS^{\mathrm{op}}
\]
which sends $B_{\uw} \langle n \rangle$ to $B_{\uw} \langle -n \rangle$ and reflects diagrams along an horizontal axis; see~\cite[\S 6.3]{ew}.

In the rest of this subsection we assume that $\K$ is in addition a complete local ring. In this case,
we denote by $\Diag$ the Karoubi envelope of the additive hull of $\DiagBS$. (This category is denoted $\mathrm{Kar}(\mathcal{D})$ in~\cite{ew}.) Again this category is monoidal, and has an autoequivalence $\langle 1 \rangle$. By~\cite[Theorem~6.25]{ew}, for any $w \in \Waff$ there exists a unique indecomposable object $B_w$ in $\Diag$ which is a direct summand of $B_{\uw}$ for any reduced expression $\uw$ for $w$, but is not a direct summand of any $B_{\uv} \langle n \rangle$ with $\ell(\uv) < \ell(\uw)$. Moreover, any indecomposable object in $\Diag$ is isomorphic to $B_w \langle n \rangle$ for a unique pair $(w,n) \in \Waff \times \Z$.

We define $\Hom^\bullet_{\Diag}(X,Y)$ in a similar way as for $\Hom^\bullet_{\DiagBS}(X,Y)$.
We will again denote by $\imath$ the autoequivalence of $\Diag$ induced by $\imath$, and by $\tau$ the anti-autoequivalence of $\Diag$ induced by $\tau$. 
We have
\[
\imath(B_w) = B_{w^{-1}}, \quad \tau(B_w)=B_w \qquad \text{for all $w \in \Waff$.}
\]

In~\S\ref{ss:hwcat} we have considered the Serre quotient of an abelian category by a Serre subcategory. In the proof below (and in the rest of the section) we will consider another, more naive, notion of quotient of an additive category $\cC$ by a full additive subcategory $\cC'$: we denote by $\cC \quo \cC'$ the category which has the same objects as $\cC$, and such that for $M,N$ in $\cC$ the  group $\Hom_{\cC \quo \cC'}(M,N)$ is the quotient of $\Hom_{\cC}(M,N)$ by the subgroup of morphisms which factor through an object of $\cC'$.

\begin{lem}
\label{lem:BS-not-minimal}
Let $\underline{w} = s_1 \cdots s_m$ be an expression.
Any indecomposable summand $B_x \langle k \rangle$ of $B_{\un{w}}$ satisfies $s_1 x < x$.
\end{lem}

\begin{proof}
Using the autoequivalence $\imath$, it is equivalent to prove that any indecomposable summand $B_x \langle k \rangle$ of $B_{\uw}$ satisfies $x s_m < x$; this is the statement we will actually prove.

Let us fix an indecomposable direct summand $B_x \langle k \rangle$ appearing in $B_{\uw}$.
Let $\Diag^{\geq x}$ be the quotient (in the sense explained just before the statement) of $\Diag$ by the full subcategory whose objects are direct sums of object $B_y \langle j \rangle$ with $y \not\geq x$ in the Bruhat order. The ``shift of the grading'' functor $\langle 1 \rangle$ induces a similar functor on $\Diag^{\geq x}$, and we define $\Hom^\bullet_{\Diag^{\geq x}}(-,-)$ as in $\Diag$. Let us fix some reduced expression $\ux$ for $x$; then the images in $\Diag^{\geq x}$ of $B_x$ and $B_{\ux}$ coincide; in particular (the image of) $B_{\ux} \langle k \rangle$ is a direct summand of (the image of) $B_{\uw}$ in $\Diag^{\geq x}$, so that we have morphisms
\[
B_{\ux} \langle k \rangle \xrightarrow{f} B_{\uw} \xrightarrow{g} B_{\ux} \langle k \rangle
\]
whose composition is $\id_{B_{\ux} \langle k \rangle}$ in $\Diag^{\geq x}$.

If 
\[
I_{\ux} \subset \Hom_{\Diag}^\bullet(B_{\uw}, B_{\ux})
\]
is the submodule considered in~\cite[\S 7.2]{ew}, then it is clear that the canonical surjection $\Hom_{\Diag}^\bullet(B_{\uw}, B_{\ux}) \twoheadrightarrow \Hom_{\Diag^{\geq x}}^\bullet(B_{\uw}, B_{\ux})$ factors through a surjection
\[
\Hom_{\Diag}^\bullet(B_{\uw}, B_{\ux}) / I_{\ux} \twoheadrightarrow \Hom_{\Diag^{\geq x}}^\bullet(B_{\uw}, B_{\ux}).
\]
On the other hand, by~\cite[Proposition~7.6]{ew}, any choice of ``light leaves morphisms'' $L_{\uw, \ue} \in \Hom_{\Diag}^\bullet(B_{\uw}, B_{\ux})$ (where $\ue$ runs over subexpressions of $\uw$ 
expressing $x$) provides a (graded) basis of $\Hom_{\Diag}^\bullet(B_{\uw}, B_{\ux}) / I_{\ux}$ as a left $R$-module. In particular, the images of these morphisms generate $\Hom_{\Diag^{\geq x}}^\bullet(B_{\uw}, B_{\ux})$. We will fix a choice of such morphisms.


Now assume for a contradiction that $x s_m > x$. Then for any subexpression $\ue \subset \uw$ expressing $x$, the symbol associated with $m$ as in~\S\ref{ss:Haff-Masph} is either D1 or U0. It follows that, by construction, any light leaves morphism $B_{\uw} \to B_{\ux} \langle j \rangle$ (with $j \in \Z$) factors as a composition
\[
B_{\un{w}} \to B_{\ux}B_{s_m} \langle j-1 \rangle \to B_{\ux} \langle j \rangle,
\]
where the second morphism is induced by the ``upper dot'' morphism $B_{s_m} \to B_{\varnothing} \langle 1 \rangle$. We deduce that this property holds for any morphism in $\Diag^{\geq x}$ from $B_{\uw}$ to some shift of $B_{\ux}$, and in particular for the image of $g$. We fix such a factorization.


The composition
\[
B_{\ux} \langle k \rangle \xrightarrow{f} B_{\uw} \to B_{\ux} B_{s_m} \langle k-1 \rangle
\]
is nonzero in $\Diag^{\geq x}$, since its composition with the morphism $B_{\ux} B_{s_m} \langle k-1 \rangle \to B_{\ux} \langle k \rangle$ induced by the upper dot is nonzero. However we have
\begin{multline*}
\Hom_{\Diag^{\geq x}}(B_{\ux} \langle k \rangle, B_{\ux s_m} \langle k-1 \rangle) \cong \Hom_{\Diag^{\geq x}}(B_{\ux s_m} \langle 1-k \rangle, B_{\ux} \langle -k \rangle) \\
\cong \Hom_{\Diag^{\geq x}}(B_{\ux s_m}, B_{\ux} \langle -1 \rangle),
\end{multline*}
where the first isomorphism is induced by $\tau$. As above $\Hom_{\Diag^{\geq x}}^\bullet(B_{\ux s_m}, B_{\ux})$ is spanned by light leaves morphisms, and the only such morphism is of degree $1$ (corresponding to the only subexpression of $\ux s_m$ expressing $x$). Hence
\[
\Hom_{\Diag^{\geq x}}(B_{\ux s_m}, B_{\ux} \langle -1 \rangle) = \{0\},
\]
which provides the desired contradiction.
\end{proof}

%

\subsection{Two lemmas on $\DiagBS$}
\label{ss:rex-moves-DBS}

Let $\K$ be again an arbitrary integral domain satisfying~\eqref{eqn:assumption-pairing} and~\eqref{eqn:assumption-Dem-surjectivity}.

The following result is observed in~\cite[Equation~(5.14)]{ew}.

\begin{lem}
\label{lem:BsBs}
For $s \in \Saff$, in the additive hull of $\DiagBS$ we have
$B_s \cdot B_s \cong B_s \langle 1 \rangle \oplus B_s \langle -1 \rangle$.
\end{lem}

Let us now recall the notion of ``rex move'' from~\cite[\S 4.2]{ew}. (Here, ``rex'' stands for ``reduced expression''.) Consider an element $w \in \Waff$, and define the \emph{rex graph} $\Gamma_w$ as the graph with vertices the reduced expressions for $w$ and edges connecting reduced expressions if they differ by one application of a braid relation. By definition, if $\ux$ and $\uy$ are two reduced expressions for $w$, a \emph{rex move} $\ux \leadsto \uy$ is a directed path in $\Gamma_w$ from $\ux$ to $\uy$. To such a path one can associate a morphism from $B_{\ux}$ to $B_{\uy}$ in $\DiagBS$, obtained by composing the $2m_{st}$-valent morphisms associated with all the braid relations encountered in this path.

Fix now two reduced expressions $\ux$ and $\uy$ for the same element, and consider a rex move $\ux \leadsto \uy$. Let us denote by $\gamma \colon B_{\ux} \to B_{\ux}$ the morphism associated with the concatenation $\ux \leadsto \uy \leadsto \ux$, where the second portion is obtained by following the same path as for the rex move we started with, but in the reversed order.

\begin{lem}
\label{lem:rex-move-smaller}
There exists a finite set $J$ and, for any $j \in J$, a morphism $\phi_j \colon B_{\ux} \to B_{\ux}$ which factors through an object of the form $B_{\uz_j} \langle k_j \rangle$ where $\uz_j$ is obtained from $\ux$ by omitting at least $2$ simple reflections (so that, in particular, $\ell(\uz_j) \leq \ell(\ux)-2$), such that $\gamma=\id_{B_{\ux}} + \sum_{j \in J} \phi_j$.
\end{lem}

\begin{proof}
Using induction on the length of the original rex move $\ux \leadsto \uy$, it suffices to prove the claim when $\ux$ is a reduced expression for the longest element $x$ in a finite parabolic subgroup $W_{s,t}$ of $W$ generated by two distinct simple relctions $s$ and $t$, and the rex move consists of the braid relation in $W_{s,t}$. In this case,
a more precise claim is stated as~\cite[Claim~7.1]{ew}.
\end{proof}

\subsection{Categorified antispherical module}
\label{ss:cat-Masph}

As in~\S\ref{ss:rex-moves-DBS},
let $\K$ be an integral domain satisfying~\eqref{eqn:assumption-pairing} and~\eqref{eqn:assumption-Dem-surjectivity}.
We define the category $\DasphBS$ with objects $\oB_{\uw} \langle n \rangle$ parametrized by pairs consisting of an expression $\uw$ and an integer $n$, and with morphisms
\[
\Hom_{\DasphBS}(\oB_{\uw}\langle n \rangle, \oB_{\uv} \langle m \rangle)
\]
defined as the quotient of $\Hom_{\DiagBS}(B_{\uw}\langle n \rangle, B_{\uv} \langle m \rangle)$ by the subspace spanned by the morphisms which factor through an object of the form $B_{\uu} \langle k \rangle$ where $\uu$ is an expression starting with a simple reflection in $\Sf$. There exists a natural functor
\[
\DiagBS \to \DasphBS
\]
sending the object $B_{\uw}$ to $\oB_{\uw}$, and the functor $\langle 1 \rangle$ induces a functor denoted similarly on $\DasphBS$.

As for objects of $\DiagBS$, for $X, Y$ in $\DasphBS$, we set
\[
\Hom^\bullet_{\DasphBS}(X,Y) := \bigoplus_{n \in \Z} \Hom_{\DasphBS}(X, Y \langle n \rangle).
\]

Let $X,Y$ be objects of $\DiagBS$, and let $\overline{X}$ and $\overline{Y}$ be their images in $\DasphBS$. Consider the natural (surjective) morphism
\begin{equation}
\label{eqn:morph-Hom-Dasph}
\Hom^\bullet_{\DiagBS}(X,Y) \to \Hom^\bullet_{\DasphBS}(\overline{X}, \overline{Y}).
\end{equation}
Using Remark~\ref{rmk:morphisms-Dasph} we see that
for any simple root $\alpha$, with associated simple reflection $s=s_\alpha \in \Sf$, and any $f \in \Hom_{\DiagBS}(X,Y \langle n \rangle)$, the morphism $\alpha^\vee \cdot f \colon X \to Y \langle n+2 \rangle$ (where by abuse we still denote by $\alpha^\vee$ the image of this coroot in $R$) factors through the object $B_s \cdot X \langle 1 \rangle$.
Hence the image of this morphism in $\Hom^\bullet_{\DasphBS}(\overline{X}, \overline{Y})$ vanishes. From this remark and~\eqref{eqn:generators-R} we deduce that~\eqref{eqn:morph-Hom-Dasph} factors through the natural surjection
\[
\Hom^\bullet_{\DiagBS}(X,Y) \to \K \otimes_R \Hom^\bullet_{\DiagBS}(X,Y)
\]
(where $\K$ is considered as the trivial $R$-module). Since $\Hom^\bullet_{\DiagBS}(X,Y)$ is finitely generated as a left $R$-module (see~\S\ref{ss:diag-SB}), it follows that $\Hom^\bullet_{\DasphBS}(\overline{X}, \overline{Y})$ is a finitely generated $\K$-module.

In the case when $\K$ is (in addition) a complete local ring,
we set
\[
\Dasph = \Diag \quo \Diag_{W \smallsetminus \fW},
\]
where $\Diag_{W \smallsetminus \fW}$ is the additive full subcategory of $\Diag$ whose objects are the direct sums of objects $B_w \langle n \rangle$ where $w \in \Waff \smallsetminus \fW$. If $w \in \fW$, we will denote by $\oB_w$ the image of $B_w$ under the canonical functor $\Diag \to \Dasph$. Then the objects $\oB_w \langle n \rangle$ with $w \in \fW$ and $n \in \Z$ are precisely the indecomposable objects in $\Dasph$.
It follows from Lemma~\ref{lem:BS-not-minimal} that $\Dasph$ is the Karoubi envelope of the additive hull of $\DasphBS$.

\subsection{Morphisms in the categorified antispherical module}
\label{ss:cat-Masph-morph}


Let $\K$ be again an arbitrary integral domain satisfying~\eqref{eqn:assumption-pairing} and~\eqref{eqn:assumption-Dem-surjectivity}.
Let us consider again the ``light leaves morphisms'' that we encountered (when $\K$ is complete local) in the proof of Lemma~\ref{lem:BS-not-minimal}, now in the special case of the empty expression. More precisely, let $\uw=(s_1, \cdots, s_m)$ be any expression. Then, in this special case, the light leaves morphisms are certain morphisms $\LL_{\uw, \ue} \colon B_{\uw} \to B_{\varnothing} \langle k \rangle$, parametrized by the set of subexpressions $\ue \subset \uw$ expressing $1$, and which form a basis of the left $R$-module $\Hom^\bullet_{\DiagBS}(B_{\uw}, B_{\varnothing})$ (see~\cite[Proposition~6.12]{ew}). The construction of these morphisms is explained in~\cite[\S 6.1]{ew};
it depends on some choices.
In particular, if $x_0, x_1, \cdots, x_m$ is the Bruhat stroll associated with $\ue$ (as in~\S\ref{ss:Haff-Masph}), then the construction depends on the choice of reduced expressions $\ux_0, \ux_1, \cdots, \ux_m$ for these elements. (In our case $x_0=x_m=1$, so that we necessarily have $\ux_0=\ux_m=\varnothing$.)

The main result of this subsection is the following counterpart of~\cite[Proposition~6.12]{ew} for $\DasphBS$.

\begin{prop}
\label{prop:ASLL}
 For any expression $\uw$, the light leaves morphisms $\LL_{\uw,\ue}$ can be chosen in such a way that the $\K$-module $\Hom^\bullet_{\DasphBS}(\overline{B}_{\uw},
 \overline{B}_{\varnothing})$ is spanned by the images of the morphisms
$\LL_{\uw, \ue}$ where $\ue$ is a subexpression of $\uw$ expressing $1$ and avoiding $W
    \smallsetminus \fW$.
\end{prop}

\begin{rmk}
\label{rmk:morphisms-Dasph-LLbasis}
It will follow from 
Theorem~\ref{thm:Dasph-parities}
below (and its proof) that 
the morphisms considered in Proposition~\ref{prop:ASLL} actually form a \emph{basis} of the $\K$-module $\Hom^\bullet_{\DasphBS}(\oB_{\uw}, \oB_{\varnothing})$,
see \S\ref{ss:LL-basis-Dasph}. However, we don't know any ``diagrammatic'' proof of this fact, and it will not be needed for our applications below.
\end{rmk}

\begin{proof}
Recall the \emph{path dominance order} on the set of subexpressions of $\uw$ defined in~\S\ref{ss:Haff-Masph}.
First
we explain how to choose the morphisms $\LL_{\uw, \ue}$ in such a way that
\begin{equation}
\label{eqn:property-LL}
\begin{array}{c}
\text{if $\ue$ does not avoid $W \smallsetminus \fW$, then the image of $\LL_{\uw, \ue}$ in} \\
\text{$\Hom^\bullet_{\DasphBS}(\overline{B}_{\uw},
 \overline{B}_{\varnothing})$ is a linear combination of images of morphisms} \\
\text{$\LL_{\uw, \ue'}$ for subexpressions $\ue' \subset \uw$ expressing $1$ and satisfying $\ue'<\ue$.}
\end{array}
\end{equation}
We will see later that this condition is sufficient to ensure that the conclusion of the proposition holds.

Let $\ue \subset \un{w}$ be a subexpression expressing $1$. If $\ue$ avoids $W \smallsetminus \fW$, we choose $\LL_{\uw, \ue}$ arbitrarily. Now, let us assume that
$\ue$ does not avoid $\Waff \smallsetminus
\fW$. By definition, with the notation used in~\S\ref{ss:Haff-Masph}, there
exists $i \in \{1, \cdots, m\}$ such that $x_{i-1} s_i \notin \fW$; we fix such an index once and for all.

If the symbol assigned to $i$ is D1 or U1 (so that in particular $x_i=x_{i-1} s_i$), then in the construction of the light leaves
morphism corresponding to $\ue$ (as illustrated
in~\cite[Figure~2]{ew}), at the $i$-th step we choose the rex move ``$\alpha$'' in
such a way that it passes through a reduced expression for $x_{i-1}
s_i$ starting with a reflection in $\Sf$. (All the other choices can be arbitrary.) Then $\LL_{\uw, \ue}$ factors
through an object of the form $B_s \cdot X$ with $s \in \Sf$, hence it
vanishes in $\DasphBS$; a fortiori it satisfies the desired property.

If 
the symbol assigned to $i$ as in~\S\ref{ss:Haff-Masph} is D0 (so that in particular $x_{i-1} s_i < x_{i-1}$, which implies that $x_{i-1} \notin \fW$, see~\eqref{eqn:not-in-fW}), then at the $i$-th step we choose the
rex move ``$\beta$'' (again with the notation of~\cite[Figure~2]{ew}) such that it passes through a reduced expression for $x_{i-1}$ starting with a reflection in $\Sf$. Then as above $\LL_{\uw, \ue}$
vanishes in $\DasphBS$, which proves the claim.

Finally, consider the case where 
the symbol assigned to $i$ is U0. Then we choose $\LL_{\uw,\ue}$ arbitrarily. 
Using once again the notation of ~\cite[Figure~2]{ew}, at the $i$-th step of the construction the morphism looks like
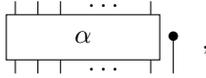
\[
\begin{array}{c}
\begin{tikzpicture}[scale=0.6]
  \draw (-0.2,1) rectangle (3.2,0);
  \node (z) at (1.5,0.5) {$\alpha$};
  \draw (0,1.3) -- (0,1);
  \draw (.5,1.3) -- (.5,1);
  \draw (1,1.3) -- (1,1);
  \draw (3,1.3) -- (3,1);
  \draw (0,-.3) -- (0,0);
  \draw (.5,-.3) -- (.5,0);
  \draw (1,-.3) -- (1,0);
  \draw (3,-.3) -- (3,0);
  \draw (3.5,-.3) -- (3.5,.5);
  \node at (3.5,.5) {$\bullet$};
  \node at (2,1.2) {$\dots$};
  \node at (2,-.2) {$\dots$};
\end{tikzpicture}
\end{array},
\]
where the top is a reduced expression $\ux_i$ for $x_i=x_{i-1}$, and the bottom is of the form $\ux_{i-1} s_i$ where $\ux_{i-1}$ is a reduced expression for $x_{i-1}$. Let us choose a reduced expression $\uy$ for $x_{i-1} s_i$ starting with an element in $\Sf$ and a rex move $\ux_{i-1} s_i \leadsto \uy$. Consider the ``reversed'' rex move $\uy \leadsto \ux_{i-1} s_i$, and denote by $\gamma \colon B_{\ux_{i-1} s_i} \to B_{\ux_{i-1} s_i}$ the morphism associated with the concatenation $\ux_{i-1} s_i \leadsto \uy \leadsto \ux_{i-1} s_i$.

By Lemma~\ref{lem:rex-move-smaller},
in $\DiagBS$ we have
\begin{equation}
\label{eqn:JWrelation}
\begin{array}{c}
\begin{tikzpicture}[scale=0.6]
  \draw (-0.2,1) rectangle (3.2,0);
  \node (z) at (1.5,0.5) {$\alpha$};
  \draw (0,1.3) -- (0,1);
  \draw (.5,1.3) -- (.5,1);
  \draw (1,1.3) -- (1,1);
  \draw (3,1.3) -- (3,1);
  \draw (0,-.3) -- (0,0);
  \draw (.5,-.3) -- (.5,0);
  \draw (1,-.3) -- (1,0);
  \draw (3,-.3) -- (3,0);
  \draw (3.5,-.3) -- (3.5,.5);
  \node at (3.5,.5) {$\bullet$};
  \node at (2,1.2) {$\dots$};
  \node at (2,-.2) {$\dots$};
\end{tikzpicture}
\end{array}
=
\begin{array}{c}
\begin{tikzpicture}[scale=0.6]
  \draw (-0.2,1) rectangle (3.2,0);
  \node (z) at (1.5,0.5) {$\alpha$};
  \node (z) at (2,-.8) {$\gamma$};
  \draw (0,1.3) -- (0,1);
  \draw (.5,1.3) -- (.5,1);
  \draw (1,1.3) -- (1,1);
  \draw (3,1.3) -- (3,1);
  \draw (0,-.3) -- (0,0);
  \draw (.5,-.3) -- (.5,0);
  \draw (1,-.3) -- (1,0);
  \draw (3,-.3) -- (3,0);
  \draw (3.5,-.3) -- (3.5,.5);
  \node at (3.5,.5) {$\bullet$};
 \draw (-0.2,-.3) rectangle (3.7,-1.3);
  \draw (0,-1.3) -- (0,-1.6);
  \draw (.5,-1.3) -- (.5,-1.6);
  \draw (1,-1.3) -- (1,-1.6);
  \draw (3,-1.3) -- (3,-1.6);
  \draw (3.5,-1.3) -- (3.5,-1.6);
  \node at (2,1.2) {$\dots$};
  \node at (2,-.15) {$\dots$};
  \node at (2.15,-1.45) {$\dots$};
\end{tikzpicture}
\end{array}
+ \ \sum_{j \in J} \begin{array}{c}
\begin{tikzpicture}[scale=0.6]
  \draw (-0.2,1) rectangle (3.2,0);
  \node (z) at (1.5,0.5) {$\alpha$};
  \node (z) at (2,-.8) {$M_j$};
  \draw (0,1.3) -- (0,1);
  \draw (.5,1.3) -- (.5,1);
  \draw (1,1.3) -- (1,1);
  \draw (3,1.3) -- (3,1);
  \draw (0,-.3) -- (0,0);
  \draw (.5,-.3) -- (.5,0);
  \draw (1,-.3) -- (1,0);
  \draw (3,-.3) -- (3,0);
 \draw (-0.2,-.3) rectangle (3.7,-1.3);
  \draw (0,-1.3) -- (0,-1.6);
  \draw (.5,-1.3) -- (.5,-1.6);
  \draw (1,-1.3) -- (1,-1.6);
  \draw (3,-1.3) -- (3,-1.6);
  \draw (3.5,-1.3) -- (3.5,-1.6);
  \node at (2,1.2) {$\dots$};
  \node at (2,-.15) {$\dots$};
  \node at (2.15,-1.45) {$\dots$};
\end{tikzpicture}
\end{array}
\end{equation}
where 
$J$ is a finite set and
for any $j \in J$, $M_j$ is a homogeneous element in $\Hom^\bullet_{\DiagBS}(B_{\ux_{i-1} s_i},B_{\ux_{i-1}})$ which factors through a shift of an object $B_{\uz_j}$ where $\uz_j$ is an expression with $\ell(\uz_j) < \ell(\ux_{i-1})=\ell(\ux_{i-1} s_i)-1$.
Performing the other steps of the construction of $\LL_{\uw,\ue}$, we obtain that this morphism can be written as a sum of morphisms of the form
\begin{equation}
\label{eqn:LL-U0}
\begin{array}{c}
\begin{tikzpicture}[scale=0.5]
\draw (0,0) -- (10,0);
\draw (0,0) -- (0,4);
\draw (0,4) -- (10,4);
\draw (10,0) -- (10,4);
\draw (0,2.5) -- (10,2.5);
\draw (5,0) -- (5,2.5);
\draw (0,1.5) -- (5,1.5);
\node at (2.5,0.75) {$L' \cdot \id_{B_i}$};
\node at (7.5,1.25) {$I$};
\node at (2.5,2) {$?$};
\node at (5,-0.35) {$\uw$};
\node at (5,3.25) {$L''$};
\node at (5,4.35) {$\varnothing$};
\end{tikzpicture}
\end{array}
\end{equation}
where $I$ denotes the identity on the object $B_{(s_{i+1}, \cdots, s_{m})}$, $L'$ and $L''$ are homogeneous elements in $\Hom^\bullet_{\DiagBS}(B_{(s_1, \cdots, s_{i-1})}, B_{\ux_{i-1}})$ and $\Hom^\bullet_{\DiagBS}(B_{(\ux_i, s_{i+1}, \cdots, s_m)}, B_{\varnothing})$ respectively, and the question mark takes as values all the morphisms appearing in the right-hand side of~\eqref{eqn:JWrelation}.

As above, the morphism involving $\gamma$ vanishes in $\DasphBS$, hence using Lemma~\ref{lem:LL-order} below we obtain that the image of $\LL_{\uw,\ue}$ in $\DasphBS$ coincides with the image of a linear combination of morphisms $\LL_{\uw,\ue'}$ with $\ue' < \ue$, and the proof of~\eqref{eqn:property-LL} is complete.


Now we prove that, with this choice of light leaves morphisms, the statement of the proposition holds. For this it suffices to prove, by induction on the path dominance order, that the image of $\LL_{\uw,\ue}$ belongs to the span of the images of the morphisms $\LL_{\uw,\ue'}$ where $\ue'$ avoids $W \smallsetminus \fW$.

The path dominance order on subexpressions of $\uw$ expressing
$1$ has a unique minimum, namely $\un{e}_{\min} := 00 \cdots
0$. Moreover, the morphism $\LL_{\uw,\ue_{\min}}$ is a sequence of upper dots. If $\ue_{\min}$ avoids $W \smallsetminus \fW$ there is nothing to prove, and if $\ue_{\min}$ does not avoid $W \smallsetminus \fW$, i.e.~if $s_i \in \Sf$ for some $i \in \{1, \cdots, m\}$, then $\LL_{\uw,\ue_{\min}}$ factors through $B_{s_i}$ (since we can ``pull'' the corresponding dot above all the other dots), hence vanishes in $\DasphBS$.

Now let $\ue$ be arbitrary. If $\ue$ avoids $W \smallsetminus \fW$ there is nothing to prove. And if $\ue$ does not avoid $W \smallsetminus \fW$ then by~\eqref{eqn:property-LL} the image of $\LL_{\uw,\ue}$ is a linear combination of images of light leaves morphisms corresponding to subexpressions which are strictly smaller that $\ue$, and induction allows to conclude.
%
\end{proof}

\begin{lem}
\label{lem:LL-order}
For any $j \in J$, the morphism~\eqref{eqn:LL-U0} (in $\DiagBS$) obtained from the diagram
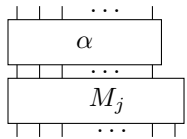
\[
\begin{array}{c}
\begin{tikzpicture}[scale=0.6]
  \draw (-0.2,1) rectangle (3.2,0);
  \node (z) at (1.5,0.5) {$\alpha$};
  \node (z) at (2,-.8) {$M_j$};
  \draw (0,1.3) -- (0,1);
  \draw (.5,1.3) -- (.5,1);
  \draw (1,1.3) -- (1,1);
  \draw (3,1.3) -- (3,1);
  \draw (0,-.3) -- (0,0);
  \draw (.5,-.3) -- (.5,0);
  \draw (1,-.3) -- (1,0);
  \draw (3,-.3) -- (3,0);
 \draw (-0.2,-.3) rectangle (3.7,-1.3);
  \draw (0,-1.3) -- (0,-1.6);
  \draw (.5,-1.3) -- (.5,-1.6);
  \draw (1,-1.3) -- (1,-1.6);
  \draw (3,-1.3) -- (3,-1.6);
  \draw (3.5,-1.3) -- (3.5,-1.6);
  \node at (2,1.2) {$\dots$};
  \node at (2,-.15) {$\dots$};
  \node at (2.15,-1.45) {$\dots$};
\end{tikzpicture}
\end{array}
\]
of~\eqref{eqn:JWrelation}
belongs to the $R$-span of the morphisms $\LL_{\uw,\uf}$ with $\uf<\ue$ in the path dominance order (for any choice of the latter morphisms).
\end{lem}

\begin{proof}
We denote by $N_j$ the morphism under consideration.

Let $\Diag_Q$ be the Karoubi envelope of
  the category obtained from $\DiagBS$ by tensoring all $\Hom$ spaces by $Q$, the fraction field of $R$.
  (This category is denoted $Kar(\cD_Q)$ in~\cite[\S
  5.4]{ew}). In
  $\Diag_Q$ we have objects $Q_{\uw}$ for any expression $\uw$, which satisfy
  \[
   \Hom_{\Diag_Q}(Q_{\uw}, Q_{\uw'}) = \begin{cases}
                                      Q & \text{if $w=w'$;} \\
                                      0 & \text{otherwise,}
                                     \end{cases}
  \]
  see~\cite[Theorem~4.8 and Theorem~5.17]{ew}. (Here, $w$ and $w'$ are the elements of $\Waff$ obtained by multiplying the simple reflections appearing in $\uw$ and $\uw'$ respectively.)
  After choosing, for any $s \in S$, a decomposition of the image of $B_s$ in $\Diag_Q$ as the direct sum $Q_{\varnothing} \oplus Q_s$ (e.g.~as in~\cite[Equation~(5.24)]{ew}), the image of $B_{\uw}$ splits canonically
  into a direct sum of objects $Q_{\uf}$ parametrized by all
  subexpressions $\uf \subset \uw$. (Here by abuse we denote by $Q_{\uf}$ the object $Q_{\ux}$ where $\ux$ is the expression obtained by omitting from $\uw$ the indices $k$ such that $f_k=0$.) Write $i_{\un{f}} \colon Q_{\un{f}}
  \to B_{\uw}$ for the inclusion of this summand in $\Diag_Q$.
For any subexpression $\un{f}$ of $\un{x}$ expressing $1$ we can write
\[
N_j \circ i_{\un{f}} = q_{\un{f}} \cdot \mathrm{can}_{\uf}
\]
where $\mathrm{can}_{\uf}$ is the canonical isomorphism $Q_{\un{f}} \simto Q_\varnothing$ and
$q_{\un{f}} \in Q$.

We claim that
\begin{equation}
\label{eqn:path-dominance-lemma}
q_{\uf}=0 \text{ unless $\uf < \ue$.}
\end{equation}
In fact this follows from the same arguments as in the proof of~\cite[Proposition~6.6]{ew}. Fix some subexpression $\uf$ expressing $1$, and let $x'_0, \cdots, x'_m$ denote the corresponding Bruhat stroll. For any $k \in \{1, \cdots, m\}$, by construction of the light leaves morphisms, $N_j$ factors through a morphism of the form $P_k \cdot \id_{B_{(s_{k+1}, \cdots, s_m)}}$, where $P_k$ is a morphism from $B_{(s_1, \cdots, s_k)}$ to $B_{\ux_k}$. If $\uf_{\leq k}:=(f_1, \cdots, f_k)$, then $P_k$ vanishes on $Q_{\uf_{\leq k}}$ unless $x'_k \leq x_k$; a fortiori, $N_j$ vanishes on $Q_{\uf}$ unless $x'_k \leq x_k$. This implies already that $q_{\uf}=0$ unless $\uf \leq \ue$. Now in the special case $k=i$, we use the fact that $M_j$ factors through an object $B_{\uz_j}$ with $\ell(\uz_j) < \ell(\ux_{i-1})$ to factor $N_j$ through a morphism of the form $P_i \cdot \id_{B_{(s_{i+1}, \cdots, s_m)}}$ where $P_i$ is a morphism from $B_{(s_1, \cdots, s_i)}$ to $B_{\uz_j}$. Then $P_i$ vanishes on $Q_{\uf_{\leq i}}$ unless $\ell(x'_i) \leq \ell(\uz_j) < \ell(x_{i-1})=\ell(x_{i})$, which shows that $q_{\ue}=0$, and the proof of~\eqref{eqn:path-dominance-lemma} is complete.

Now we can conclude: applying \cite[Proposition~6.12]{ew} we can write
\[
N_j = \sum a_{\ue'} \cdot \LL_{\uw, \ue'}
\]
where the sum is over subexpressions $\ue' \subset \ux$ expressing $1$, and $a_{\ue'} \in R$. Comparing the property~\eqref{eqn:path-dominance-lemma} with the ``path dominance upper-triangularity'' of light leaves morphisms, see~\cite[Proposition
6.6]{ew},
we obtain that $a_{\un{e}'} = 0$ unless $\un{e}' <
\un{e}$. The proposition follows.
\end{proof}

\begin{rmk}
\label{rmk:asph-general}
 The construction of the categorified antispherical module is purely diagrammatic, hence makes sense in the more general setting of~\cite{ew}. More specifically, let $(\mathcal{W}, \mathcal{S})$ be a Coxeter system, with Hecke algebra $\mathcal{H}_{(\mathcal{W}, \mathcal{S})}$. Let also $\mathcal{S}_f \subset \mathcal{S}$ be a subset, and $\mathcal{W}_f$ be the (parabolic) subgroup of $\mathcal{W}$ it generates. Then as in~\cite{soergel-comb-tilting} we have an associated antispherical right $\mathcal{H}_{(\mathcal{W}, \mathcal{S})}$-module $\mathcal{N}^f$, with a $\Z[v,v^{-1}]$-basis parametrized by the subset ${}^f \hspace{-1pt} \mathcal{W} \subset \mathcal{W}$ consisting of elements $w$ which are minimal in $\mathcal{W}_f \cdot w$. Let now $\bk$ be an integral domain, and let $\fh$ be a balanced realization of $(\mathcal{W}, \mathcal{S})$ over $\bk$ in the sense of~\cite[Definition~3.1]{ew}. We assume that this realization satisfies Demazure surjectivity. Then we have a diagrammatic category $\DiagBS$ associated with $(\mathcal{W}, \mathcal{S})$ and $\fh$ as in~\cite{ew}, and we can define the graded category $\DasphBS$ by quotienting morphisms in $\DiagBS$ by the $\bk$-span of those morphisms which either are of the form $a \cdot \varphi$ with $a \in \fh^*$, or factor through an object of the form $B_{\uv} \langle n \rangle$ with $\uv$ an expression starting with a simple reflection in $\mathcal{S}_f$. It is easily seen that the split Grothendieck group of $\DasphBS$ identifies with $\mathcal{N}^f$. Moreover, 
an analogue of Proposition~\ref{prop:ASLL} holds (with the same proof): 

\emph{For any expression $\uw$, the light leaves morphisms $\LL_{\uw,\ue}$ in $\Hom^\bullet_{\DiagBS}(B_{\uw}, B_{\varnothing})$ can be chosen in such a way that the $\bk$-module $\bigoplus_{n \in \Z} \Hom_{\DasphBS}(\oB_{\uw}, \oB_{\varnothing} \langle n \rangle)$ is spanned by the images of the morphisms
$\LL_{\uw, \ue}$ where $\ue$ is a subexpression of $\uw$ expressing $1$ and avoiding $\mathcal{W} \smallsetminus {}^f \hspace{-1pt} \mathcal{W}$.}

(Here, $\oB_{\uw}$ is the image of $B_{\uw}$ in $\DasphBS$.)
\end{rmk}





\section{Main conjecture and consequences}
\label{sec:main-conj}

\subsection{Statement of the conjecture}
\label{ss:conjecture}

Now we come back to the setting of Section~\ref{sec:blocks}. In
particular, $\bk$ is an algebraically closed field of characteristic
$p$. In order to be able to also use the results of Section~\ref{sec:Diag}, we assume from now on that $p>h$; see Remark~\ref{rmk:condition-p-h}.

Note that, for $s \in \Saff$, any choice of an adjunction $(\Trans^s,\Trans_s)$ defines morphisms $\id \to \Trans_s \Trans^s$ and $\Trans^s \Trans_s \to \id$, hence morphisms $\id \to \Theta_s$ and $\Theta_s \Theta_s \to \Theta_s$ (by composing the second morphism on the right with $\Trans^s$ and on the left with $\Trans_s$). Similarly, any choice of an adjunction $(\Trans_s, \Trans^s)$ defines morphisms $\Theta_s \to \id$ and $\Theta_s \to \Theta_s \Theta_s$.

Our conjecture roughly states that the category $\DiagBS$ acts on the right on the category $\Rep_0(G)$ via the wall-crossing functors $\Theta_s$. More formally, this conjecture can be stated as follows.

\begin{conj}
\label{conj:main}
There exists, for any $s \in \Saff$, functors
\[
\Trans^s \colon \Rep_0(G) \to \Rep_s(G) \quad \text{and} \quad \Trans_s \colon \Rep_s(G) \to \Rep_0(G)
\]
isomorphic to the translations functors $T_{\lambda_0}^{\mu_s}$ and $T^{\lambda_0}_{\mu_s}$ respectively,
together with adjunctions $(\Trans^s,\Trans_s)$ and $(\Trans_s, \Trans^s)$ and, for any pair $(s,t)$ of distinct elements in $\Saff$ such that $st \in \Waff$ has finite order $m_{st}$, a morphism
\begin{equation}
\label{eqn:2mst-vertex}
\underbrace{\cdots \Theta_t \Theta_s}_{m_{st}} \to \underbrace{\cdots \Theta_s \Theta_t}_{m_{st}},
\end{equation}
such that the assignment defined by $B_s \langle n \rangle \mapsto \Theta_s$ for all $n \in \Z$ and sending
\begin{itemize}
\item
the upper and lower dots to the morphisms
\[
\Theta_s \to \id \quad \text{and} \quad \id \to \Theta_s
\]
defined by the adjunctions $(\Trans_s, \Trans^s)$ and $(\Trans^s,\Trans_s)$ respectively;
\item
the trivalent vertices to the morphisms
\[
\Theta_s \to \Theta_s \Theta_s \quad \text{and} \quad \Theta_s \Theta_s \to \Theta_s
\]
defined by the adjunctions $(\Trans_s, \Trans^s)$ and $(\Trans^s,\Trans_s)$ respectively;
\item
the $2m_{st}$-valent vertex to~\eqref{eqn:2mst-vertex}
\end{itemize}
defines a right action of $\DiagBS$ on $\Rep_0(G)$.
\end{conj}

\begin{rmk}
\label{rmk:main-conj}
\begin{enumerate}
\item
We stated the conjecture in terms of an action on the category $\Rep_0(G)$. But in fact this statement is equivalent to the similar statement where the categories $\Rep_0(G)$ and $\Rep_s(G)$ are replaced by the subcategories $\Tilt(\Rep_0(G))$ and $\Tilt(\Rep_s(G))$. For instance, if we have defined a functor
\[
\widetilde{T}^s \colon \Tilt(\Rep_0(G)) \to \Tilt(\Rep_s(G))
\]
isomorphic to the functor induced by $T_{\lambda_0}^{\mu_s}$, then one can 
consider the diagram
\[
\xymatrix@R=0.5cm@C=1.5cm{
\Kb \Tilt(\Rep_0(G)) \ar[r]^-{\Kb(\widetilde{T}^s)} \ar[d]_-{\wr} & \Kb \Tilt(\Rep_s(G)) \ar[d]^-{\wr} \\
\Db \Rep_0(G) \ar[r] & \Db \Rep_s(G)
}
\]
where the vertical arrows are the canonical equivalences and the lower horizontal arrow is defined in such a way that the diagram commutes. Here $\Kb(\widetilde{T}^s)$ is isomorphic to $\Kb(T_{\lambda_0}^{\mu_s})$, hence the lower arrow is isomorphic to $\Db(T_{\lambda_0}^{\mu_s})$; in particular it is exact, hence it induces a functor $\Trans^s \colon \Rep_0(G) \to \Rep_s(G)$ isomorphic to $T_{\lambda_0}^{\mu_s}$. Similar remarks apply to the functor $\Trans_s$ and to the morphisms between the compositions of these functors.
\item
As explained in Remark~\ref{rmk:morphisms-Dasph}, to define the action we do not need to specify the image of the morphisms $B_{\varnothing} \to B_{\varnothing} \langle \deg(f) \rangle$ associated with homogeneous elements $f \in R$; they are determined by the other morphisms.
\item
In the proof of Theorem~\ref{thm:main} below, we will not use the fact that the morphisms $\Theta_s \to \id$, $\id \to \Theta_s$, $\Theta_s \to \Theta_s \Theta_s$, $\Theta_s \Theta_s \to \Theta_s$ are induced by adjunction, nor that the functors $\Theta_s$ can be written as a composite $\Trans_s \Trans^s$. (We will only use the fact that $\Theta_s$ is isomorphic to a composite $T_{\mu_s}^{\lambda_0} T_{\lambda_0}^{\mu_s}$.) We added these requirements in the conjecture since it is the most natural way to construct these morphisms, and since they automatically imply that certain relations hold (see~\eqref{it:rmks-conj} below). But removing these conditions one can formulate a version of Conjecture~\ref{conj:main} which is only in terms of the regular block $\Rep_0(G)$.
\item
\label{it:rmks-conj}
To prove the conjecture, the task is clear: one needs to choose the functors $\Trans^s$ and $\Trans_s$ and the appropriate adjunctions, together with the morphisms~\eqref{eqn:2mst-vertex}, and then check that the relations of~\cite[Definition~5.2]{ew} are satisfied (with~\cite[(5.1)]{ew} omitted, and ``$f$'' replaced by the composition of a lower dot and an upper dot in~\cite[(5.2)]{ew}). The relations~\cite[(5.3) and (5.4)]{ew} follow immediately from the fact that our morphisms are induced by some adjunctions. So, the first non-trivial relations one needs to consider are relations~\cite[(5.2) and (5.5)]{ew}.
\item
The statement of the conjecture depends on the choice of the weights $\lambda_0$ and $\mu_s$ ($s \in \Saff$). It is likely that the statements for different choices are equivalent, although we were not able to prove this fact. Anyway, at the level of characters (see~\eqref{eqn:char-formula-tilt}), the statement only depends on $\lambda_0$, and it is well known that the formulas for different choices of $\lambda_0$ are all equivalent.
\item
The relations defining $\DiagBS$ are ``local'' in the sense that they involve at most $3$ different simple reflections. However, the conjecture also has some ``global'' flavor since before considering the relations we first need to define the adjunctions and the morphisms~\eqref{eqn:2mst-vertex} ``globally''; they cannot vary when we change the subset of cardinality $\leq 3$ containing the simple reflections under consideration.
\item
\label{it:rmk-left-right}
It will be convenient for us (to simplify the comparison of labelings) to require that the action of $\DiagBS$ is a \emph{right} action. However, using the autoequivalence $\imath$ of~\S\ref{ss:diag-SB}, we see that it is equivalent to construct a right or a left action.
\end{enumerate}
\end{rmk}

\subsection{Tilting modules and antispherical Soergel bimodules}
\label{ss:tilting-antispherical}

The rest of this section is devoted to the study of the implications of Conjecture~\ref{conj:main} on the structure of the category $\Tilt(\Rep_0(G))$. More precisely, we will prove the following theorem (already stated in a slightly different form in~\S\ref{ss:intro-Dasph}).

\begin{thm}
\label{thm:main}
Assume that Conjecture~{\rm \ref{conj:main}} holds. Then there exists an additive functor
\[
\Psi \colon \Dasph \to \Tilt(\Rep_0(G))
\]
and an isomorphism $\zeta \colon \Psi \circ \langle 1 \rangle \simto \Psi$
which satisfy the following properties:
\begin{enumerate}
\item
\label{it:main-fully-faithful}
for any $X,Y$ in $\Dasph$, $\Psi$ and $\zeta$ induce an isomorphism of $\Bbbk$-vector spaces
\[
\Hom^\bullet_{\Dasph}(X, Y) \simto \Hom_{\Rep_0(G)}(\Psi(X), \Psi(Y));
\]
\item
\label{it:main-indec}
for any $w \in \fW$, $\Psi(\overline{B}_w) \cong \Til(w \hdot \lambda_0)$.
\end{enumerate}
\end{thm}

Theorem~\ref{thm:main} can be understood as stating that the category $\Dasph$ is a ``graded version'' of the category $\Tilt(\Rep_0(G))$; see~\cite{soergel-icc} for more details on this point of view.

The next subsections (until~\S\ref{ss:proof-thm}) are devoted to the proof of Theorem~\ref{thm:main}. \emph{In~\S\S{\rm \ref{ss:tilting-antispherical}}--{\rm \ref{ss:proof-thm}} we assume that Conjecture~{\rm \ref{conj:main}} holds, in other words that we have constructed a right action of $\DiagBS$ on $\Rep_0(G)$.} This assumption allows to consider the functor
\[
\widetilde{\Psi} \colon \DiagBS \to \Tilt(\Rep_0(G))
\]
defined by
\[
\widetilde{\Psi}(B) := \Til(\lambda_0) \cdot B.
\]
By construction there exists a natural isomorphism $\widetilde{\Psi} \circ \langle 1 \rangle \cong \widetilde{\Psi}$. Moreover, using the notation from~\S\ref{ss:BStilting}, for any expression $\uw$ we have $\Til(\uw) = \widetilde{\Psi}(B_{\uw})$. Finally, since the category $\Tilt(\Rep_0(G))$ is Karoubian, the functor $\widetilde{\Psi}$ extends to a functor $\Diag \to \Tilt(\Rep_0(G))$, which we will denote similarly.


Recall the notion of a ``rex move'' considered in~\S\ref{ss:rex-moves-DBS}. Later we will need the following property of the image under $\widetilde{\Psi}$ of these morphisms.

\begin{lem}
\label{lem:rex}
Let $\ux$ and $\uy$ be reduced expressions for the same element $w \in \fW$, and consider a rex move $\ux \leadsto \uy$. Let also $\la:=w \hdot \la_0$. The image under $\widetilde{\Psi}$ of the associated morphism $B_{\ux} \to B_{\uy}$ is invertible
in $\Rep_0(G)^{\downarrow \la}$.
\end{lem}

\begin{proof}
Denote by $\phi_{\ux,\uy} \colon B_{\ux} \to B_{\uy}$ the rex move morphism under consideration, and by $\phi_{\uy,\ux} \colon B_{\uy} \to B_{\ux}$ the morphism associated with the ``reversed'' rex move as in~\S\ref{ss:rex-moves-DBS}. Then by Lemma~\ref{lem:rex-move-smaller} we have
\[
\phi_{\uy,\ux} \circ \phi_{\ux,\uy} = \id_{B_{\ux}} + f_{\ux,\uy},
\]
where $f_{\ux,\uy}$ is a sum of morphisms which factor through various objects $B_{\ux'} \langle k \rangle$ where $\ux'$ is an expression with $\ell(\ux')<\ell(\ux)$. Hence the image of $\widetilde{\Psi}(\phi_{\uy,\ux}) \circ \widetilde{\Psi}(\phi_{\ux,\uy})$ in $\End_{\Rep_0(G)^{\downarrow \la}}(\Til(\ux))$ is the identity morphism. Similarly, the image of $\widetilde{\Psi}(\phi_{\ux,\uy}) \circ \widetilde{\Psi}(\phi_{\uy,\ux})$ in $\End_{\Rep_0(G)^{\downarrow \la}}(\Til(\uy))$ is the identity morphism, which shows that the image of $\widetilde{\Psi}(\phi_{\ux,\uy})$ in $\Rep_0(G)^{\downarrow \la}$ is an isomorphism (with inverse the image of $\widetilde{\Psi}(\phi_{\uy,\ux})$), and finishes the proof.
\end{proof}

To conclude this subsection we note the following easy fact.

\begin{lem}
\label{lem:lower-dot-not0}
Let $M \in \Rep_0(G)$, and assume that $\Theta_s(M) \neq 0$. Then the morphism $M \to \Theta_s(M)$ obtained as the image under our action of the ``lower dot'' morphism $B_\varnothing \to B_s \langle 1 \rangle$ (applied to $M$) is nonzero.
\end{lem}

\begin{proof}
If the morphism under consideration vanishes, then so does the image under our action of the morphism
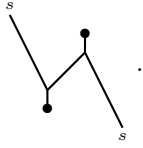
\[
\mathord{
\begin{tikzpicture}[baseline = 0,xscale=-0.5,yscale=0.25]
	\draw[-,thick] (-2,-3) to (-1,1);
	\draw[-,thick] (-1,1) to (0,-1);
	\draw[-,thick] (0,-1) to (1,3);
	\draw[-,thick] (0,-2) to (0,-1);
	\draw[-,thick] (-1,1) to (-1,2);
	\node at (-1,2) {$\bullet$};
	\node at (0,-2) {$\bullet$};
	\node at (-2,-3.5) {\tiny $s$};
	\node at (1,3.5) {\tiny $s$};
\end{tikzpicture} .
}
\]
However, this morphism is equal to $\id_{B_s}$ (by the ``zigzag relation''), hence its image is $\id_{\Theta_s(M)}$, which does not vanish by assumption.
\end{proof}

\begin{rmk}
Of course, if we use the fact that $\Theta_s = \Trans_s \Trans^s$ and that the image of the lower dot morphism is induced by adjunction, then Lemma~\ref{lem:lower-dot-not0} is obvious.
\end{rmk}

\subsection{Surjectivity}
\label{ss:surjectivity}

For two expressions $\ux$ and $\uy$, we denote by
\[
\alpha_{\ux,\uy} \colon \Hom^\bullet_{\DiagBS}(B_{\ux}, B_{\uy}) \to \Hom_{\Rep_0(G)}(\Til(\ux), \Til(\uy))
\]
the morphism induced by $\widetilde{\Psi}$. If $\ux$ is a reduced expression for an element $x \in \fW$, and if $\la:=x \hdot \la_0$, we also denote by
\[
\beta_{\ux,\uy} \colon \Hom^\bullet_{\DiagBS}(B_{\ux}, B_{\uy}) \to \Hom_{\Rep_0(G)^{\downarrow \la}}(\Til(\ux), \Til(\uy))
\]
the composition of $\alpha_{\ux,\uy}$ with the morphism induced by the quotient functor
to $\Rep_0(G)^{\downarrow \la}$.

%

\begin{prop}
\label{prop:morphisms-Phit}
Let $\ux$ and $\uy$ be expressions, and assume that $\ux$ is a reduced expression for some element $x \in \fW$. Then the morphism $\beta_{\ux,\uy}$ is surjective.
\end{prop}

Before proving the proposition in general, we consider some special cases.

\begin{lem}
\label{lem:morphisms-Phit}
Let $\uy$ be a reduced expression for an element $y \in \fW$, and let $s \in S$ be a simple reflection such that $ys>y$ in the Bruhat order and $ys \in \fW$. Then the morphisms $\beta_{\uy s, \uy s}$, $\beta_{\uy s, \uy s s}$, $\beta_{\uy,\uy s}$ and $\beta_{\uy,\uy s s}$ are surjective.
\end{lem}

\begin{proof}
By Lemma~\ref{lem:BsBs}, in $\DiagBS$ we have $B_{\uy s s} \cong B_{\uy s} \langle 1 \rangle \oplus B_{\uy s} \langle -1 \rangle$. We deduce a commutative diagram
\[
\xymatrix{
\Hom^\bullet(B_{\ux}, B_{\uy s s}) \ar[r]^-{\sim} \ar[d]_-{\alpha_{\ux,\uy s s}} & \Hom^{\bullet-1}(B_{\ux}, B_{\uy s}) \oplus \Hom^{\bullet+1}(B_{\ux}, B_{\uy s}) \ar[d]|-{\alpha_{\ux,\uy s} \oplus \alpha_{\ux,\uy s}} \\
\Hom(\Til(\ux), \Til(\uy s s)) \ar[r]^-{\sim} & \Hom(\Til(\ux), \Til(\uy s)) \oplus \Hom(\Til(\ux), \Til(\uy s))
}
\]
where $\ux$ is either $\uy$ or $\uy s$. This remark reduces the claim to the cases of $\beta_{\uy s, \uy s}$ and $\beta_{\uy, \uy s}$. The case of $\beta_{\uy s, \uy s}$ is obvious since
\[
\Hom_{\Rep_0(G)^{\downarrow \lambda}}(\Til(\uy s), \Til(\uy s)) = \bk \cdot \id
\]
(where here $\lambda = ys \hdot \lambda_0$).
So, only the surjectivity of $\beta_{\uy, \uy s}$ remains to be proved.

To simplify notation, we now set $\lambda=y \hdot \lambda_0$, and we fix a non-zero morphism $f \colon \Sta(\lambda) \to \Til(\uy)$. (Such a morphism is unique up to scalar.) Then we have a commutative diagram
\[
\xymatrix{
\Sta(\lambda) \ar[d]_-{f} \ar[r] & \Theta_s \Sta(\lambda) \ar[d]^-{\Theta_s(f)} \\
\Til(\uy) \ar[r] & \Til(\uy s),
}
\]
where both horizontal maps are induced by the image of the ``lower dot'' morphism of $\DiagBS$.
The image of $f$
in $\Rep_0(G)^{\downarrow \lambda}$ is an isomorphism, so that we have
\[
\Hom_{\Rep_0(G)^{\downarrow \lambda}}(\Til(\uy), \Til(\uy s)) \cong \Hom_{\Rep_0(G)^{\downarrow \lambda}}(\Sta(\lambda), \Til(\uy s))
\]
Now since
$\Hom_{\Rep_0(G)^{\downarrow \lambda}}(\Sta(\lambda), \cok(\Theta_s(f)))=0$,
the morphism $\Theta_s(f)$ induces an isomorphism
\[
\Hom_{\Rep_0(G)^{\downarrow \lambda}}(\Sta(\lambda), \Til(\uy s)) \simto \Hom_{\Rep_0(G)^{\downarrow \lambda}}(\Sta(\lambda), \Theta_s \Sta(\lambda)).
\]
Moreover, by Lemma~\ref{lem:Hom-Dta-y-ys}, the vector space $\Hom_{\Rep_0(G)^{\downarrow \lambda}}(\Sta(\lambda), \Theta_s \Sta(\lambda))$ is one-dimensional, and the image of the lower dot morphism in this space coincides with the image of the lower dot morphism in $\Hom_{\Rep_0(G)^{\downarrow \lambda}}(\Til(\uy), \Til(\uy s))$ under the composed isomorphism
\[
\Hom_{\Rep_0(G)^{\downarrow \lambda}}(\Til(\uy), \Til(\uy s)) \cong \Hom_{\Rep_0(G)^{\downarrow \lambda}}(\Sta(\lambda), \Theta_s \Sta(\lambda)).
\]
Hence to conclude it suffices to prove that the image of the ``lower dot'' morphism in $\Hom_{\Rep_0(G)^{\downarrow \lambda}}(\Sta(\lambda), \Theta_s \Sta(\lambda))$ does not vanish.
This fact follows from Lemma~\ref{lem:Hom-Dta-y-ys} and Lemma~\ref{lem:lower-dot-not0}.
\end{proof}

\begin{proof}[Proof of Proposition~{\rm \ref{prop:morphisms-Phit}}]
We will prove by induction on $\ell(\uy)$ that the statement holds for all reduced expressions $\ux$.
First, assume that $\ell(\uy)=0$. Then $\Til(\uy)=\Til(\varnothing)=\Til(\lambda_0)$, and 
\[
\Hom_{\Rep_0(G)^{\downarrow \lambda}}(\Til(\ux), \Til(\varnothing)) = 0
\]
unless $\ux=\varnothing$, in which case the claim is obvious since
\[
\Hom_{\Rep_0(G)^{\downarrow \lambda_0}}(\Til(\varnothing), \Til(\varnothing)) = \Bbbk \cdot \id.
\]

Now let $\uy$ be an expression such that $\ell(\uy)>0$. Write $\uy=\uv s$ where $s \in S$, so that $\Til(\uy) = \Theta_s \Til(\uv)$, and assume that the result is known for $\uv$. As in the statement, let also $\ux$ be a reduced expression for an element $x \in \fW$, and let $\lambda:=x \hdot \lambda_0$. We distinguish three cases.

\emph{Case~1: $\lambda^s \notin \bX^+$.} In this case we have
\[
\Hom_{\Rep_0(G)^{\downarrow \lambda}}(\Til(\ux), \Til(\uy)) = 0
\]
since $(\Til(\uy): \Cos(\la)) = (\Theta_s \Til(\uv) : \Cos(\la))=0$ by Lemma~\ref{lem:translation-Cos},
and there is nothing to prove.

\emph{Case~2: $\lambda^s \in \bX^+$ and $\lambda^s \uparrow \lambda$}.
In this case $xs<x$ in the Bruhat order, so that $x$ has a reduced expression ending with $s$, and $xs \in \fW$. Using Lemma~\ref{lem:rex}, we can assume that $\ux=\uu s$ for some word $\uu$ which is a reduced expression for $xs$. By induction, there exists a family $(f_i)_{i \in I}$ of elements in the image of $\alpha_{\ux,\uv}$ whose image spans $\Hom_{\Rep_0(G)^{\downarrow \la}}(\Til(\ux), \Til(\uv))$, and a family $(g_j)_{j \in J}$ of elements in the image of $\alpha_{\uu,\uv}$ whose image spans $\Hom_{\Rep_0(G)^{\downarrow \la^s}}(\Til(\uu), \Til(\uv))$. Then by Proposition~\ref{prop:morphisms-BS}\eqref{it:morphisms-BS-2} there exists morphisms $f_i' \colon \Til(\ux) \to \Til(\ux s)$ and $g_j' \colon \Til(\ux) \to \Til(\uu s) = \Til(\ux)$ such that the images of the compositions
\begin{equation}
\label{eqn:beta-surjective}
\Theta_s(f_i) \circ f'_i \quad \text{and} \quad \Theta_s(g_j) \circ g_j'
\end{equation}
span $\Hom_{\Rep_0(G)^{\downarrow \la}}(\Til(\ux), \Til(\uy))$.  By Lemma~\ref{lem:morphisms-Phit} (applied to the reduced expression $\uu$) the morphisms $\beta_{\ux,\ux}$ and $\beta_{\ux,\ux s}$ are surjective; hence we can assume that the morphisms $f_i'$ are in the image of $\alpha_{\ux,\ux s}$, and that the morphisms $g_j'$ are in the image of $\alpha_{\ux,\ux}$. Then all the morphisms in~\eqref{eqn:beta-surjective} are in the image of $\alpha_{\ux,\uy}$, and the proof is complete in this case.



\emph{Case~3: $\lambda \uparrow \lambda^s$}. In this case $xs>x$ in the Bruhat order and $xs \in \fW$.
The proof is similar to (and slightly simpler than) the proof of Case~2. In fact, by induction
there exists a family $(f_i)_{i \in I}$ of elements in the image of
$\alpha_{\ux,\uv}$ whose image spans $\Hom_{\Rep_0(G)^{\downarrow
    \la}}(\Til(\ux), \Til(\uv))$, and a family $(g_j)_{j \in J}$ of
elements in the image of $\alpha_{\ux s,\uv}$ whose image spans
$\Hom_{\Rep_0(G)^{\downarrow \la^s}}(\Til(\ux s), \Til(\uv))$. Then by
Proposition~\ref{prop:morphisms-BS}\eqref{it:morphisms-BS-1} there
exists morphisms $f_i' \colon \Til(\ux) \to \Til(\ux s)$ and
$g_j' \colon \Til(\ux) \to \Til(\ux s s)$ such
that the images of the compositions
\begin{equation}
\label{eqn:beta-surjective-2}
\Theta_s(f_i) \circ f'_i \quad \text{and} \quad \Theta_s(g_j) \circ g_j'
\end{equation}
span $\Hom_{\Rep_0(G)^{\downarrow \la}}(\Til(\ux), \Til(\uy))$. By Lemma~\ref{lem:morphisms-Phit} (applied to the reduced expression $\ux$) the morphism $\beta_{\ux, \ux s}$ and $\beta_{\ux, \ux s s}$ are surjective; hence we can assume that the morphisms $f_i'$ are in the image of $\alpha_{\ux, \ux s}$ and that the morphisms $g_j'$ are in the image of $\alpha_{\ux, \ux s s}$. Then all the morphisms in~\eqref{eqn:beta-surjective-2} are in the image of $\alpha_{\ux,\uy}$, and the proof is complete in this case also.
\end{proof}

\subsection{Dimensions of morphism spaces}
\label{ss:dim}

When one specializes the parameter $v$ of the Hecke algebra $\Haff$ to $1$, the Hecke algebra specializes to the group algebra $\Z[\Waff]$, and the antispherical module $\Masph$ specializes to the antispherical right module
\[
\mathsf{M}^{\asph} := \Z_\varepsilon \otimes_{\Z[\Wf]} \Z[\Waff]
\]
of $\Z[\Waff]$. (Here $\Z_{\varepsilon}$ is a the rank one free $\Z$-module where $\Wf$ acts via the sign character $\varepsilon$.) This $\Z$-module has a natural basis parametrized by $\fW$: for $w \in \Waff$ we denote by $N'_w$ the element $1 \otimes w$. For $\uw=s_1 \cdots s_r$ an expression, we also set
\[
\uN'_{\uw} = 1 \otimes (1+s_1) \cdots (1+s_r).
\]

\begin{lem}
\label{lem:dim-tilt}
For any expression $\uw$, the integer
\[
(\Til(\uw) : \Cos(\lambda_0))=\dim_{\Bbbk} \left( \Hom_{\Rep_0(G)}(\Til(\lambda_0), \Til(\uw)) \right)
\]
is equal to the coefficient of $\uN'_{\uw}$ on $N'_{1}$.
\end{lem}

\begin{proof}
For any $s \in \Saff$ and $w \in \fW$, using~\eqref{eqn:not-in-fW} we see that
\[
N'_w \cdot (1 + s) =
\begin{cases}
N'_w + N'_{ws} & \text{if $ws \in \fW$,} \\
0 & \text{otherwise.}
\end{cases}
\]
These formulas show that the isomorphism of $\Z$-modules
\[
\mathsf{M}^{\asph} \simto [\Rep_0(G)]
\]
sending $N'_w$ to $[\Cos(w \hdot \lambda_0)]$ intertwines right multiplication by $(1+s)$ with the morphism $[\Theta_s]$. The lemma follows.
\end{proof}

\begin{lem}
\label{lem:dim-diag}
For any expression $\uw$, the dimension of the $\Bbbk$-vector space
\[
\Hom_{\Dasph}^\bullet(\oB_\varnothing, \oB_{\uw})
\]
is at most the coefficient of $\uN'_{\uw}$ on $N'_{1}$.
\end{lem}

\begin{proof}
Using the antiequivalence $\tau$ of~\S\ref{ss:diag-SB}, it is
equivalent to prove the similar claim for
$\Hom_{\Dasph}^\bullet(\overline{B}_{\uw}, \overline{B}_\varnothing)$.
By Proposition~\ref{prop:ASLL}, the dimension of
$\Hom^\bullet_{\Dasph}(\overline{B}_{\un{w}}, \overline{B}_{\varnothing})$ is at most
the cardinality of $\{ \un{e} \subset \un{w} \mid \un{w}^{\un{e}} = 1 \text{ and
  $\un{e}$ avoids $W \smallsetminus \fW$} \}.$
By
Lemma~\ref{lem:number-subexpr-avoids} this number is the
coefficient of $\uN'_{\uw}$ on $N'_{1}$, and the lemma follows. \end{proof}
\subsection{Proof of Theorem~\ref{thm:main}}
\label{ss:proof-thm}

We can finally give the proof of Theorem~\ref{thm:main}.
For any $s \in \Sf$ we have
\[
\widetilde{\Psi}(B_s) = \Til(\lambda_0) \cdot B_s = \Theta_s(\Til(\lambda_0)) = 0.
\]
Since any object $B_w$ with $w \notin \fW$ is a direct summand in a shift of an object of the form $B_s \cdot X$ with $s \in \Sf$, it follows that $\widetilde{\Psi}$ factors through a functor
\[
\Psi \colon \Dasph \to \Tilt(\Rep_0(G)).
\]

First, let us prove~\eqref{it:main-fully-faithful}.
For this it suffices to prove that for any expressions $\ux$ and $\uw$ the functor $\Psi$ induces an isomorphism
\[
\Hom_{\Dasph}^\bullet(\oB_{\ux}, \oB_{\uw}) \simto \Hom_{\Rep_0(G)}(\Til(\ux), \Til(\uw)).
\]
We prove this claim by induction on $\ell(\ux)$. If $\ell(\ux)=0$, then $\ux=\varnothing$. In this case, Proposition~\ref{prop:morphisms-Phit} implies that the morphism induced by $\Psi$ is surjective. Comparing the dimensions using Lemmas~\ref{lem:dim-tilt} and~\ref{lem:dim-diag}, it follows that this morphism is an isomorphism, and the desired claim is proved.

Now assume that $\ell(\ux)>0$, and write
$\ux=\uy s$ with $s \in \Saff$. Consider the morphisms
\[
B_{\varnothing} \to B_s B_s \quad \text{and} \quad B_s B_s \to B_{\varnothing}
\]
obtained by composing the appropriate trivalent vertex with the appropriate dot morphism, and their images
\[
\id \to \Theta_s \Theta_s \quad \text{and} \quad \Theta_s \Theta_s \to \id.
\]
These morphisms satisfy the zigzag relations, hence are adjunction morphisms for some adjunctions $((-) \cdot B_s, (-) \cdot B_s)$ and $(\Theta_s, \Theta_s)$ respectively. Using these adjunctions we obtain 
the following commutative diagram, where vertical morphisms are induced by $\Psi$:
\[
\xymatrix{
\Hom_{\Dasph}^\bullet(\oB_{\uy s}, \oB_{\uw}) \ar[d] \ar[r]^-{\sim} & \Hom_{\Dasph}^\bullet(\oB_{\uy}, \oB_{\uw s}) \ar[d] \\
\Hom_{\Rep_0(G)}(\Til(\uy s), \Til(\uw)) \ar[r]^-{\sim} & \Hom_{\Rep_0(G)}(\Til(\uy), \Til(\uw s)).
}
\]
By induction we know that the right vertical arrow is an isomorphism, and we deduce that the left vertical arrow is also an isomorphism.

%

Now we consider~\eqref{it:main-indec}. Clearly, it is enough to prove that $\Psi(\oB_w)$ is indecomposable for any $w \in \fW$. For this we observe that by~\eqref{it:main-fully-faithful} we have an algebra isomorphism
\[
\Hom_{\Rep_0(G)}(\Psi(\oB_w), \Psi(\oB_w)) \cong \Hom^\bullet_{\Dasph}(\oB_w, \oB_w).
\]
Since $\oB_w$ is indecomposable in $\Dasph$, the subring
\[
\Hom_{\Dasph}(\oB_w, \oB_w) \subset \Hom^\bullet_{\Dasph}(\oB_w, \oB_w)
\]
is local. By~\cite[Theorem~3.1]{gg}, we deduce that the ring $\Hom^\bullet_{\Dasph}(\oB_w, \oB_w)$ is local, and then that the ring $\Hom_{\Rep_0(G)}(\Psi(\oB_w), \Psi(\oB_w))$ is local. This implies that $\Psi(\oB_w)$ is indecomposable, and finishes the proof.

\subsection{Graded form of $\Rep_0(G)$}
\label{ss:forms_of_rep0}

In this section we preserve the setup of~\S\ref{ss:conjecture}, and explain how Conjecture~\ref{conj:main} implies that the block $\Rep_0(G)$ admits a grading in the sense of~\cite[Definition~4.3.1]{bgs}.
\emph{In this subsection, we assume that Conjecture~{\rm \ref{conj:main}} holds.}

From the fundamental vanishing \eqref{eqn:morph-Sta-Cos} we deduce the vanishing
\begin{equation} \label{eq:tiltvanish}
\Ext^i_{\Rep_0(G)}(\Til(\lambda),\Til(\mu)) = 0 \quad \text{for all $i > 0$ and
  $\lambda, \mu \in \bX_0^+$.}
\end{equation}
Now $\Db(\Rep_0(G))$ is generated by $\Tilt(\Rep_0(G))$. These considerations and
Be{\u\i}lin\-son's lemma imply that the inclusion functor provides an equivalence
\[
\Kb(\Tilt(\Rep_0(G))) \simto \Db(\Rep_0(G)).
\]
Then, combining the above equivalence with Theorem \ref{thm:main}, we obtain the following result.

\begin{cor}
\label{cor:rep0}
There exists
a triangulated functor
\[
\Psi \colon \Kb(\Dasph) \to \Db(\Rep_0(G))
\]
and an isomorphism $\zeta \colon \Psi \circ \langle 1 \rangle \simto \Psi$
which satisfy the following properties:
\begin{enumerate}
\item
\label{it:cor-fully-faithful}
for any $X,Y$ in $\Kb(\Dasph)$, $\Psi$ and $\zeta$ induce an isomorphism of $\Bbbk$-vector spaces
\[
\bigoplus_{m \in \Z}
\Hom_{\Kb(\Dasph)}(X, Y \langle m \rangle) \simto \Hom_{\Db(\Rep_0(G))}(\Psi(X), \Psi(Y));
\]
\item
\label{it:cor-indec}
for any $w \in \fW$, $\Psi(\overline{B}_w) \cong \Til(w \hdot \lambda_0)$.
\end{enumerate}
\end{cor}

In particular, $\Kb(\Dasph)$ gives a ``graded version'' of
$\Db(\Rep_0(G))$. The following proposition implies that the
standard and costandard objects may be lifted to this graded version.

\begin{prop} \label{prop:stdcostd}
 For all $x \in \fW$ there exist objects $\widetilde{\Sta}_x, 
  \widetilde{\Cos}_x \in \Kb(\Dasph)$ and maps
\[
i_x \colon \widetilde{\Sta}_x \to \overline{B}_x \quad \text{and} \quad p_x :
\overline{B}_x \to \widetilde{\Cos}_x
\]
such that:
\begin{enumerate}
\item
we have isomorphisms $\Psi(\widetilde{\Sta}_x) \cong \Sta(x
  \hdot \lambda_0)$
  and $\Psi(\widetilde{\Cos}_x) \cong \Cos(x\hdot \lambda_0)$;
\item
the map $\Psi(i_x)$, resp.~$\Psi(p_x)$, identifies, up to a nonzero scalar, with
  the unique nonzero (and injective) morphism $\Delta(x\hdot \lambda_0) \into \Til(x\hdot \lambda_0)$, resp.~with the unique nonzero (and surjective) morphism
  $\Til(x\hdot \lambda_0) \onto \nabla(x\hdot \lambda_0)$.
\end{enumerate}
\end{prop}

\begin{rmk}
It can be easily checked that the objects $\widetilde{\Sta}_x$ and $\widetilde{\Cos}_x$ as in Proposition~\ref{prop:stdcostd} are unique up to isomorphism.
\end{rmk}

Before giving the proof of Proposition~\ref{prop:stdcostd} we note the following fact.

\begin{lem}
\label{lem:delnabseq}
Let $x \in \fW$ and $s \in \Saff$ be such that $xs \in \fW$ and $x<xs$ in the Bruhat order.
\begin{enumerate}
\item
\label{lem:delnabseq-1}
The morphism $\Cos(x \hdot \lambda_0) \to \Theta_s(\Cos(x \bullet \lambda_0))$ given by the image under the action of Conjecture~{\rm \ref{conj:main}} of the lower dot morphism $B_\varnothing \to B_s \langle 1 \rangle$ (applied to $\Cos(x \hdot \lambda_0)$) is injective, and its cokernel is isomorphic to $\Cos(xs \hdot \lambda_0)$.
\item
\label{lem:delnabseq-2}
The morphism $\Theta_s(\Sta(x \hdot \lambda_0)) \to \Sta(x \bullet \lambda_0)$ given by the image under the action of Conjecture~{\rm \ref{conj:main}} of the upper dot morphism $B_s \to B_\varnothing \langle 1 \rangle$ (applied to $\Sta(x \hdot \lambda_0)$) is surjective, and its kernel is isomorphic to $\Sta(xs \hdot \lambda_0)$.
\end{enumerate}
\end{lem}

\begin{proof}
We prove~\eqref{lem:delnabseq-1}; the proof of~\eqref{lem:delnabseq-2} is similar. By Lemma~\ref{lem:lower-dot-not0}, the morphism under consideration is nonzero. And by Lemma~\ref{lem:Hom-Dta-y-ys}, there exists only one such nonzero morphism, up to scalar. Hence this morphism must coincide (up to scalar) with the first morphism of the exact sequence of~\cite[Proposition~II.7.19(a)]{jantzen}, and the claim follows.
\end{proof}

\begin{proof}[Proof of Proposition~{\rm \ref{prop:stdcostd}}]
  We prove the proposition by induction on $\ell(x)$. If $\ell(x)=0$ then $x=1$, and one can take $\widetilde{\Delta}_1 =
  \widetilde{\nabla}_1 = \overline{B}_1$ and $i_{1} = p_{1} = \id_{\overline{B}_1}$.

Now fix $x \in \fW$, and assume that we have proved the proposition for
all $y \in \fW$ with $\ell(y) < \ell(x)$. Certainly we can find $s \in S$ with $xs
< x$ in the Bruhat order; then $xs \in \fW$, see~\eqref{eqn:not-in-fW}. We define $\widetilde{\Delta}_x$ and $\widetilde{\nabla}_x$ as
follows:
\[
  \widetilde{\Delta}_{x} := \cone(\widetilde{\Delta}_{xs}B_s \to
  \widetilde{\Delta}_{xs}\langle 1 \rangle)[-1], \quad
  \widetilde{\nabla}_{x} := \cone( \widetilde{\nabla}_{xs}\langle -1 \rangle \to \widetilde{\nabla}_{xs}B_s),
\]
where the morphisms are induced by the upper and lower dot maps
respectively. 
It follows from induction and Lemma \ref{lem:delnabseq} below that
$\Psi(\widetilde{\Delta}_{x}) \cong \Delta(x \hdot \lambda_0)$ and
$\Psi(\widetilde{\nabla}_{x}) \cong \nabla(x \hdot \lambda_0)$.

Now consider the compositions
\[
  \widetilde{\Delta}_{xs}B_s \xrightarrow{i_{xs}B_s}
  \overline{B}_{xs}B_s
  \to \overline{B_{x}} \quad \text{and} \quad
\overline{B}_{x} \to \overline{B}_{xs}B_s \xrightarrow{p_{xs} B_s} \widetilde{\nabla}_{xs} B_s,
\]
where the first (resp.~second) unlabelled map is a choice of
projection onto (resp.~inclusion of) the indecomposable summand
$\overline{B}_x$ of $\overline{B}_{xs} B_s$. These maps induce maps
\[
i_x \colon \widetilde{\Delta}_x \to \overline{B}_x \quad \text{and} \quad
p_x \colon  \overline{B}_x \to \widetilde{\nabla}_x.
\]
It is easily seen that $\Psi(i_x)$ and $\Psi(p_x)$ are nonzero, and the proposition follows.
\end{proof}

Let $(D^{\le 0}, D^{\ge 0})$ denote the standard t-structure on
$\Db(\Rep_0(G))$ with heart $\Rep_0(G)$. Then it is standard (and easy
to check from the axioms of a highest weight category)
that we have
\begin{gather*}
D^{\le 0} = \langle \Sta(x \hdot \lambda_0)[n] \mid
x \in \fW, \, n \in \Z, \, n \ge 0 \rangle_\ext, \\
D^{\ge 0} = \langle \Cos(x \hdot \lambda_0)[n] \mid
x \in \fW, \, n \in \Z, \, n \le 0 \rangle_\ext,
\end{gather*}
where $\langle - \rangle_\ext$ denotes the smallest full additive subcategory of
$\Db(\Rep_0(G))$ containing the specified objects and closed under
extensions.\footnote{Recall that $B$ is an extension of $A, C$ if there
  exists a distinguished triangle $A \to B \to C \triright$.}

We now explain how to lift this t-structure to $\Kb(\Dasph)$. From
Corollary \ref{cor:rep0} and the fact that $\Rep_0(G)$ is a highest weight category we deduce that for $x,y \in \fW$ and $m,n \in \Z$ we have
\begin{gather}
\label{eqn:exceptional-seq-1}
\Hom_{\Kb(\Dasph)}(\widetilde{\Cos}_x, \widetilde{\Cos}_y \langle m \rangle [n])=0 \quad \text{unless $y \hdot \lambda_0 \uparrow x \hdot \lambda_0$,}
\\
\label{eqn:exceptional-seq-2}
\Hom_{\Kb(\Dasph)}(\widetilde{\Cos}_x, \widetilde{\Cos}_x \langle m \rangle [n])= \begin{cases}
\Bbbk & \text{if $m=n=0$;} \\
0 & \text{otherwise,}
\end{cases}
\\
\label{eqn:exceptional-seq-3}
\Hom_{\Kb(\Dasph)}(\widetilde{\Cos}_x, \widetilde{\Cos}_y \langle m \rangle [n])=\Hom_{\Kb(\Dasph)}(\widetilde{\Sta}_x, \widetilde{\Sta}_y \langle m \rangle [n])=0 \text{ if $n<0$.}
\end{gather}
Similarly, using~\eqref{eqn:morph-Sta-Cos} we obtain that
\begin{equation}
  \label{eqn:morph-Sta-Cos-graded}
\Hom_{\Kb(\Dasph)}(\widetilde{\Sta}_x, \widetilde{\Cos}_y\langle m \rangle [n]) = \begin{cases}
\Bbbk & \text{if $x=y$ and $m = n=0$;} \\
0 & \text{otherwise.}
\end{cases}
\end{equation}

If we choose a total order on $\fW$ which refines the order defined by $y \preceq x$ iff $y \hdot \lambda_0 \uparrow x \hdot \lambda_0$ and such that $\fW$, endowed with this order, is isomorphic to $\Z_{\geq 0}$ (for the standard order), then properties~\eqref{eqn:exceptional-seq-1} and~\eqref{eqn:exceptional-seq-2} show that the objects $\widetilde{\Cos}_x$ form a graded exceptional sequence in $\Kb(\Dasph)$ with respect to this order (in the sense of~\cite[\S 2.1.5]{be}), and~\eqref{eqn:morph-Sta-Cos-graded} shows that the objects $\widetilde{\Sta}_x$ form the dual graded exceptional sequence.
Hence, by~\cite[Proposition~4]{be}, if we set
\begin{gather*}
K^{\le 0} := \langle \widetilde{\Sta}_x\langle m \rangle[n] \mid
x \in \fW, \, m, n \in \Z, \, n \ge 0 \rangle_\ext, \\
K^{\ge 0} := \langle \widetilde{\Cos}_x\langle m \rangle[n] \mid
x \in \fW, \, m, n \in \Z, \, n \le 0 \rangle_\ext,
\end{gather*}
then $(K^{\le 0}, K^{\ge 0})$ is a bounded t-structure on
$\Kb(\Dasph)$; we denote its heart by $\widetilde{\Rep}_0(G)$. The shift functor $\langle 1 \rangle$ obviously stabilizes this subcategory.

Property~\eqref{eqn:exceptional-seq-3} and~\cite[Proposition~4(c)]{be} show that the objects $\widetilde{\Sta}_x$ and $\widetilde{\Cos}_x$ belong to $\widetilde{\Rep}_0(G)$. Using this, standard arguments show that the category $\widetilde{\Rep}_0(G)$ is graded quasi-hereditary in the sense of~\cite[Definition~2.1]{modrap3} (with the obvious replacement of condition~(1) of \emph{loc.}~\emph{cit.}~by the condition of Definition~\ref{defn:hwcat}),
and that the realization functor provides an equivalence of triangulated categories
\[
\Db \widetilde{\Rep}_0(G) \simto \Kb(\Dasph).
\]

The following theorem is a standard consequence of these remarks and Corollary~\ref{cor:rep0}.

\begin{thm}
\label{thm:rep0grading}
There exists an exact functor
\[
\Psi \colon \widetilde{\Rep}_0(G) \to \Rep_0(G)
\]
of abelian categories and an isomorphism $\zeta \colon \Psi \circ \langle 1 \rangle \simto \Psi$
which satisfy the following properties:
\begin{enumerate}
\item
\label{it:cor-fully-faithful-blah-rep0}
for any $X,Y$ in $\widetilde{\Rep}_0(G)$ and $n \in \Z$, $\Psi$ and $\zeta$ induce an isomorphism of $\Bbbk$-vector spaces
\[
\bigoplus_{m \in \Z}
\Ext^n_{\widetilde{\Rep}_0(G)}(X, Y \langle m \rangle) \simto \Ext^n_{\Rep_0(G)}(\Psi(X), \Psi(Y));
\]
\item
\label{it:lifts}
for any $x \in \fW$, there exists objects $\overline{B}_x$,
$\widetilde{\Sta}_x$, $\widetilde{\Cos}_x$, $\widetilde{\Sim}_x$ in
$\widetilde{\Rep}_0(G)$ such that $\widetilde{\Sim}_x$ is simple, as well
as a diagram of morphisms
\[
\xymatrix@R=0.2cm@C=1.5cm{
& \oB_x \ar@{->>}[rd] & \\
\widetilde{\Sta}_x \ar@{^{(}->}[ru] \ar@{->>}[rd] & & \widetilde{\Cos}_x \\
& \widetilde{\Sim}_x \ar@{^{(}->}[ru]
}
\]
which after applying $\Psi$ becomes the diagram:
\[
\xymatrix@R=0.2cm@C=1.5cm{
& \Til(x \hdot \lambda_0) \ar@{->>}[rd] & \\
\Sta(x \hdot \lambda_0) \ar@{^{(}->}[ru] \ar@{->>}[rd] & & \Cos(x \hdot \lambda_0). \\
& \Sim(x \hdot \lambda_0) \ar@{^{(}->}[ru]
}
\]
\end{enumerate}
\end{thm}

\begin{rmk}
Theorem~\ref{thm:rep0grading} can be rephrased as stating that $\widetilde{\Rep}_0(G)$ is a grading on $\Rep_0(G)$ in the sense of~\cite[Definition~4.3.1]{bgs} and that
the standard, costandard and tilting modules, as well
as the canonical maps between them, all lift
to this grading.
\end{rmk}

%

\subsection{Integral form of $\Tilt(\Rep_0(G))$}
\label{ss:intforms}

Throughout this section we have assumed that $\Bbbk$ is a field. In this subsection we
describe how one can use an integral form of $\DiagBS$ to define a
graded integral form of $\Tilt(\Rep_0(G))$. Since we will have to consider the categories $\DiagBS$ and $\DasphBS$ over different rings of coefficients, from now on we indicate this ring in the notation. In particular, if $\bk$ is as in~\S\ref{ss:conjecture}, we can consider the categories $\DiagBSk$ and $\DasphBSk$ over $\bk$.

Recall that the connected reductive group $G$ may be obtained via extension of scalars from of reductive group scheme $G_\Z$ over $\Z$. Then for any field $\Bbbk$ as in~\S\ref{ss:conjecture} we can consider the extension of scalars $G_\bk$ of $G_\Z$ to $\bk$, and all the combinatorial data attached to the group in Sections~\ref{sec:blocks} and~\ref{sec:Diag}, in particular the groups $\Waff$ and $\Wf$ 
are independent of the choice of $\bk$.

We start with the following observation, which is clear from the proof of Theorem~\ref{thm:main} (see~\S\ref{ss:proof-thm}).

\begin{lem}
\label{lem:morph-Dasp-LL-k}
Let $\bk$ be as in~\S{\rm \ref{ss:conjecture}}, and assume that Conjecture~{\rm \ref{conj:main}} holds for $G_\bk$ and some choice of $\lambda_0$. Then for any expression $\uw$, the elements considered in Proposition~{\rm \ref{prop:ASLL}} actually form a basis of $\Hom_{\DasphBSk}(\oB_{\uw}, \oB_{\varnothing})$.
\end{lem}

%
%

We denote by $k_0'$ the product of the prime numbers which divide the determinant of the Cartan matrix of $\Phi$, and set $k_0=k_0'$ if condition~\eqref{eqn:assumption-Dem-surjectivity} holds for $\K=\Z[\frac{1}{k_0'}]$, and $k_0=2k_0'$ otherwise. Then we set $\mathfrak{R}:=\Z[\frac{1}{k_0}]$. This ring satisfies both~\eqref{eqn:assumption-pairing} and~\eqref{eqn:assumption-Dem-surjectivity}, so that we can consider the categories $\mathcal{D}_{\mathrm{BS},\mathfrak{R}}$ and $\mathcal{D}_{\mathrm{BS},\mathfrak{R}}^{\mathrm{asph}}$ over $\mathfrak{R}$. Moreover, for any field $\bk$ as in~\S\ref{ss:conjecture} there exists a unique morphism $\mathfrak{R} \to \bk$. (Note that unless $\Phi$ has a component of type $\mathbf{A}$, we have $k'_0 \leq 3$.)

\begin{lem}
\label{lem:morphisms-Dasph-integral}
Assume that there exists infinitely many primes $p>h$ such that Conjecture~{\rm \ref{conj:main}} holds for $G_{\overline{\mathbb{F}}_p}$ and some choice of $\lambda_0$.

For any expressions $\uw$ and $\uv$ and any $m \in \Z$, the $\mathfrak{R}$-module
\[
\Hom_{\mathcal{D}_{\mathrm{BS},\mathfrak{R}}^{\mathrm{asph}}}(\oB_{\uw}\langle n \rangle, \oB_{\uv} \langle m \rangle)
\]
is free of finite rank over $\mathfrak{R}$. Moreover, for any field $\bk$ as in~\S{\rm \ref{ss:conjecture}} such that Conjecture~{\rm \ref{conj:main}} holds for $G_\bk$ and some choice of $\lambda_0$, the natural morphism
\begin{equation}
\label{eqn:Hom-Dasph-basechange}
\bk \otimes_{\mathfrak{R}} \Hom_{\mathcal{D}_{\mathrm{BS},\mathfrak{R}}^{\mathrm{asph}}}(\oB_{\uw}\langle n \rangle, \oB_{\uv} \langle m \rangle) \to \Hom_{\DasphBSk}(\oB_{\uw}\langle n \rangle, \oB_{\uv} \langle m \rangle)
\end{equation}
is an isomorphism.
\end{lem}

\begin{proof}
Using adjunction as in~\S\ref{ss:proof-thm}, we can assume that $\uv=\varnothing$. In this case, consider the ``light leaves basis" considered in the proof of Proposition~\ref{prop:ASLL} for the coefficients $\mathfrak{R}$. By~\eqref{eqn:property-LL}, adding if necessary to the morphisms $\LL_{\uw,\ue}$ where $\ue$ does not avoid $\Waff \smallsetminus \fW$ some linear combinations of morphisms $\LL_{\uw,\ue'}$ with $\ue'<\ue$ in the path dominance order, one can construct a basis $(\varphi_{\ue} : \ue \subset \uw, \, \uw^{\ue}=1)$ for the left $R$-module $\Hom^\bullet_{\mathcal{D}_{\mathrm{BS},\mathfrak{R}}}(B_{\uw},B_\varnothing)$ such that $\varphi_{\ue}$ vanishes in $\mathcal{D}_{\mathrm{BS},\mathfrak{R}}^{\mathrm{asph}}$ if $\ue$ does not avoid $\Waff \smallsetminus \fW$ and such that the images of the morphism $\varphi_{\ue}$ for $\ue$ avoiding $\Waff \smallsetminus \fW$ span $\Hom_{\mathcal{D}_{\mathrm{BS},\mathfrak{R}}^{\mathrm{asph}}}(\oB_{\uw}\langle n \rangle, \oB_{\uv} \langle m \rangle)$ as an $\mathfrak{R}$-module. Then if we choose an $\mathfrak{R}$-basis $(f_i : i \in I)$ for the ideal of $R$ generated by $\fh$, then elements $\varphi_{\ue}$ and $f_i \cdot \varphi_{\ue}$ with $\ue$ as above and $i \in I$ form an $\mathfrak{R}$-basis of $\Hom^\bullet_{\mathcal{D}_{\mathrm{BS},\mathfrak{R}}}(B_{\uw},B_\varnothing)$.

We claim that, under our assumption, the kernel of the morphism
\begin{equation}
\label{eqn:Hom-Diag-Dasph}
\Hom^\bullet_{\mathcal{D}_{\mathrm{BS},\mathfrak{R}}}(B_{\uw},B_\varnothing) \to \Hom^\bullet_{\mathcal{D}_{\mathrm{BS},\mathfrak{R}}^{\mathrm{asph}}}(\oB_{\uw},\oB_\varnothing)
\end{equation}
is spanned by the elements $\varphi_{\ue}$ where $\ue$ does not avoid $\Waff \smallsetminus \fW$ and the elements $f_i \cdot \varphi_{\ue}$ for all $\ue$ and $i \in I$. Indeed, by construction all these elements are annihilated by~\eqref{eqn:Hom-Diag-Dasph}. Now let $g \in \Hom^\bullet_{\mathcal{D}_{\mathrm{BS},\mathfrak{R}}}(B_{\uw},B_\varnothing)$ be in the kernel of~\eqref{eqn:Hom-Diag-Dasph}. Consider the decomposition of $g$ on the $\mathfrak{R}$-basis considered above, and assume for a contradiction that the coefficient on some $\varphi_{\ue}$ where $\ue$ avoids $\Waff \smallsetminus \fW$ is nonzero. Choose a prime $p > h$ such that $p$ does not divide this coefficient and such that Conjecture~{\rm \ref{conj:main}} holds for $G_{\overline{\mathbb{F}}_p}$ and some choice of $\lambda_0$. Then if $g'$ is the image of $g$ in
\[
\bk \otimes_{\mathfrak{R}} \Hom^\bullet_{\mathcal{D}_{\mathrm{BS},\mathfrak{R}}}(B_{\uw},B_\varnothing) \cong \Hom^\bullet_{\DiagBSk}(B_{\uw},B_\varnothing)
\]
(where the isomorphism is canonical and follows from~\cite[Theorem~6.11]{ew}), the image of $g'$ in $\DasphBSk$ is nonzero by Lemma~\ref{lem:morph-Dasp-LL-k}; this provides the desired contradiction.

From this claim we deduce the properties stated in the lemma: the $\mathfrak{R}$-module $\Hom^\bullet_{\mathcal{D}_{\mathrm{BS},\mathfrak{R}}^{\mathrm{asph}}}(\oB_{\uw},\oB_\varnothing)$ is free and has a basis formed by the images of the elements $\varphi_{\ue}$ where $\ue$ avoids $\Waff \smallsetminus \fW$. And if Conjecture~\ref{conj:main} holds for some $\bk$, then~\eqref{eqn:Hom-Dasph-basechange} sends a basis of the left-hand side to a basis of the right-hand side (see Lemma~\ref{lem:morph-Dasp-LL-k}), hence is an isomorphism.
\end{proof}

\begin{rmk}
Lemma~\ref{lem:morphisms-Dasph-integral} can also be deduced from Theorem~\ref{thm:Dasph-parities} below without having to assume the validity of Conjecture~\ref{conj:main}; see \S\ref{ss:LL-basis-Dasph} for details.
\end{rmk}


The following theorem (an easy consequence of Theorem~\ref{thm:main} and Lemma~\ref{lem:morphisms-Dasph-integral})
shows that $\mathcal{D}_{\mathrm{BS},\mathfrak{R}}^{\mathrm{asph}}$ provides a ``graded integral form'' of
$\Tilt(\Rep_0(G))$. 

\begin{thm}
\label{thm:integral} 
Assume that there exists infinitely many primes $p>h$ such that Conjecture~{\rm \ref{conj:main}} holds for $G_{\overline{\mathbb{F}}_p}$ and some choice of $\lambda_0$.

Let $\Bbbk$ and $\lambda_0$ be as in~\S{\rm \ref{ss:conjecture}}, and assume that Conjecture~{\rm \ref{conj:main}} holds for these data. Then there exists an additive functor
\[
\mathsf{F} \colon \mathcal{D}_{\mathrm{BS},\mathfrak{R}}^{\mathrm{asph}} 
\to \Tilt(\Rep_0(G_\Bbbk))
\]
and an isomorphism $\zeta \colon \mathsf{F} \circ \langle 1 \rangle \simto
\mathsf{F}$ such that for any $X,Y$ in $\mathcal{D}_{\mathrm{BS},\mathfrak{R}}^{\mathrm{asph}}$, $\mathsf{F}$ and $\zeta$ induce an isomorphism of $\Bbbk$-vector spaces
\[
\bk \otimes_{\mathfrak{R}}
\Hom^\bullet_{\mathcal{D}_{\mathrm{BS},\mathfrak{R}}^{\mathrm{asph}}}(X, Y) \simto \Hom_{\Rep_0(G)}(\mathsf{F}(X), \mathsf{F}(Y)).
\]
\end{thm}

\begin{rmk}
  As discussed above, over any field or complete local ring \cite[Theorem 6.25]{ew} gives
  a description of the split Grothendieck group of the 
  Karoubi envelope of the additive hull of $\DiagBS$ in terms of $W$. We gave a similar
  description of the indecomposable objects in $\Dasph$ in terms of $\fW$ in~\S\ref{ss:cat-Masph}. We have no idea
  if a similar classification of the indecomposable objects in the
  Karoubi envelope of $\mathcal{D}_{\mathrm{BS},\mathfrak{R}}$ or $\mathcal{D}_{\mathrm{BS},\mathfrak{R}}^{\mathrm{asph}}$ is possible. (As
  $\mathfrak{R}$ is not a complete local ring, it is not possible to apply idempotent
  lifting arguments.)
\end{rmk}

\subsection{Integral form of $\Rep_0(G)$}
\label{ss:repintforms}

As in~\S\ref{ss:intforms}, in this subsection we assume that $G_\Z$ is a split connected reductive
group defined over $\Z$, with simply-connected derived subgroup. Given any algebraically closed field $\bk$ we let
$G_{\bk}$ denote the extension of scalars of $G_{\Z}$ to $\Bbbk$.
\emph{In this subsection we fix $N \geq h$, and 
assume that for all fields $\bk$ of characteristic $p > N$, there exists $\lambda_0$ as in~\S{\rm \ref{ss:definitions-G}} such that Conjecture~{\rm \ref{conj:main}} holds for $G_\bk$ and $\lambda_0$.}

We set $\mathfrak{S}:=\Z[\frac{1}{N!}]$, and consider the
graded $\mathfrak{S}$-linear category $\mathcal{D}_{\mathrm{BS},\mathfrak{S}}^{\mathrm{asph}}$ and its natural right action of the monoidal category $\mathcal{D}_{\mathrm{BS},\mathfrak{S}}$.
%
We define the triangulated category $\sfK$ as the bounded homotopy category of the additive hull of the category
$\mathcal{D}_{\mathrm{BS},\mathfrak{S}}^{\mathrm{asph}}$.
For any algebraically closed field $\bk$ of characteristic $p > N$, we will denote by $\Dasph_{\bk}$ the antispherical diagrammatic category of~\S\ref{ss:cat-Masph} over $\bk$, and set
\[
\sfK_{\bk} := \Kb(\Dasph_{\bk}).
\]
The functor $\sfK \to \sfK_\bk$ given by tensoring all
homomorphism spaces with $\bk$ gives rise to a triangulated functor
\[
\bk \colon \sfK \to \sfK_{\Bbbk}.
\]
We denote the objects $\widetilde{\Sta}_x,
\widetilde{\Cos}_x, \overline{B}_x, \widetilde{\Sim}_x \in \sfK_{\Bbbk}$  of~\S\ref{ss:forms_of_rep0} by $\widetilde{\Sta}_x^{\bk},
\widetilde{\Cos}_x^{\bk}, \overline{B}_x^{\bk}, \widetilde{\Sim}_x^{\bk}$ to
emphasize their dependence on the base field $\bk$.

\begin{rmk}
It might seem strange that we do not take a Karoubi envelope in the definition of $\sfK$. To make this more natural, one can remark that, for any $\bk$ as above, the natural fully-faithful functor from the bounded homotopy category of the additive hull of $\DasphBSk$ to $\sfK_\bk$ is an equivalence. In fact it is easily checked by induction on the Bruhat order that the essential image of this functor contains all objects $\oB_w$ with $w \in \fW$, and the claim follows.
\end{rmk}

In the considerations below, the order $\leq$ we consider on $W$ is the Bruhat order.
For any ideal $I \subset \Waff$ we denote by $\sfK_I$ the full triangulated graded
subcategory generated by the objects $\oB_{\un{x}}$ for $\un{x}$ a reduced
expression for $x \in I$. 

To state the next lemma we need to use the Hecke product (sometimes called the Demazure product) in $\Waff$ as considered e.g.~in~\cite[\S 3]{bm}. In fact, recall from~\cite[Proposition~3.1]{bm} that there exists a unique associative product $*$ on $\Waff$ which satisfies
\[
w * s = \begin{cases}
ws & \text{if $ws>w$;} \\
w & \text{otherwise}
\end{cases}
\]
for $w \in \Waff$ and $s \in \Saff$.
Given an expression $\un{x} = s_1 \cdots s_m$ we set
\[
*\un{x} := s_1 * s_2 * \cdots * s_m.
\]
Below we will use the following properties:
\begin{gather}
\label{eqn:Dem-product-1}
\text{if $\uw$ and $\ux$ differ by a braid relation, then $*\uw=*\uv$;} \\
\label{eqn:Dem-product-2}
\text{if $\uw$ is obtained from $\ux$ by omitting some reflections, then $*\uw \leq *\ux$.}
\end{gather}
In fact,~\eqref{eqn:Dem-product-1} is clear from associativity, and~\eqref{eqn:Dem-product-2} follows from~\cite[Proposition~3.1(c)]{bm}.

\begin{lem}
\label{lem:Bx-Demazure-product}
  For any expression $\un{x}$ we have $\overline{B}_{\un{x}} \in \sfK_{\le *\un{x}}$.
\end{lem}

\begin{proof}
  We prove the lemma by induction; first on the length $\ell(*\un{x})$
  of $*\un{x}$, and then 
on the length $\ell(\un{x})$ of
  $\un{x}$. If $\ell(*\un{x}) = 0$ then $\un{x}=\varnothing$
  and the lemma is obvious. Also, if $\ell(\un{x}) = \ell( *\un{x})$ then $\un{x}$ is
  reduced and the lemma follows from the definitions. So we fix an
  expression $\un{x}$ with $\ell(\un{x}) > \ell(*\un{x})$ and assume that the lemma is known for all
  expressions $\un{y}$ with $\ell(*\un{y}) < \ell(*\un{x})$ and all expressions
  $\un{x}'$ with $\ell(*\un{x}') = \ell(*\un{x})$ and $\ell(\un{x}') < \ell(\un{x})$.
 
If $\un{x}$ contains a subexpression of the form $ss$ then we can apply
Lemma~\ref{lem:BsBs} to deduce an isomorphism
\[
\overline{B}_{\un{x}} \cong \overline{B}_{\un{x}'}\langle 1 \rangle
\oplus \overline{B}_{\un{x}'}\langle -1 \rangle
\]
with $\ell(\un{x}') < \ell(\un{x})$ and $*\ux=*\ux'$, and we are done by
induction. If $\un{x}$ contains no such subexpression then
because $\ell(\un{x}) > \ell(*\un{x})$ there exists an expression $\uy$ and a sequence of braid relations which connects $\ux$ and $\uy$,
such that $\un{y}$ has a subexpression
of the form $ss$. If we denote by $f\colon \overline{B}_{\un{x}} \to \overline{B}_{\un{y}}$ the associated
morphism (as in~\S\ref{ss:rex-moves-DBS}, except that now we do not require that the expression is reduced) then, 
using the obvious generalization of Lemma~\ref{lem:rex-move-smaller},~\eqref{eqn:Dem-product-2}, and induction,
we conclude that $f$ is an isomorphism in the Verdier quotient $\sfK/\sfK_{<*\un{x}}$ (in the sense of~\cite[\S 2.1]{neeman}). Since $*\ux=*\uy$ by~\eqref{eqn:Dem-product-1},
$\overline{B}_{\un{y}}$ belongs to $\sfK_{\leq *\un{x}}$ by the case treated above, hence so does 
$\overline{B}_{\un{x}}$. The lemma follows.
\end{proof}

\begin{lem} \label{lem:stdgen}
Let $I \subset W$ be an ideal. Choose, for any $x \in I \cap \fW$, a reduced expression $\uw(x)$ for $x$. Then $\sfK_I$ is generated, as a graded triangulated category, by the objects $\{\oB_{\uw(x)} : x \in I \cap \fW\}$. 
\end{lem}

\begin{proof}
  Let us write $\sfK_I'$ for the category generated by the objects as in the lemma. We have to show that if
  $\un{x}$ if a reduced expression for $x \in I$ then
  $\overline{B}_{\un{x}} \in \sfK_I'$. We do so by induction on
  $\ell(\un{x}) = \ell(x)$, the case $x=1$ being obvious. 
  
  There exists a
  rex move $\un{x} \leadsto \un{y}$ such that $\un{y}$ is either $\uw(x)$ (if $x \in \fW$), or of the form
  $s_1 s_2 \cdots s_m$ with $s_1 \in \Sf$ (if $x \notin \fW$). As in the proof of Lemma~\ref{lem:Bx-Demazure-product}, using Lemma~\ref{lem:rex-move-smaller} we see that the associated morphism $\overline{B}_{\un{x}} \to
  \overline{B}_{\un{y}}$ is an isomorphism in $\sfK/\sfK_{<x}$. On the
  other hand, if $x \in \fW$ then $\oB_{\uy}$ belongs to $\sfK_{I}'$, and if $x \notin \fW$ then $\overline{B}_{\un{y}} = 0$ in $\sfK$ by definition. Since $\sfK_{<x} \subset \sfK'_I$ by induction, the claim follows.
\end{proof}

\begin{rmk}
It follows in particular from Lemma~\ref{lem:stdgen} that the subcategory $\sfK_I$ depends only on the intersection $I \cap \fW$.
\end{rmk}

\begin{prop} \label{prop:integral-standard} For all reduced
  expressions $\un{x}$ for $x \in \fW$ there exist objects $\Delta^{\mathfrak{S}}_{\un{x}}$ and $\nabla^{\mathfrak{S}}_{\un{x}}$ in $\sfK$, and maps
\[
i_{\un{x}} \colon \Delta^{\mathfrak{S}}_{\un{x}} \to \overline{B}_{\underline{x}} \quad \text{and} \quad p_{\un{x}} :
\overline{B}_{\underline{x}} \to \nabla^{\mathfrak{S}}_{\un{x}},
\]
such that:
\begin{enumerate}
\item $\Delta^{\mathfrak{S}}_{\un{x}}$ and $\nabla^{\mathfrak{S}}_{\un{x}}$ belong to $\sfK_{\le x}$;
\item $\cone(i_{\un{x}})$ and $\cone(p_{\un{x}})$ belong to $\sfK_{< x}$;
\item for any algebraically closed field $\Bbbk$ of characteristic $p > N$, we have
  isomorphisms
\[
\bk(\Delta^{\mathfrak{S}}_{\un{x}}) \cong \widetilde{\Delta}_x^{\bk}, \quad
\bk( \nabla^{\mathfrak{S}}_{\un{x}}) \cong \widetilde{\Cos}_x^{\bk}.
\]
\end{enumerate}
\end{prop}

\begin{rmk}
  We will see in Lemma~\ref{lem:CZ-DZ-indep} below that the objects $\Delta^{\mathfrak{S}}_{\un{x}}$ do not depend
  on the reduced expression $\un{x}$ for $x$ (up to canonical isomorphism).
\end{rmk}

\begin{proof}
  The proof (an easy adaption of the proof of Proposition \ref{prop:stdcostd})
  is left to the reader.
\end{proof}

Let us fix a reduced expression $\un{x}$ for all $x \in \fW$ and define
\[
\Delta^{\mathfrak{S}}_x := \Delta^{\mathfrak{S}}_{\un{x}}, \quad \nabla^{\mathfrak{S}}_x := \nabla^{\mathfrak{S}}_{\un{x}}, \quad p_x := p_{\un{x}},
\quad i_x := i_{\un{x}}.
\]
The following is immediate from Lemma \ref{lem:stdgen} and Proposition \ref{prop:integral-standard}.

\begin{cor} \label{cor:dcgen}
  For any ideal $I \subset W$, the subcategory $\sfK_I$ is generated, as a graded triangulated category, by the objects $\Delta^{\mathfrak{S}}_x$ for $x \in I \cap \fW$, and also by the objects $\nabla^{\mathfrak{S}}_x$ with $x \in I \cap \fW$. 
\end{cor}

We now ``lift" the vanishing statements of~\S\ref{ss:forms_of_rep0} to the case of $\mathfrak{S}$.

\begin{prop}
\label{prop:vanishings-KZ'}
For $x,y \in \fW$ and $m,n \in \Z$ we have:
\begin{gather}
\label{eqn:Zexceptional-seq-1}
\Hom_{\sfK}(\nabla^{\mathfrak{S}}_x, \nabla^{\mathfrak{S}}_y \langle m \rangle [n])=0 \quad
\text{unless $x \ge y$,}
\\
\label{eqn:Zexceptional-seq-2}
\Hom_{\sfK}(\nabla^{\mathfrak{S}}_x, \nabla^{\mathfrak{S}}_x \langle m \rangle [n])= \begin{cases}
\mathfrak{S} & \text{if $m=n=0$;} \\
 0 & \text{otherwise,}
\end{cases}
\\
\label{eqn:Zexceptional-seq-3}
\Hom_{\sfK}(\nabla^{\mathfrak{S}}_x, \nabla^{\mathfrak{S}}_y \langle m \rangle [n])=\Hom_{\sfK}(\Delta^{\mathfrak{S}}_x, \Delta^{\mathfrak{S}}_y \langle m \rangle [n])=0 \text{ if $n<0$.}
\end{gather}
\begin{equation}
  \label{eqn:Zmorph-Sta-Cos-graded}
\Hom_{\sfK}(\Delta^{\mathfrak{S}}_x, \nabla^{\mathfrak{S}}_y\langle m \rangle[n]) = \begin{cases}
\mathfrak{S} & \text{if $x=y$ and $m = n=0$;} \\
0 & \text{otherwise.}
\end{cases}
\end{equation}  
\end{prop}

\begin{rmk}
  These results are consequences of
  \eqref{eqn:exceptional-seq-1}, \eqref{eqn:exceptional-seq-2},
  \eqref{eqn:exceptional-seq-3} and \eqref{eqn:morph-Sta-Cos-graded},
  which in turn are consequences of the corresponding vanishing for algebraic groups. We do not know any diagrammatic proof for these statements.
\end{rmk}

\begin{proof}
  We prove \eqref{eqn:Zexceptional-seq-1}.
The argument giving the
  other relations is essentially identical. We can compute $\Hom_{\sfK}(\nabla^{\mathfrak{S}}_x, \nabla^{\mathfrak{S}}_y
  \langle m \rangle [n])$ by choosing bounded complexes $D$ and $C$ of objects in
  $\mathcal{D}_{\mathrm{BS},\mathfrak{S}}^{\mathrm{asph}}$ representing $\nabla^{\mathfrak{S}}_x$ and $\nabla^{\mathfrak{S}}_y \langle m \rangle$ respectively, and calculating the $n$-th
  cohomology of the complex $(H^\bullet, d)$ with
\[
H^k := \bigoplus_{i \in \Z} \Hom_{\mathcal{D}_{\mathrm{BS},\mathfrak{S}}^{\mathrm{asph}}}(D^i, C^{i+k}),
\]
see \cite[\S 11.7]{ks}. By Proposition \ref{prop:integral-standard},
for any field $\Bbbk$ of characteristic $p > N$ we have $\bk
(C) = \widetilde{\Delta}_x^{\bk}$ and $\bk
(D) = \widetilde{\nabla}_x^{\bk}$. From
\eqref{eqn:exceptional-seq-1} and Lemma~\ref{lem:morphisms-Dasph-integral} we deduce that $\bk
\lotimes_{\mathfrak{S}} H^\bullet = 0$ if $x \not\geq y$. Hence $H^\bullet$ 
is equal to $0$ in the derived category of $\mathfrak{S}$-modules, hence has no cohomology.
\end{proof}

\begin{prop} \label{prop:adjointsexist}
  Let $I \subset W$ be an ideal.
  \begin{enumerate}
  \item 
  \label{it:adjoint-inclusion}
  The inclusion $i \colon \sfK_I \into \sfK$ admits a left adjoint $i^L$ and a right
    adjoint $i^R$.
  \item
  \label{it:adjoint-projection}
  The quotient $j \colon \sfK \onto \sfK/\sfK_I$ admits a left
    adjoint $j^L$ and a right
    adjoint $j^R$.
  \end{enumerate}
\end{prop}

\begin{proof}
In this proof, we denote by $\langle X \rangle$ the full graded triangulated subcategory of $\sfK$ generated by a set $X$ of objects.

Let us first consider~\eqref{it:adjoint-inclusion}. We prove that $i$ has a right adjoint; the
  existence of the left adjoint follows from a dual argument. Consider
$\sfK_I^\perp$, the full subcategory of $\sfK$ consisting of objects $X$
with $\Hom_{\sfK}(Y, X) = 0$ for all $Y$ in $\sfK_I$. From Corollary~\ref{cor:dcgen} and
\eqref{eqn:Zexceptional-seq-1} we one may deduce:
\begin{enumerate}
\item
\label{eqn:adjoints-1}
Any object $X$ in $\sfK$ is part of a distinguished triangle
\[
X_I \to X \to X^I \triright
\]
with $X_I \in \langle \nabla^{\mathfrak{S}}_x : x \in I \cap \fW \rangle$ and $X^I
\in \langle \nabla^{\mathfrak{S}}_y : y \in \fW \smallsetminus (I \cap \fW) \rangle$;
\item 
\label{eqn:adjoints-2}
$\sfK_I^\perp = \langle \nabla^{\mathfrak{S}}_y : y \in \fW \smallsetminus (I \cap \fW) \rangle$.
\end{enumerate}
From this and the definition of the Verdier quotient it follows that
for $X \in \sfK_I$ and $Y \in \sfK$ the map
\[
\Hom_{\sfK}(i(X), Y) \to \Hom_{\sfK/\sfK_I^{\perp}}(qi(X), q(Y))
\]
induced by the quotient $q \colon \sfK \to \sfK/\sfK_I^\perp$ is an isomorphism. In
particular $qi$ is fully-faithful. From~\eqref{eqn:adjoints-1} and~\eqref{eqn:adjoints-2} above one deduces easily
that $qi$ is also essentially surjective; hence it is an equivalence. Now
if we fix an inverse $(qi)^{-1}$ to
$qi$ we have
\[
\Hom_{\sfK}(i(X),Y) \cong \Hom_{\sfK/\sfK_I^{\perp}}(qi(X), q(Y)) \cong \Hom_{\sfK_I}(X, (qi)^{-1} q(Y))
\]
and hence our adjoint is given by $i^R := (qi)^{-1} q$.

We now turn to~\eqref{it:adjoint-projection}. Again, we only establish the existence of $j^R$
and leave $j^L$ to the reader. If we consider the composition
\[
\sfK_I^\perp = \langle \nabla^{\mathfrak{S}}_y \; | \; y \notin I \cap \fW \rangle \xrightarrow{g} \sfK \xrightarrow{j} \sfK/\sfK_I
\]
where $g$ is the natural embedding,
then as above~\eqref{eqn:adjoints-1} and~\eqref{eqn:adjoints-2} imply that $jg$ is an equivalence. As above,
after fixing a quasi-inverse $(jg)^{-1}$ we can take $j^R := g (jg)^{-1}$.
\end{proof}

Fix an ideal $I \subset W$. Then, as in the
proof of Lemma \ref{lem:rex}, for any element $x \in \fW$ which is minimal
in $W \smallsetminus I$ and any
two reduced expressions $\un{x}'$ and $\un{x}''$ for $x$, the images of the
objects $\overline{B}_{\un{x}'}$ and $\overline{B}_{\un{x}''}$ in $\sfK/\sfK_I$ are
canonically isomorphic, via the isomorphism associated with any choice of
rex move $\un{x}' \leadsto \un{x}''$, see~\cite[\S 6.5]{ew}. We denote the resulting object
$\mathfrak{S}_x$.

\begin{lem}
\label{lem:CZ-DZ-indep}
  For any ideal $I \subset W$ and any minimal element $x \in \fW \smallsetminus
  I$, there exists canonical isomorphisms
\[
\Delta^{\mathfrak{S}}_x = j^L\mathfrak{S}_x \qquad \text{and} \qquad \nabla^{\mathfrak{S}}_x = j^R\mathfrak{S}_x.
\]
(Here $j^L, j^R$ are the adjoints to $\sfK \to \sfK/\sfK_I$ constructed in Proposition~{\rm \ref{prop:adjointsexist}}.) In particular, the objects  $\Delta^{\mathfrak{S}}_x$ and $\nabla^{\mathfrak{S}}_x$
for all $x \in \fW$ are independent of the choice of reduced
expression used to define them.
\end{lem}

\begin{proof}
Again, we give the proof of the canonical isomorphism $\nabla^{\mathfrak{S}}_x \cong j^R\mathfrak{S}_x$, and leave
  the dual statement $\Delta^{\mathfrak{S}}_x \cong j^L\mathfrak{S}_x$ to the reader. 
  Looking at the proof of Proposition~\ref{prop:adjointsexist}, and using the same notation, we see
  that, since $\mathfrak{S}_x=jg(\nabla^{\mathfrak{S}}_x)$, $j^R\mathfrak{S}_x$ may be calculated as the image of $\nabla^{\mathfrak{S}}_x$ under the
  embedding 
\[
\sfK_I^\perp = \langle \nabla^{\mathfrak{S}}_y \; | \; y \notin I \cap \fW \rangle
\xrightarrow{g} \sfK
\]
Hence $j^R\mathfrak{S}_x \cong \nabla^{\mathfrak{S}}_x$. The isomorphism is canonical because
\[
\Hom_{\sfK}(\nabla^{\mathfrak{S}}_x, j^R\mathfrak{S}_x) = \Hom_{\sfK/\sfK_I}(j(\nabla^{\mathfrak{S}}_x), \mathfrak{S}_x) = \Hom_{\sfK/\sfK_I}(\mathfrak{S}_x, \mathfrak{S}_x),
\]
where the identification $\mathfrak{S}_x  = j(\nabla^{\mathfrak{S}}_x) $ is given by $p_x$.
\end{proof}

From the above we see that if $I \subset W$ is an ideal and $x
\in \fW \smallsetminus I$ is a minimal element then setting $I'=I \cup \{y \in W \mid y \leq x\}$ then $\sfK_{I'}/\sfK_I$ identifies with the graded triangulated subcategory of $\sfK/\sfK_I$ generated by $\mathfrak{S}_x$, and we have an equivalence
of graded triangulated categories
\[
\phi_x \colon 
\Db(\mathfrak{S}\mathsf{-grmod}) \simto \sfK_{I'}/\sfK_I
\]
sending $\mathfrak{S}$ to $\mathfrak{S}_x$. (Here, $\mathfrak{S}\mathsf{-grmod}$ is the abelian category of finitely generated graded $\mathfrak{S}$-modules.)
For any $n \in \mathfrak{S}$, we define
\begin{gather*}
\Delta^{\mathfrak{S}/n\mathfrak{S}}_x :=  \begin{cases} \Delta^{\mathfrak{S}}_x & \text{if $n = 0$,} \\
\cone( \Delta^{\mathfrak{S}}_x \xrightarrow{n \cdot (-)} \Delta^{\mathfrak{S}}_x) &
\text{otherwise.} \end{cases} \\
\nabla^{\mathfrak{S}/n\mathfrak{S}}_x :=  \begin{cases} \nabla^{\mathfrak{S}}_x & \text{if $n = 0$,} \\
\cone( \nabla^{\mathfrak{S}}_x \xrightarrow{n \cdot (-)} \nabla^{\mathfrak{S}}_x) & \text{otherwise.} \end{cases}
\end{gather*}
Then via the above equivalence $\phi_x$ we have
\[
\Delta^{\mathfrak{S}/n\mathfrak{S}}_x \cong j^L(\phi_x(\mathfrak{S}/n\mathfrak{S})) \quad \text{and} \quad
\nabla^{\mathfrak{S}/n\mathfrak{S}}_x \cong j^R(\phi_x(\mathfrak{S}/n\mathfrak{S}))
\]
for all $x \in \fW$.

In the following statement we use the notation $\langle - \rangle_{\mathrm{ext}}$ introduced in~\S\ref{ss:forms_of_rep0}.

\begin{thm}
\label{thm:repintform}
  The full subcategories
\begin{align*}
\sfK^{\le 0} &:= \langle \Delta^{\mathfrak{S}}_x\langle m \rangle[k] :
x \in \fW, \,  m \in \Z, \, k \in \Z_{\ge 0} \rangle_{\ext}, \\
\sfK^{\ge 0} &:= \langle \nabla^{\mathfrak{S}/n\mathfrak{S}}_x\langle m \rangle[k] :
x \in \fW, \, n \in \mathfrak{S},  m \in \Z, \, k \in \Z_{\le 0} \rangle_{\ext}
\end{align*}
define a non-degenerate bounded $t$-structure on $\sfK$.
\end{thm}

\begin{proof}
This follows from the glueing formalism of~\cite{bbd}, as e.g.~in~\cite[Proposition~1]{bez} or in~\cite[\S 3.1]{modrap2}, starting from the standard t-structure on $\Db(\mathfrak{S}\mathsf{-grmod})$.
\end{proof}

We regard the heart of the t-structure constructed in Theorem~\ref{thm:repintform} as an ``integral form'' of $\Rep_0(G)$. One can deduce from~\eqref{eqn:Zexceptional-seq-3} that this heart contains the objects $\Delta^{\mathfrak{S}}_x$ and $\nabla^{\mathfrak{S}}_x$. It is not difficult to check (using the same arguments as in the proof of Proposition~\ref{prop:vanishings-KZ'}) that for $x,y \in \fW$ the $\mathfrak{S}$-modules
\[
\Hom_{\sfK}(\nabla^{\mathfrak{S}}_x, \nabla^{\mathfrak{S}}_y) \quad \text{and} \quad \Hom_{\sfK}(\Delta^{\mathfrak{S}}_x,\Delta^{\mathfrak{S}}_y)
\]
are free, and to deduce that the objects $\Delta^{\mathfrak{S}/n\mathfrak{S}}_x$ and $\nabla^{\mathfrak{S}/n\mathfrak{S}}_x$ also belong to this heart.

 }
\else { \newpage }
\fi

\part{The case of $\mathrm{GL}_n(\bk)$}
\label{pt:GLn}

\ifdefined\PARTCOMPILE{ 
\textbf{Overview}.
In this part we give a proof of Conjecture~\ref{conj:main} in the case of the group $G=\mathrm{GL}_n(\bk)$ (when $p>n\geq 3$), using
the theory of categorical Kac--Moody actions. Before plunging into the
details we give a brief overview of our approach.

For ease of exposition, let us first consider the case of integral category
$\mathcal{O}$ for $\mathfrak{g} := \mathfrak{gl}_n(\C)$. That is, we
consider $\mathfrak{g}$ together with its subalgebras $\mathfrak{h}
\subset \mathfrak{b}$ of diagonal and upper-triangular matrices and
define $\mathcal{O}^{\Z}$ to be the category of finitely
generated $\mathfrak{g}$-modules which are locally finite over
$\mathfrak{b}$, and $\mathfrak{h}$-diagonalisable with integral
weights (i.e.~weights which are differentials of characters of the torus of diagonal matrices in $\mathrm{GL}_n(\C)$). 

Inside $\mathcal{O}^{\Z}$ we can consider the block $\mathcal{O}_0$ of the trivial
module. It is a classical result of Soergel that $\Tilt(\mathcal{O}_0)$, the additive category of tilting modules
inside $\mathcal{O}_0$, has an action of the monoidal category of Soergel
bimodules associated with the Weyl group ${S}_n$ and its natural action on $\mathfrak{h}$. Indeed, in~\cite{soergel-garben} Soergel proves via his functor $\mathbb{V}$ that
$\Tilt(\mathcal{O}_0)$ is equivalent to a category of Soergel \emph{modules}, and
from the definitions we see that this category of Soergel modules has an
action of Soergel bimodules.

\begin{rmk*}
Soergel's results hold for any complex semi-simple Lie algebra and
 were initially phrased in terms of projective modules rather than
tilting modules. Tilting modules were considered in this context (using different terminology) by Colling\-wood--Irving~\cite{collingwood-irving}. The (mathematically easy, but conceptually
important) rephrasing of Soergel's results in terms of tilting modules
appears to be due to \cite{bbm}.
\end{rmk*}

In Soergel's approach one deduces the existence of the action of
Soergel bimodules \emph{after} obtaining a ``Hecke category description'' of the category
of tilting modules. In the setting of our conjecture we
would like information to flow the other way: we would like to
obtain a Hecke category description of the category of tilting modules from the
existence of a categorical action. The main reason is the 
difficulty to define a functor $\mathbb{V}$ in the setting of reductive groups. (In fact, no convincing definition seems to be known.)

Let us first outline how to do this in the case of
$\mathcal{O}^{\Z}$. Let us denote by $\bX \subset \mathfrak{h}^*$ the weight
lattice. Given any $\lambda \in \bX$ we have the Verma module
$\Delta_\C(\lambda) := U(\mathfrak{g}) \otimes_{U(\mathfrak{b})}
\C_{\lambda}$ and the classes of Verma modules give a basis for the complexified
Grothendieck group of $\mathcal{O}^\Z$:
\[
[\mathcal{O}^\Z]_{\C}  = \bigoplus_{\lambda \in \bX} \C [\Sta_\C(\lambda)].
\]

On $\mathcal{O}^{\Z}$ one has the endofunctors $E := V
\otimes (-)$ and $F := V^* \otimes (-)$ given by tensoring with the
natural representation $V = \C^n$ and its dual. Both functors have a natural
endomorphism $\Xbb$ (given essentially by the action of the Casimir element),
all of whose eigenvalues belong to $\Z \subset \C$. Hence we obtain
decompositions of our functors
\[
E = \bigoplus_{a \in \Z} E_a \qquad \text{and} \qquad F = \bigoplus_{a \in \Z} F_a
\]
into generalized eigenspaces.

It is a remarkable observation that on the level of Grothendieck groups these functors give
rise to an action of the Lie algebra $\mathfrak{sl}_\infty$ (whose Chevalley generators will be denoted $e_i$ and
$f_i$ for $i \in \Z$). More precisely, consider the natural module $\nat_\infty := \bigoplus_{i \in
  \Z} \C m_i$ of ``infinite column vectors'' for
$\mathfrak{sl}_\infty$. The isomorphism of vector spaces
\[
\varphi \colon [\mathcal{O}^\Z]_{\C} \simto \nat_\infty^{\otimes n}
\]
defined by
\[
\varphi \bigl( [\Sta_{\C}(\lambda)] \bigr) = m_{\lambda_1} \otimes m_{\lambda_2-1} \otimes \cdots \otimes m_{\lambda_n-n + 1}
\]
(where $\lambda = (\lambda_1, \cdots, \lambda_n)$ under the standard
identification $\bX = \Z^n$, i.e.~$\lambda$ sends a diagonal matrix with coefficients $x_1, \cdots, x_n$ to $\sum_{i=1}^n \lambda_i \cdot x_i$) 
intertwines the action of $[E_a]$ and $[F_a]$ on the left-hand side with the
action of $e_a$ and $f_a$ on the right-hand side. In particular,
the functors $E_a$ and $F_a$ induce an action of $\mathfrak{sl}_\infty$ on the
Grothendieck group of $\mathcal{O}^\Z$. Moreover, there exist
morphisms between the functors $E_a, F_a$ and their compositions which
allow us to upgrade this action on the Grothendieck group to a
``categorical $\mathfrak{sl}_\infty$-action'' in the sense of Rouquier and
Khovanov--Lauda. (The fact that the above action on the Grothendieck
group of $\mathcal{O}^{\Z}$ can be lifted to a categorical
$\mathfrak{sl}_\infty$-action is essentially due to Chuang--Rouquier~\cite{CR}.)

Now, let us return to our goal of producing an action of the Hecke
category on $\mathcal{O}_0 \subset \mathcal{O}^{\Z}$. Under the
isomorphism $\varphi$ above, the weight spaces for the
$\mathfrak{sl}_\infty$-action correspond precisely to the blocks of
$\mathcal{O}^{\Z}$. Moreover, on certain regular
blocks one can express each wall-crossing
functor in the form $F_aE_a$ for some $a \in \Z$.  The theory of
Kac--Moody actions gives us a rich supply of
morphisms between the functors $E_a$ and $F_a$, and it turns out that one can produce
the action of the Hecke category by explicitly giving the images of
all the generating morphisms and checking the relations (using known
relations in the categorification of $\mathfrak{sl}_\infty$).

The proof of our conjecture for $G = \mathrm{GL}_n(\bk)$ over a field
$\Bbbk$ of characteristic $p > n$ follows a
similar pattern, with one additional subtlety. Let us denote as in~\S\ref{ss:definitions-G} by $\Rep(G)$ the category of finite-dimensional rational representations of $G$ and by
$\Sta(\lambda)$ the Weyl module with highest weight $\lambda \in
\bX^+$. Then the Weyl modules give a basis for the (complexified)
Grothendieck group:
\[
[\Rep(G)]_\C = \bigoplus_{\lambda \in \bX^+} \C[\Delta(\lambda)]. 
\]
Again we can consider the functors $E := V \otimes (-)$ and $F := V^*
\otimes (-)$ given by tensoring with the natural representation $V$ of
$\mathrm{GL}_n(\bk)$ and its dual. As in the case of $\gl_n(\C)$, these functors have a natural endomorphism,
all of whose eigenvalues belong to the prime subfield $\Z/p\Z \subset
\Bbbk$, and we have decompositions
\[
E = \bigoplus_{a \in \Z/p\Z} E_a \qquad \text{and} \qquad F = \bigoplus_{a \in \Z/p\Z} F_a.
\]
In this setting the functors $E_a$ and $F_a$ give rise to an action of
the affine Lie algebra $\widehat{\mathfrak{gl}}_p$ on the Grothendieck
group. If this time we let $\nat_p$ denote the natural module for
$\widehat{\mathfrak{gl}}_p$ (see~\S\ref{sec:natglN}) then we have an
isomorphism of vector spaces 
\[
\varphi \colon [\Rep(G)]_{\C}  \simto \bigwedge \hspace{-3pt} {}^n \, \nat_p
\]
defined by
\[
\varphi \bigl( [\Sta(\lambda)] \bigr) = m_{\lambda_1} \wedge m_{\lambda_2-1}
\wedge \dots \wedge m_{\lambda_n-n + 1}.
\]
Again the action of the functors $E_a$ and $F_a$ on $\Rep(G)$ can be upgraded to a
categorical action of $\widehat{\mathfrak{gl}}_p$ (and again, this is due
to Chuang--Rouquier~\cite{CR}). In this setting 
each weight space is a union of blocks of $\Rep(G)$.

Let us now concentrate (as we do below) on the block containing
the Weyl module $\Sta \bigl( (n,n,\cdots,n) \bigr)$. As in the case of $\mathcal{O}^\Z$
considered above, we can express the wall-crossing functors
corresponding to ``finite'' simple reflections $s \in \Sf$ as 
compositions $F_aE_a$ for $1 \le a \le n-1$ (where we identify $a$ with its image in $\Z/p\Z$). Here the calculations
verifying the defining relations in the Hecke category are essentially
identical to those for $\mathcal{O}^{\Z}$. (These calculations are
basically already contained in the work of Mackaay--Sto{\v s}i{\'c}--Vaz~\cite{msv}.)  An additional difficulty is
provided by the wall-crossing functor corresponding to the affine
reflection. It is not difficult to see that this functor can be
expressed in terms of the functors $E_a$ and $F_a$ as 
\[
F_nF_{n+1} \dots F_{p-1} F_0E_0 E_{p-1} \dots E_{n+1}E_n,
\]
see~\S\ref{sss:comparison-Section3}.
However, checking the relations directly using this definition of the
wall-crossing functor is very complicated. (Such
calculations are performed, with different normalizations, in
\cite{mt1,mt2}.) Here we take a different approach and argue that we
can ``restrict'' part of the 2-representation given by $\Rep(G)$ to
obtain a representation of $\widehat{\mathfrak{gl}}_n$ (see Section~\ref{sec:restriction}).
After restricting we can express all
translation functors as composites $F_aE_a$ for $a \in \{ 0, \dots,
n-1\}$ and we check the relations in the Hecke category directly,
using the relations of the categorification of
$\widehat{\mathfrak{gl}}_n$.

\begin{rmk*}
\begin{enumerate}
\item
It will be more convenient for us to construct a \emph{left} action of $\DiagBS$ on $\Rep_0(G)$ rather than a right action as in Conjecture~\ref{conj:main}. However, from a left action one can obtain a right action by composing with the equivalence $\imath$; see Remark~\ref{rmk:main-conj}\eqref{it:rmk-left-right}.
\item
Recently, Maksimau~\cite{m} has given a
  general setup for restricting categorical representations
  of $\widehat{\mathfrak{sl}}_{e+1}$ to obtain categorical representations
  of $\widehat{\mathfrak{sl}}_e$. Our approach at this step is a special case of
  his construction.
\end{enumerate}
\end{rmk*}

This part consists of 3 sections, corresponding to the 3 steps of our proof. In Section~\ref{sec:2rep-GLn} we recall the construction of the categorical action of $\hgl_p$ on $\Rep(G)$ due to Chuang--Rouquier. In Section~\ref{sec:restriction} we explain how the action on a subcategory $\Rep^{[n]}(G)$ can be restricted to $\hgl_n$ if $p \geq n \geq 3$. Finally, in Section~\ref{sec:g2D} we prove that from this action one can obtain an action of $\DiagBS$ on $\Rep_0(G)$, which proves Conjecture~\ref{conj:main} in this case.

\section{Representations of $\mathrm{GL}_n$ in characteristic $p$ as a $2$-representation of $\hgl_p$}
\label{sec:2rep-GLn}

In this section we fix an algebraically closed field $\bk$ of characteristic $p>0$, and an integer $n \geq 1$. We
recall the construction of a categorical action of $\hgl_p$ on the category $\Rep(\mathrm{GL}_n(\Bbbk))$ due to Chuang--Rouquier. All the proofs are copied from~\cite{CR}. (We have tried to follow the notation and conventions of~\cite{CR} and~\cite{R2KM} as closely as possible. These authors work with the Lie algebra $\widehat{\mathfrak{sl}}_p$ rather than $\hgl_p$, but the extension is straightforward.)

\subsection{The affine Lie algebra $\hgl_N$}
\label{slp}

If $N \in \Z_{\geq 2}$, we denote by
$\hgl_N$ the affine Lie algebra associated with the Lie algebra $\gl_N(\C)$.
More precisely, we first consider the Lie algebra
\[
\widehat{\mathfrak{sl}}_N := \mathfrak{sl}_N(\C[t,t^{-1}]) \oplus \C K \oplus \C d
\]
with the Lie
bracket determined by the following rules:
\begin{gather*}
  [x \otimes t^m, y \otimes t^n] = [x,y] \otimes t^{m+n} +
  m\delta_{m,-n} (x|y)K \quad \text{for all $x, y \in \mathfrak{sl}_N(\C)$, $m,n \in
    \Z$;} \\
[d, x \otimes t^m] = m x \otimes t^m \quad \text{for all $x \in \mathfrak{sl}_N(\C)$,
  $m \in \Z$;} \\
[K, \widehat{\mathfrak{sl}}_N ] = 0.
\end{gather*}
Here $(-|-) \colon \mathfrak{sl}_N(\C) \times \mathfrak{sl}_N(\C) \to \C$ denotes the bilinear
form $(x|y) := \Tr(xy)$ on $\mathfrak{sl}_N(\C)$. Then $\hgl_N$ is the Lie algebra $\widehat{\mathfrak{sl}}_N \oplus \C$,  where we identify $(0,1)$ with the identity matrix in $\gl_N(\C)$.

Let $(e_{i,j})_{i,j \in \{1, \cdots, N\}}$ denote the matrix
units in $\gl_N(\C)$. Then
we consider the Chevalley elements
\begin{gather*}
  e_i = e_{i+1,i}, \quad f_i := e_{i,i+1} \quad \text{if $1 \le i \le N-1$;} \\ \quad e_0 := te_{1,N}, \quad
  f_0 := t^{-1}e_{N,1}; \\
h_i 
:= [e_i, f_i] = \begin{cases}
e_{1,1} - e_{N,N} + K & \text{if $i=0$;} \\
e_{i+1,i+1} - e_{i,i} & \text{if $1 \le i \le N-1$.}
\end{cases}
\end{gather*}
in $\hgl_N$.

We set $\hg := \hgf \oplus \C K \oplus \C d \subset \hgl_N$, where $\hgf \subset \gl_N(\C)$ denotes the Cartan subalgebra of
diagonal matrices. We have 
\[
\hgf^* = \bigoplus_{i=1}^N \C \e_i
\]
where
$(\e_i)_{i \in \{1, \cdots, N\}}$ is the basis dual to
$e_{1,1}, \cdots, e_{N,N}$. Any $\la_0 \in \hgf^*$ may be viewed as an element
of $\hg^*$ by setting $\langle \la_0, K \rangle = \langle \la_0, d \rangle =
0$. In this way we view $\hgf^*$ as a subspace of $\hg^*$. Similarly,
we let $K^*$ and $\delta$ denote the linear forms on $\hg$ defined by
\begin{align*}
\langle K^*, \hgf \oplus \C d \rangle=0, &\quad \langle K^*, K \rangle=1; \\
\langle \delta, \hgf \oplus \C K \rangle=0, &\quad \langle \delta, d \rangle=1.
\end{align*}
This gives
a decomposition 
\[
\hg^* = \hgf^*
\oplus \C K^* \oplus \C \delta.
\]
Let $P$ denote the \emph{weight lattice}:
\[
P := \{ \la \in \hg^* \; | \; \langle \la, h_i \rangle \in \Z \text{
  for all }0 \le i \le N-1 \}.
\]
The simple roots $\a_i \in \hg^*$ are
given by
\[
\a_0 = \delta - (\e_N - \e_1), \quad \a_i = \e_{i+1} - \e_{i} \quad
\text{for $1 \le i \le N-1$.}
\]

\begin{rmk}
The notation $\hgl_N$ is often used in the literature for the Lie algebra $\mathfrak{L}:=\gl_N(\C[t,t^{-1}]) \oplus \C K \oplus \C d$, with the obvious extension of the bracket considered above on $\widehat{\mathfrak{sl}}_N$.
Our Lie algebra $\hgl_N$ is a ``more naive'' version, namely the Kac--Moody Lie algebra associated with the affine Cartan matrix of type $\mathbf{A}$, and its realization $(\fh, \{\alpha_i, \, i=0, \cdots, N-1\}, \{h_i, \, i=0, \cdots, N-1\})$. In practice we will only consider representations on which $K$ acts trivially, i.e.~representations of the quotient $\hgl_N / \C \cdot K$, which is naturally a Lie subalgebra in $\mathfrak{L}/\C \cdot K$.
\end{rmk}

\subsection{The natural representation of $\hgl_N$} 
\label{sec:natglN}

Let $A = \bigoplus_{i = 1}^N \C a_i$ denote the natural module
for $\gl_N(\C)$. Then $A \otimes_\C \C[t,t^{-1}]$ is
naturally a module for $\mathfrak{sl}_N(\C[t,t^{-1}])$ via
\[
(x \otimes t^m) \cdot (a \otimes t^n) := x(a) \otimes t^{m+n} \quad
\text{for all $x \in \mathfrak{sl}_N(\C)$, $a \in A$ and $m,n \in \Z$}.
\]
We may extend this action to an action of $\hgl_N$ where $K$ acts by $0$, $d$ acts via
\[
d \cdot (a \otimes t^m) := ma \otimes t^m \quad \text{for all
    $a \in A$ and $m \in \Z$,}
\]
and the identity matrix acts as the identity.
We call the resulting $\hgl_N$-module the \emph{natural module} and
denote it $\nat_N$.

Given $\la \in \Z$ write $\la = \mu N + \nu$ with $1 \le \nu \le N$ and set
\[
m_\la := a_{\nu} \otimes t^{\mu}.
\]
Then $\nat_N = \bigoplus_{\la\in \Z} \C m_\la$
and the action of the Chevalley elements is given by
\begin{equation}
\label{eqn:nat-action-f}
e_i(m_\la) = \begin{cases} m_{\la + 1} & \text{if $i \equiv \la \mod N$;}
\\ 0 & \text{otherwise}  \end{cases}
\end{equation}
and
\begin{equation}
\label{eqn:nat-action-e}
f_i(m_\la) = \begin{cases} m_{\la - 1} & \text{if $i \equiv \la-1 \mod N$;} \\
0 & \text{otherwise.} \end{cases}
\end{equation}
If we write $\la = \mu N + \nu$ as above then each $m_\la$ is a weight
vector of weight $\e_{\nu} +\mu\delta$. In particular, all weight
spaces in $\nat_N$ are $1$-dimensional.

\subsection{Realization of $\bigwedge^n \nat_p$ as a Grothendieck group}
\label{ss:cat-GLn-Groth-group}

We assume from now on in this section that $N=p$.
We set
$G = \mathrm{GL}_n(\Bbbk)$, and we let $T \subset B \subset G$
be the maximal torus of diagonal matrices and the Borel subgroup of lower triangular matrices. With these data we use the notation of~\S\ref{ss:definitions-G} (but we do not have to assume that $p>h=n$ at this point).
We have
\[
\bX = X^*(T) = \bigoplus_{i=1}^n \Z \chi_i,
\]
where $ \chi_i \colon T \to \Gm $ is
given by
\begin{align*}
 \textrm{diag}(x_1, \dots, x_n) &\mapsto x_i.
\end{align*}
We will often write an element $\sum_{i=1}^n \la_i\chi_i \in \bX$ as the $n$-tuple $(\la_1,
\dots, \la_n)$. The simple roots are the characters $\chi_i-\chi_{i+1}$ for $i \in \{1, \cdots, n-1\}$, and in particular the cone of dominant weights is
\[
\bX^+ = \{ (\la_1, \dots, \la_n) \in \bX \; | \; \la_1 \ge \la_2 \ge \dots \ge \la_n \}.
\]
We also set
\[
\varsigma=(0,-1, \cdots, -n+1) \in \bX.
\]
Note that $\langle \varsigma, \beta \rangle=1$ for any simple root $\beta$, so that $w \hdot \lambda = w(\lambda + \varsigma)-\varsigma$ for all $w \in \Waff$ and $\lambda \in \bX$.

Recall that $\Rep(G)$ denotes the abelian category of finite dimensional rational representations of $G$, and that
given $\la \in \bX^+$ we have the
corresponding Weyl module $\Sta(\la) \in \Rep(G)$ with highest weight $\lambda$.  The classes of Weyl
modules give us a basis for the Grothendieck group:
\[
[\Rep(G)] = \bigoplus_{\la \in \bX^+} \Z [\Delta(\la)].
\]

Let $V=\Bbbk^n$ denote the natural representation of $G$ (which is isomorphic to $\Sim(\chi_1)$) and let $V^*$ be its
dual (which is isomorphic to $\Sim(-\chi_n)$). We have exact endofunctors of $\Rep(G)$:
\[
E := V \otimes (-) \quad \text{and} \quad F := V^* \otimes (-).
\]
Via the unit and counit morphisms
\[
\eta \colon \Bbbk \to V^* \otimes V \quad \text{and} \quad \epsilon \colon V \otimes V^* \to \Bbbk
\]
we may view $(E,F)$ as an adjoint pair of functors. We will denote again by $\eta \colon \id \to FE$ and $\epsilon \colon EF \to \id$ the corresponding adjunction morphisms.

\begin{rmk}
  Below it will be important to know that $F$ is
  also left adjoint to $E$. However it is crucial not to fix a
  preferred adjunction yet.
\end{rmk}

Let $\fg=\mathfrak{gl}_n(\Bbbk)$ denote the Lie algebra of $G$,
equipped with its natural (adjoint) $G$-module structure. Given
any $M \in \Rep(G)$ the action map
\[
a \colon \fg \otimes M \to M
\]
is a morphism of $G$-modules. Using the natural isomorphism
$\fg \cong V^* \otimes V$ and the natural adjunction $(V^* \otimes -, V \otimes -)$ we get
a morphism
\[
\Xbb_M \colon V \otimes M \to V \otimes M 
\]
given by the composition
\[
V \otimes M \xrightarrow{\adj \otimes V \otimes M} V \otimes V^* \otimes V \otimes M \xrightarrow{V \otimes a} V \otimes M
\]
which is functorial in $M$. In this way we obtain an endomorphism $\Xbb
\in \End(E)$. 
Since $F$ is right adjoint to $E$, $\Xbb$ induces an endomorphism of $F$, which by abuse we will also denote by $\Xbb$. More concretely,
$\Xbb_M \colon V^* \otimes M \to V^* \otimes M$ is given by the composition
\[
V^* \otimes M \xrightarrow{\eta \otimes V^* \otimes M}
V^* \otimes V \otimes V^* \otimes M \xrightarrow{V^* \otimes \Xbb_{V^*
    \otimes M}}  V^* \otimes V \otimes
V^* \otimes M \xrightarrow{V^* \otimes \epsilon \otimes M}  V^* \otimes M.
\]


We set
\[
\Omega := \sum_{i,j=1}^n e_{i,j} \otimes e_{j,i} \in \fg \otimes \fg,
\]
where $(e_{i,j})_{i,j \in \{1, \cdots, n\}}$ are the matrix units in $\fg=\mathfrak{gl}_n(\Bbbk)$.
The following lemma (taken from~\cite{CR}) can be easily checked by explicit computation.

\begin{lem}
For any $M$ in $\Rep(G)$:
\begin{enumerate}
\item
the endomorphism $\Xbb_M$ of $EM=V \otimes M$ is given by the action of $\Omega$;
\item
the endomorphism $\Xbb_M$ of $FM=V^* \otimes M$ is given by $-n \cdot \id - \Omega$.
\end{enumerate}
\end{lem}

For any $a \in \Bbbk$, let $E_a$
(resp. $F_a)$ denote the summand of $E$ (resp. $F$) given by the
generalized $a$-eigenspace of $\Xbb$ acting on $E$ (resp. $F$).

\begin{lem}
\label{lem:adjunctions-a}
  For any $a \in \Bbbk$, $\epsilon$ and $\eta$ induce units and counits $\eta_a \colon \id \to F_a E_a$ and $\epsilon_a \colon E_a F_a \to \id$
  making $(E_a, F_a)$ into an adjoint pair of functors.
\end{lem}

%

\begin{proof}
To prove the lemma it suffices to prove that $\eta$ and $\epsilon$ factor through morphisms
\[
\id \to \bigoplus_{a \in \Bbbk} F_a E_a \qquad \text{and} \qquad \bigoplus_{a \in \Bbbk} E_a F_a \to \id.
\]

First, the morphism $\eta_M \colon M \to FEM = V^* \otimes V \otimes M$ is defined by
\[
m \mapsto \sum_{i=1}^n v_i^* \otimes v_i \otimes m,
\]
where $(v_i)_{i \in \{1, \cdots, n\}}$ is the standard basis of $V$. Now $\Omega$ acts on $F(EM) = V^* \otimes (EM)$ as the action of
\[
\sum_{i,j=1}^n e_{i,j} \otimes e_{j,i} \otimes 1 + e_{i,j} \otimes 1 \otimes e_{j,i} \quad \in \fg \otimes \fg \otimes \fg
\]
on $V^* \otimes V \otimes M$.
It can be easily checked that this element acts on any vector of the form $\sum_{k=1}^n v_k^* \otimes v_k \otimes m$ as the element
\[
-n \cdot (1 \otimes 1 \otimes 1) - (1 \otimes \Omega).
\]
Therefore, such an element can belong to $F_a E_b M$ only if
\[
-n-a=-n-b,
\]
i.e.~only if $a=b$, proving the claim about $\eta$.

Now, consider the morphism $\epsilon_M \colon EFM=V \otimes V^* \otimes M \to M$, which is defined by
\[
v \otimes \xi \otimes m \mapsto \xi(v) \cdot m.
\]
The element $\Omega$ acts on $E(FM) = V \otimes (FM)$ as the action of
\[
\sum_{i,j=1}^n e_{i,j} \otimes e_{j,i} \otimes 1 + e_{i,j} \otimes 1 \otimes e_{j,i} \in \fg \otimes \fg \otimes \fg
\]
on $V \otimes V^* \otimes M$.
Now for any $x \in EFM$ we have
\begin{multline*}
\epsilon_M\left( \Bigl( \sum_{i,j=1}^n e_{i,j} \otimes e_{j,i} \otimes 1 + e_{i,j} \otimes 1 \otimes e_{j,i} \Bigr) \cdot x \right) \\
= \epsilon_M \Bigl( (-n \cdot (1 \otimes 1 \otimes 1) - (1 \otimes \Omega)) \cdot x\Bigr).
\end{multline*}
Therefore, if $x \in E_b F_a M$, then $\epsilon_M(x)=0$ unless the operators
\[
\sum_{i,j=1}^n e_{i,j} \otimes e_{j,i} \otimes 1 + e_{i,j} \otimes 1 \otimes e_{j,i} \text{ and } -n \cdot (1 \otimes 1 \otimes 1) - (1 \otimes \Omega)
\]
have the same generalized eigenvalue on $x$, i.e.~unless $a=-n-(-n-b)$, which finishes the proof.
\end{proof}

In the following we identify $\Z/p\Z$ with the prime subfield of
$\Bbbk$. 
We denote by $\bX/(W,\hhdot)$ the set of orbits of the ``dot-action" of the affine Weyl group $\Waff$ on $\bX$. For any $c \in \bX/(W,\hhdot)$, we denote by
\[
\Repp_c(G)
\]
the Serre subcategory of $\Rep(G)$ generated by the simple objects $\Sim(\lambda)$ for $\lambda \in c \cap \bX^+$. By the linkage principle (see~\cite[Corollary~II.6.17]{jantzen}) we have a decomposition
\begin{equation}
\label{eqn:decomposition-Rep}
\Rep(G) = \bigoplus_{c \in \bX/(W,\, \raisebox{1.7pt}{\text{\circle*{1.7}}})} \Repp_c(G).
\end{equation}

\begin{prop}
\label{prop:natcat}
  \begin{enumerate}
  \item
\label{it:natcat-eigenvalues}
We have $E_a=0$ and $F_a=0$ unless $a \in \Z/p\Z$; therefore we have
\[
E = \bigoplus_{a \in \Z/p\Z} E_a \qquad \text{and} \qquad F = \bigoplus _{a \in
      \Z/p\Z} F_a.
\]
  \item
\label{it:natcat-EF}  
The isomorphism of $\C$-vector spaces
\begin{align*}
\varphi \colon \C \otimes_\Z [\Rep(G)] & \simto \bigwedge\hspace{-3pt}{}^n \, \nat_p
\end{align*}
given by
\begin{align*}
  \varphi([\Sta(\la)]) = m_{\la_1} \wedge m_{\la_2-1} \wedge \dots
  \wedge m_{\la_n-n+1} \quad \text{for $\la \in \bX^+$}
\end{align*}
satisfies $\varphi \circ [E_a] = e_a \circ \varphi$ and $\varphi \circ [F_a] =
f_a \circ \varphi$. In particular, identifying $\Z/p\Z$ with $\{0, \cdots, p-1\}$ in the obvious way, the exact functors $E_a$, $F_a$ are part of
an action of $\hgl_p$ on $\C \otimes_\Z [\Rep(G)]$.
\item 
\label{it:natcat-weights}
Upon passing to Grothendieck
  groups the decomposition~\eqref{eqn:decomposition-Rep}
induces the decomposition of $\bigwedge^n \nat_p$ into weight spaces (via the
isomorphism $\varphi$ of~\eqref{it:natcat-EF}).
  \end{enumerate}
\end{prop}

\begin{rmk}
Sometimes one considers the Lie algebra $\hgl_p$ without the degree
  operator $d$. However 
  for statement~\eqref{it:natcat-weights} it is essential that we include $d$ in our
  action (as can already be seen in the case $n = 1$).
\end{rmk}

\begin{proof}
For simplicity we assume that $p \neq 2$; the necessary fix to treat the case $p=2$ is discussed in \cite[\S 7.5.2]{CR}.
If we
denote by 
\[
C = \sum_{1 \le i, j \le n} e_{j,i}e_{i,j} \in
U(\fg)
\]
the Casimir element (with respect to the trace form on the
natural module) then
in $U(\fg) \otimes
U(\fg)$ we have
\begin{align}
\label{eq:Cident}
\Omega = \frac{1}{2} \bigl( \Delta(C) - C \otimes 1 - 1 \otimes C \bigr),
\end{align}
where $\Delta \colon U(\fg) \to U(\fg) \otimes
U(\fg)$ denotes the comultiplication. 

After writing
\[
C = 2 \sum_{1 \le i < j \le n} e_{j,i} e_{i,j} + \sum_{1 \le i \le n} e_{i,i}^2 +
\sum_{1 \le i<
  j \le n}(e_{i,i} - e_{j,j})
\]
we see that $C$ acts on $\Delta(\la)$ as
\[
b_\la := \sum_{1 \le i \le n} \la_i^2 + \sum_{1 \le i < j \le n} (\la_i - \la_j).
\]

We have
\[
E \Sta(\la) = V \otimes \Sta(\la) = V \otimes \bigl( \Ind_B^G(-w_0 \lambda) \bigr)^* \cong \bigl( \Ind_B^G(V^* \otimes \Bbbk_B(-w_0 \lambda)) \bigr)^*.
\]
Now the $B$-module $V^* \otimes \Bbbk_B(-w_0 \lambda)$ admits a filtration with subquotients the $B$-modules $\Bbbk_B(-w_0 \lambda-\chi_i)$, and each weight $-w_0 \lambda-\chi_i$ is either dominant or has pairing $-1$ with some simple root; in any case we have $R^1 \Ind_B^G(-w_0 \lambda-\chi_i)=0$, see~\cite[Proposition~II.4.5 and Proposition~II.5.4(a)]{jantzen}. Hence the filtration on $V^* \otimes \Bbbk_B(-w_0 \lambda)$ induces a filtration on $E \Sta(\la)$. Since $w_0$ permutes the set $\{\chi_1, \cdots, \chi_n\}$, the subquotients in this filtration are the $G$-modules
$\Sta(\la + \chi_i)$ for all $i \in \{1, \cdots, n\}$ such that $\la + \chi_i \in \bX^+$. From \eqref{eq:Cident} it follows
that $\Omega$ acts on the subquotient of $E \Sta(\la)$
isomorphic to $\Sta(\la + \chi_i)$ as multiplication by
\[
\frac{1}{2} \bigl( b_{\la + \chi_i} - b_{\chi_1} - b_{\la} \bigr) = \la_i - i + 1.
\]
In particular, all the eigenvalues of $\Xbb_{\Sta(\lambda)}$ on $E\Sta(\lambda)$ belong to $\Z/p\Z$.
Because any simple object in $\Rep(G)$ is a quotient of some $\Sta(\la)$ and
every object in $\Rep(G)$ is of finite length, we deduce that, for any $M \in \Rep(G)$, all eigenvalues
of $\Xbb_M$ on $E M$ belong to the prime field $\Z/p\Z$. Similar arguments apply to $F$, and~\eqref{it:natcat-eigenvalues} is proved.

We deduce from the above calculation that
\begin{equation}
\label{eqn:nat-action-F}
[E_a] \cdot [\Sta(\la)] = \sum_{\substack{\la \to_a \mu \\ \mu \in \bX^+}} [\Delta(\mu)]
\end{equation}
where we write $\la \to_a \mu$ if there exists an $i$ such that 
\[
\text{$\la_j
= \mu_j$  for  $i \ne j$, $\mu_i = \la_i + 1$ and $\la_i - i + 1 \equiv a
\mod p$.}
\]

Similar considerations, using the fact that
\[
-n-\frac{1}{2} \bigl( b_{\mu - \chi_i} - b_{-\chi_n} - b_{\mu} \bigr) = \mu_i - i,
\]
show that
\begin{equation}
\label{eqn:nat-action-E}
[F_a] \cdot [\Delta(\mu)] = \sum_{\substack{\la \to_a \mu \\ \la \in \bX^+}} [\Delta(\la)].
\end{equation}
The relations $\varphi \circ [E_a] = e_a \circ \varphi$ and $\varphi \circ [F_a] =
f_a \circ \varphi$ are now obtained by comparing formulas~\eqref{eqn:nat-action-f}--\eqref{eqn:nat-action-e} with~\eqref{eqn:nat-action-F}--\eqref{eqn:nat-action-E}, and~\eqref{it:natcat-EF} is proved.

It remains to check~\eqref{it:natcat-weights}. The weight of a
vector
\[
m_{\zeta_1} \wedge m_{\zeta_2} \wedge \dots \wedge m_{\zeta_n} \in
\bigwedge\hspace{-3pt}{}^n \, \nat_p
\]
is $k \delta + \sum_{i = 1}^p n_i \e_i$, where $n_i = | \{ j \mid \zeta_j \equiv i \mod p
\}|$ and $k = \sum_{i=1}^n \zeta_i^1$, where $\zeta_i^1$ is uniquely defined so that $\zeta_i = p\zeta_i^1 +
\zeta_i^0$ with $1 \le \zeta_i^0 \le p$. Now given two weights $\la, \mu \in \bX^+$, the corresponding
vectors 
\[
m_{\la_1} \wedge m_{\la_2-1} \wedge \dots \wedge m_{\la_n-n+1} \quad
\text{and} \quad 
m_{\mu_1} \wedge m_{\mu_2-1} \wedge \dots \wedge m_{\mu_n-n+1} 
\]
belong to the same weight space if and only if the $n$-tuples
$\la + \varsigma$ and $\mu + \varsigma$ can be related by permutations and
addition of weights of the form $p(\chi_i - \chi_j)$ with $1 \leq i,j \leq n$. This is precisely the
condition for $\la$ and $\mu$ to be in the same $(W,\hhdot)$-orbit 
on $\bX$.
\end{proof}

\begin{rmk}
\label{rmk:weights}
As seen in the proof of Proposition~\ref{prop:natcat}\eqref{it:natcat-weights}, the set of weights of the representation $\bigwedge^n \nat_p$ is
\[
P(\bigwedge\hspace{-3pt}{}^n \, \nat_p) = \left\{k\delta + \sum_{i=1}^p n_i \e_i : k \in \Z, \, n_i \in \Z_{\geq 0}, \, \sum_{i=1}^p n_i = n \right\},
\]
and this set is in a natural bijection with $\bX/(W,\hhdot)$; we denote this bijection by
\[
\imath_{n} \colon P(\bigwedge\hspace{-3pt}{}^n \, \nat_p) \simto \bX/(W,\hhdot).
\]
\end{rmk}





\subsection{$\Rep(G)$ as a $2$-representation}
\label{Rep-2-rep}

We now describe how the
above action of the Lie algebra $\hgl_p$ on $\C \otimes_\Z [\Rep(G)]$ can be
upgraded to a categorical action of the Khovanov--Lauda--Rouquier
$2$-category $\mathcal{U}(\hgl_p)$.

\subsubsection{The degenerate affine Hecke algebra}
\label{sss:daha}

Let $\F$ be a field, let $m \in \Z_{\geq 1}$, and let $\F[X_1, \cdots , X_m]$ denote a polynomial
ring, acted upon naturally by the symmetric group $S_m$. Let $s_i \in
S_m$ denote the simple transposition $(i,i+1)$, and let 
\[
\partial_i (f ) := \frac{f - s_i(f)}{X_{i+1} - X_i}
\]
denote the Demazure operator. (Note that this is a slightly non-standard definition.) Recall that the degenerate affine Hecke algebra
$\overline{H}_m$ is the algebra which:
\begin{enumerate}
\item  is isomorphic to $\F S_m \otimes_\F
\F[X_1, \cdots, X_m]$ as an $\F$-vector space;
\item has $\F S_m$ and $\F[X_1, \cdots, X_m]$ as subalgebras;
\item satisfies the relations
\[
T_i \cdot f-s_i(f) \cdot T_i   = \partial_i (f)
\]
for all $f \in \F[X_1, \cdots, X_m]$ and $1 \le i \le m-1$,
where $ T_i$ denotes the element $s_i \in S_m$ viewed in $\overline{H}_m$.
\end{enumerate}

\begin{rmk}
\begin{enumerate}
\item
  It is perhaps more aesthetically pleasing to define $\overline{H}_m$ by
  generators and relations (in $T_i$'s for $1 \le i \le m-1$ and
  $X_i$'s for $1 \le i \le m$). For instance we have
  $T_i \cdot X_{i+1} = X_i \cdot T_i + 1$.
\item
  The degenerate affine Hecke algebra has a ``polynomial'' module
\[
\overline{H}_m \otimes_{\F S_m} \F = \F[X_1, \cdots, X_m]
\]
where $f$ acts by multiplication and $T_i$ acts as
\[
f \mapsto s_i(f) + \partial_i(f).
\]
\end{enumerate}
\end{rmk}

We can regard any $\overline{H}_m$-module as a quasi-coherent
sheaf on the affine space $\Spec \F[X_1, \cdots, X_m] = \mathbb{A}_\F^m$. Given such a module $M$, we
denote by $M_a$ the submodule consisting of sections set-theoretically supported in
the point $a \in \mathbb{A}_\F^n$. Let $\Gamma \subset \F$ denote the prime subring. We denote by $\overline{\cC}_\Gamma$ the
full subcategory of the category of $\overline{H}_m$-modules whose objects are the modules $M$ such that
\[
M = \bigoplus_{a \in \Gamma^m} M_a.
\]
(In other words, $\overline{\cC}_\Gamma$ is the category of
$\overline{H}_m$-modules which are the direct sum of their
generalized eigenspaces for each $X_i$, and such that all eigenvalues for all
$X_i$ belong to $\Gamma$.)


\subsubsection{Unravelling Rouquier's notation}
\label{subsec:R}

Recall that we have fixed
a field $\F$. We view the prime subring $\Gamma \subset \F$ as a quiver with arrows $i \to i+1$ for all $i \in
\Gamma$. (Hence $\Gamma$ is of type $A_\infty$ if $\F$ is of characteristic
$0$, and is of type $\widetilde{A}_{\ell-1}$ if $\F$ is of characteristic
$\ell>0$.)

Let us write $i \rel j$ if $i \to j$ or $j \to
i$, and $i \norel j$ if neither $i \to j$ nor $j \to
i$. Then if
we specialize the notation of \cite[\S\S 3.2.3--3.2.4]{R2KM} down to our situation we
have:
\begin{enumerate}
\item $I = \widetilde{I} = \Gamma$ and $a = \id$;
\item 
\[
i \cdot j =
\begin{cases}
  2 & \text{if $i = j$;} \\
  -1 & \text{if $i \rel j$;} \\
  0 & \text{otherwise}
\end{cases}
\]
(the Cartan matrix);
\item 
\[
Q_{ij} (u,v)=
\begin{cases}
  0 & \text{if $i = j$;} \\
  1 & \text{if $i \ne j$ and $i \norel j$;} \\
  v-u & \text{if $i \to j$;} \\
  u-v & \text{if $j \to i$.}
\end{cases}
\]
\end{enumerate}

We define scalars $t_{ij}$
for $i, j \in \Gamma$ with $i \ne j$ as follows:
\begin{gather*}
t_{ij} = \begin{cases} -1 & \text{if $i \to j$;}\\
1 & \text{if $j \to i$ or $i \norel j$.}
     \end{cases} \\
\end{gather*}

\subsubsection{The KLR algebra}
\label{sss:def-KLR}

Here we recall the definition of the quiver Hecke
algebra (or KLR algebra) associated with $\Gamma$, see~\cite[\S 3.2.1]{R2KM}. As noted in \cite[Remark 3.6]{R2KM} it
is most natural to view the quiver Hecke algebra as a category; this
is the approach we take here.

Let $\Gamma^m$ denote the set of $m$-tuples of elements of
$\Gamma$ with its natural $S_m$-action. Consider the $\F$-linear category $H_m(\Gamma)$ with objects $\nu
\in \Gamma^m$ and morphisms generated by morphisms
\begin{align*}
x_{z,\nu} \in \End(\nu)& \quad \text{for $1 \le z \le m$,} \\
\tau_{z,\nu} \in \Hom(\nu, s_z\nu)& \quad \text{for $1 \le z < m$,}
\end{align*}
satisfying the following relations:
\begin{equation}
\label{eqn:rel-KLR-1}
x_{z,\nu} x_{z',\nu} = x_{z',\nu} x_{z,\nu};
\end{equation}
\begin{equation}
\label{eqn:rel-KLR-2}
\tau_{z,s_z(\nu)} \tau_{z,\nu} = \begin{cases}
0 & \text{if $\nu_z=\nu_{z+1}$}; \\
1_{\nu} & \text{if $\nu_z \neq \nu_{z+1}$ and $\nu_z \norel \nu_{z+1}$;} \\
t_{\nu_z, \nu_{z+1}} x_{z,\nu} + t_{\nu_{z+1}, \nu_z} x_{z+1, \nu} & \text{if $\nu_z \rel \nu_{z+1}$;}
\end{cases}
\end{equation}
\begin{equation}
\label{eqn:rel-KLR-3}
\tau_{z,s_{z'}(\nu)} \tau_{z',\nu} = \tau_{z', s_z(\nu)} \tau_{z,\nu} \quad \text{if $|z-z'|>1$;}
\end{equation}
\begin{multline}
\label{eqn:rel-KLR-4}
\tau_{z+1, s_z s_{z+1}(\nu)}\tau_{z,s_{z+1}(\nu)} \tau_{z+1, \nu} - \tau_{z,s_{z+1} s_z(\nu)} \tau_{z+1, s_z(\nu)} \tau_{z,\nu} \\
= \begin{cases}
t_{\nu_z, \nu_{z+1}} \cdot 1_\nu & \text{if $\nu_z=\nu_{z+2}$ and $\nu_{z+2} \rel \nu_{z+1}$;} \\
0 & \text{otherwise;}
\end{cases}
\end{multline}
\begin{equation}
\label{eqn:rel-KLR-5}
\tau_{z,\nu} x_{z',\nu} - x_{s_z(z'),s_z(\nu)} \tau_{z,\nu} = \begin{cases}
-1_{\nu} & \text{if $z=z'$ and $\nu_z = \nu_{z+1}$;} \\
1_{\nu} & \text{if $z'=z+1$ and $\nu_z=\nu_{z+1}$;} \\
0 & \text{otherwise.}
\end{cases}
\end{equation}
(Here, $1_\nu$ is the identity morphism of the object $\nu$.)

\subsubsection{Brundan--Kleshchev--Rouquier equivalence}

A representation of $H_m(\Gamma)$ is by
definition an $\F$-linear functor from $H_m(\Gamma)$ to the category of $\F$-vector spaces. In more concrete terms it
consists of 
\begin{itemize}
\item
an $\F$-vector space $V_\nu$ for each $\nu \in \Gamma^m$;
\item
endomorphisms $x_{z,\nu}$ of $V_\nu$ for all $\nu \in \Gamma^m$ and $1 \le z \le m$;
\item
morphisms
$\tau_{z,\nu} \colon V_{\nu} \to V_{s_z(\nu)}$ for all $\nu \in \Gamma^m$ and $1 \le z < m$,
\end{itemize}
satisfying the relations from~\S\ref{sss:def-KLR}.

We write $H_m(\Gamma)\Mod_0$ for the full subcategory of the category of representations
of $H_m(\Gamma)$ consisting of objects on which $x_{z,\nu}$ is locally
nilpotent for all $\nu$ and $1 \le z \le m$. Recall that $\overline{\cC}_\Gamma$ is the category of modules over
the degenerate affine Hecke algebra $\overline{H}_m$ 
which are direct sums of their generalized eigenspaces for each $X_i$, with eigenvalues in $\Gamma$;
see~\S\ref{sss:daha}. Given $M \in \overline{\cC}_\Gamma$ and $\nu \in \Gamma^m$
we denote by $M_\nu$ the generalized $\nu$-eigenspace:
\[
M_\nu := \{ m \in M \mid (X_z - \nu_z)^Nm = 0 
\text{ for all $1 \le z \le m$ and } N \gg 0 \}.
\]

The following theorem is due to Brundan-Kleshchev~\cite{BK} and Rouquier~\cite[Theorem~3.16]{R2KM}.

\begin{thm}
\label{thm:BKR}
There exists an equivalence of categories
\[
\overline{\cC}_\Gamma \simto 
H_m(\Gamma)\Mod_0
\]
which associates to $M \in \overline{\cC}_\Gamma$ the representation $V$ defined by
\begin{enumerate}
\item $V_{\nu} = M_\nu$  for all $\nu \in \Gamma^m$;
\item
$x_{z,\nu} := X_z - \nu_z$ for all $\nu \in \Gamma^m$ and $1 \le z \le m$; 
\item
\label{it:BKR-functor}
$\tau_{z,\nu}$ given by the formulas
\[
  \tau_{z,\nu} := \begin{cases}
  \frac{1}{1 + X_z - X_{z+1}}(T_z-1) & 
    \text{if $\nu_z = \nu_{z+1}$;} \\
(X_z - X_{z+1})T_z + 1 & \text{if $\nu_{z+1} = \nu_z +1$;}\\
\frac{X_z - X_{z+1}}{1 + X_z - X_{z+1}}(T_z-1) + 1 & \text{otherwise}
\end{cases}
\]
for all $\nu \in \Gamma^m$ and $1 \le z < m$.
\end{enumerate}
\end{thm}

\begin{rmk} 
\begin{enumerate}
\item
There seems to be a typo in the third formula in~\cite[Theorem~3.16]{R2KM}. We follow the formulas of~\cite[Proposition~3.15]{R2KM}.
\item
This theorem is a highly non-trivial 
  calculation. It is not even obvious that the formulas in~\eqref{it:BKR-functor}
  do indeed give morphisms $\tau_{z,\nu} \colon M_\nu \to M_{s_i(\nu)}$.
\end{enumerate}
\end{rmk}

\subsubsection{The $2$-Kac--Moody algebra}
\label{sss:2KM-def}

From now on we come back to the setting of~\S\ref{ss:cat-GLn-Groth-group}.
Recall in particular the affine Lie
algebra $\hgl_p$ defined in \S \ref{slp}. To the triple
\[
(\hg, \, \{
\a_i : i \in \{0, \cdots, p-1\} \} \subset \hg^*, \, \{ h_i : i \in \{0, \cdots, p-1\}\} \subset
\hg)
\]
Rouquier~\cite{R2KM} has associated a strict $\Bbbk$-linear additive $2$-category $\mathcal{U}(\hgl_p)$ categorifying (an idempotented form of)
the enveloping algebra of $\hgl_p$. The definition of this $2$-category is recalled in~\cite[Definition~1.1]{Brundan}, and we will follow the (diagrammatic) notation from~\cite{Brundan}. This definition depends on the choice of some
additional data; we choose them as follows:
\begin{enumerate}
\item our ground field is chosen to be $\Bbbk$;
\item the scalars $t_{ij}$ are given as in~\S\ref{subsec:R};
\item the scalars $s_{ij}^{pq}$ are chosen identically zero.
\end{enumerate}
The category $\mathcal{U}(\hgl_p)$ has objects $P$ (the weights of $\hgl_p$, as defined in~\S\ref{slp}), generating $1$-morphisms $E_i 1_\lambda \colon \lambda \to \lambda+\alpha_i$ (which we will depict as an upward arrow decorated by $\lambda$ in the right region) and $F_i 1_\lambda \colon \lambda \to \lambda-\alpha_i$ (which we will depict as a downward arrow decorated by $\lambda$ in the right region), and generating $2$-morphisms
\[
\begin{array}{c}
\begin{tikzpicture}[baseline = 0]
\node at (0.2,.2) {\tiny $\lambda$};
	\draw[<-,thick] (0,.6) to (0,-.2);
  \node at (0,-.4) {\tiny $i$};
      \node at (0,0.2) {$\bullet$};
\end{tikzpicture}
\end{array}, \quad
\begin{array}{c}
\begin{tikzpicture}[baseline = 0]
	\draw[<-,thick] (0.25,.6) to (-0.25,-.2);
	\draw[->,thick,green] (0.25,-.2) to (-0.25,.6);
\node at (0.3,.2) {\tiny $\lambda$};
  \node at (-0.25,-.4) {\tiny $i$};
   \node at (0.25,-.4) {\tiny $j$};
\end{tikzpicture}
\end{array}, \quad
\begin{array}{c}
      \begin{tikzpicture}[baseline = 0]
        \draw[->,thick] (-0.3,0.3) to[out=-90,in=180] (0,-0.2) to[out=0,in=-90] (0.3,0.3);
        \node at (0.5,.2) {\tiny $\lambda$};
          \node at (-0.3,.5) {\tiny $i$};
      \end{tikzpicture}
\end{array}, \quad
\begin{array}{c}
      \begin{tikzpicture}[baseline = 0]
        \draw[->,thick] (-0.3,-0.2) to[out=90,in=180] (0,0.3) to[out=0,in=90] (0.3,-0.2);
        \node at (0.4,.2) {\tiny $\lambda$};
          \node at (-0.3,-.4) {\tiny $i$};
      \end{tikzpicture}
\end{array}.
\]
These $2$-morphisms are required to satisfy a number of relations described in full in~\cite[Definition~1.1]{Brundan}. Here we recall only the following relations.
\begin{equation}
\label{eqn:rel-2KM-1}
\begin{array}{c}
\begin{tikzpicture}[baseline = 0]
	\draw[<-,thick] (0.25,.6) to (-0.25,-.2);
	\draw[->,thick,green] (0.25,-.2) to (-0.25,.6);
  \node at (-0.25,-.4) {\tiny $i$};
   \node at (0.25,-.4) {\tiny $j$};
  \node at (.3,.25) {\tiny $\lambda$};
      \node at (-0.13,-0.02) {$\bullet$};
\end{tikzpicture}
\end{array}
-
\begin{array}{c}
\begin{tikzpicture}[baseline = 0]
	\draw[<-,thick] (0.25,.6) to (-0.25,-.2);
	\draw[->,thick,green] (0.25,-.2) to (-0.25,.6);
  \node at (-0.25,-.4) {\tiny $i$};
   \node at (0.25,-.4) {\tiny $j$};
  \node at (.3,.25) {\tiny $\lambda$};
      \node at (0.13,0.42) {$\bullet$};
\end{tikzpicture}
\end{array}
=
\begin{array}{c}
\begin{tikzpicture}[baseline = 0]
	\draw[<-,thick] (0.25,.6) to (-0.25,-.2);
	\draw[->,thick,green] (0.25,-.2) to (-0.25,.6);
  \node at (-0.25,-.4) {\tiny $i$};
   \node at (0.25,-.4) {\tiny $j$};
  \node at (.3,.25) {\tiny $\lambda$};
      \node at (-0.13,0.42) {$\color{green}\bullet$};
\end{tikzpicture}
\end{array}
-
\begin{array}{c}
\begin{tikzpicture}[baseline = 0]
 	\draw[<-,thick] (0.25,.6) to (-0.25,-.2);
	\draw[->,thick,green] (0.25,-.2) to (-0.25,.6);
  \node at (-0.25,-.4) {\tiny $i$};
   \node at (0.25,-.4) {\tiny $j$};
  \node at (.3,.25) {\tiny $\lambda$};
     \node at (0.13,-0.02) {$\color{green}\bullet$};
 \end{tikzpicture}
\end{array}
=
\begin{cases}
\begin{array}{c}
\begin{tikzpicture}[baseline = 0]
 	\draw[->,thick,green] (0.08,-.3) to (0.08,.4);
	\draw[->,thick] (-0.28,-.3) to (-0.28,.4);
   \node at (-0.28,-.4) {\tiny $i$};
   \node at (0.08,-.4) {\tiny $j$};
 \node at (.28,.06) {\tiny $\lambda$};
\end{tikzpicture}
\end{array}
&\text{if $i=j$;}\\
0&\text{otherwise,}\\
\end{cases}
\end{equation}
\begin{equation}
\label{eqn:rel-2KM-2}
\begin{array}{c}
\begin{tikzpicture}[baseline = 0]
	\draw[->,thick] (0.28,.4) to[out=90,in=-90] (-0.28,1.1);
	\draw[->,thick,green] (-0.28,.4) to[out=90,in=-90] (0.28,1.1);
	\draw[-,thick,green] (0.28,-.3) to[out=90,in=-90] (-0.28,.4);
	\draw[-,thick] (-0.28,-.3) to[out=90,in=-90] (0.28,.4);
  \node at (-0.28,-.45) {\tiny $i$};
  \node at (0.28,-.45) {\tiny $j$};
   \node at (.43,.4) {\tiny $\lambda$};
\end{tikzpicture}
\end{array}
=
\begin{cases}
0&\text{if $i=j$;}\\
\begin{array}{c}
\begin{tikzpicture}[baseline = 0]
	\draw[->,thick,green] (0.08,-.3) to (0.08,.4);
	\draw[->,thick] (-0.28,-.3) to (-0.28,.4);
   \node at (-0.28,-.4) {\tiny $i$};
   \node at (0.08,-.4) {\tiny $j$};
   \node at (.3,.05) {\tiny $\lambda$};
\end{tikzpicture}
\end{array}&\text{if $i \neq j$ and $i \norel j$;}\\
t_{ij}
\begin{array}{c}
\begin{tikzpicture}[baseline = 0]
	\draw[->,thick,green] (0.08,-.3) to (0.08,.4);
	\draw[->,thick] (-0.28,-.3) to (-0.28,.4);
   \node at (-0.28,-.4) {\tiny $i$};
   \node at (0.08,-.4) {\tiny $j$};
   \node at (.3,-.05) {\tiny $\lambda$};
      \node at (-0.28,0.05) {$\bullet$};
\end{tikzpicture}
\end{array}
+
t_{ji}
\begin{array}{c}
\begin{tikzpicture}[baseline = 0]
	\draw[->,thick,green] (0.08,-.3) to (0.08,.4);
	\draw[->,thick] (-0.28,-.3) to (-0.28,.4);
   \node at (-0.28,-.4) {\tiny $i$};
   \node at (0.08,-.4) {\tiny $j$};
   \node at (.3,-.05) {\tiny $\lambda$};
     \node at (0.08,0.05) {$\color{green}\bullet$};
\end{tikzpicture}
\end{array}
&\text{if $i \rel j$,}\\
\end{cases}
\end{equation}
\begin{equation}
\label{eqn:rel-2KM-3}
\begin{array}{c}
\begin{tikzpicture}[baseline = 0]
	\draw[<-,thick] (0.45,.8) to (-0.45,-.4);
	\draw[->,thick,red] (0.45,-.4) to (-0.45,.8);
        \draw[-,thick,green] (0,-.4) to[out=90,in=-90] (-.45,0.2);
        \draw[->,thick,green] (-0.45,0.2) to[out=90,in=-90] (0,0.8);
   \node at (-0.45,-.6) {\tiny $i$};
   \node at (0,-.6) {\tiny $j$};
  \node at (0.45,-.6) {\tiny $k$};
   \node at (.5,-.1) {\tiny $\lambda$};
\end{tikzpicture}
\end{array}
-
\begin{array}{c}
\begin{tikzpicture}[baseline = 0]
	\draw[<-,thick] (0.45,.8) to (-0.45,-.4);
	\draw[->,thick,red] (0.45,-.4) to (-0.45,.8);
        \draw[-,thick,green] (0,-.4) to[out=90,in=-90] (.45,0.2);
        \draw[->,thick,green] (0.45,0.2) to[out=90,in=-90] (0,0.8);
   \node at (-0.45,-.6) {\tiny $i$};
   \node at (0,-.6) {\tiny $j$};
  \node at (0.45,-.6) {\tiny $k$};
   \node at (.5,-.1) {\tiny $\lambda$};
\end{tikzpicture}
\end{array}
=
\begin{cases}
t_{ij}
\!
\begin{array}{c}
\begin{tikzpicture}[baseline = 0]
	\draw[->,thick,red] (0.44,-.3) to (0.44,.4);
	\draw[->,thick,green] (0.08,-.3) to (0.08,.4);
	\draw[->,thick] (-0.28,-.3) to (-0.28,.4);
   \node at (-0.28,-.4) {\tiny $i$};
   \node at (0.08,-.4) {\tiny $j$};
   \node at (0.44,-.4) {\tiny $k$};
  \node at (.6,-.1) {\tiny $\lambda$};
\end{tikzpicture}
\end{array}
&\text{if $i=k$ and $k \rel j$;}\\
0&\text{otherwise.}
\end{cases}
\end{equation}

\begin{rmk}
\begin{enumerate}
\item
As usual, we will write $E_i E_j 1_\lambda$ for $(E_i 1_{\lambda+\alpha_j}) \circ (E_j 1_\lambda)$, and similarly for other compositions of functors $E_k 1_\mu$ and $F_k 1_\mu$.
\item
Unless otherwise indicated, our use of colors in this context is just to remind the reader that the labels of the strands might be different (but not necessarily).
\end{enumerate}
\end{rmk}

By definition, a \emph{representation} of $\mathcal{U}(\hgl_p)$ is a $\Bbbk$-linear additive functor from $\mathcal{U}(\hgl_p)$ to the $2$-category of $\Bbbk$-linear additive categories. More concretely, a representation consists of
\begin{itemize}
\item
for each $\lambda \in P$, a $\Bbbk$-linear additive category $\mathscr{C}_\lambda$;
\item
for each $\lambda \in P$ and each $i \in \{0, \cdots, p-1\}$, additive functors
\[
E_i 1_\lambda \colon \mathscr{C}_\lambda \to \mathscr{C}_{\lambda+\alpha_i} \quad \text{and} \quad F_i 1_\lambda \colon \mathscr{C}_\lambda \to \mathscr{C}_{\lambda-\alpha_i};
\]
\item
for each $\lambda \in P$ and $i,j \in \{0, \cdots, p-1\}$, morphisms of functors
\[
x \colon E_i 1_\lambda \to E_i 1_\lambda, \ \tau \colon E_i E_j 1_\lambda \to E_j E_i 1_\lambda, \ \eta \colon \id_{\mathscr{C}_\lambda} \to F_i E_i 1_\lambda, \ \epsilon \colon E_i F_i 1_\lambda \to \id_{\mathscr{C}_\lambda};
\]
\end{itemize}
these morphisms of functors satisfying the relations from~\cite[Definition~1.1]{Brundan}.




\subsubsection{Action of the degenerate affine Hecke algebra on powers of $E$}

The functor $EE = V \otimes V \otimes (-)$ on $\Rep(G)$ has a natural endomorphism
$\Tbb$ given by
\[
\Tbb_M \colon \left\{ 
\begin{array}{ccc}
EEM & \to & EEM \\
v \otimes v' \otimes m & \mapsto & v' \otimes v
\otimes m
\end{array}
\right. .
\]
Recall the endomorphism $\Xbb \in \End(E)$ constructed in~\S\ref{ss:cat-GLn-Groth-group}. The following relation in $\End(EE)$ is fundamental (see
\cite[Lemma~7.21]{CR}):
\begin{equation}
\label{eq:deghecke}
\Tbb_M \circ (V \otimes \Xbb_M) - \Xbb_{V \otimes M} \circ \Tbb_M = -(V \otimes V \otimes M).
\end{equation}
Recall also the degenerate affine Hecke algebra $\overline{H}_m$ defined in~\S\ref{sss:daha}, which we consider here in the case $\F=\bk$.

\begin{lem}
\label{lem:dhaction}
For any $m \geq 1$, the assignment
\begin{align*}
X_i \mapsto V^{\otimes m-i} \otimes \Xbb_{V^{\otimes i-1} \otimes M} \quad
&\text{for $1 \le i \le m$,}\\
T_i \mapsto V^{\otimes m-i-1} \otimes \Tbb_{V^{\otimes i-1} \otimes M} \quad
&\text{for $1 \le i < m$}
\end{align*}
extends to an algebra morphism $\overline{H}_m \to \End(E^m)$.
\end{lem}

\begin{proof}
  The braid relations
\[
T_i T_{i+1} T_i = T_{i+1} T_i T_{i+1} \quad \text{and} \quad T_iT_j
  = T_j T_i \quad \text{if $| i - j | > 1$}
  \]
are immediate, as are the relations
  $T_i^2 = 1$ for $1 \le i \le n$ and  $X_i X_j = X_j X_i$ for all $1
  \le i, j \le n$. Hence we only need to check the relation
$T_i X_{i+1} = X_i T_i + 1$, or equivalently (since $T_i^2=1$) the relation $X_{i+1} T_i - T_i X_i = 1$.
This equality follows from the definitions and~\eqref{eq:deghecke}.
\end{proof}

\subsubsection{$2$-Kac--Moody action on $\Rep(G)$}

Let $M \in \Rep(G)$. 
It follows from Proposition~\ref{prop:natcat}\eqref{it:natcat-eigenvalues} that, for $m \geq 1$, if we view (the underlying vector space of)
$E^m M$ as a module over $\overline{H}_m$ via Lemma~\ref{lem:dhaction},
then $E^m M \in \overline{\cC}_\Gamma$. In particular we can use
Theorem~\ref{thm:BKR} to deduce that, for any $m \geq 1$, the assignment
\[
(\Z/p\Z)^m \ni \nu=(\nu_1, \cdots, \nu_m) \mapsto E_{\nu_m} E_{\nu_{m-1}} \cdots E_{\nu_1}M
\]
can be upgraded to give an object in $H_m(\Gamma)\Mod_0$. Note that, in this construction, for any $m \geq 1$ the action of $x_{i,\nu}$ on $E_{\nu_m} E_{\nu_{m-1}} \cdots E_{\nu_1}M$ is induced by the endomorphism of $E_{\nu_i}$ given by the case $m=1$ and $\nu=\nu_i$, and for any $m \geq 2$ the action of $\tau_{i,\nu}$ on $E_{\nu_m} E_{\nu_{m-1}} \cdots E_{\nu_1}M$ is induced by the morphism $E_{\nu_{i+1}} E_{\nu_{i}} \to E_{\nu_i} E_{\nu_{i+1}}$ given by the case $m=2$, $\nu=(\nu_i,\nu_{i+1})$.

Recall the bijection $\imath_{n}$ from Remark~\ref{rmk:weights}. To each $\lambda \in P$ we associate a category as follows:
\[
\lambda \mapsto \begin{cases}
\Repp_{\imath_{n}(\lambda)}(G) & \text{if $\lambda \in P(\bigwedge^n \nat_p)$;} \\
0 & \text{otherwise.}
\end{cases}
\]
Then, identifying $\Z/p\Z$ with $\{0, \cdots, p-1\}$ in the natural way, for $i \in \Z/p\Z$ and $\lambda \in P$, to $E_i 1_\lambda$ we associate the functor
\[
E_i^\lambda := E_i{}_{| \Repp_{\imath_{n}(\lambda)}(G)} \colon \Repp_{\imath_{n}(\lambda)}(G) \to \Repp_{\imath_{n}(\lambda+\alpha_i)}(G)
\]
if $\lambda$ and $\lambda+\alpha_i$ belong to $P(\bigwedge^n \nat_p)$ and $0$ otherwise,
and to $F_i 1_\lambda$ we associate the functor
\[
F_i^\lambda := F_i{}_{| \Repp_{\imath_{n}(\lambda)}(G)} \colon \Repp_{\imath_{n}(\lambda)}(G) \to \Repp_{\imath_{n}(\lambda-\alpha_i)}(G)
\]
if $\lambda$ and $\lambda-\alpha_i$ belong to $P(\bigwedge^n \nat_p)$ and $0$ otherwise.
Finally, to the generating $2$-morphisms
\[
\begin{array}{c}
\begin{tikzpicture}[baseline = 0]
\node at (0.2,.2) {\tiny $\lambda$};
	\draw[<-,thick] (0,.6) to (0,-.2);
  \node at (0,-.4) {\tiny $i$};
      \node at (0,0.2) {$\bullet$};
\end{tikzpicture}
\end{array}, \quad
\begin{array}{c}
\begin{tikzpicture}[baseline = 0]
	\draw[<-,thick,green] (0.25,.6) to (-0.25,-.2);
	\draw[->,thick] (0.25,-.2) to (-0.25,.6);
\node at (0.3,.2) {\tiny $\lambda$};
  \node at (-0.25,-.3) {\tiny $i$};
   \node at (0.25,-.3) {\tiny $j$};
\end{tikzpicture}
\end{array}, \quad
\begin{array}{c}
      \begin{tikzpicture}[baseline = 0]
        \draw[->,thick] (-0.3,0.3) to[out=-90,in=180] (0,-0.2) to[out=0,in=-90] (0.3,0.3);
        \node at (0.5,.2) {\tiny $\lambda$};
          \node at (-0.3,.4) {\tiny $i$};
      \end{tikzpicture}
\end{array}, \quad
\begin{array}{c}
      \begin{tikzpicture}[baseline = 0]
        \draw[->,thick] (-0.3,-0.2) to[out=90,in=180] (0,0.3) to[out=0,in=90] (0.3,-0.2);
        \node at (0.4,.2) {\tiny $\lambda$};
          \node at (-0.3,-.3) {\tiny $i$};
      \end{tikzpicture}
\end{array}
\]
we associate $0$ if one of the regions in the diagram is labelled by a weight which does not belong to $P(\bigwedge^n \nat_p)$, and otherwise the morphisms of functors
\begin{align*}
x_{1,i} \in \End(E_i^\lambda), &\quad \tau_{1,(j,i)} \in \End(E_i^{\lambda+\alpha_j} E_j^\lambda), \\
\eta_i \in \Hom(\id_{\Repp_{\imath_{n}(\lambda)}(G)}, F_i^{\lambda+\alpha_i} E_i^\lambda), &\quad \epsilon_i \in \Hom(E_i^{\lambda-\alpha_i} F_i^\lambda, \id_{\Repp_{\imath_{n}(\lambda)}(G)})
\end{align*}
respectively.
(Here, by abuse we denote by $x_{1,i}$, resp.~$\tau_{1,(j,i)}$, the image of this element of $H_1(\Gamma)$, resp.~of $H_2(\Gamma)$, in the corresponding morphism space obtained by the precedure described above.)

\begin{thm}
\label{thm:2rep-GLn}
The above data define a
$2$-representation of  $\mathcal{U}(\hgl_p)$.
\end{thm}

\begin{proof}
We have to check that our morphisms of functors satisfy the relations of~\cite[Definition~1.1]{Brundan}. 
The relations~\eqref{eqn:rel-2KM-1}--\eqref{eqn:rel-2KM-3} (corresponding to the relations~\cite[(1.2)--(1.4)]{Brundan}) follow from the relations~\eqref{eqn:rel-KLR-1}--\eqref{eqn:rel-KLR-5} in the KLR algebra. For instance, the relation
\begin{equation*}
\begin{array}{c}
\begin{tikzpicture}[baseline = 0]
	\draw[<-,thick] (0.25,.6) to (-0.25,-.2);
	\draw[->,thick,green] (0.25,-.2) to (-0.25,.6);
  \node at (-0.25,-.3) {\tiny $i$};
   \node at (0.25,-.3) {\tiny $j$};
  \node at (.3,.25) {\tiny $\lambda$};
      \node at (-0.13,-0.02) {$\bullet$};
\end{tikzpicture}
\end{array}
-
\begin{array}{c}
\begin{tikzpicture}[baseline = 0]
	\draw[<-,thick] (0.25,.6) to (-0.25,-.2);
	\draw[->,thick,green] (0.25,-.2) to (-0.25,.6);
  \node at (-0.25,-.3) {\tiny $i$};
   \node at (0.25,-.3) {\tiny $j$};
  \node at (.3,.25) {\tiny $\lambda$};
      \node at (0.13,0.42) {$\bullet$};
\end{tikzpicture}
\end{array}
=
\begin{cases}
\begin{array}{c}
\begin{tikzpicture}[baseline = 0]
 	\draw[->,thick,green] (0.08,-.3) to (0.08,.4);
	\draw[->,thick] (-0.28,-.3) to (-0.28,.4);
   \node at (-0.28,-.4) {\tiny $i$};
   \node at (0.08,-.4) {\tiny $j$};
 \node at (.28,.06) {\tiny $\lambda$};
\end{tikzpicture}
\end{array}
&\text{if $i=j$,}\\
0&\text{otherwise}\\
\end{cases}
\end{equation*}
in $\Hom(E_i E_j, E_j E_i)$ follows from the relation
\[
\tau_{1, (j,i)} \cdot x_{2, (j,i)} - x_{1, (i,j)} \cdot \tau_{1, (j,i)} = \begin{cases}
1_\nu & \text{if $i=j$;} \\
0 & \text{otherwise.}
\end{cases}
\]
in $H_2(\Gamma)$. Similarly, relation~\eqref{eqn:rel-2KM-2} follows from the relation
\begin{equation*}
\tau_{1,(i,j)} \tau_{1,(j,i)} = \begin{cases}
0 & \text{if $j=i$}; \\
1_{(j,i)} & \text{if $j \neq i$ and $j \norel i$;} \\
t_{ji} \cdot x_{1,(j,i)} + t_{i j} \cdot x_{2, (j,i)} & \text{if $j \rel i$}
\end{cases}
\end{equation*}
in $H_2(\Gamma)$, and relation~\eqref{eqn:rel-2KM-3} follows from the relation
\[
\tau_{2, (i,k,j)}\tau_{1,(k,i,j)} \tau_{2,(k,j,i)} - \tau_{1,(j,i,k)} \tau_{2,(j,k,i)} \tau_{1,(k,j,i)}
= \begin{cases}
t_{i, j} 1_{(k,j,i)} & \text{if $k=i$ and $i\rel j$;} \\
0 & \text{otherwise;}
\end{cases}
\]
in $H_3(\Gamma)$.

The ``right adjunction relations''~\cite[(1.5)]{Brundan} follow from the fact that $\epsilon_i$ and $\eta_i$ are indeed adjunction morphisms, see Lemma~\ref{lem:adjunctions-a}. Finally, we need to check that the images of the $2$-morphisms depicted in~\cite[(1.7)--(1.9)]{Brundan} are isomorphisms. However, the datum of $E_i$ and $F_i$ and the appropriate morphisms defines an $\mathfrak{sl}_2$-categorification in the sense of~\cite[Definition~5.20]{R2KM}, so that the images of the $2$-morphisms~\cite[(1.8), (1.9)]{Brundan} are invertible by~\cite[Theorem~5.22 and its proof]{R2KM}. Then the invertibility of the image of~\cite[(1.7)]{Brundan} is guaranteed by~\cite[Theorem~5.25 and its proof]{R2KM}.
\end{proof}

\subsubsection{Comparison with the setting of Section~{\rm \ref{sec:blocks}}}
\label{sss:comparison-Section3}

Now we assume that $p \geq n \geq 2$, and set
\[
\om := \e_1 + \e_2 + \cdots + \e_n \in P(\bigwedge\hspace{-3pt}{}^n \, \nat_p) \subset P.
\]
Then $\imath_n(\omega)$ is the $(W,\hdot)$-orbit of the weight $(n, \cdots, n) \in \bX$. This weight belongs to the fundamental alcove $C_\Z$, so that following the notation of Section~\ref{sec:blocks} we can denote it $\lambda_0$, and the corresponding subcategory $\Repp_{\imath_n(\omega)}(G)$ can be denoted $\Rep_0(G)$.

\begin{rmk}
The category $\Rep_0(G)$ is naturally equivalent to the block of the weight $0 \in \bX$ via the functor $M \mapsto \det^{\otimes (-n)} \otimes M$.
\end{rmk}


For any $i \in \{1, \cdots, n-1\}$ one can consider the functors
\[
E_i^{\omega} \colon \Repp_{\imath_n(\omega)}(G) \to \Repp_{\imath_n(\omega+\alpha_i)}(G)
\]
and
\[
F_i^{\omega+\alpha_i} \colon \Repp_{\imath_n(\omega+\alpha_i)}(G) \to \Repp_{\imath_n(\omega)}(G).
\]
Here $\imath_n(\omega + \alpha_i)$ is the orbit of the weight
\[
(\underbrace{n, \cdots, n}_{n-i}, n+1, \underbrace{n, \cdots, n}_{i-1}) \in \bX.
\]

On the other hand, the Weyl group $\Wf$ identifies canonically with the symmetric group $S_{n}$, and the simple reflections identify with the simple transpositions $\{s_j, \, j \in \{1, \cdots, n-1\}\}$. For any $j$, $s_j$ is the reflection associated with the simple root
\[
\chi_j-\chi_{j+1} = (\underbrace{0, \cdots, 0}_{j-1}, 1, -1, \underbrace{0, \cdots, 0}_{n-j-1})
\]
of $\mathrm{GL}_n(\Bbbk)$.
Hence we can set
\[
\mu_{s_j} = (\underbrace{n, \cdots, n}_{j}, n+1, \underbrace{n, \cdots, n}_{n-j-1}),
\]
so that $\imath_n(\omega + \alpha_i)$ is the $W$-orbit of $\mu_{s_{n-i}}$. Then, with the notation of Section~\ref{sec:blocks}, we have
\[
T_{\lambda_0}^{\mu_{s_j}} \cong E^{\omega}_{n-j}, \quad T_{\mu_{s_j}}^{\lambda_0} \cong F^{\omega+\alpha_{n-j}}_{n-j}
\]
Hence we can simply set
\[
\Trans^{s_j} := E^{\omega}_{n-j}, \quad \Trans_{s_j} := F^{\omega+\alpha_{n-j}}_{n-j},
\quad
\Theta_{s_j} := F^{\omega+\alpha_{n-j}}_{n-j} E^{\omega}_{n-j}.
\]

Consider now the case of the affine simple reflection $s_\infty \in W$. (See~\S\ref{ss:restriction-combinatorics} below for an explanation of why we use the notation $s_\infty$ instead of $s_0$.) We can set
\[
\mu_{s_\infty} = (p+1, n, \cdots, n).
\]
Then there exist isomorphisms of functors
\begin{align}
\label{eqn:Tinfty-E}
T_{\lambda_0}^{\mu_{s_\infty}} &\cong E_0^{\omega + \alpha_n + \cdots + \alpha_{p-1}} E_{p-1}^{\omega + \alpha_n + \cdots + \alpha_{p-2}} \cdots E_{n+1}^{\omega + \alpha_n} E_n^{\omega}, \\
\label{eqn:Tinfty-F}
T_{\mu_{s_\infty}}^{\lambda_0} &\cong F_n^{\omega+\alpha_n} F_{n+1}^{\omega + \alpha_n + \alpha_{n+1}} \cdots F_{p-1}^{\omega + \alpha_{n} + \cdots + \alpha_{p-1}} F_0^{\omega+ \alpha_n + \cdots + \alpha_{p-1} + \alpha_{0}}.
\end{align}
For instance, let us explain how to construct an isomorphism as in~\eqref{eqn:Tinfty-E}; the construction for~\eqref{eqn:Tinfty-F} is similar. 
By~\cite[Remark~II.7.6(1)]{jantzen}, the functor $T_{\lambda_0}^{\mu_{s_\infty}}$ is isomorphic to
the functor $\mathrm{pr}_{\mu_{s_\infty}} \bigl( M \otimes \mathrm{pr}_{\lambda_0}(-) \bigr)$ for any $M$ in $\Rep(G)$ such that 
\[
\dim(M_{(p+1-n) \chi_1})=1 \quad \text{and} \quad M_\nu \neq 0 \Rightarrow \nu \preceq (p+1-n)\chi_1.
\]
(Here, $\preceq$ is the standard order on $\bX$ associated to our choice of positive roots.)
For instance, the module $M=V^{\otimes (p+1-n)}$ satisfies these properties, so that $T_{\lambda_0}^{\mu_{s_\infty}}$ is a direct factor of the functor $V^{\otimes (p+1-n)} \otimes (-)$. On the other hand, by construction the functor $E_n^{\omega + \alpha_0 + \alpha_{p-1} + \cdots + \alpha_{n+1}} \cdots E_{p-1}^{\omega + \alpha_0} E_0^{\omega}$ is also a direct factor of the functor $V^{\otimes (p+1-n)} \otimes (-)$. Hence using inclusion and projection we can construct a morphism of functors
\[
T_{\lambda_0}^{\mu_{s_\infty}} \to E_n^{\omega + \alpha_0 + \alpha_{p-1} + \cdots + \alpha_{n+1}} \cdots E_{p-1}^{\omega + \alpha_0} E_0^{\omega}.
\]
Now, using~\cite[Proposition~II.7.15]{jantzen} one can easily check that this morphism induces an isomorphism on any simple object in $\Rep_0(G)$, which implies that it is an isomorphism.

From these considerations we see that we can set
\begin{align*}
\Trans^{s_\infty} &:= E_0^{\omega + \alpha_n + \cdots + \alpha_{p-1}} E_{p-1}^{\omega + \alpha_n + \cdots + \alpha_{p-2}} \cdots E_{n+1}^{\omega + \alpha_n} E_n^{\omega}, \\
\Trans_{s_\infty} &:= F_n^{\omega+\alpha_n} F_{n+1}^{\omega + \alpha_n + \alpha_{n+1}} \cdots F_{p-1}^{\omega + \alpha_{n} + \cdots + \alpha_{p-1}} F_0^{\omega+ \alpha_n + \cdots + \alpha_{p-1} + \alpha_{0}}.
\end{align*}

The goal of Sections~\ref{sec:restriction}--\ref{sec:g2D} is to show that, with these choices, if $p>n$ then Conjecture~\ref{conj:main} holds.

\section{Restriction of the representation to $\hgl_n$}
\label{sec:restriction}

From now on we assume that $p \ge n \geq 3$.\footnote{See Remark~\ref{rmk:intro-other-types}\eqref{it:rmk-n2} for comments on this restriction.}
Our goal in this section is to show that the $2$-representation of $\hgl_p$ considered in Section~\ref{sec:2rep-GLn} can be
``restricted'' to a $2$-representation of $\hgl_n$. This technical result will simplify our computations in Section \ref{sec:g2D}.

A similar 
construction in a more general setting is due to
Maksimau~\cite{m}. For the reader's convenience, we give a detailed proof.

\subsection{Combinatorics}
\label{ss:restriction-combinatorics}

Since $n \le p$, we can view $\mathfrak{sl}_n(\C)$ as a Lie subalgebra of $\mathfrak{sl}_p(\C)$
consisting of matrices all of whose non-zero entries are in the first
$n$ rows and columns. Since the trace form on $\mathfrak{sl}_p(\C)$ restricts to
the trace form on $\mathfrak{sl}_n(\C)$, this embedding extends in a natural way to an
embedding $\widehat{\mathfrak{sl}}_n \subset \widehat{\mathfrak{sl}}_p$ of Lie algebras, hence also to an embedding $\hgl_n \subset \hgl_p$.

Consider the natural representation $\nat_n$ (resp. $\nat_p$)
of $\hgl_n$ (resp. $\hgl_p$). The restriction of $\nat_p$ to $\hgl_n$
via the above inclusion is isomorphic to the direct sum of $\nat_n$ and
infinitely many copies of the trivial representation. Hence the
restriction of $\bigwedge^n \nat_p$ to $\hgl_n$ contains a  summand
$\bigwedge^n \nat_n$.

Let us describe this summand explicitly. The lattice of weights of $\hgl_n$ embeds in a natural way in the lattice $P$ of weights of $\hgl_p$. Via this embedding, the $\hgl_n$-weights in $\nat_n$ are the
weights of the form
\[
\{ \e_i + m \delta \mid 1 \le i \le n, \, m \in \Z \} \subset P.
\]
If we write $\nat_{p}^{[n]} \subset \nat_p$ for the sum of the weight spaces
corresponding to these weights, then this subspace is stable under the action of the Lie subalgebra $\hgl_n \subset \hgl_p$, and we have a canonical identification
$\nat_n =  \nat_{p}^{[n]}$.

Similarly, let us consider the weights
\begin{equation}
\label{eqn:weight-GLn}
\left\{ \sum_{i = 1}^n n_i \e_i + m \delta : n_i \in \Z_{\geq 0}, \, \sum_{i = 1}^n n_i
= n,  \, m \in \Z \right\} \subset P.
\end{equation}
Then if $(\bigwedge^n \nat_p)^{[n]} \subset \bigwedge^n \nat_p$ is the sum of the weight spaces corresponding to these weights, it is easy to see that $(\bigwedge^n \nat_p)^{[n]}$ is stable under the action of
$\hgl_n$ and that we have a canonical identification
\[
\bigwedge\hspace{-3pt}{}^n \, \nat_n =  (\bigwedge\hspace{-3pt}{}^n \, \nat_p)^{[n]}.
\]

Of course the inclusion $\hgl_n \subset \hgl_p$ is not compatible with the
Chevalley elements associated with the affine vertex $0$ unless $n = p$. 
For this reason, to avoid confusion the affine vertex for $\hgl_n$ will be denoted $\infty$ instead of $0$.
The image of $e_\infty$ and $f_\infty$ in $\hgl_p$ are complicated commutators;
however, one may
express their action on the summand $\nat_p^{[n]} \subset \nat_p$
in terms of the Chevalley elements of $\hgl_p$ as follows:
\[
e_\infty \leftrightarrow e_0 e_{p-1} \cdots e_{n+1} e_n, \qquad
f_\infty \leftrightarrow f_n f_{n+1} \cdots f_{p-1} f_0.
\]
(In other words, the complicated commutator
  expressing the image of $e_\infty$ in $\hgl_p$ has many
  terms which act by zero on $\nat_p^{[n]} \subset \nat_p$, and the above
  expression is all that remains.)

\subsection{Categorifying the combinatorics}
\label{ss:categorify-combinatorics}

We now explain how the previous construction can be categorified. 
As in~\S\ref{ss:cat-GLn-Groth-group} we set $G=\mathrm{GL}_n(\Bbbk)$, and let
\[
\Rep^{[n]}(G) \subset \Rep(G)
\]
denote the direct sum of the blocks corresponding to the weights in~\eqref{eqn:weight-GLn}.
We define functors $\widetilde{E}_i \colon \Rep^{[n]}(G) \to \Rep^{[n]}(G)$ for $i \in \{1, \dots, n-1, \infty\}$ as follows:
\begin{align*}
  \widetilde{E}_i &:= E_i{}_{|\Rep^{[n]}(G)} \quad \text{for $1 \le i \le n-1$};\\
  \widetilde{F}_i &:= F_i{}_{|\Rep^{[n]}(G)} \quad \text{for $1 \le i \le n-1$};\\
\widetilde{E}_\infty &:= (E_0 E_{p-1} \dots E_{n+1}E_n){}_{|\Rep^{[n]}(G)}; \\
\widetilde{F}_\infty &:= (F_n F_{n+1} \dots F_{p-1}F_0){}_{|\Rep^{[n]}(G)}.
\end{align*}
(It is immediate to check that these functors preserve the subcategory $\Rep^{[n]}(G)$.)

Using the adjunctions $(E_j, F_j)$ for all $j \in \{0, \cdots, p-1\}$ we obtain adjunctions
$(\widetilde{E}_i, \widetilde{F}_i )$ for all $i \in \{1, \dots, n - 1, \infty\}$.
From Proposition~\ref{prop:natcat} and the considerations of~\S\ref{ss:restriction-combinatorics}
we deduce the following.

\begin{lem}
On the Grothendieck group $[\Rep^{[n]}(G)]$, the exact functors
  $\widetilde{E}_i$ and $\widetilde{F}_i$ ($i \in \{1, \dots, n-1, \infty\}$) induce an action of
  $\hgl_n$. Moreover, the resulting $\hgl_n$-module is (canonically)
  isomorphic to $\bigwedge^n \nat_n$.
\end{lem}

The goal of the rest of the section is to show that these functors (together with extra data to be defined below) endow $\Rep^{[n]}(G)$ with a structure of $2$-representation of $\mathcal{U}(\hgl_n)$.


\subsection{First relations}
\label{ss:first-relations}

Some of our computations below will require an induction on $n$. So, for now we consider an integer $m$ with $3 \leq m \leq p$.
Let us consider the following subsets of the set $P$ of weights of $\hgl_p$:
\begin{align*}
P_+ &:= \left\{ \; \sum_{i=1}^p n_i \e_{i} + k \delta : n_i \in \Z_{\geq 0}, \, k \in \Z \right\},  \\
Y_m &:=\left\{ \; \sum_{i = 1}^m n_i \e_i + k \delta : n_i \in \Z, \, k \in \Z \right\}.
\end{align*}

We denote by $\mathcal{U}_{+}(\hgl_p)$ the quotient of $\mathcal{U}(\hgl_p)$ by the objects which do not belong to $P_+$. In other words, 
$2$-morphisms in $\mathcal{U}_+(\hgl_p)$ are obtained from the $2$-morphisms in $\mathcal{U}(\hgl_p)$ by quotienting by the $2$-morphisms which contain a region labelled by an element $\lambda \in P \smallsetminus P_+$ (or equivalently which factor through a $1$-morphism factoring through an object $\lambda \in P \smallsetminus P_+$).
Note that in the $2$-representation of Theorem~\ref{thm:2rep-GLn}, the category attached to a weight $\lambda \in P \smallsetminus P_+$ is the zero category; therefore this $2$-representation of $\mathcal{U}(\hgl_p)$ factors through a $2$-representation of $\mathcal{U}_{+}(\hgl_p)$.

In the lemmas below, by abuse we sometimes use the diagrams for $2$-morphisms in $\mathcal{U}(\hgl_p)$ to denote their image in $\mathcal{U}_{+}(\hgl_p)$.

For $\gamma \in P$, we now set
\begin{equation}
\label{eqn:def-Einfty}
E_{\infty_m} 1_\gamma := E_0 E_{p-1} \cdots E_m 1_\gamma, \qquad F_{\infty_m} 1_\gamma := F_m F_{m+1} \cdots F_{p-1} F_0 1_\gamma.
\end{equation}
The corresponding identity $2$-morphisms (in $\mathcal{U}(\hgl_p)$) are denoted as follows:
\[
\begin{array}{c}
    \begin{tikzpicture}[thick,scale=0.3,baseline]
      \draw[->] (0,-2) to (0,2);
      \node at (0,-2.4) {\tiny $\infty_m$};
            \node at (0.7,0) {\tiny $\g$};
    \end{tikzpicture}
\end{array}
=
\begin{array}{c}
    \begin{tikzpicture}[thick,yscale=0.3,xscale=0.4,baseline]
      \draw[->] (0,-2) to (0,2);
      \node at (0,-2.4) {\tiny $0$};
      \draw[->] (1.5,-2) to (1.5,2);
      \node at (1.5,-2.4) {\tiny $p-1$};
      \draw[->] (3.5,-2) to (3.5,2);
      \node at (3.5,-2.4) {\tiny $m+1$};
      \draw[->] (5,-2) to (5,2);
      \node at (5,-2.4) {\tiny $m$};
      \node at (2.6,0) {$\dots$};
                  \node at (5.5,0) {\tiny $\g$};
    \end{tikzpicture}
\end{array},
\qquad
\begin{array}{c}
    \begin{tikzpicture}[thick,scale=-0.3,baseline]
      \draw[->] (0,-2) to (0,2);
      \node at (0,-2.5) {\tiny $\infty_m$};
                  \node at (-0.7,0) {\tiny $\g$};
    \end{tikzpicture}
\end{array}
:=
\begin{array}{c}
    \begin{tikzpicture}[thick,yscale=-0.3,xscale=-0.4,baseline]
      \draw[->] (0,-2) to (0,2);
      \node at (0,-2.3) {\tiny $0$};
      \draw[->] (1.5,-2) to (1.5,2);
      \node at (1.5,-2.3) {\tiny $p-1$};
      \draw[->] (3.5,-2) to (3.5,2);
      \node at (3.5,-2.3) {\tiny $m+1$};
      \draw[->] (5,-2) to (5,2);
      \node at (5,-2.3) {\tiny $m$};
      \node at (2.5,0) {$\dots$};
                        \node at (-.5,0) {\tiny $\g$};
    \end{tikzpicture}
\end{array} .
\]
(If $m=p$, we interpret these diagrams as just an arrow labelled $0$.)
We define crossing $2$-morphisms in $\mathcal{U}(\hgl_p)$ involving this new element and $i \in
\{1, \cdots, m-1\}$ as follows:
\begin{align*}
  \begin{array}{c}
    \begin{tikzpicture}[thick,scale=0.3,baseline]
      \draw[->] (2,-2) to (-2,2);
      \draw[->,xscale=-1] (2,-2) to (-2,2);
      \node at (-2,-2.4) {\tiny $i$};
      \node at (2,-2.4) {\tiny $\infty_m$};
      \node at (2,0) {\tiny $\g$};
    \end{tikzpicture}
\end{array}
&:=
  \begin{array}{c}
    \begin{tikzpicture}[thick,scale=0.3,baseline]
      \draw[->] (-2,-2) to (2,2);
      \draw[->] (1,-2) to (-3,2);
      \draw[->] (3,-2) to (-1,2);
      \node at (-2,-2.4) {\tiny $i$}; 
      \node at (1,-1) {\tiny $\dots$};
      \node at (-1,1) {\tiny $\dots$};
      \node at (3,-2.4) {\tiny $m$};
      \node at (1,-2.5) {\tiny $0$};
      \node at (2,0) {\tiny ${\g}$};
    \end{tikzpicture}
\end{array}, \\
 \begin{array}{c}
    \begin{tikzpicture}[thick,xscale=-0.3,yscale=0.3,baseline]
      \draw[->] (2,-2) to (-2,2);
      \draw[->,xscale=-1] (2,-2) to (-2,2);
      \node at (-2,-2.4) {\tiny $i$};
      \node at (2,-2.4) {\tiny $\infty_m$};
      \node at (-2,0) {\tiny ${\g}$};
    \end{tikzpicture}
\end{array}
&:=
  \begin{array}{c}
    \begin{tikzpicture}[thick,xscale=-0.3,yscale=0.3,baseline]
      \draw[->] (-2,-2) to (2,2);
      \draw[->] (1,-2) to (-3,2);
      \draw[->] (3,-2) to (-1,2);
      \node at (-2,-2.4) {\tiny $i$}; 
      \node at (1,-1) {\tiny $\dots$};
      \node at (-1,1) {\tiny $\dots$};
      \node at (3,-2.4) {\tiny $0$};
      \node at (1,-2.4) {\tiny $m$};
      \node at (-2,0) {\tiny ${\g}$};
    \end{tikzpicture}
\end{array}, \\
 \begin{array}{c}
    \begin{tikzpicture}[thick,scale=0.3,baseline]
      \draw[->] (2,-2) to (-2,2);
      \draw[->,xscale=-1] (2,-2) to (-2,2);
      \node at (-2,-2.4) {\tiny $\infty_m$};     \node at (2,-2.4) {\tiny $\infty_m$};
      \node at (2,0) {\tiny ${\g}$};
    \end{tikzpicture}
\end{array}
&:=
  (-1)^{p-m} \begin{array}{c}
    \begin{tikzpicture}[thick,scale=0.3,baseline]
            \draw[->] (-1,-2) to (3,2);
       \draw[->] (-3,-2) to (1,2);
      \draw[->] (1,-2) to (-3,2);
       \draw[->] (3,-2) to (-1,2);
      \node at (-1.2,-1.2) {\tiny $\dots$};
      \node at (1.2,1.2) {\tiny $\dots$};
      \node at (1.2,-1.2) {\tiny $\dots$};
      \node at (-1.2,1.2) {\tiny $\dots$};
      \node at (3,-2.4) {\tiny $m$};\node at (1,-2.4) {\tiny $0$};
      \node at (-1,-2.4) {\tiny $m$};\node at (-3,-2.4) {\tiny $0$};
      \node at (3,0) {\tiny ${\g}$};
    \end{tikzpicture}
\end{array} .
\end{align*}
(Note the sign in the last definition.) Finally, we set
\[
\begin{array}{c}
    \begin{tikzpicture}[thick,scale=0.3,baseline]
      \draw[->] (0,-2) to (0,2);
      \node at (0,0) {$\bullet$};
      \node at (0,-2.4) {\tiny $\infty_m$};
            \node at (0.7,0) {\tiny $\g$};
    \end{tikzpicture}
\end{array}
:=
\begin{array}{c}
    \begin{tikzpicture}[thick,yscale=0.3,xscale=0.4,baseline]
      \draw[->] (0,-2) to (0,2);
      \node at (0,-2.4) {\tiny $0$};
      \node at (5,0) {$\bullet$};
      \draw[->] (1.5,-2) to (1.5,2);
      \node at (1.5,-2.4) {\tiny $p-1$};
      \draw[->] (3.5,-2) to (3.5,2);
      \node at (3.5,-2.4) {\tiny $m+1$};
      \draw[->] (5,-2) to (5,2);
      \node at (5,-2.4) {\tiny $m$};
      \node at (2.6,0) {$\dots$};
      \node at (5.7,0) {\tiny $\g$};
    \end{tikzpicture}
\end{array} .
\]

\begin{lem}
\label{lem:vanishing-infty}
For any $\gamma \in P$, in $\mathcal{U}(\hgl_p)$ we have
\[
      \begin{array}{c}
\begin{tikzpicture}[baseline = 0]
	\draw[->,thick] (0.28,.4) to[out=90,in=-90] (-0.28,1.1);
	\draw[->,thick] (-0.28,.4) to[out=90,in=-90] (0.28,1.1);
	\draw[-,thick] (0.28,-.3) to[out=90,in=-90] (-0.28,.4);
	\draw[-,thick] (-0.28,-.3) to[out=90,in=-90] (0.28,.4);
  \node at (-0.28,-.4) {\tiny $\infty_m$};
  \node at (0.28,-.4) {\tiny $\infty_m$};
  \node at (.43,.4) {\tiny ${\g}$};
\end{tikzpicture}
\end{array}
= 0.
\]
\end{lem}

\begin{proof}
We prove the claim by downward induction on $m$. If $m=p$ then the relation is an instance of~\eqref{eqn:rel-2KM-2}. Now assume that $m<p$. Then we have
\[
      \begin{array}{c}
\begin{tikzpicture}[baseline = 0]
	\draw[->,thick] (0.28,.4) to[out=90,in=-90] (-0.28,1.1);
	\draw[->,thick] (-0.28,.4) to[out=90,in=-90] (0.28,1.1);
	\draw[-,thick] (0.28,-.3) to[out=90,in=-90] (-0.28,.4);
	\draw[-,thick] (-0.28,-.3) to[out=90,in=-90] (0.28,.4);
  \node at (-0.28,-.4) {\tiny $\infty_m$};
  \node at (0.28,-.4) {\tiny $\infty_m$};
  \node at (.43,.4) {\tiny ${\g}$};
\end{tikzpicture}
\end{array}
=
      \begin{array}{c}
    \begin{tikzpicture}[thick,xscale=0.8,yscale=0.5]
      \draw[->] (-2,-2) to[out=90,in=-90] (-.5,0) to[out=90,in=-90] (-2,2);
            \draw[->] (-0.5,-2) to[out=90,in=-90] (-2,0) to[out=90,in=-90] (-0.5,2);
      \draw[->,green] (-1.3,-2) to[out=90,in=-90] (0.3,0) to[out=90,in=-90] (-1.3,2);
      \draw[->,green] (0.2,-2) to[out=90,in=-90] (-1.3,0) to[out=90,in=-90] (0.2,2);
      \node at (-2,-2.3) {\tiny $\infty_{m+1}$};   
      \node at (-1.3,-2.3) {\tiny $m$};   
            \node at (-.5,-2.3) {\tiny $\infty_{m+1}$};     
      \node at (0.2,-2.3) {\tiny $m$};   
      \node at (.8,0) {\tiny ${\g}$};
     \end{tikzpicture}
\end{array}
\; \stackrel{(*)}{=} \; 
 -     \begin{array}{c}
   \begin{tikzpicture}[thick,xscale=0.25,yscale=0.5]
      \draw[->] (-2,-2) to[out=90,in=-90] (-1,0) to[out=90,in=-90] (-2,2);
            \draw[->] (1,-2) to[out=90,in=-70] (-2.5,0) to[out=70,in=-90] (1,2);
      \draw[->,green] (-1,-2) to[out=90,in=-110] (2.5,0) to[out=110,in=-90] (-1,2);
      \draw[->,green] (2,-2) to[out=90,in=-90] (1,0) to[out=90,in=-90] (2,2);
     \node at (1,0) {\color{green}{$\bullet$}};
      \node at (2,-2.3) {\tiny $m$};   
      \node at (-1,-2.3) {\tiny $m$};  
      \node at (3.5,0) {\tiny ${\g}$};
     \end{tikzpicture}
\end{array}
+
      \begin{array}{c}
   \begin{tikzpicture}[thick,xscale=0.25,yscale=0.5]
      \draw[->] (-2,-2) to[out=90,in=-90] (-1,0) to[out=90,in=-90] (-2,2);
            \draw[->] (1,-2) to[out=90,in=-70] (-2.5,0) to[out=70,in=-90] (1,2);
      \draw[green,->] (-1,-2) to[out=90,in=-110] (2.5,0) to[out=110,in=-90] (-1,2);
      \draw[green,->] (2,-2) to[out=90,in=-90] (1,0) to[out=90,in=-90] (2,2);
     \node at (-1,0) {$\bullet$};
      \node at (2,-2.3) {\tiny $m$};   
      \node at (-1,-2.3) {\tiny $m$};
      \node at (3.5,0) {\tiny ${\g}$};
     \end{tikzpicture}
     \end{array},
\]
where in $(*)$ we have used relation~\eqref{eqn:rel-2KM-2} for all crossings between the strand labelled by $m$ and the strands (labelled by $0, \cdots, m+1$) in the decomposition of the strand $\infty_{m+1}$. Now both terms in the right-hand side vanish, by induction and relation~\eqref{eqn:rel-2KM-2} respectively.
\end{proof}

\begin{lem}
\label{lem:crossings-infty}
For any $\gamma \in Y_m$, in $\mathcal{U}_+(\hgl_p)$ we have
\[
      \begin{array}{c}
\begin{tikzpicture}[baseline = 0,xscale=1.5,yscale=1.2]
	\draw[<-,thick] (0.45,.8) to (-0.45,-.4);
	\draw[->,thick] (0.45,-.4) to (-0.45,.8);
        \draw[-,thick,green] (0,-.4) to[out=90,in=-90] (-.45,0.2);
        \draw[->,thick,green] (-0.45,0.2) to[out=90,in=-90] (0,0.8);
   \node at (-0.45,-.6) {\tiny $\infty_m$};
   \node at (0,-.6) {\tiny $m-1$};
  \node at (0.45,-.6) {\tiny $\infty_m$};
   \node at (.5,-.1) {\tiny ${\g}$};
\end{tikzpicture}
\end{array}
=
      \begin{array}{c}
\begin{tikzpicture}[baseline = 0,xscale=1.5,yscale=1.2]
	\draw[<-,thick] (0.45,.8) to (-0.45,-.4);
	\draw[->,thick] (0.45,-.4) to (-0.45,.8);
        \draw[-,thick,green] (0,-.4) to[out=90,in=-90] (.45,0.2);
        \draw[->,thick,green] (0.45,0.2) to[out=90,in=-90] (0,0.8);
   \node at (-0.45,-.6) {\tiny $\infty_m$};
   \node at (0,-.6) {\tiny $m-1$};
  \node at (0.45,-.6) {\tiny $\infty_m$};
   \node at (.5,-.1) {\tiny ${\g}$};
\end{tikzpicture}
\end{array}
+
      \begin{array}{c}
\begin{tikzpicture}[baseline = 0,xscale=1.5,yscale=1.2]
	\draw[<-,thick] (0.45,.8) to (0.45,-.4);
	\draw[->,thick] (-0.45,-.4) to (-0.45,.8);
        \draw[->,thick,green] (0,-.4) to (0,0.8);
   \node at (-0.45,-.6) {\tiny $\infty_m$};
   \node at (0,-.6) {\tiny $m-1$};
  \node at (0.45,-.6) {\tiny $\infty_m$};
   \node at (.7,-.1) {\tiny ${\g}$};
\end{tikzpicture}
\end{array}.
\]
\end{lem}

\begin{proof}
We prove the formula by downward induction on $m$. If $m=p$, then the relation is an instance of~\eqref{eqn:rel-2KM-3}. Now assume that $m<p$. We have
\begin{multline*}
\begin{array}{c}
\begin{tikzpicture}[baseline = 0,yscale=1.3,xscale=1.4]
	\draw[<-,thick] (0.45,.8) to (-0.45,-.4);
	\draw[->,thick] (0.45,-.4) to (-0.45,.8);
        \draw[-,thick,green] (0,-.4) to[out=90,in=-90] (-.45,0.2);
        \draw[->,thick,green] (-0.45,0.2) to[out=90,in=-90] (0,0.8);
   \node at (-0.45,-.5) {\tiny $\infty_m$};
   \node at (0,-.5) {\tiny $m-1$};
  \node at (0.5,-.5) {\tiny $\infty_m$};
   \node at (.5,-.1) {\tiny ${\g}$};
\end{tikzpicture}
\end{array}
=
-
\begin{array}{c}
\begin{tikzpicture}[baseline = 0,yscale=1.3,xscale=2.3]
	\draw[<-,thick] (0.26,.8) to (-0.64,-.4);
	        \draw[-,thick,green] (0,-.4) to[out=90,in=-90] (-.45,0.2);
        \draw[->,thick,green] (-0.45,0.2) to[out=90,in=-90] (0,0.8);
        	\draw[->,thick] (0.3,-.4) to (-0.6,.8);
	\draw[<-,thick,red] (0.55,.8) to (-0.35,-.4);
	\draw[->,thick,red] (0.55,-.4) to (-0.35,.8);
   \node at (-0.64,-.5) {\tiny $\infty_{m+1}$};
   \node at (-0.05,-.5) {\tiny $m-1$};
  \node at (0.3,-.5) {\tiny $\infty_{m+1}$};
  \node at (-0.35,-.5) {\tiny $m$};
  \node at (0.55,-.5) {\tiny $m$};
   \node at (.6,-.1) {\tiny ${\g}$};
\end{tikzpicture}
\end{array}
=
-
\begin{array}{c}
\begin{tikzpicture}[baseline = 0,yscale=1.3,xscale=2.3]
	\draw[<-,thick] (0.26,.8) to (-0.64,-.4);
	\draw[->,thick,green] (0,-.4) to[out=90,in=-90] (.1,0) to[out=90,in=-90] (-.05,.2) to[out=90,in=-90] (.1,.4) to[out=90,in=-90] (0,.8);
        	\draw[->,thick] (0.3,-.4) to (-0.6,.8);
	\draw[<-,thick,red] (0.55,.8) to (-0.35,-.4);
	\draw[->,thick,red] (0.55,-.4) to (-0.35,.8);
   \node at (-0.05,-.5) {\tiny $m-1$};
  \node at (-0.35,-.5) {\tiny $m$};
  \node at (0.55,-.5) {\tiny $m$};
   \node at (.6,-.1) {\tiny ${\g}$};
\end{tikzpicture}
\end{array}
\\
=
      \begin{array}{c}
\begin{tikzpicture}[baseline = 0,xscale=1.3,yscale=1.4]
	\draw[<-,thick] (0.45,.8) to (-0.45,-.4);
	\draw[->,thick] (0.45,-.4) to (-0.45,.8);
        \draw[-,thick,green] (0,-.4) to[out=90,in=-90] (.45,0.2);
        \draw[->,thick,green] (0.45,0.2) to[out=90,in=-90] (0,0.8);
   \node at (-0.45,-.6) {\tiny $\infty_m$};
   \node at (0,-.6) {\tiny $m-1$};
  \node at (0.45,-.6) {\tiny $\infty_m$};
   \node at (.5,-.1) {\tiny ${\g}$};
\end{tikzpicture}
\end{array}
-
\begin{array}{c}
\begin{tikzpicture}[baseline = 0,yscale=1.3,xscale=1.7]
	\draw[<-,thick] (0.26,.8) to (-0.64,-.4);
	\draw[->,thick,green] (0.1,-.4) to (0.1,.8);
        	\draw[->,thick] (0.3,-.4) to (-0.6,.8);
	\draw[->,thick,red] (-0.35,-.4) to[out=90,in=-90] (-.05,.2) to[out=90,in=-90] (-0.35,.8);
	\draw[->,thick,red] (0.55,-.4) to (0.55,.8);
   \node at (0.1,-.5) {\tiny $m-1$};
  \node at (-0.35,-.5) {\tiny $m$};
  \node at (0.55,-.5) {\tiny $m$};
   \node at (.7,-.1) {\tiny ${\g}$};
\end{tikzpicture}
\end{array} ,
\end{multline*}
where the second equality uses many instances of the second case in~\eqref{eqn:rel-2KM-3}, and the third one uses the first case in~\eqref{eqn:rel-2KM-3} and the fact that $t_{m,m-1}=1$.
Now we can use induction (since $\gamma+\alpha_m+\alpha_{m-1}$ belongs to $Y_{m+1}$) to write
\[
-
\begin{array}{c}
\begin{tikzpicture}[baseline = 0,yscale=1.3,xscale=1.7]
	\draw[<-,thick] (0.26,.8) to (-0.64,-.4);
	\draw[->,thick,green] (0.1,-.4) to (0.1,.8);
        	\draw[->,thick] (0.3,-.4) to (-0.6,.8);
	\draw[->,thick,red] (-0.35,-.4) to[out=90,in=-90] (-.05,.2) to[out=90,in=-90] (-0.35,.8);
	\draw[->,thick,red] (0.55,-.4) to (0.55,.8);
   \node at (0.1,-.5) {\tiny $m-1$};
  \node at (-0.35,-.5) {\tiny $m$};
  \node at (0.55,-.5) {\tiny $m$};
   \node at (.7,-.1) {\tiny ${\g}$};
\end{tikzpicture}
\end{array}
=
-
\begin{array}{c}
\begin{tikzpicture}[baseline = 0,yscale=1.3,xscale=1.5]
	\draw[<-,thick] (0.26,.8) to (-0.64,-.4);
	\draw[->,thick,green] (0.1,-.4) to (0.1,.8);
        	\draw[->,thick] (0.3,-.4) to (-0.6,.8);
	\draw[->,thick,red] (-0.35,-.4) to (-0.35,.8);
	\draw[->,thick,red] (0.55,-.4) to (0.55,.8);
   \node at (0.1,-.5) {\tiny $m-1$};
  \node at (-0.35,-.5) {\tiny $m$};
  \node at (0.55,-.5) {\tiny $m$};
   \node at (.7,-.1) {\tiny ${\g}$};
\end{tikzpicture}
\end{array}
+
\begin{array}{c}
\begin{tikzpicture}[baseline = 0,yscale=1.3,xscale=1.5]
	\draw[<-,thick] (-0.64,.8) to (-0.64,-.4);
	\draw[->,thick,green] (0.1,-.4) to (0.1,.8);
        	\draw[->,thick] (0.3,-.4) to[out=90,in=-90] (0,.2) to[out=90,in=-90] (0.3,.8);
	\draw[->,thick,red] (-0.35,-.4) to (-0.35,.8);
	\draw[->,thick,red] (0.55,-.4) to (0.55,.8);
   \node at (0.1,-.5) {\tiny $m-1$};
  \node at (-0.35,-.5) {\tiny $m$};
  \node at (0.55,-.5) {\tiny $m$};
   \node at (.7,-.1) {\tiny ${\g}$};
\end{tikzpicture}
\end{array} .
\]
By~\eqref{eqn:rel-2KM-2}, the second term on
the right-hand side is the desired term in the equality of the
lemma. The first term vanishes in $\mathcal{U}_+(\hgl_p)$ since it
contains a region labelled by
\[
\gamma+\alpha_m + \alpha_{m-1} + 2(\alpha_0 + \alpha_{p-1} + \cdots + \alpha_{m+1}) = \gamma+(\varepsilon_{m+1} - \varepsilon_{m-1}) + 2(\delta + \varepsilon_1-\varepsilon_{m+1}),
\]
which does not belong to $P_+$.
\end{proof}

\begin{lem}
\label{lem:crossings-all-infty}
For any $\gamma \in Y_m$, in $\mathcal{U}_+(\hgl_p)$ we have
\[
      \begin{array}{c}
\begin{tikzpicture}[baseline = 0,xscale=1.5,yscale=1.2]
	\draw[<-,thick] (0.45,.8) to (-0.45,-.4);
	\draw[->,thick] (0.45,-.4) to (-0.45,.8);
        \draw[-,thick] (0,-.4) to[out=90,in=-90] (-.45,0.2);
        \draw[->,thick] (-0.45,0.2) to[out=90,in=-90] (0,0.8);
   \node at (-0.45,-.6) {\tiny $\infty_m$};
   \node at (0,-.6) {\tiny $\infty_m$};
  \node at (0.45,-.6) {\tiny $\infty_m$};
   \node at (.5,-.1) {\tiny ${\g}$};
\end{tikzpicture}
\end{array}
=
      \begin{array}{c}
\begin{tikzpicture}[baseline = 0,xscale=1.5,yscale=1.2]
	\draw[<-,thick] (0.45,.8) to (-0.45,-.4);
	\draw[->,thick] (0.45,-.4) to (-0.45,.8);
        \draw[-,thick] (0,-.4) to[out=90,in=-90] (.45,0.2);
        \draw[->,thick] (0.45,0.2) to[out=90,in=-90] (0,0.8);
   \node at (-0.45,-.6) {\tiny $\infty_m$};
   \node at (0,-.6) {\tiny $\infty_m$};
  \node at (0.45,-.6) {\tiny $\infty_m$};
   \node at (.5,-.1) {\tiny ${\g}$};
\end{tikzpicture}
\end{array}.
\]
\end{lem}

\begin{proof}
Once again we use downward induction on $m$. If $m=p$ then the equality is an instance of the second case in~\eqref{eqn:rel-2KM-3}. Now we assume that $m<p$. Then we have
\[
      \begin{array}{c}
\begin{tikzpicture}[baseline = 0,xscale=1.5,yscale=1.2]
	\draw[<-,thick] (0.45,.8) to (-0.45,-.4);
	\draw[->,thick] (0.45,-.4) to (-0.45,.8);
        \draw[-,thick] (0,-.4) to[out=90,in=-90] (-.45,0.2);
        \draw[->,thick] (-0.45,0.2) to[out=90,in=-90] (0,0.8);
   \node at (-0.45,-.6) {\tiny $\infty_m$};
   \node at (0,-.6) {\tiny $\infty_m$};
  \node at (0.45,-.6) {\tiny $\infty_m$};
   \node at (.5,-.1) {\tiny ${\g}$};
\end{tikzpicture}
\end{array}
=
 -     \begin{array}{c}
\begin{tikzpicture}[baseline = 0,xscale=1.8,yscale=1.2]
	\draw[<-,thick] (0.45,.8) to (-0.45,-.4);
	\draw[<-,thick,green] (0.6,.8) to (-0.3,-.4);
	\draw[->,thick] (0.45,-.4) to (-0.45,.8);
		\draw[->,thick,green] (0.6,-.4) to (-0.3,.8);
        \draw[-,thick] (0,-.4) to[out=90,in=-90] (-.45,0.2);
        \draw[->,thick] (-0.45,0.2) to[out=90,in=-90] (0,0.8);
                \draw[-,thick,green] (0.15,-.4) to[out=90,in=-90] (-.3,0.2);
        \draw[->,thick,green] (-0.3,0.2) to[out=90,in=-90] (0.15,0.8);
   \node at (-0.45,-.6) {\tiny $\infty_{m+1}$};
   \node at (0,-.6) {\tiny $\infty_{m+1}$};
  \node at (0.45,-.6) {\tiny $\infty_{m+1}$};
   \node at (.5,-.1) {\tiny ${\g}$};
\end{tikzpicture}
\end{array},
\]
where the green strands are labelled by $m$. Using Lemma~\ref{lem:crossings-infty} (which is possible since $\gamma+2\alpha_m$ belongs to $Y_{m+1}$), then Lemma~\ref{lem:vanishing-infty} to omit the second term in the equality we obtain, and finally several instances of the second case in~\eqref{eqn:rel-2KM-3}, we see that
\[
      \begin{array}{c}
\begin{tikzpicture}[baseline = 0,xscale=1.8,yscale=1.2]
	\draw[<-,thick] (0.45,.8) to (-0.45,-.4);
	\draw[<-,thick,green] (0.6,.8) to (-0.3,-.4);
	\draw[->,thick] (0.45,-.4) to (-0.45,.8);
		\draw[->,thick,green] (0.6,-.4) to (-0.3,.8);
        \draw[-,thick] (0,-.4) to[out=90,in=-90] (-.45,0.2);
        \draw[->,thick] (-0.45,0.2) to[out=90,in=-90] (0,0.8);
                \draw[-,thick,green] (0.15,-.4) to[out=90,in=-90] (-.3,0.2);
        \draw[->,thick,green] (-0.3,0.2) to[out=90,in=-90] (0.15,0.8);
   \node at (-0.45,-.6) {\tiny $\infty_{m+1}$};
   \node at (0,-.6) {\tiny $\infty_{m+1}$};
  \node at (0.45,-.6) {\tiny $\infty_{m+1}$};
   \node at (.5,-.1) {\tiny ${\g}$};
\end{tikzpicture}
\end{array}
=
      \begin{array}{c}
\begin{tikzpicture}[baseline = 0,xscale=1.8,yscale=1.2]
	\draw[<-,thick] (0.45,.8) to (-0.45,-.4);
	\draw[<-,thick,green] (0.6,.8) to (-0.3,-.4);
	\draw[->,thick] (0.45,-.4) to (-0.45,.8);
		\draw[->,thick,green] (0.6,-.4) to (-0.3,.8);
        \draw[-,thick] (0,-.4) to[out=90,in=-90] (-.45,0.2);
        \draw[->,thick] (-0.45,0.2) to[out=90,in=-90] (0,0.8);
                \draw[-,thick,green] (0.15,-.4) to[out=90,in=-90] (.3,0.2);
        \draw[->,thick,green] (0.3,0.2) to[out=90,in=-90] (0.15,0.8);
   \node at (-0.45,-.6) {\tiny $\infty_{m+1}$};
   \node at (0,-.6) {\tiny $\infty_{m+1}$};
  \node at (0.45,-.6) {\tiny $\infty_{m+1}$};
   \node at (.5,-.1) {\tiny ${\g}$};
\end{tikzpicture}
\end{array} .
\]
Then we can use induction (which again is possible since $\gamma+3\alpha_m$ belongs to $Y_{m+1}$), and again several instances of the second case in~\eqref{eqn:rel-2KM-3}, to see that the equality one needs to conclude is
\[
      \begin{array}{c}
\begin{tikzpicture}[baseline = 0,xscale=1.5,yscale=1.2]
	\draw[<-,thick,green] (0.45,.8) to (-0.45,-.4);
	\draw[->,thick,green] (0.45,-.4) to (-0.45,.8);
		\draw[->,thick,green] (0.2,-.4) to[out=90,in=-90] (.3,.2) to[out=90,in=-90] (0.2,.8);
        \draw[-,thick] (0,-.4) to[out=90,in=-90] (-.45,0.2);
        \draw[->,thick] (-0.45,0.2) to[out=90,in=-90] (0,0.8);
   \node at (0,-.6) {\tiny $\infty_{m+1}$};
   \node at (.5,-.1) {\tiny ${\g}$};
\end{tikzpicture}
\end{array}
=
      \begin{array}{c}
\begin{tikzpicture}[baseline = 0,xscale=1.5,yscale=1.2]
	\draw[<-,thick,green] (0.45,.8) to (-0.45,-.4);
	\draw[->,thick,green] (0.45,-.4) to (-0.45,.8);
        \draw[-,thick] (0,-.4) to[out=90,in=-90] (.25,0.2);
        \draw[->,thick] (0.25,0.2) to[out=90,in=-90] (0,0.8);
        		\draw[->,thick,green] (0.2,-.4) to[out=90,in=-90] (.5,.2) to[out=90,in=-90] (0.2,.8);
   \node at (0,-.6) {\tiny $\infty_{m+1}$};
   \node at (.6,-.1) {\tiny ${\g}$};
\end{tikzpicture}
\end{array}.
\]
In fact this equality follows from the first case in~\eqref{eqn:rel-2KM-3} applied to the strands labelled $m+1$ and $m$, remarking that the term coming from the right-hand side in this equality vanishes by the first case in~\eqref{eqn:rel-2KM-2}.
\end{proof}

The following lemma is a straightforward consequence of relation~\eqref{eqn:rel-2KM-3}; details are left to the reader.
(In this case we do not need any restriction on $\gamma$, and do not have to work in $\mathcal{U}_+(\hgl_p)$.)

\begin{lem}
\label{lem:crossings-infty-1}
For any $\gamma \in P$, in $\mathcal{U}(\hgl_p)$ we have
\[
      \begin{array}{c}
\begin{tikzpicture}[baseline = 0,xscale=1.5,yscale=1.2]
	\draw[<-,thick,green] (0.45,.8) to (-0.45,-.4);
	\draw[->,thick,green] (0.45,-.4) to (-0.45,.8);
        \draw[-,thick] (0,-.4) to[out=90,in=-90] (-.45,0.2);
        \draw[->,thick] (-0.45,0.2) to[out=90,in=-90] (0,0.8);
   \node at (-0.45,-.6) {\tiny $1$};
   \node at (0,-.6) {\tiny $\infty_m$};
  \node at (0.45,-.6) {\tiny $1$};
   \node at (.5,-.1) {\tiny ${\g}$};
\end{tikzpicture}
\end{array}
=
      \begin{array}{c}
\begin{tikzpicture}[baseline = 0,xscale=1.5,yscale=1.2]
	\draw[<-,thick,green] (0.45,.8) to (-0.45,-.4);
	\draw[->,thick,green] (0.45,-.4) to (-0.45,.8);
        \draw[-,thick] (0,-.4) to[out=90,in=-90] (.45,0.2);
        \draw[->,thick] (0.45,0.2) to[out=90,in=-90] (0,0.8);
   \node at (-0.45,-.6) {\tiny $1$};
   \node at (0,-.6) {\tiny $\infty_m$};
  \node at (0.45,-.6) {\tiny $1$};
   \node at (.5,-.1) {\tiny ${\g}$};
\end{tikzpicture}
\end{array}
+
      \begin{array}{c}
\begin{tikzpicture}[baseline = 0,xscale=1.5,yscale=1.2]
	\draw[<-,thick,green] (0.45,.8) to (0.45,-.4);
	\draw[->,thick,green] (-0.45,-.4) to (-0.45,.8);
        \draw[->,thick] (0,-.4) to (0,0.8);
   \node at (-0.45,-.6) {\tiny $1$};
   \node at (0,-.6) {\tiny $\infty_m$};
  \node at (0.45,-.6) {\tiny $1$};
   \node at (.7,-.1) {\tiny ${\g}$};
\end{tikzpicture}
\end{array}.
\]
\end{lem}

\begin{lem}
\label{lem:crossings-infty-m-1}
For any $\gamma \in P$, in $\mathcal{U}(\hgl_p)$ we have
\[
      \begin{array}{c}
\begin{tikzpicture}[baseline = 0,xscale=1.5,yscale=1.2]
	\draw[<-,thick,green] (0.45,.8) to (-0.45,-.4);
	\draw[->,thick,green] (0.45,-.4) to (-0.45,.8);
        \draw[-,thick] (0,-.4) to[out=90,in=-90] (-.45,0.2);
        \draw[->,thick] (-0.45,0.2) to[out=90,in=-90] (0,0.8);
   \node at (-0.45,-.6) {\tiny $m-1$};
   \node at (0,-.6) {\tiny $\infty_m$};
  \node at (0.45,-.6) {\tiny $m-1$};
   \node at (.5,-.1) {\tiny ${\g}$};
\end{tikzpicture}
\end{array}
=
      \begin{array}{c}
\begin{tikzpicture}[baseline = 0,xscale=1.5,yscale=1.2]
	\draw[<-,thick,green] (0.45,.8) to (-0.45,-.4);
	\draw[->,thick,green] (0.45,-.4) to (-0.45,.8);
        \draw[-,thick] (0,-.4) to[out=90,in=-90] (.45,0.2);
        \draw[->,thick] (0.45,0.2) to[out=90,in=-90] (0,0.8);
   \node at (-0.45,-.6) {\tiny $m-1$};
   \node at (0,-.6) {\tiny $\infty_m$};
  \node at (0.45,-.6) {\tiny $m-1$};
   \node at (.5,-.1) {\tiny ${\g}$};
\end{tikzpicture}
\end{array}
-
      \begin{array}{c}
\begin{tikzpicture}[baseline = 0,xscale=1.5,yscale=1.2]
	\draw[<-,thick,green] (0.45,.8) to (0.45,-.4);
	\draw[->,thick,green] (-0.45,-.4) to (-0.45,.8);
        \draw[->,thick] (0,-.4) to (0,0.8);
   \node at (-0.45,-.6) {\tiny $m-1$};
   \node at (0,-.6) {\tiny $\infty_m$};
  \node at (0.45,-.6) {\tiny $m-1$};
   \node at (.7,-.1) {\tiny ${\g}$};
\end{tikzpicture}
\end{array}.
\]
\end{lem}

\begin{proof}
If $m=p$ then this relation is an instance of the first case in~\eqref{eqn:rel-2KM-3}. If $m \leq p-1$, it suffices to write the strand labelled $\infty_m$ as a combination of strands labelled by $\infty_{m+1}$ and $m$, to apply the first case in~\eqref{eqn:rel-2KM-3} to the strands labelled $m$ and $m-1$, and then to use the second cases in~\eqref{eqn:rel-2KM-2} and in~\eqref{eqn:rel-2KM-3}; details are left to the reader.
\end{proof}

\begin{lem}
\label{lem:crossings-2infty-1}
For any $\gamma \in Y_m$, in $\mathcal{U}_+(\hgl_p)$ we have
\[
      \begin{array}{c}
\begin{tikzpicture}[baseline = 0,xscale=1.5,yscale=1.2]
	\draw[<-,thick] (0.45,.8) to (-0.45,-.4);
	\draw[->,thick] (0.45,-.4) to (-0.45,.8);
        \draw[-,thick,green] (0,-.4) to[out=90,in=-90] (-.45,0.2);
        \draw[->,thick,green] (-0.45,0.2) to[out=90,in=-90] (0,0.8);
   \node at (-0.45,-.6) {\tiny $\infty_m$};
   \node at (0,-.6) {\tiny $1$};
  \node at (0.45,-.6) {\tiny $\infty_m$};
   \node at (.5,-.1) {\tiny ${\g}$};
\end{tikzpicture}
\end{array}
=
      \begin{array}{c}
\begin{tikzpicture}[baseline = 0,xscale=1.5,yscale=1.2]
	\draw[<-,thick] (0.45,.8) to (-0.45,-.4);
	\draw[->,thick] (0.45,-.4) to (-0.45,.8);
        \draw[-,thick,green] (0,-.4) to[out=90,in=-90] (.45,0.2);
        \draw[->,thick,green] (0.45,0.2) to[out=90,in=-90] (0,0.8);
   \node at (-0.45,-.6) {\tiny $\infty_m$};
   \node at (0,-.6) {\tiny $1$};
  \node at (0.45,-.6) {\tiny $\infty_m$};
   \node at (.5,-.1) {\tiny ${\g}$};
\end{tikzpicture}
\end{array}
-
      \begin{array}{c}
\begin{tikzpicture}[baseline = 0,xscale=1.5,yscale=1.2]
	\draw[<-,thick] (0.45,.8) to (0.45,-.4);
	\draw[->,thick] (-0.45,-.4) to (-0.45,.8);
        \draw[->,thick,green] (0,-.4) to (0,0.8);
   \node at (-0.45,-.6) {\tiny $\infty_m$};
   \node at (0,-.6) {\tiny $1$};
  \node at (0.45,-.6) {\tiny $\infty_m$};
   \node at (.7,-.1) {\tiny ${\g}$};
\end{tikzpicture}
\end{array}.
\]
\end{lem}

\begin{proof}
We prove the formula once again by downward induction on $m$. If $m=p$, then the relation is an instance of the first case in~\eqref{eqn:rel-2KM-3}. Now assume that $m<p$. We have
\begin{multline*}
\begin{array}{c}
\begin{tikzpicture}[baseline = 0,yscale=1.3,xscale=1.4]
	\draw[<-,thick] (0.45,.8) to (-0.45,-.4);
	\draw[->,thick] (0.45,-.4) to (-0.45,.8);
        \draw[-,thick,green] (0,-.4) to[out=90,in=-90] (-.45,0.2);
        \draw[->,thick,green] (-0.45,0.2) to[out=90,in=-90] (0,0.8);
   \node at (-0.45,-.5) {\tiny $\infty_m$};
   \node at (0,-.5) {\tiny $1$};
  \node at (0.5,-.5) {\tiny $\infty_m$};
   \node at (.5,-.1) {\tiny ${\g}$};
\end{tikzpicture}
\end{array}
=
-
\begin{array}{c}
\begin{tikzpicture}[baseline = 0,yscale=1.3,xscale=2.3]
	\draw[<-,thick] (0.26,.8) to (-0.64,-.4);
	        \draw[-,thick,green] (0,-.4) to[out=90,in=-90] (-.45,0.2);
        \draw[->,thick,green] (-0.45,0.2) to[out=90,in=-90] (0,0.8);
        	\draw[->,thick] (0.3,-.4) to (-0.6,.8);
	\draw[<-,thick,red] (0.55,.8) to (-0.35,-.4);
	\draw[->,thick,red] (0.55,-.4) to (-0.35,.8);
   \node at (-0.64,-.5) {\tiny $\infty_{m+1}$};
   \node at (0,-.5) {\tiny $1$};
  \node at (0.3,-.5) {\tiny $\infty_{m+1}$};
  \node at (-0.35,-.5) {\tiny $m$};
  \node at (0.55,-.5) {\tiny $m$};
   \node at (.6,-.1) {\tiny ${\g}$};
\end{tikzpicture}
\end{array}
\\
=
-
\begin{array}{c}
\begin{tikzpicture}[baseline = 0,yscale=1.3,xscale=2.3]
	\draw[<-,thick] (0.26,.8) to (-0.64,-.4);
	\draw[->,thick,green] (0,-.4) to[out=90,in=-90] (-.15,0) to[out=90,in=-90] (-.05,.2) to[out=90,in=-90] (-.15,.4) to[out=90,in=-90] (0,.8);
        	\draw[->,thick] (0.3,-.4) to (-0.6,.8);
	\draw[<-,thick,red] (0.55,.8) to (-0.35,-.4);
	\draw[->,thick,red] (0.55,-.4) to (-0.35,.8);
   \node at (-0.05,-.5) {\tiny $1$};
  \node at (-0.35,-.5) {\tiny $m$};
  \node at (0.55,-.5) {\tiny $m$};
   \node at (.6,-.1) {\tiny ${\g}$};
\end{tikzpicture}
\end{array}
+
\begin{array}{c}
\begin{tikzpicture}[baseline = 0,yscale=1.3,xscale=2]
	\draw[<-,thick] (-0.64,.8) to (-0.64,-.4);
	        \draw[->,thick,green] (0,-.4) to[out=90,in=-90] (-.2,.2) to[out=90,in=-90] (0,0.8);
        	\draw[->,thick] (0.3,-.4) to[out=90,in=-90] (0,.2) to[out=90,in=-90] (0.3,.8);
	\draw[<-,thick,red] (0.55,.8) to (-0.35,-.4);
	\draw[->,thick,red] (0.55,-.4) to (-0.35,.8);
   \node at (0,-.5) {\tiny $1$};
  \node at (-0.35,-.5) {\tiny $m$};
  \node at (0.55,-.5) {\tiny $m$};
   \node at (.6,-.1) {\tiny ${\g}$};
\end{tikzpicture}
\end{array}
\end{multline*}
by induction (which is applicable since $\gamma+2\alpha_m \in Y_{m+1}$). In the final expression, using repeatedly the second case in~\eqref{eqn:rel-2KM-3} we see that the first term coincides with the first term in the right-hand side of the equality we wish to prove. Now, consider the second term. Using Lemma~\ref{lem:crossings-infty-m-1} and the second case in~\eqref{eqn:rel-2KM-2}, we see that to conclude it suffices to prove that
\[
      \begin{array}{c}
\begin{tikzpicture}[baseline = 0,xscale=1.5,yscale=1.2]
	\draw[<-,thick,red] (0.45,.8) to (-0.45,-.4);
	\draw[->,thick,red] (0.45,-.4) to (-0.45,.8);
        \draw[-,thick] (0,-.4) to[out=90,in=-90] (.45,0.2);
        \draw[->,thick] (0.45,0.2) to[out=90,in=-90] (0,0.8);
   \node at (-0.45,-.6) {\tiny $m$};
   \node at (0,-.6) {\tiny $\infty_{m+1}$};
  \node at (0.45,-.6) {\tiny $m$};
   \node at (.5,-.1) {\tiny ${\g}$};
\end{tikzpicture}
\end{array}
=0
\]
in $\mathcal{U}_+(\hgl_p)$. This follows from the fact that
\[
\gamma+(\alpha_0 + \cdots + \alpha_{m+1}) = \gamma + (\delta + \varepsilon_1 - \varepsilon_{m+1})
\]
does not belong to $P_+$.
\end{proof}

We finish this subsection with two lemmas proving some relations involving dots.

\begin{lem}
\label{lem:crossings-dot-infty}
For any $\g \in Y_m$, in $\mathcal{U}_+(\hgl_p)$ we have:
\begin{equation*}
  \begin{array}{c}
    \begin{tikzpicture}[thick,xscale=0.15,yscale=0.25]
      \draw[->] (-2,-2) to (2,2);
      \draw[->] (2,-2) to (-2,2);
      \node at (-1,-1) {$\bullet$};
      \node at (-2,-2.6) {\tiny $\infty_m$};      
      \node at (2,-2.6) {\tiny $\infty_m$};
      \node at (2.5,0) {\tiny ${\g}$};
    \end{tikzpicture}
\end{array}
-
  \begin{array}{c}
    \begin{tikzpicture}[thick,xscale=0.15,yscale=0.25]
      \draw[->] (-2,-2) to (2,2);
      \draw[->] (2,-2) to (-2,2);
      \node at (1,1) {$\bullet$};
      \node at (-2,-2.6) {\tiny $\infty_m$};      
      \node at (2,-2.6) {\tiny $\infty_m$};
      \node at (2.5,0) {\tiny ${\g}$};
    \end{tikzpicture}
\end{array}
= 
  \begin{array}{c}
    \begin{tikzpicture}[thick,xscale=0.15,yscale=0.25]
      \draw[->] (-2,-2) to (2,2);
      \draw[->] (2,-2) to (-2,2);
      \node at (-1,1) {$\bullet$};
      \node at (-2,-2.6) {\tiny $\infty_m$};      
      \node at (2,-2.6) {\tiny $\infty_m$};
      \node at (2.5,0) {\tiny ${\g}$};
    \end{tikzpicture}
\end{array}
-
  \begin{array}{c}
    \begin{tikzpicture}[thick,xscale=0.15,yscale=0.25]
      \draw[->] (-2,-2) to (2,2);
      \draw[->] (2,-2) to (-2,2);
      \node at (1,-1) {$\bullet$};
      \node at (-2,-2.6) {\tiny $\infty_m$};      
      \node at (2,-2.6) {\tiny $\infty_m$};
      \node at (2.5,0) {\tiny ${\g}$};
    \end{tikzpicture}
\end{array}
= 
\begin{array}{c}
    \begin{tikzpicture}[thick,xscale=0.15,yscale=0.25]
      \draw[->] (-2,-2) to (-2,2);
      \draw[->] (1,-2) to (1,2);
     \node at (-2,-2.6) {\tiny $\infty_m$};      
     \node at (1,-2.6) {\tiny $\infty_m$};
      \node at (2.5,0) {\tiny ${\g}$};
    \end{tikzpicture}
\end{array}.
\end{equation*}
\end{lem}

\begin{proof}
We prove the second equality; the proof of the first one is similar. By definition we have
\[
  \begin{array}{c}
    \begin{tikzpicture}[thick,xscale=0.15,yscale=0.25]
      \draw[->] (-2,-2) to (2,2);
      \draw[->] (2,-2) to (-2,2);
      \node at (-1,1) {$\bullet$};
      \node at (-2,-2.6) {\tiny $\infty_m$};      
      \node at (2,-2.6) {\tiny $\infty_m$};
      \node at (2.5,0) {\tiny ${\g}$};
    \end{tikzpicture}
\end{array}
-
  \begin{array}{c}
    \begin{tikzpicture}[thick,xscale=0.15,yscale=0.25]
      \draw[->] (-2,-2) to (2,2);
      \draw[->] (2,-2) to (-2,2);
      \node at (1,-1) {$\bullet$};
      \node at (-2,-2.6) {\tiny $\infty_m$};      
      \node at (2,-2.6) {\tiny $\infty_m$};
      \node at (2.5,0) {\tiny ${\g}$};
    \end{tikzpicture}
\end{array}
=
    - \begin{array}{c}
    \begin{tikzpicture}[thick,xscale=-0.25,yscale=0.3]
          \draw[->] (0,-2) to (4,2);
      \draw[->] (4,-2) to (0,2);
      \draw[green,->] (-2,-2) to (2,2);
      \draw[green,->] (2,-2) to (-2,2);
      \node at (1.5,1.5) {$\color{green}\bullet$};
      \node at (0,-2.4) {\tiny $\infty_{m+1}$};      
      \node at (3.5,-2.4) {\tiny $\infty_{m+1}$};
      \node at (-2.5,0) {\tiny ${\g}$};
    \end{tikzpicture}
\end{array}
+ 
\begin{array}{c}
    \begin{tikzpicture}[thick,xscale=-0.25,yscale=0.3]
          \draw[->] (0,-2) to (4,2);
      \draw[->] (4,-2) to (0,2);
      \draw[green,->] (-2,-2) to (2,2);
      \draw[green,->] (2,-2) to (-2,2);
      \node at (-0.5,-0.5) {\color{green}$\bullet$};
      \node at (0,-2.4) {\tiny $\infty_{m+1}$};      
      \node at (3.5,-2.4) {\tiny $\infty_{m+1}$};
      \node at (-2.5,0) {\tiny ${\g}$};
    \end{tikzpicture}
\end{array},
\]
where the green strands are labelled by $m$. Using the second case in~\eqref{eqn:rel-2KM-1} to pull the green dot through the black strand, and then the first case in~\eqref{eqn:rel-2KM-1}, we see that the right-hand side is equal to
\[
-
      \begin{array}{c}
\begin{tikzpicture}[baseline = 0,xscale=1.5,yscale=1.2]
	\draw[<-,thick] (0.45,.8) to (-0.45,-.4);
	\draw[->,thick] (0.45,-.4) to (-0.45,.8);
		\draw[->,thick,green] (0.65,-.4) to (0.65,.8);
        \draw[-,thick,green] (0,-.4) to[out=90,in=-90] (.45,0.2);
        \draw[->,thick,green] (0.45,0.2) to[out=90,in=-90] (0,0.8);
   \node at (-0.45,-.6) {\tiny $\infty_{m+1}$};
  \node at (0.45,-.6) {\tiny $\infty_{m+1}$};
   \node at (.8,-.1) {\tiny ${\g}$};
\end{tikzpicture}
\end{array}
=
\begin{array}{c}
    \begin{tikzpicture}[thick,xscale=0.15,yscale=0.4]
      \draw[->] (-2,-2) to (-2,2);
      \draw[->] (1,-2) to (1,2);
     \node at (-2,-2.6) {\tiny $\infty_m$};      
     \node at (1,-2.6) {\tiny $\infty_m$};
      \node at (2.5,0) {\tiny ${\g}$};
    \end{tikzpicture}
\end{array}
-
      \begin{array}{c}
\begin{tikzpicture}[baseline = 0,xscale=1.5,yscale=1.2]
	\draw[<-,thick] (0.45,.8) to (-0.45,-.4);
	\draw[->,thick] (0.45,-.4) to (-0.45,.8);
		\draw[->,thick,green] (0.65,-.4) to (0.65,.8);
        \draw[-,thick,green] (0,-.4) to[out=90,in=-90] (-.25,0.2);
        \draw[->,thick,green] (-0.25,0.2) to[out=90,in=-90] (0,0.8);
   \node at (-0.45,-.6) {\tiny $\infty_{m+1}$};
  \node at (0.45,-.6) {\tiny $\infty_{m+1}$};
   \node at (.8,-.1) {\tiny ${\g}$};
\end{tikzpicture}
\end{array},
\]
where the equality follows from Lemma~\ref{lem:crossings-infty}. (This result can be applied because $\gamma+\alpha_m$ belongs to $Y_{m+1}$.) To conclude it suffices to remark that the rightmost term vanishes in $\mathcal{U}_+(\hgl_p)$ since it contains a region labelled by
\[
\gamma+\alpha_m + 2(\alpha_{m+1} + \cdots + \alpha_0) = \gamma+(\varepsilon_{m+1} - \varepsilon_m) + 2(\delta + \varepsilon_1 - \varepsilon_{m+1}),
\]
which does not belong to $P_+$.
\end{proof}

The following lemma shows that, in the definition of the dotted arrow labelled by $\infty_m$, we could have put the dot on any strand.

\begin{lem}
\label{lem:dotsmove}
For any $\g \in Y_m$, in $\mathcal{U}_+(\hgl_p)$ we have
\[
\begin{array}{c}
    \begin{tikzpicture}[thick,yscale=0.3,xscale=0.4,baseline]
      \draw[->] (0,-2) to (0,2);
      \node at (0,-2.4) {\tiny $0$};
      \node at (0,0) {$\bullet$};
      \draw[->] (1.5,-2) to (1.5,2);
      \node at (1.5,-2.4) {\tiny $p-1$};
      \draw[->] (3.5,-2) to (3.5,2);
      \node at (3.5,-2.4) {\tiny $m+1$};
      \draw[->] (5,-2) to (5,2);
      \node at (5,-2.4) {\tiny $m$};
      \node at (2.6,0) {$\dots$};
      \node at (5.5,0) {\tiny $\g$};
    \end{tikzpicture}
\end{array}
=
\begin{array}{c}
    \begin{tikzpicture}[thick,yscale=0.3,xscale=0.4,baseline]
      \draw[->] (0,-2) to (0,2);
      \node at (0,-2.4) {\tiny $0$};
      \node at (1.5,0) {$\bullet$};
      \draw[->] (1.5,-2) to (1.5,2);
      \node at (1.5,-2.4) {\tiny $p-1$};
      \draw[->] (3.5,-2) to (3.5,2);
      \node at (3.5,-2.4) {\tiny $m+1$};
      \draw[->] (5,-2) to (5,2);
      \node at (5,-2.4) {\tiny $m$};
      \node at (2.6,0) {$\dots$};
      \node at (5.5,0) {\tiny $\g$};
    \end{tikzpicture}
\end{array}
= \cdots =
\begin{array}{c}
    \begin{tikzpicture}[thick,yscale=0.3,xscale=0.4,baseline]
      \draw[->] (0,-2) to (0,2);
      \node at (0,-2.4) {\tiny $0$};
      \node at (5,0) {$\bullet$};
      \draw[->] (1.5,-2) to (1.5,2);
      \node at (1.5,-2.4) {\tiny $p-1$};
      \draw[->] (3.5,-2) to (3.5,2);
      \node at (3.5,-2.4) {\tiny $m+1$};
      \draw[->] (5,-2) to (5,2);
      \node at (5,-2.4) {\tiny $m$};
      \node at (2.6,0) {$\dots$};
      \node at (5.7,0) {\tiny $\g$};
    \end{tikzpicture}
\end{array}
.
\]
\end{lem}

\begin{proof}
If $m=p$ there is nothing to prove. So, we assume that $m \leq p-1$.

Note that, for any $k \in \{m, \cdots, p-1\}$, the weight 
\[
\g + \a_m + \dots + \a_{k-1} + \a_{k+1} = \begin{cases}
\gamma + \varepsilon_k - \varepsilon_m + \varepsilon_{k+2} - \varepsilon_{k+1} & \text{if $k \leq p-2$} \\
\gamma + \varepsilon_{p-1} - \varepsilon_m + (\delta - \varepsilon_p+\varepsilon_1) & \text{if $k=p-1$}
\end{cases}
\]
does not belong to $P_+$  (where $\a_p := \alpha_0$). Hence, using relation~\eqref{eqn:rel-2KM-2} we deduce that in $\mathcal{U}^{[n]}(\hgl_p)$ we have
\[
0 = 
\begin{array}{c}
    \begin{tikzpicture}[thick,yscale=0.3,xscale=0.4,baseline]
      \draw[->] (0,-2) to (0,2);
      \node at (0,-2.4) {\tiny $0$};
      \draw[->] (5,-2) to (5,2);
      \node at (5,-2.4) {\tiny $m$};
      \node at (1.2,0) {$\dots$};
      \node at (3.8,0) {$\dots$};
      \draw[->] (2,-2) to[out=90,in=-90] (3,0) to[out=90,in=-90] (2,2);
      \draw[->] (3,-2) to[out=90,in=-90] (2,0) to[out=90,in=-90] (3,2);
      \node at (5.5,0) {\tiny $\g$};
\node at (3,-2.4) {\tiny $k$};
    \end{tikzpicture}
\end{array}
=
t_{k+1,k}
\begin{array}{c}
    \begin{tikzpicture}[thick,yscale=0.3,xscale=0.4,baseline]
      \draw[->] (0,-2) to (0,2);
      \node at (0,-2.4) {\tiny $0$};
      \draw[->] (5,-2) to (5,2);
      \node at (5,-2.4) {\tiny $m$};
      \node at (1,0) {$\dots$};
      \node at (4,0) {$\dots$};
      \draw[->] (2,-2) to (2,2);
      \draw[->] (3,-2) to (3,2);
      \node at (2,0) {$\bullet$};
      \node at (5.5,0) {\tiny $\g$};
\node at (3,-2.4) {\tiny $k$};
    \end{tikzpicture}
\end{array}
+
t_{k,k+1}
\begin{array}{c}
    \begin{tikzpicture}[thick,yscale=0.3,xscale=0.4,baseline]
      \draw[->] (0,-2) to (0,2);
      \node at (0,-2.4) {\tiny $0$};
      \draw[->] (5,-2) to (5,2);
      \node at (5,-2.4) {\tiny $m$};
      \node at (1,0) {$\dots$};
      \node at (4,0) {$\dots$};
      \draw[->] (2,-2) to (2,2);
      \draw[->] (3,-2) to (3,2);
      \node at (3,0) {$\bullet$};
      \node at (5.5,0) {\tiny $\g$};
\node at (3,-2.4) {\tiny $k$};
    \end{tikzpicture}
\end{array}
\]
(where $t_{p,p-1}:=t_{0,p-1}$ and $t_{p-1,p}:=t_{p-1,0}$).
The lemma now follows because $t_{ij} = -t_{ji}$ if $i \rel j$.
\end{proof}

\subsection{Restriction of the $2$-representation to $\mathcal{U}(\hgl_n)$}

From now on we restrict to the case $m=n$, and write $\infty$ for $\infty_n$. We define the ``cup'' and ``cap'' $2$-morphisms as follows:
\[
\begin{array}{c}
      \begin{tikzpicture}[baseline = 0]
        \draw[->,thick] (-0.3,0.3) to[out=-90,in=180] (0,-0.2) to[out=0,in=-90] (0.3,0.3);
        \node at (0.5,.2) {\tiny ${\g}$};
          \node at (-0.3,.4) {\tiny ${\infty}$};
      \end{tikzpicture}
\end{array}
=
\begin{array}{c}
    \begin{tikzpicture}[thick,xscale=0.3,baseline = 0,yscale=-0.3]
            \draw[->] (-3,-2) to[out=90,in=180] (0,2) to[out=0,in=90] (3,-2);
            \draw[->] (-1,-2) to[out=90,in=180] (0,0) to[out=0,in=90] (1,-2);
      \node at (-2,-1.2) {\tiny $\dots$};
      \node at (2,-1.2) {\tiny $\dots$};
      \node at (3,-2.4) {\tiny $n$};
      \node at (1,-2.4) {\tiny $0$};
      \node at (-1,-2.4) {\tiny $0$};
      \node at (-3,-2.4) {\tiny $n$};
      \node at (3.2,0) {\tiny${\g}$};
    \end{tikzpicture}
\end{array},
\quad
\begin{array}{c}
      \begin{tikzpicture}[scale=0.3,baseline = 0]
        \draw[->,thick] (-1,-1) to[out=90,in=180] (0,1) to[out=0,in=90] (1,-1);
        \node at (1.3,.2) {\tiny ${\g}$};
          \node at (-1,-1.4) {\tiny ${\infty}$};
      \end{tikzpicture}
\end{array}
=
\begin{array}{c}
    \begin{tikzpicture}[thick,scale=0.3,baseline = 0]
            \draw[->] (-3,-2) to[out=90,in=180] (0,2) to[out=0,in=90] (3,-2);
            \draw[->] (-1,-2) to[out=90,in=180] (0,0) to[out=0,in=90] (1,-2);
      \node at (-2,-1.2) {\tiny $\dots$};
      \node at (2,-1.2) {\tiny $\dots$};
      \node at (3,-2.4) {\tiny $0$};
      \node at (1,-2.4) {\tiny $n$};
      \node at (-1,-2.4) {\tiny $n$};
      \node at (-3,-2.4) {\tiny $0$};
      \node at (3.2,0) {\tiny ${\g}$};
    \end{tikzpicture}
\end{array}.
\]


We will view $\{1, \dots, n-1, \infty\}$ as a quiver
 as follows:
\begin{equation}
\label{eqn:quiver-hgln}
\begin{array}{c}
  \begin{tikzpicture}[xscale=0.8,yscale=0.5,baseline]
    \node (a) at (1,0) {$1$};
    \node (b) at (2,0) {$2$};
    \node (c) at (3,0) {};
    \node (d) at (3.5,0) {};
    \node (e) at (5,0) {$n-1$};
    \node (i) at (3,2) {$\infty$};
    \node at (3.3,0) {$\dots$};
    \draw[->] (a) to (b);
    \draw[->] (b) to (c);
    \draw[->] (d) to (e);
    \draw[->] (e) to (i);
    \draw[->] (i) to (a);
  \end{tikzpicture}
  \end{array} .
\end{equation}
We extend the definition of our scalars $t_{ij} \in \{ \pm 1 \}$ to $i,j \in \{1, \dots, n-1, \infty \}$ as follows:
\begin{equation*}
t_{ij} = \begin{cases} -1 & \text{if $i \to j$;}\\
1 & \text{if $j \to i$ or $i \norel j$.} \end{cases}
\end{equation*}
(This is clearly consistent with the notation from~\S\ref{subsec:R}.)

Associated with these scalars we have the $2$-category $\mathcal{U}(\hgl_n)$ as in~\cite{Brundan} (i.e.~as in~\S\ref{sss:2KM-def} with $p$ replaced by $n$). 
Recall that our goal is to define a ``restriction'' of the $2$-representation of $\mathcal{U}(\hgl_p)$ (or in fact, as noticed above, of $\mathcal{U}_+(\hgl_p)$) considered in Theorem~\ref{thm:2rep-GLn} to $\mathcal{U}(\hgl_n)$. In view of this,
to a weight $\gamma$ for $\hgl_n$ we associate the category $\Repp_{\imath_n(\gamma)}(G)$ if $\gamma$ is a weight of $\bigwedge^n \nat_n$ (which we view as an element in $P$ in the obvious way, see~\S\ref{ss:restriction-combinatorics}) or $0$ otherwise. To a generator $E_i 1_\gamma$ (with $i \in \{1, \cdots, n-1\}$) we associate the action on $\Repp_{\imath_n(\gamma)}(G)$ of the element denoted similarly in $\mathcal{U}(\hgl_p)$ via the action of Theorem~\ref{thm:2rep-GLn} (or $0$). When $i=\infty$ we proceed similarly, using the $1$-morphism $E_{\infty} 1_\gamma$ defined in~\eqref{eqn:def-Einfty} (in the case $m=n$). Finally, to the various generating $2$-morphisms of $\mathcal{U}(\hgl_n)$ we associate the action of the corresponding $2$-morphisms in $\mathcal{U}_+(\hgl_p)$ via this same action of Theorem~\ref{thm:2rep-GLn} (or $0$).
The main result of the present section is the following.

\begin{thm}
\label{thm:reduction-2rep}
  These data define a $2$-representation of $\mathcal{U}(\hgl_n)$.
\end{thm}

\begin{proof}
A large part of the work in view of proving Theorem~\ref{thm:reduction-2rep} has been done in~\S\ref{ss:first-relations}. In fact, we first need to check that, in our action on $\Rep^{[n]}(G)$, the generating $2$-morphisms satisfy relations~\eqref{eqn:rel-2KM-1}--\eqref{eqn:rel-2KM-3}. 
Of course, it will suffice to prove that these relations hold in $\mathcal{U}_+(\hgl_p)$.
To check one relation below we will use the following notation:
\[
\begin{array}{c}
    \begin{tikzpicture}[thick,scale=0.3,baseline]
      \draw[double,double distance=2pt,->] (0,-2) to (0,2);
      \node at (0.8,0) {\tiny ${\g}$};
    \end{tikzpicture}
\end{array}
:= 
\begin{array}{c}
    \begin{tikzpicture}[thick,yscale=0.3,xscale=0.5,baseline]
      \draw[->] (0,-2) to (0,2);
      \node at (0,-2.4) {\tiny $p-1$};
      \draw[->] (1.5,-2) to (1.5,2);
      \node at (1.5,-2.4) {\tiny $p-2$};
      \draw[->] (3.5,-2) to (3.5,2);
      \node at (3.5,-2.4) {\tiny $n+1$};
      \draw[->] (5,-2) to (5,2);
      \node at (5,-2.4) {\tiny $n$};
      \node at (2.5,0) {$\dots$};
      \node at (5.5,0) {\tiny ${\g}$};
    \end{tikzpicture}
\end{array} .
\]
(We interpret the right-hand side as an empty diagram in case $n=p$.)

Relation~\eqref{eqn:rel-2KM-1} follows directly from the corresponding relation in $\mathcal{U}(\hgl_p)$ in the case $i \neq j$ or $i=j \neq \infty$. And the case $i=j = \infty$ has been checked in
Lemma~\ref{lem:crossings-dot-infty}.

Now we consider relation~\eqref{eqn:rel-2KM-2}. If neither $i$ nor $j$ is $\infty$ then the relation follows directly from the corresponding relation in $\mathcal{U}(\hgl_p)$. Similarly, the relation is clear if $i$ or $j$ is $\infty$ and the other label belongs to $\{2, \cdots, n-2\}$. The case $i=j=\infty$ has been checked in Lemma~\ref{lem:vanishing-infty}.

The remaining cases are $i = \infty$, $j \in \{1, n-1\}$ and $i \in \{1, n-1\}$, $j
= \infty$. We do the case $i = 1$, $j = \infty$ and leave the (very
similar) other cases to the reader:
  \begin{gather*}
      \begin{array}{c}
\begin{tikzpicture}[baseline = 0]
	\draw[->,thick] (0.28,.4) to[out=90,in=-90] (-0.28,1.1);
	\draw[->,thick] (-0.28,.4) to[out=90,in=-90] (0.28,1.1);
	\draw[-,thick] (0.28,-.3) to[out=90,in=-90] (-0.28,.4);
	\draw[-,thick] (-0.28,-.3) to[out=90,in=-90] (0.28,.4);
  \node at (-0.28,-.5) {\tiny $1$};
  \node at (0.28,-.5) {\tiny $\infty$};
  \node at (.43,.4) {\tiny ${\g}$};
\end{tikzpicture}
\end{array}
=
      \begin{array}{c}
    \begin{tikzpicture}[thick,xscale=0.15,yscale=0.5]
      \draw[->] (-2,-2) to[out=90,in=-90] (2,0) to[out=90,in=-90] (-2,2);
      \draw[double distance=2pt,double,<-,scale=-1] (-2.4,-2) to[out=60,in=-110] (1.2,0) to[out=110,in=-60] (-2.4,2);
      \draw[<-,scale=-1] (-.8,-2) to[out=60,in=-110] (3.2,0) to[out=110,in=-60] (-.8,2);
      \node at (.8,-2.3) {\tiny $0$};   
      \node at (-2.2,-2.3) {\tiny $1$};   
\node at (3.2,0) {\tiny ${\g}$};
     \end{tikzpicture}
\end{array}
 =
       \begin{array}{c}
    \begin{tikzpicture}[thick,xscale=0.15,yscale=0.5]
      \draw[->] (-2,-2) to[out=90,in=-110] (1,0) to[out=110,in=-90] (-2,2);
      \draw[double distance=2pt,double,<-,scale=-1] (-2.7,-2) to (-2.7,2);
      \draw[<-,scale=-1] (-1,-2) to[out=90,in=-110] (3,0) to[out=110,in=-90] (-1,2);
      \node at (1,-2.3) {\tiny $0$};   
      \node at (-2.2,-2.3) {\tiny $1$};   
 \node at (4.1,0) {\tiny ${\g}$};
     \end{tikzpicture}
\end{array}
= t_{1, 0} 
      \begin{array}{c}
    \begin{tikzpicture}[thick,xscale=0.1,yscale=0.3]
      \draw[->] (-2,-2) to (-2,2);
      \draw[<-,scale=-1] (-2,-2) to (-2,2);
      \node at (-2,-2.4) {\tiny $1$};      
      \node at (2,-2.4)  {\tiny $\infty$};
      \node at (-2,0) {$\bullet$};
\node at (3.9,0) {\tiny ${\g}$};
     \end{tikzpicture}
\end{array}
+  t_{0, 1} 
      \begin{array}{c}
    \begin{tikzpicture}[thick,xscale=0.1,yscale=0.3]
      \draw[->] (-2,-2) to (-2,2);
      \draw[<-,scale=-1] (-2,-2) to (-2,2);
      \node at (-2,-2.4) {\tiny $1$};      
      \node at (2,-2.4)  {\tiny $\infty$};
      \node at (2,0) {$\bullet$};
\node at (3.9,0) {\tiny ${\g}$};
     \end{tikzpicture}
\end{array}.
\end{gather*}
Here the second equality follows from the second case in~\eqref{eqn:rel-2KM-2}, and the third equality uses the third case in relation~\eqref{eqn:rel-2KM-2} and Lemma~\ref{lem:dotsmove}. We conclude using the fact that $t_{1,0}=t_{1,\infty}$ and $t_{0,1}=t_{\infty,1}$.

Next we consider relation~\eqref{eqn:rel-2KM-3}. If $i$, $j$, and $k$ all belong to $\{1, \cdots, n-1\}$, the desired relation follows from the corresponding relation in $\mathcal{U}(\hgl_p)$. If $i \neq k$, then the relation is also a direct consequence of the corresponding relation in $\mathcal{U}(\hgl_p)$. If $i=k \norel j$ and $j \neq i$, the relation is also easy to check. 
The remaining cases are as follows:
\begin{itemize}
\item 
$i=j=k=\infty$: this case was checked in Lemma~\ref{lem:crossings-all-infty};
\item
$i=k=\infty$ and $j=n-1$: this case was checked in Lemma~\ref{lem:crossings-infty};
\item
$i = k = \infty$ and $j=1$: this case was checked in Lemma~\ref{lem:crossings-2infty-1};
\item
$i=k=1$ and $j=\infty$: this case was checked in Lemma~\ref{lem:crossings-infty-1};
\item
$i=k=n-1$, $j = \infty$: this case was checked in Lemma~\ref{lem:crossings-infty-m-1}.
\end{itemize}


To conclude the proof we observe that, as explained in the proof of Theorem~\ref{thm:2rep-GLn}, once the relations~\eqref{eqn:rel-2KM-1}--\eqref{eqn:rel-2KM-3} are established
the other relations follow automatically by general results due to (Chuang--)Rouquier.
\end{proof}

\section{From categorical $\hgl_n$-actions to $\DiagBS$-modules}
\label{sec:g2D}

In this section we continue with the setting of Sections~\ref{sec:2rep-GLn}--\ref{sec:restriction}, and we assume that $p>n \geq 3$.

\subsection{Strategy}
\label{ss:strategy}

Our goal now is to finally prove that the categorical $\hgl_n$-action on $\Rep^{[n]}(G)$ obtained from Theorem~\ref{thm:reduction-2rep} induces an action of $\DiagBS$ on $\Rep_0(G)$, where the weight $\lambda_0$ is chosen as in~\S\ref{sss:comparison-Section3}. In particular, this will prove Conjecture~\ref{conj:main} in this case (see Remark~\ref{rmk:main-conj}\eqref{it:rmk-left-right}).

First, we consider a quotient of the $2$-category $\mathcal{U}(\hgl_n)$ defined as follows.
We denote by $\mathcal{U}^{[n]}(\hgl_n)$ the $2$-category obtained from $\mathcal{U}(\hgl_n)$ by quotienting the $2$-morphisms by those which contain a region labelled by a weight which is not of the form
\[
\sum_{i=1}^n n_i \varepsilon_{i} + m\delta \qquad \text{with $n_i \in \Z_{\geq 0}$, $\sum_{i=1}^n n_i = n$, and $m \in \Z$.}
\]
Using conventions similar to those in Section~\ref{sec:restriction}, we will denote in the same way a morphism in $\mathcal{U}(\hgl_n)$ and its image in $\mathcal{U}^{[n]}(\hgl_n)$. By construction (see~\S\ref{ss:restriction-combinatorics}), the $2$-representation of $\mathcal{U}(\hgl_n)$ considered in Theorem~\ref{thm:reduction-2rep} factors through a $2$-representation of $\mathcal{U}^{[n]}(\hgl_n)$.

Recall now the category $\DiagBS$ defined in~\S\ref{ss:diag-SB}.
We set
\[
\omega := \varepsilon_1 + \cdots + \varepsilon_n.
\]
Then $\Rep_0(G)$ is the category attached to $\omega$ in the $2$-representation of $\mathcal{U}^{[n]}(\hgl_n)$ under consideration. Hence to conclude our proof it suffices
to construct a strict monoidal functor
\[
\sigma \colon \DiagBS \to \End_{\mathcal{U}^{[n]}(\hgl_n)}(\omega)
\]
sending, for each $s \in S$, the object $B_s$ to the corresponding functor $\Trans_s \Trans^s$ as defined in~\S\ref{sss:comparison-Section3}.

For this 
we identify the set $S$ of simple reflections in $\Waff$ with the set $\{1,
\cdots, n-1, \infty\}$ in such a way that any $j \in \{1, \cdots, n-1\}$ corresponds to the transposition $s_{n-j}$, and that $\infty$ corresponds to the unique simple reflection in $S \smallsetminus \Sf$. Then, our requirement on images of objects $B_s$ forces us to define $\sigma$ on
objects as follows\footnote{In what follows, in $\mathcal{U}^{[n]}(\hgl_n)$, an upward arrow will most of the time have a downward arrow labelled by the same element on its left; we indicate the label only once for simplicity.}: 
\[
B_{(i_1, \cdots, i_r)} \langle k \rangle \ \mapsto \
\begin{array}{c}
\begin{tikzpicture}[baseline = 0]
	\draw[->,thick,black] (-.7,.6) to (-.7,-.2);
	\draw[<-,thick,black] (-0.4,.6) to (-.4,-.2);
	\draw[->,thick,green] (0.4,.6) to (.4,-.2);
	\draw[<-,thick,green] (.7,.6) to (.7,-.2);
	\node at (0,0.2) {$\dots$};
        \node at (-.4,-.4) {\tiny $i_1$};
        \node at (.7,-.4) {\tiny $i_r$};
        \node at (1,.2) {\tiny $\omega$};
\end{tikzpicture}.
\end{array}
\]
On the ``dot'' and trivalent morphisms, $\sigma$ is defined as follows:
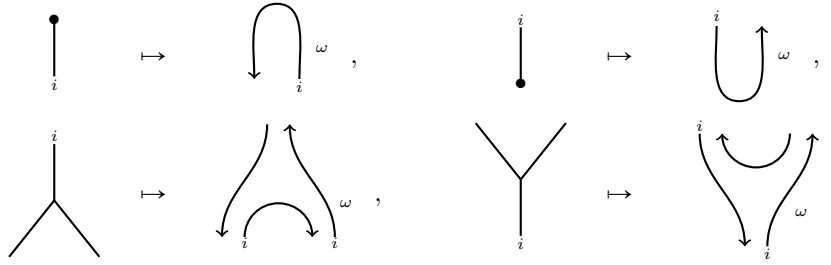
\begin{align*}
 \begin{array}{clcclc}
  \begin{array}{c}
    \begin{tikzpicture}[thick,scale=0.15]
      \draw[black] (0,-5) to (0,0);
      \node at (0,0) {$\bullet$};
      \node at (0,-5.7) {\tiny $i$};
    \end{tikzpicture}
  \end{array}
& \mapsto &\;\; 
  \begin{array}{c}
    \begin{tikzpicture}[thick,scale=0.2]
\draw[<-,black] (-1.5,-5) to[out=90,in=180] (0,0) to[out=0,in=90] (1.5,-5);
\node at (3,-3) {\tiny $\omega$};
\node at (1.5,-5.5) {\tiny $i$};
    \end{tikzpicture}  \end{array}, \qquad 
&  \begin{array}{c}
    \begin{tikzpicture}[thick,scale=-0.15]
      \draw[black] (0,-5) to (0,0);
      \node at (0,0) {$\bullet$};
      \node at (0,-5.7) {\tiny $i$};
    \end{tikzpicture}
  \end{array}
& \mapsto &\;\; 
  \begin{array}{c}
    \begin{tikzpicture}[thick,scale=-0.2]
\draw[<-,black] (-1.5,-5) to[out=90,in=180] (0,0) to[out=0,in=90] (1.5,-5);
\node at (-3,-3) {\tiny $\omega$};
\node at (1.5,-5.7) {\tiny $i$};
    \end{tikzpicture}  \end{array}, \\
  \begin{array}{c}
    \begin{tikzpicture}[thick,scale=-0.15]
      \draw[black] (-4,5) to (0,0) to (4,5);
      \draw[black] (0,-5) to (0,0);
      \node at (0,-5.7) {\tiny $i$};
    \end{tikzpicture}
  \end{array}
& \mapsto & \;\; 
  \begin{array}{c}
    \begin{tikzpicture}[thick,scale=-0.15]
\draw[<-,black] (-1,-5) to[out=90,in=-90] (-5,5);
\draw[->,black] (1,-5) to[out=90,in=-90] (5,5);
\draw[->,black] (3,5) to[out=-90,in=0] (0,2) to[out=180,in=-90] (-3,5);
\node at (-6,2) {\tiny $\omega$};
\node at (-5,5.5) {\tiny $i$};
\node at (3,5.5) {\tiny $i$};
    \end{tikzpicture}  \end{array}, 
\qquad &
  \begin{array}{c}
    \begin{tikzpicture}[thick,scale=0.15]
      \draw[black] (-4,5) to (0,0) to (4,5);
      \draw[black] (0,-5) to (0,0);
            \node at (0,-5.7) {\tiny $i$};
    \end{tikzpicture}
  \end{array}
& \mapsto & \;\; 
  \begin{array}{c}
    \begin{tikzpicture}[thick,scale=0.15]
\draw[<-,black] (-1,-5) to[out=90,in=-90] (-5,5);
\draw[->,black] (1,-5) to[out=90,in=-90] (5,5);
\draw[->,black] (3,5) to[out=-90,in=0] (0,2) to[out=180,in=-90] (-3,5);
\node at (4,-2) {\tiny $\omega$};
\node at (-5,5.7) {\tiny $i$}; 
\node at (1,-5.7) {\tiny $i$};
    \end{tikzpicture}  \end{array}.  
\end{array} 
\end{align*}
For $i, j \in \{1, \cdots, n-1,\infty\}$ with $i \norel j$ (in the quiver described in~\eqref{eqn:quiver-hgln}), $\sigma$ is defined as follows:
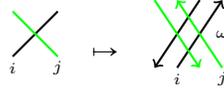
\begin{gather*}
  \begin{array}{c}
    \begin{tikzpicture}[thick,scale=0.15]
      \draw (-2,-2) to (2,2);
      \draw[green] (2,-2) to (-2,2);
      \node at (-2,-3) {\tiny $i$};
     \node at (2,-3) {\tiny $j$};
    \end{tikzpicture}
  \end{array}
 \mapsto  \;\; 
  \begin{array}{c}
    \begin{tikzpicture}[thick,scale=0.15]
     \draw[<-] (-3,-3) to (1,3);
      \draw[->] (-1,-3) to (3,3);
     \draw[green,<-] (1,-3) to (-3,3);
      \draw[green,->] (3,-3) to (-1,3);
      \node at (-1,-4) {\tiny $i$};
     \node at (3,-4) {\tiny $j$};
     \node at (3,0) {\tiny $\omega$};
    \end{tikzpicture}  \end{array} 
\end{gather*}
Finally, for $i, j \in \{1, \cdots, n-1,\infty\}$ with $j \to i$ (again in the quiver described in~\eqref{eqn:quiver-hgln}), we define $\sigma$ by:
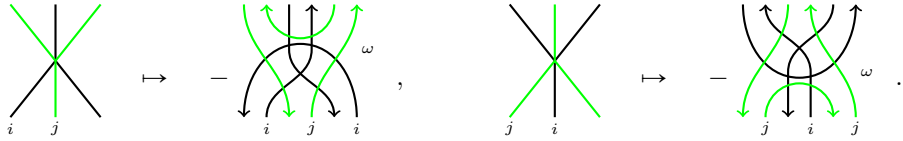
\begin{align*}
 \begin{array}{clcclc}
  \begin{array}{c}
    \begin{tikzpicture}[thick,scale=0.15]
      \draw (-4,-5) to (0,0) to (4,-5);
      \draw (0,5) to (0,0);
      \draw[green] (-4,5) to (0,0) to (4,5);
      \draw[green] (0,-5) to (0,0);
      \node at (-4,-6) {\tiny $i$};
     \node at (0,-6) {\tiny $j$};
    \end{tikzpicture}
  \end{array}
& \mapsto & \;\; -
  \begin{array}{c}
    \begin{tikzpicture}[thick,scale=0.15]
      \draw[<-] (-5,-5) to[out=90,in=180] (0,1.5) to[out=0,in=90] (5,-5);
      \draw[->] (-3,-5) to[out=90,in=-90] (1,1) to (1,5);
      \draw[<-] (3,-5) to[out=90,in=-90] (-1,1) to (-1,5);
\draw[<-,green] (-1,-5) to[out=90,in=-90] (-5,5);
\draw[->,green] (1,-5) to[out=90,in=-90] (5,5);
\draw[->,green] (3,5) to[out=-90,in=0] (0,2) to[out=180,in=-90] (-3,5);
\node (a) at (6,1) {\tiny $\omega$};
      \node at (-3,-6) {\tiny $i$};
     \node at (1,-6) {\tiny $j$};
     \node at (5,-6) {\tiny $i$};
    \end{tikzpicture}  \end{array}, \quad & \quad 
  \begin{array}{c}
    \begin{tikzpicture}[thick,scale=-0.15]
      \draw (-4,-5) to (0,0) to (4,-5);
      \draw (0,5) to (0,0);
      \draw[green] (-4,5) to (0,0) to (4,5);
      \draw[green] (0,-5) to (0,0);
      \node at (4,6) {\tiny $j$};
     \node at (0,6) {\tiny $i$};
    \end{tikzpicture}
  \end{array}
& \mapsto & \;\; - 
  \begin{array}{c}
    \begin{tikzpicture}[thick,scale=-0.15]
      \draw[<-] (-5,-5) to[out=90,in=180] (0,1.5) to[out=0,in=90] (5,-5);
      \draw[->] (-3,-5) to[out=90,in=-90] (1,1) to (1,5);
      \draw[<-] (3,-5) to[out=90,in=-90] (-1,1) to (-1,5);
\draw[<-,green] (-1,-5) to[out=90,in=-90] (-5,5);
\draw[->,green] (1,-5) to[out=90,in=-90] (5,5);
\draw[->,green] (3,5) to[out=-90,in=0] (0,2) to[out=180,in=-90] (-3,5);
\node (a) at (-6,1) {\tiny $\omega$};
      \node at (3,6) {\tiny $j$};
     \node at (-1,6) {\tiny $i$};
     \node at (-5,6) {\tiny $j$};
    \end{tikzpicture}  \end{array}.
\end{array} 
\end{align*}
(Here and below, all the $2$-morphisms in $\mathcal{U}(\hgl_n)$ we consider are
as defined in~\cite{Brundan}, and we use the same notation.)

The main result of this section, which finishes the proof of Conjecture~\ref{conj:main} in this case is the following.

\begin{thm}
\label{thm:2KM-SBim}
The above assignment defines a strict monoidal functor
\[
\sigma \colon \DiagBS
  \to \End_{\mathcal{U}^{[n]}(\hgl_n)}(\omega).
\]
\end{thm}

\begin{rmk}
The constructions in this section are inspired by a similar construction performed in~\cite{mt1, mt2} (following earlier work on the analogous finite case in~\cite{msv}). However these authors use a different choice of constants $t_{ij}$; therefore we provide a complete proof (which is slightly different from the proof in~\cite{mt1, mt2}).
\end{rmk}

To prove Theorem~\ref{thm:2KM-SBim} the strategy is clear: we have a list of
relations defining $\DiagBS$ as described in~\cite{ew}, and we need to verify that
these relations are satisfied in $\End_{\mathcal{U}^{[n]}(\hgl_n)}(\omega)$ by the images of these morphisms.

\subsection{Preliminary lemmas}

For future reference, we first record the ``sign rules'' from~\cite[(2.2), (2.9), (5.7), (5.8)]{Brundan}

\begin{lem}
\label{signrules}
For $i, j \in \{1, \dots, n-1, \infty
  \}$ we have the following sign rules in $\mathcal{U}(\hgl_n)$:
   \begin{gather}
    \label{signs:1}
    \begin{array}{c}
      \begin{tikzpicture}[thick,scale=0.4]
        \draw[->] (-1,0) to[out=90,in=180] (0,1) to[out=0,in=90] (1,0);
        \draw[green,->] (0,0) to (1.3,1.3);
        \node at (-1,-.3) {\tiny $i$};
        \node at (0,-.3) {\tiny $j$};
        \node at  (1.5,.7) {\tiny $\lambda$};
      \end{tikzpicture}
    \end{array}
=
    \begin{array}{c}
      \begin{tikzpicture}[thick,scale=0.4]
        \draw[->] (-1,0) to[out=90,in=180] (0,1) to[out=0,in=90] (1,0);
        \draw[green,->] (0,0) to (-1.3,1.3);
        \node at (-1,-.3) {\tiny $i$};
        \node at (0,-.3) {\tiny $j$};
        \node at  (1.5,.7) {\tiny $\lambda$};
      \end{tikzpicture}
    \end{array},
\quad \quad
    \begin{array}{c}
      \begin{tikzpicture}[thick,scale=-0.4]
        \draw[<-] (-1,0) to[out=90,in=180] (0,1) to[out=0,in=90] (1,0);
        \draw[green,<-] (0,0) to  (1.3,1.3);
        \node at  (-1.5,.5) {\tiny $\lambda$};
        \node at (1.3,1.7) {\tiny $j$};
        \node at (1,-.4) {\tiny $i$};
      \end{tikzpicture}
    \end{array}
=
    \begin{array}{c}
      \begin{tikzpicture}[thick,scale=-0.4]
        \draw[<-] (-1,0) to[out=90,in=180] (0,1) to[out=0,in=90] (1,0);
        \draw[green,<-] (0,0) to  (-1.3,1.3);
        \node at  (-1.5,.5) {\tiny $\lambda$};
        \node at (-1.3,1.7) {\tiny $j$};
        \node at (1,-.4) {\tiny $i$};
      \end{tikzpicture} 
     \end{array}, \\
    \label{signs:2}
    \begin{array}{c}
      \begin{tikzpicture}[thick,scale=-0.4]
        \draw[<-] (-1,0) to[out=90,in=180] (0,1) to[out=0,in=90] (1,0);
        \draw[green,->] (0,0) to (-1.3,1.3);
        \node at  (-1.5,.5) {\tiny $\lambda$};
        \node at (0,-.4) {\tiny $j$};
        \node at (1,-.4) {\tiny $i$};
      \end{tikzpicture}
    \end{array}
=
t_{ji}
    \begin{array}{c}
      \begin{tikzpicture}[thick,scale=-0.4]
        \draw[<-] (-1,0) to[out=90,in=180] (0,1) to[out=0,in=90] (1,0);
        \draw[green,->] (0,0) to (1.3,1.3);
        \node at  (-1.5,.5) {\tiny $\lambda$};
        \node at (0,-.4) {\tiny $j$};
        \node at (1,-.4) {\tiny $i$};
      \end{tikzpicture}
    \end{array},
\quad \quad
    \begin{array}{c}
      \begin{tikzpicture}[thick,scale=0.4]
        \draw[->] (-1,0) to[out=90,in=180] (0,1) to[out=0,in=90] (1,0);
        \draw[green,<-] (0,0) to  (-1.3,1.3);
        \node at (-1,-.4) {\tiny $i$};
        \node at (-1.3,1.7) {\tiny $j$};
        \node at (1.5,.5) {\tiny $\lambda$};
      \end{tikzpicture}
    \end{array}
=
    t_{ji} \begin{array}{c}
      \begin{tikzpicture}[thick,scale=0.4]
        \draw[->] (-1,0) to[out=90,in=180] (0,1) to[out=0,in=90] (1,0);
        \draw[green,<-] (0,0) to  (1.3,1.3);
        \node at (-1,-.4) {\tiny $i$};
        \node at (1.3,1.7) {\tiny $j$};
        \node at (1.5,.5) {\tiny $\lambda$};
      \end{tikzpicture} 
     \end{array},\\
    \label{signs:3}
    \begin{array}{c}
      \begin{tikzpicture}[thick,scale=0.4]
        \draw[<-] (-1,0) to[out=90,in=180] (0,1) to[out=0,in=90] (1,0);
        \draw[green,->] (0,0) to (1.3,1.3);
        \node at (1,-.4) {\tiny $i$};
        \node at (0,-.4) {\tiny $j$};
        \node at (1.5,.5) {\tiny $\lambda$};
      \end{tikzpicture}
    \end{array}
=
    t_{ij} \begin{array}{c}
      \begin{tikzpicture}[thick,scale=0.4]
        \draw[<-] (-1,0) to[out=90,in=180] (0,1) to[out=0,in=90] (1,0);
        \draw[green,->] (0,0) to (-1.3,1.3);
        \node at (1,-.4) {\tiny $i$};
        \node at (0,-.4) {\tiny $j$};
        \node at (1.5,.5) {\tiny $\lambda$};
      \end{tikzpicture}
    \end{array},
\quad \quad
    \begin{array}{c}
      \begin{tikzpicture}[thick,scale=-0.4]
        \draw[->] (-1,0) to[out=90,in=180] (0,1) to[out=0,in=90] (1,0);
        \draw[green,<-] (0,0) to  (1.3,1.3);
        \node at (-1,-.4) {\tiny $i$};
        \node at (1.3,1.7) {\tiny $j$};
        \node at (-1.5,.5) {\tiny $\lambda$};
      \end{tikzpicture}
    \end{array}
=
    t_{ij}\begin{array}{c}
      \begin{tikzpicture}[thick,scale=-0.4]
        \draw[->] (-1,0) to[out=90,in=180] (0,1) to[out=0,in=90] (1,0);
        \draw[green,<-] (0,0) to  (-1.3,1.3);
        \node at (-1,-.4) {\tiny $i$};
                \node at (-1.5,.5) {\tiny $\lambda$};
                \node at (-1.3,1.7) {\tiny $j$};
      \end{tikzpicture} 
     \end{array},\\
    \label{signs:4}
    \begin{array}{c}
      \begin{tikzpicture}[thick,scale=-0.4]
        \draw[->] (-1,0) to[out=90,in=180] (0,1) to[out=0,in=90] (1,0);
        \draw[green,->] (0,0) to (-1.3,1.3);
        \node at (-1,-.4) {\tiny $i$};
        \node at (0,-.4) {\tiny $j$};
        \node at (-1.5,.5) {\tiny $\lambda$};
      \end{tikzpicture}
    \end{array}
=
    \begin{array}{c}
      \begin{tikzpicture}[thick,scale=-0.4]
        \draw[->] (-1,0) to[out=90,in=180] (0,1) to[out=0,in=90] (1,0);
        \draw[green,->] (0,0) to (1.3,1.3);
        \node at (-1,-.4) {\tiny $i$};
        \node at (0,-.4) {\tiny $j$};
        \node at (-1.5,.5) {\tiny $\lambda$};
      \end{tikzpicture}
    \end{array},
\quad \quad
    \begin{array}{c}
      \begin{tikzpicture}[thick,scale=0.4]
        \draw[<-] (-1,0) to[out=90,in=180] (0,1) to[out=0,in=90] (1,0);
        \draw[green,<-] (0,0) to  (-1.3,1.3);
        \node at (1,-.4) {\tiny $i$};
        \node at (1.5,.5) {\tiny $\lambda$};
        \node at (-1.4,1.7) {\tiny $j$};
      \end{tikzpicture}
    \end{array}
=
    \begin{array}{c}
      \begin{tikzpicture}[thick,scale=0.4]
        \draw[<-] (-1,0) to[out=90,in=180] (0,1) to[out=0,in=90] (1,0);
        \draw[green,<-] (0,0) to  (1.3,1.3);
        \node at (1,-.4) {\tiny $i$};
        \node at (1.5,.5) {\tiny $\lambda$};
        \node at (1.3,1.7) {\tiny $j$};
      \end{tikzpicture} \\
     \end{array}.
\end{gather}
\end{lem}

We also recall that we do not have to distinguish on which side of a cup or cap a dot is put; see~\cite[(2.1), (5.6)]{Brundan}; we will therefore sometimes write this dot in the middle.

The following relations (in $\mathcal{U}(\hgl_n)$) for $i, j \in \{1, \cdots, n-1, \infty\}$ with $i \ne j$ are equivalent to the relation in~\cite[(1.19)]{Brundan}:
\begin{equation}
  \label{pullapart}
  \begin{array}{c}
    \begin{tikzpicture}[thick,scale=0.3]
      \draw[black,->] (-1,-2) to[out=90,in=-90] (1,0) to[out=90,in=-90] (-1,2);
      \draw[green,<-] (1,-2) to[out=90,in=-90] (-1,0) to[out=90,in=-90] (1,2);
      \node at (-1,-2.5) {\tiny $i$};      \node at (1,2.5) {\tiny $j$};
      \node at (1.7,0) {\tiny $\lambda$};
    \end{tikzpicture}
\end{array}
=
  \begin{array}{c}
    \begin{tikzpicture}[thick,scale=0.3]
      \draw[black,->] (-.5,-2) to (-.5,2);
      \draw[green,<-] (1,-2) to (1,2);
      \node at (-.5,-2.5) {\tiny $i$};      \node at (1,2.5) {\tiny $j$};
            \node at (1.7,0) {\tiny $\lambda$};
    \end{tikzpicture}
\end{array},
\qquad
  \begin{array}{c}
    \begin{tikzpicture}[thick,scale=0.3]
      \draw[black,<-] (-1,-2) to[out=90,in=-90] (1,0) to[out=90,in=-90] (-1,2);
      \draw[green,->] (1,-2) to[out=90,in=-90] (-1,0) to[out=90,in=-90] (1,2);
      \node at (-1,2.5) {\tiny $i$};      \node at (1,-2.5) {\tiny $j$};
      \node at (1.7,0) {\tiny $\lambda$};
    \end{tikzpicture}
\end{array}
=
  \begin{array}{c}
    \begin{tikzpicture}[thick,scale=0.3]
      \draw[black,<-] (-.5,-2) to (-.5,2);
      \draw[green,->] (1,-2) to (1,2);
      \node at (-.5,2.5) {\tiny $i$};      \node at (1,-2.5) {\tiny $j$};
            \node at (1.7,0) {\tiny $\lambda$};
    \end{tikzpicture}
\end{array}
\end{equation}

The following relation (again in $\mathcal{U}(\hgl_n)$), for $i,j \in \{1, \cdots, n-1, \infty\}$ with $i \rel j$, is a special case of~\cite[Corollary~3.4]{Brundan} (taking into account~\cite[(3.9)]{Brundan}):
\begin{equation}
  \label{ipullapart}
  \begin{array}{c}
    \begin{tikzpicture}[thick,scale=0.3]
      \draw[black,->] (-1,-2) to[out=90,in=-90] (1,0) to[out=90,in=-90] (-1,2);
      \draw[<-] (1,-2) to[out=90,in=-90] (-1,0) to[out=90,in=-90] (1,2);
      \node at (-1,-2.5) {\tiny $i$};      
      \node at (1,2.5) {\tiny $i$};
      \node[right] at (1.5,0) {\tiny $\omega + \a_i + \a_j$};
    \end{tikzpicture}
\end{array}
=
  \begin{array}{c}
    \begin{tikzpicture}[thick,scale=0.3]
      \draw[black,->] (-1,-2) to[out=90,in=180] (0,-.5) to[out=0,in=90] (1,-2);
      \draw[->] (1,2) to[out=-90,in=0] (0,.5) to[out=180,in=-90] (-1,2);
      \node at (-1,-2.5) {\tiny $i$};      
      \node at (1,2.5) {\tiny $i$};
            \node[right] at (1.5,0) {\tiny $\omega + \a_i + \a_j$};
    \end{tikzpicture}
\end{array}
-
  \begin{array}{c}
    \begin{tikzpicture}[thick,scale=0.3]
      \draw[black,->] (0,-2) to  (0,2);
      \draw[<-] (1,-2) to  (1,2);
      \node at (0,-2.5) {\tiny $i$};      
      \node at (1,2.5) {\tiny $i$};
            \node[right] at (1.5,0) {\tiny $\omega + \a_i + \a_j$};
    \end{tikzpicture}
\end{array}
\end{equation}

\begin{lem}
\label{nilhecke3}
For $i, j,k \in \{1, \cdots, n-1, \infty\}$, in $\mathcal{U}^{[n]}(\hgl_n)$ we
  have:
\begin{equation}
  \label{eq:nilhecke3}
  \begin{array}{c}
    \begin{tikzpicture}[thick,yscale=0.3,xscale=0.25]
      \draw[black,->] (-2,-2) to (2,2);
      \draw[green,->] (2,-2) to (-2,2);
      \draw[red,->] (0,-2) to[out=90,in=-90] (2,0) to[out=90,in=-90] (0,2);
      \node at (-2,-2.5) {\tiny $i$}; \node at (0,-2.5) {\tiny $j$}; \node at (2,-2.5) {\tiny $k$};
      \node at (3,0) {\tiny $\omega$};
    \end{tikzpicture}
\end{array}
=
\begin{cases}
    -t_{ij} \begin{array}{c}
    \begin{tikzpicture}[thick,yscale=0.3,xscale=0.25]
      \draw[black,->] (-2,-2) to (-2,2);
      \draw[green,->] (2,-2) to (2,2);
      \draw[red,->] (0,-2) to (0,2);
      \node at (-2,-2.5) {\tiny $i$}; \node at (0,-2.5) {\tiny $j$}; \node at (2,-2.5) {\tiny $k$};
            \node at (3,0) {\tiny $\omega$};
    \end{tikzpicture}
\end{array}
& \text{if $i = k$ and $k \rel j$,} \\
0 & \text{otherwise.}
\end{cases}
\end{equation}
\end{lem}

\begin{proof}
This follows from~\eqref{eqn:rel-2KM-3}, using the fact that the first term in the left-hand side in this equation vanishes in $\mathcal{U}^{[n]}(\hgl_n)$ since it contains a region labelled by $\omega + 2\a_i$.
\end{proof}

\begin{lem}
\label{lem:blue2} 
For any $i \in \{1, \cdots, n-1, \infty\}$, in $\mathcal{U}^{[n]}(\hgl_n)$ we have:
  \begin{equation*}
    \begin{array}{c}
            \begin{tikzpicture}[thick,scale=0.3]
      \draw[black,->] (-.5,-2) to (-.5,2);
      \draw[black,<-] (.5,-2) to (.5,2);
      \node at (2.5,0) {\tiny $\omega + \a_i$};
      \node at (-.5,-2.5) {\tiny $i$};
      \node at (.5,2.5) {\tiny $i$};
      \end{tikzpicture}
    \end{array}
\; = \;
    \begin{array}{c}
            \begin{tikzpicture}[thick,scale=0.3]
      \draw[black,->] (-1,-2) to[out=90,in=-180] (0,-1) to[out=0,in=90] (1,-2);
      \draw[black,<-] (-1,2) to[out=-90,in=-180] (0,1) to[out=0,in=-90] (1,2);
            \draw[black] (0,0) circle (.5);
      \draw[<-] (.1,.5) to (-.1,.5);
            \node[black] at (-0.5,0) {$\bullet$};
            \node at (-1.4,0) {\tiny $2$};
            \node at (3,1) {\tiny $\omega + \a_i$};
                  \node at (1,2.5) {\tiny $i$};
      \node at (-1,-2.5) {\tiny $i$};
      \end{tikzpicture}
    \end{array}
+
    \begin{array}{c}
            \begin{tikzpicture}[thick,scale=0.3]
      \draw[black,->] (-1,-2) to[out=90,in=-180] (0,-0.5) to[out=0,in=90] (1,-2);
      \draw[black,<-] (-1,2) to[out=-90,in=-180] (0,0.5) to[out=0,in=-90] (1,2);
      \node[black] at (0,0.5) {$\bullet$};
      \node at (2.5,0) {\tiny $\omega + \a_i$};
      \node at (1,2.5) {\tiny $i$};
      \node at (-1,-2.5) {\tiny $i$};
      \end{tikzpicture}
    \end{array}
\;+\;
    \begin{array}{c}
            \begin{tikzpicture}[thick,scale=0.3]
      \draw[black,->] (-1,-2) to[out=90,in=-180] (0,-0.5) to[out=0,in=90] (1,-2);
      \draw[black,<-] (-1,2) to[out=-90,in=-180] (0,0.5)
      to[out=0,in=-90] (1,2);
      \node[black] at (0,-0.5) {$\bullet$};
      \node at (2.5,0) {\tiny $\omega + \a_i$};
      \node at (1,2.5) {\tiny $i$};
      \node at (-1,-2.5) {\tiny $i$};
      \end{tikzpicture}
    \end{array}
  \end{equation*}
\end{lem}

\begin{proof}
By~\cite[Corollary~3.4]{Brundan}, in $\mathcal{U}(\hgl_n)$ we have
  \begin{equation*}
    \begin{array}{c}
 \begin{tikzpicture}[thick,baseline,scale=0.4]
\draw[->] (1,0) to [out=90,in=-90] (2,1.5) to [out=90,in=-90] (1,3);
\draw[<-] (2,0) to [out=90,in=-90] (1,1.5) to [out=90,in=-90] (2,3);
\node at (3.2,1.5) {\tiny $\omega + \a_i$};
  \end{tikzpicture}\end{array}
=
    \begin{array}{c}
            \begin{tikzpicture}[thick,scale=0.3]
      \draw[black,->] (-1,-2) to[out=90,in=-180] (0,-1) to[out=0,in=90] (1,-2);
      \draw[black,<-] (-1,2) to[out=-90,in=-180] (0,1) to[out=0,in=-90] (1,2);
            \draw[black] (0,0) circle (.5);
      \draw[->] (.1,.5) to (-.1,.5);
            \node[black] at (0.5,0) {$\bullet$};
            \node at (1.4,0) {\tiny $-2$};
            \node at (3,1) {\tiny $\omega + \a_i$};
      \end{tikzpicture}
    \end{array}
\;+\;
    \begin{array}{c}
            \begin{tikzpicture}[thick,scale=0.3]
      \draw[black,->] (-1,-2) to[out=90,in=-180] (0,-0.5) to[out=0,in=90] (1,-2);
      \draw[black,<-] (-1,2) to[out=-90,in=-180] (0,0.5)
      to[out=0,in=-90] (1,2);
      \node[black] at (0,0.5) {$\bullet$};
      \node at (3.2,0) {\tiny $\omega + \a_i$};
      \end{tikzpicture}
    \end{array}
\;+\;
    \begin{array}{c}
            \begin{tikzpicture}[thick,scale=0.3]
      \draw[black,->] (-1,-2) to[out=90,in=-180] (0,-0.5) to[out=0,in=90] (1,-2);
      \draw[black,<-] (-1,2) to[out=-90,in=-180] (0,0.5)
      to[out=0,in=-90] (1,2);
      \node[black] at (0,-0.5) {$\bullet$};
      \node at (3.2,0) {\tiny $\omega + \a_i$};
      \end{tikzpicture}
    \end{array}
-
    \begin{array}{c}
            \begin{tikzpicture}[thick,scale=0.3]
      \draw[black,->] (-.5,-2) to (-.5,2);
      \draw[black,<-] (.5,-2) to (.5,2);
      \node at (2.5,0) {\tiny $\omega + \a_i$};
      \end{tikzpicture}
    \end{array}
  \end{equation*}
(where all strands have color $i$).
Then we notice that the left-hand side has a region labelled by $\omega +2\alpha_i$, hence it vanishes in $\mathcal{U}^{[n]}(\hgl_n)$. Using the definition of negative dots in~\cite[(3.9), (3.3)]{Brundan}, what remains is precisely the relation in the lemma.
\end{proof}

\begin{lem}
Let $i, j \in \{1, \cdots, n-1, \infty\}$ with $i \to j$. Then in $\mathcal{U}^{[n]}(\hgl_n)$ we have:
  \begin{equation}
  \label{splits}
  \begin{array}{c}
    \begin{tikzpicture}[thick,scale=0.3]
      \draw[black,<-] (0,-2) to (0,2);      \draw[black,->] (1,-2) to (1,2);
      \node at (0,-2.5) {\tiny $i$};      \node at (1,-2.5) {\tiny $i$};
      \node[right] at (1.5,0) {\tiny $\omega + \a_j$};
    \end{tikzpicture}
\end{array} = 
  \begin{array}{c}
    \begin{tikzpicture}[thick,scale=0.3]
      \draw[black,<-] (-1,-2) to[out=90,in=180] (0,-.3) to[out=0,in=90] (1,-2);
     \draw[black,->] (-1,2) to[out=-90,in=180] (0,.3) to[out=0,in=-90] (1,2); 
     \node at (-1,-2.5) {\tiny $i$};      \node at (1,-2.5) {\tiny $i$};
      \node[right] at (1.5,0) {\tiny $\omega + \a_j$};
    \end{tikzpicture}
\end{array} .
\end{equation}
\end{lem}

\begin{proof}
Using~\cite[(3.17), (3.8)]{Brundan} we have in $\mathcal{U}(\hgl_n)$:
\[
 \begin{array}{c} \begin{tikzpicture}[thick,scale=0.3]
      \draw[black,<-] (-1,-2) to[out=90,in=-90] (1,0) to[out=90,in=-90] (-1,2);
      \draw[->] (1,-2) to[out=90,in=-90] (-1,0) to[out=90,in=-90] (1,2);
      \node at (-1,-2.5) {\tiny $i$};      \node at (1,-2.5) {\tiny $i$};
      \node[right] at (1.5,0) {\tiny $\omega+\alpha_j$};
    \end{tikzpicture}\end{array}
=   \begin{array}{c}
    \begin{tikzpicture}[thick,scale=0.3]
      \draw[black,<-] (-1,-2) to[out=90,in=180] (0,-.3) to[out=0,in=90] (1,-2);
     \draw[black,->] (-1,2) to[out=-90,in=180] (0,.3) to[out=0,in=-90] (1,2); 
     \node at (-1,-2.5) {\tiny $i$};      \node at (1,-2.5) {\tiny $i$};
      \node[right] at (1.5,0) {\tiny $\omega + \alpha_j$};
    \end{tikzpicture}
\end{array} -
  \begin{array}{c}
    \begin{tikzpicture}[thick,scale=0.3]
      \draw[black,<-] (0,-2) to (0,2);      \draw[black,->] (1,-2) to (1,2);
      \node at (0,-2.5) {\tiny $i$};      \node at (1,-2.5) {\tiny $i$};
      \node[right] at (1.5,0) {\tiny $\omega + \alpha_j$};
    \end{tikzpicture}
\end{array} .
\]
Then
we observe that the left-hand side vanishes in $\mathcal{U}^{[n]}(\hgl_n)$ since it contains a region labelled by $\omega+\alpha_j-\alpha_i$.
\end{proof}

\subsection{Polynomials}

As explained in Remark~\ref{rmk:morphisms-Dasph}, the elements of $R$, considered as morphisms in $\DiagBS$, can be expressed in terms of dot morphisms. The fact that this indeed defines an action of $R$ follows from the following lemma.

\begin{lem}
For any $i,j \in \{1, \cdots, n-1, \infty\}$, we have
\[
\sigma \left(
\begin{array}{c}
            \begin{tikzpicture}[thick,scale=0.3]
\draw (-1,-1) to (-1,1);
\node at (-1,-1) {$\bullet$};
\node at (-1,1) {$\bullet$};
\draw[green] (0.5,-1) to (0.5,1);
\node[green] at (0.5,-1) {$\bullet$};
\node[green] at (0.5,1) {$\bullet$};
\node at (-1,-1.7) {\tiny $i$};
\node at (0.5,-1.7) {\tiny $j$};
      \end{tikzpicture}
    \end{array}
\right)
=
\sigma \left(
\begin{array}{c}
            \begin{tikzpicture}[thick,scale=0.3]
\draw[green] (-1,-1) to (-1,1);
\node[green] at (-1,-1) {$\bullet$};
\node[green] at (-1,1) {$\bullet$};
\draw (0.5,-1) to (0.5,1);
\node at (0.5,-1) {$\bullet$};
\node at (0.5,1) {$\bullet$};
\node at (-1,-1.7) {\tiny $j$};
\node at (0.5,-1.7) {\tiny $i$};
      \end{tikzpicture}
    \end{array}
\right)
\]
\end{lem}

\begin{proof}
Of course, there is noting to prove if $i=j$. Hence from now on we assume that $i \neq j$. We observe that, by~\cite[(1.19)]{Brundan} and then Lemma~\ref{signrules}, in $\mathcal{U}(\hgl_n)$ we have
\[
\begin{array}{c}
            \begin{tikzpicture}[thick,scale=0.3]
      \draw[black] (-1.5,0) circle (1);
      \draw[<-] (-.5,.1) to (-.5,-.1);
      \node at (-1.5,-1.5) {\tiny $i$};
      \draw[green] (1.5,0) circle (1);
      \draw[green,<-] (2.5,.1) to (2.5,-.1);
      \node at (3,.5) {\tiny $\omega$};
      \node at (1.5,-1.5) {\tiny $j$};
      \end{tikzpicture}
\end{array}
=
\begin{array}{c}
            \begin{tikzpicture}[thick,scale=0.3]
      \draw[black] (-1.5,0) circle (1);
      \draw[<-] (-.5,.1) to (-.5,-.1);
      \node at (-1.5,-1.5) {\tiny $i$};
      \draw[green] (-.3,0) circle (1);
      \draw[green,<-] (.7,.1) to (.7,-.1);
      \node at (1.2,.5) {\tiny $\omega$};
      \node at (-.3,-1.5) {\tiny $j$};
      \end{tikzpicture}
\end{array}
= t_{j,i}
\begin{array}{c}
            \begin{tikzpicture}[thick,scale=0.3]
      \draw[black] (-1.5,0) circle (1);
      \draw[<-] (-.5,.1) to (-.5,-.1);
      \node at (-1.5,-1.5) {\tiny $i$};
      \draw[green] (.7,0) to (.7,.5) to[out=90,in=90] (-2.5,1.5) to[out=-90,in=90] (-1.5,0) to[out=-90,in=90] (-2.5,-1.5) to[out=-90,in=-90] (.7,-.5) to (.7,0);
      \draw[green,<-] (.7,.1) to (.7,-.1);
      \node at (1.2,.5) {\tiny $\omega$};
      \node at (-.3,-2.2) {\tiny $j$};
      \end{tikzpicture}
\end{array}.
\]
By~\cite[(2.7)]{Brundan}, this last expression is equal to
\begin{equation}
\label{eqn:sigma-polynomials-1}
\begin{array}{c}
            \begin{tikzpicture}[thick,scale=0.3]
      \draw[black] (-1.5,0) circle (1);
            \draw[green] (-1.5,0) circle (2);
                  \draw[<-] (-.5,.1) to (-.5,-.1);
                  \draw[green,<-] (.5,.1) to (.5,-.1);
                  \node at (-1.5,-1.4) {\tiny $i$};
                   \node at (-1.5,-2.4) {\tiny $j$};
                         \node at (1.2,.5) {\tiny $\omega$};
            \end{tikzpicture}
\end{array}
\end{equation}
if $i \norel j$, and to
\begin{equation}
\label{eqn:sigma-polynomials-2}
-
\begin{array}{c}
            \begin{tikzpicture}[thick,scale=0.3]
      \draw[black] (-1.5,0) circle (1);
            \draw[green] (-1.5,0) circle (2);
                  \draw[<-] (-.5,.1) to (-.5,-.1);
                  \draw[green,<-] (.5,.1) to (.5,-.1);
                  \node[green] at (-1.5,2) {$\bullet$};
                  \node at (-1.5,-1.4) {\tiny $i$};
                   \node at (-1.5,-2.4) {\tiny $j$};
                         \node at (1.2,.5) {\tiny $\omega$};
            \end{tikzpicture}
\end{array}
+
\begin{array}{c}
            \begin{tikzpicture}[thick,scale=0.3]
      \draw[black] (-1.5,0) circle (1);
            \draw[green] (-1.5,0) circle (2);
                  \draw[<-] (-.5,.1) to (-.5,-.1);
                  \draw[green,<-] (.5,.1) to (.5,-.1);
                  \node at (-1.5,1) {$\bullet$};
                  \node at (-1.5,-1.4) {\tiny $i$};
                   \node at (-1.5,-2.4) {\tiny $j$};
                         \node at (1.2,.5) {\tiny $\omega$};
            \end{tikzpicture}
\end{array}
\end{equation}
if $i \rel j$.

On the other hand, for similar reasons we have
\[
\begin{array}{c}
            \begin{tikzpicture}[thick,xscale=-0.3,yscale=0.3]
      \draw[black] (-1.5,0) circle (1);
      \draw[<-] (-2.5,.1) to (-2.5,-.1);
      \node at (-1.5,-1.5) {\tiny $i$};
      \draw[green] (1.5,0) circle (1);
      \draw[green,<-] (.5,.1) to (.5,-.1);
      \node at (-3,.5) {\tiny $\omega$};
      \node at (1.5,-1.5) {\tiny $j$};
      \end{tikzpicture}
\end{array}
=
\begin{array}{c}
            \begin{tikzpicture}[thick,xscale=-0.3,yscale=0.3]
      \draw[black] (-1.5,0) circle (1);
      \draw[<-] (-2.5,.1) to (-2.5,-.1);
      \node at (-1.5,-1.5) {\tiny $i$};
      \draw[green] (-.3,0) circle (1);
      \draw[green,<-] (-1.3,.1) to (-1.3,-.1);
      \node at (-3.2,.5) {\tiny $\omega$};
      \node at (-.3,-1.5) {\tiny $j$};
      \end{tikzpicture}
\end{array}
= t_{i,j}
\begin{array}{c}
            \begin{tikzpicture}[thick,xscale=-0.3,yscale=0.3]
      \draw[black] (-1.5,0) circle (1);
      \draw[<-] (-2.5,.1) to (-2.5,-.1);
      \node at (-1.5,-1.5) {\tiny $i$};
      \draw[green] (.7,0) to (.7,.5) to[out=90,in=90] (-2.5,1.5) to[out=-90,in=90] (-1.5,0) to[out=-90,in=90] (-2.5,-1.5) to[out=-90,in=-90] (.7,-.5) to (.7,0);
      \draw[green,<-] (-1.5,.1) to (-1.5,-.1);
      \node at (-3.2,.5) {\tiny $\omega$};
      \node at (-.3,-2.2) {\tiny $j$};
      \end{tikzpicture}
\end{array}.
\]
Using~\eqref{eqn:rel-2KM-2}, this last expression is equal to~\eqref{eqn:sigma-polynomials-1}
if $i \norel j$, and to~\eqref{eqn:sigma-polynomials-2}
if $i \rel j$, which completes the proof.
\end{proof}

\subsection{One color relations}
\label{ss:one-color}

The next relations we must consider are the one color relations. We remark that the Frobenius relations are immediate from the zigzag relations in $\mathcal{U}(\hgl_n)$ (see~\cite[(1.5) and Theorem~4.3]{Brundan}), or in other words from the fact that the dot and trivalent morphisms are induced by some adjunctions, see Remark~\ref{rmk:main-conj}\eqref{it:rmks-conj}. Next, the needle relation (i.e.~the relation described in~\cite[(5.5)]{ew}) follows from the relation
\[
    \begin{array}{c}
            \begin{tikzpicture}[thick,scale=0.3]
      \draw[black] (0,0) circle (1);
      \draw[<-] (1,-.1) to (1,.1);
      \node at (2.8,.5) {\tiny $\omega + \a_i$};
      \node at (-1.5,0) {\tiny $i$};
      \end{tikzpicture}
    \end{array}
\; = \; 0
\]
in $\mathcal{U}(\hgl_n)$, which is an application of~\cite[(3.6)]{Brundan}
since $\langle \omega + \a_i, h_i \rangle = 2$.

The polynomial sliding relation (see~\cite[(5.2)]{ew}) follows from Lemmas~\ref{lem:polynomial-1},~\ref{lem:polynomial-2} and~\ref{lem:polynomial-3} below, see Remark~\ref{rmk:morphisms-Dasph}.

\begin{lem}
\label{lem:polynomial-1}
  For any $i \in \{1, \cdots, n-1, \infty \}$ we have
\begin{equation*}
\sigma \left(
  \begin{array}{c}
    \begin{tikzpicture}[thick,scale=0.3]
\draw (0,-2) to (0,2);
\draw (-1,-1) to (-1,1);
\node at (-1,-1) {$\bullet$};
\node at (-1,1) {$\bullet$};
    \end{tikzpicture}
  \end{array}
+
  \begin{array}{c}
    \begin{tikzpicture}[thick,scale=0.3]
\draw (0,-2) to (0,2);
\draw (1,-1) to (1,1);
\node at (1,1) {$\bullet$};
\node at (1,-1) {$\bullet$};
    \end{tikzpicture}
  \end{array}
    \right)
=
  2 \;  \sigma \left(  \begin{array}{c}\begin{tikzpicture}[thick,scale=0.3]
\draw (0,-2) to (0,-.5);
\draw (0,2) to (0,.5);
\node at (0,-.5) {$\bullet$};
\node at (0,.5) {$\bullet$};
    \end{tikzpicture}
  \end{array}
  \right)
\end{equation*}
(where all strands have label $i$).
\end{lem}

\begin{proof}
By Lemma~\ref{lem:blue2}, in $\mathcal{U}^{[n]}(\hgl_n)$ we have
  \begin{equation}
\label{eqn:alpha}
    \begin{array}{c}
      \begin{tikzpicture}[thick,scale=0.3]
        \draw[<-] (0,-2) to (0,2); \draw[->] (1,-2) to (1,2);
        \draw[->] (2.5,0) circle (0.7); \draw[->] (3.2,-.1) to (3.2,+.1); 
        \node at (2.5,1.5) {\tiny $\omega$};
      \end{tikzpicture}
    \end{array}
    =
    \begin{array}{c}
      \begin{tikzpicture}[thick,scale=0.3]
        \draw[<-] (-0.6,-2) to (-0.6,2);
        \draw[->] (1,-2) to[out=90,in=180]
        (1.5,-.8) to[out=0,in=180] (2.4,-1.5) to[out=0,in=-90] (3,0)
        to[out=90,in=0] (2.4,1.5) to[out=-180,in=0] (1.5,.8)
        to[out=-180,in=-90] (1,2);
        \draw[black] (1,0) circle (0.5);
        \node (a) at (.5,0) {$\bullet$};
        \draw[->] (.9,.5) to (1.1,.5);
        \node at (2.5,2) {\tiny $\omega$};
        \node at (-0.1,0) {\tiny $2$};
      \end{tikzpicture}
    \end{array}
    \;+\;
        \begin{array}{c}
      \begin{tikzpicture}[thick,scale=0.3]
        \draw[<-] (0,-2) to (0,2); \draw[->] (1,-2) to[out=90,in=180]
        (1.5,-.3) to[out=0,in=180] (2.4,-1) to[out=0,in=-90] (3,0)
        to[out=90,in=0] (2.4,1) to[out=-180,in=0] (1.5,.3)
        to[out=-180,in=-90] (1,2); \node (a) at (1.5,-.3) {$\bullet$};
        \node at (2.5,1.5) {\tiny $\omega$};
      \end{tikzpicture}
    \end{array}
    \;+\;
    \begin{array}{c}
      \begin{tikzpicture}[thick,scale=0.3]
        \draw[<-] (0,-2) to (0,2); \draw[->] (1,-2) to[out=90,in=180]
        (1.5,-.3) to[out=0,in=180] (2.4,-1) to[out=0,in=-90] (3,0)
        to[out=90,in=0] (2.4,1) to[out=-180,in=0] (1.5,.3)
        to[out=-180,in=-90] (1,2); \node (a) at (1.5,.3) {$\bullet$};
        \node at (2.5,1.5) {\tiny $\omega$};
      \end{tikzpicture}
    \end{array}
    = 
    \begin{array}{c}
                \begin{tikzpicture}[thick,scale=0.3]
                  \draw[<-] (-0.6,-2) to (-0.6,2);
                          \draw[black] (1,0) circle (0.5); 
                  \draw[->] (2,-2) to (2,2);
                          \node (a) at (.5,0) {$\bullet$};
                  \node at (2.7,0.5) {\tiny $\omega$};
                          \node at (-0.1,0) {\tiny $2$};
                                  \draw[->] (.9,.5) to (1.1,.5);
                \end{tikzpicture}
              \end{array}
    \;+\;
    \; 2 \; \begin{array}{c}
                \begin{tikzpicture}[thick,scale=0.3]
                  \draw[<-] (0,-2) to (0,2); \draw[->] (1,-2) to
                  (1,2); \node at (1,0) {$\bullet$};
                  \node at (2,0.5) {\tiny $\omega$};
                \end{tikzpicture}
              \end{array}.
            \end{equation}
 Similarly, in $\mathcal{U}^{[n]}(\hgl_n)$ we have
   \begin{equation*}
    \begin{array}{c}
      \begin{tikzpicture}[thick,scale=0.3]
        \draw[<-] (0,-2) to (0,2);
        \draw[->] (1,-2) to (1,2);
        \draw[->] (-1.5,0) circle (0.7);
        \draw[->] (-.8,-.1) to (-.8,+.1); 
        \node at (2,0) {\tiny $\omega$};
      \end{tikzpicture}
    \end{array}
    = 
        \begin{array}{c}
                \begin{tikzpicture}[thick,scale=0.3]
                  \draw[<-] (-0.6,-2) to (-0.6,2);
                          \draw[black] (1,0) circle (0.5); 
                  \draw[->] (2,-2) to (2,2);
                          \node (a) at (.5,0) {$\bullet$};
                  \node at (2.7,0.5) {\tiny $\omega$};
                          \node at (-0.1,0) {\tiny $2$};
                                  \draw[->] (.9,.5) to (1.1,.5);
                \end{tikzpicture}
              \end{array}
              \;+\;
    \; 2 \; \begin{array}{c}
                \begin{tikzpicture}[thick,scale=0.3]
                  \draw[<-] (0,-2) to (0,2);
                  \draw[->] (1,-2) to (1,2);
                  \node at (0,0) {$\bullet$};
                  \node at (2,0.5) {\tiny $\omega$};
                \end{tikzpicture}
              \end{array}.
            \end{equation*}
On the other hand, again by Lemma \ref{lem:blue2}, in $\mathcal{U}^{[n]}(\hgl_n)$ we have
            \begin{equation}
            \label{eqn:cups}
              \begin{array}{c}
                \begin{tikzpicture}[scale=0.3,thick]
                  \draw[->] (1,-2) to[out=90,in=0] (0,-.3)
                  to[out=180,in=90] (-1,-2); \draw[<-] (1,2)
                  to[out=-90,in=0] (0,.3) to[out=180,in=-90] (-1,2);
                                    \node at (2,0.5) {\tiny $\omega$};
                \end{tikzpicture}
              \end{array}
                      \; = \;
                          \begin{array}{c}
                \begin{tikzpicture}[thick,scale=0.3]
                  \draw[<-] (-0.6,-2) to (-0.6,2);
                          \draw[black] (1,0) circle (0.5); 
                  \draw[->] (2,-2) to (2,2);
                          \node (a) at (.5,0) {$\bullet$};
                  \node at (2.7,0.5) {\tiny $\omega$};
                          \node at (-0.1,0) {\tiny $2$};
                                  \draw[->] (.9,.5) to (1.1,.5);
                \end{tikzpicture}
              \end{array}
              \;+\;
                      \begin{array}{c}
                        \begin{tikzpicture}[thick,scale=0.3]
                          \draw[<-] (0,-2) to (0,2); \draw[->] (1,-2)
                          to (1,2); \node (a) at (1,0) {$\bullet$};
                                            \node at (2,0.5) {\tiny $\omega$};
                        \end{tikzpicture}
                      \end{array} \; + \; 
                      \begin{array}{c}
                        \begin{tikzpicture}[thick,scale=0.3]
                          \draw[<-] (0,-2) to (0,2); \draw[->] (1,-2)
                          to (1,2); \node (a) at (0,0) {$\bullet$};
                                            \node at (2,0.5) {\tiny $\omega$};
                        \end{tikzpicture}
                      \end{array}.
                    \end{equation}
We deduce that in $\mathcal{U}^{[n]}(\hgl_n)$ we have
                    \begin{equation*}
                      \begin{array}{c}
                        \begin{tikzpicture}[thick,scale=0.3]
                          \draw[<-] (0,-2) to (0,2);
                          \draw[->] (1,-2) to (1,2);
                          \draw[->] (2.5,0) circle (0.7);
                          \draw[->] (3.2,-.1) to (3.2,+.1);
                          \node at (2.5,1.5) {\tiny $\omega$};
                        \end{tikzpicture}
                      \end{array}
                      +
                      \begin{array}{c}
                        \begin{tikzpicture}[thick,scale=0.3]
                          \draw[<-] (0,-2) to (0,2);
                          \draw[->] (1,-2) to (1,2);
                          \draw[->] (-1.5,0) circle (0.7);
                          \draw[->] (-0.8,-.1) to (-0.8,+.1);
                           \node at (2,0) {\tiny $\omega$};
                        \end{tikzpicture}
                      \end{array}
                      \; = \; 2 \;
                      \begin{array}{c}
                        \begin{tikzpicture}[scale=0.3,thick]
                          \draw[->] (1,-2) to[out=90,in=0] (0,-.3) to[out=180,in=90] (-1,-2);
                          \draw[<-] (1,2) to[out=-90,in=0] (0,.3) to[out=180,in=-90] (-1,2);
                          \node at (2,0) {\tiny $\omega$};
                        \end{tikzpicture}
                      \end{array},
                    \end{equation*}
which is the desired equality.
                  \end{proof}

\begin{lem}
\label{lem:polynomial-2}
  For any $i, j \in \{1, \cdots, n-1, \infty\}$ with $i \rel j$ we have:
\begin{equation*}
\sigma \left(
  \begin{array}{c}
    \begin{tikzpicture}[thick,scale=0.3]
\draw (0,-2) to (0,2);
\draw[green] (-1,-1) to (-1,1);
\node[green] at (-1,-1) {$\bullet$};
\node[green] at (-1,1) {$\bullet$};
\node at (-1,-2.5) {\tiny $j$};
\node at (0,-2.5) {\tiny $i$};
    \end{tikzpicture}
  \end{array}
+
\begin{array}{c}\begin{tikzpicture}[thick,scale=0.3]
\draw (0,-2) to (0,-.5);
\draw (0,2) to (0,.5);
\node at (0,-.5) {$\bullet$};
\node at (0,.5) {$\bullet$};
\node at (0,-2.5) {\tiny $i$};
    \end{tikzpicture}
  \end{array}
\right)
=
\sigma \left(
  \begin{array}{c}
    \begin{tikzpicture}[thick,scale=0.3]
\draw (0,-2) to (0,2);
\draw (1,-1) to (1,1);
\node at (1,1) {$\bullet$};
\node at (1,-1) {$\bullet$};
\node at (0,-2.5) {\tiny $i$};
\node at (1,-2.5) {\tiny $i$};
    \end{tikzpicture}
  \end{array}
+
  \begin{array}{c}
    \begin{tikzpicture}[thick,scale=0.3]
\draw (0,-2) to (0,2);
\draw[green] (1,-1) to (1,1);
\node[green] at (1,1) {$\bullet$};
\node[green] at (1,-1) {$\bullet$};
\node at (0,-2.5) {\tiny $i$};
\node at (1,-2.5) {\tiny $j$};
    \end{tikzpicture}
  \end{array}
\right).
\end{equation*}
\end{lem}

\begin{proof}
Applying the relations~\eqref{pullapart},~\eqref{signs:1} and~\eqref{signs:3}, and then~\eqref{eqn:rel-2KM-2} and~\cite[(1.21)]{Brundan}, we have, since $t_{ij}t_{ji}^{-1} = -1$,
  \begin{equation*}
    \begin{array}{c}
      \begin{tikzpicture}[thick,scale=0.3]
        \draw[<-] (0,-2) to (0,2);
        \draw[->] (1,-2) to (1,2);
        \draw[green,->] (2.5,0) circle (0.7);
        \draw[green,->] (3.2,-.1) to (3.2,+.1); 
        \node at (2.5,1.5) {\tiny $\omega$};
        \node at (1,-2.5) {\tiny $i$};
        \node at (2.5, -1.2) {\tiny $j$};
      \end{tikzpicture}
    \end{array}
    \;=\;
    t_{ji}^{-1} \;
    \begin{array}{c}
      \begin{tikzpicture}[thick,scale=0.3]
        \draw[<-] (-1,-2) to (-1,2);
        \draw[->] (1,-2) to (1,2);
                \node at (1,-2.5) {\tiny $i$};
        \draw[green,->] (0.5,0) circle (1.1);
        \draw[green,->] (1.6,-.1) to (1.6,+.1); 
        \node at (2.5,1) {\tiny $\omega$};
                \node at (0, -1.5) {\tiny $j$};
      \end{tikzpicture}
    \end{array}
    \;=\;
    \begin{array}{c}
      \begin{tikzpicture}[thick,scale=0.3]
        \draw[<-] (-1,-2) to (-1,2);
        \draw[->] (1,-2) to (1,2);
        \node at (1,-2.5) {\tiny $i$};
        \node at (0,-1) {\tiny $j$};
        \draw[green,->] (0,0) circle (0.6);
        \draw[green,<-] (-.1,.6) to (.1,.6);
        \node[green] at (0.6,0) {$\bullet$};
         \node at (2,0.5) {\tiny $\omega$};
      \end{tikzpicture}
    \end{array}
    - \begin{array}{c}
        \begin{tikzpicture}[thick,scale=0.3]
          \draw[<-] (-0.2,-2) to (-0.2,2); \draw[->] (1,-2) to (1,2);
          \node at (1,0) {$\bullet$};
          \node at (2,0.5) {\tiny $\omega$};
          \node at (1,-2.5) {\tiny $i$};
        \end{tikzpicture}
      \end{array} .
    \end{equation*}
Similarly, using~\cite[(2.7)]{Brundan} we have:
    \begin{gather*}
      \begin{array}{c}
        \begin{tikzpicture}[thick,scale=0.3]
          \draw[<-] (0,-2) to (0,2);
          \draw[->] (1,-2) to (1,2);
          \node at (1,-2.5) {\tiny $i$};
          \draw[green,->] (-1.5,0) circle (0.7);
          \node at (-1.5,-1.2) {\tiny $j$};
          \draw[green,->] (-.8,-.1) to (-.8,+.1);
           \node at (2,.5) {\tiny $\omega$};
        \end{tikzpicture}
      \end{array}
      \;=\;
      \begin{array}{c}
        \begin{tikzpicture}[thick,scale=0.3]
          \draw[<-] (-1,-2) to (-1,2);
          \draw[->] (1,-2) to (1,2);
          \node at (1,-2.5) {\tiny $i$};
          \node at (0,-1) {\tiny $j$};
        \draw[green,->] (0,0) circle (0.6);
        \draw[green,<-] (-.1,.6) to (.1,.6);
        \node[green] at (0.6,0) {$\bullet$};
         \node at (2,.5) {\tiny $\omega$};
        \end{tikzpicture}
      \end{array}
      - \begin{array}{c}
          \begin{tikzpicture}[thick,scale=0.3]
            \draw[<-] (-0.2,-2) to (-0.2,2);
            \draw[->] (1,-2) to (1,2);
            \node at (1,-2.5) {\tiny $i$};
            \node at (-0.2,0) {$\bullet$};
            \node at (2,.5) {\tiny $\omega$};
          \end{tikzpicture}
        \end{array}.
      \end{gather*}
      Hence in $\mathcal{U}^{[n]}(\hgl_n)$ we have
      \begin{gather*}
        \begin{array}{c}
          \begin{tikzpicture}[thick,scale=0.3]
            \draw[<-] (0,-2) to (0,2);
            \draw[->] (1,-2) to (1,2);
            \draw[green,->] (-1.5,0) circle (0.7);
          \draw[green,->] (-.8,-.1) to (-.8,+.1);
             \node at (2,.5) {\tiny $\omega$};
             \node at (1,-2.5) {\tiny $i$};
             \node at (-1.5,-1.2) {\tiny $j$};
          \end{tikzpicture}
        \end{array}
        - \begin{array}{c}
            \begin{tikzpicture}[thick,scale=0.3]
              \draw[<-] (0,-2) to (0,2);
              \draw[->] (1,-2) to (1,2);
              \node at (1,-2.5) {\tiny $i$};
              \node at (2.5,-1.2) {\tiny $j$};
              \draw[green,->] (2.5,0) circle (0.7);
              \draw[green,->] (3.2,-.1) to (3.2,+.1);
              \node at (2.5,1.5) {\tiny $\omega$};
            \end{tikzpicture}
          \end{array}
          \;=\;
          \begin{array}{c}
            \begin{tikzpicture}[thick,scale=0.3]
              \draw[<-] (-0.2,-2) to (-0.2,2);
              \draw[->] (1,-2) to (1,2);
              \node at (1,-2.5) {\tiny $i$};
              \node at (1,0) {$\bullet$};
               \node at (2,.5) {\tiny $\omega$};
            \end{tikzpicture}
          \end{array}
          - \begin{array}{c}
              \begin{tikzpicture}[thick,scale=0.3]
                \draw[<-] (-0.2,-2) to (-0.2,2);
                \draw[->] (1,-2) to (1,2); 
                \node at (-0.2,0) {$\bullet$};
                              \node at (1,-2.5) {\tiny $i$};
                               \node at (2,.5) {\tiny $\omega$};
              \end{tikzpicture}
            \end{array}
            \;=\;
            \begin{array}{c}
              \begin{tikzpicture}[thick,scale=0.3]
                \draw[<-] (0,-2) to (0,2);
                \draw[->] (1,-2) to (1,2);
                \node at (1,-2.5) {\tiny $i$};
                \node at (2.5,-1.2) {\tiny $i$};
                \draw[->] (2.5,0) circle (0.7); \draw[->] (3.2,-.1) to
                (3.2,+.1);
                 \node at (2.5,1.5) {\tiny$\omega$};
              \end{tikzpicture}
            \end{array}
            -
            \begin{array}{c}
              \begin{tikzpicture}[scale=0.3,thick]
                \draw[->] (1,-2) to[out=90,in=0] (0,-.3)
                to[out=180,in=90] (-1,-2); \draw[<-] (1,2)
                to[out=-90,in=0] (0,.3) to[out=180,in=-90] (-1,2);
                \node at (1,-2.5) {\tiny $i$};
                \node at (-1,2.5) {\tiny $i$};
                             \node at (2,.5) {\tiny $\omega$};
              \end{tikzpicture}
            \end{array}
          \end{gather*}
          by~\eqref{eqn:alpha} and~\eqref{eqn:cups}. This is easily seen
          to imply the relation in the lemma.
        \end{proof}

The proof of the following lemma is similar to (and in fact simpler than) that of Lemma~\ref{lem:polynomial-2}; it is therefore left to the reader.

\begin{lem}
\label{lem:polynomial-3}
  For any $i, j \in \{1, \cdots, n-1, \infty\}$ with $i \neq j$ and $i \norel j$ we have:
\begin{equation*}
\sigma \left(
  \begin{array}{c}
    \begin{tikzpicture}[thick,scale=0.3]
\draw (0,-2) to (0,2);
\draw[green] (-1,-1) to (-1,1);
\node[green] at (-1,-1) {$\bullet$};
\node[green] at (-1,1) {$\bullet$};
\node at (-1,-2.5) {\tiny $j$};
\node at (0,-2.5) {\tiny $i$};
    \end{tikzpicture}
  \end{array}
\right)
=
\sigma \left(
  \begin{array}{c}
    \begin{tikzpicture}[thick,scale=0.3]
\draw (0,-2) to (0,2);
\draw[green] (1,-1) to (1,1);
\node[green] at (1,1) {$\bullet$};
\node[green] at (1,-1) {$\bullet$};
\node at (0,-2.5) {\tiny $i$};
\node at (1,-2.5) {\tiny $j$};
    \end{tikzpicture}
  \end{array}
\right).
\end{equation*}
\end{lem}

\subsection{Cyclicity}

Now we prove that the images under $\sigma$ of the $4$-valent and $6$-valent vertices are cyclic. The case of the $4$-valent vertices is easy, and
left to the reader. So, we concentrate on the $6$-valent vertices.

We fix $i,j \in \{1, \cdots, n-1, \infty\}$ with
$j \to i$. Then in $\mathcal{U}(\hgl_n)$ we have
\begin{equation*}
  \begin{array}{c}
    \begin{tikzpicture}[thick,scale=0.2]
      \draw[black,->] (-6,5) to[out=-90,in=180] (-3,-3) to[out=0,in=-120] (0,-1) to[out=60,in=-90] (1,5);
      \draw[black,<-] (-5,5) to[out=-90,in=180] (-3,-2) to[out=0,in=180] (0,1) to[out=0,in=90] (3,-5);
      \draw[black,->] (-1,5) to[out=-90,in=120] (0,-1) to[out=-60,in=90] (2,-5);
      \draw[green,->] (5,-5) to[out=90,in=0] (3.5,4) to[out=180,in=0] (0,1.5) to[out=180,in=-90] (-3,5);
      \draw[green,<-] (4,-5) to[out=90,in=0] (3.5,2) to[out=180,in=90] (.5,-5);
      \draw[green,->] (-4,5) to[out=-90,in=90] (-.5,-5);
      \node at (.5,-6) {\tiny $j$};
     \node at (3,-6) {\tiny $i$};
     \node at (5,-6) {\tiny $j$};
\node (a) at (7,0) {\tiny $\omega$};
    \end{tikzpicture}
  \end{array}
=
  \begin{array}{c}
    \begin{tikzpicture}[thick,scale=0.2]
      \draw[black,->] (-6,5) to[out=-90,in=180] (-1.5,-2) to[out=0,in=-120] (0,-1) to[out=60,in=-90] (1,5);
      \draw[black,<-] (-5,5) to[out=-90,in=180] (-1.5,-1) to[out=0,in=170] (0,1) to[out=-10,in=90] (1,-5);
      \draw[black,->] (0,5) to[out=-90,in=120] (-.5,-1) to[out=-60,in=90] (0,-5);
      \draw[green,->] (-2,5) to (-3.5,-5);
      \draw[green,<-] (-1,5) to (4.5,-5);
      \draw[green,->] (-1.5,-5) to[out=70,in=-170] (0.5,-3) to[out=10,in=110] (2,-5);
      \node at (-1.5,-6) {\tiny $j$};
     \node at (1,-6) {\tiny $i$};
     \node at (4.5,-6) {\tiny $j$};
\node (a) at (3,0) {\tiny $\omega$};
    \end{tikzpicture}
  \end{array}
=
  \begin{array}{c}
    \begin{tikzpicture}[thick,xscale=-0.15,yscale=-0.2]
      \draw[<-] (-5,-5) to[out=90,in=180] (0,1.5) to[out=0,in=90] (5,-5);
      \draw[->] (-3,-5) to[out=90,in=-90] (1,1) to (1,5);
      \draw[<-] (3,-5) to[out=90,in=-90] (-1,1) to (-1,5);
\draw[<-,green] (-1,-5) to[out=90,in=-90] (-5,5);
\draw[->,green] (1,-5) to[out=90,in=-90] (5,5);
\draw[->,green] (3,5) to[out=-90,in=0] (0,2) to[out=180,in=-90] (-3,5);
\node (a) at (-6,0) {\tiny $\omega$};
      \node at (3,6) {\tiny $j$};
     \node at (-1,6) {\tiny $i$};
     \node at (-5,6) {\tiny $j$};
    \end{tikzpicture}
\end{array}
\end{equation*}
where the first equality uses each of the four sign rules in Lemma~\ref{signrules} once, and the second equality follows from~\eqref{signs:2} and~\eqref{signs:4}.

Using similar computations, we have
\[
  \begin{array}{c}
    \begin{tikzpicture}[thick,scale=-0.2]
      \draw[black,->] (-6,5) to[out=-90,in=180] (-3,-3) to[out=0,in=-120] (0,-1) to[out=60,in=-90] (1,5);
      \draw[black,<-] (-5,5) to[out=-90,in=180] (-3,-2) to[out=0,in=180] (0,1) to[out=0,in=90] (3,-5);
      \draw[black,->] (-1,5) to[out=-90,in=120] (0,-1) to[out=-60,in=90] (2,-5);
      \draw[green,->] (5,-5) to[out=90,in=0] (3.5,4) to[out=180,in=0] (0,1.5) to[out=180,in=-90] (-3,5);
      \draw[green,<-] (4,-5) to[out=90,in=0] (3.5,2) to[out=180,in=90] (.5,-5);
      \draw[green,->] (-4,5) to[out=-90,in=90] (-.5,-5);
      \node at (-1,6) {\tiny $i$};
     \node at (-6,6) {\tiny $i$};
     \node at (-4,6) {\tiny $j$};
\node (a) at (-7,0) {\tiny $\omega$};
    \end{tikzpicture}
  \end{array}
=
  \begin{array}{c}
    \begin{tikzpicture}[thick,scale=-0.2]
      \draw[black,->] (-6,5) to[out=-90,in=180] (-1.5,-2) to[out=0,in=-120] (0,-1) to[out=60,in=-90] (1,5);
      \draw[black,<-] (-5,5) to[out=-90,in=180] (-1.5,-1) to[out=0,in=170] (0,1) to[out=-10,in=90] (1,-5);
      \draw[black,->] (0,5) to[out=-90,in=120] (-.5,-1) to[out=-60,in=90] (0,-5);
      \draw[green,->] (-2,5) to (-3.5,-5);
      \draw[green,<-] (-1,5) to (4.5,-5);
      \draw[green,->] (-1.5,-5) to[out=70,in=-170] (0.5,-3) to[out=10,in=110] (2,-5);
      \node at (0,6) {\tiny $i$};
     \node at (-6,6) {\tiny $i$};
     \node at (-2,6) {\tiny $j$};
\node (a) at (-6.5,0) {\tiny $\omega$};
    \end{tikzpicture}
  \end{array}
=
  \begin{array}{c}
    \begin{tikzpicture}[thick,xscale=0.15,yscale=0.2]
      \draw[<-] (-5,-5) to[out=90,in=180] (0,1.5) to[out=0,in=90] (5,-5);
      \draw[->] (-3,-5) to[out=90,in=-90] (1,1) to (1,5);
      \draw[<-] (3,-5) to[out=90,in=-90] (-1,1) to (-1,5);
\draw[<-,green] (-1,-5) to[out=90,in=-90] (-5,5);
\draw[->,green] (1,-5) to[out=90,in=-90] (5,5);
\draw[->,green] (3,5) to[out=-90,in=0] (0,2) to[out=180,in=-90] (-3,5);
\node (a) at (6,0) {\tiny $\omega$};
      \node at (5,-6) {\tiny $i$};
     \node at (1,-6) {\tiny $j$};
     \node at (-3,-6) {\tiny $i$};
    \end{tikzpicture}
\end{array} .
\]
These relations imply that indeed the $6$-valent vertices are cyclic.

\subsection{Jones--Wenzl relations}


Now we turn to the Jones--Wenzl relations, i.e. the relations described in~\cite[(5.7)]{ew}. In our specific situation, these relations are described in the following proposition.

\begin{prop}
Let $i,j \in \{1, \cdots, n-1,\infty\}$ with $i \neq j$.
  \begin{enumerate}
  \item
\label{it:JW1}  
  If $i \norel j$ we have:
    \begin{gather*}
\sigma \left(
      \begin{array}{c}
            \begin{tikzpicture}[thick,xscale=0.1,yscale=0.15]
              \draw[green] (-3,-3) to (2,2);
              \node[green] at (2,2) {$\bullet$};
              \draw (3,-3) to (-3,3);
      \node at (-3,-4) {\tiny $i$};
     \node at (3,-4) {\tiny $j$};
            \end{tikzpicture}
      \end{array}
\right)
=
\sigma \left(
      \begin{array}{c}
            \begin{tikzpicture}[thick,xscale=0.1,yscale=0.15]
              \draw[green] (-3,-3) to (-1,-1);
              \node[green] at (-1,-1) {$\bullet$};
              \draw (3,-3) to (-3,3);
      \node at (-3,-4) {\tiny $i$};
     \node at (3,-4) {\tiny $j$};
            \end{tikzpicture}
      \end{array}
\right).
    \end{gather*}
  \item 
\label{it:JW2}  
  If $i \rel j$ we have:
  \end{enumerate}
\begin{align*}
\sigma \left(
  \begin{array}{c}
    \begin{tikzpicture}[thick,scale=0.15]
      \draw (-4,-5) to (0,0) to (4,-5);
      \draw (0,5) to (0,0);
      \draw[green] (-4,5) to (0,0) to (3,3);
      \node[green] at (3,3) {$\bullet$};
      \draw[green] (0,-5) to (0,0);
      \node at (-4,-6) {\tiny $i$};
     \node at (0,-6) {\tiny $j$};
           \node at (4,-6) {\tiny $i$};
    \end{tikzpicture}
  \end{array}
\right)
\;=\;
\sigma \left(
  \begin{array}{c}
    \begin{tikzpicture}[thick,scale=0.15]
      \draw (-4,-5) to (-2.5,-2.5);
      \node[black] at (-2.5,-2.5) {$\bullet$};
      \draw (4,-5) to[out=90,in=-90] (1,5);
      \draw[green] (0,-5) to[out=90,in=-90] (-2,5);
      \node at (-4,-6) {\tiny $i$};
     \node at (0,-6) {\tiny $j$};
           \node at (4,-6) {\tiny $i$};
    \end{tikzpicture}
  \end{array}
\;+\;
  \begin{array}{c}
    \begin{tikzpicture}[thick,scale=0.15]
      \draw (-4,-5) to (0,0) to (4,-5);
      \draw (0,5) to (0,0);
      \draw[green] (-4,5) to (-2,2);
      \node[green] at (-2,2) {$\bullet$};
      \draw[green] (0,-5) to (0,-3);
      \node[green] at (0,-3) {$\bullet$};
      \node at (-4,-6) {\tiny $i$};
     \node at (0,-6) {\tiny $j$};
    \end{tikzpicture}
  \end{array}
\right).
\end{align*}
\end{prop}

\begin{proof}
\eqref{it:JW1}~follows from Lemma~\ref{signrules},~\eqref{pullapart} and~\eqref{eqn:rel-2KM-2}; details are left to
  the reader. Now we consider~\eqref{it:JW2}. We treat the cases $i
  \to j$ and $j \to i$ separately. 

\emph{Case 1: $j \to i$.} The left-hand side of the desired equality is
\begin{gather*}
 - \;\; \begin{array}{c}
    \begin{tikzpicture}[thick,scale=0.2]
      \draw[<-] (-5,-5) to[out=90,in=180] (0,1.5) to[out=0,in=90] (5,-5);
      \draw[->] (-3,-5) to[out=90,in=-90] (1,1) to (1,5);
      \draw[<-] (3,-5) to[out=90,in=-90] (-1,1) to (-1,5);
\draw[->,green] (1,-5) to[out=90,in=-90] (4,3) to[out=90,in=0] (3,5) to[out=180,in=90] (2,3.5)
to[out=-90,in=0] (0,2) to[out=180,in=-90] (-3,5);
\draw[<-,green] (-1,-5) to[out=90,in=-90] (-5,5);
\node (a) at (5,2) {\tiny $\omega$};
\node at (-3,-5.6) {\tiny $i$};
\node at (1,-5.6) {\tiny $j$};
\node at (5,-5.6) {\tiny $i$};
    \end{tikzpicture}  \end{array}
\stackrel{\eqref{signs:3}}{=}
-t_{ji}^{-1} \;
  \begin{array}{c}
    \begin{tikzpicture}[thick,scale=0.2]
      \draw[<-] (-5,-5) to[out=90,in=180] (0,1.5) to[out=0,in=90] (5,-5);
      \draw[->] (-3,-5) to[out=90,in=-90] (1,1) to (1,5);
      \draw[<-] (3,-5) to[out=90,in=-90] (-1,1) to (-1,5);
\draw[->,green] (1,-5) to[out=90,in=-90] (3,0) to[out=90,in=-90] (-3,5);
\draw[<-,green] (-1,-5) to[out=90,in=-90] (-5,5);
\node (a) at (5,2) {\tiny$\omega$};
\node at (-3,-5.6) {\tiny $i$};
\node at (1,-5.6) {\tiny $j$};
\node at (5,-5.6) {\tiny $i$};
    \end{tikzpicture}  
    \end{array}
\stackrel{\eqref{eq:nilhecke3}}{=}
t_{ij}t_{ji}^{-1}
  \begin{array}{c}
    \begin{tikzpicture}[thick,scale=0.2]
      \draw[->] (5,-5) to[out=90,in=-90] (1,5);
      \draw[->] (-3,-5) to[out=90,in=-90] (1,0) to[out=90,in=0]
      (-1,1.5) to[out=180,in=90] (-5,-5);
      \draw[<-] (3,-5) to[out=90,in=-90] (-2.3,1) to (-1,5);
\draw[->,green] (1,-5) to[out=90,in=-90] (2,0) to[out=90,in=-90] (-3,5);
\draw[<-,green] (-1,-5) to[out=90,in=-90] (-5,5);
\node (a) at (5,2) {\tiny $\omega$};
\node at (-3,-5.6) {\tiny $i$};
\node at (1,-5.6) {\tiny $j$};
\node at (5,-5.6) {\tiny $i$};
    \end{tikzpicture}  \end{array} \\
\stackrel{\eqref{ipullapart} + \eqref{signs:4} + \eqref{pullapart}}{=} \quad
-t_{ij}t_{ji}^{-1}
  \begin{array}{c}
    \begin{tikzpicture}[thick,scale=0.2]
      \draw[->] (5,-5) to[out=90,in=-90] (1,5);
      \draw[->] (-3,-5) to[out=90,in=0] (-4,-3) to[out=180,in=90] (-5,-5);
      \draw[<-] (3,-5) to[out=90,in=-90] (-1,1) to (-1,5);
\draw[->,green] (1,-5) to[out=90,in=-90] (2,0) to[out=90,in=-90] (-3,5);
\draw[<-,green] (-1,-5) to[out=90,in=-90] (-5,5);
\node (a) at (5,2) {\tiny $\omega$};
\node at (-3,-5.6) {\tiny $i$};
\node at (1,-5.6) {\tiny $j$};
\node at (5,-5.6) {\tiny $i$};
    \end{tikzpicture}  \end{array} 
+t_{ij}t_{ji}^{-1}
  \begin{array}{c}
    \begin{tikzpicture}[thick,scale=0.2]
      \draw[->] (5,-5) to[out=90,in=-90] (1,5);
      \draw[->] (-3,-5) to[out=90,in=180] (0,-1) to[out=0,in=90] (3,-5);
      \draw[<-] (-5,-5) to[out=90,in=-90] (-1,5);
\draw[->,green] (1,-5) to[out=90,in=-90] (2,-1) to[out=90,in=-90] (-3,5);
\draw[<-,green] (-1,-5) to[out=90,in=-90] (-5,5);
\node (a) at (5,2) {\tiny $\omega$};
\node at (-3,-5.6) {\tiny $i$};
\node at (1,-5.6) {\tiny $j$};
\node at (5,-5.6) {\tiny $i$};
    \end{tikzpicture}  \end{array} \\
\stackrel{\eqref{pullapart} + \eqref{signs:2}}{=} \quad
  \begin{array}{c}
    \begin{tikzpicture}[thick,scale=0.2]
      \draw[->] (5,-5) to[out=90,in=-90] (5,5);
      \draw[->] (-3,-5) to[out=90,in=0] (-4,-3) to[out=180,in=90] (-5,-5);
      \draw[<-] (3,-5) to (3,5);
\draw[->,green] (1,-5) to (1,5);
\draw[<-,green] (-1,-5) to (-1,5);
\node (a) at (6,2) {\tiny $\omega$};
\node at (-3,-5.6) {\tiny $i$};
\node at (1,-5.6) {\tiny $j$};
\node at (5,-5.6) {\tiny $i$};
    \end{tikzpicture}  \end{array} 
+t_{ij}
  \begin{array}{c}
    \begin{tikzpicture}[thick,scale=0.2]
      \draw[->] (5,-5) to[out=90,in=-90] (4,5);
      \draw[->] (-3,-5) to[out=90,in=180] (0,-1) to[out=0,in=90] (3,-5);
      \draw[<-] (-5,-5) to[out=90,in=-90] (2,5);
\draw[->,green] (1,-5) to[out=90,in=-90] (3,-1) to[out=90,in=-90] (0,5);
\draw[<-,green] (-1,-5) to[out=90,in=-90] (1,-1) to[out=90,in=-90] (-2,5);
\node (a) at (6,2) {\tiny $\omega$};
\node at (-3,-5.6) {\tiny $i$};
\node at (1,-5.6) {\tiny $j$};
\node at (5,-5.6) {\tiny $i$};
    \end{tikzpicture}  \end{array} 
\end{gather*}
(where in the fourth equality we also use that $t_{ij}t_{ji}^{-1}
= -1$).

The first term is as expected. We
simplify the second term further as follows:
\begin{multline*}
t_{ij}
  \begin{array}{c}
    \begin{tikzpicture}[thick,scale=0.2]
      \draw[->] (5,-5) to[out=90,in=-90] (4,5);
      \draw[->] (-3,-5) to[out=90,in=180] (0,-1) to[out=0,in=90] (3,-5);
      \draw[<-] (-5,-5) to[out=90,in=-90] (2,5);
\draw[->,green] (1,-5) to[out=90,in=-90] (3,-1) to[out=90,in=-90] (0,5);
\draw[<-,green] (-1,-5) to[out=90,in=-90] (1,-1) to[out=90,in=-90] (-2,5);
\node (a) at (6,2) {\tiny $\omega$};
\node at (-3,-5.6) {\tiny $i$};
\node at (1,-5.6) {\tiny $j$};
\node at (5,-5.6) {\tiny $i$};
    \end{tikzpicture}  \end{array} \quad
\stackrel{\eqref{splits}}= \quad
t_{ij}
  \begin{array}{c}
    \begin{tikzpicture}[thick,scale=0.2]
      \draw[->] (5,-5) to[out=90,in=-90] (4,5);
      \draw[->] (-3,-5) to[out=90,in=180] (0,-1) to[out=0,in=90] (3,-5);
      \draw[<-] (-5,-5) to[out=90,in=-90] (2,5);
\draw[->,green] (1,-5) to[out=90,in=-90] (3,-1) to[out=90,in=0] (2,0)
to[out=180,in=45] (1,-2) to[out=-135,in=90] (-1,-5);
\draw[<-,green] (0,5) to[out=-90,in=90] (3,2) to[out=-90,in=0] (2,1)
to[out=180,in=-45] (0,3) to[out=135,in=-90] (-2,5);
\node (a) at (6,2) {\tiny $\omega$};
\node at (-3,-5.6) {\tiny $i$};
\node at (1,-5.6) {\tiny $j$};
\node at (5,-5.6) {\tiny $i$};
    \end{tikzpicture}  \end{array} 
\\
\stackrel{\eqref{signs:2} + \eqref{signs:4} + \eqref{pullapart}}= \quad
  \begin{array}{c}
    \begin{tikzpicture}[thick,scale=0.2]
      \draw[->] (5,-5) to[out=90,in=-90] (4,5);
      \draw[->] (-3,-5) to[out=90,in=180] (0,-1) to[out=0,in=90] (3,-5);
      \draw[<-] (-5,-5) to[out=90,in=-90] (2,5);
\draw[->,green] (1,-5) to[out=90,in=0] (0,-3) to[out=180,in=90] (-1,-5);
\draw[<-,green] (0,5) to[out=-90,in=0] (-1,3) to[out=180,in=-90] (-2,5);
\node (a) at (6,2) {\tiny $\omega$};
\node at (-3,-5.6) {\tiny $i$};
\node at (1,-5.6) {\tiny $j$};
\node at (5,-5.6) {\tiny $i$};
    \end{tikzpicture}  \end{array} .
\end{multline*}
Hence we are done in this case.

\emph{Case 2: $i \to j$.} 
Using~\eqref{signs:3} and the third case in~\eqref{eqn:rel-2KM-2}, we see that
\[
     -\;\; \begin{array}{c}
\begin{tikzpicture}[thick,scale=0.2]
      \draw[->] (-1,5) to[out=-90,in=90] (-5,-5);
      \draw[<-] (1,5) to[out=-90,in=90] (5,-5);
      \draw[->] (-3,-5) to[out=90,in=180] (0,-1) to[out=0,in=90] (3,-5);
\draw[green,->] (-5,5) to[out=-90,in=180] (0,0) to[out=0,in=-90] (4,3)
to[out=90,in=0] (3,4) to[out=180,in=60] (2,3) to[out=-120,in=80]
(-1,-1) to[out=-100,in=90] (-1,-5);
\draw[green,->] (1,-5) to (1,-1) to[out=90,in=-90] (-3,5);
\node (a) at (5,2) {\tiny $\omega$};
\node at (-3,-5.6) {\tiny $i$};
\node at (1,-5.6) {\tiny $j$};
\node at (5,-5.6) {\tiny $i$};
    \end{tikzpicture}
  \end{array}
=
-t_{ji}^{-1}t_{ji}\;
 \begin{array}{c}
    \begin{tikzpicture}[thick,scale=0.2]
      \draw[->] (-1,5) to[out=-90,in=90] (-5,-5);
      \draw[<-] (1,5) to[out=-90,in=90] (5,-5);
      \draw[->] (-3,-5) to[out=90,in=180] (0,-1) to[out=0,in=90] (3,-5);
\draw[green,->] (-5,5) to[out=-90,in=180] (0,0) to[out=0,in=-90] (1.3,1.3)
to[out=90,in=0] (0.3,2.3) to[out=180,in=60] (-0.8,1.3) to (-1,-5);
\draw[green,->] (1,-5) to (1,-1) to[out=90,in=-90] (-3,5);
\node[green] (b) at (1,2) {$\bullet$}; 
\node (a) at (4,2) {\tiny $\omega$};
\node at (-3,-5.6) {\tiny $i$};
\node at (1,-5.6) {\tiny $j$};
\node at (5,-5.6) {\tiny $i$};
    \end{tikzpicture}
\end{array} .
\]
(In fact, the second term coming from the application of~\eqref{eqn:rel-2KM-2} vanishes due to the first case in~\eqref{eqn:rel-2KM-2}.) Now, using~\eqref{eqn:rel-2KM-1}, the same vanishing property as above, and the sign rule~\eqref{signs:2}, we see that the right-hand side here is equal to
\begin{equation}
\label{eqn:JW-proof}
 \begin{array}{c}
    \begin{tikzpicture}[thick,scale=0.2]
      \draw[->] (2,5) to[out=-90,in=30] (-1.5,2) to[out=-150,in=90] (-5,-5);
      \draw[<-] (4.5,5) to[out=-90,in=90] (5,-5);
      \draw[->] (-3,-5) to[out=90,in=180] (0,-1) to[out=0,in=90] (3,-5);
\draw[green,->] (.5,-5) to[out=90,in=-60] (2,1) to[out=120,in=0] (.5,2) to[out=180,in=60] (-1,1) to[out=-120,in=90] (-1,-5);
\draw[green,<-] (-1,5) to[out=-90,in=60] (-.5,.5) to[out=-120,in=0] (-2,-0.5) to[out=180,in=-60] (-3.5,.5) to[out=120,in=-90] (-2,5);
\node (a) at (5.5,2) {\tiny $\omega$};
\node at (-3,-5.6) {\tiny $i$};
\node at (1,-5.6) {\tiny $j$};
\node at (5,-5.6) {\tiny $i$};
    \end{tikzpicture}
\end{array}\\
\stackrel{\eqref{ipullapart}}{=}
\begin{array}{c}
    \begin{tikzpicture}[thick,scale=0.2]
      \draw[->] (3,5) to[out=-90,in=30] (-1.5,2) to[out=-150,in=90] (-5,-5);
      \draw[<-] (4.5,5) to[out=-90,in=90] (5,-5);
      \draw[->] (-3,-5) to[out=90,in=180] (0,-1) to[out=0,in=90] (3,-5);
\draw[green,->] (1,-5) to (1,5);
\draw[green,->] (-1,5) to (-1,-5);
\node (a) at (5.5,2) {\tiny $\omega$};
\node at (-3,-5.6) {\tiny $i$};
\node at (1,-5.6) {\tiny $j$};
\node at (5,-5.6) {\tiny $i$};
    \end{tikzpicture}
\end{array}
-
 \begin{array}{c}
    \begin{tikzpicture}[thick,scale=0.2]
      \draw[->] (3,5) to[out=-90,in=30] (-1.5,2) to[out=-150,in=90] (-5,-5);
      \draw[<-] (4.5,5) to[out=-90,in=90] (5,-5);
      \draw[->] (-3,-5) to[out=90,in=180] (0,-1) to[out=0,in=90] (3,-5);
\draw[green,->] (1,-5) to[out=90,in=-60] (1.5,-1) to[out=120,in=0] (0,0) to[out=180,in=60] (-1.5,-1) to[out=-120,in=90] (-1,-5);
\draw[green,<-] (1,5) to[out=-90,in=60] (1.5,1.5) to[out=-120,in=0] (0,0.5) to[out=180,in=-60] (-1.5,1.5) to[out=120,in=-90] (-1,5);
\node (a) at (5.5,2) {\tiny$\omega$};
\node at (-3,-5.6) {\tiny $i$};
\node at (1,-5.6) {\tiny $j$};
\node at (5,-5.6) {\tiny $i$};
    \end{tikzpicture}
\end{array} .
\end{equation}
We now analyze the second term in the right-hand side of~\eqref{eqn:JW-proof}.
Using the sign rules~\eqref{signs:1},~\eqref{signs:2} and~\eqref{signs:3} and~\eqref{pullapart}, we see that
\[
- \begin{array}{c}
    \begin{tikzpicture}[thick,scale=0.2]
      \draw[->] (3,5) to[out=-90,in=30] (-1.5,2) to[out=-150,in=90] (-5,-5);
      \draw[<-] (4.5,5) to[out=-90,in=90] (5,-5);
      \draw[->] (-3,-5) to[out=90,in=180] (0,-1) to[out=0,in=90] (3,-5);
\draw[green,->] (1,-5) to[out=90,in=-60] (1.5,-1) to[out=120,in=0] (0,0) to[out=180,in=60] (-1.5,-1) to[out=-120,in=90] (-1,-5);
\draw[green,<-] (1,5) to[out=-90,in=60] (1.5,1.5) to[out=-120,in=0] (0,0.5) to[out=180,in=-60] (-1.5,1.5) to[out=120,in=-90] (-1,5);
\node (a) at (5.5,2) {\tiny$\omega$};
\node at (-3,-5.6) {\tiny $i$};
\node at (1,-5.6) {\tiny $j$};
\node at (5,-5.6) {\tiny $i$};
    \end{tikzpicture}
\end{array}
=    -t_{ij}^{-1} t_{ji} 
 \begin{array}{c}
    \begin{tikzpicture}[thick,scale=0.2]
      \draw[->] (3,5) to[out=-90,in=30] (-2.5,0) to[out=-150,in=90] (-5,-5);
      \draw[<-] (4.5,5) to[out=-90,in=90] (5,-5);
      \draw[->] (-3,-5) to[out=90,in=180] (0,-1) to[out=0,in=90] (3,-5);
\draw[green,->] (1,-5) to[out=90,in=0] (0,-3) to[out=180,in=90] (-1,-5);
\draw[green,<-] (1,5) to[out=-90,in=0] (0,3) to[out=180,in=-90] (-1,5);
\node (a) at (5.5,2) {\tiny $\omega$};
\node at (-3,-5.6) {\tiny $i$};
\node at (1,-5.6) {\tiny $j$};
\node at (5,-5.6) {\tiny $i$};
    \end{tikzpicture}
    \end{array}
    =
 \begin{array}{c}
    \begin{tikzpicture}[thick,scale=0.2]
      \draw[->] (3,5) to[out=-90,in=30] (-2.5,0) to[out=-150,in=90] (-5,-5);
      \draw[<-] (4.5,5) to[out=-90,in=90] (5,-5);
      \draw[->] (-3,-5) to[out=90,in=180] (0,-1) to[out=0,in=90] (3,-5);
\draw[green,->] (1,-5) to[out=90,in=0] (0,-3) to[out=180,in=90] (-1,-5);
\draw[green,<-] (1,5) to[out=-90,in=0] (0,3) to[out=180,in=-90] (-1,5);
\node (a) at (5.5,2) {\tiny $\omega$};
\node at (-3,-5.6) {\tiny $i$};
\node at (1,-5.6) {\tiny $j$};
\node at (5,-5.6) {\tiny $i$};
    \end{tikzpicture}
    \end{array} .
\]
Finally we analyze the first term in the right-hand side of~\eqref{eqn:JW-proof}. Here we use~\eqref{ipullapart} (in which the left-hand side vanishes in $\mathcal{U}^{[n]}(\hgl_n)$ in case $i \to j$ since it contains a region labelled by $\omega +2\alpha_i + \alpha_j$) and then the sign rule~\eqref{signs:1} and~\eqref{pullapart}, to see that
\[
\begin{array}{c}
    \begin{tikzpicture}[thick,scale=0.2]
      \draw[->] (3,5) to[out=-90,in=30] (-1.5,2) to[out=-150,in=90] (-5,-5);
      \draw[<-] (4.5,5) to[out=-90,in=90] (5,-5);
      \draw[->] (-3,-5) to[out=90,in=180] (0,-1) to[out=0,in=90] (3,-5);
\draw[green,->] (1,-5) to (1,5);
\draw[green,->] (-1,5) to (-1,-5);
\node (a) at (5.5,2) {\tiny $\omega$};
\node at (-3,-5.6) {\tiny $i$};
\node at (1,-5.6) {\tiny $j$};
\node at (5,-5.6) {\tiny $i$};
    \end{tikzpicture}
\end{array}
=
\begin{array}{c}
    \begin{tikzpicture}[thick,scale=0.2]
      \draw[->] (-3,-5) to[out=90,in=0] (-2,0) to[out=180,in=90] (-5,-5);
      \draw[->] (3,5) to[out=-90,in=90] (2,0) to[out=-90,in=90] (3,-5);
\draw[green,->] (-1,5) to [out=-90,in=90] (-4,0) to[out=-90,in=90] (-1,-5);
\draw[green,<-] (1,5) to [out=-90,in=90] (4,0) to[out=-90,in=90] (1,-5);
      \draw[<-] (4.5,5) to[out=-90,in=90] (5,-5);
\node (a) at (5.5,2) {\tiny $\omega$};
\node at (-3,-5.6) {\tiny $i$};
\node at (1,-5.6) {\tiny $j$};
\node at (5,-5.6) {\tiny $i$};
    \end{tikzpicture}
  \end{array}
  =
\begin{array}{c}
    \begin{tikzpicture}[thick,scale=0.2]
      \draw[->] (-3,-5) to[out=90,in=0] (-4,-2) to[out=180,in=90] (-5,-5);
      \draw[->] (3,5) to (3,-5);
\draw[green,->] (-1,5) to (-1,-5);
\draw[green,<-] (1,5) to (1,-5);
      \draw[<-] (5,5) to (5,-5);
\node (a) at (6,2) {\tiny $\omega$};
\node at (-3,-5.6) {\tiny $i$};
\node at (1,-5.6) {\tiny $j$};
\node at (5,-5.6) {\tiny $i$};
    \end{tikzpicture}
  \end{array}.
\]
The proposition is proved. \end{proof}

\subsection{Two color associativity}

We now turn to the ``two color associativity'' relations, see~\cite[(5.6)]{ew}. As for the Jones--Wenzl relations, these relations are associated with pairs $(i,j)$ of elements in $\{1, \cdots, n-1, \infty\}$ with $i \neq j$. In the case when $i \norel j$, these relations are easy and left to the reader. Now we assume that $i \rel j$.

Once one has the Jones-Wenzl relations and the one colour relations, one can check that the
two color associativity relations associated with the pair $(i,j)$ is implied by the following relation:
\begin{gather}
\label{2colassoc}
\begin{array}{c}
  \begin{tikzpicture}[thick,baseline=0,scale=0.4]
    \draw[green] (-3,-3) to[out=90,in=-120] (-2,-1);
    \draw[green] (1,3) to[out=-90,in=60] (0,1);
    \draw[green] (-2,-1) to[out=-60,in=150] (-1,-2);
    \draw[green] (0,1) to[out=120,in=-30] (-1,2);
   \draw[green] (-1,-2) to[out=30,in=-90] (0,1);
    \draw[green] (-1,2) to[out=-150,in=90] (-2,-1);
\draw[green] (-1,-3) to (-1,-2);
\draw[green] (-1,2) to (-1,3);
\draw (-2,-3) to[out=90,in=-150] (-1,-2);
\draw (0,-3) to[out=90,in=-30] (-1,-2);
\draw (-2,3) to[out=-90,in=150] (-1,2);
\draw (0,3) to[out=-90,in=30] (-1,2);
\draw (-1,-2) to (-1,2);
\node at (-3,-3.4) {\tiny $j$};
\node at (-2,-3.4) {\tiny $i$};
\node at (-1,-3.4) {\tiny $j$};
\node at (0,-3.4) {\tiny $i$};
  \end{tikzpicture}
\end{array}
=
\begin{array}{c}
  \begin{tikzpicture}[thick,baseline=0,scale=0.4]
    \draw[green] (-3,-3) to[out=90,in=-120] (-2,-1);
    \draw[green] (1,3) to[out=-90,in=60] (0,1);
    \draw[green] (-2,-1) to[out=-60,in=150] (-1,-2);
    \draw[green] (0,1) to[out=120,in=-30] (-1,2);
   \draw[green] (-1,-2) to[out=30,in=-90] (0,1);
    \draw[green] (-1,2) to[out=-150,in=90] (-2,-1);
\draw[green] (-1,-3) to (-1,-2);
\draw[green] (-1,2) to (-1,3);
\draw (-2,-3) to (-2,-1);
\draw (0,-3) to[out=90,in=-30] (0,1);
\draw (-2,3) to[out=-90,in=150] (-2,-1);
\draw (0,3) to (0,1);
\draw (-2,-1) to (0,1);
\node at (-3,-3.4) {\tiny $j$};
\node at (-2,-3.4) {\tiny $i$};
\node at (-1,-3.4) {\tiny $j$};
\node at (0,-3.4) {\tiny $i$};
  \end{tikzpicture}
\end{array}
\end{gather}
(see e.g.~\cite[Example~3.7]{ek}).
If we denote the left hand side of~\eqref{2colassoc} by a box, it
is easy to see, using cyclicity, that this relation is equivalent to the relation:
\begin{gather}
  \label{boxrotate}
  \begin{array}{c}
\begin{tikzpicture}[thick,scale=0.3]
  \draw (0.5,0) rectangle (4.5,4);
  \draw (1,4) to[out=90,in=0] (.5,5) to[out=180,in=90] (-.5,4) to
  (-.5,-2);
  \draw (4,0) to[out=-90,in=180] (4.5,-1) to[out=0,in=-90] (5.5,0) to
  (5.5,6);
\draw[green] (1,0) to (1,-2);
  \draw[green] (2,4) to (2,6); \draw (2,0) to (2,-2);
  \draw (3,4) to (3,6); \draw[green] (3,0) to (3,-2);
  \draw[green] (4,4) to (4,6);
  \node at (-.5,-2.4) {\tiny $i$};
\node at (1,-2.4) {\tiny $j$};
\node at (2,-2.4) {\tiny $i$};
\node at (3,-2.4) {\tiny $j$};
\end{tikzpicture}  \end{array}
=  \begin{array}{c}
\begin{tikzpicture}[xscale=-1,thick,scale=0.3]
  \draw (0.5,0) rectangle (4.5,4);
  \draw[green] (1,4) to[out=90,in=0] (.5,5) to[out=180,in=90] (-.5,4) to
  (-.5,-2);
  \draw[green] (4,0) to[out=-90,in=180] (4.5,-1) to[out=0,in=-90] (5.5,0) to
  (5.5,6);
\draw[black] (1,0) to (1,-2);
  \draw[black] (2,4) to (2,6); \draw[green]  (2,0) to (2,-2);
  \draw[green]  (3,4) to (3,6); \draw[black] (3,0) to (3,-2);
  \draw[black] (4,4) to (4,6);
  \node at (-.5,-2.4) {\tiny $j$};
  \node at (1,-2.4) {\tiny $i$};
\node at (2,-2.4) {\tiny $j$};
\node at (3,-2.4) {\tiny $i$};
\end{tikzpicture}
  \end{array}.
\end{gather}

In order to prove this relation we need some preparatory lemmas.

\begin{lem}
Suppose that $i, j \in \{1, \cdots, n-1, \infty\}$ and that $j \to i$.
  \begin{enumerate}
  \item 
  In $\mathcal{U}^{[n]}(\hgl_n)$
  we have
  \begin{gather} \label{ef3}
      \begin{array}{c}
    \begin{tikzpicture}[thick,scale=0.3]
      \draw[black,->] (0,-2) to (0,2);
      \draw[black,<-] (1,-2) to (1,2);
      \node at (0,-2.5) {\tiny $i$};      \node at (1,-2.5) {\tiny $i$};
      \node[right] at (1.5,0) {\tiny $\omega + 2\a_i + \a_j$};
    \end{tikzpicture}
\end{array} 
=
      \begin{array}{c}
    \begin{tikzpicture}[thick,scale=0.3]
      \draw[black,->] (-1,-2) to[out=90,in=180] (0,-0.4) to[out=0,in=90] (1,-2);
      \draw[black,<-] (-1,2) to[out=-90,in=180] (0,0.4)
      to[out=0,in=-90] (1,2);
      \node at (0,0.4) {$\bullet$};
      \node at (0,-0.4) {$\bullet$};
      \node at (-1,-2.5) {\tiny $i$};      
      \node at (1,2.5) {\tiny $i$};
      \node[right] at (1.5,0) {\tiny $\omega + 2\a_i + \a_j$};
    \end{tikzpicture}
\end{array} 
+ R
  \end{gather}
where $R$ is a linear combination of terms of the form $
      \begin{array}{c}
    \begin{tikzpicture}[thick,scale=0.2]
      \draw[black,->] (-1,-2) to[out=90,in=180] (0,-0.4)
      to[out=0,in=90] (1,-2);
      \draw (-1,2) rectangle (1,0);\node at (0,1) {\small $f$};
      to[out=0,in=-90] (1,2);
      \node at (-1,-2.7) {\tiny $i$};
      \node[right] at (1.5,0) {\tiny $\omega + 2\a_i + \a_j$};
    \end{tikzpicture}
\end{array}$ and $
      \begin{array}{c}
    \begin{tikzpicture}[thick,scale=-0.2]
      \draw[black,->] (-1,-2) to[out=90,in=180] (0,-0.4)
      to[out=0,in=90] (1,-2);
      \draw (-1,2) rectangle (1,0); \node at (0,1) {\small $f'$};
      to[out=0,in=-90] (1,2);
      \node at (-1,-2.6) {\tiny $i$};      
      \node[right] at (-1.5,0) {\tiny $\omega + 2\a_i + \a_j$};
    \end{tikzpicture}
\end{array}$ 
for some $f \in \Hom(\id, E_iF_i)$ and $f' \in \Hom(E_iF_i, \id)$.
 \item 
In $\mathcal{U}^{[n]}(\hgl_n)$
 we have
  \begin{gather} \label{ef-3}
      \begin{array}{c}
    \begin{tikzpicture}[thick,scale=0.3]
      \draw[black,<-] (0,-2) to (0,2);
      \draw[black,->] (1,-2) to (1,2);
      \node at (0,-2.5) {\tiny $i$};      
      \node at (1,-2.5) {\tiny $i$};
      \node[right] at (1.5,0) {\tiny $\omega - \a_i + \a_j$};
    \end{tikzpicture}
\end{array} 
=
      \begin{array}{c}
    \begin{tikzpicture}[thick,scale=0.3]
      \draw[black,<-] (-1,-2) to[out=90,in=180] (0,-0.4) to[out=0,in=90] (1,-2);
      \draw[black,->] (-1,2) to[out=-90,in=180] (0,0.4)
      to[out=0,in=-90] (1,2);
      \node at (0,0.4) {$\bullet$};
      \node at (0,-0.4) {$\bullet$};
      \node at (-1,2.5) {\tiny $i$};      
      \node at (1,-2.5) {\tiny $i$};
      \node[right] at (1.5,0) {\tiny $\omega + 2\a_i + \a_j$};
    \end{tikzpicture}
\end{array} 
+ R
  \end{gather}
where $R$ is a linear combination of terms of the form $
      \begin{array}{c}
    \begin{tikzpicture}[thick,scale=0.2]
      \draw[black,<-] (-1,-2) to[out=90,in=180] (0,-0.4)
      to[out=0,in=90] (1,-2);
      \draw (-1,2) rectangle (1,0);\node at (0,1) {\small $f$};
      to[out=0,in=-90] (1,2);
      \node at (1,-2.7) {\tiny $i$};
      \node[right] at (1.5,0) {\tiny $\omega + 2\a_i + \a_j$};
    \end{tikzpicture}
\end{array}$ and $
      \begin{array}{c}
    \begin{tikzpicture}[thick,scale=-0.2]
      \draw[black,<-] (-1,-2) to[out=90,in=180] (0,-0.4)
      to[out=0,in=90] (1,-2);
      \draw (-1,2) rectangle (1,0); \node at (0,1) {\small $f'$};
      to[out=0,in=-90] (1,2);
      \node at (1,-2.6) {\tiny $i$};
      \node[right] at (-1.5,0) {\tiny $\omega + 2\a_i + \a_j$};
    \end{tikzpicture}
\end{array}$ 
for some $f \in \Hom(\id, F_iE_i)$ and $f' \in \Hom(F_iE_i, \id)$.
  \end{enumerate}
\end{lem}

\begin{proof} 
  We remark that in $\mathcal{U}^{[n]}(\hgl_n)$ we have
\[
\begin{array}{c}
    \begin{tikzpicture}[thick,scale=0.3]
      \draw[black,->] (-1,-2) to[out=90,in=-90] (1,0) to[out=90,in=-90] (-1,2);
      \draw[black,<-] (1,-2) to[out=90,in=-90] (-1,0) to[out=90,in=-90] (1,2);
      \node at (-1,-2.5) {\tiny $i$};      
      \node at (1,-2.5) {\tiny $i$};
      \node[right] at (1.5,0) {\tiny $\omega + 2\a_i + \a_j$};
    \end{tikzpicture}
\end{array}
=0
\]
since this diagram contains a region labelled by $\omega + 3\a_i + \a_j$. Then~\eqref{ef3} follows from~\cite[(3.16)]{Brundan}. (In our case we have $\langle \omega + 2\a_i + \a_j, h_i \rangle = 3$.)

The proof of~\eqref{ef-3} is similar, using~\cite[(3.17)]{Brundan} instead of~\cite[(3.16)]{Brundan}.
%
\end{proof}

\begin{lem}
\label{crisscross}
  Suppose that $i, j \in \{1, \cdots, n-1, \infty\}$ and that $j \to i$. Then in $\mathcal{U}^{[n]}(\hgl_n)$ we have
\begin{equation}
\label{eq:crisscross}
\begin{array}{c}
 \begin{tikzpicture}[thick,baseline,xscale=0.4,yscale=0.6]
\draw[<-] (0,0) to [out=90,in=180] (1.5,1.3) to[out=0,in=90] (3,0);
\draw[->] (0,3) to [out=-90,in=180] (1.5,1.7) to[out=0,in=-90] (3,3);
\draw[->] (1,0) to [out=45,in=-90] (2,1.5) to [out=90,in=-45] (1,3);
\draw[<-] (2,0) to [out=135,in=-90] (1,1.5) to [out=90,in=-135] (2,3);
\node at (4,1.5) {\tiny $\omega + \a_j$};
\node at (3,-.3) {\tiny $i$};
\node at (1,-.3) {\tiny $i$};
\node at (0,3.3) {\tiny $i$};
\node at (2,3.3) {\tiny $i$};
  \end{tikzpicture}\end{array}
=
\begin{array}{c}
 \begin{tikzpicture}[thick,baseline,xscale=0.4,yscale=0.6]
\draw[<-] (0,0) to (2,3);
\draw[->] (1,0) to (3,3);
\draw[<-] (2,0) to (0,3);
\draw[->] (3,0) to (1,3);
\node at (4,1.5) {\tiny $\omega + \a_j$};
\node at (3,-.3) {\tiny $i$};
\node at (1,-.3) {\tiny $i$};
\node at (0,3.3) {\tiny $i$};
\node at (2,3.3) {\tiny $i$};
  \end{tikzpicture}
\end{array}.
\end{equation}
\end{lem}

\begin{proof}
By \eqref{ef3}, in $\mathcal{U}^{[n]}(\hgl_n)$ we have
\begin{gather*}
\begin{array}{c}
 \begin{tikzpicture}[thick,baseline,xscale=0.4, yscale=0.6]
\draw[<-] (0,0) to [out=90,in=-90] (1,1.5) to[out=90,in=-90] (0,3);
\draw[->] (3,0) to [out=90,in=-90] (2,1.5) to[out=90,in=-90] (3,3);
\draw[->] (1,0) to [out=45,in=-90] (3,1.5) to [out=90,in=-45] (1,3);
\draw[<-] (2,0) to [out=135,in=-90] (0,1.5) to [out=90,in=-135] (2,3);
\node at (4.5,1.5) {\tiny $\omega + \a_j$};
\node at (2,1.5) {$\bullet$};
\node at (1,1.5) {$\bullet$};
\node at (1,-.3) {\tiny $i$};
\node at (3,-.3) {\tiny $i$};
\node at (0,3.3) {\tiny $i$};
\node at (2,3.3) {\tiny $i$};
  \end{tikzpicture}\end{array}
=
\begin{array}{c}
 \begin{tikzpicture}[thick,baseline,xscale=0.4, yscale=0.6]
\draw[<-] (0,0) to [out=90,in=180] (1.5,1.3) to[out=0,in=90] (3,0);
\draw[->] (0,3) to [out=-90,in=180] (1.5,1.7) to[out=0,in=-90] (3,3);
\draw[->] (1,0) to [out=45,in=-90] (2,1.5) to [out=90,in=-45] (1,3);
\draw[<-] (2,0) to [out=135,in=-90] (1,1.5) to [out=90,in=-135] (2,3);
\node at (3.7,1.5) {\tiny $\omega + \a_j$};
\node at (1,-.3) {\tiny $i$};
\node at (3,-.3) {\tiny $i$};
\node at (0,3.3) {\tiny $i$};
\node at (2,3.3) {\tiny $i$};
  \end{tikzpicture}\end{array}.
\end{gather*}
(The terms denoted $R$ in~\eqref{ef3}  do not contribute by the first case in~\eqref{eqn:rel-2KM-2} and its ``reversed" counterpart~\cite[(2.7)]{Brundan}.) On the other hand, by the nil Hecke relations~\eqref{eqn:rel-2KM-1} and \eqref{eqn:rel-2KM-2} and their ``reversed'' counterparts (see~\cite[(2.6)--(2.7)]{Brundan}), in $\mathcal{U}(\hgl_n)$ we have:
\begin{gather*}
\begin{array}{c}
 \begin{tikzpicture}[thick,baseline,xscale=0.4, yscale=0.6]
\draw[<-] (0,0) to [out=90,in=-90] (1,1.5) to[out=90,in=-90] (0,3);
\draw[->] (3,0) to [out=90,in=-90] (2,1.5) to[out=90,in=-90] (3,3);
\draw[->] (1,0) to [out=45,in=-90] (3,1.5) to [out=90,in=-45] (1,3);
\draw[<-] (2,0) to [out=135,in=-90] (0,1.5) to [out=90,in=-135] (2,3);
\node at (4.5,1.5) {\tiny $\omega + \a_j$};
\node at (2,1.5) {$\bullet$};
\node at (1,1.5) {$\bullet$};
\node at (1,-.3) {\tiny $i$};
\node at (3,-.3) {\tiny $i$};
\node at (0,3.3) {\tiny $i$};
\node at (2,3.3) {\tiny $i$};
  \end{tikzpicture}\end{array}
=
\begin{array}{c}
 \begin{tikzpicture}[thick,baseline,xscale=0.4,yscale=0.6]
\draw[<-] (0,0) to (2,3);
\draw[->] (1,0) to (3,3);
\draw[<-] (2,0) to (0,3);
\draw[->] (3,0) to (1,3);
\node at (4,1.5) {\tiny $\omega + \a_j$};
\node at (3,-.3) {\tiny $i$};
\node at (1,-.3) {\tiny $i$};
\node at (0,3.3) {\tiny $i$};
\node at (2,3.3) {\tiny $i$};
  \end{tikzpicture}
\end{array}
\end{gather*} 
Combining these equalities gives~\eqref{eq:crisscross}.
\end{proof}

\begin{lem} \label{crisscross2}
  Suppose that $i, j \in \{1, \cdots, n-1, \infty\}$ and that $j \to i$. Then in $\mathcal{U}^{[n]}(\hgl_n)$ we have:
\begin{equation}
\label{eq:crisscross2}
\begin{array}{c} 
 \begin{tikzpicture}[thick,baseline,xscale=0.4,yscale=0.4]
\draw[->] (3,0) to[out=90,in=-30] (-0.3,2.8) to[out=150,in=-90] (-1,6);
\draw[->] (-3,6) to[out=-90,in=150] (0.3,3.2) to[out=-30,in=90] (1,0);
\draw[->] (-1,0) to[out=20,in=-90] (3,6);
\draw[->] (1,6) to[out=-160,in=90] (-3,0);
\node at (3,-.3) {\tiny $i$};
\node at (-1,-.3) {\tiny $i$};
\node at (1,6.3) {\tiny $i$};
\node at (-3,6.3) {\tiny $i$};
\node at (4.2,3) {\tiny $\omega + \a_j$};
  \end{tikzpicture}\end{array}
=
\begin{array}{c}
 \begin{tikzpicture}[thick,baseline,xscale=0.4,yscale=0.4]
\draw[->] (-3,3) to[out=-90,in=-90] (3,3);
\draw[<-] (-3,-3) to[out=90,in=90] (3,-3);
\draw[->] (-1,-3) to (-1,3);
\draw[<-] (1,-3) to (1,3);
\node at (3,-3.3) {\tiny $i$};
\node at (-1,-3.3) {\tiny $i$};
\node at (1,3.3) {\tiny $i$};
\node at (-3,3.3) {\tiny $i$};
\node at (3.3,0) {\tiny $\omega + \a_j$};
  \end{tikzpicture}\end{array} .
\end{equation}
\end{lem}

\begin{proof}
First we note that, by~\cite[(3.17)]{Brundan}, in $\mathcal{U}(\hgl_n)$ we have
\begin{equation}
 \label{ipullapart2}
  \begin{array}{c}
    \begin{tikzpicture}[thick,scale=0.4]
      \draw[black,<-] (-1,-2) to[out=90,in=-90] (1,0) to[out=90,in=-90] (-1,2);
      \draw[->] (1,-2) to[out=90,in=-90] (-1,0) to[out=90,in=-90] (1,2);
      \node at (-1,2.5) {\tiny $i$};      
      \node at (1,-2.5) {\tiny $i$};
      \node[right] at (1.7,0) {\tiny $\omega + \a_j$};
    \end{tikzpicture}
\end{array}
=
  \begin{array}{c}
    \begin{tikzpicture}[thick,scale=0.4]
      \draw[black,<-] (-1,-2) to[out=90,in=180] (0,-.5) to[out=0,in=90] (1,-2);
      \draw[<-] (1,2) to[out=-90,in=0] (0,.5) to[out=180,in=-90] (-1,2);
      \node at (-1,2.5) {\tiny $i$};      
      \node at (1,-2.5) {\tiny $i$};
      \node at (2.5,0) {\tiny $\omega + \a_j$};
    \end{tikzpicture}
\end{array}
-
  \begin{array}{c}
    \begin{tikzpicture}[thick,scale=0.4]
      \draw[black,<-] (0,-2) to  (0,2);
      \draw[->] (1,-2) to  (1,2);
      \node at (0,2.5) {\tiny $i$};      
      \node at (1,-2.5) {\tiny $i$};
      \node at (2.8,0) {\tiny $\omega + \a_j$};
    \end{tikzpicture}
\end{array}.
  \end{equation}
Using the first case in~\eqref{eqn:rel-2KM-2}, we deduce that
\begin{multline*}
\begin{array}{c}
 \begin{tikzpicture}[thick,baseline,scale=0.3]
\draw[<-] (0,0) to[out=90,in=-90] (0,4) to[out=90,in=-90] (2,6);
\draw[->] (2,0) to[out=90,in=-90] (2,4)
to[out=90,in=-90] (0,6);
\draw[->] (-2,0) to[out=90,in=-150] (1,4) to[out=30,in=-120] (4,6);
\node at (4.5,3) {\tiny $\omega + \alpha_j$};
\node at (-2, -.5) {\tiny $i$};
\node at (2, -.5) {\tiny $i$};
\node at (2,6.5) {\tiny $i$};
  \end{tikzpicture}\end{array}
\stackrel{\eqref{ipullapart2}}{=}
- \begin{array}{c}
 \begin{tikzpicture}[thick,baseline,scale=0.3]
\draw[<-] (0,0) to[out=90,in=-90] (2,2) to[out=90,in=-90] (0,4) to[out=90,in=-90] (2,6);
\draw[->] (2,0) to[out=90,in=-90] (0,2) to[out=90,in=-90] (2,4)
to[out=90,in=-90] (0,6);
\draw[->] (-2,0) to[out=90,in=-150] (1,4) to[out=30,in=-120] (4,6);
\node at (4.5,3) {\tiny $\omega + \alpha_j$};
\node at (-2, -.5) {\tiny $i$};
\node at (2, -.5) {\tiny $i$};
\node at (2,6.5) {\tiny $i$};
  \end{tikzpicture}\end{array}
\stackrel{(*)}{=}
- \begin{array}{c}
 \begin{tikzpicture}[thick,baseline,scale=0.3]
\draw[<-] (0,0) to[out=90,in=-90] (2,2) to[out=90,in=-90] (0,4) to[out=90,in=-90] (2,6);
\draw[->] (2,0) to[out=90,in=-90] (0,2) to[out=90,in=-90] (2,4)
to[out=90,in=-90] (0,6);
\draw[->] (-2,0) to[out=30,in=-150] (1,2) to[out=30,in=-90] (4,6);
\node at (4.5,3) {\tiny $\omega+\alpha_j$};
\node at (-2, -.5) {\tiny $i$};
\node at (2, -.5) {\tiny $i$};
\node at (2,6.5) {\tiny $i$};
  \end{tikzpicture}\end{array}
  \\
\stackrel{\eqref{ipullapart}}{=}
\begin{array}{c}
 \begin{tikzpicture}[thick,baseline,scale=0.3]
\draw[<-] (0,0) to[out=90,in=-90] (2,2) to[out=90,in=-90] (2,6);
\draw[->] (2,0) to[out=90,in=-90] (0,2) to[out=90,in=-90] (0,6);
\draw[->] (-2,0) to[out=30,in=-150] (1,2) to[out=30,in=-90] (4,6);
\node at (4.5,3) {\tiny $\omega+\alpha_j$};
\node at (-2, -.5) {\tiny $i$};
\node at (2, -.5) {\tiny $i$};
\node at (2,6.5) {\tiny $i$};
  \end{tikzpicture}\end{array},
\end{multline*}
where $(\ast)$ follows from~\cite[(2.4)]{Brundan}. Attaching a
rightward cap to the top right strand and a rightward cup to bottom
left strand and using~\cite[(1.6) \& (1.13)]{Brundan}, we also obtain
that 
\[
\begin{array}{c}
 \begin{tikzpicture}[thick,baseline,scale=0.3]
\draw[->] (0,0) to[out=90,in=-90] (2,2) to[out=90,in=-90] (2,6);
\draw[<-] (2,0) to[out=90,in=-90] (0,2) to[out=90,in=-90] (0,6);
\draw[<-] (-2,0) to[out=30,in=-150] (1,2) to[out=30,in=-90] (4,6);
\node at (5.5,3) {\tiny $\omega+\alpha_i + \alpha_j$};
\node at (-2, -.5) {\tiny $i$};
\node at (2, -.5) {\tiny $i$};
\node at (2,6.5) {\tiny $i$};
  \end{tikzpicture}\end{array}
=
\begin{array}{c}
 \begin{tikzpicture}[thick,baseline,scale=0.3]
\draw[->] (0,0) to[out=90,in=-90] (0,4) to[out=90,in=-90] (2,6);
\draw[<-] (2,0) to[out=90,in=-90] (2,4)
to[out=90,in=-90] (0,6);
\draw[<-] (-2,0) to[out=90,in=-150] (1,4) to[out=30,in=-120] (4,6);
\node at (5,3) {\tiny $\omega + \alpha_i + \alpha_j$};
\node at (-2, -.5) {\tiny $i$};
\node at (2, -.5) {\tiny $i$};
\node at (2,6.5) {\tiny $i$};
  \end{tikzpicture}\end{array}.
\]
Using these relations (and again~\eqref{eqn:rel-2KM-2} and its ``reversed'' counterpart~\cite[(2.7)]{Brundan}) we see that the left diagram in~\eqref{eq:crisscross2} (seen in $\mathcal{U}^{[n]}(\hgl_n)$) is equal to
\begin{multline*}
\begin{array}{c}
 \begin{tikzpicture}[thick,baseline,scale=0.3]
\draw[->] (3,0) to[out=90,in=-30] (-1,2) to[out=150,in=-90] (-1,6);
\draw[->] (-3,6) to[out=-90,in=150] (1,4) to[out=-30,in=90] (1,0);
\draw[->] (-1,0) to (3,6);
\draw[->] (1,6) to (-3,0);
\node at (3,-.4) {\tiny $i$};
\node at (-1,-.4) {\tiny $i$};
\node at (1,6.4) {\tiny $i$};
\node at (-3,6.4) {\tiny $i$};
\node at (3.5,3) {\tiny $\omega+\alpha_j$};
  \end{tikzpicture}\end{array}
\stackrel{\eqref{ef-3}}{=}
\begin{array}{c}
 \begin{tikzpicture}[thick,baseline,scale=0.3]
\draw[->] (3,0) to[out=90,in=-30] (-1,2) to[out=150,in=-90] (-1,6);
\draw[->] (-3,6) to[out=-90,in=150] (1,4) to[out=-30,in=90] (1,0);
\draw[->] (-1,0) to[out=60,in=0] (-0.5,2.7) to[out=180,in=60] (-3,0);
\draw[->] (1,6) to[out=-120,in=180] (0.5,3.3) to[out=0,in=-120] (3,6);
\node at (-0.5,2.7) {$\bullet$};
 \node at (0.5,3.3) {$\bullet$};
\node at (3,-.4) {\tiny $i$};
\node at (-1,-.4) {\tiny $i$};
\node at (1,6.4) {\tiny $i$};
\node at (-3,6.4) {\tiny $i$};
\node at (3.5,3) {\tiny $\omega+\alpha_j$};
  \end{tikzpicture}\end{array} 
=
\begin{array}{c}
 \begin{tikzpicture}[thick,baseline,scale=0.3]
\draw[->] (3,0) to[out=90,in=-30] (-1,2) to[out=150,in=-90] (-1,6);
\draw[->] (-3,6) to[out=-90,in=150] (1,4) to[out=-30,in=90] (1,0);
\draw[->] (-1,0) to[out=60,in=-60] (-0.5,2.7) to[out=120,in=0]
(-2,3.5) to[out=180,in=90] (-3,0);
\draw[->] (1,6) to[out=-120,in=120] (0.5,3.3) to[out=-60,in=180] (2,2.5) to[out=0,in=-90] (3,6);
\node at (-0.5,2.7) {$\bullet$};
 \node at (0.5,3.3) {$\bullet$};
\node at (3,-.4) {\tiny $i$};
\node at (-1,-.4) {\tiny $i$};
\node at (1,6.4) {\tiny $i$};
\node at (-3,6.4) {\tiny $i$};
\node at (4,2) {\tiny $\omega+\alpha_j$};
  \end{tikzpicture}\end{array} 
\\
=
\begin{array}{c}
 \begin{tikzpicture}[thick,baseline,scale=0.3]
\draw[->] (-3,3) to[out=-90,in=-90] (3,3);
\draw[<-] (-3,-3) to[out=90,in=90] (3,-3);
\draw[->] (-1,-3) to (-1,3);
\draw[<-] (1,-3) to (1,3);
\node at (3,-3.4) {\tiny $i$};
\node at (-1,-3.4) {\tiny $i$};
\node at (1,3.4) {\tiny $i$};
\node at (-3,3.4) {\tiny $i$};
\node at (3.5,0) {\tiny $\omega+\alpha_j$};
  \end{tikzpicture}\end{array} ,
\end{multline*}
where the last equality follows from~\eqref{eqn:rel-2KM-1} and \cite[(2.6)]{Brundan}.
This finishes the proof.
\end{proof}

We can finally prove~\eqref{boxrotate}.

\begin{proof}[Proof of~\eqref{boxrotate}]
We proceed as follows: we first simplify the picture for the box in~\eqref{boxrotate}, and then check the relation on this simplified picture.

\emph{Case 1: $j \to i$.} The left hand side of \eqref{2colassoc} is:
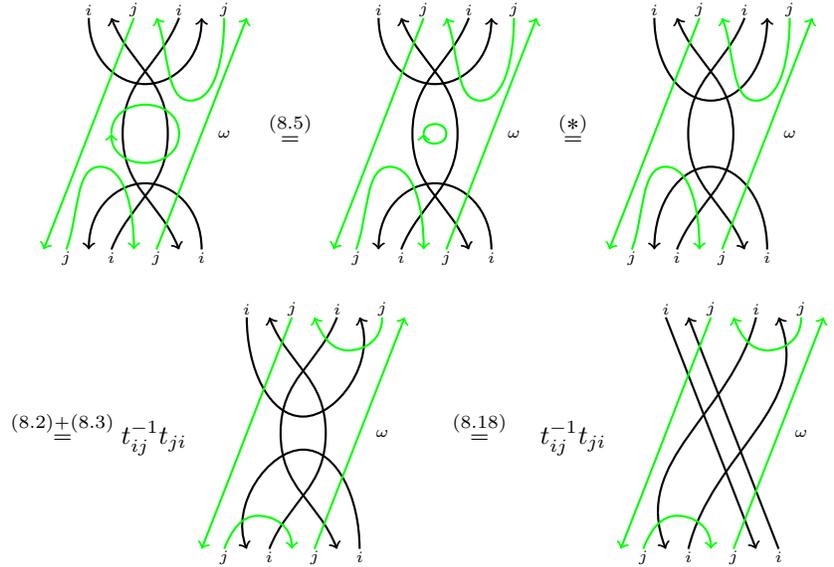
\begin{gather*}
  \begin{array}{c}
  \begin{tikzpicture}[xscale=0.15,yscale=0.22]
\draw[<-,thick] (-5,0) to[out=90,in=180] (0,4) to[out=0,in=90] (5,0);
\draw[->,thick] (-3,0) to[out=60,in=-90] (2,7) to[out=90,in=-60] (-3,14);
\draw[<-,thick] (3,0) to[out=120,in=-90] (-2,7) to[out=90,in=-120] (3,14);
\draw[->,thick] (-5,14) to[out=-90,in=180] (0,10) to[out=0,in=-90] (5,14);
\draw[green,<-,thick] (-9,0) to[out=60,in=-120] (-1,14);
\draw[green,->,thick] (-7,0) to[out=60,in=180] (-4,5) to[out=0,in=90] (-1,0);
\draw[green,<-,thick] (1,14) to[out=-60,in=180] (4,9) to[out=0,in=-90]
(7,14);
\draw[green,->,thick] (1,0) to (9,14);
\draw[green,->,thick] (-3,7) to[out=90,in=90] (3,7) to[out=-90,in=-90]
(-3,7);
\node (a) at (7,7) {\tiny $\omega$};
\node at (-7,-.5) {\tiny $j$};
\node at (-3,-.5) {\tiny $i$};
\node at (1,-.5) {\tiny $j$};
\node at (5,-.5) {\tiny $i$};
\node at (-5,14.5) {\tiny $i$};
\node at (-1,14.5) {\tiny $j$};
\node at (3,14.5) {\tiny $i$};
\node at (7,14.5) {\tiny $j$};
  \end{tikzpicture}
  \end{array}
\stackrel{\eqref{pullapart}}{=}
\begin{array}{c}
  \begin{tikzpicture}[xscale=0.15,yscale=0.22]
\draw[<-,thick] (-5,0) to[out=90,in=180] (0,4) to[out=0,in=90] (5,0);
\draw[->,thick] (-3,0) to[out=60,in=-90] (2,7) to[out=90,in=-60] (-3,14);
\draw[<-,thick] (3,0) to[out=120,in=-90] (-2,7) to[out=90,in=-120] (3,14);
\draw[->,thick] (-5,14) to[out=-90,in=180] (0,10) to[out=0,in=-90] (5,14);
\draw[green,<-,thick] (-9,0) to[out=60,in=-120] (-1,14);
\draw[green,->,thick] (-7,0) to[out=60,in=180] (-4,5) to[out=0,in=90] (-1,0);
\draw[green,<-,thick] (1,14) to[out=-60,in=180] (4,9) to[out=0,in=-90]
(7,14);
\draw[green,->,thick] (1,0) to (9,14);
\draw[green,->,thick] (-1,7) to[out=90,in=90] (1,7) to[out=-90,in=-90]
(-1,7);
\node (a) at (7,7) {\tiny $\omega$};
\node at (-7,-.5) {\tiny $j$};
\node at (-3,-.5) {\tiny $i$};
\node at (1,-.5) {\tiny $j$};
\node at (5,-.5) {\tiny $i$};
\node at (-5,14.5) {\tiny $i$};
\node at (-1,14.5) {\tiny $j$};
\node at (3,14.5) {\tiny $i$};
\node at (7,14.5) {\tiny $j$};
  \end{tikzpicture}
\end{array}
 \stackrel{(*)}{=}
\begin{array}{c}
  \begin{tikzpicture}[xscale=0.15,yscale=0.22]
\draw[<-,thick] (-5,0) to[out=90,in=180] (0,5) to[out=0,in=90] (5,0);
\draw[->,thick] (-3,0) to[out=60,in=-90] (2,7) to[out=90,in=-60] (-3,14);
\draw[<-,thick] (3,0) to[out=120,in=-90] (-2,7) to[out=90,in=-120] (3,14);
\draw[->,thick] (-5,14) to[out=-90,in=180] (0,9) to[out=0,in=-90] (5,14);
\draw[green,<-,thick] (-9,0) to[out=60,in=-120] (-1,14);
\draw[green,->,thick] (-7,0) to[out=60,in=180] (-4,5) to[out=0,in=90] (-1,0);
\draw[green,<-,thick] (1,14) to[out=-60,in=180] (4,9) to[out=0,in=-90]
(7,14);
\draw[green,->,thick] (1,0) to (9,14);
\node (a) at (7,7) {\tiny $\omega$};
\node at (-7,-.5) {\tiny $j$};
\node at (-3,-.5) {\tiny $i$};
\node at (1,-.5) {\tiny $j$};
\node at (5,-.5) {\tiny $i$};
\node at (-5,14.5) {\tiny $i$};
\node at (-1,14.5) {\tiny $j$};
\node at (3,14.5) {\tiny $i$};
\node at (7,14.5) {\tiny $j$};
  \end{tikzpicture}
\end{array}
\\
\stackrel{\eqref{signs:2} + \eqref{signs:3}}{=}  t_{ij}^{-1}t_{ji}
\begin{array}{c}
 \begin{tikzpicture}[xscale=0.15,yscale=0.22]
\draw[<-,thick] (-5,0) to[out=120,in=180] (0,6) to[out=0,in=90] (5,0);
\draw[->,thick] (-3,0) to[out=60,in=-90] (2,7) to[out=90,in=-60] (-3,14);
\draw[<-,thick] (3,0) to[out=120,in=-90] (-2,7) to[out=90,in=-120] (3,14);
\draw[->,thick] (-5,14) to[out=-90,in=180] (0,8) to[out=0,in=-60] (5,14);
\draw[green,<-,thick] (-9,0) to[out=60,in=-120] (-1,14);
\draw[green,->,thick] (-7,0) to[out=60,in=180] (-4,2) to[out=0,in=90] (-1,0);
\draw[green,<-,thick] (1,14) to[out=-60,in=180] (4,12) to[out=0,in=-90]
(7,14);
\draw[green,->,thick] (1,0) to (9,14);
\node (a) at (7,7) {\tiny $\omega$};
\node at (-7,-.5) {\tiny $j$};
\node at (-3,-.5) {\tiny $i$};
\node at (1,-.5) {\tiny $j$};
\node at (5,-.5) {\tiny $i$};
\node at (-5,14.5) {\tiny $i$};
\node at (-1,14.5) {\tiny $j$};
\node at (3,14.5) {\tiny $i$};
\node at (7,14.5) {\tiny $j$};
  \end{tikzpicture}
\end{array}
\quad \stackrel{\eqref{eq:crisscross}}{=} \quad t_{ij}^{-1}t_{ji}
\begin{array}{c}
 \begin{tikzpicture}[xscale=0.15,yscale=0.22]
\draw[<-,thick] (-5,0) to[out=120,in=-110] (3,14);
\draw[->,thick] (-3,0) to[out=70,in=-60] (5,14);
\draw[<-,thick] (3,0) to (-5,14); 
\draw[->,thick] (5,0) to (-3,14); 
\draw[green,<-,thick] (-9,0) to[out=60,in=-120] (-1,14);
\draw[green,->,thick] (-7,0) to[out=60,in=180] (-4,2) to[out=0,in=90] (-1,0);
\draw[green,<-,thick] (1,14) to[out=-60,in=180] (4,12) to[out=0,in=-90]
(7,14);
\draw[green,->,thick] (1,0) to (9,14);
\node (a) at (7,7) {\tiny $\omega$};
\node at (-7,-.5) {\tiny $j$};
\node at (-3,-.5) {\tiny $i$};
\node at (1,-.5) {\tiny $j$};
\node at (5,-.5) {\tiny $i$};
\node at (-5,14.5) {\tiny $i$};
\node at (-1,14.5) {\tiny $j$};
\node at (3,14.5) {\tiny $i$};
\node at (7,14.5) {\tiny $j$};
  \end{tikzpicture}
\end{array}.
\end{gather*}
Here, $(\ast)$ follows from~\cite[(1.20)]{Brundan}, which allows to omit the central bubble.

Using this description of the box in~\eqref{boxrotate},
it is a matter of bookkeeping to check that the relation holds. For instance, using the relations~\eqref{signs:1}--\eqref{signs:4} exactly once on each side (and the fact that any part of a diagram involving only one color is isotopy-invariant), one can check that both sides in~\eqref{boxrotate} are equal to
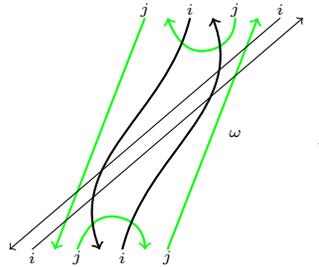
\[
\begin{array}{c}
\begin{tikzpicture}[xscale=0.15,yscale=0.22]
\draw[green,<-,thick] (-9,0) to[out=60,in=-120] (-1,14);
\draw[green,->,thick] (-7,0) to[out=60,in=180] (-4,2) to[out=0,in=90] (-1,0);
\draw[green,<-,thick] (1,14) to[out=-60,in=180] (4,12) to[out=0,in=-90]
(7,14);
\draw[green,->,thick] (1,0) to (9,14);
\draw[<-,thick] (-5,0) to[out=120,in=-110] (3,14);
\draw[->,thick] (-3,0) to[out=70,in=-60] (5,14);
\draw[->] (-11,0) to (13,14);
\draw[<-] (-13,0) to (11,14);
\node at (7,14.5) {\tiny $j$};
\node at (-1,14.5) {\tiny $j$};
\node at (1,-.5) {\tiny $j$};
\node at (-7,-.5) {\tiny $j$};
\node at (-11,-.5) {\tiny $i$};
\node at (11,14.5) {\tiny $i$};
\node at (-3,-.5) {\tiny $i$};
\node at (3,14.5) {\tiny $i$};
\node (a) at (7,7) {\tiny $\omega$};
\end{tikzpicture}
\end{array},
\]
hence that they coincide.

\emph{Case 2: $i \to j$.} The left hand side
of \eqref{2colassoc} is
\begin{multline*}
  \begin{array}{c}
    \begin{tikzpicture}[thick,xscale=0.15,yscale=0.22]
\draw[<-] (-5,-7) to[out=90,in=-90] (-1,0) to [out=90,in=-90] (-5,7);      
\draw[->] (5,-7) to[out=90,in=-90] (1,0) to [out=90,in=-90] (5,7);
\draw[->] (-3,-7) to[out=90,in=180] (0,-4) to [out=0,in=90] (3,-7);
\draw[<-] (-3,7) to[out=-90,in=180] (0,4) to [out=0,in=-90] (3,7);
\draw[green,<-] (-1,-7) to[out=90,in=-150] (0,-2.5) to[out=30,in=-90]
(3,0) to[out=90,in=-30] (0,2.5) to[out=150,in=-90] (-1,7); 
\draw[green,->] (1,-7) to[out=90,in=-30] (0,-2.5) to[out=150,in=-90]
(-3,0) to[out=90,in=-150] (0,2.5) to[out=30,in=-90] (1,7);
\draw[green,->] (-7,-7) to[out=90,in=180] (-3.5,-2) to[out=0,in=180]
(0,-3.5) to[out=0,in=-90] (5,0) to[out=90,in=-90] (9,7);
\draw[green,->] (7,7) to[out=-90,in=0] (3.5,2) to[out=180,in=0]
(0,3.5) to[out=180,in=90] (-5,0) to[out=-90,in=90] (-9,-7);
\node (a) at (7,0) {\tiny $\omega$};
\node at (-7,-7.5) {\tiny $j$};
\node at (-3,-7.5) {\tiny $i$};
\node at (1,-7.5) {\tiny $j$};
\node at (5,-7.5) {\tiny $i$};
\node at (-5,7.5) {\tiny $i$};
\node at (-1,7.5) {\tiny $j$};
\node at (3,7.5) {\tiny $i$};
\node at (7,7.5) {\tiny $j$};
    \end{tikzpicture}
   \end{array}
\quad \stackrel{\eqref{pullapart} + \eqref{signs:2} + \eqref{signs:3}}= \quad t_{ji}t_{ij}^{-1} 
  \begin{array}{c}
    \begin{tikzpicture}[thick,xscale=0.15,yscale=0.22]
\draw[<-] (-5,-7) to[out=90,in=-90] (-7,0) to [out=90,in=-90] (-5,7);      
\draw[->] (5,-7) to[out=90,in=-90] (7,0) to [out=90,in=-90] (5,7);
\draw[->] (-3,-7) to[out=90,in=180] (0,-4) to [out=0,in=90] (3,-7);
\draw[<-] (-3,7) to[out=-90,in=180] (0,4) to [out=0,in=-90] (3,7);
\draw[green,<-] (-1,-7) to[out=90,in=-150] (0,-2.5) to[out=30,in=-90]
(3,0) to[out=90,in=-30] (0,2.5) to[out=150,in=-90] (-1,7); 
\draw[green,->] (1,-7) to[out=90,in=-30] (0,-2.5) to[out=150,in=-90]
(-3,0) to[out=90,in=-150] (0,2.5) to[out=30,in=-90] (1,7);
\draw[green,->] (-7,-7) to[out=90,in=180] (-3.5,-2) to[out=0,in=180]
(0,-3.5) to[out=0,in=-90] (5,0) to[out=90,in=-90] (9,7);
\draw[green,->] (7,7) to[out=-90,in=0] (3.5,2) to[out=180,in=0]
(0,3.5) to[out=180,in=90] (-5,0) to[out=-90,in=90] (-9,-7);
\node (a) at (9,0) {\tiny $\omega$};
\node at (-7,-7.5) {\tiny $j$};
\node at (-3,-7.5) {\tiny $i$};
\node at (1,-7.5) {\tiny $j$};
\node at (5,-7.5) {\tiny $i$};
\node at (-5,7.5) {\tiny $i$};
\node at (-1,7.5) {\tiny $j$};
\node at (3,7.5) {\tiny $i$};
\node at (7,7.5) {\tiny $j$};
    \end{tikzpicture}
   \end{array}
\\
\quad \stackrel{\eqref{eq:crisscross2}}{=} \quad t_{ji}t_{ij}^{-1} 
  \begin{array}{c}
    \begin{tikzpicture}[thick,xscale=0.15,yscale=0.22]
\draw[<-] (-5,-7) to[out=90,in=-90] (-7,0) to [out=90,in=-90] (-5,7);      
\draw[->] (5,-7) to[out=90,in=-90] (7,0) to [out=90,in=-90] (5,7);
\draw[->] (-3,-7) to[out=90,in=180] (0,-4) to [out=0,in=90] (3,-7);
\draw[<-] (-3,7) to[out=-90,in=180] (0,4) to [out=0,in=-90] (3,7);
\draw[green,->] (-1,7) to[out=-90,in=180] (3,1) to[out=0,in=-90] (9,7); 
\draw[green,->] (1,-7) to[out=90,in=0] (-3,-1) to[out=180,in=90] (-9,-7); 
\draw[green,->] (-7,-7) to[out=90,in=-90] (1,7);
\draw[green,->] (7,7) to[out=-90,in=90] (-1,-7);
\node (a) at (9,0) {\tiny $\omega$};
\node at (-7,-7.5) {\tiny $j$};
\node at (-3,-7.5) {\tiny $i$};
\node at (1,-7.5) {\tiny $j$};
\node at (5,-7.5) {\tiny $i$};
\node at (-5,7.5) {\tiny $i$};
\node at (-1,7.5) {\tiny $j$};
\node at (3,7.5) {\tiny $i$};
\node at (7,7.5) {\tiny $j$};
    \end{tikzpicture}
   \end{array}.
\end{multline*}
(Here, the second equality again also uses the fact that the central part of the diagram involving only $j$ is isotopy-invariant.) Using this expression for the box in~\eqref{boxrotate},
it is again bookkeeping to check that the relation holds.
\end{proof}

\subsection{Zamolodchikov (three colour) relations}

The Zamolodchikov relations are associated with triples $(i,j,k)$ of pairwise distinct elements of $\{1, \cdots, n-1, \infty\}$, see~\cite[(5.8)--(5.12)]{ew}. The precise form of the relation depends on the type of the parabolic subgroup of $W$ generated by $s_i$, $s_j$ and $s_k$. The cases when this subgroup is of type $\mathbf{A}_1 \times \mathbf{A}_1 \times \mathbf{A}_1$ or $\mathbf{A}_2 \times \mathbf{A}_1$ are easy, and left to the reader. (In fact, in these cases, the relation follows from the observation that in $\mathcal{U}(\hgl_n)$ one can pull a strand labelled $i$ through any crossing of strands labelled $j$ and $k$ if $i \norel j$, $i \neq j$, $i \norel k$ and $i \neq k$; see Lemma~\ref{lem:reid} below and its proof.)

Now we consider the case when the subgroup is of type $\mathbf{A}_3$. To fix notation, we assume that $i \rel j \rel k$ and $i \neq k$. In this case, using cyclicity, the relation (stated in~\cite[(5.10)]{ew}) can be equivalently formulated as the following equality:
\begin{gather}
\label{zam}
  \begin{array}{c}
  \begin{tikzpicture}[thick,scale=-0.4]
  \node[inner sep=0mm,minimum size=0mm] (n) at (-1,4) {};
  \node[inner sep=0mm,minimum size=0mm] (s) at (1,-4) {};
  \node[inner sep=0mm,minimum size=0mm] (e) at (3,2) {};
 \node[inner sep=0mm,minimum size=0mm] (w) at (-3,-2) {};
 \draw (-1,-6) to (w) to (e) to (5,-6); \draw (-5,6) to (w); \draw (1,6) to (e);
 \draw[green] (3,-6) to (s) to (n) to (3,6); \draw[green] (-3,-6) to (s); \draw[green] (-3,6) to (n);
 \draw[red] (-5,-6) to (w) to (n) to (-1,6); \draw[red] (s) to (w);
 \draw[red] (1,-6) to (s) to (e) to (5,6); \draw[red] (e) to (n);
 \node at (-5,-6.5) {\tiny $j$};
 \node at (-5,6.5) {\tiny $i$};
 \node at (-3,-6.5) {\tiny $k$};
  \node at (-3,6.5) {\tiny $k$};
  \node at (-1,-6.5) {\tiny $i$};
  \node at (-1,6.5) {\tiny $j$};
  \node at (1,-6.5) {\tiny $j$};
  \node at (1,6.5) {\tiny $i$};
  \node at (3,-6.5) {\tiny $k$};
  \node at (3,6.5) {\tiny $k$};
  \node at (5,-6.5) {\tiny $i$};
  \node at (5,6.5) {\tiny $j$};
\end{tikzpicture}
\end{array}
=
  \begin{array}{c}
  \begin{tikzpicture}[thick,scale=0.4]
  \node[inner sep=0mm,minimum size=0mm] (n) at (1,3) {};
  \node[inner sep=0mm,minimum size=0mm] (s) at (-1,-3) {};
  \node[inner sep=0mm,minimum size=0mm] (e) at (3,-2) {};
 \node[inner sep=0mm,minimum size=0mm] (w) at (-3,2) {};
 \draw (-1,-6) to (s) to (n) to (1,6); \draw (5,-6) to (s); \draw (-5,6) to (n);
 \draw[green] (-3,-6) to (w) to (e) to (3,6); \draw[green] (3,-6) to (e); \draw[green] (-3,6) to (w);
 \draw[red] (-5,-6) to[out=90,in=120] (w) to (n) to (-1,6); \draw[red] (s) to (w);
 \draw[red] (1,-6) to (s) to (e) to[out=-60,in=-90] (5,6); \draw[red] (e) to (n);
  \node at (-5,-6.5) {\tiny $j$};
 \node at (-5,6.5) {\tiny $i$};
 \node at (-3,-6.5) {\tiny $k$};
  \node at (-3,6.5) {\tiny $k$};
  \node at (-1,-6.5) {\tiny $i$};
  \node at (-1,6.5) {\tiny $j$};
  \node at (1,-6.5) {\tiny $j$};
  \node at (1,6.5) {\tiny $i$};
  \node at (3,-6.5) {\tiny $k$};
  \node at (3,6.5) {\tiny $k$};
  \node at (5,-6.5) {\tiny $i$};
  \node at (5,6.5) {\tiny $j$};
\end{tikzpicture}
\end{array} .
\end{gather}
It can be easily checked that if this relation is known for a triple $(i,j,k)$ as above, then the relation follows for the triple $(k,j,i)$. Hence it suffices to consider the case $k \to j \to i$.
Moreover,
if we denote the left hand side in~\eqref{zam} by $\begin{array}{c} \tikz[scale=0.2]{ \draw (0,0) rectangle (4,4); \node at (2,2) {$Z$};} \end{array}$, then we can restate this relation as the following equality:
\begin{gather} \label{zam2}
  \begin{array}{c} \begin{tikzpicture}[scale=0.3] \draw (-3,-3) rectangle (3,3); \node at (0,0) {$Z$};
\draw(-2.5,3) to (-2.5,6);
\draw[green](-1.5,3) to (-1.5,6);
\draw[red](-0.5,3) to (-0.5,6);
\draw[red] (2.5,3) to (2.5,6);
\draw[green](1.5,3) to (1.5,6);
\draw(0.5,3) to (0.5,6);
\node at (-2.5, -6.5) {\tiny $j$};
\node at (-2.5,6.5) {\tiny $i$};
\node at (-1.5,-6.5) {\tiny $k$};
\node at (-1.5,6.5) {\tiny $k$};
\node at (-.5,-6.5) {\tiny $i$};
\node at (-.5,6.5) {\tiny $j$};
\node at (.5,-6.5) {\tiny $j$};
\node at (.5,6.5) {\tiny $i$};
\node at (1.5,-6.5) {\tiny $k$};
\node at (1.5,6.5) {\tiny $k$};
\node at (2.5,-6.5) {\tiny $i$};
\node at (2.5,6.5) {\tiny $j$};
\draw (2.5,-3) to (2.5,-6);
\draw[green](1.5,-3) to (1.5,-6);
\draw[red](0.5,-3) to (0.5,-6);
\draw[red] (-2.5,-3) to (-2.5,-6);
\draw[green](-1.5,-3) to (-1.5,-6);
\draw (-0.5,-3) to (-0.5,-6);
 \end{tikzpicture} \end{array}
=
  \begin{array}{c} \begin{tikzpicture}[xscale=-0.3,yscale=0.3] \draw (-3,-3) rectangle (3,3); \node at (0,0) {$Z$};
\draw[red] (-2.5,3) to (-2.5,6);
\draw[green](-1.5,3) to (-1.5,6);
\draw (-0.5,3) to (-0.5,6);
\draw[red] (0.5,3) to[out=90,in=180] (3,5.5) to[out=0,in=90] (5.5,-6);
\draw[green] (1.5,3) to[out=90,in=180] (3,4.5) to[out=0,in=90] (4.5,-6);
\draw (2.5,3) to[out=90,in=180] (3,3.5) to[out=0,in=90] (3.5,-6);
\draw[red] (2.5,-3) to (2.5,-6);
\draw[green](1.5,-3) to (1.5,-6);
\draw (0.5,-3) to (0.5,-6);
\draw[red] (-0.5,-3) to[out=-90,in=0] (-3,-5.5) to[out=180,in=-90] (-5.5,6);
\draw[green] (-1.5,-3) to[out=-90,in=0] (-3,-4.5) to[out=180,in=-90] (-4.5,6);
\draw (-2.5,-3) to[out=-90,in=0] (-3,-3.5) to[out=180,in=-90] (-3.5,6);
\node at (5.5, -6.5) {\tiny $j$};
\node at (-.5,6.5) {\tiny $i$};
\node at (4.5,-6.5) {\tiny $k$};
\node at (-1.5,6.5) {\tiny $k$};
\node at (3.5,-6.5) {\tiny $i$};
\node at (-2.5,6.5) {\tiny $j$};
\node at (2.5,-6.5) {\tiny $j$};
\node at (-3.5,6.5) {\tiny $i$};
\node at (1.5,-6.5) {\tiny $k$};
\node at (-4.5,6.5) {\tiny $k$};
\node at (.5,-6.5) {\tiny $i$};
\node at (-5.5,6.5) {\tiny $j$};
\end{tikzpicture} 
\end{array}.
\end{gather}
This is the relation that we will check below (in the case $k \to j \to i$); this will finish the proof of Theorem~\ref{thm:2KM-SBim}.

We start with some preliminary lemmas. In the first statement, $i,j,k$ are arbitrary elements of $\{1, \cdots, n-1, \infty\}$.

\begin{lem}
\label{lem:reid}
For $i,j,k \in \{1, \cdots, n-1, \infty\}$, in $\mathcal{U}(\hgl_n)$ we have
\begin{align}
  \label{reid}
  \begin{array}{c}
    \begin{tikzpicture}[thick,scale=0.3]
      \draw[black,->] (-2,-2) to (2,2);
      \draw[green,->] (2,-2) to (-2,2);
      \draw[red,->] (0,-2) to[out=90,in=-90] (2,0) to[out=90,in=-90] (0,2);
      \node at (-2,-2.5) {\tiny $i$}; \node at (0,-2.5) {\tiny $j$}; \node at (2,-2.5) {\tiny $k$};
      \node at (2.7,0) {\tiny $\lambda$};
    \end{tikzpicture}
\end{array}
&=     \begin{array}{c}\begin{tikzpicture}[thick,scale=0.3]
      \draw[black,->] (-2,-2) to (2,2);
      \draw[green,->] (2,-2) to (-2,2);
      \draw[red,->] (0,-2) to[out=90,in=-90] (-2,0) to[out=90,in=-90] (0,2);
      \node at (-2,-2.5) {\tiny $i$}; \node at (0,-2.5) {\tiny $j$}; \node at (2,-2.5) {\tiny $k$};
      \node at (2,0) {\tiny $\lambda$};
    \end{tikzpicture} 
\end{array}\quad \text{unless $i = k$ and $i \rel j$;}
\\
  \label{reid-reversed}
  \begin{array}{c}
    \begin{tikzpicture}[thick,scale=0.3]
      \draw[black,<-] (-2,-2) to (2,2);
      \draw[green,<-] (2,-2) to (-2,2);
      \draw[red,<-] (0,-2) to[out=90,in=-90] (2,0) to[out=90,in=-90] (0,2);
      \node at (-2,2.5) {\tiny $i$}; \node at (0,2.5) {\tiny $j$}; \node at (2,2.5) {\tiny $k$};
      \node at (2.7,0) {\tiny $\lambda$};
    \end{tikzpicture}
\end{array}
&=     \begin{array}{c}\begin{tikzpicture}[thick,scale=0.3]
      \draw[black,<-] (-2,-2) to (2,2);
      \draw[green,<-] (2,-2) to (-2,2);
      \draw[red,<-] (0,-2) to[out=90,in=-90] (-2,0) to[out=90,in=-90] (0,2);
      \node at (-2,2.5) {\tiny $i$}; \node at (0,2.5) {\tiny $j$}; \node at (2,2.5) {\tiny $k$};
      \node at (2,0) {\tiny $\lambda$};
    \end{tikzpicture} 
\end{array}\quad \text{unless $i = k$ and $i \rel j$;}
\\
  \label{reid'}
  \begin{array}{c}
    \begin{tikzpicture}[thick,scale=0.3]
      \draw[black,->] (-2,-2) to (2,2);
      \draw[green,<-] (2,-2) to (-2,2);
      \draw[red,->] (0,-2) to[out=90,in=-90] (2,0) to[out=90,in=-90] (0,2);
      \node at (-2,-2.5) {\tiny $i$}; \node at (0,-2.5) {\tiny $j$}; \node at (-2,2.5) {\tiny $k$};
      \node at (2.7,0) {\tiny $\lambda$};
    \end{tikzpicture}
\end{array}
&=     \begin{array}{c}\begin{tikzpicture}[thick,scale=0.3]
      \draw[black,->] (-2,-2) to (2,2);
      \draw[green,<-] (2,-2) to (-2,2);
      \draw[red,->] (0,-2) to[out=90,in=-90] (-2,0) to[out=90,in=-90] (0,2);
      \node at (-2,-2.5) {\tiny $i$}; \node at (0,-2.5) {\tiny $j$}; \node at (-2,2.5) {\tiny $k$};
      \node at (2,0) {\tiny $\lambda$};
    \end{tikzpicture} 
\end{array}\quad \text{unless $j = k$ and $k \rel i$;}
\\
  \label{reid2}
  \begin{array}{c}
    \begin{tikzpicture}[thick,scale=0.3]
      \draw[black,->] (-2,-2) to (2,2);
      \draw[green,->] (2,-2) to (-2,2);
      \draw[red,<-] (0,-2) to[out=90,in=-90] (2,0) to[out=90,in=-90] (0,2);
      \node at (-2,-2.5) {\tiny $i$}; \node at (0,2.5) {\tiny $j$}; \node at (2,-2.5) {\tiny $k$};
      \node at (2.7,0) {\tiny $\lambda$};
    \end{tikzpicture}
\end{array}
&=     \pm \begin{array}{c}\begin{tikzpicture}[thick,scale=0.3]
      \draw[black,->] (-2,-2) to (2,2);
      \draw[green,->] (2,-2) to (-2,2);
      \draw[red,<-] (0,-2) to[out=90,in=-90] (-2,0) to[out=90,in=-90] (0,2);
      \node at (-2,-2.5) {\tiny $i$}; \node at (0,2.5) {\tiny $j$}; \node at (2,-2.5) {\tiny $k$};
      \node at (2,0) {\tiny $\lambda$};
    \end{tikzpicture} 
\end{array}\quad \text{unless $i = j = k$;}
\\
  \label{reid3}
  \begin{array}{c}
    \begin{tikzpicture}[thick,xscale=0.3,yscale=-0.3]
      \draw[black,->] (-2,-2) to (2,2);
      \draw[green,->] (2,-2) to (-2,2);
      \draw[red,<-] (0,-2) to[out=90,in=-90] (2,0) to[out=90,in=-90] (0,2);
      \node at (-2,-2.5) {\tiny $i$}; \node at (0,2.5) {\tiny $j$}; \node at (2,-2.5) {\tiny $k$};
      \node at (2.7,0) {\tiny $\lambda$};
    \end{tikzpicture}
\end{array}
&=     \pm \begin{array}{c}\begin{tikzpicture}[thick,xscale=0.3,yscale=-0.3]
      \draw[black,->] (-2,-2) to (2,2);
      \draw[green,->] (2,-2) to (-2,2);
      \draw[red,<-] (0,-2) to[out=90,in=-90] (-2,0) to[out=90,in=-90] (0,2);
      \node at (-2,-2.5) {\tiny $i$}; \node at (0,2.5) {\tiny $j$}; \node at (2,-2.5) {\tiny $k$};
      \node at (2,0) {\tiny $\lambda$};
    \end{tikzpicture} 
\end{array}\quad \text{unless $i = j = k$;}
\\
  \label{reid4}
  \begin{array}{c}
    \begin{tikzpicture}[thick,scale=0.3]
      \draw[black,<-] (-2,-2) to (2,2);
      \draw[green,->] (2,-2) to (-2,2);
      \draw[red,->] (0,-2) to[out=90,in=-90] (2,0) to[out=90,in=-90] (0,2);
      \node at (2,2.5) {\tiny $i$}; \node at (0,-2.5) {\tiny $j$}; \node at (2,-2.5) {\tiny $k$};
      \node at (2.7,0) {\tiny $\lambda$};
    \end{tikzpicture}
\end{array}
&=  \pm   \begin{array}{c}\begin{tikzpicture}[thick,scale=0.3]
      \draw[black,<-] (-2,-2) to (2,2);
      \draw[green,->] (2,-2) to (-2,2);
      \draw[red,->] (0,-2) to[out=90,in=-90] (-2,0) to[out=90,in=-90] (0,2);
      \node at (2,2.5) {\tiny $i$}; \node at (0,-2.5) {\tiny $j$}; \node at (2,-2.5) {\tiny $k$};
      \node at (2,0) {\tiny $\lambda$};
    \end{tikzpicture} 
\end{array}\quad \text{unless $i = j$ and $j \rel k$;}
\\
  \label{reid5}
  \begin{array}{c}
    \begin{tikzpicture}[thick,scale=0.3]
      \draw[black,->] (-2,-2) to (2,2);
      \draw[green,<-] (2,-2) to (-2,2);
      \draw[red,<-] (0,-2) to[out=90,in=-90] (2,0) to[out=90,in=-90] (0,2);
      \node at (-2,-2.5) {\tiny $i$}; \node at (0,2.5) {\tiny $j$}; \node at (2,-2.5) {\tiny $k$};
      \node at (2.7,0) {\tiny $\lambda$};
    \end{tikzpicture}
\end{array}
&=  \pm   \begin{array}{c}\begin{tikzpicture}[thick,scale=0.3]
      \draw[black,->] (-2,-2) to (2,2);
      \draw[green,<-] (2,-2) to (-2,2);
      \draw[red,<-] (0,-2) to[out=90,in=-90] (-2,0) to[out=90,in=-90] (0,2);
      \node at (-2,-2.5) {\tiny $i$}; \node at (0,2.5) {\tiny $j$}; \node at (2,-2.5) {\tiny $k$};
      \node at (2,0) {\tiny $\lambda$};
    \end{tikzpicture} 
\end{array}\quad \text{unless $i = j$ and $j \rel k$;}
\\
  \label{reid6}
  \begin{array}{c}
    \begin{tikzpicture}[thick,scale=0.3]
      \draw[black,<-] (-2,-2) to (2,2);
      \draw[green,->] (2,-2) to (-2,2);
      \draw[red,<-] (0,-2) to[out=90,in=-90] (2,0) to[out=90,in=-90] (0,2);
      \node at (2,2.5) {\tiny $i$}; \node at (0,2.5) {\tiny $j$}; \node at (2,-2.5) {\tiny $k$};
      \node at (2.7,0) {\tiny $\lambda$};
    \end{tikzpicture}
\end{array}
&=  \pm   \begin{array}{c}\begin{tikzpicture}[thick,scale=0.3]
      \draw[black,<-] (-2,-2) to (2,2);
      \draw[green,->] (2,-2) to (-2,2);
      \draw[red,<-] (0,-2) to[out=90,in=-90] (-2,0) to[out=90,in=-90] (0,2);
      \node at (2,2.5) {\tiny $i$}; \node at (0,2.5) {\tiny $j$}; \node at (2,-2.5) {\tiny $k$};
      \node at (2,0) {\tiny $\lambda$};
    \end{tikzpicture} 
\end{array}\quad \text{unless $j = k$ and $k \rel i$.}
\end{align}
\end{lem}

\begin{proof}
Equality~\eqref{reid} is a special case of~\eqref{eqn:rel-2KM-3},~\eqref{reid-reversed} is a special case of~\cite[2.8]{Brundan} and~\eqref{reid'} is a special case of~\cite[(2.4)]{Brundan}. Equality~\eqref{reid4} can be obtained from~\eqref{reid} (applied to strands $j$, $k$ an $i$ from left to right) by adding a leftward cup to the right and a leftward cap to the left. Equalities~\eqref{reid5} and~\eqref{reid6} can be deduced from~\eqref{reid-reversed} by a similar procedure.

To prove~\eqref{reid2}, we distinguish two cases.
First, assume that $j \neq k$. Then
\begin{multline*}
\begin{array}{c}
    \begin{tikzpicture}[thick,xscale=0.2,yscale=0.4]
      \draw[black,->] (-2,-2) to (2,2);
      \draw[green,->] (2,-2) to (-2,2);
      \draw[red,<-] (0,-2) to[out=90,in=-90] (2,0) to[out=90,in=-90] (0,2);
      \node at (-2,-2.5) {\tiny $i$}; \node at (0,-2.5) {\tiny $j$}; \node at (2,-2.5) {\tiny $k$};
      \node at (2.7,0) {\tiny $\lambda$};
    \end{tikzpicture}
\end{array}
\stackrel{\eqref{pullapart}}{=}
  \begin{array}{c}
    \begin{tikzpicture}[thick,scale=0.2]
      \draw[black,->] (-2,-4) to (2,4);
      \draw[green,->] (2,-4) to (-2,4);
      \draw[red,<-] (0,-4) to[out=90,in=-90] (2,-1) to[out=90,in=-90] (-2,2.5) to[out=90,in=-130] (0,4);
      \node at (-2,-4.5) {\tiny $i$}; \node at (0,-4.5) {\tiny $j$}; \node at (2,-4.5) {\tiny $k$};
      \node at (2.7,0) {\tiny $\lambda$};
    \end{tikzpicture}
    \end{array}
=
      \begin{array}{c}
    \begin{tikzpicture}[thick,scale=0.2]
      \draw[black,->] (-2,-4) to (2,4);
      \draw[green,->] (2,-4) to (-2,4);
      \draw[red,<-] (0,-4) to[out=90,in=-90] (2,-1) to[out=90,in=0] (0,1.5) to[out=180,in=-90] (-2,1) to[out=90,in=-130] (0,4);
      \node at (-2,-4.5) {\tiny $i$}; \node at (0,-4.5) {\tiny $j$}; \node at (2,-4.5) {\tiny $k$};
      \node at (2.7,0) {\tiny $\lambda$};
    \end{tikzpicture}
\end{array}
= \pm
      \begin{array}{c}
    \begin{tikzpicture}[thick,scale=0.2]
      \draw[black,->] (-2,-4) to (2,4);
      \draw[green,->] (2,-4) to (-2,4);
      \draw[red,<-] (0,-4) to[out=90,in=-90] (2.5,2.5) to[out=90,in=-90] (-2,0) to[out=90,in=-130] (0,4);
      \node at (-2,-4.5) {\tiny $i$}; \node at (0,-4.5) {\tiny $j$}; \node at (2,-4.5) {\tiny $k$};
      \node at (2.7,0) {\tiny $\lambda$};
    \end{tikzpicture}
\end{array}
\\
\stackrel{\eqref{reid}}{=} \pm
      \begin{array}{c}
    \begin{tikzpicture}[thick,scale=0.2]
      \draw[black,->] (-2,-4) to (2,4);
      \draw[green,->] (2,-4) to (-2,4);
      \draw[red,<-] (0,-4) to[out=90,in=-90] (2,-.5) to[out=90,in=0] (-2,-2) to[out=135,in=-130] (0,4);
      \node at (-2,-4.5) {\tiny $i$}; \node at (0,-4.5) {\tiny $j$}; \node at (2,-4.5) {\tiny $k$};
      \node at (2.7,0) {\tiny $\lambda$};
    \end{tikzpicture}
\end{array}%
\end{multline*}
Here 
we have used the zigzag relations~\cite[(1.5),(4.8)]{Brundan} and the
sign relations of Lemma~\ref{signrules} multiple times.

Now, assume that $j \neq i$. Then, by similar arguments, we have
\begin{multline*}
  %
\end{multline*}

Equality~\eqref{reid3} can be proved similarly, starting from~\cite[(2.8)]{Brundan} instead of~\eqref{eqn:rel-2KM-3}; details are left to the reader.
\end{proof}

\begin{lem}
\label{lem:rel-zam}
Assume that $k \to j \to i$. Then in $\mathcal{U}(\hgl_n)$ we have
  \begin{equation}
    %
  \end{equation}
\end{lem}

\begin{proof}
Using~\cite[(3.17)]{Brundan}, then~\eqref{eqn:rel-2KM-3} (and its reversed counterpart~\cite[(2.8)]{Brundan}) and finally~\cite[(3.16)]{Brundan} on both sides we see that
  \begin{multline*}
    %
 .
  \end{multline*}
\end{proof}

\begin{proof}[Proof of~\eqref{zam2} in the case $k \to j \to i$]
We proceed as in the proof of the ``2 color associativity" relation: namely we first simplify the diagram for the box in~\eqref{zam2}, and then check the relation with this simplified diagram.

The manipulations for this simplification are described in Figure~\ref{fig:zam}.
\begin{figure}
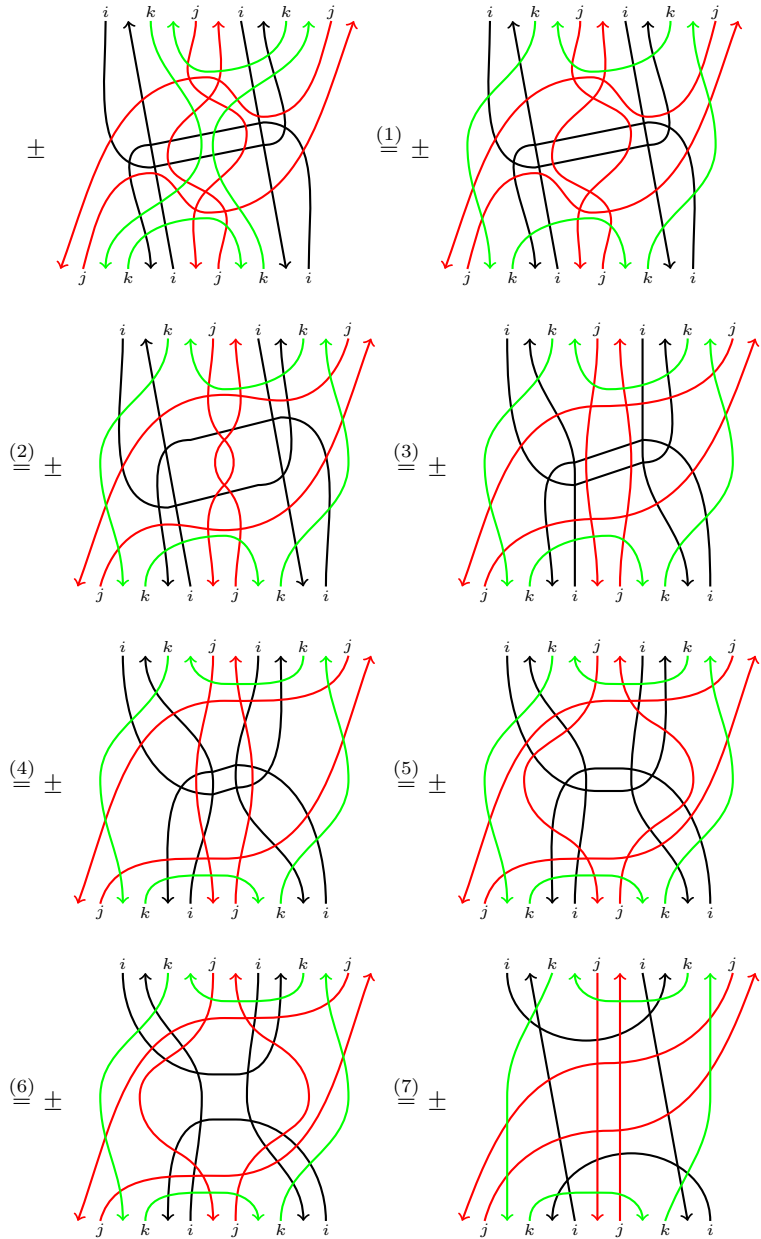

\begin{gather*}
 \pm %
\end{gather*}
\caption{Manipulations for the proof of~\eqref{zam2}.}\label{fig:zam}
\end{figure}
In these manipulations we work ``up to sign'', so that we do not need
to be careful with the sign relations in Lemma~\ref{signrules}. More
generally, this means that the diagrams become isotopy-invariant.

Now we justify the various equalities:
\begin{itemize}
\item
Equality~(1) follows from 18 applications of the ``Reidemester moves'' of Lemma~\ref{lem:reid} to move the vertical green strands to the right and the left of the diagram.
\item
Equality~(2) follows from~\eqref{reid2} and~\eqref{reid3}, used to pull the black strands through the crossings of red strands in the middle of the diagram.
\item
Equality~(3) follows from~\cite[(3.17)]{Brundan} (applied in the middle of the diagram, for the weight $\omega + \alpha_i + \alpha_j + \alpha_k$ and the color $j$).
\item
Equality~(4) follows from the following manipulation together with its 180 degrees rotation (which can be checked similarly):
\[
\begin{array}{c}
\begin{tikzpicture}[thick,xscale=0.4,yscale=0.5]
\draw[->,red] (0,-2) to (0,2);
\draw[->] (1,2) to (1,-2);
\draw[->] (2,-2) to[out=90, in=0] (-2,2);
\draw[->] (-2,-2) to[out=0,in=-90] (2,2);
\end{tikzpicture}
\end{array}
\stackrel{\eqref{pullapart}}{=}
\begin{array}{c}
\begin{tikzpicture}[thick,xscale=0.4,yscale=0.5]
\draw[->,red] (0,-2) to[out=45,in=-90] (1.2,-1.2) to[out=120,in=-90] (0,0) to (0,2);
\draw[->] (1,2) to (1,-2);
\draw[->] (2,-2) to[out=90, in=0] (-2,2);
\draw[->] (-2,-2) to[out=0,in=-90] (2,2);
\end{tikzpicture}
\end{array}
\stackrel{\eqref{reid2}}{=}
\pm
\begin{array}{c}
\begin{tikzpicture}[thick,xscale=0.4,yscale=0.5]
\draw[->,red] (0,-2) to[out=90,in=-90] (1.3,-.5) to[out=90,in=-90] (0,1) to (0,2);
\draw[->] (1,2) to (1,-2);
\draw[->] (2,-2) to[out=90, in=0] (-2,2);
\draw[->] (-2,-2) to[out=0,in=-90] (2,2);
\end{tikzpicture}
\end{array}
\stackrel{\eqref{reid4}}{=}
\pm
\begin{array}{c}
\begin{tikzpicture}[thick,xscale=0.4,yscale=0.5]
\draw[->,red] (0,-2) to[out=90,in=-90] (1.25,-1) to (1.25,1) to[out=90,in=-90] (0,2);
\draw[->] (1,2) to (1,-2);
\draw[->] (2,-2) to[out=90, in=0] (-2,2);
\draw[->] (-2,-2) to[out=0,in=-90] (2,2);
\end{tikzpicture}
\end{array}.
\]
\item
To prove~(5), on the right part of the diagram we use the relation
\[
 \begin{array}{c}\begin{tikzpicture}[thick,xscale=0.2,yscale=0.3]
      \draw[->] (-2,-2) to (2,2);
      \draw[->] (2,-2) to (-2,2);
      \draw[red,->] (0,-2) to[out=90,in=-90] (-2,0) to[out=90,in=-90] (0,2);
    \end{tikzpicture} 
\end{array}
=
 \begin{array}{c}
    \begin{tikzpicture}[thick,xscale=0.2,yscale=0.3]
      \draw[->] (-2,-2) to (2,2);
      \draw[->] (2,-2) to (-2,2);
      \draw[red,->] (0,-2) to[out=90,in=-90] (2,0) to[out=90,in=-90] (0,2);
    \end{tikzpicture}
\end{array}
+t_{ij}
  \begin{array}{c}
    \begin{tikzpicture}[thick,xscale=0.2,yscale=0.3]
      \draw[->] (-2,-2) to (-2,2);
      \draw[->] (2,-2) to (2,2);
      \draw[red,->] (0,-2) to (0,2);
    \end{tikzpicture}
\end{array},
\]
see~\eqref{eqn:rel-2KM-3}, and then check (using~\cite[(3.16)]{Brundan} and then~\eqref{eqn:rel-2KM-2}) that the term coming from the rightmost diagram vanishes. Then we do a similar manipulation on the left part of the diagram using the reversed variants of these relations, see~\cite[(2.7), (2.8)]{Brundan}.
\item
Equality~(6) follows from Lemma~\ref{lem:rel-zam}.
\item
In equality~(7) we use the relation
\[
\begin{array}{c}
\begin{tikzpicture}[thick,scale=0.3]
\draw[->] (-1,.7) to[out=0,in=-90] (2,3);
\draw[<-] (-1,-.7) to[out=0,in=90] (2,-3);
\draw[->,red] (0,-3) to[out=90,in=-90] (1.5,-1.2) to (1.5,1.2) to[out=90,in=-90] (0,3);
\draw[->] (1,3) to (1,-3);
\end{tikzpicture}
\end{array}
\stackrel{\eqref{reid2}}{=}
\pm
\begin{array}{c}
\begin{tikzpicture}[thick,scale=0.3]
\draw[->] (-1,.7) to[out=0,in=-90] (2,3);
\draw[<-] (-1,-.7) to[out=0,in=90] (2,-3);
\draw[->,red] (0,-3) to[out=90,in=-90] (1.5,-1.2) to[out=90,in=-90] (0,3);
\draw[->] (1,3) to (1,-3);
\end{tikzpicture}
\end{array}
\stackrel{\eqref{reid4}+\eqref{pullapart}}{=}
\pm
\begin{array}{c}
\begin{tikzpicture}[thick,scale=0.3]
\draw[->] (-1,.7) to[out=0,in=-90] (2,3);
\draw[<-] (-1,-.7) to[out=0,in=90] (2,-3);
\draw[->,red] (0,-3) to (0,3);
\draw[->] (1,3) to (1,-3);
\end{tikzpicture}
\end{array},
\]
then the 180 degrees rotation of this equality (which follows from similar manipulations), then the equalities
\begin{multline*}
\begin{array}{c}
\begin{tikzpicture}[thick,scale=0.3]
\draw[->] (-4,2) to[out=-90,in=180] (0,-.5) to[out=0,in=-90] (4,2);
\draw[->,red] (4,-2) to[out=90,in=0] (0,.5) to[out=180,in=90] (-4,-2);
\draw[->] (2,2) to (2,-2);
\draw[->,red] (1,-2) to (1,2);
\draw[->,red] (-1,2) to (-1,-2);
\draw[->] (-2,-2) to (-2,2);
\end{tikzpicture}
\end{array}
\stackrel{\eqref{reid2}}{=}
\pm
\begin{array}{c}
\begin{tikzpicture}[thick,scale=0.3]
\draw[->] (-4,2) to[out=-90,in=180] (0,-.5) to[out=0,in=-90] (4,2);
\draw[->,red] (4,-2) to[out=90,in=-90] (1.5,-.5) to[out=90,in=0] (0,.5) to[out=180,in=90] (-4,-2);
\draw[->] (2,2) to (2,-2);
\draw[->,red] (1,-2) to (1,2);
\draw[->,red] (-1,2) to (-1,-2);
\draw[->] (-2,-2) to (-2,2);
\end{tikzpicture}
\end{array}
\stackrel{\eqref{reid}}{=}
\begin{array}{c}
\begin{tikzpicture}[thick,scale=0.3]
\draw[->] (-4,2) to[out=-90,in=180] (0,-.5) to[out=0,in=-90] (4,2);
\draw[->,red] (4,-2) to[out=90,in=-90] (0,-.5) to[out=90,in=-90] (0,0) to[out=90,in=90] (-4,-2);
\draw[->] (2,2) to (2,-2);
\draw[->,red] (1,-2) to (1,2);
\draw[->,red] (-1,2) to (-1,-2);
\draw[->] (-2,-2) to (-2,2);
\end{tikzpicture}
\end{array}
\\
\stackrel{\eqref{reid3}}{=}
\begin{array}{c}
\begin{tikzpicture}[thick,scale=0.3]
\draw[->] (-4,2) to[out=-90,in=180] (0,-.5) to[out=0,in=-90] (4,2);
\draw[->,red] (4,-2) to[out=90,in=-90] (-1.5,-.5) to[out=90,in=0] (-2,.5) to[out=180,in=90] (-4,-2);
\draw[->] (2,2) to (2,-2);
\draw[->,red] (1,-2) to (1,2);
\draw[->,red] (-1,2) to (-1,-2);
\draw[->] (-2,-2) to (-2,2);
\end{tikzpicture}
\end{array}
\stackrel{\eqref{reid'}+\eqref{pullapart}}{=}
\begin{array}{c}
\begin{tikzpicture}[thick,scale=0.3]
\draw[->] (-4,2) to[out=-90,in=180] (0,.5) to[out=0,in=-90] (4,2);
\draw[->,red] (4,-2) to[out=90,in=0] (0,-.5) to[out=180,in=90] (-4,-2);
\draw[->] (2,2) to (2,-2);
\draw[->,red] (1,-2) to (1,2);
\draw[->,red] (-1,2) to (-1,-2);
\draw[->] (-2,-2) to (-2,2);
\end{tikzpicture}
\end{array}
\end{multline*}
and the 180 degrees rotation of this relation (which again follows from similar manipulations).
\end{itemize}

Finally we consider the last diagram in Figure~\ref{fig:zam} (without the sign, and not considered ``up to isotopy'' anymore), and denote it $Z'$. To prove~\eqref{zam2}, it suffices to prove the similar relation where ``$Z$'' is replaced by ``$Z'$''. In this setting, if we draw the right-hand side, we can first use the zigzag relations on the black strands (corresponding to $i$), then on the green strands (corresponding to $k$), and finally on the red strands (corresponding to $j$) to obtain the left-hand side. In this process we use the sign relations from Lemma~\ref{signrules}. In the first step, the corresponding crossings involve colors which are either equal or distant, so no sign appears. In the second step we use the relations
\[
    \begin{array}{c}
      \begin{tikzpicture}[thick,scale=0.4]
        \draw[->,green] (-1,0) to[out=90,in=180] (0,1) to[out=0,in=90] (1,0);
        \draw[red,->] (0.2,0) to (1.5,1.3);
        \draw[red,<-] (-.2,0) to (1.1,1.3);
      \end{tikzpicture}
    \end{array}
=
t_{jk}^{-1}
    \begin{array}{c}
      \begin{tikzpicture}[thick,scale=0.4]
        \draw[->,green] (-1,0) to[out=90,in=180] (0,1) to[out=0,in=90] (1,0);
        \draw[red,->] (0.2,0) to (-1.1,1.3);
        \draw[red,<-] (-.2,0) to (-1.5,1.3);
      \end{tikzpicture}
    \end{array},
\quad \text{and} \quad
    \begin{array}{c}
      \begin{tikzpicture}[thick,scale=-0.4]
        \draw[->,green] (-1,0) to[out=90,in=180] (0,1) to[out=0,in=90] (1,0);
        \draw[red,->] (0.2,0) to (1.5,1.3);
        \draw[red,<-] (-.2,0) to (1.1,1.3);
      \end{tikzpicture}
    \end{array}
=
t_{kj}
    \begin{array}{c}
      \begin{tikzpicture}[thick,scale=-0.4]
        \draw[->,green] (-1,0) to[out=90,in=180] (0,1) to[out=0,in=90] (1,0);
        \draw[red,->] (0.2,0) to (-1.1,1.3);
        \draw[red,<-] (-.2,0) to (-1.5,1.3);
      \end{tikzpicture}
    \end{array}.
\]
And in the third step we use the relations
\begin{gather*}
    \begin{array}{c}
      \begin{tikzpicture}[thick,scale=0.4]
        \draw[->,red] (-1,0) to[out=90,in=180] (0,1) to[out=0,in=90] (1,0);
        \draw[<-,red] (-1.4,0) to[out=90,in=180] (0,1.4) to[out=0,in=90] (1.4,0);
        \draw[->] (0.2,0) to (1.5,1.3);
        \draw[green,<-] (-.2,0) to (1.1,1.3);
      \end{tikzpicture}
    \end{array}
=
t_{ji} t_{kj}^{-1}
    \begin{array}{c}
      \begin{tikzpicture}[thick,scale=0.4]
        \draw[->,red] (-1,0) to[out=90,in=180] (0,1) to[out=0,in=90] (1,0);
        \draw[<-,red] (-1.4,0) to[out=90,in=180] (0,1.4) to[out=0,in=90] (1.4,0);
        \draw[->] (0.2,0) to (-1.1,1.3);
        \draw[green,<-] (-.2,0) to (-1.5,1.3);
      \end{tikzpicture}
    \end{array},
\qquad
    \begin{array}{c}
      \begin{tikzpicture}[thick,scale=-0.4]
        \draw[->,red] (-1,0) to[out=90,in=180] (0,1) to[out=0,in=90] (1,0);
        \draw[<-,red] (-1.4,0) to[out=90,in=180] (0,1.4) to[out=0,in=90] (1.4,0);
        \draw[->] (0.2,0) to (1.5,1.3);
        \draw[green,<-] (-.2,0) to (1.1,1.3);
      \end{tikzpicture}
    \end{array}
=
t_{jk} t_{ij}^{-1}
    \begin{array}{c}
      \begin{tikzpicture}[thick,scale=-0.4]
        \draw[->,red] (-1,0) to[out=90,in=180] (0,1) to[out=0,in=90] (1,0);
        \draw[<-,red] (-1.4,0) to[out=90,in=180] (0,1.4) to[out=0,in=90] (1.4,0);
        \draw[->] (0.2,0) to (-1.1,1.3);
        \draw[green,<-] (-.2,0) to (-1.5,1.3);
      \end{tikzpicture}
    \end{array},
\\
    \begin{array}{c}
      \begin{tikzpicture}[thick,scale=-0.4]
        \draw[<-] (-1,0) to[out=90,in=180] (0,1) to[out=0,in=90] (1,0);
        \draw[red,->] (0.2,0) to (1.5,1.3);
        \draw[red,<-] (-.2,0) to (1.1,1.3);
      \end{tikzpicture}
    \end{array}
=
t_{ji}^{-1}
    \begin{array}{c}
      \begin{tikzpicture}[thick,scale=-0.4]
        \draw[<-] (-1,0) to[out=90,in=180] (0,1) to[out=0,in=90] (1,0);
        \draw[red,->] (0.2,0) to (-1.1,1.3);
        \draw[red,<-] (-.2,0) to (-1.5,1.3);
      \end{tikzpicture}
    \end{array},
\qquad
    \begin{array}{c}
      \begin{tikzpicture}[thick,scale=0.4]
        \draw[<-] (-1,0) to[out=90,in=180] (0,1) to[out=0,in=90] (1,0);
        \draw[red,->] (0.2,0) to (1.5,1.3);
        \draw[red,<-] (-.2,0) to (1.1,1.3);
      \end{tikzpicture}
    \end{array}
=
t_{ij}
    \begin{array}{c}
      \begin{tikzpicture}[thick,scale=0.4]
        \draw[<-] (-1,0) to[out=90,in=180] (0,1) to[out=0,in=90] (1,0);
        \draw[red,->] (0.2,0) to (-1.1,1.3);
        \draw[red,<-] (-.2,0) to (-1.5,1.3);
      \end{tikzpicture}
    \end{array}.
\end{gather*}
Hence all the signs cancel, and we obtain the desired relation.
\end{proof}

 }
\else {  \newpage }
\fi

\part{Relation to parity sheaves}
\label{pt:parities}

\ifdefined\PARTCOMPILE{ \textbf{Overview}.
This part is devoted to the proof of the relation between the
diagrammatic Hecke category and parity complexes on flag
varieties. Many proofs in this section are similar to some proofs in
Part~\ref{pt:general-conj}. However we repeat most of the details, for
the benefit of readers interested in geometry of flag varieties but
not necessarily in representation theory of reductive groups.

In Section~\ref{sec:parity} we recall some notions related to (partial) flag varieties of a Kac--Moody group $\GKM$ and to parity complexes on these varieties. In particular we introduce and study the concept of ``section of the $!$-flag'', which is an analogue for parity complexes of the ``sections of the $\Cos$-flag'' of tilting modules in Part~\ref{pt:general-conj}. Using this notion we prove Proposition~\ref{prop:morphisms-BS-parity}, which explains how to ``generate" morphisms between ``Bott--Samelson type" parity complexes and is the analogue for parity complexes of Proposition~\ref{prop:morphisms-BS}.

In Section~\ref{sec:parity-Hecke} we prove the main result of this part, namely an equivalence between the diagrammatic Hecke category attached to $\GKM$ and the category of Borel-equivariant parity complexes on the full flag variety of $\GKM$. (This equivalence is a generalization of Soergel's description of the category of equivariant semisimple complexes with complex coefficients on a flag variety in terms of Soergel bimodules.) In particular, this provides a geometric description of the category $\Diag$ of Part~\ref{pt:general-conj}.

Finally, in Section~\ref{sec:Whittaker} we extend these constructions to some ``Whittaker-type" parity complexes, and derive a geometric description of the category $\Dasph$ of Part~\ref{pt:general-conj}.

\section{Parity complexes on flag varieties}
\label{sec:parity}

\subsection{Reminder on Kac--Moody groups and their flag varieties}
\label{ss:KMgroups}

Let $A=(a_{i,j})_{i,j \in I}$ be a generalized Cartan matrix, with rows and columns parametrized by a finite set $I$, and let $(\Lambda, \{\alpha_i : i \in I\}, \{\alpha_i^\vee : i \in I\})$ be an associated Kac--Moody root datum. In other words, $\Lambda$ is a finitely generated free $\Z$-module, $\{\alpha_i : i \in I\}$ is a collection of elements of $\Lambda$, $\{\alpha_i^\vee : i \in I\}$ is a collection of elements of $\Hom_{\Z}(\Lambda,\Z)$, and we assume that
\[
a_{i,j}=\langle \alpha_i^\vee, \alpha_j \rangle 
\]
for any $i,j \in I$.
To such a datum one can associate following Mathieu (see~\cite[p.~45]{mathieu-KM}) a group ind-scheme $\GKM_\Z$ over $\Z$. This group ind-scheme has a canonical Borel subgroup $\BKM_\Z$ (which is a group subscheme) and a canonical maximal torus $\TKM_\Z$, whose group of characters is canonically isomorphic to $\Lambda$.
For any subset $J \subset I$ of finite type 
we also have an associated parabolic subgroup $\PKM_{J,\Z} \subset \GKM_\Z$, see~\cite[p.~61]{mathieu0}.

We let $W$ be the Weyl group of $\GKM_\Z$, 
$S \subset W$ be the subset of simple reflections, and $\ell$ be the corresponding length function. There exist canonical bijections $I \simto S$ and $S \simto I$, which will be denoted $i \mapsto s_i$ and $s \mapsto i_s$. We also denote by $\leq$ the Bruhat order on $W$. We will use the same terminology and notation as in Part~\ref{pt:general-conj} for the objects attached to the Coxeter group $(W,S)$; see in particular~\S\ref{ss:BStilting} for the notion of an expression.

\begin{rmk}
\label{rmk:KM}
\begin{enumerate}
\item
\label{it:KM-conditions}
To be more precise, in~\cite{mathieu-KM} Mathieu works under some technical conditions on the Kac--Moody root datum;
see also~\cite[\S 6.5]{tits} or~\cite[Remarque~3.5]{rousseau} for a discussion of this condition. See~\cite[\S 6.8]{tits} for a sketch of an argument explaining how to generalize this construction to general Kac--Moody root data, and~\cite[\S 3.19]{rousseau} for more details.
\item
As claimed in~\cite{mathieu-KM} (and explained in more detail in~\cite[\S 3.8]{rousseau}), Mathieu's construction generalizes the construction of split connected reductive groups. In particular, the constructions in this part
apply when $\GKM$ is a 
connected reductive group.
\item
We work with Mathieu's Kac--Moody groups and not with the version of Tits (considered in particular in~\cite{kumar}) for two reasons. The first one is that we need an algebro-geometric structure on the group itself (and not only on its flag variety) in order to define convolution of parity complexes. The second one is that we want to consider base fields of positive characteristic also. (This case is crucial for the considerations in Section~\ref{sec:Whittaker}.)
\end{enumerate}
\end{rmk}



We now fix an algebraically closed field $\bbL$. 
We set
\[
 \GKM := \mathrm{Spec}(\bbL) \times_{\mathrm{Spec}(\Z)} \GKM_\Z, \quad
 \BKM := \mathrm{Spec}(\bbL) \times_{\mathrm{Spec}(\Z)} \BKM_\Z, \quad
 \TKM := \mathrm{Spec}(\bbL) \times_{\mathrm{Spec}(\Z)} \TKM_\Z.
\]
Then $\GKM$ is a group ind-scheme over $\bbL$, and $\BKM$ and $\TKM$ are group subschemes of $\GKM$. For any simple reflection $s \in S$, we also set
\[
 \PKM_s := \mathrm{Spec}(\bbL) \times_{\mathrm{Spec}(\Z)} \PKM_{\{i_s\}, \Z}.
\]

Let us briefly recall the construction of $\GKM$ and its flag variety $\Flag := \GKM/\BKM$. By Remark~\ref{rmk:KM}\eqref{it:KM-conditions} one can assume that the technical conditions of~\cite{mathieu-KM} are satisfied, which we will do from now on. Choose $\lambda \in \Lambda$ such that $\langle \lambda, \alpha_i^\vee \rangle >0$ for any $i \in I$, and consider the corresponding integrable module $L(\lambda)$ for the Lie agebra over $\bbL$ associated with $A$ (see~\cite[p.~28]{mathieu0} for the case $\mathrm{char}(\bbL)=0$, and~\cite[p.~246]{mathieu0} for the general case, which is obtained by extension of scalars from an integral form of the module in characteristic $0$). Then $L(\lambda)$ has a canonical $\BKM$-module structure. For $w \in W$ we denote by $L(\lambda)_{w\lambda}$ the weight space of $L(\lambda)$ of weight $w\lambda$, and by $E_w(\lambda)$ the $\BKM$-submodule of $L(\lambda)$ generated by $L(\lambda)_{w\lambda}$. We set
\[
 S_{w,\lambda} := \overline{\BKM \cdot L(\lambda)_{w\lambda}} \subset \mathbb{P}(E_w(\lambda)).
\]
Then $S_{w,\lambda}$ is a projective variety over $\bbL$. Clearly $S_{w,\lambda}^\circ := \BKM \cdot L(\lambda)_{w\lambda}$ is an open subvariety of $S_{w,\lambda}$, which can be seen to be isomorphic to the group denoted $U(w)$ in~\cite[Lemme~4]{mathieu-KM} (see in particular~\cite[Lemme~133 on p.~247]{mathieu0}), hence to an affine space of dimension $\ell(w)$. 

If $v \leq w$ there exists a natural closed embedding $S_{v,\lambda} \subset S_{w,\lambda}$, and we have
\[
S_{w,\lambda} = \bigsqcup_{v \leq w} S_{v,\lambda}^\circ
\]
(see~\cite[p.~65]{mathieu-KM} for the case $\mathrm{char}(\bbL)=0$).

For any expression $\uw=s_1 \cdots s_r$ we can consider the Bott--Samelson resolution
\[
 \BSvar(\uw) := \PKM_{s_1} \times^\BKM \cdots \times^\BKM \PKM_{s_r}/\BKM,
\]
a smooth projective variety.
If $\uw$ is a reduced expression for $w$ there exists a canonical (proper, birational) morphism
\[
 \nu_{\uw, \lambda} \colon \BSvar(\uw) \to S_{w,\lambda},
\]
see~\cite[Lemme~29 on p.~65 and p.~246]{mathieu0}.

Now for $w \in W$ we denote by $\Flag_{\leq w}$ the normalization of $S_{w,\lambda}$. Then $\Flag_{\leq w}$ is a projective variety, see~\cite[Proposition~A.6]{kumar}. The inverse image $\Flag_w$ of $S_{w,\lambda}^\circ$ under the normalization morphism is an open subvariety of $\Flag_{\leq w}$, isomorphic to $\mathbb{A}_\bbL^{\ell(w)}$. If $\uw$ is a reduced expression for $w$, since $\BSvar(\uw)$ is smooth, the morphism $\nu_{\uw,\lambda}$ factors (uniquely) through a morphism
\[
 \nu_{\uw} \colon \BSvar(\uw) \to \Flag_{\leq w},
\]
see~\cite[Proposition~A.7]{kumar}. Moreover, since
$\nu_{\uw, \lambda}$ is birational, $\nu_{\uw}$ is also birational. Hence, by Zariski's main theorem,
it satisfies $(\nu_{\uw})_* \cO_{\BSvar(\uw)} \cong \cO_{\Flag_{\leq w}}$, see~\cite[Theorem~A.9]{kumar}. It is proved in~\cite[Lemme~32 on p.~69]{mathieu0} that the normalization morphism $\Flag_{\leq w} \to S_{w,\lambda}$ is a homeomorphism. In particular, this implies that $\Flag_{\leq w}$ can be described as the ringed space $\bigl( S_{w,\lambda}, (\nu_{\uw, \lambda})_* \cO_{\BSvar(\uw)} \bigr)$.

 
Using this description, Mathieu checks that $\Flag_{\leq w}$ is independent of $\lambda$, see~\cite[Lemme~33 on p.~71 and \S XVIII]{mathieu0}. 
%

If $v \leq w$, then the closed embedding $S_{v,\lambda} \hookrightarrow S_{w,\lambda}$ induces a scheme morphism $\Flag_{\leq v} \to \Flag_{\leq w}$, which is known to be a closed embedding, see~\cite[Th\'eor\`eme~2(b) on p.~116 and Th\'eor\`eme~5(6) on p.~241]{mathieu0}. In particular $\Flag_v$ is a locally closed subvariety of $\Flag_{\leq w}$, and we have
\begin{equation}
 \label{eqn:Bruhat-Schubert}
 \Flag_{\leq w} = \bigsqcup_{v \leq w} \Flag_v
\end{equation}
(which justifies our notations). Finally there exists a natural $\BKM$-action on $\Flag_{\leq w}$, such that~\eqref{eqn:Bruhat-Schubert} identifies with the decomposition into $\BKM$-orbits.

\begin{rmk}
 It is proved in~\cite[Th\'eor\`eme~2 on p.~116]{mathieu0} that the normalization morphism $\Flag_{\leq w} \to S_{w,\lambda}$ is an isomorphism if $\lambda$ is ``sufficiently large,'' and in~\cite[Corollaire~1]{mathieu-KM} that this is in fact the case for any $\lambda$ as above. 
\end{rmk}

The varieties $\Flag_{\leq w}$ play the role of Schubert varieties in the flag variety. Next one defines what will play the role of their inverse image in $\GKM$. For this we observe that given a $\BKM$-module $M$ one can define a vector bundle $\mathscr{L}_w(M)$ on $\Flag_{\leq w}$ by taking the direct image under $\nu_{\uw}$ of the vector bundle on $\BSvar(\uw)$ naturally associated with $M$. (Here $\uw$ is a reduced expression for $w$.) When $M=\bbL[\BKM]$ is the regular representation of $\BKM$, this vector bundle has a natural structure of sheaf of algebras. Therefore there exists a scheme $\BKM(w)$ and an affine morphism $\mu_w \colon \BKM(w) \to \Flag_{\leq w}$ such that $\mathscr{L}_w(\bbL[\BKM]) \cong (\mu_w)_* \cO_{\BKM(w)}$. One next proves that $\BKM(w)$ is an affine scheme, and is the affinization of the variety
\[
 \PKM_{s_1} \times^\BKM \cdots \times^\BKM \PKM_{s_r}
\]
if $(s_1, \cdots, s_r)$ is a reduced expression for $w$; see~\cite[p.~129, Lemme~70 on p.~133, and Proposition~26 on p.~256]{mathieu0}. By~\cite[Lemme~7]{mathieu-KM}, one also knows that the natural $\BKM$-action on $\BKM(w)$ on the right is locally free, and that $\mu_w$ is the quotient morphism for this action. If $v \leq w$, then the closed embedding $\Flag_{\leq v} \hookrightarrow \Flag_{\leq w}$ induces a closed embedding $\BKM(v) \hookrightarrow \BKM(w)$.

Finally, one defines $\GKM$ as the ind-scheme associated with the varieties $\{\BKM(w) : w \in W\}$ and the closed embeddings $\BKM(v) \hookrightarrow \BKM(w)$ for $v \leq w$. This ind-scheme has a natural group structure where the multiplication and inverse morphisms are induced by the natural operations on the schemes $\PKM_{s_1} \times^\BKM \cdots \times^\BKM \PKM_{s_r}$, see~\cite[p.~45]{mathieu-KM} for details. Then it makes sense to consider the quotient $\Flag=\GKM/\BKM$, which is the ind-variety associated with the varieties $\{\Flag_{\leq w} : w \in W\}$ and the closed embeddings $\Flag_{\leq v} \hookrightarrow \Flag_{\leq w}$ for $v \leq w$.

By construction
we have a Bruhat decomposition
\[
\Flag = \bigsqcup_{w \in W} \Flag_w,
\]
and each $\Flag_w$ is a locally-closed subvariety of $\Flag$ isomorphic to an affine space of dimension $\ell(w)$. We denote by $i_w \colon \Flag_w \to \Flag$ the corresponding inclusion morphism.
For any $w \in W$, we will also consider the open ind-subvariety
\[
\Flag_{\geq w} = \bigsqcup_{y \geq w} \Flag_y.
\]
Note that $\Flag_w$ is closed in $\Flag_{\geq w}$.

\subsection{Partial flag varieties}
\label{ss:partial-flags}

Let $J \subset I$ be a subset of finite type. We denote by $W_J$ the (finite) subgroup of $W$ generated by $\{s_j : j \in J\}$, and by $W^J$ the subset of $W$ consisting of elements $w$ which are minimal in $wW_J$. We also denote by $w_0^J$ the longest element of $W_J$, so that $\{w w_0^J : w \in W^J\}$ is the subset of $W$ consisting of elements $v$ which are maximal in $vW_J$. We set
\[
\PKM_J := \mathrm{Spec}(\mathbb{L}) \times_{\mathrm{Spec}(\Z)} \PKM_{J,\Z}.
\]

Let $w \in W^J$.
In~\cite{mathieu0}, Mathieu also considers some varieties $S_{w,\mu}$ defined as above, but now under the assumption that $\langle \mu, \alpha_j^\vee \rangle = 0$ if $j \in J$ and $\langle \mu, \alpha_i^\vee \rangle > 0$ if $i \in I \smallsetminus J$. The open subvariety $S_{w,\mu}^\circ \subset S_{w,\mu}$ is also defined as above, and isomorphic to an affine space of dimension $\ell(w)$. If $\uw$ is a reduced expression for $w$, we also have a ``Demazure resolution" $\nu_{\uw,\mu} \colon \BSvar(\uw) \to S_{w,\mu}$, which is proper and birational.

As in the case $J=\varnothing$, it is proved in~\cite{mathieu0} that the normalization $\Flag^J_{\leq w}$ of $S_{w,\mu}$ does not depend on the choice of $\mu$.
If $v \leq w$ are in $W^J$ there exists a canonical closed embedding $\Flag_{\leq v}^J \hookrightarrow \Flag^J_{\leq w}$; hence we can define the partial flag variety $\Flag^J$ as the ind-scheme associated with the varieties $\{\Flag^J_{\leq w} : w \in W^J\}$ and the closed embeddings $\Flag^J_{\leq v} \hookrightarrow \Flag^J_{\leq w}$ for $v \leq w$. For $w \in W^J$ we denote by $\Flag^J_w$ the inverse image of $S_{w,\mu}^\circ$ in $\Flag^J_{\leq w}$; then each $\Flag^J_w$ is an affine space of dimension $\ell(w)$, and we have a Bruhat decomposition
\[
 \Flag^J = \bigsqcup_{w \in W^J} \Flag^J_w.
\]
For $w \in W^J$, we will denote by $i_w^J \colon \Flag^J_w \hookrightarrow \Flag^J$ the embedding.

Now, let us explain how to construct a $\BKM$-equivariant ind-scheme morphism $q^J \colon \Flag \to \Flag^J$, following a construction in~\cite{kumar}. In fact it suffices to construct compatible morphisms $q^J_w \colon \Flag_{\leq w w_0^J} \to \Flag^J_{\leq w}$ for any $w \in W^J$. Consider the natural morphism $f_{\lambda, \mu} \colon L(\lambda + \mu) \to L(\lambda) \otimes L(\mu)$, see e.g.~\cite[Corollaire~2]{mathieu-KM}. This map identifies the line $L(\lambda + \mu)_{w w_0^J(\lambda+\mu)} = L(\lambda + \mu)_{ww_0^J(\lambda) + w(\mu)}$ with $L(\lambda)_{ww_0^J(\lambda)} \otimes L(\mu)_{w(\mu)}$, hence maps $E_{ww_0^J}(\lambda+\mu)$ into $E_{ww_0^J}(\lambda) \otimes E_w(\mu)$. It is clear also that the embedding $S_{ww_0^J, \lambda + \mu} \hookrightarrow \mathbb{P}(E_{ww_0^J}(\lambda+\mu))$ factors through an embedding
\[
S_{ww_0^J, \lambda + \mu} \hookrightarrow \mathbb{P} \bigl( E_{ww_0^J}(\lambda+\mu) / \ker(f_{\lambda,\mu}) \cap E_{ww_0^J}(\lambda+\mu) \bigr).
\]

Consider now the Segre embedding
\[
\mathbb{P}(E_{ww_0^J}(\lambda)) \times \mathbb{P}(E_w(\mu)) \hookrightarrow \mathbb{P}(E_{ww_0^J}(\lambda) \otimes E_w(\mu)),
\]
see e.g.~\cite[Lemma 7.1.14]{kumar}. This morphism is clearly $\BKM$-equivariant (for the diagonal action on the left-hand side), and its image contains the image of $L(\lambda+\mu)_{w w_0^J(\lambda+\mu)}$. Hence the composition
\[
S_{ww_0^J, \lambda+\mu} \hookrightarrow \mathbb{P} \bigl( E_{ww_0^J}(\lambda+\mu) / \ker(f_{\lambda,\mu}) \cap E_{ww_0^J}(\lambda+\mu) \bigr) \hookrightarrow \mathbb{P}(E_{ww_0^J}(\lambda) \otimes E_w(\mu))
\]
factors through a closed embedding $S_{ww_0^J, \lambda+\mu} \hookrightarrow \mathbb{P}(E_{ww_0^J}(\lambda)) \times \mathbb{P}(E_w(\mu))$, which itself defines a closed embedding
\[
S_{ww_0^J,\lambda+\mu} \hookrightarrow S_{ww_0^J}(\lambda) \times S_w(\mu).
\]
Composing with the projection on the second component, and passing to normalizations (using again~\cite[Proposition~A.7]{kumar}), we deduce the wished-for morphism $q^J_w \colon \Flag_{\leq w w_0^J} \to \Flag^J_{\leq w}$.

Since the varieties $\Flag_{\leq w w_0^J}$ and $\Flag^J_{\leq w}$ are projective, the morphism $q^J$ is ind-proper. It is also clear that if $w \in W^J$, 
the set-theoretic inverse image of $\Flag^J_w$ under $q^J$ is $\bigsqcup_{v \in W_J} \Flag_{wv}$, and that
the restriction of $q^J$ to the stratum $\Flag_{wv}$ is isomorphic to the standard projection $\mathbb{A}_{\bbL}^{\ell(w) + \ell(v)} \to \mathbb{A}_{\bbL}^{\ell(w)}$.

\begin{rmk}
In view of the familiar case of reductive groups, we would expect the morphism $q^J$ to be smooth (and in fact a locally trivial $\PKM_J/\BKM$-bundle). This property would follow e.g.~if we knew that the natural $\PKM_J$-action on $\BKM(ww_0^J)$ on the right (for $w \in W^J$) is locally free. However, we were not able to prove this fact in the case when $\mathrm{char}(\bbL) \neq 0$. (The case $\mathrm{char}(\bk)=0$ can e.g.~be treated using~\cite[Corollary~7.4.15 and Exercise~7.4.E.5]{kumar}.) This is the source of some technical difficulties encountered below.
\end{rmk}

Below we will mainly consider the special case $J=\{i_s\}$ for some $s \in S$. In this case we will replace the superscript $\{i_s\}$ by $s$. In particular, we have
the partial flag variety $\Flag^s$,
the Schubert cells $\Flag^s_w$ for any $w \in W^s$, and the embeddings $i^s_w \colon \Flag_w^s \to \Flag^s$. Below we will sometimes identify $W^s$ with $W/W_s$.
For $w \in W^s$, we also set
\[
\Flag^s_{\geq w} := \bigsqcup_{\substack{y \in W^s \\ y \geq w}} \Flag^s_y.
\]
Here again, $\Flag^s_{\geq w}$ is an open ind-subvariety in $\Flag^s$, and $\Flag^s_w$ is closed in $\Flag^s_{\geq w}$.


\subsection{Derived categories of sheaves on $\Flag$ and $\Flag^s$}
\label{ss:DbX}


We will consider some categories of sheaves of $\K$-modules on $\Flag$ and $\Flag^J$, where $\K$ is a commutative ring. We will assume that we are in one of the following two contexts:
\begin{enumerate}
\item
$\bbL=\C$, $\K$ is Noetherian of finite global dimension, and we consider sheaves for the analytic topology on our varieties;
\item
$\bbL$ is arbitrary, $\K$ is either an algebraic closure of $\Q_p$, or a finite extension of $\Q_p$, or the ring of integers in such an extension, or a finite field of characteristic $p$ (where $p$ is a prime number different from $\mathrm{char}(\bbL)$), and we work with \'etale $\K$-sheaves.
\end{enumerate}
The first context will be referred to as the ``classical context," and the second one as the ``\'etale context."

We denote by $\Db_\BKM(\Flag, \K)$, resp.~$\Db_\BKM(\Flag^J, \K)$, the $\BKM$-equivariant derived category of $\Flag$, resp.~$\Flag^J$. (See~\cite[\S 2.7]{weidner} for the construction of the equivariant derived category in the \'etale context.) We also denote by $\Db_{(\BKM)}(\Flag,\K)$, resp.~$\Db_{(\BKM)}(\Flag^J,\K)$, the corresponding derived category of complexes which are constructible with respect to the stratification given by the Bruhat decomposition. (The technical details of the definition of such categories are discussed in~\cite[\S 2.2]{nadler}, and will not be repeated here. Let us only note that, by definition, any object of $\Db_{\BKM}(\Flag,\K)$ or $\Db_{(\BKM)}(\Flag,\K)$ is supported on a \emph{finite} union of Bruhat cells.) There exists a natural forgetful functor $\Db_{\BKM}(\Flag,\K) \to \Db_{(\BKM)}(\Flag,\K)$, which will usually be omitted from the notation.

For $\cF,\cG$ in $\Db_{\BKM}(\Flag,\K)$ or in $\Db_{(\BKM)}(\Flag,\K)$, we set
\begin{align*}
\Hom^\bullet_{\Db_{\BKM}(\Flag,\K)}(\cF,\cG) &= \bigoplus_{n \in \Z} \Hom_{\Db_{\BKM}(\Flag,\K)}(\cF,\cG[n]), \\
\Hom^\bullet_{\Db_{(\BKM)}(\Flag,\K)}(\cF,\cG) &= \bigoplus_{n \in \Z} \Hom_{\Db_{(\BKM)}(\Flag,\K)}(\cF,\cG[n])
\end{align*}
respectively. We use a similar notation for $\Flag^J$.

The morphism $q^J$ induces functors
\begin{multline*}
(q^J)_*, (q^J)_! \colon \Db_{\BKM}(\Flag,\K) \to \Db_{\BKM}(\Flag^s,\K), \\
(q^J)^*, (q^J)^! \colon \Db_{\BKM}(\Flag^s,\K) \to \Db_{\BKM}(\Flag,\K)
\end{multline*}
and similarly functors
\begin{multline*}
(q^J)_*, (q^J)_! \colon \Db_{(\BKM)}(\Flag,\K) \to \Db_{(\BKM)}(\Flag^s,\K), \\
(q^J)^*, (q^J)^! \colon \Db_{(\BKM)}(\Flag^s,\K) \to \Db_{(\BKM)}(\Flag,\K),
\end{multline*}
these functors being compatible with the forgetful functor in the obvious sense. Since $q^J$ is proper, we have a canonical isomorphism $(q^J)_! \simto (q^J)_*$.

For $w \in W$, we set
\[
\Cos_w := (i_w)_* \underline{\K}_{\Flag_w} [\ell(w)] \in \Db_{\BKM}(\Flag, \K), \quad
\Sta_w := (i_w)_! \underline{\K}_{\Flag_w} [\ell(w)] \in \Db_{\BKM}(\Flag, \K).
\]
If $w \in W^J$, we also set
\[
\Cos_w^J := (i^J_w)_* \underline{\K}_{\Flag_w^J}[\ell(w)] \in \Db_{\BKM}(\Flag^J, \K), \quad
\Sta_w^J := (i^J_w)_! \underline{\K}_{\Flag_w^J}[\ell(w)] \in \Db_{\BKM}(\Flag^J, \K).
\]

The category $\Db_{\BKM}(\Flag,\K)$ can be endowed with a natural convolution product
\[
(-) \star^{\BKM} (-) \colon \Db_{\BKM}(\Flag,\K) \times \Db_{\BKM}(\Flag,\K) \to \Db_{\BKM}(\Flag,\K)
\]
defined as follows. Let $\mu \colon \GKM \to \Flag$ be the projection, and $\mathsf{m} \colon \GKM \times^{\BKM} \Flag \to \Flag$ be the morphism induced by the $\GKM$-action on $\Flag$. Then, given $\cE$ and $\cF$ in $\Db_{\BKM}(\Flag,\K)$, there exists a unique object $\cE \, \widetilde{\boxtimes} \, \cF$ in $\Db_{\BKM}(\GKM \times^{\BKM} \Flag,\K)$ whose pullback to $\GKM \times \Flag$ is isomorphic to $\mu^*\cE \boxtimes \cF$; by definition we have
\[
\cE \star^\BKM \cF = \mathsf{m}_*(\cE \, \widetilde{\boxtimes} \, \cF).
\]
This convolution product, together with the natural associativity constraint and the natural identity object, make $\Db_{\BKM}(\Flag,\K)$ into a monoidal category.

The same construction also defines bifunctors
\begin{align*}
\Db_{(\BKM)}(\Flag,\K) \times \Db_{\BKM}(\Flag,\K) &\to \Db_{(\BKM)}(\Flag,\K), \\
\Db_{\BKM}(\Flag,\K) \times \Db_{\BKM}(\Flag^J,\K) &\to \Db_{\BKM}(\Flag^J,\K)
\end{align*}
which we will denote similarly, and which define a right action of the monoidal category $\Db_{\BKM}(\Flag,\K)$ on $\Db_{(\BKM)}(\Flag,\K)$, and a left action on $\Db_{\BKM}(\Flag^J,\K)$.

\begin{rmk}
As above, in the description of $\star^\BKM$ we have ignored the subtleties related to the fact that $\GKM$ is an ind-variety and not an honest variety.
\end{rmk}

\subsection{Parity complexes on flag varieties}
\label{ss:parity-flag}

We continue with the setting of~\S\ref{ss:DbX}.
Recall from~\cite{jmw} the notion of \emph{even} objects, \emph{odd} objects, and \emph{parity complexes} in $\Db_{\BKM}(\Flag, \K)$, $\Db_{(\BKM)}(\Flag, \K)$, $\Db_{\BKM}(\Flag^J, \K)$ or $\Db_{(\BKM)}(\Flag^J, \K)$.
(In~\cite{jmw}, only certain rings of coefficients are considered; but the definitions have obvious analogues in our more general setting; see~\cite[\S 2.1]{mr}.) We will denote by $\Par_{\BKM}(\Flag, \K)$, $\Par_{(\BKM)}(\Flag, \K)$, $\Par_{\BKM}(\Flag^J, \K)$, $\Par_{(\BKM)}(\Flag^J, \K)$ the corresponding full subcategories of parity complexes. Then by definition the forgetful functors send $\Par_{\BKM}(\Flag, \K)$ into $\Par_{(\BKM)}(\Flag, \K)$ and $\Par_{\BKM}(\Flag^J, \K)$ into $\Par_{(\BKM)}(\Flag^J, \K)$. 

\begin{lem}
\label{lem:q^J_*-parity}
The functor $(q^J)_*$ sends even objects, resp.~odd objects, resp.~parity complexes on $\Flag$ to even objects, resp.~odd objects, resp.~parity complexes on $\Flag^J$ (in both the $\BKM$-equivariant and the $\BKM$-constructible settings).
\end{lem}

\begin{proof}
In the classical context, this follows from the facts that $q^J$ is an even morphism (with respect to the Bruhat stratifications of $\Flag$ and $\Flag^J$) in the sense of~\cite[Definition~2.33]{jmw}, and that such morphisms preserve parity objects when they are proper; see~\cite[Proposition~2.34]{jmw}. Similar arguments apply in the \'etale context. (Note that the fibers of the morphisms $q^J_{w}$ might a priori not be reduced; however this does not play any role in cohomology, see~\cite[Remark~II.3.17]{milne}.)
\end{proof}

We set
\[
\cE_{w_0^J} := \underline{\K}_{\PKM_J/\BKM}[\ell(w_0^J)].
\]
Since $\PKM_J/\BKM=\Flag_{\leq w^J_0}$ is smooth, it is easy to see that $\cE_{w_0^J}$ is a parity complex on $\Flag$. It is also clear that
\begin{equation}
\label{eqn:inverse-image-skyscraper}
(q^J)^* \underline{\K}_{\Flag^J_{\leq 1}} \cong \cE_{w_0^J} [-\ell(w_0^J)],
\end{equation}
where here $1$ is the unit element in $W$. (Here again, some non-reducedness questions might arise; but this can be ignored when considering sheaves by~\cite[Remark~II.3.17]{milne}.)

Consider the variety $\GKM \times^\BKM (\PKM_J / \BKM)$ (which is well defined since the right $\BKM$-action on $\GKM$ is locally trivial). There are natural proper morphisms
\[
a^1_J, a^2_J \colon \GKM \times^\BKM (\PKM_J / \BKM) \to \Flag
\]
induced by projection on the first factor and the $\GKM$-action on $\Flag$ respectively. Moreover, the morphism $a_J^1$ is smooth (more precisely, a locally trivial $\PKM_J/\BKM$-fibration). These morphisms fit in the following commutative diagram:
\[
\xymatrix@C=1.5cm{
\GKM \times^\BKM (\PKM_J/\BKM) \ar[r]^-{a_J^2} \ar[d]_-{a_J^1} & \Flag \ar[d]^-{q^J} \\
\Flag \ar[r]^-{q^J} & \Flag^J.
}
\]
It is clear also that there exists a canonical isomorphism
\begin{equation}
\label{eqn:convolution-constant}
(-) \star^\BKM \cE_{w_0^J}[-\ell(w_0^J)] \cong (a_J^2)_* (a_J^1)^* (-)
\end{equation}
of endofunctors of $\Db_\BKM(\Flag,\K)$ and of $\Db_{(\BKM)}(\Flag,\K)$.

\begin{lem}
\label{lem:qJ-parity}
\begin{enumerate}
\item
\label{it:qJ1}
There exist isomorphisms
\[
(q^J)^* (q^J)_* \cong (a_J^2)_* (a_J^1)^*, \qquad (q^J)^! (q^J)_* \cong (a_J^2)_* (a_J^1)^*[2\ell(w_0^J)]
\]
of endofunctors of $\Db_\BKM(\Flag,\K)$ and of $\Db_{(\BKM)}(\Flag,\K)$.
\item
\label{it:qJ2}
The functors $(q^J)^* (q^J)_*$ and $(q^J)^! (q^J)_*$ send parity complexes to parity complexes (in both the $\BKM$-equivariant and the $\BKM$-constructible settings).
\end{enumerate}
\end{lem}

\begin{proof}
\eqref{it:qJ1}
Let us start with the first isomorphism. For a general algebraic group $K$ and subgroup $H$, given a $K$-variety $X$ there exists a natural ``convolution" bifunctor 
\[
\star^H \colon \Db_{\mathrm{c}}(K/H, \K) \times \Db_H(X,\K) \to \Db_{\mathrm{c}}(X,\K),
\]
where $\Db_{\mathrm{c}}(-)$ means the constructible derived category.
If $Y$ is another $K$-variety and $f \colon X \to Y$ is a $K$-equivariant morphism, it is easy to see that the functor $f^*$ commutes with these operations on $X$ and $Y$ in the obvious sense. Applying this to the morphism $q^J$ we obtain, for any $\cF$ in $\Db_{\mathrm{c}}(\Flag, \K)$ and $\cG$ in $\Db_\BKM(\Flag^J, \K)$, a canonical isomorphism
\[
(q^J)^*(\cF \star^\BKM \cG) \cong \cF \star^\BKM (q^J)^* \cG.
\]
In particular, when $\cG=\underline{\K}_{\Flag^J_{\leq 1}}$, the functor $(-) \star^\BKM \cG$ identifies with the functor $(q^J)_*$. Using isomorphisms~\eqref{eqn:inverse-image-skyscraper} and~\eqref{eqn:convolution-constant}, we deduce the desired isomorphism of functors in the constructible setting. The equivariant setting is similar.

The second isomorphism is obtained by conjugating the first one with Verdier duality, using the facts that $q^J$ and $a^2_J$ are proper, while $a^1_J$ is smooth.

\eqref{it:qJ2}
We treat the case of $(q^J)^* (q^J)_*$; the case of $(q^J)^! (q^J)_*$ is similar.

By Lemma~\ref{lem:q^J_*-parity}, the functor $(q^J)_*$ sends even objects to even objects. On the other hand, it is clear that the functor $(q^J)^*$ sends $*$-even complexes to $*$-even complexes. Hence our functor sends even complexes to $*$-even complexes. Since this functor is Verdier self-dual up to shift by $2\ell(w_0^J)$ by~\eqref{it:qJ1}, it in fact sends even complexes to even complexes. The claim follows.
\end{proof}

In particular, for any $s \in S$, choosing $J=\{i_s\}$ we obtain the object
\[
\cE_s:= \underline{\K}_{\PKM_s/\BKM}[1].
\]
For any expression $\uw=s_1 \cdots s_r$, we set
\[
\cE_{\uw} := \cE_{s_1} \star^\BKM \cdots \star^\BKM \cE_{s_r}.
\]
It follows from~\eqref{eqn:convolution-constant} and Lemma~\ref{lem:qJ-parity} that
this object belongs to $\Par_{\BKM}(\Flag, \K)$. 

There exists a canonical graded algebra isomorphism
\begin{equation}
\label{eqn:equiv-cohom-pt}
\mathsf{H}^\bullet_{\BKM}(\mathrm{pt}; \K) \cong \mathrm{Sym}(\Lambda \otimes_{\Z} \K),
\end{equation}
where 
$\Lambda \otimes_\Z \K$ is placed in degree $2$. Moreover, if $\cF,\cG \in \Db_{\BKM}(\Flag,\K)$ are parity complexes, the $\mathsf{H}^\bullet_{\BKM}(\mathrm{pt}; \K)$-module $\Hom^\bullet_{\Db_{\BKM}(\Flag,\K)}(\cF,\cG)$ is free of finite rank, and the forgetful functor induces an isomorphism
\begin{equation}
\label{eqn:morphisms-parity-For}
\K \otimes_{\mathsf{H}^\bullet_{\BKM}(\mathrm{pt}; \K)} \Hom^\bullet_{\Db_{\BKM}(\Flag,\K)}(\cF,\cG) \simto \Hom^\bullet_{\Db_{(\BKM)}(\Flag,\K)}(\cF,\cG),
\end{equation}
see~\cite[Lemma~2.2]{mr}. 


Assume from now on that $\K$ is a field, or more generally a complete local ring.
By~\cite[Theorem~2.12 \& \S 4.1]{jmw},
for any $w \in W$ there exists (up to isomorphism)
a unique indecomposable parity complex
\[
\cE_w \in \Par_\BKM(\Flag,\K)
\]
which is characterized (among indecomposable parity complexes) by the properties that $\cE_w$ is supported on $\overline{\Flag_w}$ and that
\[
i_w^*(\cE_w) \cong \underline{\K}_{\Flag_w} [\ell(w)].
\]
By~\cite[Lemma~2.4]{mr}, the image of $\cE_w$ in $\Par_{(\BKM)}(\Flag,\K)$ is still indecomposable, hence it is isomorphic to the parity complex in $\Db_{(\BKM)}(\Flag,\K)$ characterized by the similar conditions (which will therefore also be denoted $\cE_w$). In particular, if $J \subset I$ is of finite type, $\cE_{w_0^J}$ is the object considered above.
(This object is canonical, whereas in general $\cE_w$ is defined only up to isomorphism.)


Similarly, given $s \in S$, for any $w \in W^s$ there exists a unique indecomposable parity complex
\[
\cE_w^s \in \Par_\BKM(\Flag^s, \K)
\]
which is characterized by the properties that $\cE_w^s$ is supported on $\overline{\Flag^s_w}$ and that
\[
(i_w^s)^*(\cE_w^s) \cong \underline{\K}_{\Flag^s_w} [\ell(w)].
\]
As above the image of $\cE_w^s$ in $\Par_{(\BKM)}(\Flag^s, \K)$ is indecomposable, and will still be denoted $\cE_w^s$.

If $w \in W^s$, then $\Flag^s_w$ is open in the support of $(q^s)_* \cE_w$, and the restriction of this complex to $\Flag^s_w$ is $\underline{\K}_{\Flag^s_w} [\ell(w)]$. Therefore,
$\cE_w^s$ is a direct summand of $(q^s)_* \cE_w$; we fix split embeddings and projections
\begin{equation}
\label{eqn:qs-parity-2}
\cE_w^s \to (q^s)_* \cE_w \to \cE_w^s.
\end{equation}

\begin{lem}
\label{lem:qs^*-parity}
Assume that $\K$ is a complete local ring.
The functors $(q^s)^*$ and $(q^s)^!$ send parity complexes to parity complexes (in both the $\BKM$-equivariant and the $\BKM$-constructible settings).
\end{lem}

\begin{proof}
It follows from the comments before the statement that any parity complex on $\Flag^s$ is isomorphic to a direct summand of an object of the form $(q^s)_* \cE$ where $\cE$ is a parity complex on $\Flag$. Hence the claim follows from Lemma~\ref{lem:qJ-parity}\eqref{it:qJ2}.
\end{proof}

As in the case of $(q^s)_*$, it follows from Lemma~\ref{lem:qs^*-parity} that for any $w \in W^s$, $\cE_{ws}$ is a direct summand of $(q^s)^! \cE_w^s [-1]$ and of $(q^s)^* \cE_w^s [1]$. We fix a split embedding and projection respectively
\begin{equation}
\label{eqn:qs-parity-1}
\cE_{ws} \to (q^s)^! \cE^s_w[-1],
\quad 
(q^s)^* \cE^s_w[1] \to \cE_{ws}.
\end{equation}

\begin{rmk}
Using the arguments in~\cite[Proposition~3.5]{williamson-IC}, one can prove that both morphisms in~\eqref{eqn:qs-parity-1} are in fact isomorphisms.
\end{rmk}


\subsection{Sections of the $!$-flag}

In the rest of this section we 
assume that $\K$ is a field, and we denote it by $\F$.
Recall that an object $\cF \in \Db_{(\BKM)}(\Flag,\F)$ is \emph{$!$-even} if for any $w \in W$ we have
\[
\cH^k(i_w^! \cF) = 0 \qquad \text{unless $k$ is even.}
\]
$\cF$ is called \emph{$!$-odd} if $\cF[1]$ is $!$-even. Finally, we will say that $\cF$ is \emph{$!$-parity} if it is the direct sum of a $!$-even and a $!$-odd object. (Note that this terminology is slightly different from the one used in~\cite{jmw}.)

\begin{defn}
Let $\cF \in \Db_{(\BKM)}(\Flag,\F)$ be $!$-parity. 
A \emph{section of the $!$-flag} of $\cF$ is a quadruple $(\Pi,e,d,(\varphi_\pi^\cF)_{\pi \in \Pi})$ where
\begin{itemize}
\item
$\Pi$ is a finite set;
\item
$e \colon \Pi \to W$ and $d \colon \Pi \to \Z$ are maps;
\item
for each $\pi \in \Pi$, $\varphi_\pi^\cF$ is an element in $\Hom_{\Db_{(\BKM)}(\Flag,\F)}(\cE_{e(\pi)},\cF[d(\pi)])$
\end{itemize}
such that for any $w \in W$ the images of the morphisms in
\[
\{\varphi^\cF_\pi \colon \cE_{w} \to \cF [d(\pi)] : \pi \in e^{-1}(w) \}
\]
form a basis of $\Hom^\bullet_{\Db_{(\BKM)}(\Flag_{\geq w},\F)}(\cE_w, \cF)$ (where we omit the functor of restriction from $\Flag$ to $\Flag_{\geq w}$ from the notation).
\end{defn}

\begin{rmk}
\begin{enumerate}
\item
We have an isomorphism
\[
\Hom^\bullet_{\Db_{(\BKM)}(\Flag_{\geq w},\F)}(\cE_w, \cF) \cong \mathsf{H}^{\bullet-\ell(w)}(\Flag_w, i_w^! \cF).
\]
In particular, $|e^{-1}(w)|=\dim_{\F} \mathsf{H}^{\bullet}(\Flag_w, i_w^! \cF)$.
\item
It follows from arguments similar to those in the proof of~\cite[Corollary~2.9]{jmw}
that sections of the $!$-flag always exist for 
$!$-parity objects.
\end{enumerate}
\end{rmk}

Of course, we have similar concepts for the ind-varieties $\Flag^s$ instead of $\Flag$ (replacing the objects $\cE_w$ by the objects $\cE_w^s$), where now $e$ takes values in $W^s$ (or equivalently in $W/W_s$); we will not repeat the definition.

\subsection{Sections of the $!$-flag and pushforward to $\Flag^s$}
\label{ss:section-!-flag-push}

We fix a simple reflection $s \in S$. Let 
$\cF \in \Db_{(\BKM)}(\Flag,\F)$ be $!$-parity, 
and let $(\Pi,e,d,(\varphi_\pi^\cF)_{\pi \in \Pi})$ be a section of the $!$-flag of $\cF$. 
The proof of Lemma~\ref{lem:q^J_*-parity} shows
that the object $\cG:=(q^s)_* \cF$ is $!$-parity, and the goal of this subsection is to explain how one can define a section of the $!$-flag for this object out of $(\Pi,e,d,(\varphi_\pi^\cF)_{\pi \in \Pi})$.

We set $\Pi':=\Pi$, and define $e' \colon \Pi' \to W/W_s$ as the composition of $e$ with the surjection $W \to W/W_s$. Now let $\pi \in \Pi'$. First, we assume that $e(\pi) > e(\pi)s$. Then we set $d'(\pi):=d(\pi)-1$, and we define $\varphi^{\cG}_\pi$ as the composition
\begin{multline*}
\cE_{e'(\pi)}^s \xrightarrow{\adj} (q^s)_* (q^s)^* \cE_{e'(\pi)}^s \to (q^s)_* \cE_{e(\pi)}[-1] \\
\xrightarrow{(q^s)_*(\varphi^{\cF}_\pi)[-1]} (q^s)_* \cF [d(\pi)-1]=\cG[d'(\pi)],
\end{multline*}
where 
the 
second morphism
is induced by 
the second map in~\eqref{eqn:qs-parity-1}. In other words, $\varphi^{\cG}_\pi$ is the image of the composition of $\varphi_\pi^\cF[-1]$ with the projection $(q^s)^* \cE_{e'(\pi)}^s \to \cE_{e(\pi)} [-1]$ under the isomorphism
\[
\Hom_{\Db_{(\BKM)}(\Flag,\F)}((q^s)^* \cE^s_{e'(\pi)}, \cF [d(\pi)-1]) \cong \Hom_{\Db_{(\BKM)}(\Flag^s,\F)}(\cE^s_{e'(\pi)}, \cG[d'(\pi)])
\]
induced by
adjunction.

Now, assume that $e(\pi)<e(\pi)s$, or in other words that $e(\pi) \in W^s$. Then we set $d'(\pi):=d(\pi)$, and we define $\varphi_\pi^\cG$ as the composition
\[
\cE_{e'(\pi)}^s \to (q^s)_* \cE_{e(\pi)} \xrightarrow{(q^s)_* (\varphi^\cF_\pi)} (q^s)_* \cF [d(\pi)] = \cG [d'(\pi)],
\]
where the first map is the first morphism in~\eqref{eqn:qs-parity-2}.

\begin{rmk}
\label{rmk:section-!-push}
An important point for us is that in both cases the morphism $\varphi^{\cG}_\pi$ factors through a shift of $(q^s)_* (\varphi^\cF_\pi) \colon (q^s)_* \cE_{e(\pi)} \to \cG[d(\pi)]$.
\end{rmk}

\begin{prop}
\label{prop:section-!-push}
The quadruple $(\Pi',e',d',(\varphi^{\cG}_\pi)_{\pi \in \Pi'})$ constructed above is a section of the $!$-flag of $\cG=(q^s)_* \cF$.
\end{prop}

Before proving Proposition~\ref{prop:section-!-push} in general we consider a special case.

\begin{lem}
\label{lem:section-!-push}
Let $w \in W$, and
assume that $\cF$ is isomorphic to a direct sum of shifts of $\Cos_w$. Then the quadruple $(\Pi',e',d',(\varphi^{\cG}_\pi)_{\pi \in \Pi'})$ constructed above is a section of the $!$-flag of $\cG$.
\end{lem}

\begin{proof}
Let us fix a non-zero morphism $f_w \colon \cE_w \to \Cos_w$ (which is unique up to scalar); then the morphism
\[
\bigoplus_{\pi \in \Pi} \varphi_\pi [-d(\pi)] \colon \bigoplus_{\pi \in \Pi} \cE_w[-d(\pi)] \to \cF
\]
factors through $\bigoplus_{\pi \in \Pi} f_w [-d(\pi)]$, and induces an isomorphism
\[
\bigoplus_{\pi \in \Pi} \Cos_w[-d(\pi)] \simto \cF;
\]
therefore we can assume that $\cF=\Cos_w$, $\Pi=\{w\}$, $e(w)=w$, $d(w)=0$ and $\varphi^{\cF}_w = f_w$. In this setting $\cG$ is isomorphic to $\Cos^s_w$ or $\Cos^s_w[1]$, so that to conclude it suffices to prove that $\varphi^\cG_w \neq 0$.

If $w>ws$, then the fact that $\varphi^\cG_w \neq 0$ follows from the facts that the composition of $\varphi^\cF_w [-1]$ with the projection $(q^s)^* \cE_{e'(\pi)}^s \to \cE_{e(\pi)} [-1]$ is nonzero, and that $\varphi^\cG_w$ is obtained from the latter morphism by adjunction. If $w<ws$, it is easy to see that $(i_w^s)^*(\varphi^{\cG}_w)$ is an isomorphism, hence is non-zero; it follows that $\varphi^\cG_w \neq 0$ also in this case.
\end{proof}

\begin{proof}[Proof of Proposition~{\rm \ref{prop:section-!-push}}]
We have to prove that for any $w \in W^s$ the images of the morphisms $\varphi^\cG_\pi$ with $e'(\pi)=w$ form a basis of $\Hom^\bullet_{\Db_{(\BKM)}(\Flag^s_{\geq w},\F)}(\cE_w^s, \cG)$. For this we fix $w$ and choose some closed subvariety $\mathscr{Y} \subset \Flag$ such that
\[
\mathscr{Y} \supset \Flag_w, \quad \mathscr{Y} \not\supset \Flag_{ws},
\]
and such that both
\[
\mathscr{Y}':=\mathscr{Y} \smallsetminus \Flag_w \quad \text{and} \quad \mathscr{Y}'':=\mathscr{Y} \cup \Flag_{ws}
\]
are closed subvarieties of $\Flag$. We denote by
\[
i \colon \mathscr{Y} \to \Flag, \quad i' \colon \mathscr{Y}' \to \Flag, \quad i'' \colon \mathscr{Y}'' \to \Flag
\]
the embeddings, so that we have natural morphisms
\[
(i')_! (i')^! \cF \to i_! i^! \cF \to (i'')_! (i'')^! \cF \to \cF
\]
induced by adjunction. Note that all of these objects are $!$-parity.

One can easily check, using the standard triangle
\[
i_! i^! \cF \to \cF \to j_* j^* \cF \triright
\]
where $j \colon \Flag \smallsetminus \mathscr{Y} \to \Flag$ is the (open) embedding, resp.~the similar triangle for $i''$, that every morphism $\varphi^\cF_\pi$ where $e(\pi)=w$, resp.~$e(\pi)=ws$, factors (uniquely) through a morphism
\[
\varphi_\pi^{i_! i^! \cF} \colon \cE_w \to i_! i^! \cF[d(\pi)], \quad \text{resp.} \quad \varphi_\pi^{(i'')_! (i'')^! \cF} \colon \cE_{ws} \to (i'')_! (i'')^! \cF[d(\pi)].
\]
Then by construction the corresponding morphism $\varphi^\cG_\pi$ factors through a morphism
\begin{multline*}
\varphi_\pi^{(q^s)_* i_! i^! \cF} \colon \cE^s_w \to (q^s)_* i_! i^! \cF [d'(\pi)], \\
\text{resp.} \quad \varphi_\pi^{(q^s)_* (i'')_! (i'')^! \cF} \colon \cE^s_w \to (q^s)_* (i'')_! (i'')^! \cF[d'(\pi)].
\end{multline*}
It is clear also that the morphisms obtained in this way coincide with the morphisms obtained by the above procedure from the section of the $!$-flag of $i_! i^! \cF$, resp.~$(i'')_! (i'')^! \cF$, obtained (in the obvious way) from $(\Pi,e,d,(\varphi_\pi^\cF)_{\pi \in \Pi})$ by restriction to $\mathscr{Y}$, resp.~to $\mathscr{Y}''$.

Now if $e(\pi)=w$, resp.~$e(\pi)=ws$, we can consider the composition
\[
\varphi_\pi^{(i_w)_* (i_w)^! \cF} \colon \cE_w \xrightarrow{\varphi_\pi^{i_! i^! \cF}} i_! i^! \cF[d(\pi)] \to (i_w)_* (i_w)^! \cF[d(\pi)]
\]
where the second morphism is induced by the adjunction $\bigl( (i_w)^*, (i_w)_* \bigr)$, resp.~the composition
\[
\varphi_\pi^{(i_{ws})_* (i_{ws})^! \cF} \colon \cE_{ws} \xrightarrow{\varphi_\pi^{(i'')_! (i'')^! \cF}} (i'')_! (i'')^! \cF[d(\pi)] \to (i_{ws})_* (i_{ws})^! \cF[d(\pi)]
\]
where the second morphism is induced by the adjunction $\bigl( (i_{ws})^*, (i_{ws})_* \bigr)$, and the corresponding morphisms $\varphi_\pi^{(q^s)_* (i_w)_* (i_w)^! \cF}$, resp.~$\varphi_\pi^{(q_s)_* (i_{ws})_* (i_{ws})^! \cF}$, obtained by the procedure above, which can also be described as the compositions
\[
\cE^s_w \xrightarrow{\varphi_\pi^{(q^s)_* i_! i^! \cF}} (q^s)_* i_! i^! \cF [d'(\pi)] \to (q^s)_* (i_w)_* (i_w)^! \cF [d'(\pi)],
\]
resp.
\[
\cE^s_w \xrightarrow{\varphi_\pi^{(q^s)_* (i'')_! (i'')^! \cF}} (q^s)_* (i'')_! (i'')^! \cF[d'(\pi)] \to (q^s)_* (i_{ws})_* (i_{ws})^! \cF [d'(\pi)].
\]

Finally, let us denote by $i'_{w} \colon \Flag_w \sqcup \Flag_{ws} \to \Flag$ the inclusion. Then if $e(\pi)=w$, resp.~$e(\pi)=ws$, we can consider the composition
\[
\varphi_\pi^{(i'_w)_* (i'_w)^! \cF} \colon \cE_w \xrightarrow{\varphi_\pi^{(i_w)_* (i_w)^! \cF}} (i_w)_* (i_w)^! \cF[d(\pi)] \to (i'_w)_* (i'_w)^! \cF[d(\pi)],
\]
resp.~the composition
\[
\varphi_\pi^{(i'_w)_* (i'_w)^! \cF} \colon \cE_{ws} \xrightarrow{\varphi_\pi^{(i'')_! (i'')^! \cF}} (i'')_! (i'')^! \cF[d(\pi)] \to (i'_w)_* (i'_w)^! \cF[d(\pi)],
\]
where in both cases the second morphism is again induced by adjunction, and the corresponding morphisms
\[
\varphi_\pi^{(q^s)_* (i'_w)_* (i'_w)^! \cF} \colon \cE^s_w \to (q^s)_* (i'_w)_* (i'_w)^! \cF[d'(\pi)].
\]

Now, consider the natural distinguished triangle
\[
(i_w)_* (i_w)^! \cF \to (i'_w)_* (i'_w)^! \cF \to (i_{ws})_* (i_{ws})^! \cF \triright
\]
and its image 
\[
(q^s)_* (i_w)_* (i_w)^! \cF \to (q^s)_* (i'_w)_* (i'_w)^! \cF \to (q^s)_* (i_{ws})_* (i_{ws})^! \cF \triright
\]
under the functor $(q^s)_*$. Since the image under $(q^s)_*$ of any morphism $\Cos_{ws} \to \Cos_w [k]$ is zero, this triangle is split.
Hence, taking the image under the functor $\Hom_{\Db_{(\BKM)}(\Flag^s_{\geq w}, \F)}(\cE_w^s, -)$ we obtain an exact sequence of $\F$-vector spaces
\begin{multline*}
0 \to \Hom^\bullet_{\Db_{(\BKM)}(\Flag^s_{\geq w}, \F)}(\cE_w^s, (q^s)_* (i_w)_* (i_w)^! \cF) \\
\to \Hom^\bullet_{\Db_{(\BKM)}(\Flag^s_{\geq w}, \F)}(\cE_w^s, (q^s)_* (i'_w)_* (i'_w)^! \cF) \\
\to \Hom^\bullet_{\Db_{(\BKM)}(\Flag^s_{\geq w}, \F)}(\cE_w^s, (q^s)_* (i_{ws})_* (i_{ws})^! \cF) \to 0.
\end{multline*}
It follows from Lemma~\ref{lem:section-!-push} that the images of the morphisms $\varphi_\pi^{(q^s)_* (i_w)_* (i_w)^! \cF}$, resp.~$\varphi_\pi^{(q^s)_* (i_{ws})_* (i_{ws})^! \cF}$, with $e(\pi)=w$, resp.~$e(\pi)=ws$, form a basis of the first, respectively third, term in this exact sequence. Hence the morphisms $\varphi_\pi^{(q^s)_* (i'_w)_* (i'_w)^! \cF}$ with $e(\pi) \in \{w,ws\}$ form a basis of the middle term. Now we remark that the natural morphisms
\[
(i'')_! (i'')^! \cF \to \cF \quad \text{and} \quad (i'')_! (i'')^! \cF \to (i_w')_* (i_w')^! \cF
\]
induce isomorphisms
\begin{multline*}
\Hom^\bullet_{\Db_{(\BKM)}(\Flag^s_{\geq w}, \F)}(\cE_w^s, (q^s)_* \cF) \simto \Hom^\bullet_{\Db_{(\BKM)}(\Flag^s_{\geq w}, \F)}(\cE_w^s, (q^s)_* (i'')_! (i'')^! \cF) \\
\simto \Hom^\bullet_{\Db_{(\BKM)}(\Flag^s_{\geq w}, \F)}(\cE_w^s, (q^s)_* (i'_w)_* (i'_w)^! \cF),
\end{multline*}
and we deduce the desired claim.
\end{proof}

\subsection{Sections of the $!$-flag and pullback from $\Flag^s$}
\label{ss:section-!-flag-pull}

As in~\S\ref{ss:section-!-flag-push} we fix a simple reflection $s \in S$. Let $\cF \in \Db_{(\BKM)}(\Flag^s,\F)$ be an object which is $!$-parity, and let $(\Pi,e,d,(\varphi_\pi^\cF)_{\pi \in \Pi})$ be a section of the $!$-flag of $\cF$. It is 
clear
that the object $\cH:=(q^s)^! \cF$ is $!$-parity, and the goal of this subsection is to explain how one can define a section of the $!$-flag for this object out of $(\Pi,e,d,(\varphi_\pi^\cF)_{\pi \in \Pi})$.

We set $\Pi':=\Pi \times \{0,1\}$. We define a map $e' \colon \Pi' \to W$ as follows. Given $\pi \in \Pi$, the elements $e'(\pi,0)$ and $e'(\pi,1)$ are characterized by the following properties:
\begin{itemize}
\item
the images in $W/W_s$ of both $e'(\pi,0)$ and $e'(\pi,1)$ are equal to $e(\pi)$;
\item
$e'(\pi,1) = e'(\pi,0)s$;
\item
$e'(\pi,0) < e'(\pi,1)$ in the Bruhat order.
\end{itemize}
Then we define a map $d' \colon \Pi' \to \Z$ and morphisms $\varphi^{\cH}_{(\pi,0)}$ and $\varphi^{\cH}_{(\pi,1)}$ for $\pi \in \Pi$ as follows. First, we set $d'(\pi,1)=d(\pi)-1$, and $\varphi^{\cH}_{(\pi,1)}$ is defined as the composition
\[
\cE_{e'(\pi,1)} \to (q^s)^! \cE^s_{e(\pi)}[-1] \xrightarrow{(q^s)^! (\varphi^{\cF}_\pi) [-1]} (q^s)^! \cF[d(\pi)-1] = \cH[d'(\pi,1)],
\]
where the first morphism is the first map in~\eqref{eqn:qs-parity-1}.
On the other hand, we set $d'(\pi,0)=d(\pi)$ and we define $\varphi^{\cH}_{(\pi,0)}$ as the composition
\begin{multline*}
\cE_{e'(\pi,0)} \xrightarrow{\adj} (q^s)^! (q^s)_! \cE_{e'(\pi,0)} = (q^s)^! (q^s)_* \cE_{e'(\pi,0)} \to \\
(q^s)^! \cE^s_{e(\pi)} \xrightarrow{(q^s)^! (\varphi^{\cF}_\pi)} (q^s)^! \cF[d(\pi)] = \cH[d'(\pi,0)].
\end{multline*}
Here the 
the third morphism is induced by the second map in~\eqref{eqn:qs-parity-2}.

\begin{rmk}
\label{rmk:section-!-pull}
As in Remark~\ref{rmk:section-!-push},
an important point for us is that both $\varphi^{\cH}_{(\pi,0)}$ and $\varphi^{\cH}_{(\pi,1)}$ factor through a shift of the morphism $(q^s)^! (\varphi^\cF_\pi) \colon (q^s)^! \cE^s_{e(\pi)} \to \cH[d(\pi)]$.
\end{rmk}

\begin{prop}
\label{prop:section-!-pull}
The quadruple $(\Pi',e',d',(\varphi^{\cH}_\pi)_{\pi \in \Pi'})$ constructed above is a section of the $!$-flag of $\cH=(q^s)^! \cF$.
\end{prop}

Before proving Proposition~\ref{prop:section-!-pull} in general we consider a special case.

\begin{lem}
\label{lem:section-!-pull}
Let $w \in W^s$, and
assume that $\cF$ is isomorphic to a direct sum of shifts of $\Cos^s_w$. Then the quadruple $(\Pi',e',d',(\varphi^{\cH}_\pi)_{\pi \in \Pi'})$ constructed above is a section of the $!$-flag of $\cH$.
\end{lem}

\begin{proof}
As in the proof of Lemma~\ref{lem:section-!-push}, we can assume that $\cF=\Cos^s_w$, $\Pi=\{w\}$, $e(w)=w$, $d(w)=0$. In this case, since $(q^s)^{-1}(\Flag^s_w)$ is the disjoint union of the closed subset $\Flag_w$ and the open subset $\Flag_{ws}$, adjunction provides a canonical distinguished triangle
\begin{equation*}
\Cos_w[-2] \to \cH \to \Cos_{ws}[-1] \triright.
\end{equation*}
Moreover, the morphisms in this triangle induce isomorphisms
\begin{align*}
\Hom^\bullet_{\Db_{(\BKM)}(\Flag_{\geq w},\F)}(\cE_w, \cH) &\cong \Hom^{\bullet-2}_{\Db_{(\BKM)}(\Flag_{\geq w},\F)}(\cE_w, \Cos_w), \\
\Hom^\bullet_{\Db_{(\BKM)}(\Flag_{\geq ws},\F)}(\cE_{ws}, \cH) &\cong \Hom^{\bullet-1}_{\Db_{(\BKM)}(\Flag_{\geq ws},\F)}(\cE_{ws}, \Cos_{ws}),
\end{align*}
and all of these spaces are $1$-dimensional.
Hence to conclude it suffices to prove that the restriction of $\varphi^{\cH}_{(\pi,0)}$ to $\Flag_w$ and the restriction of $\varphi^{\cH}_{(\pi,1)}$ to $\Flag_{ws}$ are non-zero; in both cases this is clear from construction.
\end{proof}

\begin{proof}[Proof of Proposition~{\rm \ref{prop:section-!-pull}}]
The proof is very similar to, and in fact simpler than, that of Proposition~\ref{prop:section-!-push}; details are therefore left to the reader.
\end{proof}

\subsection{Morphisms between ``Bott--Samelson type'' parity complexes}
\label{ss:morphisms-BS-parity}

We set
\[
\Upsilon_s := (q^s)^! (q^s)_*[-1] \colon \Db_{(\BKM)}(\Flag,\F) \to \Db_{(\BKM)}(\Flag,\F).
\]
%
By Lemma~\ref{lem:qJ-parity}\eqref{it:qJ1} and~\eqref{eqn:convolution-constant}, we have a canonical isomorphism
\begin{equation}
\label{eqn:Upsilon-convolution}
\Upsilon_s \cong (-) \star^{\BKM} \cE_s.
\end{equation}



Recall the parity complexes $\cE_{\uw}$ (where $\uw$ is an expression) defined in~\S\ref{ss:parity-flag}.

\begin{prop}
\label{prop:morphisms-BS-parity}
Let $\ux$ and $\uv$ be expressions, and assume that $\ux$ is a reduced expression for some element $x \in W$.
\begin{enumerate}
\item
\label{it:morphisms-BS-parity-1}
Assume that $x \in W^s$, so that $\ux s$ is a reduced expression for $xs \in W$. Let $(f_i)_{i \in \mathfrak{I}}$ be a family of homogeneous elements in $\Hom^\bullet_{\Db_{(\BKM)}(\Flag,\F)}(\cE_{\ux}, \cE_{\uv})$ whose images span the vector space $\Hom^\bullet_{\Db_{(\BKM)}(\Flag_{\geq x},\F)}(\cE_{\ux}, \cE_{\uv})$, and let $(g_j)_{j \in \mathfrak{J}}$ be a family of homogeneous elements in $\Hom^\bullet_{\Db_{(\BKM)}(\Flag,\F)}(\cE_{\ux s}, \cE_{\uv})$ whose images span the vector space $\Hom^\bullet_{\Db_{(\BKM)}(\Flag_{\geq xs},\F)}(\cE_{\ux s}, \cE_{\uv})$. Then there exist integers $n_i$ and morphisms $f'_i \colon \cE_{\ux} \to \cE_{\ux s}[n_i]$ (for $i \in \mathfrak{I}$) and integers $m_j$ and morphisms $g'_j \colon \cE_{\ux} \to \cE_{\ux s s}[m_j]$ (for $j \in \mathfrak{J}$) such that the images of the compositions
\[
\cE_{\ux} \xrightarrow{f'_i} \cE_{\ux s}[n_i] \xrightarrow{f_i \star^\BKM \cE_s [n_i]} \cE_{\uv s} [n_i + \deg(f_i)]
\]
together with the images of the compositions
\[
\cE_{\ux} \xrightarrow{g'_j} \cE_{\ux s s}[m_j] \xrightarrow{g_j \star^\BKM \cE_s[m_j]} \cE_{\uv s} [m_j + \deg(g_j)]
\]
span the vector space $\Hom^\bullet_{\Db_{(\BKM)}(\Flag_{\geq x},\F)}(\cE_{\ux}, \cE_{\uv s})$.
\item
\label{it:morphisms-BS-parity-2}
Assume that $\ux=\uy s$ for some expression $\uy$ (which is automatically a reduced expression for $xs$). Let $(f_i)_{i \in \mathfrak{I}}$ be a family of homogeneous elements in $\Hom^\bullet_{\Db_{(\BKM)}(\Flag,\F)}(\cE_{\ux}, \cE_{\uv})$ whose images span $\Hom^\bullet_{\Db_{(\BKM)}(\Flag_{\geq x},\F)}(\cE_{\ux}, \cE_{\uv})$, and let $(g_j)_{j \in \mathfrak{J}}$ be a family of homogeneous elements in $\Hom^\bullet_{\Db_{(\BKM)}(\Flag,\F)}(\cE_{\uy}, \cE_{\uv})$ whose images span $\Hom^\bullet_{\Db_{(\BKM)}(\Flag_{\geq xs},\F)}(\cE_{\uy}, \cE_{\uv})$. Then there exist integers $n_i$ and morphisms $f'_i \colon \cE_{\ux} \to \cE_{\ux s}[n_i]$ (for $i \in \mathfrak{I}$) and integers $m_j$ and morphisms $g'_j \colon \cE_{\ux} \to \cE_{\ux}[m_j]$ (for $j \in \mathfrak{J}$) such that the images of the compositions
\[
\cE_{\ux} \xrightarrow{f'_i} \cE_{\ux s}[n_i] \xrightarrow{f_i \star^\BKM \cE_s[n_i]} \cE_{\uv s} [n_i + \deg(f_i)]
\]
together with the images of the compositions
\[
\cE_{\ux} \xrightarrow{g'_j} \cE_{\ux}[m_j] \xrightarrow{g_j \star^\BKM \cE_s[m_j]} \cE_{\uv s} [m_j + \deg(g_j)]
\]
span the vector space $\Hom^\bullet_{\Db_{(\BKM)}(\Flag_{\geq x},\F)}(\cE_{\ux}, \cE_{\uv s})$.
\end{enumerate}
\end{prop}

\begin{proof}
\eqref{it:morphisms-BS-parity-1}
We have $\cE_{\ux} \cong \cE_x$ in $\Db_{(\BKM)}(\Flag_{\geq x}, \F)$, and $\cE_{\ux s} \cong \cE_{xs}$ in $\Db_{(\BKM)}(\Flag_{\geq xs}, \F)$. Hence we can fix split embeddings $\cE_x \to \cE_{\ux}$ and $\cE_{xs} \to \cE_{\ux s}$, and assume that the compositions
\[
\cE_x \to \cE_{\ux} \xrightarrow{f_i} \cE_{\uv}[\deg(f_i)] \quad \text{and} \quad \cE_{xs} \to \cE_{\ux s} \xrightarrow{g_j} \cE_{\uv}[\deg(g_j)]
\]
are part of a section of the $!$-flag of $\cE_{\uv}$. Then Proposition~\ref{prop:section-!-push} provides a section of the $!$-flag of $(q^s)_* \cE_{\uv}$ whose morphisms in $\Hom^\bullet(\cE^s_x, (q^s)_* \cE_{\uv})$ are parametrized by $\mathfrak{I} \sqcup \mathfrak{J}$, in such a way that the morphism associated with $i \in \mathfrak{I}$ factors through a shift of $(q^s)_* (f_i) \colon (q^s)_* \cE_{\ux} \to (q^s)_* \cE_{\uv}[\deg(f_i)]$, and the morphism associated with $j \in \mathfrak{J}$ factors through a shift of $(q^s)_* (g_j) \colon (q^s)_* \cE_{\ux s} \to (q^s)_* \cE_{\uv}[\deg(g_j)]$ (see in particular Remark~\ref{rmk:section-!-push}). Applying Proposition~\ref{prop:section-!-pull}, we then obtain a section of the $!$-flag of $\Upsilon_s \cE_{\uv}$ whose morphisms in $\Hom^\bullet(\cE_x, \Upsilon_s \cE_{\uv})$ are parametrized by $\mathfrak{I} \sqcup \mathfrak{J}$, in such a way that the morphism associated with $i \in \mathfrak{I}$ factors through a shift of $\Upsilon_s (f_i) \colon \Upsilon_s \cE_{\ux} \to \Upsilon_s \cE_{\uv}[\deg(f_i)]$, and the morphism associated with $j \in \mathfrak{J}$ factors through a shift of $\Upsilon_s (g_j) \colon \Upsilon_s \cE_{\ux s} \to \Upsilon_s \cE_{\uv}[\deg(g_j)]$ (see in particular Remark~\ref{rmk:section-!-pull}). Composing with an arbitrarily chosen split projection $\cE_{\ux} \to \cE_x$ and using~\eqref{eqn:Upsilon-convolution},
we deduce the desired claim.

\eqref{it:morphisms-BS-parity-2} The proof is identical to the proof of~\eqref{it:morphisms-BS-parity-1}, and is therefore omitted.
\end{proof}

\section{Parity complexes and the Hecke category}
\label{sec:parity-Hecke}

\subsection{Diagrammatic category associated with $\GKM$}
\label{ss:Diag-GKM}

We now define a realization of $(W,S)$ over $\Z$ (in the sense of~\cite[Definition~3.1]{ew}) as follows:
\begin{itemize}
\item
the underlying $\Z$-module is $\fh^\Z_{\GKM} := \Hom_{\Z}(\Lambda, \Z)$;
\item
for any $s \in S$, the elements ``$\alpha_s$'' and ``$\alpha_s^\vee$'' are the simple root and coroot attached to $s$ respectively.
\end{itemize}
This realization is balanced in the sense of~\cite[Definition~3.6]{ew}.

For any integral domain $\K$ one can consider the realization $\fh_{\GKM}^\K:=\fh_{\GKM}^\Z \otimes_\Z \K$ over $\K$. If the ``Demazure surjectivity'' condition~\cite[Assumption~3.7]{ew} holds for $\fh_{\GKM}^\Z$, we set $\Z':=\Z$; otherwise we set $\Z':=\Z[\frac{1}{2}]$. Then Demazure surjectivity holds for $\fh_{\GKM}^{\Z'}$, hence for $\fh_{\GKM}^\K$ for any integral domain $\K$ such that there exists a ring morphism $\Z' \to \K$. For such a ring, we denote by
$\DiagBS^\K(\GKM)$ the diagrammatic category of~\cite[Definition~5.2]{ew} associated with this realization. (The definition of this category is similar to the definition of the category $\DiagBS$ in~\S\ref{ss:diag-SB}. In particular, the morphisms are generated by the same diagrams as in~\S\ref{ss:diag-SB}, with now $S$ and $W$ defined as in~\S\ref{ss:KMgroups}, and $\fh$ replaced by $\fh_{\GKM}^\K$.)

By construction, for $\K$ as above and any $M,N$ in $\DiagBS^\K(\GKM)$, the morphism space
\[
\Hom^\bullet_{\DiagBS^\K(\GKM)}(M,N) :=
\bigoplus_{i \in \Z} \Hom_{\DiagBS^\K(\GKM)}(M,N \langle i \rangle)
\]
is a graded bimodule over
\begin{equation}
\label{eqn:equiv-cohom-K}
\cO(\fh_{\GKM}^\K) = \mathrm{Sym} \bigl( \K \otimes_\Z \Lambda \bigr).
\end{equation}
Note also that if $\K \to \K'$ is a ring morphism (where again $\K'$ is an integral domain), then there exists a natural functor $\DiagBS^\K(\GKM) \to \DiagBS^{\K'}(\GKM)$ which is the identity on objects. Moreover, this functor induces an isomorphism
\begin{equation}
\label{eqn:Diag-K-K'}
\K' \otimes_{\K} \Hom_{\DiagBS^\K(\GKM)}(M,N) \simto \Hom_{\DiagBS^{\K'}(\GKM)}(M,N)
\end{equation}
for any $M,N$ in $\DiagBS^\K(\GKM)$. (In fact, this follows from the double leaves theorem~\cite[Theorem~6.11]{ew}.)

If $\K$ is in addition a complete local ring, we denote by $\Diag^\K(\GKM)$ the Karoubi envelope of the additive hull of $\DiagBS^\K(\GKM)$. We define $\Hom^\bullet_{\Diag^\K(\GKM)}(M,N)$ in the obvious way.

\begin{rmk}
In this case, contrary to the situation in~\S\ref{ss:diag-SB}, in general the roots do \emph{not} generate $\K \otimes_\Z \Lambda$. Hence we need to consider the ``polynomial'' morphisms in $\DiagBS^\K(\GKM)$; they cannot be expressed in terms of the other generators in general.
\end{rmk}

\subsection{More on Bott--Samelson parity complexes}
\label{ss:BS-parity}

Let us come back to the setting of~\S\ref{ss:DbX}.
Below we will also use an analogue of the construction of $\star^\BKM$ in the case of $3$ variables: given $\cF_1, \cF_2, \cF_3$ in $\Db_{\BKM}(\Flag,\K)$, we set
\[
\mathsf{Conv}_3(\cF_1, \cF_2, \cF_3) := \mathsf{m}_{3*} (\cF_1 \, \widetilde{\boxtimes} \, \cF_2 \, \widetilde{\boxtimes} \, \cF_3),
\]
where $\mathsf{m}_3 \colon \GKM \times^\BKM \GKM \times^\BKM \Flag \to \Flag$ is the morphism induced by multiplication in $\GKM$, and $\cF_1 \, \widetilde{\boxtimes} \, \cF_2 \, \widetilde{\boxtimes} \, \cF_3$ is the unique object in $\Db_{\BKM}(\GKM \times^\BKM \GKM \times^\BKM \Flag, \K)$ whose pullback to $\GKM \times \GKM \times \Flag$ is $\mu^*\cF_1 \boxtimes \mu^*\cF_2 \boxtimes \cF_3$. Of course, there exist canonical isomorphisms
\[
(\cF_1 \star^\BKM \cF_2) \star^\BKM \cF_3 \simto \mathsf{Conv}_3(\cF_1, \cF_2, \cF_3) \simto \cF_1 \star^\BKM (\cF_2 \star^\BKM \cF_3)
\]
whose composition is the associativity constraint for the product $\star^\BKM$.

The ``Bott--Samelson parity complex" $\cE_{\uw}$ defined in~\S\ref{ss:parity-flag} is defined only up to (canonical) isomorphism, since one needs to choose the order in which the convolution products are taken. To remedy this we introduce a \emph{canonical} object $\cE_{\K}(\uw)$ as follows.
For this, recall the Demazure (or Bott--Samelson) resolution $\nu_{\uw} \colon \BSvar(\uw) \to \Flag$ 
defined in~\S\ref{ss:KMgroups}.
Then we set
\[
\cE_\K(\uw) := (\nu_{\uw})_* \underline{\K}_{\BSvar(\uw)} [\ell(\uw)].
\]

\begin{lem}
\label{lem:convolution-Ew}
For any expressions $\uw$ and $\uv$, there exists a canonical isomorphism
\[
\cE_\K(\uw) \star^\BKM \cE_\K(\uv) \cong \cE_\K(\uw \uv),
\]
where $\uw\uv$ is the concatenation of $\uw$ and $\uv$.
\end{lem}

\begin{proof}
Write $\uw=s_1 \cdots s_r$.
It can be easily checked that the complex $\cE_\K(\uw) \, \widetilde{\boxtimes} \, \cE_\K(\uv)$ is canonically isomorphic to (the extension by 0 of) the direct image of the constant sheaf on
\[
\bigl( \PKM_{s_1} \times^\BKM \cdots \times^\BKM \PKM_{s_r} \bigr) \times^\BKM \BSvar(\uv) = \BSvar(\uw\uv)
\]
under the natural projection to $\GKM \times^\BKM \Flag$. The composition of this projection with
the morphism induced by $\mathsf{m}$ identifies with $\nu_{\uw \uv}$, and the claim follows.
\end{proof}

We will consider the category
\[
\Par_{\BKM}^{\mathrm{BS}}(\Flag, \K)
\]
whose objects are the pairs $(\uw, n)$ where $\uw$ is an expression and $n \in \Z$, and whose morphisms are defined as follows:
\[
\Hom_{\Par_{\BKM}^{\mathrm{BS}}(\Flag, \K)} \bigl( (\uw, n), (\uv, m) \bigr) := \Hom_{\Db_{\BKM}(\Flag,\K)} \bigl( \cE_\K(\uw)[n], \cE_\K(\uv)[m] \bigr).
\]
We endow this category with a monoidal product $\star$ by declaring that
\[
(\uw,n) \star (\uv,m) := (\uw\uv, n+m)
\]
and using the canonical isomorphism in Lemma~\ref{lem:convolution-Ew} to define the product of morphisms.


If $\K$ and $\K'$ are Noetherian commutative rings of finite global dimension, and if we are given a ring morphism $\K \to \K'$, then we have an ``extension of scalars'' functor
\[
\K' := \K' \lotimes_{\K} (-) \colon \Db_{\BKM}(\Flag,\K) \to \Db_{\BKM}(\Flag,\K').
\]
Standard compatibility properties of extension of scalars with $!$-pushforward functors (see e.g.~\cite[Proposition~2.6.6]{ks1} in the classical context) show that for any expression $\uw$ there exists a canonical isomorphism $\K'(\cE_\K(\uw)) \cong \cE_{\K'}(\uw)$. In particular, it follows that the functor $\K'$ defines in a natural way a monoidal functor from $\Par_{\BKM}^{\mathrm{BS}}(\Flag, \K)$ to $\Par_{\BKM}^{\mathrm{BS}}(\Flag, \K')$, which we will denote by the same symbol. The following result is standard, see e.g.~\cite[Lemma~2.2(2)]{mr}.

\begin{lem}
\label{lem:morphisms-parityBS}
For any expressions $\uw$ and $\uv$ and any $n,m \in \Z$, the $\K$-module
\[
\Hom_{\Par_{\BKM}^{\mathrm{BS}}(\Flag, \K)} \bigl( (\uw, n), (\uv, m) \bigr)
\]
is free, and the functor $\K'$ induces an isomorphism
\[
\K' \otimes_\K \Hom_{\Par_{\BKM}^{\mathrm{BS}}(\Flag, \K)} \bigl( (\uw, n), (\uv, m) \bigr) \simto \Hom_{\Par_{\BKM}^{\mathrm{BS}}(\Flag, \K')} \bigl( (\uw, n), (\uv, m) \bigr).
\]
\end{lem}

Assume now in addition that $\K$ is a 
complete local ring, and recall the category $\Par_{\BKM}(\Flag, \K)$ defined in~\S\ref{ss:parity-flag}.
We can define a natural fully-faithful monoidal functor
\begin{equation}
\label{eqn:functor-parityBS}
\Par_{\BKM}^{\mathrm{BS}}(\Flag, \K) \to \Par_{\BKM}(\Flag, \K)
\end{equation}
sending $(\uw,n)$ to $\cE_\K(\uw)[n]$. In this case any indecomposable object in the category $\Par_{\BKM}(\Flag, \K)$ is isomorphic to a direct summand of an object $\cE_\K(\uw)[n]$, see~\cite{jmw}. We deduce the following lemma.

\begin{lem}
\label{lem:parity-Karoubi}
Assume that $\K$ is a Noetherian commutative complete local ring. The functor~\eqref{eqn:functor-parityBS} realizes $\Par_{\BKM}(\Flag,\K)$, as a monoidal category, as the Karou\-bi envelope of the additive hull of the monoidal category $\Par_{\BKM}^{\mathrm{BS}}(\Flag, \K)$.
\end{lem}

\subsection{Statement of the equivalences}
\label{ss:statement-Diag-parity}

We continue with the setting of~\S\ref{ss:DbX} and~\S\ref{ss:BS-parity}, and assume in addition that $\K$ is an integral domain.
The main result of this subsection is the following.

\begin{thm}
\label{thm:main-parity}
Assume that $\K$ is a 
complete local ring and that there exists a ring morphism $\Z' \to \K$. Then there exists an equivalence of additive monoidal categories
\[
\Delta \colon \Diag^\K(\GKM) \simto \Par_{\BKM}(\Flag,\K).
\]
\end{thm}

In view of Lemma~\ref{lem:parity-Karoubi} and the construction of the category $\Diag^\K(\GKM)$, Theorem~\ref{thm:main-parity} will follow from the following result, which applies to more general coefficients.

\begin{thm}
\label{thm:main-parity-BS}
Assume that there exists a ring morphism $\Z' \to \K$. Then there exists an equivalence of monoidal categories
\[
\Delta_{\mathrm{BS}} \colon \DiagBS^\K(\GKM) \simto \Par^{\mathrm{BS}}_{\BKM}(\Flag,\K).
\]
\end{thm}

The remainder of this section is devoted to the proof of Theorem~\ref{thm:main-parity-BS}: in~\S\S\ref{ss:construction-Diag-parity}--\ref{ss:verification-Diag-parity} we construct a monoidal functor $\Delta_{\mathrm{BS}} \colon \DiagBS^\K(\GKM) \to \Par^{\mathrm{BS}}_{\BKM}(\Flag,\K)$ which is obviously essentially surjective. Then in~\S\ref{ss:Diag-parity-ff} we show that this functor is fully-faithful, which completes the proof of Theorem~\ref{thm:main-parity-BS}, hence also of Theorem~\ref{thm:main-parity}.

To fix notation we assume that we work in the classical setting. The \'etale setting can be treated in a similar way, replacing $\Z'$ by a ring of $p$-adic integers.

\subsection{Construction of the functor $\Delta_{\mathrm{BS}}$}
\label{ss:construction-Diag-parity}

\subsubsection{Principle of the construction}

In this subsection we assume that $\K$ is a Noetherian integral domain of finite global dimension, and
that there exists a ring morphism $\Z' \to \K$.
Our goal is to construct a monoidal functor
\[
\Delta_{\mathrm{BS}} \colon \DiagBS^\K(\GKM) \to \Par^{\mathrm{BS}}_{\BKM}(\Flag,\K).
\]
The definition on objects is obvious: we simply set $\Delta_{\mathrm{BS}}(B_{\uw} \langle n \rangle) = (\uw, n)$. To define $\Delta_{\mathrm{BS}}$ on morphisms, we will explain how to define the image of a morphism $\phi \colon B_{\uw} \langle n \rangle \to B_{\uw'} \langle n \rangle$ where $\uw'$ is either equal to $\uw$ or obtained from $\uw$ by one of the substitutions
\begin{equation}
\label{eqn:elementary-morphisms}
s \dashrightarrow \varnothing, \ \ \varnothing \dashrightarrow s, \ \ ss \dashrightarrow s, \ \ s \dashrightarrow ss, \ \ st \cdots \dashrightarrow ts \cdots
\end{equation}
(where $s,t \in S$, $s \neq t$, and $st$ has finite order $m_{st}$, the number of terms on each side of the last substitution being $m_{st}$), and $\phi$ is induced by the corresponding ``elementary'' morphism (polynomial, upper dot, lower dot, trivalent morphism or $2m_{st}$-valent morphism). Then we will check that these images satisfy the relations from~\cite{ew}. In fact we will only consider the case when $\K=\Z'$. Then using Lemma~\ref{lem:morphisms-parityBS} one can deduce the definition of the morphisms in the case of any $\K$, and the fact that the relations hold over $\Z'$ implies that they also hold over $\K$.

We only need to define the images of the morphisms associated with the polynomials and the substitutions in~\eqref{eqn:elementary-morphisms}. For instance, if one knows the definition of the image $\psi \colon \cE(s) \to \cE(\varnothing)[1]$ of the ``upper dot morphism'' $B_s \to B_\varnothing \langle 1 \rangle$, for any expressions $\uu$ and $\uv$ one defines the image of the induced morphism $B_{\uu s \uv} \to B_{\uu \uv} \langle 1 \rangle$ as the composition
\begin{multline*}
\cE(\uu s \uv) \simto \mathsf{Conv}_3(\cE(\uu), \cE(s), \cE(\uv)) \\
\xrightarrow{\mathsf{Conv}_3(\cE(\uu), \psi, \cE(\uv))} \mathsf{Conv}_3(\cE(\uu), \cE(\varnothing), \cE(\uv)) [1] \simto \cE(\uu \uv)[1]
\end{multline*}
where the first and third morphisms are the canonical isomorphisms (given by the obvious analogue of Lemma~\ref{lem:convolution-Ew}).

The definition of these images occupies the rest of this subsection. Then in~\S\ref{ss:verification-Diag-parity} we prove that these morphisms satisfy the required relations.

\subsubsection{Polynomials}

As noted in~\eqref{eqn:equiv-cohom-pt},
for $m \in \Z_{\geq 0}$ the Borel isomorphism gives
us a canonical identification
\[
\Hom_{\Db_{\BKM}(\Flag, \Z')}(\cE_{\Z'}(\varnothing), \cE_{\Z'}(\varnothing)[2m]) \cong \mathsf{H}_{\BKM}^{2m}(\pt; \Z') = \mathrm{Sym}^m_{\Z'} \bigl( \Z' \otimes_\Z \Lambda \bigr).
\]
We send the morphism $B_{\varnothing} \to B_{\varnothing} \langle 2m \rangle$ given a region labelled by $f \in \mathrm{Sym}^m_{\Z'}\bigl( \Z' \otimes_\Z \Lambda \bigr)$ to the corresponding map
$\cE_{\Z'}(\varnothing) \to \cE_{\Z'}(\varnothing)[2m]$ under this identification:
\[
\Delta_{\mathrm{BS}} \left ( 
  \begin{array}{c}
    \begin{tikzpicture}[thick,scale=0.07,baseline]
      \node at (0,0) {$f$};
    \end{tikzpicture}
  \end{array}  \right )
:= \cE_{\Z'}(\varnothing) \xrightarrow{f} \cE_{\Z'}(\varnothing)[2m].
\]

\subsubsection{Dot morphisms}

Recall the inclusion $i_{1} \colon \Flag_{\leq 1} = \BKM/\BKM \into \GKM/\BKM$. 
We define the image of the upper dot morphism attached to a simple reflection $s \in S$ to be the
adjunction morphism:
\[
\Delta_{\mathrm{BS}} \left ( 
  \begin{array}{c}
    \begin{tikzpicture}[thick,scale=0.07,baseline]
      \draw (0,-5) to (0,0);
      \node at (0,0) {$\bullet$};
      \node at (0,-6.7) {\tiny $s$};
    \end{tikzpicture}
  \end{array}  \right )
:= a_* \colon  \cE_{\Z'}(s) \to (i_{1})_*(i_1)^*\cE_{\Z'}(s) = \cE_{\Z'}(\varnothing)[1].
\]
(Because $(f \circ g)^* \cong g^* f^*$ for two maps $f$ and $g$, we have canonically $(i_{1})^* \cE_{\Z'}(s) = (i_{1})^*(p_s)^*\underline{\Z'}_{\pt}[1] =
\underline{\Z'}_{\BKM/\BKM}[1]$ where $p_s \colon \PKM_s/\BKM \to \pt$ denotes the projection.)

We define the image of the lower dot morphism attached to $s \in S$ to be the adjunction morphism:
\[
\Delta_{\mathrm{BS}} \left ( 
  \begin{array}{c}
    \begin{tikzpicture}[thick,baseline,xscale=0.07,yscale=-0.07]
      \draw (0,-5) to (0,0);
      \node at (0,0) {$\bullet$};
      \node at (0,-6.7) {\tiny $s$};
    \end{tikzpicture}
  \end{array}  \right )
:= a_! \colon \cE_{\Z'}(\varnothing) = (i_{1})_!(i_{1})^! \cE_{\Z'}(s)[1] \to \cE_{\Z'}(s)[1].
\]
(Because $(f \circ g)^! \cong g^! f^!$ for two maps $f$ and $g$, we have canonically $(i_1)^! \cE_{\Z'}(s)[1] =
(i_{1})^!\underline{\mathbb{D}}^{\Z'}_{\PKM_s/\BKM} = (i_{1})^!(p_s)^!\underline{\Z'}_{\pt} =
\underline{\Z'}_{\BKM/\BKM}$, where as above $p_s \colon \PKM_s/\BKM \to \pt$ denotes the projection. Here
the identification $\cE_{\Z'}(s) = \underline{\Z'}_{\PKM_s/\BKM} [1]= \underline{\mathbb{D}}^{\Z'}_{\PKM_s/\BKM}[-1]$ is canonical because
$\PKM_s/\BKM$ is smooth of complex dimension $1$ and we
have chosen once and for all $\sqrt{-1} \in \C$.)

\subsubsection{Trivalent vectices}

We fix a simple reflection $s \in S$.

\begin{lem}
\label{lem:Ess}
There exists an isomorphism $\cE_{\Z'}(ss) \cong \cE_{\Z'}(s)[1] \oplus \cE_{\Z'}(s)[-1]$.
\end{lem}

\begin{proof}
There exist isomorphisms
\[
\BSvar(ss) \cong \PKM_s/\BKM \times \PKM_s/\BKM \cong \mathbb{P}^1 \times \mathbb{P}^1
\]
such that the morphism $\nu_{ss}$ identifies with the second projection $\mathbb{P}^1 \times \mathbb{P}^1 \to \mathbb{P}^1$. Then the decomposition follows e.g.~from the projection formula together with the known description of the cohomology of $\mathbb{P}^1$.
\end{proof}

In particular, from Lemma~\ref{lem:Ess} we deduce the following isomorphisms:
\begin{gather*}
  \Hom(\cE_{\Z'}(s), \cE_{\Z'}(ss)[-1]) = \Hom(\cE_{\Z'}(s), \cE_{\Z'}(s)) =
  \mathsf{H}^0_{\BKM}(\PKM_s/\BKM; \Z')  = \Z', \\ 
  \Hom(\cE_{\Z'}(ss), \cE_{\Z'}(s)[-1]) = \Hom(\cE_{\Z'}(s)[-1], \cE_{\Z'}(s)[-1]) =
  \mathsf{H}^0_{\BKM}(\PKM_s/\BKM; \Z') = \Z'.
\end{gather*}

\begin{lem}
\label{lem:morphisms-ss-s}
\begin{enumerate}
\item
\label{it:morphisms-ss-s-1}
Composition with the morphism
\[
\cE_{\Z'}(ss)[-1] \simto
\cE_{\Z'}(s) \star^{\BKM} \cE_{\Z'}(s) [-1] \xrightarrow{\cE_{\Z'}(s) \star^\BKM a_* [-1]} \cE_{\Z'}(s) \star^\BKM \cE_{\Z'}(\varnothing) \simto \cE_{\Z'}(s)
\]
induces an isomorphism
\[
\Hom_{\Db_{\BKM}(\Flag, \Z')}(\cE_{\Z'}(s), \cE_{\Z'}(ss)[-1]) \simto \Hom_{\Db_{\BKM}(\Flag, \Z')}(\cE_{\Z'}(s), \cE_{\Z'}(s)).
\]
\item
\label{it:morphisms-ss-s-2}
Composition with the morphism
\[
\cE_{\Z'}(s)[-1] \simto \cE_{\Z'}(s) \star^\BKM \cE_{\Z'}(\varnothing)[-1] \xrightarrow{\cE_{\Z'}(s) \star^\BKM a_! [-1]} \cE_{\Z'}(s) \star^\BKM \cE_{\Z'}(s) \simto \cE_{\Z'}(ss)
\]
induces an isomorphism
\[
\Hom_{\Db_{\BKM}(\Flag, \Z')}(\cE_{\Z'}(ss), \cE_{\Z'}(s)[-1]) \simto \Hom_{\Db_{\BKM}(\Flag, \Z')}(\cE_{\Z'}(s)[-1], \cE_{\Z'}(s)[-1]).
\]
  \end{enumerate}
\end{lem}

\begin{proof}
Statement~\eqref{it:morphisms-ss-s-1} follows from the observation that one can choose the decomposition $\cE_{\Z'}(ss) \cong \cE_{\Z'}(s)[1] \oplus \cE_{\Z'}(s)[-1]$ of Lemma~\ref{lem:Ess} so that our morphism identifies with the shift by $[-1]$ of the projection on the first factor. The proof of~\eqref{it:morphisms-ss-s-2} is similar.
%
\end{proof}

We now define
\[
b_1 \in \Hom_{\Db_{\BKM}(\Flag, \Z')}(\cE_{\Z'}(s), \cE_{\Z'}(ss)[-1]) \ \text{and} \
b_2 \in \Hom_{\Db_{\BKM}(\Flag, \Z')}(\cE_{\Z'}(ss), \cE_{\Z'}(s)[-1]) 
\]
to be the unique elements which map to the 
identity in $\Hom_{\Db_{\BKM}(\Flag, \Z')}(\cE_{\Z'}(s), \cE_{\Z'}(s))$ and $\Hom_{\Db_{\BKM}(\Flag, \Z')}(\cE_{\Z'}(s)[-1], \cE_{\Z'}(s)[-1])$ respectively under the isomorphisms of
Lemma~\ref{lem:morphisms-ss-s}. We set:
\[
\Delta_{\mathrm{BS}} \left ( 
  \begin{array}{c}
    \begin{tikzpicture}[thick,baseline,scale=0.07]
      \draw (-4,5) to (0,0) to (4,5);
      \draw (0,-5) to (0,0);
      \node at (0,-6.7) {\tiny $s$};
      \node at (-4,6.4) {\tiny $s$};
            \node at (4,6.4) {\tiny $s$};
    \end{tikzpicture}
  \end{array} \right ) := b_1
\qquad \text{and} \quad
\Delta_{\mathrm{BS}} \left ( 
  \begin{array}{c}
    \begin{tikzpicture}[thick,baseline,scale=-0.07]
      \draw (-4,5) to (0,0) to (4,5);
      \draw (0,-5) to (0,0);
      \node at (0,-6.7) {\tiny $s$};
            \node at (-4,6.4) {\tiny $s$};
            \node at (4,6.4) {\tiny $s$};
    \end{tikzpicture}
  \end{array} \right ) := b_2.
\]

\subsubsection{$2m_{st}$-valent vectices}

Fix $s, t \in S$ and define
\[
m_{st} := \begin{cases} 2 & \text{if $\langle \alpha_{i_s}^\vee, \alpha_{i_t}
    \rangle = \langle \alpha_{i_t}^\vee, \alpha_{i_s} \rangle = 0$;} \\
3 & \text{if $\langle \alpha_{i_s}^\vee, \alpha_{i_t}
    \rangle \langle \alpha_{i_t}^\vee, \alpha_{i_s} \rangle = 1$;} \\
4 & \text{if $\langle \alpha_{i_s}^\vee, \alpha_{i_t}
    \rangle \langle \alpha_{i_t}^\vee, \alpha_{i_s} \rangle = 2$;} \\
6 & \text{if $\langle \alpha_{i_s}^\vee, \alpha_{i_t}
    \rangle \langle \alpha_{i_t}^\vee, \alpha_{i_s} \rangle = 3$;} \\
\infty & \text{if $\langle \alpha_{i_s}^\vee, \alpha_{i_t}
    \rangle \langle \alpha_{i_t}^\vee, \alpha_{i_s} \rangle > 3$.}
\end{cases}
\]
Then $m_{st}$ is the order of $st \in W$. 

From now on we fix a pair $s,t \in S$ with $m_{st} < \infty$ and
abbreviate $m := m_{st}$. Set
\[
w_I := sts \cdots \qquad \text{(with $m$ terms)}
\]
and let $W_I := \langle s, t \rangle$.
To simplify notation we also set
\begin{gather*}
 \cF_s :=  \cE_{\Z'}(st \cdots) \qquad \text{($m$ terms)},\\
 \cF_t :=  \cE_{\Z'}(ts \cdots) \qquad \text{($m$ terms)}.
\end{gather*}

\begin{lem} \label{lem:mZ}
The $\Z'$-module $\Hom_{\Db_{\BKM}(\Flag, \Z')}(\cF_s, \cF_t)$ is free of rank $1$.
\end{lem}

\begin{proof}
By Lemma~\ref{lem:morphisms-parityBS}, the $\Z'$-module under consideration is free, and we have
\[
\Q \otimes_{\Z'} \Hom_{\Db_{\BKM}(\Flag, \Z')}(\cF_s, \cF_t)\cong \Hom_{\Db_{\BKM}(\Flag, \Q)}(\Q(\cF_s) , \Q(\cF_t)).
\]
By Kazhdan--Lusztig theory and a calculation in the Hecke algebra we have isomorphisms:
\[
\Q(\cF_s) \cong \mathcal{IC}_{w_I} \oplus \bigoplus_{\substack{x \in W_I,
  \\ x \ne w_I,
  sx < x}} \mathcal{IC}_x^{\oplus m_x},
\qquad
\Q(\cF_t) \cong \mathcal{IC}_{w_I} \oplus \bigoplus_{\substack{x \in W_I,
  \\ x \ne w_I,
  tx < x}} \mathcal{IC}_x^{\oplus m_x}
\]
for some integers $m_x$, where $\mathcal{IC}_y$ is the intersection cohomology complex associated with the constant rank-$1$ $\Q$-local system on $\Flag_y$.
In particular, $\Q(\cF_s)$ and $\Q(\cF_t)$ are semi-simple perverse sheaves and only
have one simple factor in common. Hence $\Hom(\Q(\cF_s),
\Q(\cF_t)) = \Q$, and the claim follows.
\end{proof}

\begin{lem}
 \label{lem:split}
We have decompositions
\[
\cF_s = \underline{\Z'}_{\overline{\Flag_{w_I}}}[m] \oplus
  \cC_s \quad \text{and} \quad \cF_t = \underline{\Z'}_{\overline{\Flag_{w_I}}}[m] \oplus \cC_t
\]
where $\cC_s, \cC_t$ are complexes supported on $\overline{\Flag_{w_I}} \smallsetminus \Flag_{w_I}$.
\end{lem}

\begin{proof}
By symmetry it is enough to show one decomposition. Set $\uw:=st \cdots$ (with $m$ terms);
then by definition we have $\cF_s = \nu_{\uw*} \underline{\Z'}_{\BSvar(\uw)}[m]$. The morphism $\nu_{\uw}$ is birational and has connected fibers, see~\cite[Lemme~32 on p.~69]{mathieu0}. Hence if
$\tau_{\le -m}$ denotes the truncation for the standard t-structure,
we have a map
\[
\alpha \colon \underline{\Z'}_{\overline{\Flag_{w_I}}}[m] \cong \tau_{\le -m} \cF_s \to \cF_s
\]
which is an isomorphism over the Schubert cell $\Flag_{w_I} \subset
\Flag$. Dualizing $\a$ we get a map (by the self-duality of both
sheaves)
\[
\beta \colon \cF_s \to \underline{\Z'}_{\overline{\Flag_{w_I}}}[m]
\]
which is again an isomorphism over $\Flag_{w_I} \subset \Flag$. Hence the composition
\[
\underline{\Z'}_{\overline{\Flag_{w_I}}}[m] \to \cF_s \to \underline{\Z'}_{\overline{\Flag_{w_I}}}[m] 
\]
is an isomorphism, since the morphism
\[
\Hom(\underline{\Z'}_{\overline{\Flag_{w_I}}}[m],  \underline{\Z'}_{\overline{\Flag_{w_I}}}[m]) \to \Hom(\underline{\Z'}_{\Flag_{w_I}}[m], \underline{\Z'}_{\Flag_{w_I}}[m]) \cong \Z'
\]
induced by restriction
is an isomorphism.
Hence $\underline{\Z'}_{\overline{\Flag_{w_I}}}[m]$ is a summand of $\cF_s$ as claimed. Then we can write
\[
\cF_{s} = \underline{\Z'}_{\overline{\Flag_{w_I}}}[m] \oplus \cC_s,
\]
and the above considerations imply that $i_{w_I}^*\cC_s = 0$.
\end{proof}

Since both $\nu_{st \cdots}$ and $\nu_{ts \cdots}$ (with $m$ terms in both expressions once again) are isomorphisms over $\Flag_{w_I}$, we have canonical isomorphisms
\begin{equation} \label{eq:stalk}
\mathsf{H}^{-m}((\cF_s)_{w_I \BKM/\BKM}) = \mathsf{H}^0(\pt,\Z') = \mathsf{H}^{-m}((\cF_t)_{w_I\BKM/\BKM}).
\end{equation}
We let
\[
f_{s,t} \colon \cF_s \to \cF_t
\]
be the unique morphism which restricts to the identity on the stalk at $w_I \BKM/\BKM$ under the
identification~\eqref{eq:stalk}. (The uniqueness of
such a map follows from Lemma  \ref{lem:mZ}, and the existence
follows from Lemma \ref{lem:split}.)
Then we define the image of the $2m_{st}$-valent vertex to be $f_{s,t}$:
\[
\Delta_{\mathrm{BS}} \left (
\begin{tikzpicture}[yscale=0.5,xscale=0.3,baseline,thick]
\draw (-2.5,-1) to (0,0) to (-1.5,1);
\draw (-0.5,-1) to (0,0) to (0.5,1);
\draw (1.5,-1) to (0,0) to (2.5,1);
\draw[red] (-1.5,-1) to (0,0) to (-2.5,1);
\draw[red] (0.5,-1) to (0,0) to (-0.5,1);
\draw[red] (2.5,-1) to (0,0) to (1.5,1);
\node at (-2.5,-1.3) {\tiny $s$};
\node at (-1.5,1.3) {\tiny $s$};
\node at (-0.5,-1.3) {\tiny $\cdots$};
\node at (-1.5,-1.3) {\tiny $t$};
\node at (-2.5,1.3) {\tiny $t$};
\node at (-0.5,1.3) {\tiny $\cdots$};
    \end{tikzpicture}
    \right ) := 
f_{s,t}.
\]

\subsection{Verification of the relations}
\label{ss:verification-Diag-parity}

Now we prove that the morphisms defined in~\S\ref{ss:construction-Diag-parity} satisfy the relations of the definition of~$\DiagBS^{\Z'}(\GKM)$. In fact, using Lemma~\ref{lem:morphisms-parityBS} once again, we see that it is enough to check that the similar morphisms for coefficients $\Q$ satisfy the desired relations. For simplicity, we will use the same notation as in~\S\ref{ss:construction-Diag-parity}, but now our ring of coefficients is $\Q$ instead of $\Z'$. Also, we set $R:=\cO(\fh_\GKM^\Q) \cong \mathsf{H}^\bullet_{\BKM}(\pt; \Q)$ (see~\eqref{eqn:equiv-cohom-pt}), we denote by $R-\mathsf{Bim}^\Z$ the category of $\Z$-graded $R$-bimodules, and we denote by $\langle 1 \rangle$ the functor of shift of the grading, normalized\footnote{Note that this normalization is opposite to the choice in~\cite{mr} or~\cite{ar}.} by $(M\langle 1 \rangle)_n=M_{n+1}$.

The total cohomology functor $\mathsf{H}^\bullet_{\BKM}(\Flag, -)$ induces a functor
\begin{equation}
\label{eqn:functor-H}
\Par^{\mathrm{BS}}_{\BKM}(\Flag, \Q) \to \mathsf{H}^\bullet_{\BKM}(\Flag; \Q)-\mathsf{Mod}^\Z,
\end{equation}
where the right-hand side denotes the category of graded $\mathsf{H}^\bullet_{\BKM}(\Flag; \Q)$-modules. Composition with the ``restriction of scalars'' functor associated with the natural morphism
\[
R \otimes_{\Q} R \to \mathsf{H}^\bullet_{\BKM}(\Flag; \Q)
\]
induced by~\eqref{eqn:equiv-cohom-pt} (for the two actions of $\BKM$ on $\GKM$), we obtain a functor
\[
\bbH \colon \Par^{\mathrm{BS}}_{\BKM}(\Flag, \Q) \to R-\mathsf{Bim}^\Z.
\]

The following facts are known about $\bbH$:
\begin{itemize}
\item[(H1)]
\label{it:H1}
$\bbH$ is monoidal, i.e.~there exists canonical isomorphisms
\[
\bbH(\cF \star \cG) \cong \bbH(\cF) \otimes_R \bbH(\cG)
\]
for $\cF, \cG$ in $\Par^{\mathrm{BS}}_{\BKM}(\Flag, \Q)$;
\item[(H2)]
\label{it:H2}
we have $\bbH \bigl( (s,0) \bigr) = R \otimes_{R^s} R \langle 1 \rangle$ and $\bbH \bigl( (\varnothing,n) ) = R \langle n\rangle$ for any $s \in S$ and $n \in \Z$.
\item[(H3)]
\label{it:H3}
$\bbH$ is faithful.
\end{itemize}
(To prove (H3), we remark that the functor~\eqref{eqn:functor-H} is fully-faithful by standard arguments, see~\cite[\S 3.3]{by} or~\cite[\S 3.9]{mr}, and that a ``restriction of scalars'' functor is always faithful.)

We now calculate the image of our morphisms under $\bbH$. For simplicity we abbreviate $B_s^B := \bbH \bigl( (s,0) ) = R \otimes_{R^s} R \langle 1 \rangle$.
\begin{enumerate}
\item
\emph{Polynomials:}
$\bbH(\Delta_{\mathrm{BS}}(f))$ for $f \in \mathrm{Sym}^m(\Q \otimes_\Z \Lambda)$ is given by
  multiplication by $f$ on $R$ (as is immediate from
  the definitions).
\item \emph{The upper dot:} Because $B_s^B$ and $R$ are both cyclic as $R$-bimodules
  any morphism
\[
B_s^B \to R \langle 1 \rangle
\]
is a scalar multiple of the morphism $m \colon B_s^B \to R \langle 1 \rangle \colon f \otimes g \mapsto fg$.
From the definitions $\bbH(a_*)$ must induce the identity in degree $-1$. Hence
$\bbH(a_*) = m$.
\item
\emph{The lower dot:}
It is easy to see that the space of graded $R$-bimodule homomorphisms
\[
R \to B_s^B \langle 1 \rangle
\]
is of dimension 1, with generator $\delta$, where $\delta$ is given by
$\delta(1) = \frac{1}{2}(\alpha_s \otimes 1 + 1 \otimes \alpha_s)$. In
particular, we have $\bbH(a_!) = x \cdot \delta$ for some $x \in \Q$.

However, the composition
\begin{multline*}
(i_1)_!(i_{1})^! \cE_{\Q}(s) \to \cE_{\Q}(s) \to (i_{1})_*(i_1)^*\cE_{\Q}(s) \in \Hom(\cE_{\Q}(\varnothing)[-1],
\cE_{\Q}(\varnothing)[1] ) \\
= \mathsf{H}_{\BKM}^2(\pt,\Q)
\end{multline*}
is given by the $\BKM$-weight on the tangent space of $\PKM_s/\BKM$ at $\BKM/\BKM$,
which is $\alpha_s$. We conclude that $x = 1$, or in other words that
$\bbH(a_!) = \delta$.
\item \emph{The trivalent vertices.}
Any choice of isomorphism $R = R^s \oplus R^s \langle -2 \rangle$ of $R^s$-bimodules gives a decomposition
\[
B_s^B \otimes_R B_s^B = B_s^B \langle 1 \rangle \oplus B_s^B \langle -1 \rangle.
\]
In particular, we have
\begin{align*}
\Hom(B_s^B, B_s^B \otimes_R B_s^B \langle -1 \rangle) = \Q, \\
\Hom(B_s^B \otimes_R B_s^B, B_s^B \langle -1 \rangle) = \Q.
\end{align*}
Moreover, these spaces are generated by the maps
\begin{gather*}
t_1 \colon f \otimes g \mapsto f \otimes 1 \otimes g \quad \in \Hom(B_s^B, B_s^B \otimes_R B_s^B \langle -1 \rangle)\\
t_2 \colon f \otimes g \otimes h \mapsto f (\partial_s g) \otimes h \quad \in \Hom(B_s^B \otimes_R B_s^B, B_s^B \langle -1 \rangle)
\end{gather*}
where in both cases we identify $B_s^B \otimes_R B_s^B = R
\otimes_{R^s} R \otimes_{R^s} R \langle 2 \rangle$. (In the formula for $t_2$, $\partial_s$ is the Demazure operator associated with $s$; see~\cite[\S 3.3]{ew}.) In particular, $\bbH(b_i)$ is a
scalar multiple of $t_i$ for $i \in \{1, 2\}$.

However one checks easily that one has
\begin{align*}
(t_2 \langle 1 \rangle) \circ ( \delta \otimes_R \id_{B_s^B})
= \id_{B_s^B}, \\
(m \langle -1 \rangle \otimes_R \id_{B_s^B}) \circ t_1
= \id_{B_s^B}. 
\end{align*}
Hence
\[
\bbH(b_1) = t_1 \quad \text{and} \quad \bbH(b_2) = t_2
\]
as follows from applying $\bbH$
to the defining properties of $b_1$, $b_2$.
\item
\emph{$2m_{st}$-valent vertices.}
Let $\ux = (s,t, \cdots)$ and $\uy = (t, s, \cdots)$, where both sequences have $m_{st}$ elements.
Then by (H1) and (H2) we have
\begin{align*}
\bbH(\ux) = B_s^B \otimes_R B_t^B \otimes_R \dots \quad
\text{($m_{st}$ terms),} \\
\bbH(\uy) = B_t^B \otimes_R B_s^B \otimes_R \dots \quad
\text{($m_{st}$ terms).}
\end{align*}
Analogously to the definition of the image of the $2m_{st}$-valent
vertex one can use the theory of Soergel bimodules to see that the
space of graded $R$-bimodule homomorphisms
$\bbH(\ux)  \to \bbH(\uy)$
is of dimension 1 over $\Q$, see~\cite[Proposition 4.3]{libedinsky}. It follows from the definition of $f_{s,t}$ that
$\bbH(f_{s,t})$ is the unique morphism
\[
B_s^B \otimes_R B_t^B \otimes_R \dots \to B_t^B \otimes_R B_s^B \otimes_R \dots 
\]
which induces the identity in degree $-m_{st}$.
\end{enumerate}

To conclude the proof, by faithfulness of $\bbH$ we only have to check that the morphisms of $R$-bimodules considered above satisfy the relations of~\cite{ew}. Each of these relations involves a subset $S'$ of $S$ (of cardinality at most $3$) which generates a finite subgroup $W'$ of $W$. Let us fix some relation, and the corresponding subset $S'$. Let $\mathfrak{k}_{S'} \subset \fh^\Q_{\GKM}$ be the intersection of the kernels of the images of the roots $\alpha_s$ for $s \in S'$, and let $\mathfrak{l}_{S'} \subset \fh^\Q_{\GKM}$ be the subspace generated by the images of the coroots $\alpha_s^\vee$ for $s \in S'$. Then since $W'$ is finite the matrix $(\langle \alpha_s^\vee, \alpha_t \rangle)_{s,t \in S'}$ is invertible, which implies that $\dim(\mathfrak{l}_{S'}) = \# S'$ and that $\fh^\Q_{\GKM} = \mathfrak{k}_{S'} \oplus \mathfrak{l}_{S'}$. This decomposition reduces the verification of the relation to the similar relation in the case of the realization of $(W',S')$ given by $\mathfrak{l}_{S'}$, i.e.~by the standard Cartan realization. In this case the relations can be (and have been) checked by computer, using the localization method described in~\cite[\S 5.5]{ew}.


\begin{rmk}
\label{rmk:Soergel-H}
As explained above,
the relations in the definition of $\DiagBS^\K(\GKM)$ only involve subsets of $S$ which generate a \emph{finite} parabolic subgroup of $W$. Hence in the proof above one only needs the fully-faithfulness of~\eqref{eqn:functor-H} in the case when $\GKM$ is an ordinary complex connected reductive group. In this generality this result was first proved by Soergel, see~\cite[Proposition~2]{soergel-Langlands}. 
\end{rmk}

\subsection{Fully-faithfulness of $\Delta_{\mathrm{BS}}$}
\label{ss:Diag-parity-ff}

To conclude the proof of Theorem~\ref{thm:main-parity-BS}, it remains to prove that our functor $\Delta_{\mathrm{BS}}$ is fully-faithful. Using~\eqref{eqn:Diag-K-K'} and Lemma~\ref{lem:morphisms-parityBS},
one sees that it is enough to prove fully-faithfulness over $\Z'$.

\begin{lem}
\label{lem:graded-ranks}
Let $\uw$ and $\uv$ be expressions. The graded $\cO(\fh_\GKM^{\Z'})$-modules
\[
\Hom_{\DiagBS^{\Z'}(\GKM)}^\bullet(B_{\uw}, B_{\uv}) \quad \text{and} \quad \bigoplus_{n \in \Z}Ê\Hom_{\Par^{\mathrm{BS}}_{\BKM}(\Flag,\Z')}((\uw,0), (\uv,n))
\]
are free of finite rank, and they have the same graded rank.
\end{lem}

\begin{proof}[Sketch of proof]
For $\DiagBS^{\Z'}(\GKM)$, the freeness follows from~\cite[Theorem~6.11]{ew}. For the category $\Par^{\mathrm{BS}}_{\BKM}(\Flag,\Z')$ this follows from~\cite[Lemma~2.2(2)]{mr}, see~\S\ref{ss:parity-flag}.

To prove that the graded ranks coincide, using adjunction it suffices to consider the case $\uv=\varnothing$ (see e.g.~the arguments in~\S\ref{ss:proof-thm}). In this case, using~\cite[Lemma~2.10 \& Proposition~6.12]{ew}, if $\uw=s_1 \cdots s_r$ one sees that the graded rank of the left-hand side is the coefficient of $H_{\id}$ in the product of the Kazhdan--Lusztig elements $\uH_{s_i}$ in the Hecke algebra of $(W,S)$. On the other hand, the left-hand side is the dual of
\[
\mathsf{H}^\bullet_{\BKM} \bigl( (i_{1})^! \cE_{\Z'}(\uw) \bigr).
\]
The graded rank of the latter module can also be expressed in terms of the Hecke algebra of $(W,S)$ using the methods of ~\cite[\S 3.10]{mr} (which go back at least to~\cite{springer}).
The formula for the graded rank of $\bigoplus_{n \in \Z}Ê\Hom_{\Par^{\mathrm{BS}}_{\BKM}(\Flag,\Z')}((\uw,0), (\uv,n))$ obtained in this way is the same as for $\DiagBS^{\Z'}(\GKM)$, and the lemma follows.
\end{proof}


It is easy to see that a graded morphism $\phi$ between two graded free $\cO(\fh_\GKM^{\Z'})$-modules 
is an isomorphism iff $\F \otimes_{\Z'} \phi$ is an isomorphism for any field $\F$ such that there exists a ring morphism $\Z' \to \F$.
This remark, together with~\eqref{eqn:Diag-K-K'}, Lemma~\ref{lem:morphisms-parityBS} and Lemma~\ref{lem:graded-ranks}, reduce the proof of fully-faithfulness to the case of coefficients in a field $\F$ (which admits a ring morphism $\Z' \to \F$). Finally, it is easy to see that a morphism $\phi \colon M \to N$ between two graded free $\cO(\fh_\GKM^{\F})$-modules with the same graded rank is an isomorphism iff the composition
\[
M \xrightarrow{\phi} N \to \F \otimes_{\cO(\fh_\GKM^{\F})} N
\]
is surjective (where here $\F$ is considered as the trivial $\cO(\fh_\GKM^{\F})$-module). Hence, using~\eqref{eqn:morphisms-parity-For}, we have finally reduced the proof of the fact that $\Delta_{\mathrm{BS}}$ is fully-faithful to proving the following claim.

\begin{prop}
\label{prop:Delta-surjective}
For any field $\F$ such that there exists a morphism $\Z' \to \F$, and any expressions $\uw$ and $\uv$,
the morphism
\[
\gamma_{\uw,\uv} \colon \Hom^\bullet_{\DiagBS^{\F}(\GKM)}(B_{\uw}, B_{\uv}) \to \Hom^\bullet_{\Db_{(\BKM)}(\Flag, \F)}(\cE_{\F}(\uw), \cE_{\F}(\uv))
\]
induced by the composition of $\Delta_{\mathrm{BS}}$ with the forgetful functor is surjective.
\end{prop}

As in the proof of Theorem~\ref{thm:main} (see~\S\ref{ss:proof-thm}), using adjunction one can reduce the proof of Proposition~\ref{prop:Delta-surjective} to the case when $\uw=\varnothing$. This special case follows from the following more general claim.

\begin{prop}
\label{prop:surjectivity-parity}
For $\F$ as in Proposition~{\rm \ref{prop:Delta-surjective}}, any reduced expression $\uw$ for an element $w \in W$ and any expression $\uv$, the composition
\begin{multline*}
\delta_{\uw,\uv} \colon \Hom^\bullet_{\DiagBS^{\F}(\GKM)}(B_{\uw}, B_{\uv}) \xrightarrow{\gamma_{\uw,\uv}}  \Hom^\bullet_{\Db_{(\BKM)}(\Flag, \F)}(\cE_{\F}(\uw), \cE_{\F}(\uv)) \\
\to \Hom^\bullet_{\Db_{(\BKM)}(\Flag_{\geq w}, \F)}(\cE_{\F}(\uw), \cE_{\F}(\uv))
\end{multline*}
(where the second arrow is induced by restriction to $\Flag_{\geq w}$) is surjective.
\end{prop}

Before proving Proposition~\ref{prop:surjectivity-parity} in full generality, we begin with some special cases.

\begin{lem}
\label{lem:surjectivity-parity}
Let $\uy$ be a reduced expression for an element $y \in W$, and let $s \in S$ be a simple reflection such that $ys>y$ in the Bruhat order. Then the morphisms $\delta_{\uy s, \uy s}$, $\delta_{\uy s, \uy s s}$, $\delta_{\uy,\uy s}$ and $\delta_{\uy,\uy s s}$ are surjective.
\end{lem}

\begin{proof}
Using Lemma~\ref{lem:BsBs} (but in $\DiagBS^\F(\GKM)$ now), as in the proof of Lemma~\ref{lem:morphisms-Phit} one reduces the case of $\delta_{\uy s, \uy s s}$ to that of $\delta_{\uy s, \uy s}$, and the case of $\delta_{\uy,\uy s s}$ to that of $\delta_{\uy,\uy s}$. The case of $\delta_{\uy s, \uy s}$ is obvious since
\[
\bigoplus_{n \in \Z} \Hom_{\Db_{(\BKM)}(\Flag_{\geq ys}, \F)}(\cE_{\F}(\uy s), \cE_{\F}(\uy s)[n]) = \F \cdot \id,
\]
hence only the case of $\delta_{\uy,\uy s}$ remains. In this case, recall the inclusion $i_y' \colon \Flag_y \sqcup \Flag_{ys} \hookrightarrow \Flag$. Since the Bott--Samelson resolution $\nu_{\uy s}$ is an isomorphism over the open subset $\Flag_y \sqcup \Flag_{ys} \subset \overline{\Flag_{ys}}$, we have
\[
(i_y')^* \cE_{\F}(\uy s) \cong \underline{\F}_{\Flag_y \sqcup \Flag_{ys}} [\ell(y)+1].
\]
Similarly we have
\[
(i_y')^* \cE_{\F}(\uy) \cong \underline{\F}_{\Flag_y} [\ell(y)]
\]
(where we omit the direct image functor associated with the closed embedding $\Flag_y \hookrightarrow \Flag_y \sqcup \Flag_{ys}$),
and the functor $(i'_y)^*$ induces an isomorphism
\begin{multline*}
\Hom^\bullet_{\Db_{(\BKM)}(\Flag_{\geq y}, \F)}(\cE_{\F}(\uy), \cE_{\F}(\uy s)) \simto \\
\Hom^\bullet_{\Db_{(\BKM)}(\Flag_{y} \sqcup \Flag_{ys}, \F)}(\underline{\F}_{\Flag_y} [\ell(y)], \underline{\F}_{\Flag_y \sqcup \Flag_{ys}} [\ell(y)+1]).
\end{multline*}
It is easily checked that the right-hand side is concentrated in degree $1$, and $1$-dimensional. It is also easy to see that the image of the morphism $\cE_{\F}(\uy) \to \cE_{\F}(\uy s)[1]$ induced by the lower dot morphism is non-zero, and the claim follows.
\end{proof}

\begin{proof}[Proof of Proposition~{\rm \ref{prop:surjectivity-parity}}]
We prove by induction on $\ell(\uv)$ that the statement holds for all reduced expressions $\uw$. If $\ell(\uv)=0$, i.e.~$\uv=\varnothing$, then the codomain of $\delta_{\uw,\uv}$ vanishes unless $\uw=\varnothing$, in which case it is spanned by the identity morphism; hence there is nothing to prove in this case.

Now, let $\uv$ be an expression such that $\ell(\uv)>0$. Write $\uv=\uu s$ where $s \in S$, 
and assume the result is known for $\uu$. We distinguish two cases.

\emph{Case~1:} $ws<w$. In this case, $w$ has a reduced expression $\uw'$ ending with $s$. It is clear from definitions that the restriction to $\Flag_w$ of the image under $\Delta_{\mathrm{BS}}$ of the morphism $\cE_\F(\uw) \to \cE_{\F}(\uw')$ induced by a rex move $\uw \leadsto \uw'$ (see~\S\ref{ss:rex-moves-DBS}) is invertible. Hence it is enough to prove the surjectivity of $\delta_{\uw',\uv}$; in other words we can assume (replacing $\uw$ by $\uw'$ if necessary) that $\uw=\ux s$ for some reduced expression $\ux$ (expressing $ws$). By induction, there exists a family $(f_i)_{i \in \mathfrak{I}}$ of homogeneous elements in the image of $\gamma_{\uw,\uu}$ whose image spans the vector space
\[
\Hom^\bullet_{\Db_{(\BKM)}(\Flag_{\geq w}, \F)}(\cE_{\F}(\uw), \cE_{\F}(\uu)),
\]
and a family $(g_j)_{j \in \mathfrak{J}}$ of homogeneous elements in the image of $\gamma_{\ux, \uu}$ whose image spans the vector space
\[
\Hom^\bullet_{\Db_{(\BKM)}(\Flag_{\geq ws}, \F)}(\cE_{\F}(\ux), \cE_{\F}(\uu)).
\]
By Proposition~\ref{prop:morphisms-BS-parity}\eqref{it:morphisms-BS-parity-2}, there exist morphisms
\[
f'_i \colon \cE_\F(\uw) \to \cE_\F(\uw s) [n_i], \qquad g'_j \colon \cE_{\F}(\uw) \to \cE_{\F}(\uw)[m_j]
\]
such that the images of the compositions
\[
(f_i \star^\BKM \cE_\F(s) [n_i]) \circ f'_i \quad \text{and} \quad (g_j \star^\BKM \cE_\F(s) [m_j]) \circ g'_j
\]
span the vector space
\begin{equation}
\label{eqn:hom-space-delta}
\Hom^\bullet_{\Db_{(\BKM)}(\Flag_{\geq w}, \F)}(\cE_{\F}(\uw), \cE_{\F}(\uv)).
\end{equation}
By Lemma~\ref{lem:surjectivity-parity} we can assume that the morphisms $f'_i$ are in the image of $\gamma_{\uw,\uw s}$, and that the morphisms $g'_j$ are in the image of $\gamma_{\uw, \uw}$; then we obtain that~\eqref{eqn:hom-space-delta} is spanned by images of vectors in the image of $\gamma_{\uw,\uv}$, and the claim follows.

\emph{Case~2:} $ws>w$. By induction, there exists a family $(f_i)_{i \in \mathfrak{I}}$ of homogeneous elements in the image of $\gamma_{\uw,\uu}$ whose image spans the vector space
\[
\Hom^\bullet_{\Db_{(\BKM)}(\Flag_{\geq w}, \F)}(\cE_{\F}(\uw), \cE_{\F}(\uu)),
\]
and a family $(g_j)_{j \in \mathfrak{J}}$ of homogeneous elements in the image of $\gamma_{\uw s, \uu}$ whose image spans the vector space
\[
\Hom^\bullet_{\Db_{(\BKM)}(\Flag_{\geq ws}, \F)}(\cE_{\F}(\uw s), \cE_{\F}(\uu)[n]).
\]
By Proposition~\ref{prop:morphisms-BS-parity}\eqref{it:morphisms-BS-parity-1}, there exist morphisms
\[
f'_i \colon \cE_\F(\uw) \to \cE_\F(\uw s) [n_i], \qquad g'_j \colon \cE_{\F}(\uw) \to \cE_{\F}(\uw s s)[m_j]
\]
such that the compositions
\[
(f_i \star^\BKM \cE_\F(s) [n_i]) \circ f'_i \quad \text{and} \quad (g_j \star^\BKM \cE_\F(s) [m_j]) \circ g'_j
\]
span the vector space
\begin{equation}
\label{eqn:hom-space-delta-2}
\Hom^\bullet_{\Db_{(\BKM)}(\Flag_{\geq w}, \F)}(\cE_{\F}(\uw), \cE_{\F}(\uv)).
\end{equation}
By Lemma~\ref{lem:surjectivity-parity} we can assume that the morphisms $f'_i$ are in the image of $\gamma_{\uw,\uw s}$, and that the morphisms $g'_j$ are in the image of $\gamma_{\uw, \uw ss}$; then we obtain that~\eqref{eqn:hom-space-delta-2} is spanned by images of vectors in the image of $\gamma_{\uw,\uv}$, and the claim follows.
\end{proof}



\subsection{The case of the affine flag variety}
\label{ss:affine-flag-variety}

We now explain the relation between the results stated in~\S\ref{ss:statement-Diag-parity} and the setting considered in
Sections~\ref{sec:blocks}--\ref{sec:Diag}. We consider a connected reductive algebraic group $G$ over an algebraically closed field $\F$ of characteristic $p$, and use the same notation as in~\S\ref{ss:definitions-G} (but do not assume that $p \geq h$).
We denote by $\Gv$ the 
connected reductive group over $\bbL$ which is
Langlands dual to $G$. That is to say, $\Gv$ has a fixed maximal torus
$\Tv \subset \Gv$ endowed with a fixed isomorphism $X^*(\Tv) \cong
X_*(T)$, such that the root datum of $(\Gv,\Tv)$ identifies with the
dual of the root datum of $(G,T)$. We also denote by $\Bv \subset
\Gv$ the Borel subgroup whose roots are the negative coroots of
$G$. Next, we denote by $\Gw$ the simply connected cover of 
the derived subgroup of 
$\Gv$. In particular, we have a natural group
morphism $\Gw \to \Gv$, and we denote by $\Tw$, resp.~$\Bw$, the
inverse image of $\Tv$, resp.~$\Bv$, under this morphism.

Let $\mathscr{K}:=\bbL( \hspace{-1.5pt} (z) \hspace{-1.5pt} )$ and $\mathscr{O}:=\bbL[ \hspace{-1.5pt} [z] \hspace{-1.5pt} ]$. Then we can consider the ind-group scheme $\GwK$ over $\bbL$ and its group subscheme $\GwO$. We define the Iwahori subgroup $\Iw \subset \GwO$ as the inverse image of $\Bw$ under the evaluation morphism $\GwO \to \Gw$ (at $z=0$). We can similarly define $\GvK$, $\GvO$ and $\Iv$, and consider the affine flag varieties
\[
\Fl^\vee := \GvK/\Iv, \qquad \Fl^\wedge:=\GwK/\Iw
\]
and their natural ind-variety structure.
The morphism $\Gw \to \Gv$ induces a closed embedding
\[
\Fl^\wedge \to \Fl^\vee
\]
which identifies $\Fl^\wedge$ with the connected component of the base point $\Iv/\Iv$ in $\Fl^\vee$. For a detailed account of these constructions, see e.g.~\cite{goertz}.

Both in the classical or \'etale setting of~\S\ref{ss:DbX}, we can consider the equivariant derived category $\Db_{\Iw}(\Fl^\wedge, \K)$.
It is well known that we have a ``Bruhat decomposition''
\[
\Fl^\wedge = \bigsqcup_{w \in \Waff} \Fl^\wedge_w
\]
where $\Fl^\wedge_w$ is the $\Iw$-orbit of the point of $\Fl^\wedge$ associated naturally with $w$, and that moreover $\Fl^\wedge_w$ is isomorphic to an affine space of dimension $\ell(w)$; see~\cite[Theorem~2.18]{goertz}. 
We also have ``Bott--Samelson varieties'' $\BSvar(\uw)$, hence 
we can consider the corresponding objects $\cE_{\K}(\uw)$ in 
$\Db_{\Iw}(\Fl^\wedge, \K)$, and the category $\Par^{\mathrm{BS}}_{\Iw}(\Fl^\wedge, \K)$ of $\Iw$-equivariant Bott--Samelson parity complexes on $\Fl^\wedge$. This category has a natural convolution product $\star^{\Iw}$, which makes it a monoidal category. Finally for any $s \in S$ we have a corresponding partial affine flag variety $\Fl^{\wedge,s}$ and a morphism $\Fl^\wedge \to \Fl^{\wedge, s}$; see e.g.~\cite{pappas-rapoport}.

In case $\K$ is complete local, one can also consider the monoidal category of all $\Iw$-equivariant parity complexes on $\Fl^\wedge$, denoted $\Par_{\Iw}(\Fl^\wedge, \K)$, which identifies with the Karoubi envelope of the additive hull of $\Par^{\mathrm{BS}}_{\Iw}(\Fl^\wedge, \K)$.

The proof of the following theorem, which provides a geometric description of the category $\Diag$ used in Sections~\ref{sec:Diag}--\ref{sec:main-conj}, is identical to the proof of Theorem~\ref{thm:main-parity} and Theorem~\ref{thm:main-parity-BS}. Here we set $\Z'=\Z$ if the morphisms $\alpha \colon \Z\Phi^\vee \to \Z$ and $\alpha^\vee \colon \Z\Phi \to \Z$ are surjective for any simple root $\alpha$, and $\Z'=\Z[\frac{1}{2}]$ otherwise. (In this statement, since we only consider the categories $\Diag$ and $\DiagBS$, we do not need to assume that~\eqref{eqn:assumption-pairing} is satisfied.)

\begin{thm}
\label{thm:Diag-parity-affine}
Assume that there exists a ring morphism $\Z' \to \K$. Then
there exists a canonical equivalence of monoidal categories
\[
\DiagBS \simto \Par_{\Iw}^{\mathrm{BS}}(\Fl^\wedge, \K).
\]
If $\K$ is furthermore a complete local ring, this equivalence
induces an equivalence of additive monoidal categories
\[
\Diag \simto \Par_{\Iw}(\Fl^\wedge, \K).
\]
\end{thm}

\begin{rmk}
\label{rmk:p-can-basis-realization}
\begin{enumerate}
\item
\label{it:KM-affine}
Assume for simplicity that $G$ is quasi-simple.
Let $\Sigma^\vee$ be the set of simple coroots of $G$, which constitute a basis of the root system of $(\Gw, \Tw)$. To avoid confusion, for $\alpha \in \Phi^\vee$ we will denote by $\widetilde{\alpha} \in \Phi$ the corresponding coroot of $\Gw$.
Let $\theta$ be the highest root of $\Gw$, and let $A$ be the generalized Cartan matrix with rows and columns parametrized by $\Sigma^\vee \cup \{0\}$ and coefficients
\[
a_{\alpha,\beta} = \begin{cases}
\langle \beta, \widetilde{\alpha} \rangle & \text{if $\alpha, \beta \in \Sigma^\vee$;} \\
2 & \text{if $\alpha=\beta=0$;} \\
- \langle \beta, \widetilde{\theta} \rangle & \text{if $\beta \in \Sigma^\vee$ and $\alpha=0$;} \\
- \langle \theta, \widetilde{\alpha} \rangle & \text{if $\alpha \in \Sigma^\vee$ and $\beta = 0$.}
\end{cases}
\]
There exists a natural Kac--Moody root datum for $A$ with underlying $\Z$-module $\Hom_\Z(\Z\Phi,\Z)$, and with the simple roots and coroots defined in a way similar to the construction in~\S\ref{ss:diag-SB}. It is likely that the Kac--Moody group over $\bbL$ associated with this root datum is the group ind-scheme $\Gw(\mathscr{K})$ considered above. 
This fact would make Theorem~\ref{thm:Diag-parity-affine} an actual special case of Theorem~\ref{thm:main-parity} and Theorem~\ref{thm:main-parity-BS}.
However, such a statement does not seem to be known. (See~\cite[\S 9.h]{pappas-rapoport} for a comparison of the corresponding flag varieties.)
\item
Assume again that $G$ is quasi-simple.
Instead of the ``degenerate" Kac--Moody root datum of~\eqref{it:KM-affine}, one can consider the more traditional
Kac--Moody root datum consisting of $\Lambda = \Hom_\Z(\Z \Phi \oplus \Z c \oplus \Z d, \Z)$, the simple roots $\Sigma^\vee \cup \{\delta - \theta\}$ (where $\delta(d)=1$ and $\delta_{|\Z \Phi \oplus \Z c}=0$, and where $\Sigma^\vee$ is considered as a subset of $\Lambda$ in the obvious way) and the simple coroots $\Sigma \cup \{c-\widetilde{\theta}\}$. In~\S\ref{ss:intro-tilting-characters} we recalled the definition of the $p$-canonical basis of $\cH$ using the category $\Diag$. It might be more natural to define this $p$-canonical basis using the category $\Diag^\bk(\GKM)$, where $\GKM$ is the Kac--Moody group associated with $\Lambda$ and the above roots and coroots. 
But in fact the two definitions coincide. Indeed, we have natural $\Waff$-invariant morphisms
\[
\Z \Phi \oplus \Z c \oplus \Z d \twoheadrightarrow \Z \Phi \oplus \Z d \quad \text{and} \quad \Z\Phi \hookrightarrow \Z \Phi \oplus \Z d,
\]
which allow to construct monoidal functors
\[
\eta_1 \colon \Diag' \to \Diag^\bk(\GKM) \quad \text{and} \quad \eta_2 \colon \Diag' \to \Diag,
\]
where $\Diag'$ is the category from~\cite{ew} associated with the natural realization of $\Waff$ with underlying $\bk$-module $\bk \otimes_{\Z} (\Z \Phi \oplus \Z d)$. Hence to prove our claim it suffices to prove that $\eta_1$ and $\eta_2$ send indecomposable objects to indecomposable objects. However, if we define the category $\overline{\Diag}$ with the same objects as $\Diag$ and morphisms from $M$ to $N$ given by the degree-$0$ part of
\[
\bk \otimes_{\cO(\bk \otimes_\Z \Z\Phi)}
\Hom^\bullet_{\Diag}(M,N),
\]
then $\overline{\Diag}$ is a Krull--Schmidt category, and an object in $\Diag$ is indecomposable if and only if its image in $\overline{\Diag}$ is indecomposable. Similar remarks apply to $\Diag'$ and $\Diag^\bk(\GKM)$. This suffices to conclude since, by
the double leaves theorem~\cite[Theorem~6.11]{ew}, for any $M,N$ in $\Diag'$ the morphisms
\[
\bk \otimes_{\cO(\bk \otimes_\Z (\Z \Phi \oplus \Z d))}
\Hom^\bullet_{\Diag'}(M,N) \to \bk \otimes_{\cO(\bk \otimes_\Z (\Z \Phi \oplus \Z c \oplus \Z d))} \Hom^\bullet_{\Diag^\bk(\GKM)}(\eta_1(M), \eta_1(N))
\]
induced by $\eta_1$ and
\[
\bk \otimes_{\cO(\bk \otimes_\Z (\Z \Phi \oplus \Z d))}
\Hom^\bullet_{\Diag'}(M,N) \to \bk \otimes_{\cO(\bk \otimes_\Z \Z \Phi)} \Hom^\bullet_{\Diag}(\eta_2(M), \eta_2(N))
\]
induced by $\eta_2$ are isomorphisms.
\end{enumerate}
\end{rmk}

\section{Whittaker sheaves and antispherical diagrammatic categories}
\label{sec:Whittaker}

\subsection{Definition of Whittaker sheaves}
\label{ss:IW}


Let us come back to the general setting of Sections~\ref{sec:parity}--\ref{sec:parity-Hecke}. In~\S\ref{ss:statement-Diag-parity} we have obtained a description of the category $\Diag^\K(\GKM)$ in terms of parity complexes on a flag variety. The goal of this section is to obtain a similar description for the antispherical quotient defined in Remark~\ref{rmk:asph-general}. This relies on the consideration of some ``Whittaker sheaves," which exist only in the \'etale context.

So, from now on we restrict to the \'etale context of~\S\ref{ss:DbX}, assuming in addition that $\mathrm{char}(\bbL)>0$, and that there exists a nontrivial (additive) character $\psi$ from the prime subfield of $\bbL$ to $\K^\times$ (which we fix once and for all).
Then we can consider the corresponding Artin--Schreier local system $\cL_\psi$ on $\mathbb{G}_{\mathrm{a}, \bbL}$, see~\cite{ab, by, modrap1}, i.e.~the rank-$1$ local system defined as the $\psi$-isotypic component in the direct image of the constant sheaf under the Artin--Schreier map $\mathbb{G}_{\mathrm{a}, \bbL} \to \mathbb{G}_{\mathrm{a}, \bbL}$ defined by $x \mapsto x^{\mathrm{char}(\bbL)}-x$.

Let $J \subset I$ be a subset of finite type. We will denote by $\JW \subset W$ the subset of elements $v$ which are minimal in $W_J v$ (i.e.~the image of $W^J$ under $w \mapsto w^{-1}$).
To $J$ we have associated the parabolic subgroup $\PKM_J$ of $\GKM$. We denote by $\UKM^J$ the pro-unipotent radical of $\PKM_J$ (i.e.~the base change to $\bbL$ of the group scheme denoted $\mathfrak{U}_J^{\mathrm{ma}+}$ in~\cite{rousseau}), and by $\LKM_J$ its Levi factor (i.e.~the base change to $\bbL$ of the group scheme denoted $\mathfrak{G}(J)$ in~\cite{rousseau}), which is a connected reductive $\bbL$-group. Finally, let $\UKM_J^-$ be the unipotent radical of the Borel subgroup of $\LKM_J$ which is opposite to $\BKM \cap \LKM_J$. Then the $\UKM_J^- \UKM^J$-orbits on $\Flag$ are parametrized by $W$ (in the obvious way). For $w \in W$, we will denote by $\Flag_{w,J}$ the orbit corresponding to $w$, and by $d_w^J$ its dimension.

For any $s \in \{s_j : j \in J\}$ we have a root subgroup $\UKM_s^- \subset \UKM_J^-$. Moreover, the natural embedding induces an isomorphism of algebraic groups
\[
 \prod_{s \in J} \UKM_s^- \simto \UKM_J^- / [\UKM_J^-, \UKM_J^-].
\]
If we choose once and for all, for any $s \in J$, an isomorphism $\mathbb{G}_{\mathrm{a}, \bbL} \simto \UKM_s^-$, we deduce a morphism of algebraic groups
\[
  \UKM_J^- \UKM^J \to \UKM_J^- \to \UKM_J^- / [\UKM_J^-, \UKM_J^-] \simto \prod_{s \in J} \UKM_s^- \simto (\mathbb{G}_{\mathrm{a}, \bbL})^J \xrightarrow{+} \mathbb{G}_{\mathrm{a}, \bbL}
\]
which we will denote $\chi_J$.

The local system $\cL_\psi$ is multiplicative in the sense of~\cite[Appendix~A]{modrap1}, hence so is $(\chi_J)^* \cL_\psi$. We will denote by
\[
 \Db_{\Whit,J}(\Flag, \K)
\]
the triangulated category of $(\UKM_J^- \UKM^J, (\chi_J)^* \cL_\psi)$-equivariant complexes on the variety $\Flag$ (see~\cite[Definition~A.1]{modrap1}). Note that if $w \in W$, then the $\UKM_J^- \UKM^J$-orbit parametrized by $w$ supports a $(\UKM_J^- \UKM^J, (\chi_J)^* \cL_\psi)$-equivariant local system iff $w \in \JW$. In this case there exists a unique such local system of rank one (up to isomorphism), which we will denote $\cL_w^J$, and we have $d^J_w=\ell(w)+\ell(w_0^J)$. We denote by $i^J_{w,\circ} \colon \Flag_{w,J} \hookrightarrow \Flag$ the embedding, and set
\[
\Sta^J_{w,\circ} := (i^J_{w,\circ})_! \cL^J_{w}[\ell(w) + \ell(w_0^J)], \qquad \Cos^J_{w,\circ} := (i^J_{w,\circ})_* \cL^J_{w}[\ell(w) + \ell(w_0^J)].
\]
By minimality, when $w=1$ is the neutral element we have $\Sta^J_{1,\circ} \cong \Cos^J_{1,\circ}$, see~\cite[Corollary~4.2.2]{by}.
These objects satisfy
\[
\Hom^n_{\Db_{\Whit,J}(\Flag, \K)}(\Sta^J_{w,\circ}, \Cos^J_{v,\circ}) = \begin{cases}
\K & \text{if $w=v$ and $n=0$;} \\
0 & \text{otherwise.}
\end{cases}
\]

There exists a natural convolution functor
\[
(-) \star^{\BKM} (-) \colon \Db_{\Whit,J}(\Flag, \K) \times \Db_{\BKM}(\Flag, \K) \to \Db_{\Whit,J}(\Flag, \K),
\]
defined in a way similar to the convolution on $\Db_{\BKM}(\Flag, \K)$. Then we can define the
``averaging functor''
\[
\mathsf{Av}_{J} \colon \Db_{\BKM}(\Flag, \K) \to \Db_{\Whit,J}(\Flag, \K),
\]
as the functor $\cF \mapsto \Delta^J_{1,\circ} \star^\BKM \cF$. (Alternatively, this functor can be described as the composition of the forgetful functor to the constructible derived category, followed by any choice of the functors defined as in~\cite[\S A.2]{modrap1}; see~\cite[Lemma~4.4.3]{by}.) 
By construction, for $\cF,\cG$ in $\Db_{\BKM}(\Flag, \K)$ there exists a canonical and functorial isomorphism
\begin{equation}
\label{eqn:Av-convolution}
\mathsf{Av}_{J} (\cF \star^{\BKM} \cG) \cong \mathsf{Av}_{J}(\cF) \star^{\BKM} \cG.
\end{equation}

For $s \in S$,
we can consider in a similar way the category
\[
\Db_{\Whit,J}(\Flag^{s}, \K).
\]
The $\UKM_J^- \UKM^J$-orbits on $\Flag^{s}$ are parametrized in a natural way by $W^s$;  those which support a $(\UKM_J^- \UKM^J, (\chi_J)^* \cL_\psi)$-equivariant local system correspond to the elements in 
\[
\JW^s_\circ := \{w \in W^s \mid w \in \JW \text{ and } ws \in \JW\};
\]
we denote the corresponding standard and costandard sheaves $\Sta^{s,J}_{w,\circ}$ and $\Cos^{s,J}_{w,\circ}$ respectively.

We also have 
functors $(q^s)_!=(q^s)_*$, $(q^s)^*$, $(q^s)^!$ between the $(\UKM_J^- \UKM^J, (\chi_J)^* \cL_\psi)$-equivariant categories.

\begin{lem}
\label{lem:qs-standard-costandard}
Let $w \in \JW$.
\begin{enumerate}
\item
If $w \in W^s$ and $ws \notin \JW$, then
\[
(q^s)_! \Sta^{J}_{w,\circ} = (q^s)_! \Cos^{J}_{w,\circ} = 0.
\]
\item
If $w \in W^s$ and $ws \in \JW$, then 
\[
(q^s)_! \Sta^{J}_{w,\circ} \cong \Sta^{s,J}_{w,\circ} \quad \text{and} \quad (q^s)_! \Cos^{J}_{w,\circ} \cong \Cos^{s,J}_{w,\circ}.
\]
\item
If $ws<w$ (which implies that $ws \in \JW$), then
\[
(q^s)_! \Sta^{J}_{w,\circ} \cong \Sta^{s,J}_{ws,\circ}[-1] \quad \text{and} \quad (q^s)_! \Cos^{J}_{w,\circ} \cong \Cos^{s,J}_{ws,\circ}[1].
\]
\end{enumerate}
\end{lem}

\begin{proof}
Consider the restriction of $q^s$ to $\Flag_{w,J}$. If $w \in W^s$ and $ws \notin \JW$, then this restriction is an $\mathbb{A}^1$-fibration, and the restriction of $\cL^J_{w}$ to the fibers are isomorphic to $\cL_\psi$, so that $(q^s)_! \Sta^{J}_{w,\circ} =(q^s)_! \Cos^{J}_{w,\circ} = 0$. 

If $w \in W^s$ and $ws \in \JW$, then the restriction of $q^s$ is an isomorphism, so that $(q^s)_! \Sta^{J}_{w,\circ} \cong \Sta^{s,J}_{w,\circ}$ and $(q^s)_! \Cos^{J}_{w,\circ} \cong \Cos^{s,J}_{w,\circ}$.

Finally, if $w \notin W^s$, then the restriction of $q^s$ to $\Flag_{w,J}$ is an $\mathbb{A}^1$-fibration, and the restriction of $\cL^J_{w}$ to the fibers are constant; it follows that $(q^s)_! \Sta^{J}_{w,\circ} \cong \Sta^{s,J}_{ws,\circ}[-1]$ and $(q^s)_! \Cos^{J}_{w,\circ} \cong \Cos^{s,J}_{ws,\circ}[1]$.
\end{proof}

\begin{lem}
\label{lem:qs-inverse-standard-costandard}
Let $w \in \JW^s_\circ$.
There exists distinguished triangles
\[
\Sta^J_{ws,\circ}[-1] \to (q^s)^* \Sta^{s,J}_{w,\circ} \to \Sta^J_{w,\circ} \triright
\quad \text{and} \quad
\Cos^J_{w,\circ} \to (q^s)^! \Cos^{s,J}_w \to \Cos^J_{ws,\circ}[1] \triright.
\]
\end{lem}

\begin{proof}
These distinguished triangles are obtained by using the base change theorem and applying the standard distinguished triangles of functors attached to the decomposition of the inverse image under $q^s$ of the orbit corresponding to $w$
as the disjoint union of the closed piece $\Flag_{w,J}$ and the open piece $\Flag_{ws,J}$.
\end{proof}

\subsection{Whittaker parity complexes}
\label{ss:IW-parity}

The general theory of parity complexes developed in~\cite{jmw} applies verbatim in the setting of Whittaker sheaves. In particular, a complex $\cF$ in $\Db_{\Whit,J}(\Flag, \K)$ is said to be $*$-even, resp.~$!$-even, if for any $w \in W$ we have
\[
\cH^{\mathrm{odd}} \bigl( (i^J_{w,\circ})^* \cF \bigr)=0, \quad \text{resp.} \quad \cH^{\mathrm{odd}} \bigl( (i^J_{w,\circ})^! \cF \bigr)=0,
\]
and if moreover for $w \in W$ and $k$ even the local system
\[
\cH^k \bigl( (i^J_{w,\circ})^* \cF \bigr), \quad \text{resp.} \quad \cH^k \bigl( (i^J_{w,\circ})^! \cF \bigr),
\]
is projective (equivalently, free) over $\K$.
(These conditions are automatic if $w \notin \JW$ since in this case $(i^J_{w,\circ})^* \cF=(i^J_{w,\circ})^! \cF=0$ for any $\cF$ in $\Db_{\Whit,J}(\Flag, \K)$.) Then an object $\cF$ is called a \emph{parity complex} if $\cF \cong \cF_0 \oplus \cF_1$ where the objects $\cF_0$ and $\cF_1[1]$ are both $*$-even and $!$-even. We will denote by
\[
\Par_{\Whit,J}(\Flag, \K)
\]
the additive category of parity complexes in $\Db_{\Whit,J}(\Flag, \K)$. For any $s \in S$, we can similarly consider the category
\[
\Par_{\Whit,J}(\Flag^{s}, \K)
\]
of parity complexes in $\Db_{\Whit,J}(\Flag^{s}, \K)$.

\begin{lem}
\label{lem:parity-IW-direct}
Let $\cF \in \Db_{\Whit,J}(\Flag, \K)$ and $s \in S$.
\begin{enumerate}
\item
\label{it:parity-IW-*}
If $\cF$ is $*$-even, then $(q^s)_* \cF$ is $*$-even.
\item
\label{it:parity-IW-!}
If $\cF$ is $!$-even, then $(q^s)_* \cF$ is $!$-even.
\item
\label{it:parity-IW-*!}
If $\cF$ is a parity complex, then $(q^s)_* \cF$ is a parity complex.
\end{enumerate}
\end{lem}

\begin{proof}
Clearly,~\eqref{it:parity-IW-*!} follows from~\eqref{it:parity-IW-*} and~\eqref{it:parity-IW-!}. We explain the proof of~\eqref{it:parity-IW-*}; the proof of~\eqref{it:parity-IW-!} is similar.

So, let us assume that $\cF$ is $*$-even. We proceed by induction on the support of $\cF$; so we assume that $\cF$ is supported on a closed union $\mathscr{Y} \subset \Flag$ of $\UKM_J^- \UKM^J$-orbits, and that $w \in W$ is such that $\Flag_{w,J}$ is open in $\mathscr{Y}$ and $(i^J_{w,\circ})^*\cF \neq 0$. Necessarily we have $w \in \JW$. Let $k \colon \mathscr{Y} \smallsetminus \Flag_{w,J} \hookrightarrow \mathscr{Y}$ be the closed embedding, and consider the canonical distinguished triangle
\[
(i^J_{w,\circ})_! (i^J_{w,\circ})^! \cF \to \cF \to k_* k^* \cF \triright.
\]
The object $k_* k^* \cF$ is $*$-even and supported on $\mathscr{Y} \smallsetminus \Flag_{w,J}$, hence by induction we know that $(q^s)_* k_* k^* \cF$ is $*$-even. On the other hand, we have an isomorphism
\[
(i^J_{w,\circ})_! (i^J_{w,\circ})^! \cF \cong \bigoplus_{i \in \Z} \bigl( \Sta^J_{w,\circ}[i] \bigr)^{\oplus n_i},
\]
where $n_i \in \Z_{\geq 0}$ and $n_i=0$ unless $\dim(\Flag_{w,J})+i$ is even. Now it follows from Lemma~\ref{lem:qs-standard-costandard} that under this condition the object $(q^s)_* (\Sta^J_{w,\circ}[i])$ is $*$-even, and the claim follows.
\end{proof}

%

Using arguments similar to those for Lemma~\ref{lem:qJ-parity}, one can check that the functors $(q^s)^! (q^s)_*$ and $(q^s)^* (q^s)_*$ are isomorphic up to cohomological shift by $2$. (In our present setting Verdier duality exchanges the Whittaker categories defined in terms of $\psi$ and $\psi^{-1}$; but this does not cause any trouble.) We deduce the following.

\begin{lem}
\label{lem:parity-IW-inverse}
If $\cF \in \Db_{\Whit,J}(\Flag,\K)$ is parity, then so are $(q^s)^* (q^s)_* \cF$ and $(q^s)^! (q^s)_* \cF$.
\end{lem}

\begin{cor}
\label{cor:Av-parity}
For any $\cE$ in $\Par_{\BKM}(\Flag, \K)$ the object $\mathsf{Av}_{J}(\cE)$
is a parity complex.
\end{cor}

\begin{proof}
Since the functor $(q^s)^* (q^s)_*$ commutes with $\mathsf{Av}_{J}$ (see~\eqref{eqn:Av-convolution}) and since any object in $\Par_{\BKM}(\Flag, \K)$ is isomorphic to a direct sum of shifts of direct summands of objects obtained from $\Sta_1$ by repeated application of functors $(q^s)^* (q^s)_*$ for various $s$, the corollary follows from the observation that $\mathsf{Av}_{J}(\Sta_1)=\Sta_{1,\circ}^J$ is parity (since $\Sta_{1,\circ}^J \cong \Cos_{1,\circ}^J$) and Lemma~\ref{lem:parity-IW-inverse}.
\end{proof}

\begin{rmk}
\label{rmk:indec-IW-parity}
For any $w \in \JW$,
the general theory of parity complexes of~\cite{jmw} guarantees the uniqueness (up to isomorphism) of an indecomposable parity complex supported on $\overline{\Flag_{w,J}}$ and whose restriction to $\Flag_{w,J}$ is $\cL^J_{w}[\dim(\Flag_{w,J})]$, but not its existence. This existence follows from Corollary~\ref{cor:Av-parity}: in fact $\mathsf{Av}_{J}(\cE_w)$ is a parity complex supported on $\overline{\Flag_{w,J}}$ and whose restriction to $\Flag_{w,J}$ is $\cL^J_{w}[\dim(\Flag_{w,J})]$; hence it admits an indecomposable direct summand whose restriction to $\Flag_{w,J}$ is $\cL^J_{w}[\dim(\Flag_{w,J})]$. Similar comments apply to the variety $\Flag^s$. Using this, as in Lemma~\ref{lem:qs^*-parity} one can check that the functors $(q^s)^*$ and $(q^s)^!$ send parity complexes to parity complexes.
\end{rmk}

\begin{lem}
\label{lem:Av-Ew-zero}
If $w \in W \smallsetminus \JW$, then $\Av_J(\cE_w)=0$.
\end{lem}

\begin{proof}
Let $\uw$ be a reduced expression for $w$ starting with a simple reflection $s_j$ for some $j \in J$. Then $\cE_w$ is isomorphic to a direct summand of $\cE_{\uw}$. On the other hand $\cE_{\uw}$ is the image in $\Db_{\BKM}(\Flag,\K)$ of an object in the equivariant derived category of the parabolic subgroup $\PKM_{s_j}$. Therefore we have $\Av_J(\cE_{\uw})=0$ (see e.g.~\cite[Lemma~4.4.6 and its proof]{by}), which implies that $\Av_J(\cE_w)=0$.
\end{proof}

\subsection{Sections of the $!$-flag for Whittaker parity complexes}
\label{ss:section-Whit}

In this subsection we assume that $\K$ is a field. Our goal is to explain how to develop a theory of ``sections of the $!$-flag" for Whittaker parity complexes. Since the constructions are close to those performed in Section~\ref{sec:parity}, we leave most of the details to the reader.

We will denote by $(\cE^J_{w,\circ} : w \in \JW)$ the (normalized) indecomposable objects in the category $\Par_{\Whit,J}(\Flag,\K)$ and by $(\cE^{s,J}_{w,\circ} : w \in \JW^s_\circ)$ the similar objects in $\Par_{\Whit,J}(\Flag^s,\K)$. Then one can define the sections of the $!$-flag for $!$-parity objects in $\Db_{\Whit,J}(\Flag, \K)$ and $\Db_{\Whit,J}(\Flag^{s}, \K)$ in the obvious way. (In this case, one replaces the Bruhat order which appears in the condition for the morphisms $\varphi^\cF_\pi$ by the order on $W$ induced by inclusions of orbit closures in $\Flag$ or in $\Flag^s$; the corresponding open ind-subvarieties of $\Flag$ will be denoted $\Flag_{\geq w, J}$.)

Next, we fix some split embeddings and projections
\[
\cE_{w,\circ}^{s,J} \to (q^s)_* \cE^J_{w,\circ} \to \cE_{w,\circ}^{s,J}
\]
for $w \in \JW^s_\circ$, and a split embedding and split projection
\[
\cE^J_{ws,\circ} \to (q^s)^! \cE^{s,J}_{w,\circ}[-1], \quad (q^s)^* \cE^{s,J}_{w,\circ}[1] \to \cE^J_{ws, \circ}
\]
for $w \in \JW^s_\circ$
respectively. This allows to obtain analogues of Proposition~\ref{prop:section-!-push} and Proposition~\ref{prop:section-!-pull}. Namely, if $(\Pi,e,d,(\varphi^\cF_\pi)_{\pi \in \Pi})$ is a section of the $!$-flag of $\cF \in \Db_{\Whit,J}(\Flag,\K)$, then one can define a section of the $!$-flag of $(q^s)_*\cF$ as follows. We set $\Pi':=\{\pi \in \Pi \mid e(\pi)s \in \JW\}$, and define $e' : \Pi' \to \JW^s_\circ$ as sending $\pi$ to the shortest element among $e(\pi)$ and $e(\pi) s$. Then we can define $d'$ and $(\varphi^{(q^s)_* \cF}_{\pi})_{\pi \in \Pi'}$ by the same procedure as in~\S\ref{ss:section-!-flag-push}. (The proof that this indeed defines a section of the $!$-flag of $(q^s)_* \cF$ is similar to the proof of Proposition~\ref{prop:section-!-push}; details are left to the reader.) On the other hand, if $(\Pi,e,d,(\varphi^\cF_\pi)_{\pi \in \Pi})$ is a section of the $!$-flag of $\cF \in \Db_{\Whit,J}(\Flag^s,\K)$, then one can define a section of the $!$-flag of $(q^s)^! \cF$ by exactly the same procedure as in~\S\ref{ss:section-!-flag-pull}.

Using these constructions one can prove the following analogue of Proposition~\ref{prop:morphisms-BS-parity}.

\begin{prop}
\label{prop:morphisms-BS-parity-Whit}
Let $\ux$ and $\uv$ be expressions, and assume that $\ux$ is a reduced expression for some element $x \in \JW$. Let also $s \in S$, and assume that $xs \in \JW$.
\begin{enumerate}
\item
\label{it:morphisms-BS-parity-Whit-1}
Assume that $x \in W^s$, so that $\ux s$ is a reduced expression for $xs \in \JW$. Let $(f_i)_{i \in \mathfrak{I}}$ be a family of homogeneous elements in 
\[
\Hom^\bullet_{\Db_{\Whit,J}(\Flag,\K)}(\Av_J(\cE_{\ux}), \Av_J(\cE_{\uv})) 
\]
whose images span 
\[
\Hom^\bullet_{\Db_{\Whit,J}(\Flag_{\geq x,J},\K)}(\Av_J(\cE_{\ux}), \Av_J(\cE_{\uv})),
\]
and let $(g_j)_{j \in \mathfrak{J}}$ be a family of homogeneous elements in 
\[
\Hom^\bullet_{\Db_{\Whit,J}(\Flag,\K)}(\Av_J(\cE_{\ux s}), \Av_J(\cE_{\uv}))
\]
whose images span the vector space
\[
\Hom^\bullet_{\Db_{\Whit,J}(\Flag_{\geq xs,J},\K)}(\Av_J(\cE_{\ux s}), \Av_J(\cE_{\uv})).
\]
Then there exist integers $n_i$ and morphisms $f'_i \colon \Av_J(\cE_{\ux}) \to \Av_J(\cE_{\ux s})[n_i]$ (for $i \in \mathfrak{I}$) and integers $m_j$ and morphisms $g'_j \colon \Av_J(\cE_{\ux}) \to \Av_J(\cE_{\ux s s})[m_j]$ (for $j \in \mathfrak{J}$) such that the images of the compositions
\[
\Av_J(\cE_{\ux}) \xrightarrow{f'_i} \Av_J(\cE_{\ux s})[n_i] \xrightarrow{f_i \star^\BKM \cE_s [n_i]} \Av_J(\cE_{\uv s}) [n_i + \deg(f_i)]
\]
together with the images of the compositions
\[
\Av_J(\cE_{\ux}) \xrightarrow{g'_j} \Av_J(\cE_{\ux s s})[m_j] \xrightarrow{g_j \star^\BKM \cE_s[m_j]} \Av_J(\cE_{\uv s}) [m_j + \deg(g_j)]
\]
span the vector space $\Hom^\bullet_{\Db_{\Whit,J}(\Flag_{\geq x,J},\K)}(\Av_J(\cE_{\ux}), \Av_J(\cE_{\uv s}))$.
\item
\label{it:morphisms-BS-parity-Whit-2}
Assume that $\ux=\uy s$ for some expression $\uy$ (which is automatically a reduced expression for $xs$). Let $(f_i)_{i \in \mathfrak{I}}$ be a family of homogeneous elements in
\[
\Hom^\bullet_{\Db_{\Whit,J}(\Flag,\K)}(\Av_J(\cE_{\ux}), \Av_J(\cE_{\uv}))
\]
whose images span
\[
\Hom^\bullet_{\Db_{\Whit,J}(\Flag_{\geq x,J},\K)}(\Av_J(\cE_{\ux}), \Av_J(\cE_{\uv})),
\]
and let $(g_j)_{j \in \mathfrak{J}}$ be a family of homogeneous elements in
\[
\Hom^\bullet_{\Db_{\Whit,J}(\Flag,\K)}(\Av_J(\cE_{\uy}), \Av_J(\cE_{\uv}))
\]
whose images span
\[
\Hom^\bullet_{\Db_{\Whit,J}(\Flag_{\geq xs,J},\K)}(\Av_J(\cE_{\uy}), \Av_J(\cE_{\uv})).
\]
Then there exist integers $n_i$ and morphisms $f'_i \colon \Av_J(\cE_{\ux}) \to \Av_J(\cE_{\ux s})[n_i]$ (for $i \in \mathfrak{I}$) and integers $m_j$ and morphisms $g'_j \colon \Av_J(\cE_{\ux}) \to \Av_J(\cE_{\ux})[m_j]$ (for $j \in \mathfrak{J}$) such that the images of the compositions
\[
\Av_J(\cE_{\ux}) \xrightarrow{f'_i} \Av_J(\cE_{\ux s})[n_i] \xrightarrow{f_i \star^\BKM \cE_s[n_i]} \Av_J(\cE_{\uv s}) [n_i + \deg(f_i)]
\]
together with the images of the compositions
\[
\Av_J(\cE_{\ux}) \xrightarrow{g'_j} \Av_J(\cE_{\ux})[m_j] \xrightarrow{g_j \star^\BKM \cE_s[m_j]} \Av_J(\cE_{\uv s}) [m_j + \deg(g_j)]
\]
span the vector space $\Hom^\bullet_{\Db_{\Whit,J}(\Flag_{\geq x,J},\K)}(\Av_J(\cE_{\ux}), \Av_J(\cE_{\uv s}))$.
\end{enumerate}
\end{prop}

\subsection{Surjectivity}

We drop the assumption that $\K$ is a field.
Corollary~\ref{cor:Av-parity} shows that the functor $\mathsf{Av}_J$ restricts to a functor
\[
\Par_{\BKM}(\Flag, \K) \to \Par_{\Whit,J}(\Flag, \K);
\]
we will denote this restriction also by $\mathsf{Av}_J$.

\begin{prop}
\label{prop:morphism-IW-surjective}
For any $\cE, \cF$ in $\Par_{\BKM}(\Flag, \K)$, the morphism
\[
\Hom^\bullet_{\Par_{\BKM}(\Flag, \K)}(\cE,\cF) \to \Hom^\bullet_{\Par_{\Whit,J}(\Flag, \K)}(\mathsf{Av}_J(\cE), \mathsf{Av}_J(\cF))
\]
induced by the functor $\mathsf{Av}_J$ is surjective.
\end{prop}

\begin{proof}[Sketch of proof]
Using the Nakayama lemma and the appropriate analogue of Lem\-ma~\ref{lem:morphisms-parityBS}, it suffices to consider the case when $\K$ is a field, to which we restrict from now on.

In this setting, the surjectivity
can be proved by using once again the technique of the proof of~\S\ref{ss:proof-thm} and~\S\ref{ss:Diag-parity-ff}. Namely, one can assume that $\cE$ and $\cF$ are of Bott--Samelson type, and then using adjunction that $\cE=\cE_\K(\varnothing)$ and $\cF=\cE_\K(\uv)$ for some expression $\uv$. To treat this case we prove more generally that if $\uw$ is a reduced expression for $w \in \JW$ and $\uv$ is any expression, the composition
\begin{multline*}
\delta^J_{\uw,\uv} \colon \Hom^\bullet_{\Par_{\BKM}(\Flag,\K)}(\cE_\K(\uw), \cE_\K(\uv)) \xrightarrow{\Av_J} \\
\Hom^\bullet_{\Par_{\Whit,J}(\Flag, \K)}(\mathsf{Av}_J(\cE_\K(\uw)), \mathsf{Av}_J(\cE_\K(\uv))) \to \\
\Hom^\bullet_{\Par_{\Whit,J}(\Flag_{\geq w, J}, \K)}(\mathsf{Av}_J(\cE_\K(\uw)), \mathsf{Av}_J(\cE_\K(\uv)))
\end{multline*}
(where the second map is induced by restriction) is surjective. For this we first consider as in Lemma~\ref{lem:surjectivity-parity} the case of morphisms $\delta^J_{\uy s, \uy s}$, $\delta^J_{\uy s, \uy s s}$, $\delta^J_{\uy, \uy s}$, $\delta^J_{\uy, \uy s s}$ where $\uy$ is a reduced expression for an element $y \in \JW^s_\circ$.
Then we treat the general case by induction on $\ell(\uv)$ as follows.

If $\ell(\uv)=0$ the claim is obvious. If $\ell(\uv)>0$, we write $\uv=\uu s$. If $ws<w$, then $ws \in \JW$, and it suffices to treat the case when $\uw=\ux s$ for some reduced expression $\ux$ for $ws$. In this case we invoke Proposition~\ref{prop:morphisms-BS-parity-Whit}\eqref{it:morphisms-BS-parity-Whit-2} as in the proof of Proposition~\ref{prop:surjectivity-parity}. If $ws>w$ and $ws \in \JW$, then we can invoke Proposition~\ref{prop:morphisms-BS-parity-Whit}\eqref{it:morphisms-BS-parity-Whit-1}.

Finally, it remains to consider the case when
$ws>w$ and $ws \notin \JW$. In this case, since $\Av_J(\cE_\K(\uv))$ is the image under $(q^s)^!$ of an object in $\Db_{\Whit,J}(\Flag^s,\K)$, its co-restriction to a stratum $\Flag_{x,J}$ can be nonzero only if $x$ and $xs$ are both in $\JW$. In particular, its co-restriction to $\Flag_{w,J}$ vanishes, so that the codomain of $\delta^J_{\uw,\uv}$ vanishes. Hence there is nothing to prove in this case.
%
\end{proof}

\begin{cor}
If $w \in \JW$, then $\mathsf{Av}_J(\cE_w)$ is a nonzero indecomposable parity complex.
\end{cor}

\begin{proof}
If $w \in \JW$, then $\mathsf{Av}_J(\cE_w)$ is supported on $\overline{\Flag_{w,J}}$, and its restriction to the orbit $\Flag_{w,J}$ is $\cL^J_{w}[\dim(\Flag_{w,J})]$, see Remark~\ref{rmk:indec-IW-parity}. It is indecomposable by Proposition~\ref{prop:morphism-IW-surjective} and the fact that a quotient of a local ring is local.
\end{proof}

\subsection{Description of the antispherical diagrammatic category in terms of Whittaker sheaves}
\label{ss:Dasph-parity}

By Corollary~\ref{cor:Av-parity}, we can define $\Par_{\Whit,J}^{\mathrm{BS}}(\Flag, \K)$ as the full (but not strictly full) subcategory of $\Par_{\Whit,J}(\Flag, \K)$ whose objects are the parity complexes
\[
\Av_J(\cE_\K(\uw))[n]
\]
for $\uw$ an expression and $n \in \Z$. Using~\eqref{eqn:Av-convolution}, we see that this category admits a natural action of the monoidal category $\Par^{\mathrm{BS}}_{\BKM}(\Flag, \K)$ on the right. Moreover, using Remark~\ref{rmk:indec-IW-parity}, one can easily check that the Karoubi envelope of the additive hull of $\Par^{\mathrm{BS}}_{\Whit,J}(\Flag, \K)$ is $\Par_{\Whit,J}(\Flag, \K)$.

On the other hand, consider the antispherical category $\Diag^{\mathrm{asph},\K}_{\mathrm{BS},J}(\GKM)$ defined as in Remark~\ref{rmk:asph-general}, for the Coxeter group $(W,S)$, its realization $\fh^\K_\GKM$, and the subset $\{s_j : j \in J\} \subset S$. We also denote by $\Diag^{\mathrm{asph},\K}_J(\GKM)$ the Karoubi envelope of the additive hull of $\Diag^{\mathrm{asph},\K}_{\mathrm{BS},J}(\GKM)$.

\begin{thm}
\label{thm:Dasph-parities-Whit}
Assume that there exists a ring morphism $\Z' \to \K$. Then
there exists a canonical equivalence of categories
\[
\Delta_{\mathrm{BS}}^J \colon \Diag^{\mathrm{asph},\K}_{\mathrm{BS},J}(\GKM) \simto \Par^{\mathrm{BS}}_{\Whit,J}(\Flag, \K)
\]
which is compatible with the right actions of $\DiagBS^\K(\GKM)$ and $\Par^{\mathrm{BS}}_{\BKM}(\Flag, \K)$ through the equivalence of Theorem~{\rm \ref{thm:main-parity-BS}}. As a consequence, we obtain an equivalence of categories
\[
\Delta^J \colon \Diag^{\mathrm{asph},\K}_J(\GKM) \simto \Par_{\Whit,J}(\Flag, \K)
\]
compatible with the right actions of $\Diag^\K(\GKM)$ and $\Par_{\BKM}(\Flag, \K)$ through the equivalence of Theorem~{\rm \ref{thm:main-parity}}.
\end{thm}

\begin{proof}
We consider the diagram
\[
\xymatrix@C=2cm@R=0.5cm{
\DiagBS^\K(\GKM) \ar[r]^-{\Delta_{\mathrm{BS}}}_-{\sim} \ar[d] & \Par^{\mathrm{BS}}_{\BKM}(\Flag, \K) \ar[d]^-{\mathsf{Av}_J} \\
\Diag^{\mathrm{asph},\K}_{\mathrm{BS},J}(\GKM) & \Par^{\mathrm{BS}}_{\Whit,J}(\Flag, \K),
}
\]
where the left vertical arrow is the natural quotient functor, and the horizontal arrow is the equivalence of Theorem~\ref{thm:main-parity-BS}. Since $\mathsf{Av}_\chi(\cE_\K(\uw))=0$ for any expression $\uw$ starting with a simple reflection in $\{s_j : j \in J\}$ (see the proof of Lemma~\ref{lem:Av-Ew-zero}), and $\Av_J(\Delta_{\mathrm{BS}}(f))=0$ for any $f \in (\fh^\K_\GKM)^*$ (because $\Hom_{\Par^{\mathrm{BS}}_{\Whit,J}(\Flag, \K)}(\Sta^J_{1,\circ},\Sta^J_{1,\circ}[2])=0$), the composition $\Av_J \circ \Delta_{\mathrm{BS}}$ factors through a functor 
\begin{equation}
\label{eqn;functor-DiagBS}
\Delta_{\mathrm{BS}}^J \colon \Diag^{\mathrm{asph},\K}_{\mathrm{BS},J}(\GKM) \simto \Par^{\mathrm{BS}}_{\Whit,J}(\Flag, \K)
\end{equation}
which is the natural bijection on objects. It follows from Theorem~\ref{thm:main-parity-BS} and Proposition~\ref{prop:morphism-IW-surjective} that this functor induces surjections on morphisms, and what remains is to prove that these surjections are in fact isomorphisms. 

To prove this property, using adjunction (as e.g.~in~\S\ref{ss:proof-thm}), it suffices to consider morphisms (of any degree) from $\overline{B}_{\uw}$ to $\oB_\varnothing$ where $\uw$ is any expression. In this case, the version of Proposition~\ref{prop:ASLL} in our present setting (see Remark~\ref{rmk:asph-general}) provides a finite generating family for the $\K$-module $\Hom^\bullet_{\Diag^{\mathrm{asph},\K}_{\mathrm{BS},J}(\GKM)}(\oB_{\uw}, \oB_\varnothing)$. Now, as in Lemma~\ref{lem:graded-ranks}, one can prove that the $\K$-module
\[
\bigoplus_{n \in \Z} \Hom(\mathsf{Av}_J(\cE_\K(\uw)),\mathsf{Av}_J(\cE_\K(\varnothing))[n])
\]
is free, and compute its rank in terms of the antispherical module for the Hecke algebra of $(W,S)$ associated with $J$ (based on Lemma~\ref{lem:qs-standard-costandard} and Lemma~\ref{lem:qs-inverse-standard-costandard}). The rank one obtains in this way is precisely the cardinality of the generating family of $\Hom^\bullet_{\Diag^{\mathrm{asph},\K}_{\mathrm{BS},J}(\GKM)}(\oB_{\uw}, \oB_\varnothing)$ considered above (by the appropriate generalization of Lemma~\ref{lem:number-subexpr-avoids}). Since any surjective morphism from a $\K$-module generated by $m$ elements to $\K^m$ must be an isomorphism, we deduce the desired claim, and the first equivalence of the theorem.
The second equivalence follows by taking the Karoubi envelope of the additive hull of each of these categories.
\end{proof}

\subsection{Application to the light leaves basis in the antispherical category}
\label{ss:LL-basis-Dasph}

As noticed in the course of the proof of Theorem~\ref{thm:Dasph-parities-Whit}, it follows from the constructions of the present section that the generating family for the $\K$-module $\Hom^\bullet_{\Diag^{\mathrm{asph},\K}_{\mathrm{BS},J}(\GKM)}(\oB_{\uw}, \oB_\varnothing)$ considered in Proposition~\ref{prop:ASLL} actually forms a basis of this $\K$-module, when the \'etale derived category with coefficients in $\K$ makes sense.
In particular, this implies that for any expressions $\uw$ and $\uv$, the graded $\K$-module $\Hom^\bullet_{\Diag^{\mathrm{asph},\K}_{\mathrm{BS},J}(\GKM)}(\oB_{\uw}, \oB_{\uv})$ is free. We conclude the paper with the following lemma, which allows to deduce that the same property holds for more general rings of coefficients.

\begin{lem}
\label{lem:LLbasis-asph}
Let $\K$ and $\K'$ be integral domains such that there exists a ring morphism $\Z' \to \K$,
and consider a ring morphism $\varphi \colon \K \to \K'$.
\begin{enumerate}
\item
\label{it:LLbasis-asph-1}
Assume that $\varphi$ is injective, and that the family of Proposition~{\rm \ref{prop:ASLL}} forms a basis of $\Hom^\bullet_{\Diag^{\mathrm{asph},\K'}_{\mathrm{BS},J}(\GKM)}(\oB_{\uw}, \oB_\varnothing)$ for coefficients $\K'$. Then the same property holds for $\K$.
\item
\label{it:LLbasis-asph-2}
Assume that $\K'$ is free over $\K$, and that the family of Proposition~{\rm \ref{prop:ASLL}} forms a basis of $\Hom^\bullet_{\Diag^{\mathrm{asph},\K}_{\mathrm{BS},J}(\GKM)}(\oB_{\uw}, \oB_\varnothing)$ for coefficients $\K$. Then the same property holds for $\K'$.
\end{enumerate}
\end{lem}

\begin{proof}
Consider the following diagram:
\[
\xymatrix@C=1.5cm@R=0.6cm{
\K' \otimes_{\K} \Hom^\bullet_{\DiagBS^\K(\GKM)}(B_{\uw}, B_\varnothing) \ar[r]^-{\sim} \ar@{->>}[d] & \Hom^\bullet_{\DiagBS^{\K'}(\GKM)}(B_{\uw}, B_\varnothing) \ar@{->>}[d] \\
\K' \otimes_{\K} \Hom^\bullet_{\Diag^{\mathrm{asph},\K}_{\mathrm{BS},J}(\GKM)}(\oB_{\uw}, \oB_\varnothing) & \Hom^\bullet_{\Diag^{\mathrm{asph},\K'}_{\mathrm{BS},J}(\GKM)}(\oB_{\uw}, \oB_\varnothing).
}
\]
Here the upper horizontal morphism is as in~\eqref{eqn:Diag-K-K'}, and the vertical arrows are induced by the respective quotient functors. From this diagram we deduce a natural surjective morphism
\begin{equation}
\label{eqn:Dasph-Hom-K-K'}
\K' \otimes_{\K} \Hom^\bullet_{\Diag^{\mathrm{asph},\K}_{\mathrm{BS},J}(\GKM)}(\oB_{\uw}, \oB_\varnothing) \twoheadrightarrow \Hom^\bullet_{\Diag^{\mathrm{asph},\K'}_{\mathrm{BS},J}(\GKM)}(\oB_{\uw}, \oB_\varnothing).
\end{equation}

Now we can prove~\eqref{it:LLbasis-asph-1}. In fact if the property is known for $\K'$ then the surjection in~\eqref{eqn:Dasph-Hom-K-K'} must be an isomorphism. Hence the image of our family is free (over $\K'$) in $\K' \otimes_{\K} \Hom^\bullet_{\Diag^{\mathrm{asph},\K}_{\mathrm{BS},J}(\GKM)}(\oB_{\uw}, \oB_\varnothing)$, which implies that the family is free (over $\K$) in $\Hom^\bullet_{\Diag^{\mathrm{asph},\K}_{\mathrm{BS},J}(\GKM)}(\oB_{\uw}, \oB_\varnothing)$ since $\varphi$ is injective.

Finally, we prove~\eqref{it:LLbasis-asph-2}. Consider the $\K$-basis of $\Hom^\bullet_{\DiagBS^\K(\GKM)}(B_{\uw}, B_\varnothing)$ constructed as in the proof of Lemma~\ref{lem:morphisms-Dasph-integral}. The image of this basis in $\Hom^\bullet_{\DiagBS^{\K'}(\GKM)}(B_{\uw}, B_\varnothing)$ is a $\K'$-basis of this $\K'$-module. As in the proof of Lemma~\ref{lem:morphisms-Dasph-integral}, what we have to prove is that any homogeneous element in $\Hom^\bullet_{\DiagBS^{\K'}(\GKM)}(B_{\uw}, B_\varnothing)$ which factors through an object $B_{\ux} \langle k \rangle$ where $\ux$ starts with a simple reflection $s_j$ with $j \in J$ is a linear combination of the elements $\varphi_{\ue}$ where $\ue$ does not avoid $\Waff \smallsetminus \JW$ and the elements $f_i \cdot \varphi_{\ue}$. Choosing a basis for $\K'$ over $\K$, from the knowledge of this property over $\K$ one easily deduces the property over $\K'$, and the proof is complete.
\end{proof}

Using Lemma~\ref{lem:LLbasis-asph}\eqref{it:LLbasis-asph-2} one obtains that the property holds over any field 
admitting a ring morphism from $\Z'$, and
then using~\eqref{it:LLbasis-asph-1} one deduces that it holds over any integral domain satisfying this condition.

\subsection{Iwahori--Whittaker sheaves on the affine flag variety}
\label{ss:Whit-affine}

Let us consider the setting of~\S\ref{ss:affine-flag-variety} (in the \'etale context), and let us assume in addition that $\mathrm{char}(\bbL)>0$, that
$\K$ satisfies~\eqref{eqn:assumption-pairing}, that there exists a ring morphism $\Z' \to \K$
(so that~\eqref{eqn:assumption-Dem-surjectivity} is satisfied), and that there exists a nontrivial (additive) character $\psi$ from the prime subfield of $\bbL$ to $\K^\times$ (which we fix once and for all). Then we can consider
the categories $\DasphBS$ and $\Dasph$ of~\S\ref{ss:cat-Masph}, with our present choice of ring of coefficients $\K$.

Let $\Iw_\circ \subset \Gw(\mathscr{O})$ be the inverse image under the evaluation map $\Gw(\mathscr{O}) \to \Gw$ of the unipotent radical of the Borel subgroup which is opposite to $\Bw$ (with respect to $\Tw$). Then after choosing isomorphisms between $\mathbb{G}_{\mathrm{a}, \bbL}$ and the appropriate root subgroups, we obtain a morphism $\chi \colon \Iw_\circ \to \mathbb{G}_{\mathrm{a}, \bbL}$, and we can consider the corresponding category $\Par_{(\Iw_\circ, \chi)}(\Fl^\wedge, \K)$ of Whittaker parity complexes. (Of course, in this setting we can also consider more general parahoric subgroups, but we restrict to this special case for simplicity and concreteness.)

The proof of the following result is very similar to the proof of Theorem~\ref{thm:Dasph-parities-Whit}; see the remarks in~\S\ref{ss:affine-flag-variety}.

\begin{thm}
\label{thm:Dasph-parities}
Assume that there exists a ring morphism $\Z' \to \K$. Then
there exists a canonical equivalence of categories
\[
\DasphBS \simto \Par^{\mathrm{BS}}_{(\Iw_\circ, \chi)}(\Fl^\wedge, \K)
\]
which is compatible with the right actions of $\DiagBS$ and $\Par^{\mathrm{BS}}_{\Iw}(\Fl^\wedge, \K)$ through the first equivalence in Theorem~{\rm \ref{thm:Diag-parity-affine}}. As a consequence, we obtain an equivalence of categories
\[
\Dasph \simto \Par_{(\Iw_\circ, \chi)}(\Fl^\wedge, \K)
\]
compatible with the right actions of $\Diag$ and $\Par_{\Iw}(\Fl^\wedge, \K)$ through the second equivalence in Theorem~{\rm \ref{thm:Diag-parity-affine}}.
\end{thm}

 }
\else { \newpage }
\fi

\newpage

\addtocontents{toc}{\protect\addvspace{1.5em}}





\section*{List of notations}


Section~\ref{sec:tilting-hw}
\vspace{1mm}

\begin{tabular}{|l|l|l|}
\hline
Notation & \S & Description \\
\hline
$\mathcal{A}_{\Omega}$ & \ref{ss:hwcat} & Subcategory of $\mathcal{A}$ generated by simples with label in $\Omega$ \\
$\mathcal{A}^{\Omega}$ &  \ref{ss:hwcat} & Quotient $\mathcal{A}/\mathcal{A}_{\Lambda \smallsetminus \Omega}$ \\
$\Tilt(\mathcal{A})$ & \ref{ss:tiltings} & Category of tilting objects in $\mathcal{A}$ \\
\hline
\end{tabular}
\bigskip

Section~\ref{sec:blocks}
\vspace{1mm}

\begin{tabular}{|l|l|l|}
\hline
Notation & \S & Description \\
\hline
$\bk$ & \ref{ss:definitions-G} & Algebraically closed field of characteristic $p$ \\
$G$ & \ref{ss:definitions-G} & Reductive group with simply-connected derived subgroup \\
$T \subset B$ & \ref{ss:definitions-G} &  Maximal torus and Borel subgroup \\
$h$ & \ref{ss:definitions-G} & Coxeter number \\
$\bX$, $\bX^+$ & \ref{ss:definitions-G} & Weight lattice and dominant weights \\
$\Phi$, $\Phi^\vee$ & \ref{ss:definitions-G} & Roots and coroots \\
$\Sigma \subset \Phi^+$ & \ref{ss:definitions-G} & Simple and positive roots \\
$\rho$ & \ref{ss:definitions-G} & half sum of positive roots \\
$\Sf \subset \Wf$ & \ref{ss:definitions-G} & (Finite) Weyl group and simple reflections \\
$\Saff \subset \Waff$ & \ref{ss:definitions-G} & Affine Weyl group and simple reflections \\
$C_\Z$ & \ref{ss:definitions-G} & Intersection of $\bX$ with fundamental alcove \\
$\lambda_0$ & \ref{ss:definitions-G} & Fixed weight in $C_\Z$ \\
$\mu_s$ & \ref{ss:definitions-G} & Fixed weight on the $s$-wall \\
$\bX_0^+$ & \ref{ss:definitions-G} & Intersection of $\bX^+$ and $\Waff \hdot \lambda_0$ \\
$\bX_s^+$ & \ref{ss:definitions-G} & Intersection of $\bX^+$ and $\Waff \hdot \mu_s$ \\
$\Rep(G)$ & \ref{ss:definitions-G} & Category of finite-dimensional algebraic $G$-modules \\
$\Sta(\lambda)$, $\Cos(\lambda)$ & \ref{ss:definitions-G} & Standard and costandard modules in $\Rep(G)$ \\
$\Sim(\lambda)$, $\Til(\lambda)$ & \ref{ss:definitions-G} & Simple and tilting modules in $\Rep(G)$ \\
$\Rep_0(G)$ & \ref{ss:definitions-G} & Serre subcategory generated by the $\Sim(\lambda)$'s with $\lambda \in \bX^+_0$ \\
$\Rep_s(G)$ & \ref{ss:definitions-G} & Serre subcategory generated by the $\Sim(\lambda)$'s with $\lambda \in \bX^+_s$ \\
$\Trans^s$, $\Trans_s$ & \ref{ss:translation-functors} & Fixed functors isomorphic to translation functors \\
$\Theta_s$ & \ref{ss:BStilting} & Wall-crossing functor $\Trans_s \Trans^s$ \\
$\Til(\uw)$ & \ref{ss:BStilting} & Bott--Samelson-type tilting module \\
\hline
\end{tabular}
\bigskip

Section~\ref{sec:Diag}
\vspace{1mm}

\begin{tabular}{|l|l|l|}
\hline
Notation & \S & Description \\
\hline
$\Haff$, $\Hf$ & \ref{ss:Haff-Masph} & Hecke algebras of $(\Waff,\Saff)$ and $(\Wf, \Sf)$ \\
$H_w$, $\uH_w$ & \ref{ss:Haff-Masph} & Standard and Kazhdan--Lusztig basis elements in $\Haff$ \\
$ \uH_{\uw}$ & \ref{ss:Haff-Masph} & Bott--Samelson-type element in $\Haff$ \\
$\Masph$ &  \ref{ss:Haff-Masph} & Antispherical $\Haff$-module \\
$N_w$, $\uN_w$ & \ref{ss:Haff-Masph} & Standard and Kazhdan--Lusztig basis elements in $\Masph$ \\
$\fh$ & \ref{ss:diag-SB} & Realization of $(\Waff,\Saff)$ over $\bk$ \\
$R$ & \ref{ss:diag-SB} & Symmetric algebra of $\fh^*$ \\
$\DiagBS$, $\Diag$ & \ref{ss:diag-SB} & Diagrammatic Hecke categories \\
$B_{\uw}$, $B_w$ & \ref{ss:diag-SB} & Bott--Samelson and indecomposable objects in $\Diag$ \\
$\imath$, $\tau$ & \ref{ss:diag-SB} & Autoequivalences of $\DiagBS$ \\
$\LL_{\uw,\ue}$ & \ref{ss:diag-SB} & Elements of the light leaves basis \\
$\DasphBS$, $\Dasph$ & \ref{ss:cat-Masph} & Antispherical diagrammatic Hecke categories \\
$\oB_{\uw}$, $\oB_w$ & \ref{ss:cat-Masph} & Bott--Samelson and indecomposable objects in $\Dasph$ \\
\hline
\end{tabular}
\bigskip

Section~\ref{sec:main-conj}
\vspace{1mm}

\begin{tabular}{|l|l|l|}
\hline
Notation & \S & Description \\
\hline
$\Psi$ & \ref{ss:tilting-antispherical} & Functor in Theorem~\ref{thm:main} \\
$\alpha_{\ux,\uy}$, $\beta_{\ux,\uy}$ & \ref{ss:surjectivity} & Morphisms induced by the functor $\Psi$ \\
$\mathcal{D}_{\mathrm{BS}, \mathfrak{R}}$ & \ref{ss:intforms} & Version of $\DiagBS$ over $\mathfrak{R}$ \\
$\mathcal{D}_{\mathrm{BS}, \mathfrak{R}}^{\mathrm{asph}}$ & \ref{ss:intforms} & Version of $\Dasph$ over $\mathfrak{R}$ \\
\hline
\end{tabular}
\bigskip


Section~\ref{sec:2rep-GLn}
\vspace{1mm}

\begin{tabular}{|l|l|l|}
\hline
Notation & \S & Description \\
\hline
$\hgl_N$ & \ref{slp} & Affine Lie algebra associated with $\gl_N(\C)$ \\
$K$, $d$ & \ref{slp} & Elements in $\hgl_N$ \\
$e_i$, $h_i$, $f_i$ & \ref{slp} & Chevalley elements in $\hgl_N$ \\
$\varepsilon_i$ & \ref{slp} & Basis of the dual of diagonal matrices in $\gl_N(\C)$ \\
$P$ & \ref{slp} & Weight lattice of $\hgl_N$ \\
$\alpha_i$ & \ref{slp} & Simple roots for $\hgl_N$ \\
$\nat_N$ & \ref{sec:natglN} & Natural module for $\hgl_N$ \\
$m_\lambda$ & \ref{sec:natglN} & Basis of $\nat$ \\
$G$ & \ref{ss:cat-GLn-Groth-group} & Group $\mathrm{GL}_n(\bk)$ \\
$B$ & \ref{ss:cat-GLn-Groth-group} & Borel subgroup of lower triangular matrices \\
$T$ & \ref{ss:cat-GLn-Groth-group} & Maximal torus of diagonal matrices \\
$\chi_i$ & \ref{ss:cat-GLn-Groth-group} & Standard basis of $\bX=X^*(T)$ \\
$\varsigma$ & \ref{ss:cat-GLn-Groth-group} & Element in $\bX$ \\
$V$ & \ref{ss:cat-GLn-Groth-group} & Natural representation of $V$ \\
$E$, $F$ & \ref{ss:cat-GLn-Groth-group} & Functors of tensoring with $V$ and $V^*$ \\
$\eta$, $\epsilon$ & \ref{ss:cat-GLn-Groth-group} & Adjunction morphisms for $(E,F)$ \\
$\Xbb$ & \ref{ss:cat-GLn-Groth-group} & Endomorphism of $E$ and $F$ \\
$E_a$, $F_a$ & \ref{ss:cat-GLn-Groth-group} & Generalized eigenspaces of $X$ on $E$ and $F$ \\
$\Repp_c(G)$ & \ref{ss:cat-GLn-Groth-group} & Direct factor of $\Rep(G)$ associated with $c \in \bX/(W,\hdot)$ \\
$\imath_n$ & \ref{ss:cat-GLn-Groth-group} & Bijection between weights of $\wedge^n \nat$ and $\bX/(W,\hdot)$ \\
$\partial_i$ & \ref{Rep-2-rep} & Demazure operators \\
$\overline{H}_m$ & \ref{Rep-2-rep} & Degenerate affine Hecke algebra \\
$\Gamma$ & \ref{Rep-2-rep} & Prime ring of the base field, considered as a quiver \\
$\overline{\mathcal{C}}_\Gamma$ &  \ref{Rep-2-rep} & Category of $\overline{H}_m$-modules supported on $\Gamma^m$ \\
$t_{ij}$ & \ref{Rep-2-rep} & Scalars in the definition of the KLR algebra \\
$H_m(\Gamma)$ & \ref{Rep-2-rep} & KLR algebra \\
$H_m(\Gamma)\Mod_0$ & \ref{Rep-2-rep} & Category of locally nilpotent modules for $H_m(\Gamma)$ \\
$\mathcal{U}(\hgl_p)$ & \ref{Rep-2-rep} & KLR 2-category attached to $\hgl_p$ \\
$\Tbb$ & \ref{Rep-2-rep} & Endomorphism of $EE$ \\
$\omega$ & \ref{Rep-2-rep} & Weight in $P$ \\
$\mu_{s_j}$, $\mu_{s_\infty}$ & \ref{Rep-2-rep} & Choice of weights for $G$ \\
\hline
\end{tabular}
\bigskip

Section~\ref{sec:restriction}
\vspace{1mm}

\begin{tabular}{|l|l|l|}
\hline
Notation & \S & Description \\
\hline
$\Rep^{[n]}(G)$ & \ref{ss:categorify-combinatorics} & Subcategory of $\Rep(G)$ \\
$Y_m$ & \ref{ss:first-relations} & Subset of $P$ \\
$\mathcal{U}_+(\hgl_p)$ & \ref{ss:first-relations} & Quotient of $\mathcal{U}(\hgl_p)$ \\
\hline
\end{tabular}
\bigskip
\bigskip

Section~\ref{sec:g2D}
\vspace{1mm}

\begin{tabular}{|l|l|l|}
\hline
Notation & \S & Description \\
\hline
$\mathcal{U}^{[n]}(\hgl_n)$ & \ref{ss:strategy} & Quotient of $\mathcal{U}(\hgl_n)$ \\
$\sigma$ & \ref{ss:strategy} & Functor from $\DiagBS$ to endomorphisms of $\omega$ \\
\hline
\end{tabular}
\vspace{5cm}


Section~\ref{sec:parity}
\vspace{1mm}

\begin{tabular}{|l|l|l|}
\hline
Notation & \S & Description \\
\hline
$A$, $I$ & \ref{ss:KMgroups} & Generalized Cartan matrix and set of parameters \\
$(\Lambda, \{\alpha_i\}, \{\alpha_i^\vee\})$ & \ref{ss:KMgroups}  & Kac--Moody root datum \\
$W$, $S$ & \ref{ss:KMgroups} & Weyl group and simple reflections \\
$\GKM$, $\BKM$, $\TKM$ & \ref{ss:KMgroups} & Kac--Moody group, Borel subgroup, maximal torus \\
$\Flag$ & \ref{ss:KMgroups} & Flag variety of $\GKM$ \\
$\Flag_w$, $i_w$ & \ref{ss:KMgroups} & Bruhat cell and embedding \\
$\BSvar(\uw)$ & \ref{ss:KMgroups} & Bott--Samelson resolution \\
$\nu_{\uw}$ & \ref{ss:KMgroups} & Morphism from $\BSvar(\uw)$ to $\Flag$ \\
$\PKM_J$ & \ref{ss:partial-flags} & Parabolic subgroup attached to $J \subset I$ \\
$\Flag^J$ & \ref{ss:partial-flags} & Partial flag variety attached to $J$ \\
$\Flag^J_w$, $i^J_w$ & \ref{ss:KMgroups} & Bruhat cell in $\Flag^J$ and embedding \\
$q^J$ & \ref{ss:KMgroups} & Projection $\Flag \to \Flag^J$ \\
$\K$ & \ref{ss:DbX} & ring of coefficients \\
$\Db_\BKM(\Flag, \K)$ & \ref{ss:DbX} & $\BKM$-equivariant derived category of $\Flag$ \\
$\Db_\BKM(\Flag^J, \K)$ & \ref{ss:DbX} & $\BKM$-equivariant derived category of $\Flag^J$ \\
$\Db_{(\BKM)}(\Flag, \K)$ & \ref{ss:DbX} & $\BKM$-constructible derived category of $\Flag$ \\
$\Db_{(\BKM)}(\Flag^J, \K)$ & \ref{ss:DbX} & $\BKM$-constructible derived category of $\Flag^J$ \\
$\Cos_w$, $\Cos_w^J$ & \ref{ss:DbX} & Costandard perverse sheaves on $\Flag$ and $\Flag^J$ \\
$\Sta_w$, $\Sta_w^J$ & \ref{ss:DbX} & Standard perverse sheaves on $\Flag$ and $\Flag^J$ \\
$\mu$ & \ref{ss:DbX} & Projection from $\GKM$ to $\Flag$ \\
$\star^\BKM$ & \ref{ss:DbX} & Convolution action of $\Db_\BKM(\Flag, \K)$ \\
$\Par_{\BKM}(\Flag, \K)$ & \ref{ss:parity-flag} & Category of $\BKM$-equiv.~parity complexes on $\Flag$ \\
$\Par_{(\BKM)}(\Flag, \K)$ & \ref{ss:parity-flag} & Category of $\BKM$-const.~parity complexes on $\Flag$ \\
$\Par_{\BKM}(\Flag^J, \K)$ & \ref{ss:parity-flag} & Category of $\BKM$-equiv.~parity complexes on $\Flag^J$ \\
$\Par_{(\BKM)}(\Flag^J, \K)$ & \ref{ss:parity-flag} & Category of $\BKM$-const.~parity complexes on $\Flag^J$ \\
$\cE_{\uw}$ & \ref{ss:parity-flag} & Bott--Samelson-type parity complex on $\Flag$ \\
$\cE_w$ & \ref{ss:parity-flag} & Indecomposable parity complex on $\Flag$ \\
$\cE^s_w$ & \ref{ss:parity-flag} & Indecomposable parity complex on $\Flag^s$ \\
$\Upsilon_s$ & \ref{ss:morphisms-BS-parity} & Functor on $\Db_{(\BKM)}(\Flag, \F)$ \\
\hline
\end{tabular}
\vspace{2cm}

Section~\ref{sec:parity-Hecke}
\vspace{1mm}

\begin{tabular}{|l|l|l|}
\hline
Notation & \S & Description \\
\hline
$\DiagBS^\K(\GKM)$, $\Diag^\K(\GKM)$ & \ref{ss:Diag-GKM} & Diagrammatic categories attached to $\GKM$ \\
$\cE_\K(\uw)$ & \ref{ss:BS-parity} & Canonical Bott--Samelson parity complex on $\Flag$ \\
$\Par_{\BKM}^{\mathsf{BS}}(\Flag, \K)$ & \ref{ss:BS-parity} & Category of Bott--Samelson parity complexes on $\Flag$ \\
$\Delta$, $\Delta_{\mathsf{BS}}$ & \ref{ss:statement-Diag-parity} & Functors from diagrams to parity complexes \\
$\mathbb{H}$ & \ref{ss:verification-Diag-parity} & Functor from parity complexes to $R$-bimodules \\
$B_s^B$ & \ref{ss:verification-Diag-parity} & Soergel bimodule attached to $s$ (over $\Q$) \\
$\gamma_{\uw,\uv}$, $\delta_{\uw,\uv}$ & \ref{ss:Diag-parity-ff} & Morphisms induced by the functor $\Delta_{\mathsf{BS}}$ \\
$\Gv$ & \ref{ss:affine-flag-variety} & Reductive group Langlands-dual to $G$ \\
$\Tv$, $\Bv$ & \ref{ss:affine-flag-variety} & Maximal torus and Borel subgroup in $\Gv$ \\
$\Gw$ & \ref{ss:affine-flag-variety} & S.-c.~cover of the derived subgroup of $\Gv$ \\
$\Tw$, $\Bw$ & \ref{ss:affine-flag-variety} & Maximal torus and Borel subgroup in $\Gw$ \\
$\mathscr{K}$, $\mathscr{O}$ & \ref{ss:affine-flag-variety} & Laurent series and power series over $\bbL$ \\
$\Iv$, $\Iw$ & \ref{ss:affine-flag-variety} & Iwahori subgroups for $\Gv$ and $\Gw$ \\
$\Fl^\vee$, $\Fl^\wedge$ & \ref{ss:affine-flag-variety} & Affine flag variety for $\Gv$ and $\Gw$ \\
$\Fl^\wedge_w$ & \ref{ss:affine-flag-variety} & Bruhat cell in $\Fl^\wedge$ \\
$\Db_{\Iw}(\Fl^\wedge, \K)$ & \ref{ss:affine-flag-variety} & $\Iw$-equivariant derived category of $\Fl^\wedge$ \\
$\star^{\Iw}$ & \ref{ss:affine-flag-variety} & Convolution bifunctor \\
$\Par^{\mathsf{BS}}_{\Iw}(\Fl^\wedge, \K)$ & \ref{ss:affine-flag-variety} & $\Iw$-equ. Bott--Samelson parity compl.~on $\Fl^\wedge$ \\
$\Par_{\Iw}(\Fl^\wedge, \K)$ & \ref{ss:affine-flag-variety} & $\Iw$-equ. parity complexes on $\Fl^\wedge$ \\
\hline
\end{tabular}
\bigskip
\bigskip

Section~\ref{sec:Whittaker}
\vspace{1mm}

\begin{tabular}{|l|l|l|}
\hline
Notation & \S & Description \\
\hline
$\psi$ & \ref{ss:IW} & Nontrivial character of the prime subfield of $\mathbb{F}$ \\
$\mathcal{L}_\psi$ & \ref{ss:IW} & Artin--Schreier local system attached to $\psi$ \\
$\UKM^J$ & \ref{ss:IW} & Pro-unipotent radical of $\PKM_J$ \\
$\LKM_J$ & \ref{ss:IW} & Levi factor of $\PKM_J$ \\
$\UKM^-_J$ & \ref{ss:IW} & Unip.~rad.~of the Borel subgr.~opp.~to $\LKM_J \cap \BKM$ \\
$\Flag_{w,J}$ & \ref{ss:IW} & $\UKM_J^- \UKM^J$-orbit on $\Flag$ \\
$\chi_J$ & \ref{ss:IW} & Character of $\UKM_J^- \UKM^J$ \\
$\Db_{\Whit,J}(\Flag, \K)$ & \ref{ss:IW} & Whittaker derived category of $\Flag$ \\
$\cL_w^J$ & \ref{ss:IW} & Local system on $\Flag_{w,J}$ \\
$\Sta^J_{w,\circ}$ & \ref{ss:IW} & Standard perverse sheaf in $\Db_{\Whit,J}(\Flag, \K)$ \\
$\Cos^J_{w,\circ}$ & \ref{ss:IW} & Costandard perverse sheaf in $\Db_{\Whit,J}(\Flag, \K)$ \\
$\mathsf{Av}_J$ & \ref{ss:IW} & Averaging functor \\
$\Db_{\Whit,J}(\Flag^s, \K)$ & \ref{ss:IW} & Whittaker derived category of $\Flag$ \\
$\Par_{\Whit,J}(\Flag, \K)$ & \ref{ss:IW-parity} & Parity complexes in $\Db_{\Whit,J}(\Flag, \K)$ \\
$\Par_{\Whit,J}(\Flag^s, \K)$ & \ref{ss:IW-parity} & Parity complexes in $\Db_{\Whit,J}(\Flag^s, \K)$ \\
$\cE_{w,\circ}^J$ & \ref{ss:section-Whit} & Indec.~Whittaker parity complexes on $\Flag$ \\
$\cE_{w,\circ}^{s,J}$ & \ref{ss:section-Whit} & Indec.~Whittaker parity complexes on $\Flag$ \\
$\Delta^J$, $\Delta^J_{\mathsf{BS}}$ & \ref{ss:Dasph-parity} & Functors from diag.~to Whittaker par.~compl. \\
$\Iw_\circ$ & \ref{ss:Whit-affine} & Opposite Iwahori subgroup for $\Gw$ \\
$\Par_{(\Iw_\circ, \chi)}(\Fl^\wedge, \K)$ & \ref{ss:Whit-affine} & Iwahori--Whittaker parity complexes on $\Fl^\wedge$ \\
\hline
\end{tabular}

\end{document}